\documentclass[reqno,10pt]{amsart}
\usepackage{amsthm,amsfonts,amssymb,euscript,mathrsfs,graphics,color,amsmath,latexsym,marginnote,hyperref}
\usepackage[latin1]{inputenc}  

\theoremstyle{plain}

\usepackage{times}
\numberwithin{equation}{section}

\newcommand{\tb}{\mathtt{b}}

\newcommand{\calO}{\mathcal{O}}

\newcommand{\td}{{\mathtt D}}
\newcommand{\gotp}{{\mathfrak p}}

\def\td{{\mathtt d}}

\newtheorem{Teo}{Theorem}
\newtheorem{Def}{Definition}
\newtheorem{theorem}{Theorem}[section]
\newtheorem{proposition}[theorem]{Proposition}

\newtheorem{lemma}[theorem]{Lemma}

\newtheorem{remark}[theorem]{Remark}
\newtheorem{remarks}[theorem]{Remark}
\newtheorem{definition}[theorem]{Definition}
\newcommand{\be}{\begin{equation}}
\newcommand{\ee}{\end{equation}}

\newcommand{\e}{\varepsilon}

\newcommand{\al}{\alpha}

\newcommand{\ov}{\overline}

\newcommand{\R}{\mathbb R}
\newcommand{\C}{\mathbb C}

\newcommand{\Z}{\mathbb Z}

\newcommand{\N}{\mathbb N}
\newcommand{\T}{\mathbb T}
\newcommand{\sign}{{\rm sign}}

\newcommand{\s }{\sigma }
\newcommand{\ii }{{\rm i} }
\renewcommand{\d }{\delta }

\newcommand{\g }{\gamma}

\newcommand{\m }{\mu }
\newcommand{\vphi}{\varphi }

\newcommand{\bral}{[ \! [} 
\newcommand{\brar}{] \! ]}

\newcommand{\z }{\zeta }

\newcommand{\dps}{\displaystyle}

\newcommand{\gotL}{\mathfrak{L}}

\newcommand{\x }{\xi }

\newcommand{\gotM}{\mathfrak{M}}

\newcommand{\pa}{\partial}

\newcommand{\op}{{Op}}
\newcommand{\opw}{{Op^{\mathrm{W}}}}

\def\hat{\widehat}

\def\bar{\overline}

\def\cal{\mathcal}

\newcommand{\f}{\vphi}
\newcommand{\su}{\s_1}

\def\ba{\begin{aligned}}
\def\ea{\end{aligned}}
\def\beginm{\begin{multline}}
\def\endm{\end{multline}}

\providecommand{\vect}[2]{{\bigl[\begin{smallmatrix}#1\\#2\end{smallmatrix}\bigr]}} \providecommand{\sign}{\mathrm{sgn}\,}  
\providecommand{\sm}[4]{{\bigl[\begin{smallmatrix}#1&#2\\#3&#4\end{smallmatrix}\bigr]}}

\makeatletter

\renewcommand{\tocsection}[3]{%
  \indentlabel{\@ifnotempty{#2}{\bfseries\ignorespaces#1 #2\quad}}\bfseries#3}
\def\l@subsection{\@tocline{2}{0pt}{2.5pc}{5pc}{}}
\def\l@subsubsection{\@tocline{3}{0pt}{4.5pc}{5pc}{}}

\renewcommand\tocchapter[3]{%
  \indentlabel{\@ifnotempty{#2}{\ignorespaces#2.\quad}}#3%
}
\newcommand\@dotsep{4.5}
\def\@tocline#1#2#3#4#5#6#7{\relax
  \ifnum #1>\c@tocdepth % then omit
  \else
    \par \addpenalty\@secpenalty\addvspace{#2}%
    \begingroup \hyphenpenalty\@M
    \@ifempty{#4}{%
      \@tempdima\csname r@tocindent\number#1\endcsname\relax
    }{%
      \@tempdima#4\relax
    }%
    \parindent\z@ \leftskip#3\relax \advance\leftskip\@tempdima\relax
    \rightskip\@pnumwidth plus1em \parfillskip-\@pnumwidth
    #5\leavevmode\hskip-\@tempdima{#6}\nobreak
    \leaders\hbox{$\m@th\mkern \@dotsep mu\hbox{.}\mkern \@dotsep mu$}\hfill
    \nobreak
    \hbox to\@pnumwidth{\@tocpagenum{#7}}\par
    \nobreak
    \endgroup
  \fi}
\def\l@subsection{\@tocline{2}{0pt}{2.5pc}{5pc}{}}

\begin{document}

\title{Quasi-periodic Traveling Waves on an Infinitely Deep Perfect Fluid Under Gravity
%solutions of  the periodic gravity Water Waves system
} %in infinite depth}

\date{}

\author{Roberto Feola}
\address{University of Nantes}
\email{roberto.feola@univ-nantes.fr}

\author{Filippo Giuliani}
\address{UPC, Barcelona}
\email{filippo.giuliani@upc.edu}

\thanks{
Roberto Feola has been supported of the Centre Henri Lebesgue 
ANR-11-LABX-0020-01 and by ANR-15-CE40-0001-02 
``BEKAM'' of the Agence Nationale de la Recherche.
Filippo Giuliani has received funding from 
the European Research Council (ERC) 
under the European Union's Horizon 2020
research and innovation programme under grant agreement 
No 757802.}

%\thanks{This research was supported by PRIN 2015 ``Variational methods, with applications to problems in 
%mathematical physics and geometry''.
%The third author was supported in part by Start-up grants from Princeton University and the University of Toronto,  
%and NSERC grant RGPIN-201}

 \begin{abstract} 
We consider the gravity water  
waves system with a periodic one-dimensional interface 
in infinite depth and we establish the existence and the linear stability of small amplitude, 
quasi-periodic in time, traveling waves.
This provides the first existence result of quasi-periodic water waves 
solutions bifurcating from a \emph{completely resonant} elliptic fixed point. 
The proof is based on a Nash-Moser scheme, Birkhoff normal form methods 
and pseudo differential calculus techniques. We deal with the combined problems 
of \emph{small divisors}
and the \emph{fully-nonlinear} nature of the equations.\\
The lack of parameters, like the capillarity or the depth of the ocean, 
demands a refined \emph{nonlinear} bifurcation analysis involving 
several non-trivial resonant wave interactions, as the well-known ``Benjamin-Feir resonances''. 
We develop a novel normal form approach to deal with that.
Moreover, by making full use of the Hamiltonian structure, 
we are able to provide the existence of a wide class of solutions 
which are free from restrictions of parity in the time and space variables.
\end{abstract}

\maketitle
\tableofcontents

\section{Introduction}\label{sec:1}

We consider a two dimensional (1d interface)
incompressible and irrotational perfect fluid
with periodic boundary conditions under the action of gravity.
The fluid layer is assumed to be infinitely deep
and the motion is governed by the free surface Euler equations.  
This paper is concerned with the existence and the stability of small amplitude, quasi-periodic in time, traveling waves
on the surface of the fluid. By a quasi-periodic traveling wave we mean a motion that, at the first orders of amplitude, is a superposition of an arbitrarily large number of periodic traveling waves with rationally independent frequencies.
The irrationality of the frequencies of oscillations 
excludes the existence of a moving frame for 
which such motions are stationary. 
As a consequence the traveling quasi-periodic
 waves are ``truly'' time dependent solutions and the search for them is a \emph{small divisors} problem. 
A natural approach to deal with it 
is to implement a Nash-Moser implicit function scheme.

\noindent
Besides the small divisors difficulty, the pure gravity water waves system with infinite depth presents several issues regarding the bifurcation of small amplitude solutions.
Indeed the absence of parameters, like the capillarity of the fluid or the depth of the ocean, makes the linearized problem at the rest surface \emph{completely resonant}, in the sense that the kernel of the linearized operator has infinite dimension.
%The combination of the completely resonance and the small divisors problem requires a \emph{nonlinear} bifurcation analysis. \\

\noindent
In the search for periodic waves, the same issues have been tackled by Iooss-Plotnikov-Toland for proving the existence of standing periodic unimodal waves \cite{IoossPloTol1} and, later, by Iooss-Plotnikov \cite{IossPlo111} for the multimodal generalization. \\
In the quasi-periodic case, the complexity of the resonant waves interactions and the analysis of the linearized operator in a neighborhood of the equilibrium demand more refined techniques.\\  
The first results on the existence of quasi-periodic waves have been provided only recently  
and are due to Berti-Montalto \cite{BM1} for the gravity-capillary 
case with infinite depth and by Baldi-Berti-Haus-Montalto \cite{BBHM} 
for the pure gravity case with finite depth for $2$d oceans. 
In both cases the existence of quasi-periodic solutions
is provided for some asymptotically full-measure set of the parameters of the problem, respectively capillarity and depth (or equivalently wavelength). We remark that such solutions are standing, i.e. even in space, and reversible (they enjoy an additional symmetry in the space-time variable).\\
As far as we know, all the previous results on periodic and quasi-periodic in time water waves take advantage from the presence of physical parameters and / or assumptions of parity conditions.\\
The purpose of the present paper is to address two natural questions: (i) we work on a fixed equation for which the only possible parameters to modulate are the initial data of the solutions.\\
(ii) we look for a general class of quasi-periodic traveling waves which are free from restrictions of parity in the spatial and time variables.
%To the best of our knowledge, all the previous results on periodic and quasi-periodic in time water waves take advantage from the presence of parameters and/or the simplification coming from symmetries
%

%We shall use Birkhoff normal form methods and fully exploit the (approximate) constants of motion of the pure gravity water waves system.\\
% the natural question of finding a (large) family of time quasi-periodic waves parametrized by their \emph{amplitudes}. Due to the complete resonant nature of the system such solutions are to be found by continuation of \emph{nonlinear} waves.\\
 We deal with these issues by using Birkhoff normal form methods and fully exploiting the (approximate) constants of motion of the pure gravity water waves system.\\ 
The theorem that we prove is given below and it has been announced in \cite{FGnote}.
\begin{Teo}%{\bf (Quasi-periodic traveling waves).}
\label{thm:mainmain}
There exist non-trivial small amplitude, linearly stable quasi-periodic traveling waves solutions 
of the two dimensional pure gravity water waves problem in infinite depth.
\end{Teo}
To the best of our knowledge the above theorem is the first existence result concerning
quasi-periodic solutions of the water waves equations bifurcating from a completely resonant elliptic fixed point. 
%traveling waves for 
%irrotational fluids under the action of gravity.

%\comment{As far as we know the theorem  above is the first existence result concerning
%quasi-periodic 
%traveling waves for 
%irrotational fluids under the action of gravity.
%We mention the very recent paper  \cite{BFMasp}
%by Berti, Franzoi, Maspero, where is proved the bifurcation of quasi-periodic
%traveling waves for fluids with constant vorticity under the action of 
%gravity and capillary forces.
%Due to the presence of the surface tension such result
%result, obtained simultaneously with ours,
%is based on different techniques and holds  for ``almost all '' values
%of the capillarity parameter.
% }
% \comment{qui scriverei : Moreover, we actually provide the existence...}
%Such restrictions are usually exploited to simplify the analysis of the linearized operator.\\
%\comment{al posto di ``As far as we know'' metterei: To the best of our knowledge}
%As far as we know, 
%{
%Actually we provide the existence of a general class of quasi-periodic traveling waves, meaning that those solutions are free from restrictions of parity in the spatial and time variables. 
%To the best of our knowledge, all the previous results on periodic and quasi-periodic in time water waves take advantage from these restrictions to simplify the analysis of the linearized operator. 
%In the present paper we adopt a new approach that allows to get rid of them by making full use of the known conserved quantities of the system, which are the Hamiltonian and the momentum. \\}
As it is well known one of the main difficulties in the search for quasi-periodic solutions is to deal with the resonances between the frequencies of the expected solutions and the eigenvalues of the linearized operator.\\
This requires a complete control on the frequencies and a good knowledge of the spectrum in a whole neighborhood of the equilibrium. Especially in the infinite dimensional context, this demands the use of some parameters that modulate the frequencies.\\
In a completely resonant case such parameters are to be found through a \emph{nonlinear} bifurcation analysis. In the small amplitude regime this can be successfully done by performing Birkhoff normal form methods. The bifurcation parameters 
appear naturally as the amplitudes of 
an appropriate approximate solution, which is 
 obtained by 
exciting  a finite number  of modes that we call  \emph{tangential sites}.  It turns out that the tangential sites have to be chosen in a suitable way and such choice is a fundamental ingredient of the proof of Theorem \ref{thm:mainmain}.
% the existence of the quasi-periodic solutions given by Theorem 1 strongly depends on the choice of the tangential sites.
\\
A serious  difficulty is to prove that the amplitudes
%a finite number of parameters 
provide a sufficiently good  modulation 
 to impose \emph{infinitely many} non-resonance conditions.
This is usually achieved by exploiting some 
non-degeneracy
 of the Birkhoff normal form 
at order four, which in turn depends on the choice of the tangential sites. 
Proving the non-degeneracy of the normal form is tantamount to prove a "twist" 
condition between the bifurcation parameters and the first 
order corrections of the eigenvalues 
of the linearized operator at the equilibrium.
%The non-degeneracy %of the normal form 
%depends on the choice of the tangential sites.
%We use algebraic arguments to prove that 
%the degeneracy of the normal form occurs  only for
%\emph{non-generic}  choices of the tangential sites.
%the non-degeneracy conditions hold for a \emph{generic} choice
%of the tangential sites.
\\
Due to the quasi-linear nature of the equations and the presence of non-trivial resonances at order four (called  \emph{Benjamin-Feir resonances}) the proof of the aforementioned "twist" 
condition requires the implementation of a new set of ideas.\\ 
The key ingredient to overcome these problems is a novel identification argument of normal forms. We refer to section \ref{sec:integrability} for a detailed discussion. We remark that, by the lack of parameters, classical uniqueness arguments concerning \emph{non-resonant} Birkhoff normal forms (see for instance \cite{KdVeKAM}) cannot be applied. 
Actually our approach relies on the presence 
of approximate conserved quantities and the formal integrability of the Hamiltonian of order four proved in \cite{Zak2}, \cite{CW}, \cite{CS}.

\subsection{Formulation of the problem}\label{sec:1.1}
We consider an incompressible and irrotational perfect fluid, under the action of gravity
occupying, at time
$t$, a two dimensional domain with infinite depth, periodic in the horizontal variable, given by 
\be\label{Deta}
    {\mathcal D}_{\eta} := \big\{ (x,y)\in \T \times\R \, ;  \ - \infty <y<\eta(t,x) \big\},  \quad 
    \T := \R \slash (2 \pi \Z) \, ,
\ee 
where $\eta$ %:\R\times\T\to \R$ 
is a smooth function. 
The velocity field in the time dependent domain 
$ {\mathcal D}_{\eta} $ is the gradient of a harmonic function $\Phi$, called the velocity potential. 
The time-evolution of the fluid is determined by a system of equations for 
the two functions $(t,x)\to \eta(t,x) $, $ (t,x,y)\to \Phi(t,x,y)$.
Following Zakharov \cite{Zak1} and Craig-Sulem \cite{CrSu} we denote by
$\psi(t,x) = \Phi(t,x,\eta(t,x))$ the restriction 
of the velocity potential to the free interface.
Given the shape $\eta(t,x)$ of the domain $ {\cal D}_\eta $ 
and the Dirichlet value $\psi(t,x)$ of the velocity potential at the top boundary, one can recover 
$\Phi(t,x,y)$ as the unique solution of the elliptic problem 
\begin{equation} \label{BoundaryPr}
\Delta \Phi = 0  \ \text{in } 
{\cal D}_\eta  \, , \quad 
\partial_y \Phi \to 0  \ \text{as } y \to - \infty \, , \quad 
\Phi = \psi \ \;\; \text{on }\;\; \{y = \eta(t,x)\}. 
\end{equation}
The $(\eta,\psi)$ variables then satisfy the gravity water waves system
\begin{equation} \label{eq:113}
\begin{cases}
    \partial_t \eta = G(\eta)\psi \cr
\partial_t\psi = \dps -g\eta  -\frac{1}{2} \psi_x^2 +  \frac{1}{2}\frac{(\eta_x  \psi_x + G(\eta)\psi)^2}{1+\eta_x^2}
\end{cases}
\end{equation}
where $ G(\eta)\psi $ is the Dirichlet-Neumann operator  
\begin{equation}
  \label{eq:112a}
  G(\eta)\psi := \sqrt{1+\eta_x^2}(\partial_n\Phi)\vert_{y=\eta(t,x)} = (\partial_y\Phi -\eta_x \partial_x\Phi)(t,x,\eta(t,x))
\end{equation}
and $n$ is the outward unit normal at the free interface $y= \eta(t,x)$.
%The Dirichlet-Neumann operator
$ G(\eta)$ is a pseudo differential operator with principal symbol $|D|:=-\mathrm{i} \partial_x$
,
self-adjoint with respect to the $L^2$ scalar product, positive-semidefinite, 
and its kernel contains only the constant functions. 
Without loss of generality, we set the gravity constant to $g = 1$.

It was first observed by Zakharov \cite{Zak1} that \eqref{eq:113} is a Hamiltonian system with respect to the symplectic form $d\psi\wedge d\eta$ and it can be written as
\begin{equation}\label{HS}
\begin{aligned}
& \qquad \pa_t \eta = \nabla_\psi H (\eta, \psi) \, , \quad  \pa_t \psi = - \nabla_\eta H (\eta, \psi)  \, , 
\end{aligned}
\end{equation}
where $ \nabla $ denotes the $ L^2 $-gradient, with Hamiltonian
\begin{equation}\label{Hamiltonian}
H(\eta, \psi) := \frac12 \int_\T \psi \, G(\eta ) \psi \, dx + \frac{1}{2} \int_{\T} \eta^2  \, dx
\end{equation}
given by the sum of the kinetic 
and potential energy of the fluid.
%{
%Recall that  the Poisson bracket between functions  $ H(\eta,\psi), F(\eta,\psi)$ is defined as
%\begin{equation}\label{poissonBra}
%\{F, H\}=\int_{\T}\big(\nabla_{\eta}H\nabla_{\psi}F-\nabla_{\psi}H\nabla_{\eta}F\big)dx \, .
%\end{equation}
%}
The invariance of the system \eqref{eq:113} in the $y$ and $x$ variable implies the existence of two prime integrals, respectively the ``mass'' $\int_\T \eta \, dx$ and the momentum 
\begin{equation}\label{Hammomento}
\mathtt{M}:= \int_{\T} \eta_x (x) \psi (x)  \, dx. 
\end{equation}

The Hamiltonian \eqref{Hamiltonian}  is defined on the spaces
\be\label{phase-space}
(\eta, \psi) \in H^s_0 (\T; \mathbb{R}) \times {\dot H}^s (\T; \mathbb{R})  
\ee
where $ H^s (\T;\mathbb{R})$, $ s\in \R $, denotes the Sobolev space of $ 2 \pi $-periodic functions of $ x $, ${\dot H}^s (\T;\mathbb{R}):= H^s (\T;\mathbb{R}) \slash{ \sim}$ 
is the homogeneous Sobolev space
obtained by the equivalence relation
$\psi_1 (x) \sim \psi_2 (x)$ if and only if $ \psi_1 (x) - \psi_2 (x) $ is a constant\footnote{The 
fact that $ \psi \in \dot H^s $ is coherent with the fact 
that only the velocity field $ \nabla_{x,y} \Phi $ has physical meaning, 
and the velocity potential $ \Phi $ is defined up to a constant.
For simplicity of notation 
we denote the equivalence class $ [\psi] $ by $ \psi $ and,  
since the quotient map induces an isometry of $ {\dot H}^s (\T;\mathbb{R}) $ onto 
$ H^s_0 (\T;\mathbb{R}) $, 
we will conveniently identify $ \psi $ with a function with zero average.},
and $H^s_0(\T;\mathbb{R})$ is the subspace of 
$H^s(\T;\mathbb{R})$ of zero average functions. \\
Since 
$ \widehat \eta_0 (t)  := \frac{1}{2\pi} \int_{\T} \eta (t, x) \, dx $, 
$\widehat \psi_0 (t) := \frac{1}{2\pi} \int_{\T} \psi (t, x) \, dx $
evolve according to the decoupled equations\footnote{Since  the ocean has infinite depth,   if $\Phi$ solves \eqref{BoundaryPr},
then $\Phi_{c}(x,y):=\Phi(x,y-c)$ solves the same problem in $\mathcal{D}_{\eta+c}$ 
assuming the Dirichlet datum $\psi$ at the free boundary
 $\eta+c $. Therefore $G(\eta+c)=G(\eta) $, $ \forall c \in \R $, 
and $ \int_\T \nabla_\eta K \, dx = 0$ where $ K := \frac12 \int_{\T} \psi G(\eta) \psi \, dx $ denotes the kinetic energy.}
\be\label{medie00}
\pa_t {\widehat \eta}_0 (t) = 0 \, , \quad \pa_t  \widehat \psi_0 (t) = - g \widehat \eta_0 (t) \, ,  
\ee
we may restrict the study of the dynamics to the invariant subspace of $(\eta, \psi)$ such that
\begin{equation}\label{zeromode}
\int_{\T} \eta \, dx = \int_{\T} \psi \, dx = 0  \, . 
\end{equation}

\noindent
{\bf Linear water waves system.} Small amplitude solutions are close to the solutions of the linearized system of \eqref{eq:113}
at the equilibrium $(\eta,\psi)=(0,0)$, namely
\begin{equation} \label{eq:113linear}
\begin{cases}
    \partial_t \eta =G(0)\psi \cr
\partial_t\psi = \dps -\eta  
\end{cases}
\end{equation}
where 
 the Dirichlet-Neumann operator 
at the flat surface $\eta=0$ is the Fourier multiplier
$G(0)=|D|$.
The solution of the linear system \eqref{eq:113linear} restricted to the subspace \eqref{zeromode}
are 
\begin{equation}\label{solLineari}
\begin{aligned}
\eta(t,x)&=\frac{1}{\sqrt{2\pi}}\sum_{j\in \mathbb{Z}\setminus\{0\}}\Big(
\eta_{j}(0)\cos(\sqrt{|j|} t)+
|j|^{\frac{1}{2}}\psi_{j}(0)\sin(\sqrt{|j|}t)\Big)e^{\ii jx}\,,\\
\psi(t,x)&=\frac{1}{\sqrt{2\pi}}\sum_{j\in \mathbb{Z}\setminus\{0\}} \Big( 
\psi_{j}(0)\cos(\sqrt{|j|}t)- 
|j|^{-\frac{1}{2}}\eta_{j}(0)\sin(\sqrt{|j|}t)\Big)e^{\ii jx}
\end{aligned}
\end{equation}
with
coefficients satisfying $\ov{\eta_{-j}(0)}=\eta_{j}(0)$
and $\ov{\psi_{-j}(0)}=\psi_{j}(0)$.
The frequency of oscillation of the $j$-th mode is $\sqrt{|j|}$ and
we refer to the map 
\begin{equation}\label{dispersionLaw}
j\to \sqrt{|j|}\,,\quad j\in \mathbb{Z}\setminus\{0\}\,,
\end{equation}
as the \emph{dispersion law} of \eqref{eq:113linear}.
We note that the dispersion law is even in $j$, then there are infinitely many multiple eigenvalues. 

\noindent
Passing to the complex coordinates, the system \eqref{eq:113linear}
is equivalent to the \emph{completely resonant} equation
\begin{equation} \label{eq:113CWW}
%\pa_{t}u_{j}=-\ii \sqrt{|j|}u_{j}
\pa_{t}u=-\ii |D|^{\frac{1}{2}} u\,,\qquad 
u=\frac{1}{\sqrt{2}}\big( |D|^{-\frac{1}{4}}\eta+\ii|D|^{\frac{1}{4}}\psi \big)\,.
\end{equation}
The Fourier multipliers in \eqref{eq:113CWW} are well-defined thanks to
the choice \eqref{zeromode}. The solutions \eqref{solLineari} assume the form
\begin{equation}\label{solLinearicomplex}
\begin{aligned}
u(t,x)&=\sum_{j\in \mathbb{Z}\setminus\{0\}}
u_{j}(0)e^{-\ii\sqrt{|j|}t+\ii jx}\,,
\qquad u_j(0):=\frac{1}{\sqrt{2}}\big( |j|^{-\frac{1}{4}}\eta_j(0)+\ii|j|^{\frac{1}{4}}\psi_{j}(0) \big)\,
\end{aligned}
\end{equation}
and it is clear that they can be either periodic or quasi-periodic depending on the Fourier support.
%In this paper we look for  solutions 
%of \eqref{eq:113} which are \emph{quasi}-periodic in time 
%and \emph{traveling}
%according to the following definition.

\vspace{0.3em}
\noindent
{\bf Traveling quasi-periodic solutions.}
A \emph{quasi}-periodic solution for the system \eqref{eq:113}
 with an irrational frequency vector $\omega\in\mathbb{R}^{\nu}$, $\nu\geq 1$, i.e. $\omega \cdot \ell\neq 0$ for all $\ell\in\mathbb{Z}^{\nu}\setminus\{0\}$, is defined by a smooth embedding  
\begin{equation}\label{6.6}
\begin{aligned}
\T^{\nu} & \to  
H_0^1(\mathbb{T};\mathbb{R})\times
 {H}_0^1(\mathbb{T};\mathbb{R})\\
 \varphi & \mapsto U(\varphi,x):=(\tilde{\eta}(\varphi, x), \tilde{\psi}(\varphi, x))
 \end{aligned}
\end{equation}
such that 
\begin{equation}\label{invarianceequationH}
U \circ \Psi^t_{\omega}=\Phi^t_H\circ U\,, 
\qquad \Psi^t_{\omega}(\varphi_0):=\varphi_0+\omega t\,,
\end{equation}
where $\Phi^t_H$ is the flow of \eqref{Hamiltonian}. 
By differentiating \eqref{invarianceequationH} at $t=0$ 
we get that $U$ has to be solution of the following partial differential equation
\begin{equation}\label{HamiltonianequationH}
\omega\cdot \partial_{\varphi} U-X_H(U)=0\,,
\end{equation}
where $X_H$ corresponds to the r.h.s of the system \eqref{eq:113}.\\
Since the Hamiltonians $H$ and $\mathtt{M}$ commute,
then $s\mapsto \Phi_{\mathtt{M}}^{s}\circ\Phi_{H}^{t}\circ U$, 
where $\Phi^t_{\mathtt{M}}$ is the flow of \eqref{Hammomento},
 is a one-parameter group of solutions of \eqref{HamiltonianequationH}. 
 An embedding $U$ as in \eqref{6.6}, with frequency
 $\mathtt{v}\in \mathbb{R}^{\nu}$, which is a solution of the Hamiltonian $\mathtt{M}$
 in \eqref{Hammomento}   is such that
\begin{equation}\label{invarianceequationM}
U \circ \Psi^{s}_{-\mathtt{v}}=\Phi^{s}_{\mathtt{M}}\circ U\,,\qquad
\leftrightarrow 
\qquad U(\vphi_0-\mathtt{v}s,x)=U(\vphi_0,x+s)\,.
\end{equation}
We refer to $\mathtt{v}$ as its  \emph{velocity vector}.
 By differentiating \eqref{invarianceequationM} at $s=0$ we find the following quasi-periodic transport equation with constant coefficients
\begin{equation}\label{HamiltonianequationM}
\mathtt{v}\cdot \partial_{\varphi} U+X_{\mathtt{M}}(U)=0, \qquad X_{\mathtt{M}}(U)=(\tilde{\eta}_x, \tilde{\psi}_x).
\end{equation}
If $U$ satisfies \eqref{invarianceequationH}, \eqref{invarianceequationM} then
\begin{equation}\label{parametergroup}
U(\vphi_0+\omega t, x+s)=
\Phi^t_H\circ \Phi_{\mathtt{M}}^{s} \circ U
%=U(\varphi_0-\mathtt{v} s,x)=
%U(\cdot-\mathtt{v} s,x)\circ\Psi_{\omega}^t(\varphi_0)
=U \circ \Psi^s_{-\mathtt{v}}\circ\Psi_{\omega}^t(\varphi_0)
=U(\varphi_0+\omega t-\mathtt{v} s,x)\,.\footnote{Note that the freedom in the choice of the phase $\varphi_0$ comes from the autonomous nature of the systems involved. } 
%=U(\vphi_0+\omega t,x+s)\,.
\end{equation}

The \emph{quasi-periodic traveling waves} we look for are solutions of the form \eqref{parametergroup}. We point out that the linear solutions \eqref{solLineari} are of this type with velocity vector depending on the Fourier support. 
%This will determine the choice of the velocity vector of the expected solutions.

\begin{Def}{\bf (Quasi-periodic traveling waves).}
$(i)$ We say that a function $(\eta(t,x),\psi(t,x)) : \mathbb{R}\times \mathbb{T}\to\mathbb{R}^{2}$
is a \emph{quasi}-periodic solution of \eqref{eq:113} with  irrational 
frequency vector $\omega\in \mathbb{R}^{\nu}$, if there is an embedding 
$U\colon \mathbb{T}^{\nu}\to \mathbb{R}^2$ as in \eqref{6.6}
%\begin{equation*}%\label{6.6}
%\T^{\nu}\ni\vphi\mapsto U(\vphi,x) \in  
%H_0^1(\mathbb{T};\mathbb{R})\times
%{H}_0^1(\mathbb{T};\mathbb{R})
%\end{equation*}
such that $(\eta(t,x),\psi(t,x))=U(\omega t,x)$ solves \eqref{eq:113}.

\noindent
$(ii)$ A quasi-periodic solution is \emph{traveling} with \emph{velocity vector}
$\mathtt{v}\in \mathbb{Z}^{\nu}$ if there is a function 
$\widetilde{U} : \mathbb{T}^{\nu}\to \mathbb{R}^{2}$ such that
\begin{equation}\label{travellingWaves}
(\eta(t,x),\psi(t,x))=U(\omega t, x)=\widetilde{U}(\omega t-\mathtt{v}x)\,.
\end{equation}
\end{Def}
We remark that an 
embedding $U$ satisfies  \eqref{travellingWaves}, i.e.
$U(\vphi,x)=\widetilde{U}(\vphi-\mathtt{v}x)$ for all $\varphi\in\mathbb{T}^{\nu}$ (because of the irrationality of $\omega$), if and only if $U(\vphi,x)$
solves the transport equation \eqref{HamiltonianequationM}.\\
We shall construct 
such solutions localized  
 in Fourier space at 
$\nu$ distinct  \textit{tangential sites} 
\begin{equation}\label{TangentialSitesDP}
S:=S^{+}\cup S^{-}\,, \quad
S^+:=\{\overline{\jmath}_1, \dots, \overline{\jmath}_m \}\subset \mathbb{N}\setminus\{0\}\,,\quad
S^-:=\{ \overline{\jmath}_{m+1}, \dots, \overline{\jmath}_{\nu}\}\subset -\mathbb{N}\setminus\{0\}\,,
\end{equation}
for some  $1\leq m\leq \nu$ and where
\begin{equation}\label{TangentialSitesDPbis}
k \neq -j\,,\quad \forall j\in S^+,\,\,\forall k\in S^-.
\end{equation}
Let us denote
\begin{equation}\label{maxS}
\max(S):=\max\{ \lvert j \rvert \,:\,  j\in S\}.
\end{equation}
%\begin{equation}\label{TangentialSitesDPbis}
%\overline{\jmath}_{k}\neq -\overline{\jmath}_i\,,\quad \forall \;1\leq k\leq m\,,\;\;\; 
%m+1\leq i\leq\nu\,.
%\end{equation}
The solutions of \eqref{eq:113linear} that originate by exciting the tangential modes are superpositions of periodic traveling linear waves with velocity $\overline{\jmath}_i$ and frequency $\sqrt{|\overline{\jmath}_i|}$. Such motions are quasi-periodic (or periodic) traveling waves of the form \eqref{travellingWaves} with frequency vector
\begin{equation}\label{LinearFreqDP}
\overline{\omega}:=\left(
\sqrt{|\overline{\jmath}_1|},\ldots, \sqrt{|\overline{\jmath}_{\nu}|}
\right)\in\mathbb{R}^{\nu}
\end{equation}
and velocity vector
\begin{equation}\label{velocityvec}
\mathtt{v}:=\left( \overline{\jmath}_1, \dots, \overline{\jmath}_{\nu} \right)\in\mathbb{Z}^{\nu}.
\end{equation}
%Notice that the linear solution 
%$u(t,x)=\frac{1}{\sqrt{2\pi}}\sum_{j\in S}u_{j}(0)e^{-\ii\sqrt{j}t+\ii jx }$ 
%(which is the restriction to $S$ of the function in  \eqref{solLinearicomplex})
%can be written in the form (recall \eqref{LinearFreqDP}, \eqref{velocityvec}) 
%\begin{equation}\label{UTTT}
%u(t,x)=G(\,\ov{\omega} t-\mathtt{v}x)\,, \quad G\in H^{s}(\mathbb{T}^{\nu};\mathbb{C})\,.
%\end{equation}
%We  define  the map $\mathtt{l} : S\to \mathbb{Z}^{\nu}$ as
%\begin{equation}\label{LTTT}
%\ov{\jmath}_{i} \mapsto \mathtt{l}(\ov{\jmath}_i)
%:=(0, \ldots, \underbrace{-1}_{i-th},0,\ldots,0)=-e_{i}\,,\quad i=1,\ldots, \nu\,,
%\end{equation}
%where $e_i$ is the $i$-th vector of the canonical basis of $\mathbb{Z}^{\nu}$,
%and the function $G(\theta)$, $\theta\in \mathbb{T}^{\nu}$ as
%\begin{equation}\label{GTTT}
%G(\theta):=\sum_{\ell\in \mathbb{Z}}\hat{G}(\ell)e^{\ii\cdot\theta}\,,\qquad \hat{G}(\ell):=\left\{
%\begin{aligned}
%&u_{\ov{\jmath}_i}(0)\,,\;\;\; {\rm if}\;\;\; \ell=\mathtt{l}(\ov{\jmath}_i)\,, \; i=1,\ldots,\nu\,,\\
%&0\,, \qquad {\rm otherwise}\,.
%\end{aligned}
%\right.
%\end{equation}
%Then, recalling \eqref{LinearFreqDP}, \eqref{velocityvec} and \eqref{LTTT},
%it is easy to check that \eqref{UTTT} holds with $G$ as in \eqref{GTTT}.
We construct quasi-periodic traveling waves solutions of \eqref{eq:113} 
which are ``close'' to the linear ones, namely
%Our solutions \eqref{travellingWaves} will be close to these ones, meaning that 
they will be of the form
\begin{equation}\label{SoluzioneEsplicitaWW}
\begin{aligned}
\eta(t,x)&=\sum_{j\in S} \sqrt{2 \zeta_j} |j|^{1/4} \cos(\omega(j) t-j x)+o(\sqrt{|\zeta|})\,,\\
\psi(t,x)&=-\sum_{j\in S}
\sqrt{2 \zeta_j} |j|^{1/4} \sin(\omega(j) t-j x)+o(\sqrt{|\zeta|})\,, \\
&\omega=\bar\omega +O(\lvert \zeta \rvert) \,, 
%a_{j}:=\frac{|j|^{\frac{1}{4}}}{\sqrt{2}}\big(\sqrt{\zeta_j}+\sqrt{\zeta_{-j}}\big)\,,
%\quad
%b_{j}:=\frac{|j|^{-\frac{1}{4}}}{\ii \sqrt{2}}\big(\sqrt{\zeta_j}-\sqrt{\zeta_{-j}}\big)
%\ov{a_{-j}}=a_j\,,\;\; \ov{b_{-j}}=b_j\,,\;\;\;
%a_j, b_j\sim O(\sqrt{|\zeta|})
\end{aligned}
\end{equation}
where $o(\sqrt{\lvert \zeta \rvert})$
is meant  in the 
 $H^{s}$-topology with $s$ large. The vectors $(\sqrt{\zeta_j})_{j\in S}\in \mathbb{R}^{\nu}$ represent the \emph{amplitudes} of the approximate solutions from which we have the bifurcation.\\
 Due to the absence of physical (or external) parameters in the system \eqref{eq:113}, we shall use these unperturbed amplitudes to modulate the frequencies of the expected solutions. It turns out that such strategy is not doable for any choice of the tangental sites $\overline{\jmath}_i$, but we will prove that \emph{generically} it is. When we refer to a generic choice of the tangential sites we mean that the $\overline{\jmath}_i$'s are chosen such that the vector $\left( \overline{\jmath}_1, \dots, \overline{\jmath}_{\nu} \right)$ is not a zero of a certain non-trivial poynomial $\mathbb{C}^{\nu}\to\mathbb{C}$. We remark that such choice is equivalent to the choice of the velocity vector $\mathtt{v}$ in \eqref{velocityvec}.
 
 \smallskip
 
\noindent \emph{Comment on assumption \eqref{TangentialSitesDPbis}.}
{ The genericity arguments are used to impose non-resonance conditions between the tangential linear frequencies of oscillations \eqref{LinearFreqDP}.
 %which arise in normal form
 The assumption \eqref{TangentialSitesDPbis} guarantees the simplicity of the tangential linear eigenvalues, excluding some delicate non-trivial resonant interactions. 
  %In particular this fact shall be used in the analysis of the linearized operator whenever genericity arguments cannot be applied just by using the momentum conservation. 
In section \ref{sec:generalStrategy}-$(iv)$ we explain where this assumption comes in handy in the analysis of the linearized operator.\\
The \eqref{TangentialSitesDPbis} simplifies also the computation of the first order corrections of the frequency of the expected quasi-periodic solutions, which are important to impose non-degeneracy conditions. See also section \ref{sec:generalStrategy}-$(i)$.
 }

\vspace{0.5em}

Denoting by $B(0,\varrho)$ 
the ball centered at the origin of $\mathbb{R}^{\nu}$ 
of radius $\varrho>0$, 
our result can be stated as follows.
\begin{Teo}{\bf (Quasi-periodic traveling  gravity waves).}\label{thm:main}
Let $\nu\geq 1$.
For a generic
choice of the velocity vector $\mathtt{v}$ as in \eqref{velocityvec} there exist $s\gg1$,  $0<\varrho\ll1$     
and a positive measure  
 Cantor-like set $\mathfrak{A}\subseteq B(0,\varrho)$
 such that
the following holds. 
For any $\zeta\in \mathfrak{A}$,
the equation \eqref{eq:113} possesses a small amplitude 
quasi-periodic solution $(\eta, \psi)(t,x; \zeta)=U(\omega t,x; \zeta)$ of the form \eqref{SoluzioneEsplicitaWW} which is a traveling wave with velocity vector $\mathtt{v}$, 
$U(\vphi,x)\in H^{s}(\mathbb{T}^{\nu+1},\mathbb{R}^{2})$ and 
$\omega:=\omega(\zeta)\in\mathbb{R}^{\nu}$ is a diophantine frequency vector. For  $0<\e\le \sqrt{\varrho}$, the set  $\mathfrak A$ has asymptotically full relative measure in $[\e^2,2\e^2]^\nu$.
Moreover these solutions are \textbf{linearly stable}.
\end{Teo}
%
%Theorem \ref{thm:main}
%is formulated in the typical style of results on reducible KAM tori for PDEs.

Now we discuss the main issues and the novelties of the paper.

\smallskip
\noindent
$\bullet$  The general form of the linear frequencies of oscillations for the water waves equations is the following 
\[
\sqrt{| j | \tanh(\mathtt{h} |j|) (g +\kappa j^{2})},
\]
where $\mathtt{h}$ and $\kappa$ are respectively the depth and the capillarity of the fluid. If $\mathtt{h}<\infty$ or $\kappa\neq 0$ such parameters may be used to impose non-resonance conditions (see for instance \cite{BBHM}, \cite{BM1}, \cite{PloTol1}). In our case $\mathtt{h}=\infty$ and $\kappa= 0$, thus the linear frequencies of oscillations are
$
\sqrt{g \,|j|}
$
and the elements of the infinite dimensional space $\mbox{span}\{ e^{\mathrm{i} |n| t}\,e^{\mathrm{i} n^2 x} : n\in\mathbb{Z} \}$ are \emph{periodic} solutions of the linearized problem at the origin (completely resonant case).\\
The physical parameter $g$ clearly does not modulate the frequencies, hence if we look for quasi-periodic solutions
 we need to extract parameters directly from the nonlinearities of the equation. We do that by means of \emph{Birkhoff normal form} (BNF) techniques (see section \ref{sec:integrability} for a detailed discussion). In this way the bifurcation parameters are essentially the ``initial data'' or the amplitudes
of an appropriate approximate solution (see \eqref{SoluzioneEsplicitaWW}) from which the bifurcation occurs. 
The choice of the Fourier support of such approximate solution 
plays a fundamental role in proving some non-degeneracy conditions.
Roughly speaking, both the amplitudes $\sqrt{\zeta_{\overline{\jmath}_i}}$ and the tangential sites $\overline{\jmath}_i$
will be ``parameters'' of our problem.

\smallskip
\noindent
$\bullet$ The frequency of the expected quasi-periodic solutions are close to resonant vectors, then the diophantine constant $\gamma$ appearing in non-resonance conditions, see for instance \eqref{ZEROMEL}, has small size as the amplitudes. This clearly produces difficulties in the application of perturbative methods. 

\smallskip
\noindent
$\bullet$
In performing BNF procedures we shall deal with resonances among linear frequencies.
It is known that the pure gravity case in infinite depth 
has no $3$-waves resonant interactions. 
On the other hand, there are many \emph{non-trivial} $4$-wave
interactions, called 
\emph{Benjamin-Feir} resonances (see \eqref{straBFR}).
 We then exploit
a fundamental property of the pure gravity waver waves Hamiltonian 
\eqref{Hamiltonian} in infinite depth: 
the formal integrability, up to order four, of the 
Birkhoff normal form.
This has been proved in \cite{Zak2}, \cite{CW}, \cite{CS}
by showing explicit key algebraic cancellations 
occurring for the coefficients of the Hamiltonian.

\smallskip
\noindent
$\bullet$ In order to show that the ``initial data'' of the expected 
solutions tune in an efficient way the frequencies we 
shall find the explicit expression of the first order 
corrections of the tangential frequencies and of the spectrum 
of the linearized operator in the normal directions. 
 To do that we use an identification argument of normal 
 forms based on the presence of approximate constants 
 of motion (see section \ref{sec:3.33} 
 and Proposition \ref{identificoBNF}).  
Actually, even after this procedure, the Hamiltonian is still partially degenerate. Indeed it turns out that there is 
 a finite number of eigenvalues which are still in resonance.
 %is corrected in a way that make .....}
% It turns out that a finite number of eigenvalues 
% are corrected by terms which makes difficult
% the proof of the ``non degeneracy''.
%We get rid of these finite rank 
%corrections 
We overcome this difficulty 
by passing to suitable rotating coordinates, see 
Lemma \ref{phaseshift}.

\smallskip
\noindent
$\bullet$ 
%As already explained the water waves system \eqref{eq:113}
%admits several symmetries. 
We exhibit the existence of a wide class of traveling quasi-periodic solutions with no parity restrictions in time and space by using the Hamiltonian structure and the $x$-translation invariance of \eqref{eq:113}. It is well known that 
the water waves system \eqref{eq:113} 
exhibits additional  symmetries.
For instance the vector field $X_{H}$
%$X_{H}:=\vect{\nabla_{\eta}H}{-\nabla_{\psi}H}$
in \eqref{HS} is \begin{itemize}
%\vspace{0.5em}
%\noindent
\item[(i)]  \emph{reversible} with respect to  the involution
 \be\label{involutionReal}
S : \left(\begin{matrix} 
\eta(x) \\ \psi(x)
\end{matrix}\right)\;  \mapsto \;
 \left(\begin{matrix} 
\eta(-x) \\ -\psi(-x)
\end{matrix}\right)
\ee
i.e. it satisfies $X_{H}\circ S=-S\circ X_{H}$;
%\begin{equation}\label{involution2Real}
%X_{H}\circ S=-S\circ X_{H}\,;
%\end{equation}
\item[(ii)]
 \emph{even-to-even}, 
i.e. maps $x$-even functions into $x$-even functions. 
%In other words it preserves the subspace 
%\begin{equation}\label{evenFuncReal}
%\Big\{\eta(x)=\eta(-x) \,, \;\; \psi(x)=\psi(-x)\Big\}\,.
%\end{equation}
\end{itemize}

{
In several papers (for instance \cite{BBHM}, \cite{BM1}, \cite{IoossPloTol1}, \cite{PloTol1}) such symmetries are adopted to remove degeneracies due to translation invariance in $x$ and $t$. In the aforementioned works the authors look for solutions which are \\
- standing, namely even in $x$
\[
(\eta, \psi) (t, x)=(\eta, \psi) (t, -x),
\]
- reversible, namely
\begin{equation}\label{sottocaso2}
(\eta, \psi)(t, x)=(\eta, -\psi)(-t, -x)=S(\eta, \psi)(-t, x).
\end{equation}
We observe that combining these properties, the solutions have to satisfy the following parity conditions
\[
\eta\,\, \mbox{is even in $x$ and even in $t$}, \qquad \psi \,\,\mbox{is odd in $t$ and even in $x$}.
\]
}
Since we do not look for solutions in the subspace of reversible functions,
the existence of a solution $U(t, x)$ of \eqref{eq:113} implies the existence of a possibly different solution $S U(-t, x)$.

\smallskip
\noindent
$\bullet$ The vector field in \eqref{eq:113} is a \emph{singular perturbation}
of the linearized  system at the origin \eqref{eq:113linear}
since
the nonlinearity
%, in the complex variable in \eqref{complexExample}, 
%the field
%\eqref{eq:113}
contains derivatives of the first order, while \eqref{eq:113CWW}
contains only derivatives of order $1/2$.
We remark that this is not that case 
when the capillarity $\kappa\neq 0$.

\smallskip
\noindent
$\bullet$ The dispersion law \eqref{dispersionLaw} is \emph{sub-linear}.
This is a major difference in developing KAM theory 
for equations with \emph{super-linear} dispersion law, such as in the gravity-capillary case. 
Indeed
% the linear frequencies are
%$\sqrt{j(g+\kappa j^{2})}\sim |j|^{\frac{3}{2}}$
%where $\kappa>0$ is surface tension parameter. 
%In our case,
 a weaker dispersion law
implies \emph{bad} separation properties of the eigenvalues. The main issue concerns the verification of non-resonance conditions between the tangential frequencies and the differences of the normal ones, called ``second order Melnikov conditions''. The bad separation properties force to impose very weak conditions. In this paper we adopt a different strategy with respect to \cite{BBHM} where the same problem occurs.\\
Our approach is well adapted to the completely resonant case since it allows to reduce the number of steps of nonlinear bifurcation.
It is based on the conservation of momentum, the algebraic structure of the equations and a careful analysis of the resonant regimes. We refer to section \ref{sec:generalStrategy} for a detailed discussion.

\smallskip
\noindent
$\bullet$ The stability result is the same given in Theorem $1.2$ in \cite{BBHM}.
This is of course an interesting dynamical information itself. We also remark that
it is a consequence of an important ingredient of the proof, a reducibility argument 
of the linearized equation at the quasi-periodic solution.\\
By a linearly stable $U(\omega t, x)$ we mean that the linearized operator at the embedded torus $U(\vphi,x)$ has purely imaginary spectrum. In particular
we are able to provide a set of coordinates in which the linear problem is diagonal in the directions normal to the torus. 
 As a consequence the Cauchy problem of the linearized equation is stable, 
 i.e. the Sobolev norms are uniformly bounded in $t$.

\vspace{0.5em}

\noindent\textbf{Literature.}
We  present some results on the water waves systems.

\vspace{0.3em}
\noindent
{\emph{Euclidean case.}}
In the Euclidean case, i.e. when $x\in \mathbb{R}^{d}$, 
the problem of local/global well-posedness has been 
addressed by several authors. Without trying to be exhaustive 
we mention  Coutand-Shkoller \cite{Cou},  Lindblad \cite{Lind1}, Lannes \cite{Lan1} 
Alazard-Burq-Zuily \cite{ABZduke, ABZ1} for local well-posedness results.
We 
 refer to  \cite{IP2}
for a survey on the subject (and reference therein).
In the Euclidean case it is also possible to construct global in time solution 
by exploiting the dispersive effect of the linearized problem.
We quote Germain-Masmoudi-Shatah \cite{GMS}, 
Wu \cite{Wu2}, Ionescu-Pusateri \cite{IP}, Alazard-Delort \cite{ADel1, ADel2},
 Ifrim-Tataru \cite{IFRT}.
 See also \cite{Wu}, \cite{IP3}, \cite{ifrTat}.

We now briefly discuss some known results about the 
the existence of \emph{special} solutions for the 
water waves equations.
% under periodic boundary conditions. 
%We mainly focus on results regarding the existence of
%``special'' solutions such has periodic and quasi-periodic ones.

\vspace{0.3em}
\noindent
{\emph{Traveling waves.}}
%The study of special solutions for the water waves equations has been carried out since the $19$-th century. 
Early results about \emph{traveling} waves date back to $19$-th century. 
In \cite{stokes} Stokes provided  the first nonlinear analysis of the two 
dimensional {traveling} gravity waves, 
by computing a nonlinear approximation of the flow. The rigorous
bifurcation of small amplitude bi-dimensional traveling gravity water waves 
solutions 
has been obtained by Levi-Civita \cite{levi}, and 
Struik \cite{struik}.
The three-dimensional case has been successfully approached more recently.
We quote the paper  \cite{CN1} (and reference therein)
by  Craig-Nicholls where the
 existence of
traveling wave solutions
 is proved
in the gravity-capillary case with space periodic boundary conditions.
%We remark that
%periodic traveling waves are steady 
%in a frame moving with the velocity of the wave. 
%In particular the result is obtained 
%by using the Hamiltonian formulation of the water waves problem given by Zakharov
%in \cite{Zak1} combined with a variational method.
In absence of capillarity the existence of periodic traveling waves is 
a small divisor problem. 
This case has been treated by 
 Iooss and Plotnikov in \cite{IoossPlo1, IoossPlo2}.

\vspace{0.3em}

\noindent
{\emph{Standing  waves.}} The time periodic  \emph{standing waves}
%are ``truly'' time dependent solutions, namely they 
are not stationary 
with respect to a moving reference frame.
%We also mention some results \emph{standing waves} which are 
%``truly'' time dependent solutions, i.e. 
%not stationary w.r.t. a moving reference frame.  
We quote Plotinkov and Toland \cite{PloTol1} for the gravity case in finite depth,  
Iooss-Plotnikov-Toland  \cite{IoossPloTol1} 
and Ioss-Plotnikov \cite{IossPlo111} in infinite depth.
In \cite{AB1} Alazard-Baldi show the existence of periodic 
gravity-capillary standing waves.\\
We remark that the above quoted  results concern \emph{periodic} in time solutions.

%The paper \cite{IossPlo111} deals exactly with the problem \eqref{eq:113}
%which is completely resonant requires a \emph{nonlinear} bifurcation theory.
%This is an important similarity with our case, but, since we look for time \emph{quasi}-periodic 
%solutions, the construction of our first nonlinear approximate
%solution is based on a different argument with respect to \cite{IossPlo111}. 

\vspace{0.3em}
\noindent
{\emph{Time quasi-periodic waves.}}
The existence of time quasi-periodic solutions for PDEs 
has been widely studied from the $'80$s. We mention the pioneering works
by Kuksin \cite{ImaginarySpectrum}, Craig-Wayne \cite{CWay},
Bourgain \cite{JeanB},
Kuksin-P{\"o}schel \cite{KP}.
The KAM theory for PDEs with \emph{unbounded} perturbation is quite more
recent, see for instance \cite{Ku2,LY,BBiP2}. 
The problem of dealing with \emph{quasi-linear} equations
has been tackled first by
  Baldi-Berti-Montalto in 
\cite{Airy, KdVAut} for the KdV equation. Their method has been
extended to deal with 
 many models of interest in hydrodynamics,
 such as the Schr\"odinger equation \cite{FP1},
 generalized KdV \cite{Giuliani},
 Kirchoff \cite{Mon1}, Degasperis-Procesi equation \cite{FGP1}.
 Regarding the water waves problem we quote
 the novel results of existence of 
%The existence of \emph{quasi}-periodic in time
{standing} waves %has been proved in
 by Berti-Montalto  \cite{BM1} for the gravity-capillary case 
 and by Baldi-Berti-Haus-Montalto
\cite{BBHM} for the gravity case in finite depth.
Eventually we mention the recent work  \cite{BFMasp} and the preprint { \cite{BFMasp2}
by Berti-Franzoi-Maspero} on the  existence  of quasi-periodic
traveling waves for fluids with constant vorticity, and
\cite{BaMon} by Baldi-Montalto regarding time quasi-periodic 
solutions for the incompressible Euler equation on $\mathbb{T}^{3}$.
%In these papers 
%the small amplitude solutions bifurcate 
%from the linearized problem 
%at the origin and the linear frequencies 
%of oscillation depend on physical parameters
%such as the surface tension (\cite{BM1}) or the depth of the ocean (\cite{BBHM}).
%In this cases the linear approximation is sufficient to enter the perturbative regime 
%required by the Nash-Moser scheme.
%These papers are based
%on a recent and innovative strategy proposed by 
%Berti, Baldi, Montalto 
%in  
%\cite{Airy, KdVAut} to deal with quasi-linear 
%and fully nonlinear PDEs on the circle.
%Such approach has been 
% applied to many other equations of interest in hydrodynamics,
% such as Schr\"odinger equation, \cite{FP1},
% generalized KdV, \cite{Giuliani},
% Kirchoff \cite{Mon1}, Degasperis-Procesi equation, \cite{FGP1}.

\vspace{0.3em}
\noindent
{\emph{Birkhoff normal form and long time stability on tori.}}
The normal form theory plays a fundamental role in our approach. 
In particular we refer to the pioneering works
\cite{Zak2, DLzak, CW} (see also \cite{CS}) where
is proved the integrability (at the formal level) 
of the Birkhoff normal form
at order four of the pure gravity water waves in infinite depth. 
The paper quoted above provide just a ``formal'' result in the sense 
that
no actual relation can be established between the flow generated by $H$ in 
\eqref{Hamiltonian} and the one generated by the Hamiltonian in Birkhoff normal form.

Rigorous long time existence results for water waves, 
based on a normal form approach, are quite more recent. 
%Regarding the $1$-d gravity-capillary water waves equations,
We mention 
\cite{BD} by Berti-Delort, where
the authors  
provide 
the existence of solutions
 evolving from $\e$-small data
up to time of order $\e^{-M}$ for the 
 $2$-d gravity-capillary water waves equations.
Ionescu-Pusateri in \cite{IP6} prove a life span of order $\e^{-5/3+}$
for the gravity-capillary waves in $3$-d.
These results hold 
for almost all values of the gravity  and surface tension  parameters.

Without the use of any external parameter, we mention 
the paper by Berti-Feola-Pusateri \cite{BFP} 
on the pure gravity case \eqref{eq:113} (see also \cite{BFPnote}).
Using 
a novel identification of normal form argument introduced in \cite{FGP1}, 
the authors proved  
an $\e^{-3}$ existence result providing 
a rigorous proof of the Zakharov-Dyachenko
conjecture (see \cite{Zak2}, \cite{DLzak}).
With a similar idea Berti-Feola-Franzoi proved in \cite{BFF} 
a quadratic life span
for gravity-capillary water waves for any values of   
the gravity  and surface tension.

\subsection{Scheme of the proof}\label{sec:generalStrategy}
  Here we 
 summarize the steps of the proof of Theorem \ref{thm:main} highlighting the key ingredients.
 For further details see \cite{FGnote}.

\smallskip
\noindent
$(i)$ \emph{Nash-Moser theorem of hypothetical conjugation.}
The quasi-periodic solutions are found as zeros of the 
nonlinear functional equation \eqref{EquazioneFunzionaleDP}, \eqref{NonlinearFunctionalDP}. 
We apply a Nash-Moser scheme 
(see Theorem \ref{NashMoserDP}) 
that provides the zeros of such functional 
as limit of a sequence of approximate solutions 
convergent in some Sobolev space.\\
The main issues concern the invertibility of the 
linearized operator in a neighborhood of the 
equilibrium and the search for a good 
approximate solution that initializes the scheme. \\
In sections \ref{sec:weak} and \ref{sec:5} we construct 
the first nonlinear approximate solution from 
which the expected solutions bifurcate. From a geometrical point of view, we determine the "unperturbed" embedded torus that we want to continue to a torus which is invariant for the full Hamiltonian system \eqref{Hamiltonian}.\\
In a suitable set of coordinates the linear dynamics 
of the tangential (to the torus) variables is decoupled by 
the dynamics of the normal ones (we follow the Berti-Bolle method \cite{BertiBolle}). 
In section \ref{sezione6DP} we solve the 
equations corresponding to the tangential 
part and in sections 
\ref{sec:Linopnormal}-\ref{sez:KAMredu} 
we deal with the linearized operator in the normal directions.

\smallskip
\noindent
$(ii)$ \emph{Weak Birkhoff normal form.}
In order to find the first nonlinear approximation 
we implement a Birkhoff normal form method, 
more precisely we construct
a map, which is close to the identity up to a 
finite rank operator, such that 
the Hamiltonian $H$ in \eqref{Hamiltonian}
assumes the form (see \eqref{piotta}) $H_{\rm Birk}+R$ 
with $R$ a small remainder
and 
\begin{itemize}
\item[1.]  there exists a finite dimensional subspace
$U_S$  invariant for $H_{\rm Birk}$; 
\item[2.] the Hamiltonian restricted to $U_S$ is integrable and non-degenerate
 in the sense that
the ``frequency-to-amplitude'' map (see \eqref{FreqAmplMapDP}) 
is invertible.
\end{itemize}
We refer to this procedure as a \emph{weak} Birkhoff normal form, see the discussion in section \ref{sec:integrability}.\\
At the $n$-th step of the BNF procedure one has to deal
with $n$-waves resonant interactions, 
which are zeros of the algebraic equations
\begin{equation}\label{rambo3}
\sum_{i=1}^{n}\s_{i} j_i=0\,,\qquad 
\sum_{i=1}^{n}\s_{i}\sqrt{|j_{i}|}=0\,,\qquad \s_{i}=\pm\,,\;\;\;
i=1,\ldots,n\,.
\end{equation}
We say that a $n$-tuple $(j_1,\ldots,j_{n})$  is a \emph{trivial} 
resonance
if $n$ is even, 
$\s_{i}=-\s_{i+1}$, $i=1,\ldots,n-1$ (up to permutations), and the $n$-tuple has the form
%$\s_1=\ldots=\s_{n/2}=-\s_{n}=\ldots=-\s_{n/2+1}=1$
%and the $n$-tupla 
$(j,j,k,k,\ldots)$. % solves \eqref{rambo3}.
It is easy to note that monomials $u^{\sigma_1}_{j_1}\dots u^{\sigma_n}_{j_n}$, with $u_j^{+}:=u_j, u_j^-:=\overline{u_j}$, supported on 
trivial resonances are \emph{integrable} or \emph{action preserving}, meaning that depend only on the actions $|u_j|^2$.
Unfortunately, even for $n=4$,  there are infinitely 
many non-trivial solutions of \eqref{rambo3}.
For $n=4$ they are called Benjamin-Feir resonances 
(see \eqref{straBFR}). To deal with these resonances we reason as follows.
In  \cite{Zak2}, \cite{CW}, \cite{CS} it has been proved that 
%there are algebraic cancellations, performing a \emph{full} BNF procedure 
%(see section \ref{compaBNFpro}), 
%on
 the coefficients of the normalized Hamiltonian 
 (obtained by a \emph{full} Birkhoff normal form procedure, 
 see section \ref{sec:integrability}) at order four
of the monomials corresponding to the Benjamin-Feir resonances vanish.
In subsections \ref{ApproxCons}, \ref{sec:3.33},
by using suitable algebraic arguments, we actually prove that 
such cancellations of \cite{Zak2}, \cite{CW}, \cite{CS}
occur also performing a \emph{weak} version of BNF which involves only \emph{finitely}
many \emph{tangential sites} (see \eqref{TangentialSitesDP}). 
This is the content of Proposition \ref{LemmaBelloWeak}.
In this way we conclude the \emph{integrability}
of the weak BNF Hamiltonian at degree 4 and we obtain its explicit formula 
\eqref{theoBirH}.
We remark that formula \eqref{theoBirH} is fundamental to prove the non-degeneracy of the frequency shift.  {Without assumption \eqref{TangentialSitesDPbis} this formula would be more complicated and in principle it would be not clear whether the twist condition (see Lemma \ref{Twist1}) is satisfied.}\\
In order to deal with higher order resonances we use the  genericity
argument of section \ref{algresoreso}. {We remark that this argument exploit the conservation of momentum, as it is evident from the proof of Proposition \ref{resonances}.}

\smallskip
\noindent
$(iii)$ \emph{Invertibility 
of the linearized operator.} 
In section \ref{sec:Linopnormal} we compute the linearized operator in the normal directions $\mathcal{L}_{\omega}$ (see \eqref{BoraMaledetta}).\\
The invertibility is obtained by a \emph{reducibility argument} that consists into two main steps:
\begin{itemize}
\item[(a)]\label{stepREGU} A pseudo differential reduction 
in decreasing order of $\mathcal{L}_{\omega}$
which conjugates the linearized operator to 
a pseudo differential  one
with constant coefficients
up to a bounded remainder;

\item[(b)] A reduction of bounded operators and a KAM
 scheme which completes the diagonalization;
 \end{itemize}

 The main new issues are the following:
 \begin{itemize}
 \item[(I)] By the complete resonance of the linear equation \eqref{eq:113CWW} the frequency $\omega$ is close to a resonant vector, then the diophantine constant $\gamma$ in  \eqref{gammaDP} (see also \eqref{ZEROMEL}) is small with the size $\varepsilon$ of the amplitudes. This implies that many terms in $\mathcal{L}_{\omega}$ are not perturbative, in the sense that, roughly speaking, the size of their coefficients divided by $\gamma$ is big ($\varepsilon \gamma^{-1}\gg 1$).
 \item[(II)] We consider the linearization on a quasi-periodic traveling function $U(\varphi, x)$ without any assumption on the parity of $\varphi$ and $x$. Usually such conditions provide some algebraic cancellations in performing steps $(a)$-$(b)$ and reduce the multiplicity of the eigenvalues simplifying the proof.
 \end{itemize}
 
 {The regularization procedure (a) consists in applying several changes of coordinates. Each change of variables makes constant the coefficient of an unbounded pseudo differential operator and it is constructed in two steps. This is because the coefficients of the linearized operator are sum of two parts, one is non-perturbative in size, in the sense of (I), and the other one is just small enough. Then in the first step, that we call \emph{preliminary step}, we treat the non-perturbative part. In the second one we make constant the whole coefficient. In the regularization procedure (a) the pseudo differential structure is fundamental, because it guarantees that the linearized operator is diagonal up to a remainder whose order is getting smaller and smaller at any step. This fact relies on the property that the commutator between two pseudo differential operators gains one derivative (see Lemma \ref{James2}).\\
The procedure (b) completes the reduction of the linearized operator. It is again a transformation method but it deals only with bounded terms. Also in this case the changes of coordinates are constructed in two steps, for the non-perturbative and perturbative terms. The steps that deal with the non-perturbative terms are called  \emph{linear Birkhoff normal form steps}. The combination of the preliminary steps, the Linear BNF and the identification argument of the Proposition \ref{identificoBNF} allows to construct a normal form around the approximately invariant embedded tori. Once we obtained this normal form we can apply a KAM scheme as in a semilinear case.\\
In both procedures (a) and (b) the fact that the coefficients can be made constant is guaranteed by the assumption  \eqref{TangentialSitesDPbis} and certain symmetries. Since we do not assume parity conditions on the approximate solutions on which we linearize, as explained in (II), we need to exploit the Hamiltonian structure and the conservation of momentum. At the linear level the presence of such symmetries comes from the fact that we linearize a Hamiltonian operator on a quasi-periodic traveling function (see section \ref{sec:linstruct} for more details). Along the reducibility procedure we need to ensure that this structure is left invariant by all the changes of coordinates we perform.  }
 
%We deal with $(I)$ by splitting each step of the regularization into $2$ parts. In the first part we deal with non perturbative terms of $\mathcal{L}_{\omega}$ by using algebraic arguments involving Birkhoff resonances. We refer to these steps as \emph{preliminary steps} and \emph{linear Birkhoff normal form steps} when they are performed in the procedures described in items $(a)$ and $(b)$ respectively. As a consequence, together with the identification argument of the Proposition \ref{identificoBNF}, we obtain a normal form around the approximately invariant embedded tori that allows to control the resonances between the tangential and the the normal frequencies.\\
%In the second part we complete the conjugation to constant coefficients of the symbols which at this point are ``small''  in size.
 
% Concerning $(II)$, {the fact that we linearize on a quasi-periodic traveling function guarantees that the Hamiltonian of the linearized operator commutes with momentum. We exploit both the Hamiltonian structure
% and the conservation of momentum to obtain information on the spectrum, more precisely that is simple and purely imaginary. To do that we need to preserve this structure along the reducibility procedure.} 
 
 \smallskip
 
One of the main issues that we deal with by using the conservation of momentum is when we need to reduce to constant coefficient the  transport 
 term 
 \begin{equation}\label{figaro10}
\omega\cdot\pa_{\vphi}+V(\vphi,x)\pa_{x}\,
\end{equation}
 appearing in $\mathcal{L}_{\omega}$. This is the most delicate point of procedure (a) and we describe it below.
 Actually this problem is equivalent to straighten the degenerate 
 vector field
 $
 \omega\cdot\frac{\pa}{\pa_{\vphi}}+V(\vphi,x)\frac{\pa}{\pa{x}}\,
 $
 on the $\nu+1$-dimensional torus.
 By using the fact that we linearized on a quasi-periodic traveling wave (or using momentum conservation)
  we can reduce it to the study of the following \emph{non-degenerate} vector field (on a lower dimensional torus)
  \[
  (\omega-{\bf V}(\Theta)\mathtt{v})\cdot\frac{\pa}{\pa{\Theta}}\,,
  \qquad \Theta\in \mathbb{T}^{\nu}\,,
  \]
  where ${\bf V}(\vphi-\mathtt{v}x)=V(\vphi,x)$. Then we can apply a result of straightening of weakly perturbed constant vector fields on tori given in \cite{FGMP}.  
  
  \smallskip
\noindent
$(iv)$ 
  \emph{Non-resonance conditions.}
  Since the dispersion law is sub-linear, 
 at each step of 
  items (a) and (b) small divisors problems arise. We discuss 
  the non-resonance conditions that we shall require on the frequencies.\\
  {In the preliminary steps described above we encounter the same small divisors appearing in the weak BNF procedure, namely the combinations of the tangential linear frequencies of oscillations
  \begin{equation}\label{smalldivpre}
\sum_{i=1}^n \sigma_i \sqrt{|j_i|}, \qquad j_i\in S, \quad \sigma_i=\pm.
  \end{equation}
  In the weak BNF we use momentum conservation  
  \begin{equation}\label{lucca1}
  \sum_{i=1}^n \sigma_i j_i=0 \qquad j_i\in S, \quad \sigma_i=\pm
  \end{equation}
   to apply genericity arguments ensuring that the above relations vanish just in the trivial case. Indeed if a non-trivial $(\sigma_i, j_i)_{i=1}^n$ makes \eqref{smalldivpre} vanish , then one can choose the tangential sites such that \eqref{lucca1} is not satisfied (see the proof of Proposition \ref{resonances}). \\
    This corresponds to a generic choice because the conditions \eqref{lucca1} provide a system of \textbf{finitely many} polynomial equations. 
  However the same reasoning cannot be applied in the preliminary steps, where the conservation of momentum reads as
  \[
  \sum_{i=1}^n \sigma_i j_i+j-k=0 \qquad \forall j_i\in S, \qquad j, k\in S^c,
  \]
and so it involves infinitely many polynomial equations. In this case we are forced to choose the tangential sites out of the zero set of the following finitely many algebraic functions
\[
\sum_{i=1}^n \sigma_i \sqrt{|j_i|}=0, \qquad j_i\in S, \quad \sigma_i=\pm.
\]
This is certainly a generic choice, but the problem is that the above functions are identically zero if $(\sigma_i, |j_i|)_{i=1}^n$ is a trivial resonance, and not just $(\sigma_i, j_i)_{i=1}^n$. This is the typical issue that arises in a case with multiple eigenvalues. Then by assuming the condition \eqref{TangentialSitesDPbis} on the tangential set we get rid of this problem. 
}

%As second step to invert the linearized operator we shall apply 
%a KAM reducibility scheme to the operator $\mathcal{L}_{6}$ in \eqref{elle6intro}.
%To straighten the vector field above we implement an iterative scheme
%involving small divisors. As a consequence we shall impose
In order to deal with the operator in \eqref{figaro10}
we shall impose the following
``zero order Melnikov conditions'':
\begin{equation}\label{zerointro}
|\big(\omega-\mathfrak{m}_1\mathtt{v}\big)\cdot\ell|
\geq \gamma \langle \ell\rangle^{-\tau}\,, \qquad \forall\, \ell\in \mathbb{Z}^{\nu}\setminus\{0\}
\end{equation}
where 
$\mathfrak{m}_1\approx (2\pi)^{-\nu}\int_{\mathbb{T}^{\nu}} {\bf V}(\Theta)d\Theta $,
$\tau>\nu+1$.

In section \ref{sec:linBNF} we compute the first order corrections
to the eigenvalues of the linearized operator  
in the normal directions. This step is fundamental in order to impose
the non-resonance conditions required for the KAM reducibility scheme of the procedure (b).
Thanks to the identification argument of section \ref{stepsLower}
it turns out that we need to prove
the non-degeneracy of the following map: 
\[
\omega\mapsto  \mathbb{A}\,\zeta(\omega)\cdot \ell + m_1(\omega)(j-k)+c_j(\omega)j
-c_{k}(\omega)k\,,\qquad
\mathtt{v}\cdot\ell+j-k=0\,,\;\;\; \forall\, \ell\in \mathbb{Z}^{\nu}\,, \; j,k\in S^{c}\,,
\]
where $\mathbb{A}$ is the matrix in \eqref{parco1} 
(which only depends on the choice of the tangential sites), 
$m_1$ is in \eqref{constM1}
and $c_j$ is in \eqref{cj}.
In principle the presence of the  $c_{j}$'s produces serious difficulties
in imposing the non-resonance conditions. 
However only a finite number of  such corrections are different from  zero.
% except a finite number.
We show that in an appropriate set of rotating coordinates
those corrections disappear. See Lemma \ref{phaseshift}.
%we can get get rid of these finite
%rank corrections (see Lemma \ref{phaseshift}). 

In the KAM reducibility scheme in the procedure (b)
we shall  impose suitable lower bounds on the functions
\begin{equation}\label{rambo13}
\psi_{\ell, j,k}=\omega\cdot\ell +\mathfrak{m}_1(j-k)
+(1+\mathfrak{m}_{\frac{1}{2}})(\sqrt{|j|}-\sqrt{|k|})+
\mathfrak{m}_0(\sign(j)-\sign(k))+r_{j}-r_{k}
\end{equation}
$\ell\in \mathbb{Z}^{\nu}$, $j,k\in \mathbb{Z}\setminus\{0\}$, with
$\mathfrak{m}_{1/2}$, $\mathfrak{m}_{0}$, $r_{j}$ depending on $\omega$ and satisfying
\[
(\sup_{j}|j|^{-1/2}|r_{j}|+|\mathfrak{m}_{1/2}|+|\mathfrak{m}_0|)\gamma^{-1}\ll 1\,.
\]
These conditions are called ``second order Melnikov conditions'' and are of the form
\begin{equation}\label{secondintro}
|\psi_{\ell,j,k}|\geq \frac{\eta}{\langle \ell\rangle^{\tau}}\,,\quad
\ell\in\mathbb{Z}^{\nu}\,,\;\; j,k\in\mathbb{Z}\setminus\{0\}\,,\quad (\ell,j,k)\neq (0,j,j)\,,
\end{equation}
where $\eta, \tau>0$ are some constants to be fixed.

We prove that the
\eqref{secondintro} hold with
\begin{equation}\label{musmall}
\eta:=\gamma^{3}\ll \gamma\,,
\end{equation}
where $\gamma$ is the constant appearing in \eqref{zerointro}.
Actually we show that the measure of the complementary of the set of  frequencies $\omega$
such that \eqref{secondintro} holds goes to zero as $\e\to 0$.
The crucial problem is the summability in the indexes $\ell,j,k$.
Since in this case the linear 
frequencies grow at infinity only sublinearly,  the differences
 $\sqrt{|j|}-\sqrt{|k|}$ accumulate everywhere in $\R$. This means that for any fixed $\ell$ 
there are \emph{infinitely many} indexes $j,k$ to be taken into account. 
The key idea that we use is the following. By the conservation of momentum
we 
 need to impose the conditions
\eqref{secondintro} only for $(\ell,j,k)$ satisfying
\begin{equation}\label{momentointro}
\mathtt{v}\cdot\ell+j-k=0\,,
\end{equation}
where $\mathtt{v}$ is given in \eqref{velocityvec}.
This allows us to show (see Lemma \ref{delpiero}) that, if $|j|,|k|$
are much larger than $|\ell|$, then
the conditions \eqref{secondintro} are implied by the
\eqref{zerointro}
and then we are left 
 to control the small divisors only for \emph{finitely many} 
indexes $j,k$.\\
{We point out that, despite the multiplicity of the linear eigenvalues $\sqrt{|j|}$, the linearized operator is diagonal in the Fourier basis, and not just block-diagonal. This is usually the case because one cannot impose the \eqref{secondintro} for $\ell=0, j=\pm k$. Thanks to the conservation of momentum, when $k=-j$ in \eqref{rambo13} the relation \eqref{momentointro} implies that
\[
v\cdot \ell=-2 j
\]
and
\[
\psi_{\ell, j,-j}=\omega\cdot\ell +2\,\mathfrak{m}_1\,j
+
2\mathfrak{m}_0\,\sign(j)+r_{j}-r_{-j}=\overline{\omega}\cdot \ell+\e^2 (\mathbb{A}\zeta(\omega)-m_1(\omega) \mathtt{v})\,\cdot \ell+o(\e^2).
\]
By choosing $\omega$ in the set $\mathcal{G}^{(1)}_0$ in \eqref{ZEROMEL}, the first order expansion (the $\e^2$-terms) of $\psi_{\ell, j, -j}$ does not vanish and we are able to impose a lower bound as \eqref{secondintro}.} 

\vspace{0.5em}

We mention that an alternative approach to impose the non-resonance conditions required by the KAM scheme is to prove 
Melnikov conditions which \emph{lose} derivatives 
by setting
\begin{equation}\label{muloss}
\eta:=\frac{\gamma}{\langle j\rangle^{\mathtt{d}}\langle k\rangle^{\mathtt{d}}}\,,
\quad \mathtt{d}> 1\,.
\end{equation}
With this choice  ``many'' frequencies $\omega$ satisfy
the \eqref{secondintro}, but the small divisors creates a true loss 
of space derivatives.

The class of non perturbative terms is enlarged by the choice \eqref{musmall} but such terms are compactly time-Fourier supported ($|\ell|\leq C$ for some constant $C>0$).
It turns out that if at least one between $|j|, |k|$
is large enough
then there are no small divisors (see Lemma \ref{lowerboundDel1}).
Otherwise we choose to impose the second Melnikov conditions 
with $\eta$ as in \eqref{muloss}.
Actually this does not create any loss of 
derivatives since $|j|,|k|$ are taken into a ball with finite radius.

  Another possibility would be to use $\eta$ as in \eqref{muloss}
  in any regimes $\ell,j,k$ as done in \cite{BBHM}.
  As a drawback this choice would require
  several steps of reduction to constant coefficients up to smoothing remainders 
  before applying the KAM scheme of the procedure (b) above. 
  Since   we are dealing with a  resonant case,
  it turns out that
  the accuracy necessary for the bifurcation point 
  depends on the number of reduction steps.
In other words the number of weak Birkhoff normal form steps
in section \ref{sec:weak} increases rapidly as
the number of reduction steps increases.
Choosing to reduce to constant coefficients 
only the pseudo differential operators 
of non-negative orders (up to remainders which are
 only $1/2$-smoothing)  provides the optimal balancing of the two procedures.

\subsection{Comparison with previous KAM results} 
In this section we discuss the main differences between the present work and related KAM results for PDEs.

\smallskip

\noindent{\emph{ KAM results without external or physical parameters.} As it is well-known the implementation of KAM schemes in the PDE context requires the use of parameters to impose (infinitely many) non-resonance conditions.\\
Then it is common to use physical or external parameters to modulate the linear frequencies of oscillation. Excluding the cases of the water waves systems that we discussed before, this can be done artificially by adding convolution or multiplicative potentials, as for instance in nonlinear Schr\"odinger equations equations, or more naturally by considering mass terms, as in the case of Wave/Klein-Gordon equation.\\
However many important physical phenomena are modeled by equations without parameters. In these cases the parameters are extracted directly from the equation by using Birkhoff normal form methods. This approach has been implemented successfully in 
\cite{KP, GeYou1, GeYou2, GeXuYou, PP, FGP1}. 
In all these works the authors exploit also the conservation of momentum. Indeed usually the modulation of the inner parameters is not strong enough to impose non-resonance conditions.\\
As we discussed in section \ref{sec:generalStrategy}, in our case the conservation of momentum is fundamental to deal with 
the multiplicity of the linear eigenvalues and the sublinear dispersion law. 
We remark that in some papers, for instance \cite{KdVAut, FGP1}, 
the preservation of momentum is used just up to some order. In fact in \cite{KdVAut} the authors may deal with non-$x$-translation invariant nonlinearities. However in that case the dispersion law is super-linear (stronger separation property of the eigenvalues) and the linear spectrum is simple. In \cite{FGP1} the dispersion law is linear and this fact is used to prove that many second order Melnikov conditions are implied by first order Melnikov conditions. Thanks to this argument one does not need to fully exploit the conservation of momentum at any order.
}

\vspace{0.5em}
  
\noindent{\emph{Comparison with previous KAM results for fluid dynamics models }.
Here we make a brief comparison between the present paper and the works \cite{KdVAut, BM1, BBHM, FGP1} (mentioning also \cite{Giuliani}, \cite{BFMasp})
concerning results of existence and stability of small amplitude 
quasi-periodic solutions for nonlinear shallow waters 
and water waves systems in spatial periodic domains of dimension one. 
The overall strategy for these KAM results for quasi-linear PDEs, based on Nash-Moser theory, 
has been developed in \cite{BertiBolle} and nowadays it is well-established. However, as usual in the PDEs context, every equation presents different issues.  
We believe that the main ingredients to take into account are
\begin{itemize}
\item The (non-) resonant nature of the equilibrium and the presence of parameters;
\item The presence of non-trivial resonances of order $4$;
\item The use of the (approximate) integrability of the \emph{unperturbed} system;
\item The dispersion law: super-linear, linear and sub-linear cases;
\item The analysis of the linearized operator: invertibility in a whole neighborhood of the equilibrium.
\end{itemize}
The main distinction among the aforementioned papers relies is the fact that the works \cite{KdVAut, FGP1} 
are parameter-free and completely resonant problems, 
while for \cite{BM1, BBHM} the elliptic equilibrium is non-resonant for the set of the chosen parameters. 
We start by commenting the former group of results.\\
 In \cite{KdVAut} the authors consider the KdV equation 
 with a high order perturbation (a smooth nonlinearity of order $4$ at the origin), hence very small close to the origin. 
 The linear frequencies of oscillations are $j^3$, $j\in \Z$ and there are no 
 non-trivial resonances at order four. 
 This fact depends on the particular form of the linear eigenvalues and the conservation of momentum at order four, and it does not rely on the integrability of the KdV equation, 
 which is not exploited to prove this result. 
 The first order corrections of the eigenvalues, which are fundamental to prove a non-degeneracy condition, it is not affected by the presence of the perturbation. 
 This makes easier the nonlinear bifurcation analysis. Following the strategy of \cite{KdVAut}, in \cite{Giuliani} the second author of the present paper deals with quadratic quasi-linear 
 perturbations of KdV \footnote{The equations considered in \cite{Giuliani}  cannot be considered as small perturbations of the integrable KdV equation. This can be seen as an evidence that the integrability does not play a role in \cite{KdVAut}. }. The lower order of the nonlinearity implies a stronger perturbative effect on the linear dynamics and one of the main problems is to keep track of the corrections 
 to the eigenvalues also in the regularization procedure 
 described in section \ref{sec:generalStrategy}-$(iii)$ (\emph{preliminary steps}). All these computations are made without the use of identification arguments. This is doable thanks to  the classical differential structure of the equation.\\
 Similar difficulties are encountered in \cite{FGP1} to prove existence of 
 reducible KAM tori for the Degasperis-Procesi equation and its (high order) 
 Hamiltonian perturbations. 
 In this case the bifurcation analysis is challenging due to the complicated structure of the resonances. Indeed non-trivial resonances are present also at order four, as the Benjamin-Feir resonances for the pure gravity system with infinite depth. 
 The authors fully exploit the integrability of the equation and develop 
 a novel argument of identification of normal forms to provide non-degeneracy conditions. 
 Another fundamental difference with the works on KdV 
 equations is the fact that the dispersion law is asymptotically linear $j \to j+O(1/j)$, 
 instead of super-linear $j\to j^3$. 
 This complicates both the imposition of non-resonance conditions, since there is a weaker separation of the eigenvalues, 
 and the analysis of the linearized operator. Concerning the latter, in the KdV case the linearized operator is purely differential, while in the DP case this has a pseudo differential structure.
% The latter is a pseudo differential Hamiltonian linear operator, 
% while in the KdV case this is a non-constant coefficient differential operator.  
% In the DP case the use of pseudo differential 
% calculus comes in handy to deal with the intricate symplectic structure of the equation.\\
A KAM result for the pure gravity water waves system in infinite depth collects all the difficulties previously mentioned: a weak dispersion law, absence of parameters, 
presence of many non-trivial resonances of lower order, pseudo differential  structure of the linearized operator. 
The main differences with the DP case rely on the sub-linear dispersion law $j\to \sqrt{|j|}$, 
the lack of constants of motion (except momentum) 
and the fact that the linear analysis is a \emph{singular} perturbation problem. In section \ref{sec:integrability} we explain in details how to use the \emph{formal} approximate integrability of the pure gravity system in finite depth and the characterization of the Benjamin-Feir resonances to provide the identification argument of normal forms.\\
Concerning the works \cite{BM1, BBHM}, the dispersion law of the gravity-capillary case in infinite depth and pure gravity with finite depth is respectively super-linear and sub-linear. In both cases the measure estimates for the set of the good parameters (for which the existence of quasi-periodic water waves is ensured) are provided by using degenerate KAM theory. The present work and \cite{BBHM} share the difficulty of the weak separation of the linear eigenvalues, but we deal with this issue in a different way, as it has been explained in section \ref{sec:generalStrategy}-$(iv)$. In \cite{BBHM} the authors impose Melnikov conditions that imply a possible loss of (spatial) derivatives of the KAM transformations. This is compensated by a preliminary long regularization procedure that makes constant coefficients the linearized operator up to very smooth terms. In a resonant case such procedure is not convenient, because the number of steps of the normal form methods that we need to perform increases with the number of steps of the regularization procedure. Hence we do not impose Melnikov conditions with loss of derivatives thanks to a careful analysis of the small divisors (we refer to section \ref{sec:linBNF} for more details). \\ Concerning the role of integrability, we remark that in non-resonant cases the integrability is usually used just at the linear level. Roughly speaking the equation is seen as a small perturbation of the linear part, that is integrable, and the \emph{unperturbed tori} are linear solutions obtained by exciting a finite number of modes.\\
Regarding the analysis of the linearized operator, except for the number of steps of the regularization procedure, an important difference with respect to \cite{BM1, BBHM} is the fact that we look for traveling quasi-periodic waves, instead of standing and reversible. The main novelty is providing that the terms that cannot be eliminated or normalized in the reducibility procedure of the linearized operator are constant coefficients operators. In the standing/reversible case this is guaranteed by parity conditions on the solutions that one look for, which reflect on parity properties of the non-constant coefficients of the linearized operator. In the present work we use the Hamiltonian structure and the momentum conservation to verify that such coefficients have nice properties, as for instance that the averages in time or space of some of them are constant. This clearly requires that all the changes of coordinates that we perform leave invariant those symmetries along the reducibility procedure. In other words, the reducibility of the linearized operator has to be performed within a class of Hamiltonian operators, $x$-translation invariant (to which it belongs the linearized operator at a quasi-periodic traveling function), see section \ref{sec:group}.\\
To conclude, we mention the work \cite{BFMasp}. The authors provide the existence of traveling quasi-periodic water waves for the gravity-capillary case with constant vorticity. Also in this case the solutions are reversible (parity conditions in the time-space variable) and the vorticity is used as a parameter to modulate the linear frequencies of oscillations.
}

\vspace{0.5em}

\noindent
{\emph{Possible extensions to other water waves problems.}
It would be interesting to investigate the existence of quasi-periodic solutions 
for the models considered in \cite{BM1, BBHM, BFMasp} 
keeping \emph{fixed}
the parameters of the problems, respectively the capillarity, the depth of the ocean and the vorticity, and looking for solutions free from parity restrictions. 
Depending on the value of the fixed parameter, these equations present the same issues of a  
resonant case, like the pure gravity system in infinite depth.
Our method could be applied to these cases provided that: 
(i) the normal form is formally integrable at order four, namely 
the coefficients of the Hamiltonian vanishes on non-trivial resonances (if they exist);
(ii) we have some information on non-trivial resonances at order four; 
(iii) the twist condition holds, namely the frequency-amplitude map
is a local diffeomorphism. All these ones are algebraic conditions that have to be checked for each choice of the fixed parameter. At the moment
it is not clear to us whether these conditions are satisfied 
for some (large) set of the physical parameters.\\ Concerning the analysis of the linearized operator: the specific changes of coordinates that reduce the linear operator depend strongly on the considered model. However we believe that the reducibility strategy that we implement in the present work, made of an analysis of perturbative and non-perturbative terms, is quite general. Regarding the parity conditions, one should apply the reducibility argument in a class of linear operators with some algebraic structure dictated by the equation considered. We think that in water waves models is natural to exploit the Hamiltonian structure and the conservation of momentum. In this paper we provide a guideline to take advantage of these symmetries which are present also in the finite depth case and the gravity-capillary system.}

\vspace{0.5em}

\noindent\textbf{Plan of the paper.} 
In section \ref{sec:2} we introduce the functional setting 
and the classes of linear operators that 
we shall consider in the
reducibility argument.

In section \ref{sec:integrability}
we discuss the formal integrability of the pure gravity water waves at order four
and we show how to exploit it 
to prove the normal form identification 
argument  of section \ref{sec:3.33}.

In section \ref{sec:weak}  we apply a 
``weak'' version of a Birkhoff normal form 
procedure 
to the Hamiltonian \eqref{Hamiltonian} (see Proposition \ref{WBNFWW})
and we show that there exists a finite dimensional approximately 
invariant subspace foliated by embedded tori. 
Then, in section \ref{sec:5},
we prove that the dynamics on such subspace is 
 \emph{non-isochronous} 
(see the Hamiltonian system \eqref{HamiltonianSisteminoDP} and 
Lemma \ref{Twist1}).
In this way we find a first good nonlinear approximate solution.
In section \ref{sezioneNonlinearFunct} we state the 
Theorem \ref{IlTeoremaDP}
from which we deduce the main result.
This theorem is proved through a Nash-Moser nonlinear iteration, 
see Theorem \ref{NashMoserDP}. 

In section \ref{sezione6DP} we prove Theorem \ref{TeoApproxInvDP}
that concerns the inversion of the linearized 
operator at an approximate solution.
Sections \ref{sec:Good}-\ref{sez:KAMredu}
are devoted to the verification of the assumptions of Theorem \ref{TeoApproxInvDP},
in particular we prove the invertibility of the linearized operator in the normal
directions, whose properties
are discussed in section \ref{sec:Linopnormal}.
 The key results of  sections \ref{sec:Good}-\ref{sec:block} 
 are  Propositions \ref{lem:compVar}, 
  \ref{lem:mappa2}, \ref{blockTotale}.
  In section \ref{ordiniPositivi}
  we show how to reduce the linearized operator
  to a diagonal one up to a smoothing remainder, see Proposition \ref{riduzSum}.
In section \ref{sez:KAMredu}  we prove the reducibility Theorem 
 \ref{ReducibilityWW} which implies 
the invertibility  result in Proposition \ref{InversionLomegaDP}.

In section \ref{sec:linBNF} we deal with the terms 
which still are not \emph{perturbative} for the KAM scheme of Theorem
\ref{ReducibilityWW}.
By Lemma \ref{stepLBNF2} and Proposition \ref{identificoBNF}
we find the first order corrections of the normal eigenvalues.
Then we split the reduction of the remaining non perturbative terms 
by performing a \emph{low} / \emph{high} modes analysis.

We show that the analysis of high modes 
in section \ref{stepsHigher} 
does not involve small divisors. 
In section \ref{prepreKAMKAM} we perform 
the analysis for low modes.
In section \ref{sezione9WW}
we conclude the proof of the main theorem
by providing the measure estimates of the set 
of ``good'' frequencies.

\vspace{0.9em}
\noindent
{\bf Acknowledgements.}
The authors wish to thank Michela Procesi, Marcel Guardia and Raffaele Scandone for many 
useful discussions and comments. This paper has received funding from 
the European Research Council (ERC) 
under the European Union's Horizon 2020
research and innovation programme under grant agreement 
No 757802.

\section{Functional setting}\label{sec:2}
 
\subsection{Function spaces}
We consider functions $u(\varphi, x)$ defined on $\T^{\nu}\times \T$
with zero average in $x\in \mathbb{T}$. 
Passing to the Fourier representation we write
\begin{equation}\label{realfunctions}
u(\varphi, x)=\frac{1}{\sqrt{2\pi}}\sum_{j\in\mathbb{Z}\setminus\{0\}} 
u_{j}(\varphi)\,e^{\mathrm{i} j x}
=\frac{1}{\sqrt{2\pi}}\sum_{\ell\in\mathbb{Z}^{\nu}, j \in \mathbb{Z}\setminus\{0\}} 
u_{\ell j} \,e^{\mathrm{i}(\ell \cdot \varphi+j x)}\,.
%\quad 
%\overline{u}_j(\varphi)=u_{-j}(\varphi)\,. 
%\quad \overline{u}_{\ell j}=u_{-\ell, -j}\,.
\end{equation}
For $\vphi$-independent Fourier coefficients we also  
use the notation
\begin{equation*}%\label{notaFou}
u_n^+ := u_n:=\hat{u}(n) \qquad  {\rm and} \qquad  
u_n^- := \ov{u_n}:=\ov{\hat{u}(n)}  \, . 
\end{equation*}
Let $s\in\mathbb{R}$. We define the scale of Sobolev spaces 
\begin{equation}\label{space} 
H^{s}:=\Big\{ u(\varphi, x)\in L^{2}(\T^{\nu+1}; \mathbb{C}) : 
\lVert u \rVert_s^2:=\sum_{\ell\in\mathbb{Z}^{\nu}, j\in\mathbb{Z}\setminus\{0\}} \lvert u_{\ell j} 
\rvert^2 \langle\ell, j \rangle^{2 s}<\infty \Big\}\,,
\end{equation}
where $\langle \ell, j \rangle:=\max\{ 1, \lvert \ell \rvert, \lvert j \rvert\}$, 
$\lvert \ell \rvert:=\sum_{i=1}^{\nu} \lvert \ell_i \rvert$. \\
In the following we shall consider the scale of Sobolev 
spaces $(H^s)_{s=s_0}^{\mathcal{S}}$ where
\begin{equation}\label{scala}
\mathcal{S}\gg s_0 :=[\nu/2]+2.
\end{equation}
We remark that $H^{s_0}$ is continuously embedded in 
$L^{\infty}(\mathbb{T}^{\nu+1}; \mathbb{C})$.

\smallskip
\noindent
{\bf Lipschitz norm.} Fix $\nu\in \mathbb{N}\setminus\{0\}$ 
and let $\calO$ be a compact subset of $\mathbb{R}^{\nu}$. 
For a function $u\colon \calO\to E$, where $(E, \lVert \cdot \rVert_E)$ is a Banach space, we define the sup-norm and the lip-seminorm of $u$ as
\begin{equation*}%\label{suplip}
\begin{aligned}
&\lVert u \rVert_E^{sup}:=\lVert u \rVert_{E}^{sup, \calO}
:=\sup_{\omega\in\calO} \lVert u(\omega) \rVert_E,\qquad \lVert u \rVert_{E}^{lip}
:=\lVert u \rVert_{E}^{lip, \calO}
:=\sup_{\substack{\omega_1, \omega_2\in \calO,\\ \omega_1\neq \omega_2}} 
\frac{\lVert u(\omega_1)-u(\omega_2)\rVert_E}{\lvert \omega_1-\omega_2\rvert}\,.
\end{aligned}
\end{equation*}
If $E$ is finite dimensional,  for any $\gamma>0$ we introduce the 
weighted  Lipschitz norm
\begin{equation*}%\label{tazzone}
\lVert u \rVert_E^{\g, \calO}
:=\lVert u \rVert_E^{sup, \calO}+\gamma \lVert u \rVert_{E}^{lip, \calO}\,.
\end{equation*}
If $E$ is a scale of Banach spaces, say $E=H^{s}$, for $\gamma>0$ we introduce the 
weighted  Lipschitz norm
\begin{equation}\label{tazza10}
\lVert u \rVert_s^{\g, \calO}:=\lVert u \rVert_s^{sup, \calO}
+\gamma \lVert u \rVert_{s-1}^{lip, \calO}\,, \quad \forall s\geq s_0
\end{equation}
where we denoted by $[ r ]$ the integer part of $r\in\R$.

\subsection{Linear operators}
We introduce general classes of linear operators 
which will be used in the paper
following \cite{FGP1}.
We consider $\vphi$-dependent families of 
operators $A : \mathbb{T}^{\nu}\mapsto 
\mathcal{L}(L^{2}(\mathbb{T};\mathbb{R}))$ (or 
$\mathcal{L}(L^{2}(\mathbb{T};\mathbb{C}))$
$\vphi\mapsto A(\vphi)$ acting on functions 
$u(x)\in L^{2}(\mathbb{T};\mathbb{R})$
(or $L^{2}(\mathbb{T};\mathbb{C})$). 
We also regard $A$
as an operator acting on functions $u(\vphi,x)\in 
L^{2}(\mathbb{T}^{\nu+1};\mathbb{R}) $ (or 
$L^{2}(\mathbb{T}^{\nu+1};\mathbb{C})$) of the space-time.
In other words the action of 
$A\in \mathcal{L}(L^{2}(\mathbb{T}^{\nu+1}; \mathbb{R}))$ 
is defined by
\[
(Au)(\vphi,x)=(A(\vphi)u(\vphi,\cdot))(x)\,.
\]
We represent a linear operator $\mathcal{Q}$
acting on $L^{2}(\mathbb{T}^{\nu+1}; \mathbb{R}^{2})$
by a matrix
\begin{equation}\label{matrix1}
\mathcal{Q}:=\left(
\begin{matrix} A & B \\ C & D
\end{matrix}
\right)
\end{equation}
where $A$, $B$, $C$, $D$ 
are  operators acting on the scalar valued components
$\eta,\psi\in L^{2}(\mathbb{T}^{\nu+1}; \mathbb{R})$.
The action of $A$ on periodic functions 
as in \eqref{realfunctions}
is given by
\begin{equation}\label{actio}
(Au)(\vphi,x)=\frac{1}{\sqrt{2\pi}}
\sum_{l\in \mathbb{Z}^{\nu}, j\in \mathbb{Z}}e^{\ii l\cdot\vphi}e^{\ii jx}
\Big(\sum_{p\in \mathbb{Z}^{\nu}, k\in \mathbb{Z}}
A_{j}^{k}(l-p)\hat{u}(p,k)
\Big)\,.
\end{equation}
We shall identify the operator $A$ with the 
matrix $(A_{j}^{k}(\ell))_{\ell\in \mathbb{Z}^{\nu},j,k\in \mathbb{Z}}$.
We say that $A$ is a \textit{real} operator if it maps real valued functions in real valued functions. For the matrix coefficients this means that
\[
\overline{A_j^{k}(\ell)}=A_{-j}^{-k}(-\ell)\,.
\]
It will be convenient to work with the the complex
variables $(u,\bar{u})=\mathcal{C}(\eta, \psi)$ introduced in  \eqref{matriciCompVar}.
Thus we study how the linear operators $\mathcal{Q}$ as in \eqref{matrix1}
 transform under the map $\mathcal{C}$.
We have
\begin{equation}\label{forma-complessa}
\begin{aligned}
\mathcal{T}&:=\mathcal{C}\mathcal{Q}\mathcal{C}^{-1}:=(\mathcal{T}_{\s}^{\s'})_{\s,\s'=\pm}:=
\left(
\begin{matrix} 
\mathcal{T}_{+}^{+} & \mathcal{T}_{+}^{-} 
\vspace{0.2em}\\ 
{\mathcal{T}_{-}^{+}} & 
{\mathcal{T}_{-}^{-}}
\end{matrix}
\right)\,,
\quad 
\mathcal{T}_{\s}^{\s'}=\ov{\mathcal{T}_{-\s}^{-\s'}}\,,\;\;\s,\s'=\pm\\
 \mathcal{T}_{+}^{+}&:=\frac{1}{2}\big\{(A+D)-\ii (B-C)\big\}\,,
 \qquad
 \mathcal{T}_{+}^{-}:=\frac{1}{2}\big\{(A-D)+\ii (B+C)\big\}\,,
\end{aligned}
\end{equation}
where the \emph{conjugate} operator $\ov{\mathcal{T}_{\s}^{\s'}}$ is defined as
\begin{equation}\label{conjugate-Op}
\ov{\mathcal{T}_{\s}^{\s'}}[h]:= \ov{{\mathcal{T}_{\s}^{\s'}}[\bar{h}]}\,.
\end{equation}
By setting 
$u^{+}=u$, $u^{-}=\bar{u}$ for any $u\in L^{2}(\mathbb{T}^{\nu+1};\mathbb{C})$,
%Recalling the notation in \eqref{notaFou}
 the action of an operator 
$\mathcal{T}$ of the form \eqref{forma-complessa} is given by
\begin{equation}\label{azioneSobsigma}
(\mathcal{T}\vect{u}{\bar{u}})(\vphi,x)=
\left(
\begin{matrix}
\mathcal{T}_{+}^{+}u^{+}+\mathcal{T}_{+}^{-}u^{-}
\vspace{0.2em}\\
\mathcal{T}_{-}^{+}u^{+}+\mathcal{T}_{-}^{-}u^{-}
\end{matrix}
\right)\stackrel{\eqref{forma-complessa}}{=}
\left(
\begin{matrix}
\mathcal{T}_{+}^{+}u^{+}+\mathcal{T}_{+}^{-}u^{-}
\vspace{0.2em}\\
\ov{\mathcal{T}_{+}^{-}}u^{+}+\ov{\mathcal{T}_{+}^{+}}u^{-}
\end{matrix}
\right)\,.
\end{equation}
By formul\ae\,  \eqref{actio}, \eqref{azioneSobsigma}
we shall identify the operator $\mathcal{T}$ with the 
matrix 
$(\mathcal{T}_{\s,j}^{\s',k}(\ell))_{\ell\in \mathbb{Z}^{\nu},j,k\in \mathbb{Z},\s,\s'=\pm}$, where
\[
\mathcal{T}_{\s,j}^{\s',k}(\ell):=\int_{\mathbb{T}^{\nu+1}} \mathcal{T}_{\s}^{\s'} [e^{\mathrm{i} (\s' k-\s j) x} \,e^{-\mathrm{i} \ell \cdot \varphi}]\,dx\,d\varphi.
\]

\begin{definition}{\bf (Real-to-real).}\label{realtoreal}
We say that an operator  $\mathcal{T}$ acting on $L^2(\T^{\nu+1}; \mathbb{C}^2)$ of the form \eqref{forma-complessa} is 
\emph{real-to-real}. Equivalently an operator is real-to real
if it preserves the following subspace
 \begin{equation*}%\label{calU}
 \mathcal{U}:=\big\{
(u,v)\in L^{2}(\mathbb{T}^{\nu+1};\mathbb{C}^{2})\, : \, v=\ov{u}\,\big\}\,.
 \end{equation*}
\end{definition}

\begin{remark}
Notice that the conjugate of a real  operator as in \eqref{matrix1}
under the map $\Lambda$ in \eqref{CVWW}
has the form \eqref{forma-complessa} and hence it is real-to-real.
\end{remark}

\subsubsection{Tame operators}
%\paragraph{Notation.} 
We use the notation $A\lesssim B$ to denote $A\le C B$ where 
$C$ is a positive constant possibly depending 
on fixed parameters given by the problem. 
We use the notation $A\lesssim_y B$ to denote 
$A\le C(y) B$ if we wish to highlight 
the dependence on the variable $y$ of the constant $C(y)>0$.\\
We use the notation $\calO\Subset \mathbb{R}^{\nu}$ to denote a compact subset of $\mathbb{R}^{\nu}$.

\medskip
\noindent
{\bf Linear Tame operators.} Here we introduce rigorously the spaces and the classes of operators on which we work.

\begin{definition}{\bf ($\s$-Tame operators).}\label{TameConstants}
Let $0<s_0< \mathcal{S}\le \infty$ be two integer numbers and $\s\geq 0$.
We say that 
a linear operator $A$ is $\sigma$-\textit{tame} w.r.t. a 
non-decreasing sequence of positive real numbers
$\{\mathfrak M_A(\s,s)\}_{s=s_0}^\mathcal{S}$ if (recall \eqref{space})
\begin{equation*}%\label{SigmaTame}
\lVert A u \rVert_{s}\le \mathfrak{M}_A(\s,s) 
\lVert u \rVert_{s_0+\sigma}+\mathfrak{M}_A(\s,s_0) \lVert u \rVert_{s+\sigma} 
\qquad u\in H^s\,,
\end{equation*} 
for any $s_0\le s\le \mathcal{S}$. 
We call $\mathfrak{M}_A(\s,s)$  a {\sc tame constant} 
for the operator $A$. When the index $\s$ is not relevant 
we write $\gotM_{A}(\s,s)=\gotM_{A}(s)$. 
For an operator $\mathcal{T}$ acting on 
$L^2(\T^{\nu+1}; \mathbb{C}^2)$ of the form \eqref{forma-complessa}
we set 
\[
\gotM_{\mathcal{T}}(\s,s):=\max\{\gotM_{\mathcal{T}_{+}^{+}}(\s,s),
\gotM_{\mathcal{T}_{+}^{-}}(\s,s)\}.
\]
\end{definition}
%\comment{$\mathcal{O}$ compatto!}

\begin{definition}{\bf (Lip-$\s$-Tame operators).}\label{LipTameConstants}
Let $\s\geq 0$ and $A=A(\omega)$ be a linear operator defined for 
$\omega\in \calO\Subset \mathbb{R}^{\nu}$.
Let us define
\begin{equation}\label{defDELTAomega}
\Delta_{\omega,\omega'}A:=\frac{A(\omega)-A(\omega')}{|\omega-\omega'|}\,, \quad \omega,\omega'\in \calO\,,\,\,\omega\neq \omega'.
\end{equation}
Then $A$ is \emph{Lip-$\s$-tame} w.r.t. a 
non-decreasing sequence $\{\mathfrak{M}^{\gamma}_A(\s,s)\}_{s=s_0}^\mathcal{S}$ if
 the following estimate holds
\begin{equation*}%\label{lipTAME}
\sup_{\omega\in \calO}\|Au\|_{s}\,,\,\,
\g\sup_{\omega\neq\omega'}\|(\Delta_{\omega,\omega'}A)\|_{s-1}
\leq _{s}\gotM^{\g}_{A}(\s,s)\|u\|_{s_0+\s}+
\gotM^{\g}_{A}(\s,s)\|u\|_{s+\s}, \quad u\in H^{s}\,.
\end{equation*}
We call $\mathfrak{M}^{\gamma}_A(\s,s)$ a {\sc Lip-tame constant} of the operator $A$. 
When the index $\s$ is not relevant 
we write $\gotM^{\gamma}_{A}(\s,s)=\gotM^{\gamma}_{A}(s)$. 
\end{definition}

\noindent 
{\bf Modulo-tame operators and majorant norms.}
We now introduce a class of operators which are tame respect to a norm stronger than the operator norm.

\begin{definition}{\bf (Majorant operator).}\label{opeMagg}
Let $s\in\mathbb{R}$ and $u\in H^s$, we define the majorant function
$$\underline{u} (\varphi, x):=(2\pi)^{-1/2}\sum_{\ell\in\mathbb{Z}^{\nu}, j\in\mathbb{Z}\setminus\{0\}} \lvert u_{\ell j} \rvert e^{\mathrm{i}(\ell\cdot \varphi+j x)}.$$
Note that $\lVert u \rVert_s=\lVert \underline{u} \rVert_s$.
Let $A\in\mathcal{L}(H^s)$ and recall its matrix representation \eqref{actio}.
We define the majorant matrix $\underline{A}$ as the matrix with entries
\[
\big(\underline{A}\big)_j^{k}(\ell):=\lvert (A)_{j}^{k}(\ell) \rvert 
\qquad j, k\in\mathbb{Z}\setminus\{0\},\,\,\ell\in\mathbb{Z}^{\nu}\,.
\]
\end{definition}

\noindent
We consider the majorant operator norms 
\begin{equation*}%\label{majorantnorm}
\|\underline{A}\|_{\mathcal L(H^s)}:= 
\sup_{\lVert u \rVert_s\le 1} \lVert\underline{A}u \rVert_{s}\,.
\end{equation*}
We have a partial ordering relation in the set of the infinite dimensional matrices, i.e.
if
\begin{equation}\label{partialorder}
A \preceq B \Leftrightarrow |A_{j}^{k}(\ell)|\leq |B_{j}^{k}(\ell)|\;\;\forall j, k,\ell\; 
\Rightarrow \lVert \underline{A} \rVert_{\mathcal L(H^s)}\le 
\lVert \underline{B} \rVert_{\mathcal L(H^s)}\,, 
\quad \lVert  {A}u \rVert_s\le 
\lVert  \underline{A}\,\underline u \rVert_s \le \lVert\underline{B}\, 
\underline u \rVert_s\,.
\end{equation}
Since we are working on a majorant norm we have 
the continuity of the projections 
on monomial subspace, in particular 
we define the following functor acting on the matrices
\begin{equation}\label{operegula}
\Pi_K A:=
\begin{cases}
 A_{j}^{k}(\ell) \qquad \qquad \text{if} \; |\ell|\le K\,, \\
0 \qquad \qquad \qquad \mbox{otherwise}
\end{cases}
\qquad \qquad  \Pi_K^\perp:= \mathrm{I}-\Pi_K\,.
\end{equation}
Finally we define for $\mathtt b\in \N$ the operator $\langle \pa_{\vphi}\rangle^{\mathtt{b}}A$
whose matrix element are
\begin{equation}\label{funtore}
(\langle \pa_\vphi \rangle^{\mathtt b} A )_{j}^{k}(\ell) :=  
\langle \ell  \rangle^{\mathtt b} A_j^{k}(\ell)\,.
\end{equation}
%In the sequel let $1>\g>\g^{3/2}>0$ be fixed constants.

\begin{definition}{\bf ($\s$-modulo tame operators).}\label{TameConstantsSharpOP}
Let $0<s_0< \mathcal{S}\le \infty$ be two integer numbers and $\s\geq 0$.
We say that 
a linear operator $A$ is $\sigma$-\textit{modulo tame} w.r.t. a 
non-decreasing sequence of positive real numbers
$\{\mathfrak{M}^{\sharp}_A(\s,s)\}_{s=s_0}^\mathcal{S}$ if (recall \eqref{space})
\begin{equation*}%\label{SigmaTame}
\lVert \underline{A}\, \underline{u} \rVert_{s}\le \mathfrak{M}^{\sharp}_A(\s,s) 
\lVert u \rVert_{s_0+\sigma}+\mathfrak{M}^{\sharp}_A(\s,s_0) \lVert u \rVert_{s+\sigma} 
\qquad u\in H^s\,,
\end{equation*} 
for any $s_0\le s\le \mathcal{S}$. 
We call $\mathfrak{M}^{\sharp}_A(\s,s)$  a {\sc modulo tame constant} 
for the operator $A$. When the index $\s$ is not relevant 
we write $\gotM_{A}^{\sharp}(\s,s)=\gotM^{\sharp}_{A}(s)$. 
For an operator $\mathcal{T}$ acting on 
$L^2(\T^{\nu+1}; \mathbb{C}^2)$ of the form \eqref{forma-complessa}
we set 
\[
\gotM^{\sharp}_{\mathcal{T}}(\s,s):=\max\{\gotM^{\sharp}_{\mathcal{T}_{+}^{+}}(\s,s),
\gotM^{\sharp}_{\mathcal{T}_{+}^{-}}(\s,s)\}.
\]
\end{definition}

\begin{definition}\label{menounomodulotamesenzaLip}
We say that $A$ is $(-1/2)$-modulo tame if 	
$\langle D \rangle^{1/4}{A}  \langle D \rangle^{1/4}$ 
is $0$-modulo tame, where $\langle D \rangle$ is the Fourier multiplier with symbol $\langle j \rangle$. We denote
\begin{equation*}%\label{anagrafe} 
\begin{aligned}
&\mathfrak M^{{\sharp}}_{A}(-1/2,s):=  
\mathfrak M^{\sharp}_{ \langle D \rangle^{1/4}A 
\langle D \rangle^{1/4}}(0,s), \\
& \mathfrak M^{\sharp}_{A}(-1/2,s,a):= 
\mathfrak M^{\sharp}_{\langle \pa_\vphi \rangle^{a}  
\langle D \rangle^{1/4}A \langle D \rangle^{1/4}}(0, s)\,, 
\quad a\ge 0\,.
\end{aligned}
\end{equation*}
\end{definition}

\begin{definition}{\bf (Lip-$\sigma$-modulo tame).}\label{def:op-tame} 
Let $\s\geq 0$. A  linear operator $ A := A(\omega) $, $\omega\in \calO\Subset\mathbb{R}^{\nu}$,  is  
Lip-$\s$-modulo-tame  w.r.t. a non-decreasing sequence 
$\{{\mathfrak M}_{A}^{\sharp, \g} 
(\s, s) \}_{s=s_0}^{\mathcal{S}}$ if 
the majorant operators 
$  \underline{ A }$, 
$\underline{\Delta_{\omega,\omega'} A}$ 
(see \eqref{defDELTAomega}) 
are Lip-$\s$-tame w.r.t. these constants, i.e. they 
satisfy the following weighted tame estimates:  
  for all %$\omega\neq \omega'\in \cO$,
 $ s \geq s_0 $ and  for any $u \in H^{s} $,  
\begin{equation*}%\label{CK0-tame}	
\sup_{\omega\in\calO}\| \underline{A} u\|_s\,,
 \sup_{\omega\neq \omega'\in \mathcal{O}}{\g}  
 \|\underline{\Delta_{\omega,\omega'} A}  u\|_s 
 \leq  
{\mathfrak M}_{A}^{\sharp, \g} (\s,s_0) \| u \|_{s+\s} +
{\mathfrak M}_{A}^{\sharp, \g} (\s,s) \| u \|_{s_0+\s} \,.
\end{equation*}
The constant $ {\mathfrak M}_A^{{\sharp, \g}} (\s,s) $ 
is called the {\sc modulo-tame constant} of the operator $ A $. 
When the index $\s$ is not relevant 
we write $ {\mathfrak M}_{A}^{{\sharp, \g}} (\s, s) 
= {\mathfrak M}_{A}^{{\sharp, \g}} (s) $. 
\end{definition}

\begin{definition}\label{menounomodulotame}

We say that $A$ is Lip-$(-1/2)$-modulo tame if 	
$\langle D \rangle^{1/4}{A}  \langle D \rangle^{1/4}$ 
is Lip-$0$-modulo tame, where $\langle D \rangle$ is the Fourier multiplier with symbol $\langle j \rangle$. We denote
\begin{equation*}%\label{anagrafe} 
\begin{aligned}
&\mathfrak M^{{\sharp, \g}}_{A}(-1/2,s):=  
\mathfrak M^{\sharp, \g}_{ \langle D \rangle^{1/4}A 
\langle D \rangle^{1/4}}(0,s), \\
& \mathfrak M^{\sharp, \g}_{A}(-1/2,s,a):= 
\mathfrak M^{\sharp, \g}_{\langle \pa_\vphi \rangle^{a}  
\langle D \rangle^{1/4}A \langle D \rangle^{1/4}}(0, s)\,, 
\quad a\ge 0\,.
\end{aligned}
\end{equation*}
\end{definition}
 
We refer the reader to the Appendix of \cite{FGP}, \cite{FGP1} 
for the properties of tame and modulo-tame operators. 
In particular we shall apply several times the results of 
Lemma $A.5$ in \cite{FGP}. 
We remark that in order to estimate the 
tame constants of operators $A$ acting on 
$L^2(\mathbb{T}^{\nu+1}; \mathbb{C}^2)$ 
we shall use these results for each entry of the matrix $A$.

\medskip
\noindent
{\bf Classes of  ``smoothing'' operators.}
The following class of smoothing (in space) 
operators has been introduced in \cite{FGP}, \cite{FGP1}.

\begin{definition}\label{ellerho}
Let $p\in \mathbb{N}$ with $s_0\le p< \mathcal{S}\le \infty$ (see \eqref{scala}).
Fix $\rho\in\mathbb{N}$,  with $\rho\geq 3$ 
and consider $\calO\Subset\R^\nu$. 
We denote by $\gotL_{\rho, p}=\gotL_{\rho, p}(\calO)$ 
the set of the linear operators 
$A=A(\omega)\colon H^s\to H^{s}$, $\omega\in \calO$ with the following properties:

\vspace{0.5em}
\noindent
$\bullet$ the operator $A$ is Lipschitz in $\omega$,

\vspace{0.5em}
\noindent		
$\bullet$  the operators $\partial_{\varphi}^{\vec{\tb}} A$,  $[\partial_{\varphi}^{\vec{\tb}} A, \partial_x]$, for all 
$\vec{\tb}=(\tb_1,\ldots,\tb_\nu)\in \mathbb{N}^{\nu}$ with 
$0\le |\vec{\tb}| \leq \rho-2$  have the following properties, 
for any $s_0\le s< \mathcal{S}$:
\begin{itemize}
		
\item[(i)] 
for any $m_{1},m_{2}\in \mathbb{R}$, $m_{1},m_{2}\geq0$
and $m_{1}+m_{2}=\rho-|\vec{\tb}|$ 
one has that  
$\langle D\rangle^{m_{1}} \pa_{\vphi}^{\vec{\mathtt{b}}}A \langle D\rangle^{m_{2}}$
is Lip-$0$-tame according to Definition \ref{LipTameConstants} 
and we set
\begin{equation*}%\label{megaTame2}
\gotM^{\gamma}_{\pa_{\vphi}^{\vec{\mathtt{b}}}A}(-\rho+|\vec{\tb}|,s):=
\sup_{\substack{m_{1}+m_{2}=\rho-|\vec{\tb}|\\
m_{1},m_{2}\geq0}}\gotM^{\gamma}_{\langle D\rangle^{m_{1}} \pa_{\vphi}^{\vec{\mathtt{b}}}A 
\langle D\rangle^{m_{2}}}(0,s)\,;
\end{equation*}
		
\item[(ii)]
for any $m_{1},m_{2}\in \mathbb{R}$, $m_{1},m_{2}\geq0$
and $m_{1}+m_{2}=\rho- |\vec{\tb}|-1$ 
one has that  
$\langle D\rangle^{m_{1}} [\pa_{\vphi}^{\vec{\mathtt{b}}}A,\pa_{x}] \langle D\rangle^{m_{2}}$
is Lip-$0$-tame %according to Definition \ref{LipTameConstants} 
and we set
\begin{equation*}%\label{megaTame4}
\gotM^{\gamma}_{[\pa_{\vphi}^{\vec{\mathtt{b}}}A,\pa_{x}]}(-\rho+|\vec{\tb}|+1,s):=
\sup_{\substack{m_{1}+m_{2}=\rho-|\vec{\tb}|-1\\
m_{1},m_{2}\geq0}}\gotM^{\gamma}_{\langle D\rangle^{m_{1}} 
[\pa_{\vphi}^{\vec{\mathtt{b}}}A,\pa_{x}] \langle D\rangle^{m_{2}}}(0,s)\,.
\end{equation*}
\end{itemize}
We define for $0\le \tb\le \rho-2$
\begin{equation*}%\label{Mdritta}
\begin{aligned}\mathbb{M}^{\gamma}_{A}(s, \mathtt{b}):=&\max_{0\leq |\vec{\tb}|\leq \tb}\max\left(
\gotM_{\pa_{\vphi}^{\vec{\mathtt{b}}}A}^{\gamma}(-\rho+|\vec{\tb}|,s),
\gotM_{\pa_{\vphi}^{\vec{\mathtt{b}}}[A,\pa_{x}]}^{\gamma}(-\rho+|\vec{\tb}|+1,s)\right).
\end{aligned}
\end{equation*}
Let $s\geq p$ and consider $i_1=i_1(\omega), i_2=i_2(\omega)$, defined for $\omega\in \mathcal{O}_1$ and $\mathcal{O}_2$ respectively, such that $\lVert i_k \rVert^{\gamma, \calO_k}_{s}\lesssim 1$, $k=1, 2$.
Consider an operator $A=A(\omega, i_k(\omega))$ defined for $\omega\in \mathcal{O}(i_k)\subseteq \mathcal{O}_k$, $k=1, 2$, and
we define 
\begin{equation*}%\label{DELTA12}
\Delta_{12}A=\Delta_{12} A(\omega, i_1(\omega), i_2(\omega)):=A(i_1)-A(i_2)\,, \quad \mbox{for} \quad \omega\in \mathcal{O}(i_1)\cap\mathcal{O}(i_2).
\end{equation*}
We require the following:

\vspace{0.5em}
\noindent		
$\bullet$ The operators
$\pa_{\vphi}^{\vec{\mathtt{b}}}\Delta_{12}A $, 
$[\pa_{\vphi}^{\vec{\mathtt{b}}}\Delta_{12}A ,\pa_{x}]$, for  $0\le |\vec{\tb}|\leq \rho-3$, 
have the following properties:

\begin{itemize}
\item[(iii)] 
for any $m_{1},m_{2}\in \mathbb{R}$, $m_{1},m_{2}\geq0$
and $m_{1}+m_{2}=\rho-|\vec{\tb}|-1$ 
one has that  
$\langle D\rangle^{m_{1}} \pa_{\vphi}^{\vec{\mathtt{b}}} \Delta_{12}A  \langle D\rangle^{m_{2}}$
is  bounded on $H^{p}$. More precisely there is a positive constant 
$\mathfrak{N}_{\pa_{\vphi}^{\vec{\mathtt{b}}}\Delta_{12}A }(-\rho+|\vec{\tb}|+1,p)$ such that, for any $h\in H^{p}$,
we have 
\begin{equation*}%\label{megalipTame2}
\sup_{\substack{m_{1}+m_{2}=\rho-|\vec{\tb}|-1\\
m_{1},m_{2}\geq0}}
\|\langle D \rangle^{m_{1}}
\pa_{\vphi}^{\vec{\mathtt{b}}}\Delta_{12}A\langle D \rangle^{m_{2}} h\|_{p}\leq 
\mathfrak{N}_{\pa_{\vphi}^{\vec{\mathtt{b}}}\Delta_{12}A }(-\rho+|\vec{\tb}|+1,p)\|h\|_{p}\,
\end{equation*}
uniformly for $\omega\in \mathcal{O}(i_1)\cap\mathcal{O}(i_2)$	.	
\item[(iv)] 
for any $m_{1},m_{2}\in \mathbb{R}$, $m_{1},m_{2}\geq0$
and $m_{1}+m_{2}=\rho-|\vec{\tb}|-2$ 
one has that  %the operator
$\langle D\rangle^{m_{1}} [\pa_{\vphi}^{\vec{\mathtt{b}}} \Delta_{12}A ,\pa_{x}] 
\langle D\rangle^{m_{2}}$
is  bounded on $H^{p}$. More precisely there is a positive constant 
$\mathfrak{N}_{[\pa_{\vphi}^{\vec{\mathtt{b}}}\Delta_{12}A,\pa_{x}] }(-\rho+|\vec{\tb}|+2,p)$ such that for any 
$h\in H^{p}$ one has
\begin{equation*}%\label{megalipTame4}
\sup_{\substack{m_{1}+m_{2}=\rho-|\vec{\tb}|-2\\
m_{1},m_{2}\geq0}}
\|\langle D\rangle^{m_{1}}[
\pa_{\vphi}^{\vec{\mathtt{b}}}\Delta_{12}A,\pa_{x}]
\langle D\rangle^{m_{2}} h\|_{p}\leq 
\mathfrak{N}_{[\pa_{\vphi}^{\vec{\mathtt{b}}}
\Delta_{12}A, \pa_{x}] }(-\rho+|\vec{\tb}|+2,p)\|h\|_{p}\,
\end{equation*}
uniformly for $\omega\in \mathcal{O}(i_1)\cap\mathcal{O}(i_2)$.	
\end{itemize}
		
We define for $0\le \tb\le \rho-3$
\begin{equation*}%\label{Mdrittaconlai}
\begin{aligned}\mathbb{M}_{\Delta_{12}A }(p, \mathtt{b})
:=&\max_{0\le |\vec{\tb}|\leq \tb}\max \left(
\mathfrak{N}_{\pa_{\vphi}^{\vec{\mathtt{b}}}\Delta_{12}A }(-\rho+|\vec{\tb}|+1,p)\,,
\mathfrak{N}_{\pa_{\vphi}^{\vec{\mathtt{b}}}
[\Delta_{12}A ,\pa_{x}]}(-\rho+|\vec{\tb}|+2,p)\right)\,.
\end{aligned}
\end{equation*}

\noindent
By construction one has that 
$\mathbb{M}^{\gamma}_{A}(s, \mathtt{b}_1)\leq 
\mathbb{M}^{\gamma}_{A}(s, \mathtt{b}_2)$
 if \,$0\le \tb_1\le \tb_2\le \rho-2$ and 
$\mathbb{M}_{\Delta_{12}A }(p, \mathtt{b}_1)\leq 
\mathbb{M}_{\Delta_{12}A }(p, \mathtt{b}_2)$
if \,$0\le \mathtt{b}_1\leq \mathtt{b}_2\leq \rho-3$.
\end{definition}
\noindent
For the properties of operators in the class $\gotL_{\rho,p}$
we refer the reader to Appendix $B$
in \cite{FGP}.

\subsubsection{pseudo differential operators and 
symbolic calculus}\label{sec:pseudopseudo}

Following \cite{BM1} %and \cite{Taylor} 
we give the following definitions.

\begin{definition}{\bf (Symbols).}\label{pseudoSimbo}
Let $m\in \R$.  We define the class $S^m$ as the set of symbols
$a(x,j)$ which are
the restriction to 
$\mathbb{T}\times \mathbb{Z}$ of a complex valued function 
$a(x, \x)$ which is $C^{\infty}$ smooth on $\mathbb{T}\times\mathbb{R}$, 
$2\pi$-periodic in $x$ and satisfies
\begin{equation}\label{space3}
|\pa_{x}^{\al}\pa_{\xi}^{\beta}a(x,\x)|\leq C_{\al,\beta}\langle \x\rangle^{m-\beta}\,,
\;\;\forall \; \al,\beta\in \mathbb{N}\,.
\end{equation} 
\end{definition}

\begin{definition}{\bf (Pseudo differential operators).}\label{pseudoR}
Let $m\in \R$.
A linear operator $A$ is called pseudo differential 
of order $\le m$ if its action on any $H^s(\T; \mathbb{C})$ with 
$s\ge m$ is given by
\begin{equation}\label{weylquanti}
A u:=\opw(a(x,\x))[u]:=
%\sum_{j\in\Z} u_j e^{\mathrm{i} j x} =
\frac{1}{\sqrt{2\pi}}\sum_{k\in \mathbb{Z}}
\Big(\sum_{\substack{j\in \Z\setminus\{0\}}}\hat{a}\Big(k-j,\frac{k+j}{2}\Big)u_{j}\Big)
\frac{e^{\ii kx}}{\sqrt{2\pi}}\,,\qquad u=\frac{1}{\sqrt{2\pi}}
\sum_{j\in\Z\setminus\{0\}}u_{j}e^{\ii jx}\,,
%\sum_{j\in\Z} a(x,j) u_j e^{\mathrm{i}j x} \,,
\end{equation}
where   $a(x, \x)$ is a symbol in the class 
$S^{m}$ (see Def. \ref{pseudoSimbo})
and $\hat{a}(\eta,\x)$ 
denotes the Fourier transform of $a(x,\x)$ in the variable $x\in\mathbb{T}$.
We shall also write 
%We denote by 
$A[\cdot]=\opw(a)[\cdot]=\opw(a(x,\x))[\cdot]$.
%the pseudo operator with symbol $a:=a(x, \x)$.
We call $OPS^m$ the class of the pseudo 
differential operator of order less or equal to $m$ and
$OPS^{-\infty}:=\bigcap_m OPS^m$.
%We define the class $S^m$ as the set of symbols which satisfies \eqref{space3}. 
\end{definition}

We will consider mainly operators acting on $H^s(\T;\mathbb{C})$ with a quasi-periodic time dependence.  
In the case of pseudo differential operators this corresponds\footnote{since $\omega$ 
is diophantine we can replace the time variable with angles $\varphi\in\T^{\nu}$. 
The time dependence is recovered by setting $\varphi=\omega t$.} 
to considering symbols $a(\varphi, x, \x)$ with $\varphi\in\T^{\nu}$. 
Clearly these operators can be thought as acting on functions 
$u(\vphi,x)=\sum_{j\in\Z}u_{j}(\vphi)e^{\ii jx}$ in 
$H^s(\T^{\nu+1}; \mathbb{C})$ in the following sense:
\[
(Au)(\vphi,x)=\frac{1}{\sqrt{2\pi}}
\sum_{k\in \mathbb{Z}}
\Big(\sum_{\substack{j\in \Z}}
\hat{a}\Big(\vphi,k-j,\frac{k+j}{2}\Big)u_{j}(\vphi)\Big)\frac{e^{\ii kx}}{\sqrt{2\pi}}\,,
%\sum_{j\in\mathbb{Z}}a(\vphi,x,j) u_{j}(\vphi)e^{\mathrm{i} jx}\,, 
\quad a(\vphi,x,j)\in S^{m}\,.
\]
The symbol $a(\varphi, x, \xi)$ is $C^{\infty}$ smooth also in the variable $\varphi$. 
We still denote 
$A:=A(\varphi)=\opw(a(\varphi, \cdot))=\opw(a)$.

\begin{definition}
Let $a:=a(\vphi,x, \x)\in S^{m}$ and set $A:=\opw(a)\in OPS^{m}$,
\begin{equation}\label{norma}
|A|_{m,s,\alpha}:=\max_{0\leq \beta\leq \alpha} 
\sup_{\x\in\mathbb{R}}\|\pa_{\xi}^{\beta}a(\cdot,\cdot, \x)\|_{s}
\langle \x\rangle^{-m+\beta}\,.
\end{equation}
We will use also the notation 
$\lvert a \rvert_{m, s, \alpha}:=|A|_{m,s,\alpha}$.
\end{definition}
Note that the norm $|\cdot|_{m,s,\alpha}$ is non-decreasing in $s$ and $\alpha$.
Moreover given a symbol $a(\vphi,x)$ independent of $\x\in\mathbb{R}$, the norm 
of the associated multiplication operator $\opw(a)$ 
is just the $H^{s}$ norm of the function $a$.
If on the contrary the symbol $a(\x)$ depends only on $\x$, 
then the norm of the corresponding 
Fourier multipliers $\opw(a(\x))$ 
is just controlled by a constant.
We shall use the following notation, used also in \cite{AB1}.
For any $m\in \mathbb{R}\setminus\{0\}$ we set
\begin{equation}\label{Demme}
|D|^{m}:=\opw(\chi(\x)|\x|^{m})\,,
\end{equation}
where $\chi\in C^{\infty}(\mathbb{R}, \mathbb{R})$ is 
an even and positive cut-off function such that
\begin{equation}\label{cutofffunct}
\chi(\x)=\left\{
\begin{aligned}
&0 \quad {\rm if}\,\, |\x|\leq 1/3 %\frac{1}{3} 
\\
&1 \quad {\rm if}\,\, |\x|\geq 2/3 %\frac{2}{3}
\end{aligned}\right.\qquad \pa_{\x}\chi(\x)>0\quad \forall\, 
\x\in \left(\frac{1}{3}, \frac{2}{3}\right)\,.
\end{equation}
To simplify the notation we shall also write
\begin{equation}\label{notasignXi}
\opw(\sign(\x)) \qquad {\rm to \; denote}\qquad \opw(\chi(\x)\sign(\x))\,,
\qquad {\rm where}\qquad \sign(\x):=\x/|\x|\,.
\end{equation}
%where $\sign(\x):=\x/|\x|$.

As in formula \eqref{tazza10}, if $A=\opw(a(\omega,\vphi,x, \x))\in OPS^{m}$ 
is a family of 
pseudo differential operators with symbols 
$a(\omega,\vphi,x, \x)$ belonging to $S^{m}$ and 
depending in a Lipschitz way on some parameter 
$\omega\in \calO\subset \mathbb{R}^{\nu}$, 
we set
\begin{equation*}%\label{norma2}
|A|_{m,s,\alpha}^{\g,\calO}:=\sup_{\omega\in \calO}|A|_{m,s,\alpha}+
\g \sup_{\omega_1,\omega_2\in \calO}\frac{|\opw\big(a(\omega_1,\vphi,x, \x)-a(\omega_2,\vphi,x, \x)\big)|_{m,s-1,\alpha}}{|\omega_1-\omega_2|}\,.
\end{equation*}
For the properties  of compositions, adjointness and quantitative estimates of the actions
on the Sobolev spaces $H^{s}$  of pseudo differential operators we 
refer to Appendix B of \cite{FGP1}.

\medskip

From now on we consider symbols
$a(\varphi, x, \xi)=a(\omega, i(\omega); \vphi,x,\x)$ 
where $\omega\in\mathcal{O} \Subset\mathbb{R}^{\nu}$ 
and $i=i(\omega)$ is a lipschitz function such that $\lVert i \rVert^{\gamma, \calO}_{\mathfrak{s}_0}\le 1$ for some $\mathfrak{s}_0\geq [\nu/2]+2$. \\
Let us consider two subsets $\mathcal{O}_1, \mathcal{O}_2\Subset \mathbb{R}^{\nu}$ and two lipschitz functions $i_1(\omega), i_2(\omega)$ defined respectively on $\mathcal{O}_1$, $\mathcal{O}_2$.
We denote by $\Delta_{12} a=\Delta_{12} a(\vphi,x,\x)$ the symbol
\[
\Delta_{12} a=\Delta_{12} a(\omega, i_1(\omega), i_2(\omega))
:=a(\omega, i_1(\omega))-a(\omega, i_2(\omega))
\] 
defined for $\omega\in \mathcal{O}_1\cap\mathcal{O}_2$.
To simplify the notation we shall omit the dependence on $(\omega,i(\omega))$
in symbols and operators.

\vspace{0.9em}
\noindent
{\bf Weyl/Standard quantizations.} 
The quantization of a symbol in $S^{m}$ given in \eqref{weylquanti}
is called \emph{Weyl} quantization.  
% Notice that Definition \ref{pseudoR}
%is slightly different form Definition $2.8$ in \cite{FGP1} 
%where,
We shall compare this quantization with the \emph{standard one}, that we now recall:
 given a symbol $b(x,j)\in S^{m}$, the associated pseudo differential operator
is defined as $B:=\op(b)$ with
\begin{equation}\label{standardquanti}
\op(b)\sum_{j\in\Z\setminus\{0\}} u_j \frac{e^{\mathrm{i} j x}}{\sqrt{2\pi}} =
\sum_{j\in\mathbb{Z}\setminus\{0\}}b(x,j) u_{j}\frac{e^{\mathrm{i} jx}}{\sqrt{2\pi}}
=\frac{1}{\sqrt{2\pi}}\sum_{k\in \mathbb{Z}}
\Big(\sum_{\substack{j\in \Z\setminus\{0\}}}\hat{b}(k-j,j)u_{j}\Big)\frac{e^{\ii kx}}{\sqrt{2\pi}}\,.
\end{equation}
It is known (see paragraph $18.5$ in \cite{Hormo}) that the two quantizations
are equivalent. 
%up to symbols in the class $S^{-\rho}$, for any $\rho>0$.
Moreover one can transform the symbols between 
different quantizations by using the formul\ae
\begin{equation}\label{bambola}
\opw(a)=\op(b)\qquad \hat{a}(k,j)=\hat{b}\big(k, j-\frac{k}{2}\big)\,.
\end{equation}
Along the paper {we shall use only the Weyl quantization, which is convenient in a Hamiltonian setting. We use the equivalence with the standard one to recover some known results of pseudo differential calculus on tori that have been proved in the standard case. We mostly refer to the results in \cite{FGP, FGP1} where symbolic calculus is provided with sharp estimates on the seminorms of the symbols.}

\smallskip
 
We shall need an asymptotic expansion of the symbol
$a$ in terms of the symbol $b$.
The distributional kernel of the operator $\opw(a)$ is given by
the oscillatory integral
\begin{equation}\label{quantiguai}
K(x,y)=\frac{1}{2\pi}\int_{\mathbb{R}}e^{\ii(x-y)\x}a\big(\frac{x+y}{2},\x\big)d\x\,.
\end{equation}
Hence we may recover the Weyl symbol $a$ from the kernel by 
the inverse formula
\begin{equation}\label{quantiguai2}
a(x,\x)=\int_{\mathbb{R}}K\big(x+\frac{t}{2},x-\frac{t}{2}\big)e^{-\ii t\x}dt\,.
\end{equation}
If we consider the standard quantization $\op(b)=\op^W(a)$ 
we can express the distributional kernel in terms of $b$
\begin{equation*}%\label{quantiguai3}
K(x,y)=\frac{1}{2\pi}\int_{\mathbb{R}}e^{\ii (x-y)\x}b(x,\x)d\xi.
\end{equation*}
Then by \eqref{quantiguai} and \eqref{quantiguai2} 
we obtain an expression of the symbol $a$ in terms of $b$
\begin{equation}\label{quantiguai4}
a(x,\x)=\frac{1}{2\pi}\int_{\mathbb{R}^2}e^{-\ii z \zeta}
b\big(x+\frac{z}{2},\x-\zeta\big)dz d\zeta\,
\end{equation}
and viceversa
\begin{equation}\label{escobar}
b(x,\x)=\frac{1}{2\pi}\int_{\mathbb{R}^2}e^{-\ii z \zeta}
a\big(x-\frac{z}{2},\x-\zeta\big)dz d\zeta\,.
\end{equation}

\noindent
We have the following important Lemma (see e.g. Lemma $3.5$ in \cite{BD}).
\begin{lemma}\label{weyl/standard}
Let $b(\vphi,x,\x)\in S^{m}$ and consider the function $a(\vphi,x,\x)$ 
in \eqref{quantiguai4}.
Then for any $N\geq 1$ we have the expression
\begin{equation*}%\label{quantiguai5}
a(\vphi,x,\x)=\sum_{p=0}^{N-1}\frac{(-1)^{p}}{2^{p}p!}
\big(\pa_{x}^{p}D_{\x}^{p}b\big)(\vphi,x,\x)+\widetilde{a}(\vphi,x,\x)
\end{equation*}
where $D_{\x}=-\ii \pa_{\x}$ and $\widetilde{a}$ satisfies
the following bounds 
\begin{equation}\label{quantiguai6}
\begin{aligned}
|\widetilde{a}|^{\gamma, \calO}_{m-N, s, \alpha}
&\lesssim
|b|^{\gamma,\mathcal{O}}_{m,s+s_0+N,\alpha+2+2N}\,,\\
| \Delta_{12}\widetilde{a}|_{m-N, p, \alpha}&\lesssim
|\Delta_{12} b|_{m,p+s_0+N,\alpha+2+2N}\,.
\end{aligned}
\end{equation}
%\begin{equation}\label{quantiguai6bis}
%| \Delta_{12}\widetilde{a}|_{m-N, p, \alpha}\lesssim
%|\Delta_{12} b|_{m,p+s_0+N,\alpha+2+2N}\,.
%\end{equation}
\end{lemma}

\begin{proof}
The proof follows word by word the proof of Lemma $3.5$ in \cite{BD} taking into account the norm \eqref{norma} to estimate the regularity of the symbol. The bound for the variations $\Delta_{\omega, \omega'} \widetilde{a},\,\Delta_{12} \widetilde{a}$ follows by repeating the proof with $a\rightsquigarrow \Delta_{\omega, \omega'}a,\, \Delta_{12} a$, $b\rightsquigarrow \Delta_{\omega, \omega'} b,\,\Delta_{12} b$.
%The proof is postponed in Appendix %\ref{sec:proofs}.
%\ref{proof213}.
\end{proof}

As a consequence of Lemma \ref{weyl/standard} we have the following. 
\begin{lemma}\label{lem:escobar}
Fix $\rho,p$ as in Definition \ref{ellerho}. Let $m\in\mathbb{R}$, $N:=m+\rho$ and consider a symbol $b(\vphi,x,\x)$ in $S^{m}$.
%$a=a(\omega, i(\omega))$ 
%n $S^{m}$
%\comment{credo che sia $a\in S^m$? se no cos'e' $m$?}
%depending on   $\omega\in \calO\subset 
%\mathbb{R}^{\nu}$ and on $\mathfrak{I}$  
%in a Lipschitz way.
 Then there exists a remainder $R_{\rho}\in \gotL_{\rho,p}$ such that 
\begin{equation}\label{passo}
\op(b)=\opw(c)+R_{\rho}, \quad c(\vphi,x,\x):=\sum_{p=0}^{N-1}\frac{(-1)^{p}}{2^{p}p!}
\big(\pa_{x}^{p}D_{\x}^{p}b\big)(\vphi,x,\x)
\end{equation} 
with
\begin{equation}\label{passo2}
|c|_{m,s,\alpha}^{\gamma,\calO}
%+\mathbb{M}^{\gamma}_{R_{\rho}}(s,\tb)
\leq_{s,\rho}|b|^{\gamma,\calO}_{m,s+N,\alpha+N}\,,
\qquad
\mathbb{M}^{\gamma}_{R_{\rho}}(s,\tb)
\leq_{s,\rho}|b|^{\gamma,\calO}_{m,s+N,N}\,,
\end{equation}
for all $0\le \tb\le \rho-2$ and $s_0\le s\le \mathcal{S}$.
Moreover one has
\begin{equation}\label{passo3}
\lvert \Delta_{12} c   \rvert_{m+m', p, \alpha}
%+\mathbb{M}_{\Delta_{12} R_{\rho} } (p, \tb)
\leq_{p,\rho}
|\Delta_{12}b|_{m,s+N,\alpha+N}\,,
\qquad \mathbb{M}_{\Delta_{12} R_{\rho} } (p, \tb)\leq_{p,\rho}
|\Delta_{12}b|_{m,s+N,N}\,,
\end{equation}
for all $0\le \tb \le \rho-3$ and where $p$ is the constant 
given in Definition \ref{ellerho}.
\end{lemma}

\begin{proof}
Let $a(\vphi,x,\x)$ be a symbol such that $\op^W(a)=\op(b)$.
We  apply Lemma \ref{weyl/standard}. 
and we set $R_{\rho}:=\opw(\widetilde{a})$.
The bounds \eqref{passo2} and \eqref{passo3} 
on the symbol $c(\vphi,x,\x)$
follow by an explicit computation.
The estimates on the remainder $R_{\rho}$ follows 
by \eqref{quantiguai6}
%, \eqref{quantiguai6bis} 
and {Lemma B.2
in \cite{FGP}.}
\end{proof}

\begin{remark}\label{standard/Weyl}
The formula \eqref{escobar} gives a explicit expression of a standard symbol $b(x, \xi)$ as a function of a Weyl symbol $a(x, \xi)$. Then one can prove Lemmata \ref{weyl/standard}, \ref{lem:escobar} inverting the role of the two symbols. More precisely one can obtain 
\begin{equation}\label{escobar2}
b(\vphi,x,\x)=f(\varphi, x, \xi)  +\widetilde{b}(\vphi,x,\x), \qquad f(\varphi, x, \xi):=\sum_{p=0}^{N-1}\frac{1}{2^{p}p!}
\big(\pa_{x}^{p}D_{\x}^{p}a\big)(\vphi,x,\x)
\end{equation}
with $\widetilde{b}$ satisfying estimates like \eqref{quantiguai6}. Hence $\op^W(a)=\op(f)+R_{\rho}$ for some remainder $R_{\rho}\in\mathfrak{L}_{\rho, p}$ satisfying bounds like \eqref{passo2}, \eqref{passo3}.
\end{remark}

It is known that the $L^2$-adjoint
of a pseudo differential operator is pseudo differential.
In the Weyl quantization it is easier, w.r.t. the standard quantization,
to determine the symbol of the adjoint operator.
In particular we have the following:

\vspace{0.5em}
\noindent
$\bullet$ if $A:=\opw(a(x,\x))$, with $a(x,\x)\in S^{m}$
then the operator $\ov{A}$ defined in \eqref{conjugate-Op}
 and the $L^{2}$-adjoint
 have the form
\begin{equation}\label{pseudobarra}
\ov{A}:=\opw\left(\ov{a(x,-\x)}\right)\,,\qquad
A^{*}:=\opw\left(\ov{a(x,\x)}\right)\,.
\end{equation}

\noindent
{\bf Composition of pseudo differential operators.}
Let 
$$ 
\s(D_{x},D_{\x},D_{y},D_{\eta}) := D_{\x}D_{y}-D_{x}D_{\eta} 
$$
where $D_{x}:=- \ii \pa_{x}$ and $D_{\x},D_{y},D_{\eta}$ are similarly defined. 

\begin{definition}{\bf (Asymptotic expansion of composition symbol).}\label{cancelletti}
Let $\rho\geq 0$.
Consider symbols 
$a(x,\x)\in S^{m}$ and $b(x,\x)\in S^{m'}$. 
We define
\begin{equation}\label{espansione2}
(a\#_{\rho}^{W} b)(x,\x):=\sum_{k=0}^{\rho}\frac{1}{k!}
\left(
\frac{\ii}{2}\s(D_{x},D_{\x},D_{y},D_{\eta})\right)^{k}
\Big[a(x,\x)b(y,\eta)\Big]_{|_{\substack{x=y, \x=\eta}}}
\end{equation}
modulo symbols in $S^{m+m'-\rho}$.
\end{definition}

\noindent
$ \bullet $
We have  the expansion  
\[
 a\#^{W}_{\rho}b =ab+\frac{1}{2 \ii }\{a,b\} -\frac{1}{8}\Big(\pa_{\x\x}a\pa_{xx}b
 -2\pa_{\x x}a\pa_{x\x}b+\pa_{xx}a\pa_{\x\x}b\Big)
% \cdots 
 \]
up to a symbol in 
$S^{m+m'-3}$, 
where 
\begin{equation*}%\label{poissonSimbo}
\{a,b\}:=\pa_{\x}a\pa_{x}b-\pa_{x}a\pa_{\x}b
\end{equation*}
denotes the Poisson bracket. 
%It is well know that composition of pseudo differential operator
%is a pseudo differential operator. 
In the following Lemma 
we prove that the composition of pseudo differential operators
is  pseudo differential up to 
up to a remainder which is smoothing in the $x$-variable.
We provide
precise estimates
on the $|\cdot|_{m,s,\alpha}^{\g,\calO}$-norm
of the symbol of the composition operator.
\begin{lemma}{\bf (Composition).}\label{James}
Let $m,m'\in \mathbb{R}$.
Fix  $\rho,p$ as in Definition \ref{ellerho},
such that $\rho \geq \!\max\{-(m+m'+1), 3\}$
and  define $N:=m+m'+\rho\geq1$.
Consider two symbols
$a(\vphi,x,\x)\in S^m$, $b(\vphi,x,\x)\in S^{m'}$.
There exist an operator $R_{\rho}\in \gotL_{\rho, p}$ 
and a constant $\s=\s(N)\sim N$
such that (recall Def. \ref{cancelletti}) 
\begin{equation}\label{escobar12}
\opw(a)\circ\opw(b)
=\opw(c)+R_{\rho}, \qquad c:=a\#^{W}_{N} b\in S^{m+m'}
\end{equation}
where
\begin{align}
\lvert c \rvert^{\g, \calO}_{m+m', s, \alpha}&\le_{s, \rho, \alpha, m, m'} 
\lvert a \rvert^{\g, \calO}_{m, s+\s,\s+\alpha}
\lvert b \rvert^{\g, \calO}_{m', s_0+\s, \alpha+\s}
\lvert a \rvert^{\g, \calO}_{m, s_0+\s,\s+\alpha}+
\lvert b \rvert^{\g, \calO}_{m', s+\s, \alpha+\s}\,,\label{crawford}\\
\mathbb{M}^{\g}_{R_{\rho}}(s, \tb)
&\le_{s, \rho,  m , m'} 
\lvert a \rvert^{\g, \calO}_{m,s+\s, \s}
\lvert b \rvert^{\g, \calO}_{m', s_0+\s, \s}
+\lvert a \rvert^{\g, \calO}_{m, s_0+\s, \s}
\lvert b \rvert^{\g, \calO}_{m', s+\rho+\s, \s}\,,\label{jamal}
\end{align}
for all $0\le \tb\le \rho-2$ and $s_0\le s\le \mathcal{S}$.
Moreover one has
%if $a$ and $b$ depends in a $C^1$ way on a parameter $i$ then
%\begin{equation}
%\Delta_{12} \op(a\# b)  =\op(\Delta_{12} c  )+\Delta_{12} R_{\rho} 
%\end{equation}
%and
\begin{align}
\!\!\!\lvert \Delta_{12} c   \rvert_{m+m', p+\s, \alpha+\s} 
&\le_{p, \alpha, \rho, m, m'} 
\lvert \Delta_{12} a \rvert_{m, p+\s, \s+\alpha}\lvert b \rvert_{m', p+\s, \alpha+\s}
+\lvert a \rvert_{m, p+\s, \s+\alpha}\lvert \Delta_{12} b   \rvert_{m', p+\s, \alpha+\s}
\,,\label{crawford2}\\
\mathbb{M}_{\Delta_{12} R_{\rho} } (p, \tb) 
&\le_{p, \rho, m, m'} 
\lvert \Delta_{12} a   \rvert_{m+1, p+\s, \s}
\lvert b \rvert_{m', p+\s, \s}+\lvert a \rvert_{m, p+\s, \s}
\lvert \Delta_{12} b   \rvert_{m'+1, p+\s, \s}\,,\label{jamal2}
\end{align}
for all $0\le \tb \le \rho-3$ and where $p$ 
is the constant given in Definition \ref{ellerho}.
%$s_0\le p\le \mathcal{S}$.
\end{lemma}

\begin{proof}
By Lemma \ref{lem:escobar} and   Remark \ref{standard/Weyl} 
we can write
\begin{equation}\label{escobar10}
\opw(a)\circ\opw(b)=\op({a}_{N})\circ\op({b}_{N})+Q_{\rho}
\end{equation}
where the symbols ${a}_{N}$ and ${b}_{N}$
are defined as (recall \eqref{passo}, \eqref{escobar2})
%are given (see  \eqref{escobar}, \eqref{escobar2})
\begin{equation}\label{escobar16}
{a}_{N}=\sum_{p=0}^{N-1}\frac{1}{2^{p}p!}
\big(\pa_{x}^{p}D_{\x}^{p}a\big)(\vphi,x,\x)\,,\qquad
{b}_{N}=\sum_{p=0}^{N-1}\frac{1}{2^{p}p!}
\big(\pa_{x}^{p}D_{\x}^{p}b\big)(\vphi,x,\x)\,.
\end{equation}
Moreover the remainder $Q_{\rho}$ satisfies bounds like \eqref{passo2},\eqref{passo3}.
 Lemma B.4 in \cite{FGP}  implies that
 \begin{equation}\label{escobar11}
 \op({a}_N)\circ\op({b}_N)=\op({c}_N)+L_{\rho}
% \stackrel{{\rm Lem. \ref{lem:escobar}}}{=}
% \opw(c)+\widetilde{R}_{\rho}+L_{\rho}
 \end{equation}
 for some $L_{\rho}\in \gotL_{\rho, p}$ and where the symbol 
 $c_N\in S^{m}$
 has the form 
 \begin{equation}\label{escobar15}
{c}_N(x,\x):=
 \sum_{n=0}^{N-1}\frac{1}{n! \mathrm{i}^{n}}(\pa_{\x}^{n}{a}_{N})(x,\x)\,\cdot
 (\pa_{x}^{n}{b}_{N})(x,\x)
 \end{equation}
 By Lemma  \ref{lem:escobar} applied to the operator $\op(c_N)$ in \eqref{escobar11} 
 we get
 $  \op({c}_{N})=\opw(c)+\widetilde{R}_{\rho}$
% \begin{equation}\label{escobar18}
%  \op({c}_{N})=\opw(c)+\widetilde{R}_{\rho}
% \end{equation}
 where
 $\widetilde{R}_{\rho}\in \gotL_{\rho,p}$
 and
  \begin{equation}\label{escobar19}
 c=\sum_{p=0}^{N-1}\frac{(-1)^{p}}{2^{p}p!}
\big(\pa_{x}^{p}D_{\x}^{p}{c}_{N}\big)(\vphi,x,\x)\,.
 \end{equation}
 By an explicit computation using \eqref{escobar15}, \eqref{escobar16}
 we deduce that the symbol $c$ in \eqref{escobar19}
 has the form $c=a\#^{W}_{N} b$ (see \eqref{espansione2}).
Using \eqref{escobar10}, \eqref{escobar11} we get the \eqref{escobar12}
  with $R_{\rho}:=\widetilde{R}_{\rho}+L_{\rho}+Q_{\rho}$.
  The estimates \eqref{crawford}-\eqref{jamal2}
  follow by combining the estimates in 
  Lemmata B.2 and B.4 in \cite{FGP} and Lemmata
  \ref{weyl/standard}, \ref{lem:escobar} and
  Remark \ref{standard/Weyl}.
\end{proof}

\begin{lemma}{\bf (Commutator).}\label{James2}
%Fix $\rho\in \mathbb{N}$, $m,m'\in \mathbb{R}$, 
%such that $\rho \geq \max\{-(m+m'+1), 3\}$ 
%and define $N:=m+m'+\rho\geq1$.
Let $m,m'\in \mathbb{R}$.
Fix  $\rho,p$ as in Definition \ref{ellerho},
such that $\rho \geq \max\{-(m+m'+1), 3\}$
and  define $N:=m+m'+\rho\geq1$.
Consider two symbols
$a(\vphi,x,\x)\in S^m$, $b(\vphi,x,\x)\in S^{m'}$.
%Let $a=a(\omega)\in S^m$, $b=b(\omega)\in S^{m'}$ be defined 
%on some subset $\calO\subset \mathbb{R}^{\nu}$ with 
%$m, m'\in\mathbb{R}$ and consider any  $\rho \geq \max\{-(m+m'+1), 3\}$. 
%Define also $N=m+m'+\rho\ge 1$ and 
% assume  that $a$ and $b$ depend in a Lipschitz way 
%on the parameter $\mathfrak{I}$.
There exist an operator $R_{\rho}\in \gotL_{\rho, p}$ 
and $\s=\s(N)\sim N$
such that the commutator between $\opw(a)$ and $\opw(b)$ has the form
\begin{equation}\label{escobar12commu}
[\opw(a), \opw(b)]
=\opw(c)+R_{\rho}, \qquad c:=a\star_{N} b\in S^{m+m'-1}
\end{equation}
where (recall \eqref{espansione2})
\begin{equation}\label{moyal}
a\star_{N} b:=a\#^{W}_{N}b-b\#^{W}_{N}a=\frac{1}{\ii}\{a,b\}+r\,,\qquad r\in S^{m+m'-3}\,.
\end{equation}
Moreover the symbols $\{a,b\}$, $r$ satisfy bounds as 
\eqref{crawford}, \eqref{crawford2} with $m+m'$ replaced by $m+m'-1$, 
$m+m'-3$ respectively. Finally the smoothing operator $R_{\rho}$
satisfies bounds as \eqref{jamal} and \eqref{jamal2}.
\end{lemma}

\begin{proof}
It is a direct application of Lemma \ref{James}. The expansion 
\eqref{escobar12commu}, \eqref{moyal}
follows  using formula  \eqref{espansione2}.
\end{proof}

\noindent
{\bf Notation.}
\begin{itemize}
%\item We shall write 
%$\Delta_{12}u:=u(i_1)-u(i_2)$.

\item 
Let  $a,b,c,d$ be symbols in the class
 $S^{m}$, $m\in \mathbb{R}$ (see Def. \ref{pseudoR})
depending in a Lipschitz way on $\omega\in \mathcal{O}$.
We shall write
\begin{equation}\label{matrixSym}
A:=A(\vphi,x,\x):=\sm{a}{b}{c}{d}
%\left(\begin{matrix}a & b\\
%c & d\end{matrix}\right)
\in S^{m}\otimes\mathcal{M}_{2}(\mathbb{C})
\end{equation}
to denote $2\times2$ matrices of symbols. 
With abuse of notation
we write $|A|_{m,s,\alpha}^{\gamma,\calO}$   to denote
$\max\{|f|_{m,s,\alpha}^{\gamma,\calO}\,, f=a,b,c,d\}$.

\item Similarly we shall write 
$\mathcal{Q}\in\gotL_{\rho,p}\otimes\mathcal{M}_{2}(\mathbb{C}) $ 
to denote a $2\times2$-matrix whose entries are 
smoothing operators in $\gotL_{\rho,p}$.
%Similarly for the class $\mathfrak{C}_{-q}$.
\end{itemize}

\subsection{Hamiltonian and translation invariant operators}\label{sec:group}
Let $\mathtt{v}\in \mathbb{Z}^{\nu}$ as in \eqref{velocityvec}.
We define  the
 subspace
\begin{align}
S_{\mathtt{v}}&:=\{u(\vphi,x)\in L^{2}(\mathbb{T}^{\nu+1};\mathbb{C}^{2})\,:\,
u(\vphi,x)=U(\vphi-\mathtt{v}x)\,, U(\Theta)\in L^{2}(\mathbb{T}^{\nu};\mathbb{C}^2)
\}\,.\label{funzmom}
\end{align}
With abuse of notation we denote by $S_{\mathtt{v}}$ also the subspace of
$L^{2}(\mathbb{T}^{\nu+1};\mathbb{C})$ of scalar functions $u(\vphi,x)$
of the form 
$u(\vphi,x)=U(\vphi-\mathtt{v}x)$, 
$U(\Theta)\in L^{2}(\mathbb{T}^{\nu};\mathbb{C})$.
We consider the symplectic form in the extended phase space
$(\vphi,Q,h,\bar{h})\in \mathbb{T}^{\nu}\times\mathbb{R}^{\nu}
\times H^{s}(\mathbb{T};\mathbb{C})\times H^{s}(\mathbb{T};\mathbb{C})$
\begin{equation}\label{formaestesa}
\Omega_{e}(\vphi,Q,h,\bar{h})=dQ\wedge d\vphi-\ii dh\wedge d\bar{h}\,
%\end{equation}
\;\;\footnote{namely, for any  $U_1:=(\vphi_1,Q_1, h_1,\ov{h_1})$,
$U_2:=(\vphi_2,Q_2, h_2,\ov{h_2})$, one has $
\Omega_e(U_1,U_2)= J \vect{\vphi_1}{Q_1}\cdot \vect{\vphi_2}{Q_2}
-\int_{\mathbb{T}} \vect{h_1}{\ov{h_1}}\cdot \ii J\vect{h_2}{\ov{h_2}}dx$,\;
$J:=\sm{0}{1}{-1}{0}$.
%\begin{equation*}%\label{formaestesa2}
%\Omega_e(U_1,U_2)= J \vect{\vphi_1}{Q_1}\cdot \vect{\vphi_2}{Q_2}
%-\int_{\mathbb{T}} \vect{h_1}{\ov{h_1}}\cdot \ii J\vect{h_2}{\ov{h_2}}dx\,,
%\qquad J:=\sm{0}{1}{-1}{0}\,,
%\end{equation*}
}
\end{equation}
where $\cdot$ denotes the standard scalar product on $\mathbb{R}^{2}$.
Given a function $F : \mathbb{T}^{\nu}\times\mathbb{R}^{\nu}
\times H^{s}(\mathbb{T};\mathbb{C})\times H^{s}(\mathbb{T};\mathbb{C})\to\mathbb{R}$
 its Hamiltonian vector field is given by
\[
X_{F}=\Big(\pa_{Q}F, -\pa_{\vphi}F, -\ii \pa_{\bar{h}}F, \ii \pa_{h}F    \Big)\,,
\]
since
\[
dH(U)[\hat{\eta}]=\Omega_e(\hat{\eta},X_{F}(U))\,,\quad \forall \;U,\hat{\eta}\in 
 \mathbb{T}^{\nu}\times\mathbb{R}^{\nu}
\times H^{s}(\mathbb{T};\mathbb{C})\times H^{s}(\mathbb{T};\mathbb{C})\,.
\]
Here $\pa_{h}=(\pa_{{\rm Re}(h)}-\ii \pa_{{\rm Im}(h)})/\sqrt{2}$, 
$\pa_{h}=(\pa_{{\rm Re}(h)}+\ii \pa_{{\rm Im}(h)})/\sqrt{2}$
denote the $L^{2}$-gradient.
The Poisson brackets between two Hamiltonians $F, G$
are defined as
%: \mathbb{R}^{\nu}\times\mathbb{R}^{\nu}
%\times H^{s}(\mathbb{T};\mathbb{C})\times H^{s}(\mathbb{T};\mathbb{C})\to\mathbb{R}$
\begin{equation}\label{poissonEstese}
\{G,F\}_{e}:=\Omega_e(X_{G},X_{F})=\pa_{Q}G\pa_{\vphi}F-\pa_{\vphi}G\pa_{Q}F
+\frac{1}{\ii }\int_{\mathbb{T}}
(\pa_{h}F\pa_{\bar{h}}G-\pa_{\bar{h}}F\pa_{{h}}G)dx.
\end{equation}
We give the following Definition.

\begin{definition}\label{admiOpComp}
Consider an operator $\mathcal{T}=\mathcal{T}(\vphi)$ 
acting on $L^2(\T^{\nu+1}; \mathbb{C}^2)$
of the form \eqref{forma-complessa}. We say that $\mathcal{T}(\vphi)$ is

\vspace{0.5em}
\noindent
$(i)$ {\bf Hamiltonian}
if it has the form
\begin{equation}\label{Ham1comp}
\mathcal{T}=\ii E \widetilde{\mathcal{T}}
\,, \qquad E=\sm{1}{0}{0}{-1}
\end{equation}
where 
$\widetilde{\mathcal{T}}$ has the form \eqref{forma-complessa}
and
\begin{equation}\label{self-adjT}
(\widetilde{\mathcal{T}}_{+}^{+})^{*}=\widetilde{\mathcal{T}}_{+}^{+}\,,
\qquad (\widetilde{\mathcal{T}}_{+}^{-})^{*}=\ov{\widetilde{\mathcal{T}}_{+}^{-}}
\end{equation}
where $(\mathcal{T}_{\s}^{\s'})^{*}_i$, $\s,\s'=\pm$ 
denotes the adjoint operator with respect 
to the complex scalar product of $L^{2}(\mathbb{T};\mathbb{C})$;

\vspace{0.5em}
\noindent
$(ii)$ {\bf $x$-translation invariant} if  (recall \eqref{funzmom})
\begin{equation}\label{traInvcomp}
\mathcal{T}\; :\; S_{\mathtt{v}}\to S_{\mathtt{v}}\,. 
\end{equation}

Consider a real operator $\mathcal{Q}=\mathcal{Q}(\vphi)$ 
acting on $L^{2}(\mathbb{T}^{\nu+1}; \mathbb{R}^{2})$
of the form \eqref{matrix1}. We say that $\mathcal{Q}$ is,
respectively, \emph{Hamiltonian}, 
%\emph{Even-to-even},
%\emph{Reversible}, \emph{Reversibility preserving}, 
\emph{$x$-translation invariant},
if the operator $\mathcal{T}=\mathcal{C}\mathcal{Q}\mathcal{C}^{-1}$ 
defined in \eqref{forma-complessa}
is, respectively,
Hamiltonian, 
%Even-to-even,
%Reversible, Reversibility preserving, 
$x$-translation invariant.
%according to Definition \ref{admiOpComp}.
We say that an operator $\mathcal{L}=\mathcal{L}(\vphi)$
of the form
\begin{equation}\label{ELLEOMEGA}
\mathcal{L}:=\omega\cdot\pa_{\vphi}+\mathcal{T}(\vphi)
\end{equation}
is respectively Hamiltonian, $x$-translation invariant
if $\mathcal{T}(\vphi)$ 
is respectively Hamiltonian, $x$-translation invariant.
Similar for $\mathcal{L}:=\omega\cdot\pa_{\vphi}+\mathcal{Q}(\vphi)$
where $\mathcal{Q}(\vphi)$ has the form \eqref{matrix1}.
\end{definition}
%We have the analogous of Lemma \ref{lem:momento}.

Let $\mathcal{T}=\mathcal{T}(\vphi)$ 
acting on $L^2(\T^{\nu+1}; \mathbb{C}^2)$
of the form \eqref{forma-complessa} be an Hamiltonian operator according to 
Definition \ref{admiOpComp}.
We associate to $\mathcal{T}$  a real valued  Hamiltonian 
function 
\begin{equation}\label{HamilDiTT}
\mathcal{H}(U)=\omega\cdot Q+\frac{1}{2}\int_{\mathbb{T}} 
\widetilde{\mathcal{T}}(\vphi)\vect{h}{\bar{h}} \cdot\vect{\bar{h}}{h}dx\,,
\quad\omega\in \mathbb{R}^{\nu}\,,
\quad U=(\vphi,Q,h,\bar{h})\in 
 \mathbb{T}^{\nu}\times\mathbb{R}^{\nu}
\times H^{s}(\mathbb{T};\mathbb{C})\times H^{s}(\mathbb{T};\mathbb{C})
\end{equation}
with $\widetilde{\mathcal{T}}$ as in \eqref{Ham1comp}, \eqref{self-adjT}.
Its Hamiltonian vector field $X_{\mathcal{H}}$, w.r.t. the symplectic 
form \eqref{formaestesa}, 
as the form
\[
\pa_{t}{\vphi}=\omega\,,\quad \pa_{t}{Q}=-\frac{1}{2}\pa_{\vphi}\mathcal{H}(U)\,,
\qquad
\pa_{t}\vect{h}{\bar{h}}=-\ii E\widetilde{\mathcal{T}}(\vphi)\vect{h}{\bar{h}}
=-\mathcal{T}(\vphi)\vect{h}{\bar{h}}\,.
\]
The second equation is not relevant for the dynamics since it is 
completely determined from the equation on the variables $(\vphi, h,\bar{h})$.
Moreover, setting $\vphi=\omega t$, $z(\omega t) := h(t)$, 
$z(\vphi)\in H^{s}(\mathbb{T}^{\nu+1};\mathbb{C})$,
we can write the third equation as
\[
0=\omega\cdot\pa_{\vphi}\vect{z}{\bar{z}}(\vphi)+\mathcal{T}(\vphi)\vect{z}{\bar{z}}(\vphi)
{=}\mathcal{L}(\vphi)\vect{z}{\bar{z}}(\vphi)\,,
\]
where $\mathcal{L}(\vphi)$ is the operator defined 
in \eqref{ELLEOMEGA} acting on functions in $H^{s}(\mathbb{T}^{\nu+1};\mathbb{C})$.
This justifies Definition \ref{admiOpComp}.
In the following we shall characterize Hamiltonian operators which are 
$x$-\emph{translation} invariant, namely \eqref{traInvcomp} holds,
by giving a condition on the associated Hamiltonian function $\mathcal{H}$
in \eqref{HamilDiTT}.
\begin{lemma}\label{funzioniSVVV}
The following conditions are equivalent:
\begin{itemize}
\item[(a)] the function $h=h(\vphi,x)$ belong to $S_{\mathtt{v}}$ in \eqref{funzmom};

\item[(b)] one has that
\begin{equation}\label{SVVV2}
(\mathtt{v}\cdot\pa_{\vphi}+\pa_{x})h=0\;;
\end{equation}

\item[(c)] $h(\vphi,x)=(2\pi)^{-1/2}\sum_{\ell,j}h_{\ell j}e^{\ii \ell\cdot\vphi+\ii jx}$ is such that
\begin{equation}\label{vascoblasco31}
{\rm if} \quad h_{\ell j}\neq0 \quad \Rightarrow \quad \mathtt{v}\cdot\ell+j=0\,,\;\;\forall \, 
\ell\in \mathbb{Z}^{\nu}\,, \;j\in \mathbb{Z}\,.
\end{equation}
\end{itemize}
\end{lemma}
\begin{proof}
$(a)\Rightarrow (b)$. 
If $h\in S_{\mathtt{v}}$ then we can write 
$h(\vphi,x)=G(\vphi-\mathtt{v}x)$ for some $G(\Theta)$, $\Theta\in \mathbb{T}^{\nu}$.
Hence $(\mathtt{v}\cdot\pa_{\vphi}+\pa_{x})h(\vphi,x)=
(\mathtt{v}\cdot\pa_{\Theta}G)(\vphi-\mathtt{v}x)-(\mathtt{v}\cdot\pa_{\Theta}G)(\vphi-\mathtt{v}x)=0$ which is the \eqref{SVVV2}.

\noindent
$(b)\Rightarrow (c)$.
By \eqref{SVVV2} and passing to the Fourier basis we have
\[
\sum_{\ell\in \mathbb{Z}^{\nu},j\in\mathbb{Z}}(\ii \mathtt{v}\cdot\ell+\ii j)
h_{\ell j}e^{\ii \ell\cdot\vphi+\ii jx}=0\,, 
\]
and hence we get the \eqref{vascoblasco31}. The implications $(c)\Rightarrow (b)$ and 
$(c)\Rightarrow(a)$ are  trivial.
\end{proof}

\begin{lemma}{\bf ($x$-translation invariant operators).}\label{lem:momentocomp}
Consider an operator $\widetilde{\mathcal{T}}$ 
acting on $L^2(\T^{\nu+1}; \mathbb{C}^2)$
of the form \eqref{forma-complessa}
satisfying \eqref{self-adjT}.
 Consider the Hamiltonian $\mathcal{H}$ in \eqref{HamilDiTT},
 the operator $\mathcal{T}:=\ii E\widetilde{\mathcal{T}}$ (see \eqref{Ham1comp})
 and 
the \emph{momentum} Hamiltonian 
 \begin{equation}\label{extMom0Comp}
\mathcal{M}=\mathcal{M}(\vphi,Q,h,\bar{h})=-\mathtt{v}\cdot Q+\frac{1}{2}
\int_{\mathbb{T}}\ii E \vect{h_x}{\bar{h_x}} \cdot\vect{\bar{h}}{h}dx\,.
\end{equation}
The following conditions are equivalent:
\begin{itemize}
\item[(a)]
the  Hamiltonian operator
$\mathcal{T}$ is $x-$translation invariant, namely satisfies \eqref{traInvcomp};

\item[(b)]
the Hamiltonians $\mathcal{H}, \mathcal{M}$ Poisson commutes, i.e.
\begin{equation}\label{extMomcomp}
\{\mathcal{M},\mathcal{H}\}_{e}=0\,,
%\mathtt{v}\cdot\pa_{\vphi}\mathcal{H}
%+\{\mathcal{H},\mathcal{M}\}=0\,,
\end{equation}
where $\{\cdot, \cdot\}_e$ are the Poisson brackets in \eqref{poissonEstese};

\item[(c)]
by identifying the operator $\widetilde{\mathcal{T}}$ 
with the matrix coefficients $(\widetilde{\mathcal{T}})_{\s,j}^{\s',k}(\ell)$, 
$\ell\in \mathbb{Z}^{\nu}, j,k\in \mathbb{Z}$, $\s,\s'=\pm$
(recall \eqref{actio}, \eqref{azioneSobsigma})
one has that
%we have that \eqref{extMomcomp} holds if and only if 
\begin{equation}\label{constcoeffMomentum}
(\widetilde{\mathcal{T}})_{\s,j}^{\s',k}(\ell)\neq0\quad \Rightarrow\quad
v\cdot\ell+ j-k=0\,, \quad \forall\, \ell\in \mathbb{Z}^{\nu}\,,\; j,k\in \mathbb{Z}\,,
\; \s,\s'=\pm\,.
\end{equation}
\end{itemize}
\end{lemma}

\begin{proof} 
Consider the condition \eqref{extMomcomp} in item $(b)$.
By using the \eqref{poissonEstese}, \eqref{HamilDiTT}, \eqref{extMom0Comp}
we have that
\[
\begin{aligned}
0=\big\{\mathcal{M}, \mathcal{H}\big\}_{e}&=
\Omega_{e}\big(X_{\mathcal{M}},X_{\mathcal{H}}\Big)
=\frac{1}{2}
\int_{\mathbb{T}}\Big( -\mathtt{v}\cdot\pa_{\vphi}\widetilde{\mathcal{T}}(\vphi)
+\widetilde{\mathcal{T}}(\vphi)\pa_{x}-\pa_{x}\widetilde{\mathcal{T}}(\vphi)
\Big)\vect{h}{\bar{h}}\cdot\vect{\bar{h}}{h}dx\,.
\end{aligned}
\]
The equation above is equivalent, passing to the Fourier representation, to
\[
\ii \big(-\mathtt{v}\cdot\ell+ k- j\big) (\widetilde{\mathcal{T}})_{\s,j}^{\s',k}=0\,,
\qquad \forall\, \ell\in \mathbb{Z}^{\nu}\,,\;\; j,k\in \mathbb{Z}\,, \s,\s'=\pm\,.
\]
This condition holds if and only if \eqref{constcoeffMomentum} is satisfied.
This proves $(b)\Leftrightarrow (c)$.\\
Consider the operator $\widetilde{\mathcal{T}}_{+}^{+}$ and let $h\in S_{\mathtt{v}}$.
We have
\[
\widetilde{\mathcal{T}}_{+}^{+}h=\sum_{\ell\in \mathbb{Z}^{\nu}, j\in \mathbb{Z}}
e^{\ii \ell\cdot\vphi+\ii j x}\Big(
\sum_{p\in \mathbb{Z}^{\nu}, k\in \mathbb{Z}}
(\widetilde{\mathcal{T}})_{+,j}^{+,k}(\ell-p)h_{pk}
\Big)\,.
\]
We also recall that $\mathtt{v}\cdot p+k=0$, for any $p,k$ since
$h\in S_{\mathtt{v}}$.
Then if $\widetilde{\mathcal{T}}_{+}^{+}h\in S_{\mathtt{v}}$
we must have (recall \eqref{vascoblasco31}) $\mathtt{v}\cdot\ell+j=0$
which implies \eqref{constcoeffMomentum}
on the coefficients $(\widetilde{\mathcal{T}})_{+,j}^{+,k}(\ell-p)$.
This proves $(a)\Rightarrow (c)$. The converse is similar.
%Then we proved $(b)\Leftrightarrow (c)\Leftrightarrow (a)$. 
\end{proof}

\begin{lemma}\label{coeffFouHamilto}
Let ${\mathcal{A}}=(\mathcal{A}_{\s}^{\s'})_{\s,\s'}$ 
be an operator of the form \eqref{forma-complessa}
satisfying \eqref{self-adjT}, i.e. $(\mathcal{A}_{\s}^{\s})^{*}=
\mathcal{A}_{\s}^{\s}$ and 
$(\mathcal{A}_{\s}^{-\s})^{*}=\ov{\mathcal{A}_{\s}^{-\s}}$,
and let 
$\big(({\mathcal{A}})_{\s,j}^{\s',k}(\ell)\big)$, 
$i=1,2$, $\ell\in\mathbb{Z^{\nu}}$, $j,k\in \mathbb{Z}$ be
 the matrices representing the operators ${\mathcal{A}}_{\s}^{\s'}$.
Then
\begin{align}
&\ov{{\mathcal{A}_{\s}^{\s'}}}:=\Big((\ov{\mathcal{A}})_{\s,j}^{\s',k}(\ell)\Big)\,, \quad (\ov{\mathcal{A}})_{\s,j}^{\s',k}(\ell)=
\ov{({\mathcal{A}})_{\s,-j}^{\s',-k}(-\ell)}\,,\label{ruben1}\\
&({\mathcal{A}_{\s}^{\s'}})^{*}:=(\ov{\mathcal{A}_{\s}^{\s'}})^{T}:=
\Big(({\mathcal{A}}^{*})_{\s,j}^{\s',k}(\ell)\Big)\,, \quad ({\mathcal{A}}^{*})_{\s,j}^{\s',k}(\ell)=
\ov{(\mathcal{A})_{\s,k}^{\s',j}(-\ell)}\,,\label{ruben2}
\end{align}
In particular, since $\mathcal{A}$ is self-adjoint (see \eqref{self-adjT}) and real-to real, one has
\begin{equation}\label{ruben3}
(\mathcal{A})_{\s,j}^{\s,k}(\ell)=\ov{(\mathcal{A})_{\s,k}^{\s,j}(-\ell)}\,,
\qquad 
(\mathcal{A})_{\s,-j}^{-\s,-k}(\ell)=
{\big((\mathcal{A}_{\s}^{-\s})^T\big)_{k}^{j}(\ell)}={(\mathcal{A})_{\s,k}^{-\s,j}(\ell)}
\end{equation}
and
\begin{equation}
\mathcal{A}_{-\s, j}^{-\s', k}(\ell)=\ov{({\mathcal{A}})_{\s,-j}^{\s',-k}(-\ell)}
\end{equation}
\end{lemma}
\begin{proof}
We prove \eqref{ruben1}, \eqref{ruben2} for $\mathcal{A}_1:=\mathcal{A}_{+}^{+}$.
Consider a function $h=(\sqrt{2\pi})^{-1}\sum_{k}h_ke^{\ii kx}$. Then
\[
\begin{aligned}
\ov{\mathcal{A}_1(\vphi)}[h]&\stackrel{\eqref{conjugate-Op}}{:=}
\ov{\mathcal{A}_1(\vphi)[\bar{h}]}=\ov{\sum_{j}e^{\ii jx} \sum_{k,\ell}e^{\ii\ell\cdot\vphi}
(\mathcal{A}_1)_{j}^{k}(\ell)\ov{h_{-k}}}
=\sum_{j,k,\ell}\ov{(\mathcal{A}_1)_{-j}^{-k}(-\ell)}h_{k}e^{\ii jx}e^{\ii \ell\cdot\vphi}\,.
\end{aligned}
\]
This implies the \eqref{ruben1}.  %Recalling the \eqref{prodottoL2}, 
Moreover we have
\[
\begin{aligned}
\sum_{j,k,\ell}e^{\ii \ell\cdot\vphi} (\mathcal{T}_{1})_{j}^{k}(\ell) h_{k}\ov{v_j}&
=(\mathcal{T}_{1}(\vphi)h,v)_{L^2}=
(h,\mathcal{T}_{1}^{*}(\vphi)v)_{L^2}=\int_{\mathbb{T}}h\cdot \ov{\mathcal{T}_{1}^*(\vphi)v}dx=\sum_{j,k,\ell} \ov{ (\mathcal{T}_{1}^{*})_{k}^{j}(-\ell)}h_{k}\ov{v_{j}}e^{\ii \ell\cdot\vphi}
\end{aligned}
\]
This implies the \eqref{ruben2}. Then the \eqref{ruben3} follows
by conditions \eqref{self-adjT}.
\end{proof}

\begin{lemma}{\bf ($x$-translation invariant symbols).}\label{lem:momentoSimbolo}
Consider a symbol $a(\vphi,x,\x)$ in $S^{m}$, $m\in \mathbb{R}$.
We have that $a$ has the form
\begin{equation}\label{space3Xinv}
a(\vphi,x,\x)=A(\vphi-\mathtt{v}x,\x)
\end{equation}
for some $A(\Theta,\x)$, $\Theta\in \mathbb{T}^{\nu},\x\in \mathbb{R}$ satisfying 
\begin{equation}\label{space3Xinv2}
|\pa_{\Theta}^{\al}\pa_{\x}^{\beta}A(\Theta,\x)|\leq C_{\al,\beta}\langle \x\rangle^{m-\beta}\,,
\;\;\forall \; \al,\beta\in \mathbb{N}\,
\end{equation} 
if and only if
\begin{equation}\label{883}
\{\mathcal{M},\mathcal{H}\}_{e}=0\,,\qquad
\mathcal{H}(\vphi,Q,h,\bar{h})=\omega\cdot Q+\int_{\mathbb{T}} 
\opw(a(\vphi,x,\x)) h \cdot\bar{h}dx\,,
\end{equation}
\end{lemma}

\begin{proof}
By \eqref{poissonEstese}, \eqref{883}, \eqref{extMom0Comp} we have
\[
0=\int_{\mathbb{T}}\opw\Big(-(\mathtt{v}\cdot\pa_{\vphi}a)(\vphi,x,\x)
+\frac{1}{\ii }\{a(\vphi,x,\x),\ii\x\}\Big) h\cdot\bar{h}dx
=\int_{\mathbb{T}}\opw\Big(-(\mathtt{v}\cdot\pa_{\vphi}a)(\vphi,x,\x)-(\pa_{x}a)(\vphi,x,\x)
\Big) h\cdot\bar{h}dx\,.
\]
Therefore we must have
\[
0=\Big(- \mathtt{v}\cdot\pa_{\vphi}a-\pa_{x}a\Big)(\vphi,x,\x)
=\sum_{\ell\in\mathbb{Z}^{\nu}, j\in\mathbb{Z}} 
(-\ii) (\mathtt{v}\cdot\ell+j)\tilde{a}(\ell,j,\x)e^{\ii\ell\cdot\vphi+\ii jx}
\]
where $\widetilde{a}(\ell,j,\x)$ denotes the Fourier transform of $a(\vphi,x,\x)$
in the variables $(\vphi,x)\in \mathbb{T}^{\nu}\times\mathbb{T}$.
Then the \eqref{883}  is verified if and only if 
$\mathtt{v}\cdot\ell+j=0$ for any $\ell\in \mathbb{Z}^{\nu}, j\in\mathbb{Z}$.
Then we can write the symbol $a(\vphi,x,\x)$ as in \eqref{space3Xinv}.
The bound \eqref{space3Xinv2} follows by the estimates \eqref{space3}
on $a(\vphi,x,\x)$.
%The result follows by following the ideas used in the proof of Lemma
%\ref{lem:momentocomp} and using \eqref{weylquanti}.
\end{proof}

If condition \eqref{space3Xinv} holds we shall say that the symbol 
$a(\vphi,x,\x)$ is \emph{$x$-translation invariant}. 

\begin{lemma}{\bf (Real-to-real/Self-adjoint matrices of symbols).}\label{realHamMtrixSimbo}
Consider the  operator  (recall \eqref{matrixSym})
\[
L=\opw(A(\vphi,x,\x))\,, \qquad A\in S^{m}\otimes\mathcal{M}_{2}(\mathbb{C})\,.
\]
We have that $L$ is \emph{real-to-real} according to Definition \ref{realtoreal}
if and only if the matrix $A$ has the form
\begin{equation}\label{retoreMat}
A(\vphi,x,\x):=\left(\begin{matrix}a(\vphi,x,\x) & b(\vphi,x,\x)\\
\ov{b(\vphi,x,-\x)} & \ov{a(\vphi,x,-\x)}
\end{matrix}
\right)\,,
\end{equation}
and $L$ is \emph{self-adjoint}, i.e. satisfies \eqref{self-adjT},
if and only if the matrix of symbols $A$ satisfies 
\begin{equation}\label{selfSimbo}
\ov{(a(\vphi,x,\x))}=a(\vphi,x,\x)\,, \qquad b(\vphi,x,-\x)=b(\vphi,x,\x)\,.
\end{equation}
Finally $L$ is \emph{Hamiltonian}, i.e. satisfies \eqref{Ham1comp} if and only if
the matrix $A$ 
satisfies 
\begin{equation}\label{selfSimbo2}
\ov{(a(\vphi,x,\x))}=-a(\vphi,x,\x)\,, \qquad b(\vphi,x,-\x)=b(\vphi,x,\x)\,.
\end{equation}
\end{lemma}

\begin{proof}
The Lemma follows by usign formul\ae \, \eqref{forma-complessa},
 \eqref{pseudobarra}. % and \eqref{pseudoadj}.
\end{proof}

\noindent
We conclude this section with the definition of a special class of operator
we shall use in sections \ref{sec:Good}-\ref{ordiniNegativi}.
\begin{definition}\label{compaMulti}
We say that a linear operator $\mathcal{T}(\vphi)\in 
\mathcal{L}(L^{2}(\mathbb{T}^{\nu+1};\mathbb{C}^{2}))$
belongs to the class $\mathfrak{S}_0$ if  it is
 \begin{itemize}

\item[(i)] \emph{real-to-real}, i.e. satisfies \eqref{forma-complessa};

\item[(ii)] 
\emph{Hamiltonian}, i.e. satisfies  \eqref{Ham1comp};

\item[(iii)] \emph{$x$-translation invariant}, 
i.e. satisfies the \eqref{traInvcomp} 
with $\mathtt{v}$ in \eqref{def:m1}.
\end{itemize}

We say that a linear operator of the form
\begin{equation*}
\mathcal{L}=\omega\cdot\pa_{\vphi}+\opw(A(\vphi,x,\x))+\mathcal{R}
\end{equation*}
for some $A\in S^{m}\otimes\mathcal{M}_{2}(\mathbb{C})$ and 
$\mathcal{R}\in \gotL_{\rho,p}\otimes\mathcal{M}_{2}(\mathbb{C})$
belongs to $\mathfrak{S}_{1}$ if
$\mathcal{L}$ and  $\opw(A(\vphi,x,\x))$ belong to $\mathfrak{S}_0$.
\end{definition}

\begin{remark}
Consider $\mathcal{L}\in \mathfrak{S}_{1}$. By Lemmata 
\ref{lem:momentoSimbolo}, \ref{realHamMtrixSimbo}
we deduce that the operator $A(\vphi,x,\x)$ 
  satisfies \eqref{retoreMat}, \eqref{selfSimbo2} and 
 \eqref{space3Xinv}.
 Moreover we deduce that
 also the remainder 
$\mathcal{R}$ belongs to $\mathfrak{S}_0$.
\end{remark}

In the following  we shall 
look for transformations 
of coordinates 
which preserves the structure 
the class $\mathfrak{S}_{1}$. %, $i=1,2,3$.
In particular we give the following definition.

\begin{definition}\label{goodMulti}
We say that
a map $\Phi=\Phi(\vphi)$  
belongs to $\mathfrak{T}_{1}$
if it is \emph{symplectic} and $x$-translation invariant.
\end{definition}

%%%%%%%%%%%%%%%%%%%%%%%%%%%%%%%%%%%%%%%
%%%%%%%%%%%%%%%%%%%%%%%%%%%%%%%%%%%%%%%
%%%%%%%%%%%%%%%%%%%%%%%%%%%%%%%%%%%%%%%
%%%%%%%%%%%%%%%%%%%%%%%%%%%%%%%%%%%%%%%

\section{Normal forms and integrability properties of the pure gravity water waves}\label{sec:integrability}

In this section we recall some properties of 
the water waves system \eqref{eq:113},
we discuss its Hamiltonian structure and normal form. 
In particular we focus on the formal integrability of the 
Hamiltonian at order four which has been proved in \cite{CW}, \cite{CS}, \cite{Zak2}. 
Due to the quasi-linear nature of the water waves equations 
the Birkhoff normal form procedure turns out to be not well defined. 
To overcome this problem we perform a  "weaker" but  \emph{rigorous} 
normal form algorithm: first we normalize the dynamics 
on a finite dimensional subspace to find an approximately invariant torus; 
secondly we normalize the dynamics in the normal 
directions around the embedded torus. 
These procedures are discussed in details in section \ref{compaBNFpro}.
In sections \ref{ApproxCons}, \ref{sec:3.33} 
we prove, thanks to an argument of identification of normal forms, 
that the formal (approximate) integrability implies 
the (approximate) integrability of the linearized problem at the torus.

\subsection{Hamiltonian structure of water waves}
Let $W_1=\vect{\eta_1}{\psi_1}$, 
$W_2=\vect{\eta_2}{\psi_2}$. We
consider the 
 the symplectic form
\begin{equation}\label{symReal}
\widetilde{\Omega}(W_1,W_2)
:=\int_{\mathbb{T}}W_1
\cdot J^{-1}W_2dx=
\int_{\mathbb{T}}-(\eta_1\psi_2-\psi_1\eta_2)dx\,,
\quad
J=\sm{0}{1}{-1}{0}\,.
\end{equation}
The vector field (see \eqref{HS})
\begin{equation*}%\label{realvecWW}
X(\eta,\psi)=X_{H}(\eta,\psi)=J \nabla H(\eta,\psi)
=\left(\begin{matrix}\nabla_{\psi}H(\eta,\psi) \\ 
-\nabla_{\eta} H(\eta,\psi) \end{matrix}\right)
\end{equation*}
 is the Hamiltonian 
vector field of $H$ in \eqref{Hamiltonian} w.r.t. the symplectic form
\eqref{symReal}, i.e.
\[
-\widetilde{\Omega}(X_{H}(\eta,\psi), h)=dH((\eta,\psi))[h]\,\quad 
h=\vect{\hat{\eta}}{\hat{\psi}}\,,
\]
while the Poisson  bracket between functions 
$F(\eta,\psi), H(\eta,\psi)$
%in \eqref{poissonBra}
are defined as
\begin{equation}\label{poissonBra22}
\{F,H\}=\widetilde{\Omega}(X_{F},X_{H})
=\int_{\T}\big(\nabla_{\eta}H\nabla_{\psi}F-\nabla_{\psi}H\nabla_{\eta}F\big)dx\,.
\end{equation}
Consider the maps
\begin{equation}\label{matriciCompVar}
\mathcal{C}:=\frac{1}{\sqrt{2}}
\sm{1}{\ii }{1}{-\ii}
\,,\qquad
%\mathcal{C}^{-1}:=
%\frac{1}{\sqrt{2}}
%\sm{1}{1}{-\ii }{\ii}\,, \qquad
\mathfrak{F}:=
\sm{|D|^{-\frac{1}{4}}}{0}{0}{|D|^{\frac{1}{4}}}\,,
\qquad \Lambda:=\mathcal{C}\circ \mathfrak{F}:=\frac{1}{\sqrt{2}}
\sm{|D|^{-\frac{1}{4}}}{ \ii |D|^{\frac{1}{4}}}{ |D|^{-\frac{1}{4}}}{-\ii |D|^{\frac{1}{4}}}\,.
\end{equation}
We note that the operator $|D|^{-1/4}$ is well defined on the space of functions with zero spatial average.
We introduce the complex symplectic variables 
\begin{equation}\label{CVWW}
\!\!\!\!\!\!\left(\begin{matrix}
u  \\
\overline{u}
\end{matrix} 
\right) = 
\Lambda 
\left(\begin{matrix}
\eta \\\psi\end{matrix}\right)
 : =\frac{1}{\sqrt{2}} 
\left(\begin{matrix}
|D|^{-\frac{1}{4}}\eta+ \ii |D|^{\frac{1}{4}}\psi   \\
|D|^{-\frac{1}{4}}\eta - \ii |D|^{\frac{1}{4}}\psi  
\end{matrix} 
\right) \, , 
\qquad 
\left(\begin{matrix}\eta \\ \psi\end{matrix}\right) 
= \Lambda^{-1} 
\left(\begin{matrix}
u \\ \bar u \end{matrix}\right) = 
\frac{1}{\sqrt{2}} 
\left(\begin{matrix}
|D|^{\frac14}( u + \bar u ) \\
 - \ii  |D|^{- \frac14}(  u - \bar u )
\end{matrix}\right).
\end{equation}
The symplectic form in \eqref{symReal} transforms, for 
$U=\vect{u}{\bar{u}}$, $V=\vect{v}{\bar{v}}$, into
\begin{equation}\label{symComp}
\Omega(U,V):=-\int_{\mathbb{T}}  U \cdot \ii J V dx=-\int_{\mathbb{T}}
\ii (u\bar{v}-\bar{u}v)dx\,.
%=-2{\rm Re}(\ii u,v)_{L^{2}}\,,
\end{equation}
%where $(\cdot,\cdot)_{L^2}$ is the standard complex $L^2$-scalar product, i.e.
%\begin{equation}\label{prodottoL2}
%(u,v)_{L^{2}}:=\int_{\mathbb{T}}u\,\ov{v}dx\,.
%\end{equation}
The Poisson bracket in \eqref{poissonBra22} 
assumes the form 
\begin{equation}\label{poissonBraComp}
\{F,H\}=\frac{1}{\ii}
\int_{\mathbb{T}}(\nabla_{u}H\nabla_{\bar{u}}F-\nabla_{\bar{u}}H\nabla_{{u}}F)dx
=\frac{1}{\ii}\sum_{k\in \Z\setminus\{0\}}
  \big(\pa_{u_{k}}H\pa_{\bar{u_k}}F - \pa_{\bar{u_k}}H\pa_{u_k}F\big) \, ,
\end{equation}
where 
%$\nabla_{u}:=\frac{1}{\sqrt{2}}(|D|^{1/4}\nabla_{\eta}-\ii \nabla_{\psi})$, 
%$\nabla_{\bar{u}}:=\frac{1}{\sqrt{2}}(\nabla_{\eta}+\ii \nabla_{\psi})$ and
\begin{equation*}%\label{poissonBraComp22}
\pa_{u_{k}}=\frac{1}{\sqrt{2}}\big(|k|^{-\frac{1}{4}}\pa_{\eta_{k}}
-\ii |k|^{\frac{1}{4}}\pa_{\psi_{k}} \big)\,,
\qquad
\pa_{\ov{u_{-k}}}=
\frac{1}{\sqrt{2}}\big(|k|^{-\frac{1}{4}}\pa_{\eta_{k}}+\ii |k|^{\frac{1}{4}}\pa_{\psi_{k}} \big)\,.
\end{equation*}
In these coordinates the vector field $X=X_{H}$ in \eqref{eq:113}  
assumes the form
(setting $H_{\mathbb{C}}:=H\circ\Lambda^{-1}$)
\begin{equation}\label{compvecWW}
X^{\mathbb{C}}:=X_{H_{\C}}=
 \left(
\begin{matrix}
-\ii\pa_{\bar{u}}H_{\mathbb{C}} \\
\ii\pa_{u}H_{\mathbb{C}}
\end{matrix}
\right)
=
\frac{1}{\sqrt{2 \pi}} \sum_{k\in\Z\setminus\{0\}} 
  \left(\begin{matrix}
- \ii  \pa_{ \ov{u_k}} H_{\C} \, e^{\ii k x}  \\
 \ii  \pa_{u_k} H_{\C}   \, e^{- \ii k x} 
\end{matrix} 
\right) \, ,
\end{equation}
that we also identify, using the  standard vector field notation, with 
\begin{equation*}%\label{HSuk}
X^{\mathbb{C}} 
=\sum_{k\in\Z\setminus\{0\}, \sigma=\pm} 
-\ii \s \pa_{u^{-\s}_{k}} H_\C  \, \pa_{u^{\s}_{k}} \, . 
\end{equation*}
The Hamiltonian of the momentum (recall \eqref{Hammomento}) reads as
\begin{equation}\label{HamMom2}
M_{\mathbb{C}}:=\mathtt{M}\circ\Lambda^{-1}=\int_{\mathbb{T}} \mathrm{i}\, u_x\cdot\bar{u} dx\,.
\end{equation}
%\end{itemize}

\noindent
{\bf Taylor expansion at the origin.}
Consider
the Dirichlet-Neumann operator $G(\eta)$  in \eqref{eq:112a}.
The map $(\eta,\psi)\to G(\eta)\psi$ is linear with respect to $\psi$
and nonlinear with respect to the profile $\eta$.
The derivative with respect to $\eta$ (which is called ``shape derivative'')
is given by the formula (see for instance \cite{Lan1}) %, \cite{Lan2})
\begin{equation}\label{shapeDer}
G'(\eta)[\hat{\eta}]\psi=\lim_{\epsilon\to0} \frac{\big(G(\eta+\epsilon\hat{\eta})\psi-
G(\eta)\psi
\big)}{\varepsilon}=-G(\eta)(B\hat{\eta})-\pa_{x}\big(V\hat{\eta}\big)\,,
\end{equation}
where we denoted 
the horizontal and vertical components 
of the velocity field at the free interface by
\begin{align} 
\label{def:V}
& V =  V (\eta, \psi) :=  (\pa_x \Phi) (x, \eta(x)) = \psi_x - \eta_x B \, , 
\\
\label{form-of-B}
& B =  B(\eta, \psi) := (\pa_y \Phi) (x, \eta(x)) =  \frac{G(\eta) \psi + \eta_x \psi_x}{ 1 + \eta_x^2} \, .
\end{align}
It is also known that
%The Dirichlet-Neumann operator 
$\eta\mapsto G(\eta)$ % in \eqref{eq:112a}
is analytic and admits
%in a neighborhood of Lipschitz functions $\eta$ (see \cite{CoMe}), 
% validating 
 the Taylor expansion of the Hamiltonian near equilibrium
 $(\eta=0, \psi=0)$
 \begin{equation}\label{DNexp3bis}
 G(\eta)=\sum_{m=0}^{\infty}G^{(m)}(\eta)\,.
 \end{equation}
Each term $G^{(m)}$
in the Taylor expansion is homogeneous of degree $m$ and can be computed
explicitly. For instance  (see e.g.  formula (2.5) of \cite{CS}) 
we have
\begin{equation}\label{DNexp3}
\begin{aligned}
G^{(0)}&:=|D|\,, \qquad G^{(1)}(\eta):= - \pa_x \eta \pa_x  - |D|  \eta |D|, \\
G^{(2)} (\eta) & :=  - \frac12 \Big(  D^2 \eta^2 |D| 
+ |D| \eta^2 D^2 - 2 |D| \eta |D| \eta |D| \Big)  \,.
\end{aligned}
\end{equation}
For further properties about the Dirichlet-Neumann operator we refer the reader to
Appendix \ref{propDirNeu}.\\
The Hamiltonian $H$  in \eqref{Hamiltonian} 
has a convergent Taylor expansion
\begin{equation*}%\label{Hamtay}
H=H^{(2)}+H^{(3)}+H^{(4)}+\ldots+H^{(m)}+R^{(m+1)}\,.
\end{equation*}
In the complex coordinates given 
by \eqref{matriciCompVar}-\eqref{CVWW} the Hamiltonian reads as
\begin{equation}\label{Rinco}
\begin{aligned}
H_{\mathbb{C}}:= H\circ \Lambda^{-1}&=\sum_{p=2}^{N}H^{(p)}_{\mathbb{C}}
+H^{(\geq N+1)}_{\mathbb{C}}\, , \\
H^{(2)}_\C& = \sum_{j \in \Z\setminus\{0\}} \sqrt{|j|}  u_j \bar{u_j} \,,
\qquad
H^{(p)}_\C = 
\sum_{\s_1 j_1  +\ldots+ \s_p j_p = 0 } H_{j_1, \ldots, j_p}^{\s_1, \ldots, \s_p} 
u_{j_1}^{\s_1} \cdots u_{j_p}^{\s_p}
\end{aligned}
\end{equation}
where $N\geq 2$, $H_{j_1, \ldots, j_p}^{\s_1, \ldots, \s_p} \in \mathbb{C}$ 
and the  Hamiltonian  
$H^{(\geq N+1)}_\C $ collects all the monomials
of homogeneity $\geq N+1$.
%The restriction on the indeces in the sum in \eqref{Rinco}
%is a consequence of the conservation of momentum \eqref{HamMom2}.
%greater or equal $N+1$. 

\noindent 
\textbf{Splitting of the phase space.}
Recall $S$ in \eqref{TangentialSitesDP} and define 
$S^{c}:=\mathbb{Z}\setminus\big(S\cup\{0\}\big)$.
We decompose the phase space as 
\begin{equation}\label{decomposition}
\begin{aligned}
 H_0^1(\mathbb{T})\times H_0^1(\mathbb{T}) &:=H_S\oplus H_S^{\perp}\,, \\ 
 H_S:=\Big\{
 \begin{pmatrix} \eta_{S} \\ \psi_{S} \end{pmatrix}
 :=\sum_{j\in S} \begin{pmatrix} \eta_j\\ \psi_j \end{pmatrix} 
 e^{\mathrm{i}\,j\,x}  \Big\}\,,
 & \quad H_S^{\perp}:=
 \Big\{ 
 \begin{pmatrix} \tilde{\eta} \\ \tilde{\psi} \end{pmatrix}
 := \sum_{j\in S^c} \begin{pmatrix} \eta_j\\ \psi_j \end{pmatrix} 
 e^{\mathrm{i}\,j\,x} \Big\}\,,%\nonumber
\end{aligned}
\end{equation}
and we denote by $\Pi_S, \Pi_S^{\perp}$ the corresponding orthogonal projectors. 
The subspaces $H_S$ and $H_S^{\perp}$ are symplectic 
orthogonal respect to the $2$-form 
$\Omega$ (see \eqref{symComp}). 
We use the following notations for the complex variables, 
\begin{equation}\label{splitto2}
\left(\begin{matrix}
\eta \\ \psi
\end{matrix}\right)=
\left(\begin{matrix}
\eta_{S} \\ \psi_{S}
\end{matrix}\right)
+\left(\begin{matrix}
\widetilde{\eta} \\ \widetilde{\psi}
\end{matrix}\right)\,,\qquad
\Lambda
\left(\begin{matrix}
\eta_{S} \\ \psi_{S}
\end{matrix}\right)=:
\left(\begin{matrix}
v \\ \ov{v}
\end{matrix}\right)\,,\qquad
\Lambda
\left(\begin{matrix}
\widetilde{\eta} \\ \widetilde{\psi}
\end{matrix}\right)=:
\left(\begin{matrix}
z \\ \ov{z}
\end{matrix}\right)\,.
\end{equation}
The notation 
$R_k(v^{\alpha_1}\, \bar{v}^{\beta_1}\, z^{\alpha_2}\,\bar{z}^{\beta_2})$
indicates a homogeneous polynomial of degree 
$k$ in $(v, \bar{v}, z, \bar{z})$ of the form
\begin{equation}\label{notaTRA}
R_k(v^{\alpha_1}\, \bar{v}^{\beta_1}\, z^{\alpha_2}\,\bar{z}^{\beta_2})
=M[\underbrace{v, \dots, v}_{\alpha_1-times}, 
\underbrace{\bar{v}, \dots, \bar{v}}_{\beta_1-times}, 
\underbrace{z, \dots, z}_{\alpha_2-times}, 
\underbrace{\bar{z}, \dots, \bar{z}}_{\beta_2-times}\,]\,,
\end{equation}
\[
M=k\mbox{-linear}\, \quad k:=\alpha_1+\alpha_2+\beta_1+\beta_2.
\]
We denote with $H^{(n, \geq k)}, H^{(n, k)}, H^{(n, \le k)}$ 
the terms of type 
$R_n(v^{\alpha_1}\, \bar{v}^{\beta_1}\, z^{\alpha_2}\,\bar{z}^{\beta_2})$, 
where, respectively, 
$\alpha_2+\beta_2\geq k, \alpha_2+\beta_2=k, \alpha_2+\beta_2\le k$, 
that appear in the homogeneous polynomial 
$H^{(n)}$ of degree $n$ in the variables $(v, \bar{v}, z, \bar{z})$.
We denote by
$\Pi^{d_z\leq k}$, respectively  $\Pi^{d_z=k}$,
 the projector of a homogeneous Hamiltonian of degree $n$ 
 on the monomials with degree less or equal than $k$, 
 respectively equal $k$,
 in the normal variable $z$, i.e. 
\begin{equation*}
\Pi^{d_z \leq k} H^{(n)}:=H^{(n, \leq k)}\,, 
\quad \Pi^{d_z = k} H^{(n)}:=H^{(n, k)}\,.
\end{equation*}

\smallskip
\noindent
Given  two Hamiltonians $Q, K$
we shall define the adjoint action  ${\rm ad}_{Q}(H):=\{Q,H\}$. 
We denote
by $\Pi_{\mbox{Ker}(Q)}$,   the projection on the 
kernel of the adjoint action. 
We define the projector on the range of the adjoint action as 
$\Pi_{\mbox{Rg}(Q)}:=\mathrm{I}-\Pi_{{\mbox{Ker}(Q)}}$.

\subsection{Comparison among Birkhoff normal form procedures}\label{compaBNFpro}
In this section we present three kinds of Birkhoff normal form procedures and we compare them.  All these procedures are \emph{formal} and they differ by the terms of the Hamiltonian that are supposed to be normalized along the process.\\
In this paper we shall implement a, not just formally defined, modification of the third one, the ``Weak'' plus ``Linear'' Bikhoff normal form.

\medskip

{In the following we use the notations introduced above related to the Hamiltonian of the pure gravity water waves system \eqref{Rinco}, but we point out that the procedures that we describe below apply to general analytic Hamiltonians close to an elliptic fixed point of the form $H=H^{(2)}+H^{(3)}+\dots$ commuting with momentum and such that there are no $3$-waves resonant interactions.} 

\smallskip

\noindent
{\bf ``Full'' Bikhoff normal form.} 
We refer to  (formal) full Birkhoff normal form method as the
 normalization 
procedure of the 
cubic and quartic terms $H^{(3)}_{\mathbb{C}}$ and $H^{(4)}_{\mathbb{C}}$ 
of the 
 Hamiltonian  in \eqref{Rinco}. This is the strongest normal form method and, in the PDEs context, it is usually hard to implement. {Indeed the normalizing transformations are constructed as time-one flow maps of certain nonlinear PDEs whose vector field may lose derivatives, making the equations possibly ill-posed. There are two main sources of this loss of regularity:
 \begin{enumerate}
 \item The presence of small divisors - in the full procedure one has to get lower bounds for infinitely many combinations of the linear eigenvalues;
 \item The unboundness of the vector field of the PDE that we want to normalize.
 \end{enumerate}
 In the pure gravity case with infinite depth one has to deal with both issues.}
The normalization is done by applying the change of 
 coordinates 
 $\Phi_{FB}:=\Phi_{F_3}\circ\Phi_{F_4}$
 where $\Phi_{F_i}$ are the time one (formal) flow map generated by the Hamiltonians
\begin{equation}\label{F3F4}
F_3:=ad_{H_{\mathbb{C}}^{(2)}}^{-1} H^{(3)}\,, \qquad 
F_4:=ad_{H_{\mathbb{C}}^{(2)}}^{-1} 
\Pi_{Rg(H^{(2)}_{\mathbb{C}})}\Big( H^{(4)}_{\mathbb{C}}
+\frac{1}{2}\{ F_3, H^{(3)}_{\mathbb{C}} \} \Big)\,.
\end{equation}
{By using the fact that there are no $3$-waves interactions} it is easy to check (using Lie series) 
that one obtains
\begin{align}
&H_{\mathbb{C}}\circ\Phi_{FB}=H_{\mathbb{C}}^{(2)}+H_{FB}^{(4)}+
{\rm quintic\; terms}\,,\label{larubia}\\
&H_{FB}^{(4)}:=\Pi_{Ker(H^2_{\mathbb{C}})}\Big( H^{(4)}_{\mathbb{C}}+\frac{1}{2}\{ F_3, H^{(3)}_{\mathbb{C}} \} \Big)\,.\label{larubia2}
\end{align}

\noindent
{\bf ``Partial'' Bikhoff normal form.} 
We refer to  (formal) partial Birkhoff normal form method as the
 normalization 
procedure of the 
cubic and quartic terms 
with at most two wave numbers  outside $S$, i.e.
the terms
$H^{(3,\leq 2)}_{\mathbb{C}}$ and $H^{(4,\leq2)}_{\mathbb{C}}$ 
of the 
 Hamiltonian  in \eqref{Rinco}. Such method has been successfully applied for many KAM results for \emph{semilinear} PDEs, starting from the pioneering work \cite{KP}. Roughly speaking it is the optimal normal form procedure that provides a control on the tangent bundle of the expected invariant torus. However, by the quasi-linear nature of the equations, this method cannot be applied for the water waves system (see discussions below for more details). \\
The normalization procedure is done by applying the change of 
 coordinates 
 $\Phi_{PB}:=\Phi_{F_{3}^{(3,\leq2)}}\circ\Phi_{\tilde{F}^{(4,\leq2)}_4}$
 where $\Phi_{F_3^{(3,\leq2)}}$, $\Phi_{\tilde{F}_{4}^{(4,2)}}$ 
 are the time one (formal) flow map generated by the Hamiltonians
\begin{equation}\label{F3F4partial}
\begin{aligned}
&F_3^{(3,\leq2)}:=ad_{H_{\mathbb{C}}^{(2)}}^{-1} H^{(3,\leq2)}\,, \qquad 
\tilde{F}^{(4,\leq2)}_4:=ad_{H_{\mathbb{C}}^{(2)}}^{-1} 
\Pi_{Rg(H^{(2)}_{\mathbb{C}})}
\Pi^{d_z\leq 2} H_1^{(4)}\,,\\
&\qquad 
H_{1}^{(4)}:=H^{(4)}_{\mathbb{C}}
+\frac{1}{2}\{ F_3^{(3,\leq2)}, H^{(3,\leq2)}_{\mathbb{C}} \}+
\{ F_3^{(3,\leq2)}, H^{(3,3)}_{\mathbb{C}} \}\,.
%\Big( H^{(4)}_{\mathbb{C}}
%+\frac{1}{2}\{ F_3, H^{(3)}_{\mathbb{C}} \} \Big)\,.
\end{aligned}
\end{equation}
It is easy to check (using Lie series) 
that one obtains
\begin{equation*}%\label{larubiaPartial}
H_{\mathbb{C}}\circ\Phi_{PB}=H_{\mathbb{C}}^{(2)}+H_{\mathbb{C}}^{(3,3)}
+H_{PB}^{(4)}+
H_1^{(4,\geq3)}
+
{\rm quintic\; terms}\,,
\end{equation*}
\begin{equation}\label{larubia2Partial}
\begin{aligned}
H_{PB}^{(4)}&:=\Pi_{Ker(H^2_{\mathbb{C}})}\Pi^{d_z\leq2} H_1^{(4)}\,.
 \end{aligned}
\end{equation}

\noindent
{\bf ``Weak'' plus ``Linear'' Bikhoff normal form.} 
In this case the normalization procedure is split into two steps: (i) a weak BNF, which is always rigorously defined if the Hamiltonian terms that we want to normalize commute with the momentum (this assumption holds naturally for many physical PDE models); (ii) a linear BNF, which is defined just at the formal level, which aims to normalize the tangent bundle of the expected invariant torus.\\
(i) We first define the map
$\Phi_{WB}:=\Phi_{F_3^{(3,\leq1)}}\circ\Phi_{\widehat{F}_4^{(4,\leq1)}}$
where $\Phi_{F_3^{(3,\leq1)}}, \Phi_{\widehat{F}_4^{(4,\leq1)}}$
are the time-one flow maps generated by the Hamiltonians
\begin{equation}\label{F3F4Weak}
\begin{aligned}
&F_3^{(3,\leq1)}:=ad_{H_{\mathbb{C}}^{(2)}}^{-1} H^{(3,\leq1)}\,, \qquad 
\widehat{F}^{(4,\leq1)}_4:=ad_{H_{\mathbb{C}}^{(2)}}^{-1} 
\Pi_{Rg(H^{(2)}_{\mathbb{C}})}
\Pi^{d_z\leq 1} \widehat{H}_1^{(4)}\,,\\
&\qquad 
\widehat{H}_{1}^{(4)}:=H^{(4)}_{\mathbb{C}}
+\frac{1}{2}\{ F_3^{(3,\leq1)}, H^{(3,\leq1)}_{\mathbb{C}} \}+
\{ F_3^{(3,\leq1)}, H^{(3,\geq2)}_{\mathbb{C}} \}\,.
%\Big( H^{(4)}_{\mathbb{C}}
%+\frac{1}{2}\{ F_3, H^{(3)}_{\mathbb{C}} \} \Big)\,.
\end{aligned}
\end{equation}
{Observe that homogenous Hamiltonians of the form $F^{(n, \le 1)}$ that commute with momentum generate a finite dimensional vector field. This comes from the fact that the conservation of momentum for a monomial $u_{j_1}^{\sigma_1}\dots u_{j_n}^{\sigma_n}$ reads as $\sum_{i=1}^n\sigma_i j_i=0$.  Since there is at most one index $j_i\notin S$ this implies that all the $j_i$'s are contained in the ball centered at the origin with radius $(n-1)\max(S)$ (recall \eqref{maxS}). Then the flow  maps $\Phi_{F_3^{(3,\leq1)}}, \Phi_{\widehat{F}_4^{(4,\leq1)}}$ are well-defined.\\}
One obtains
\begin{equation}\label{larubiaWeak}
H_{\mathbb{C}}\circ\Phi_{WB}=
H_{\mathbb{C}}^{(2)}+H_{\mathbb{C}}^{(3,\geq2)}
+H_{WB}^{(4)}+
\widehat{H}_1^{(4,\geq2)}
+
{\rm quintic\; terms}\,,
\end{equation}
\begin{equation}\label{larubia2Weak}
\begin{aligned}
H_{WB}^{(4)}&:=\Pi_{Ker(H^2_{\mathbb{C}})}\Pi^{d_z\leq1} \widehat{H}_1^{(4)}\,.
 \end{aligned}
\end{equation}

\noindent(ii) As a second step we define the map
$\Phi_{LB}:=\Phi_{F_3^{(3,2)}}\circ\Phi_{\mathcal{F}_4^{(4,2)}}$
where $\Phi_{F_3^{(3,\leq1)}}, \Phi_{\mathcal{F}_4^{(4,\leq1)}}$
are the flows generated by the Hamiltonians
\begin{equation}\label{F3F4Linear}
\begin{aligned}
&F_3^{(3,2)}:=ad_{H_{\mathbb{C}}^{(2)}}^{-1} H^{(3,2)}\,, \qquad 
\mathcal{F}^{(4,2)}_4:=ad_{H_{\mathbb{C}}^{(2)}}^{-1} 
\Pi_{Rg(H^{(2)}_{\mathbb{C}})}
\Pi^{d_z=2} \mathcal{H}^{(4)}\,\\
&\qquad 
\mathcal{H}^{(4)}:=H^{(4)}_{\mathbb{C}}
+\frac{1}{2}\{ F_3^{(3,\leq1)}, H^{(3,\leq1)}_{\mathbb{C}} \}
+\frac{1}{2}\{ F_3^{(3,2)}, H^{(3,2)}_{\mathbb{C}} \}+
\{ F_3^{(3,\leq1)}, H^{(3,2)}_{\mathbb{C}} \}
+\{ F_3^{(3,\leq2)}, H^{(3,3)}_{\mathbb{C}} \}
\,.
%\Big( H^{(4)}_{\mathbb{C}}
%+\frac{1}{2}\{ F_3, H^{(3)}_{\mathbb{C}} \} \Big)\,.
\end{aligned}
\end{equation}
One obtains
\begin{equation}\label{larubiaLinear}
H_{\mathbb{C}}\circ\Phi_{WB}\circ\Phi_{LB}=
H_{\mathbb{C}}^{(2)}+H_{\mathbb{C}}^{(3,3)}
+H_{WB}^{(4)}
+H_{LB}^{(4)}+
\mathcal{H}^{(4,\geq3)}
+
{\rm quintic\; terms}
\end{equation}
\begin{equation}\label{larubia2Linear}
\begin{aligned}
H_{LB}^{(4)}&:=\Pi_{Ker(H^2_{\mathbb{C}})}\Pi^{d_z=2} \mathcal{H}^{(4)}\,.
 \end{aligned}
\end{equation}

We remark again that the procedures described above are ``formal''
since the maps we used are not \emph{a priori} well-defined.
Indeed they are flows at time one of possibly ill-posed PDEs.
This is due to two main reasons:  the presence of ``small divisors'' in the 
inversion of the adjoint action ${\rm ad}_{H^{(2)}_{\mathbb{C}}}$
and the fact that the vector field of $H_{\mathbb{C}}$
is \emph{quasi-linear}.

{
\begin{remark}
 In \cite{BFP} the authors provide a result of long-time stability of the pure gravity system. One may wonder if there is a connection between this result and the strongest normalization procedures described above (partial and full BNF) and whether this could be applied in the search for quasi-periodic solutions for the pure gravity system.\\
 In \cite{BFP} the authors use a \emph{modified energies} method which is based on a normal form approach (and the integrability at order four), but no genuine Birkhoff maps, meant as real transformations of the phase space, are provided. This is usually requested in KAM theory for PDEs. \\
  At the best of our knowledge, full or partial Birkhoff transformations are still not available for water waves systems. In any case we believe that our approach is well-suited for proving the existence of finite-dimensional quasi-periodic invariant tori, because in some sense it requires minimal assumptions.
 \end{remark}
}
Notice that the map $\Phi_{WB}$ generated by Hamiltonians in \eqref{F3F4Weak}
is actually well-posed
because is the flow of an $ODE$.
This is true since $\Pi^{d_{z}\leq 1}$, combined with the conservation of momentum,
gives the projection over a finite number of modes.
In the paper we 
shall follow a procedure similar to the \emph{Weak} plus
\emph{Linear} Birkhoff normal form.
The aim of the \emph{Weak} procedure is to find the first nonlinear approximate 
solution $v_{I}$ of \eqref{eq:113}.  Such approximate solution will be the starting point
for the Nash-Moser scheme
and the leading term of the expected quasi-periodic solution (see \eqref{SoluzioneEsplicitaWW}).
Actually, in section \ref{sec:linBNF}, we will construct  a map
admitting the same Taylor expansion (at low degree)
of  the map $\Phi_{LB}$ 
but with rigorous estimates on the Sobolev spaces $H^{s}$.
Therefore the main order of  corrections of the eigenvalues of the linearized operator
at $v_{I}$ will be given by the vector field 
of the Hamiltonian $H_{LB}^{(4)}$ in \eqref{larubia2Linear}.\\
To compute such corrections one needs to study 
the kernel of the adjoint action of $H^{(2)}_{\mathbb{C}}$\,.
By an explicit computation 
one has that a fourth order monomial 
$u^{\s_1}_{j_1}\dots u^{\s_4}_{j_4}$
belongs to $Ker(H^{(2)}_{\mathbb{C}})$
if it 
is Fourier supported on a \emph{4-waves resonances} 
that is, non-trivial integer solutions of
\begin{equation}\label{torires}
 \s_1\sqrt{|j_1|}+ \s_2\sqrt{|j_2|}+ \s_3\sqrt{|j_3|}+ \s_4\sqrt{|j_4|}=0\,,\qquad
 \s_1 j_1+\s_2 j_2+\s_3 j_3+\s_4 j_4=0\,.
\end{equation}

We know that (see for instance \cite{Zak2}, \cite{CW}, \cite{CS}) %there are no cubic resonances and
 the quartic resonances are given by two families of quartets of wave numbers:
\begin{itemize}
\item[$(i)$] the \emph{trivial} ones, which have the form  $j_1=j_2$, $\sigma_1=-\sigma_2$ and $j_3=j_4$, $\sigma_3=-\sigma_4$;
\item[$(ii)$] the \emph{Benjamin-Feir} resonances, which consist in 
the two parameter family of solutions
\begin{equation}\label{straBFR}
\bigcup_{\lambda\in\Z \setminus \{0\}, b\in\N} 
\Big\{ 
j_1 = -\lambda b^2, 
 \, j_2 = \lambda (b+1)^2 \, , \, j_3 = \lambda (b^2+b+1)^2, \, j_4= \lambda (b+1)^2 b^2 \Big\} \, ,
\end{equation}
with $\sigma_1=\s_3=-\sigma_2=-\s_4$,
\end{itemize}
In \cite{Zak2} (see also \cite{CW}, \cite{CS})
the authors show that
the coefficients of the Hamiltonian $H_{FB}^{(4)}$ in  \eqref{larubia2}
related to monomials supported on Benjamin-Feir resonances 
are zero. The consequence of this ``null condition''
of the gravity water waves 
system in infinite depth is that $H_{FB}^{(4)}$ is \emph{integrable}
and in particular
\begin{equation}\label{HamZakDyaCraig}
\begin{aligned}
H_{FB}^{(4)} &=\Pi_{\mathrm{triv}}\Big( H^{(4)}_{\mathbb{C}}+\frac{1}{2}\{ F_3, H^{(3)}_{\mathbb{C}} \} \Big)
%\\&
=\frac{1}{4 \pi} \sum_{k \in \Z} |k|^3  \big(  |z_k|^4  - 2 |z_{k}|^2 |z_{-k}|^2   \big)
\\
&\qquad+  \frac{1}{\pi} \sum_{\substack{k_1, k_2 \in \Z, \, \sign(k_1) = \sign( k_2 )  \\ |k_2| < |k_1|}} |k_1| |k_2|^2  
  \big( - |z_{-k_1}|^2 |z_{k_2} |^2 + |z_{k_1}|^2  |z_{k_2}|^2 \big)\,,
\end{aligned}
\end{equation}
where $\Pi_{\mathrm{triv}}$ denotes the projection on 
\emph{trivial} resonances.
To prove that the corrections to the eigenvalues are integrable is equivalent 
to show that 
\begin{align}
&H_{WB}^{(4)}=\Pi_{\mathrm{triv}}H_{WB}^{(4)}\,,\label{WBB1}\\
&H_{LB}^{(4)}=\Pi_{\mathrm{triv}}H_{LB}^{(4)}\,.\label{WLL1}
\end{align}
The condition in \eqref{WBB1}
could be proved 
using a ``generic'' choice of the tangential sites. % (see Def. \ref{Generic}).
On the other hand it seems that the same argument does not fit for the proof of the equality 
\eqref{WLL1}.
%Actually the \eqref{WLL1} is not a direct consequence of the null conditions 
%leading to the integrability of $H_{FB}^{(4)}$ (see \eqref{HamZakDyaCraig}), indeed
%by  \eqref{larubia2Linear}, \eqref{F3F4Linear}, \eqref{larubia2},
%{one could have}
%\[
%H^{(4)}_{LB}\neq \Pi^{d_{z}=2} H_{FB}^{(4)}\,.
%\]
So, to prove the \eqref{WLL1} we use an argument of identification of normal forms 
based on the ideas developed in  \cite{FGP1}.
In that paper the authors exploit the complete  integrability of the Degasperis-Procesi 
equation (see for instance \cite{Deg}, \cite{FGPa})
using a sequence of constants of 
motion in order to deal with non-trivial $4$-waves interactions.
As it is well known the pure gravity water waves equation is not integrable, as it has been proved in \cite{CW}; furthermore no other constants of motion are known, except for the energy and the momentum.\\
Actually the identification argument of \cite{FGP1}
works with \emph{approximate} constants of motion and for our purpose it is sufficient to find one of it.
 In section \ref{ApproxCons} 
we construct such approximate conserved 
quantity by exploiting the formal integrability at order four.

\subsection{Approximate constant of motion for the gravity water waves}\label{ApproxCons}

The idea of the construction of an approximate conserved quantity $K$ is the following: we know that the full BNF is integrable at order four, this implies that the Sobolev norms of the solutions of the truncated system $H^{(2)}_{\mathbb{C}}+H^{(4)}_{FB}$ (recall \eqref{larubia}) are preserved. Then quadratic Hamiltonians of the form $\sum_{j}j^{2n}|u_{j}|^{2}$, $n\geq 0$, are constants of motion for the truncated system.
Written in formula this means that, setting $K^{(2)}(u)=\sum_{j}j^{2n}|u_{j}|^{2}$,
we have 
\[
\{ K^{(2)}\circ \Phi_{F^{(4)}}^{-1}\circ  \Phi_{F^{(3)}}^{-1}, H \}\circ \Phi_{BF}=\{K^{(2)}, H\circ \Phi_{BF}\}=O(u^{5})\,.
\]
It is easy to check that
$ \Phi_{F^{(4)}}^{-1}\circ  \Phi_{F^{(3)}}^{-1}$ differes from $\Phi^{-1}_{F^{(3)}+F^{(4)}}$
by terms of order $O(u^{5})$.
Then by Taylor expanding we get
\[
K^{(2)}\circ\Phi^{-1}_{F^{(3)}+F^{(4)}}=K^{(2)}+\{K^{(2)}, F^{(3)}\}+
\{K^{(2)}, F^{(4)}\}+
\frac{1}{2}\{\{K^{(2)},F^{(3)}\}, F^{(3)} \}+
O(u^{5})\,.
\]
This justifies the choice of the Hamiltonian $K$ constructed in the following proposition.
\begin{proposition}\label{modifieden}
Consider the formal polynomial 
 $K=K^{(2)}+K^{(3)}+K^{(4)}$ where
\begin{equation}\label{K3K4}
K^{(2)}:=\sum_{j\in \mathbb{Z}\setminus\{0\}} j^{2}\,\lvert u_j \rvert^2\,,
\qquad
K^{(3)}:=ad_{K^{(2)}} F_3, \quad K^{(4)}:=ad_{K^{(2)}} F_4
+\frac{1}{2} \{ ad_{K^{(2)}} F_3, F_3 \}\,
\end{equation}
and $F_3$, $F_4$ are given in \eqref{F3F4}.
Then we have
\begin{equation}\label{commutazione}
\{ H_{\mathbb{C}}, K \}=O(u^5)
\end{equation}
where $O(u^5)$ denotes a formal power series with a zero at the origin of order $5$.
\end{proposition}

\begin{proof}
We first observe that $K^{(2)}$ is the Hamiltonian 
of the linear Schr\"odinger equation 
$\mathrm{i} u_t=u_{xx}$ restricted to the invariant 
subspace of zero-average functions $u(x)$, 
$x\in\mathbb{T}$. It is well known that $ad_{K^{(2)}}$ 
has only trivial resonances of order $3$ and $4$. 
In particular, this implies that cubic Hamiltonians, 
as for instance $F_3$, are in the range of the adjoint action of $K^{(2)}$.\\
The equation \eqref{commutazione} holds if and only if
\begin{align}
\{ K^{(2)}, H^{(2)}_{\mathbb{C}}\}=0\,, \label{eqquadratico}\\
\{ K^{(2)}, H^{(3)}_{\mathbb{C}}\}+\{  K^{(3)}, H^{(2)}_{\mathbb{C}}\}=0\,, \label{eqcubico}\\   \label{eqquartico}
\{ K^{(2)}, H^{(4)}_{\mathbb{C}}\}+\{ K^{(4)}, H^{(2)}_{\mathbb{C}}  \}
+\{ K^{(3)}, H^{(3)}_{\mathbb{C}}\}=0\,.
\end{align}
We note that $K^{(2)}$ is diagonal, 
hence \eqref{eqquadratico} holds. From this fact we deduce that the adjoint actions $ad_{K^{(2)}}$ and $ad_{H^{(2)}_{\mathbb{C}}}$ commutes on the intersection of their range $Rg(H_{\mathbb{C}}^{(2)})\cap Rg(K^{(2)})$.\\
The equation \eqref{eqcubico} holds since
\[
ad_{H^{(2)}_{\mathbb{C}}}  [K^{(3)}]\stackrel{\eqref{K3K4}}{=}ad_{K^{(2)}}\, 
ad_{H^{(2)}_{\mathbb{C}}}[F_3]\stackrel{\eqref{F3F4}}{=}\{ K^{(2)}, H^{(3)}_{\mathbb{C}}\}\,,
\]
where we used that the adjoint actions of $H^{(2)}_{\mathbb{C}}$ and $K^{(2)}$ commute and $F_3\in Rg(H_{\mathbb{C}}^{(2)})\cap Rg(K^{(2)})$.\\
% Since there are no cubic resonances for $ad(H^{(2)}_{\mathbb{C}})$, $ad(K_2)$ then $K_3=\Pi_{Rg(H_{\mathbb{C}}^{(2)})}\Pi_{Rg(K_2)} K_3$. This implies that
The equality \eqref{eqquartico}  is equivalent to
\begin{equation}\label{beremolto5}
K^{(4)}=ad_{H^{(2)}_{\mathbb{C}}}^{-1} \Pi_{Rg(H^{(2)}_{\mathbb{C}})} 
\Big(   \{ K^{(2)}, H^{(4)}_{\mathbb{C}}\}  
+\{ K^{(3)}, H^{(3)}_{\mathbb{C}}\}  \Big)\,.
\end{equation}
We need to check that the choice of $K^{(4)}$ in \eqref{K3K4}
implies the \eqref{beremolto5}.
First of all
notice that,
 by \eqref{eqcubico}, \eqref{F3F4}, \eqref{K3K4} 
 and the Jacobi identity, we have 
\[
\begin{aligned}
&\{ K^{(3)}, H^{(3)}_{\mathbb{C}}\} \stackrel{\eqref{K3K4}}{=}
\{ \{K^{(2)}, F_3\}, H^{(3)}_{\mathbb{C}}\} =
-\{ \{ H^{(3)}_{\mathbb{C}}, K^{(2)}\}, F_3\}-\{ \{F_3, H^{(3)}_{\mathbb{C}}\}, K^{(2)}\}\,,\\
& \{F_3, \{H^{(3)}_{\mathbb{C}}, K^{(2)}\}\}-\{ H^{(3)}_{\mathbb{C}}, \{K^{(2)}, F_3\}\} 
\stackrel{\eqref{eqcubico}}{=} -\{F_3, ad_{H^{(2)}_{\mathbb{C}}}
 (K^{(3)})\}-\{ H^{(3)}_{\mathbb{C}}, K^{(3)}\}
=
 ad_{H^{(2)}_{\mathbb{C}}} \{ K^{(3)}, 
F_3\}\,.
\end{aligned}
\]
Therefore
\begin{equation}\label{beremolto2}
\begin{aligned}
 \{ K^{(2)}, H^{(4)}_{\mathbb{C}}\}  +\{ K^{(3)}, H^{(3)}_{\mathbb{C}}\}&
 =ad_{K^{(2)}} ( H^{(4)}_{\mathbb{C}})
 +ad_{K^{(2)}} (\{ F_3, H^{(3)}_{\mathbb{C}}\})-\{ \{H^{(3)}_{\mathbb{C}}, K^{(2)}\}, F_3 \}\\
&=ad_{K^{(2)}} (H^{(4)}_{\mathbb{C}}+\frac{1}{2}\{F_3, H^{(3)}_{\mathbb{C}}\})
+\frac{1}{2} \Big( \{F_3, \{H^{(3)}_{\mathbb{C}}, K^{(2)}\}\}
-\{ H^{(3)}_{\mathbb{C}}, \{K^{(2)}, F_3\}\} \Big)\\
&\stackrel{\eqref{F3F4},\eqref{K3K4}}{=}
 ad_{K^{(2)}} (H^{(4)}_{\mathbb{C}}+\frac{1}{2}\{F_3, H^{(3)}_{\mathbb{C}}\})+
 \frac{1}{2}ad_{H^{(2)}_{\mathbb{C}}}
  \{ ad_{K^{(2)}} F_3, F_3 \}\,.
\end{aligned}
\end{equation}
Moreover using \eqref{eqquadratico} we have
\begin{equation}\label{orologio}
\begin{aligned}
\Pi_{Ker(H^{(2)}_{\mathbb{C}})}  ad_{K^{(2)}} (H^{(4)}_{\mathbb{C}}
+\frac{1}{2}\{F_3, H^{(3)}_{\mathbb{C}}\})&=
\Pi_{Ker(H^{(2)}_{\mathbb{C}})}  ad_{K^{(2)}} 
\Pi_{Ker(H^{(2)}_{\mathbb{C}})}\big(H^{(4)}_{\mathbb{C}}
+\frac{1}{2} \{F_3, H^{(3)}_{\mathbb{C}}\}\big)\,,\\
&\stackrel{ \eqref{HamZakDyaCraig}}{=}
ad_{K^{(2)}} \Pi_{\mathrm{triv}} \big(H^{(4)}_{\mathbb{C}}
+\frac{1}{2}\{F_3, H^{(3)}_{\mathbb{C}}\}\big)=0\,,
\end{aligned}
\end{equation}
since  $K^{(2)}$ commutes 
with terms supported on trivial resonances.
Then we have
\[
\eqref{beremolto5}\stackrel{\eqref{beremolto2}}{=}
 ad_{K^{(2)}} ad_{H^{(2)}_{\mathbb{C}}}^{-1} \Pi_{Rg(H^{(2)}_{\mathbb{C}})}
 \big(H^{(4)}_{\mathbb{C}}+\frac{1}{2}\{F_3, H^{(3)}_{\mathbb{C}}\}\big)+
 \frac{1}{2}
  \{ ad_{K^{(2)}} F_3, F_3 \}\stackrel{\eqref{F3F4},\eqref{K3K4}}{=}K^{(4)}\,.
\]
This proves that \eqref{eqquartico} holds with $K^{(4)}$ as in \eqref{K3K4}.
\end{proof}

\subsection{Identification of normal forms}\label{sec:3.33}
In this section 
we prove that the Hamiltonians $H_{WB}^{(4)}$, $H_{LB}^{(4)}$ 
in \eqref{larubia2Weak}, 
\eqref{larubia2Linear}
are actually given in terms of the projections 
$\Pi^{d_z=0}$ and $\Pi^{d_z=2}$ of the full Birkhoff normal form Hamiltonian
$H_{FB}^{(4)}$ in \eqref{larubia2}.
This will be done into two steps. First in Propositions \ref{LemmaBelloWeak}
and \ref{LemmaBelloLinear}
we show that 
\eqref{WBB1} and \eqref{WLL1} hold, i.e. the coefficients of monomials in 
$H_{WB}^{(4)}$, $H_{LB}^{(4)}$ supported on Benjamin-Feir resonances (see \eqref{straBFR})
are zero. 
Then in Proposition \ref{fineIdentif} we conclude the identification.

\medskip

The next two propositions are based on the fact that two commuting Hamiltonians can be put simultaneously in Birkhoff normal form and the normalized terms have to be supported in the intersection of the kernels of \emph{both} adjoint actions.\\
Then the main point is that the adjoint action of the constructed Hamiltonian $K$ possesses only \emph{trivial} $4$-waves resonant interactions with two plus signs and two minus signs and the Benjamin-Feir are of this type (see \eqref{straBFR}). Hence the common $4$-waves resonances  are just the trivial ones.\\
{\textbf{Notation:} Following \cite{FGP1}, given a Hamiltonian $H$ we use the sub-index $H_n$ to denote the transformed  Hamiltonian after $n$ steps of a normal form procedure.}
\begin{proposition}\label{LemmaBelloWeak}
Consider $H_{\mathbb{C}}$ in \eqref{Rinco} and its approximate constant of motion $K$ given in Proposition \ref{modifieden}. Recalling \eqref{F3F4Weak} let us call $\Phi:=\Phi_{F_3^{(3, \le 1)}}\circ \Phi_{\widehat{F}_4^{(4, \le 1)}}$.  Then
\begin{align}
&H_{\mathbb{C}}\circ \Phi=H^{(2)}+H_{\mathbb{C}}^{(3,\geq2)}
+H_{WB}^{(4)}+\widehat{H}^{(4,2)}_{1}+\widehat{H}_1^{(4,\geq3)}+R_2^{(\geq 5)}
\,,\label{wnormalform1}\\
&K\circ \Phi=K^{(2)}+K^{(3,\geq2)}+W_{2}^{(4, 0)}
+Q_{2}^{(\geq 5, \le 1)}
+K_{1}^{(\geq 4, \geq 2)}\,,\label{wnormalform2} 
\end{align}
where $\widehat{H}_1^{(4)}$ is given in \eqref{F3F4Weak}, $H_{WB}^{(4)}$ is in \eqref{larubia2Weak}, 
and 
$H_{WB}^{(4)}, W_{2}^{(4, 0)} \in  
\mbox{Ker}(K^{(2)}) \cap \mbox{Ker}(H_{\mathbb{C}}^{(2)})$.
\end{proposition}
\begin{proof}
One can reason as in  Proposition $3.6$ in \cite{FGP1}. The explicit form 
 \eqref{wnormalform1}
comes from \eqref{F3F4Weak}, \eqref{larubiaWeak}, \eqref{larubia2Weak}.
\end{proof}

\begin{proposition}\label{LemmaBelloLinear}
Consider $H_{\mathbb{C}}$ in \eqref{Rinco} and $K$ given in Proposition \ref{modifieden}. Recalling \eqref{F3F4Linear} 
let us call $\Psi:=\Phi_{F_3^{(3,  2)}}\circ \Phi_{\mathcal{F}_4^{(4, 2)}}$.  Then
\begin{equation}\label{wnormalform}
\begin{aligned}
&H_{\mathbb{C}}\circ \Phi\circ \Psi=H_{\mathbb{C}}^{(2)}+
H_{WB}^{(4)}+H_{LB}^{(4)}+R_4^{(\geq 5, \le 2)}
+H_{4}^{(\geq 3, \geq 3)}\,,\\
&K\circ \Phi\circ \Psi=K^{(2)}+W_{2}^{(4, 0)}+W_{2}^{(4, 2)}
+Q_{4}^{(\geq 5, \le 2)}
+K_{4}^{(\geq 3, \geq 3)}\,, 
\end{aligned}
\end{equation}
where 
$H_{LB}^{(4)}$ is given in \eqref{larubia2Linear}
and 
$H_{LB}^{(4)}, 
W_{2}^{(4, 2)} \in  \mbox{Ker}(K^{(2)}) \cap \mbox{Ker}(H_{\mathbb{C}}^{(2)})$.
\end{proposition}

\begin{proof}
%We have $F^{(3, \le 2)}:=\Pi^{d_z\le 2} F_3$. Then by the definition of $F_3$ we have that
%\begin{equation}
%ad_{H^{(2)}_{\mathbb{C}}}[F^{(3, \le 2)}]=H^{(3, \le 2)}_{\mathbb{C}}
%\end{equation}
Recall the definition of $F^{(3,2)}$ in \eqref{F3F4Linear}.
The equality \eqref{eqcubico} projected on $d_z\le 2$ gives
\begin{equation}
ad_{K^{(2)}}[F^{(3, 2)}]=K^{(3,  2)}.
\end{equation}
Hence
\begin{equation}
\begin{aligned}
&H\circ \Phi\circ{\Phi}_{F^{(3,2)}}=H^{(2)}+
H_{\mathbb{C}}^{(3,3)}+H_{WB}^{(4)}+\mathcal{H}^{(4,2)}+\mathcal{H}^{(4,\geq3)}
+R_3^{(\geq 5, \le 2)}
+H_{3}^{(\geq 3, \geq 3)}\,,\\
&K\circ \Phi\circ{\Phi}_{F^{(3,2)}}=K^{(2)}+W_{2}^{(4,0)}
+Q_{3}^{(\geq 4, \le 2)}
+K_{3}^{(\geq 3, \geq 3)}\,, 
\end{aligned}
\end{equation}
where $\mathcal{H}^{(4)}$ is in \eqref{F3F4Linear}.
Since \eqref{commutazione} holds and the map $\Phi\circ{\Phi}_{F^{(3,2)}}$ is close to the identity we have that
$
\{ H_{\mathbb{C}}, K \}\circ \Phi\circ{\Phi}_{F^{(3,2)}}=O(u^{5})
$ then
\[
\{ \mathcal{H}^{(4, 2)}, K^{(2)}\}=\{ Q_3^{(4,  2)}, H^{(2)}_{\mathbb{C}} \}\,.
\]
We note the following fact, which derives from the Jacobi identity and \eqref{eqquadratico}: 
if $f\in {Ker}(H^{(2)})$ then $\{ f, K^{(2)}\}\in {Ker}(H^{(2)})$, if $f\in {Rg}(H^{(2)})$ then $\{ f, K^{(2)}\}\in {Rg}(H^{(2)})$.

\noindent
Then we have that $\{ \Pi_{{Ker}(H^{(2)})} \mathcal{H}^{(4, 2)}, K^{(2)}  \}
\in {Ker}(H^{(2)})$ 
\[
\{ \Pi_{{Ker}(H^{(2)})}\mathcal{H}^{(4, 2)}, K^{(2)}  \}=
-\{ \Pi_{{Rg}(H^{(2)})} \mathcal{H}^{(4, 2)}, K^{(2)}  \}
+\{ H^{(2)}, Q_{3}^{(4, 2)}\}\in {Rg}(H^{(2)})\,.
\]
Thus $\{ \Pi_{{Ker}(H^{(2)})} \mathcal{H}^{(4, 2)}, K^{(2)}  \}=0$ and 
\[
\Pi_{{Ker}(H^{(2)})} \mathcal{H}^{(4, 2)}
=\Pi_{{Ker}(H^{(2)})} \Pi_{{Ker}(K^{(2)})} \mathcal{H}^{(4, 2)}\,.
\]
By symmetry 
$\Pi_{{Ker}(K^{(2)})} Q_{3}^{(4, 2)}
=\Pi_{{Ker}(H^{(2)})} \Pi_{{Ker}(K^{(2)})} Q_{3}^{(4, 2)}$. 
Hence
\begin{equation}\label{prisoner}
\Pi_{{Rg}(H^{(2)})} \Pi_{{Ker}(K^{(2)})} Q_{3}^{(4,  2)}
=\Pi_{{Rg}(K^{(2)})} \Pi_{{Ker}(H^{(2)})} \mathcal{H}^{(4, 2)}=0\,.
\end{equation}
The function $\mathcal{F}_4^{(4, 2)}$ solves the homological equation
\[
\{ H^{(2)}, \mathcal{F}_4^{(4,2)}\}=\Pi_{{Rg}(H^{(2)})} \mathcal{H}^{(4, 2)}
\stackrel{(\ref{prisoner})}{=}
\Pi_{{Rg}(K^{(2)})}\Pi_{{Rg}(H^{(2)})} \mathcal{H}^{(4, 2)}\,.
\]
We now show that $\mathcal{F}_4^{(4, 2)}$ solves also the homological equation 
for the commuting Hamiltonian
$K\circ \Phi\circ\Phi_{F^{(3,2)}}$. 
Indeed, by the fact that $\mathrm{ad}_{H^{(2)}}^{-1}$ 
commutes with $\mathrm{ad}_{K^{(2)}}$ 
on the intersection ${Rg}(H^{(2)})\cap {Rg}(K^{(2)})$, we have
\[
\{ K^{(2)}, \mathcal{F}_4^{(4, 2)}\}=
\mathrm{ad}_{H^{(2)}}^{-1}\{ K^{(2)}, 
\Pi_{Rg(K^{(2)})}\Pi_{Rg(H^{(2)})} \mathcal{H}^{(4, 2)} \}\,,
\]
and we get
\[
\{ K^{(2)}, \Pi_{{Rg}(K^{(2)})}
\Pi_{{Rg}(H^{(2)})} \mathcal{H}^{(4, 2)}\}
=\{H^{(2)}, \Pi_{{Rg}(K^{(2)})}\Pi_{{Rg}(H^{(2)})} 
Q_{1}^{(4,  2)}  \}\,.
\]
By \eqref{prisoner} we have that the resonant term 
$H_{LB}^{(4)}:=\Pi_{{Ker}(H^{(2)})} \mathcal{H}^{(4, 2)}$ 
belongs to the intersection of the kernels. This concludes the proof. 
\end{proof}

%
%\subsection{Formal equivalence between  \emph{weak} $+$ \emph{linear} 
%and \emph{full} BNF procedure}\label{weakLinugualeFULL}

\begin{remark}\label{conclusioneTrivial}
Notice that the monomials of degree $4$ with two bars (possibly supported on the Benjamin-Feir resonances) that belong to the kernel of ${ ad}_{K^{(2)}}$ (see \eqref{K3K4})
are action-preserving ( they are supported just on \emph{trivial} resonances).
Therefore Propositions \ref{LemmaBelloWeak}, \ref{LemmaBelloLinear}
imply the \eqref{WBB1}, \eqref{WLL1}.
\end{remark}

\begin{proposition}\label{fineIdentif}
One has that (see \eqref{larubia2Weak}, \eqref{larubia2Linear})
\begin{equation}\label{equivalenzaZZZ}
H_{WB}^{(4)}=\Pi_{\mathrm{triv}}H_{WB}^{(4)}
=\Pi^{d_z=0}H_{FB}^{(4)}\,, \qquad H_{LB}^{(4)}=\Pi_{\mathrm{triv}}H_{LB}^{(4)}
=\Pi^{d_z=2}H_{FB}^{(4)}
\end{equation}
where $H_{FB}^{(4)}$ is in \eqref{larubia2} (see also \eqref{HamZakDyaCraig}).
\end{proposition}

\begin{proof}
%Recalling  \eqref{larubia}, \eqref{larubia2},
%we have that after two steps of full Birkhoff normal form
%(with generators in 
% \eqref{F3F4}) one obtains  that the water waves Hamiltonian $H_{\mathbb{C}}$ in \eqref{Rinco}
% becomes
% 
%One step of formal BNF means that we apply the formal change of variables generated by
%\begin{equation}\label{fine2}
%\mathfrak{F}^{(3)}:=[\mathrm{ad}_{H^{(2)}}]^{-1}H^{(3)}
%\end{equation}
%which  removes completely $H^{(3)}$ and conjugates $H$ (using \eqref{formLIE}) to
%$H\circ{\Phi}_{(\mathfrak{F}^{(3)})^{-1}}=H^{(2)}+\frac12 \{\mathfrak{F}^{(3)}, H^{(3)}\} +$ h.o.t. 
%\\
%The scecond step of formal BNF removes all the non-resonant terms of degree four thus we get
%\[
%H^{(2)}_{\mathbb{C}}+\frac12
%\Pi_{{\mbox{Ker}}(H_{\mathbb{C}}^{(2)})} 
%\{{F}_{3},H_{\mathbb{C}}^{(3)}\} +\mbox{h.o.t.}\,.
%\]
We have already remarked that 
in \cite{Zak2}, \cite{CW}, \cite{CS}
%In \cite{FGPa} 
it has been proved that
\[
\big(\Pi_{{{Ker}}(H_{\mathbb{C}}^{(2)})}
-\Pi_{\mathrm{triv}}\big)
(H^{(4)}_{\mathbb{C}}+\frac{1}{2}\{ F_3, H^{(3)}_{\mathbb{C}} \})=0\,.
\]
 %(see  the notation of Definition \ref{chebellosiproietta}) 
 %and h
 Hence
\begin{equation}\label{megliotrivial}
\Pi^{d_z=2}\Pi_{{{Ker(H_{\mathbb{C}}^{(2)})}}}
(H^{(4)}_{\mathbb{C}}+\frac{1}{2}\{ F_3, H^{(3)}_{\mathbb{C}} \})
= \Pi^{d_z=2}\Pi_{\mathrm{triv}}
(H^{(4)}_{\mathbb{C}}+\frac{1}{2}\{ F_3, H^{(3)}_{\mathbb{C}} \})\,.
\end{equation}
We claim that
\begin{equation}\label{occhio1}
\begin{aligned}
&\Pi_{\mathrm{triv}}\{F_{3}^{(3,1)},H_{\mathbb{C}}^{(3,3)}\}
=\Pi_{\mathrm{triv}}\{F_{3}^{(3,3)},H_{\mathbb{C}}^{(3,1)}\}=
\Pi_{\mathrm{triv}}\{F_{3}^{(3,2)},H_{\mathbb{C}}^{(3,0)}\}=
\Pi_{\mathrm{triv}}
\{F_{3}^{(3,0)},H_{\mathbb{C}}^{(3,2)}\}=0\,.
%\\
%&\Pi_{\mathrm{triv}}
%\{F_{3}^{(3,2)},H^{(3,0)}\}=
%\Pi_{\mathrm{triv}}
%\{F_{3}^{(3,0)},H^{(3,2)}\}=0\,.
\end{aligned}
\end{equation}
This implies that 
\begin{equation}\label{rambo5}
\begin{aligned}
\Pi_{{{Ker}}(H_{\mathbb{C}}^{(2)})}
&\Pi^{d_z=2}(H^{(4)}_{\mathbb{C}}
+\frac{1}{2}\{ F_3, H^{(3)}_{\mathbb{C}} \})
%=\\&
=
\Pi_{\mathrm{triv}}\Pi^{d_z=2}\Big(
H^{(4,2)}_{\mathbb{C}}+
\{F_{3}^{(3,\leq1)},H_{\mathbb{C}}^{(3,\leq1)}\}
+\{F_{3}^{(3,\leq1)},H_{\mathbb{C}}^{(3,\geq 2)}\}
\Big)\\
&+\Pi_{\mathrm{triv}}\Pi^{d_z=2}\Big(
H^{(4,2)}_{\mathbb{C}}+
\{F_{3}^{(3,2)},H_{\mathbb{C}}^{(3,0)}\}
+\{F_{3}^{(3,2)},H_{\mathbb{C}}^{(3,2)},\}
+\{F_{3}^{(3,3)},H_{\mathbb{C}}^{(3,1)}\}
\Big)\\
&=\Pi_{\mathrm{triv}}\Pi^{d_z=2}\Big(
H^{(4,2)}_{\mathbb{C}}+
\{F_{3}^{(3,\leq1)},H_{\mathbb{C}}^{(3,\leq1)}\}+\{F_{3}^{(3,2)},H_{\mathbb{C}}^{(3,2)}\}
\Big)\,.
\end{aligned}
\end{equation}
\begin{proof}[Proof of the claim \eqref{occhio1}]
Let us consider the term $\{F_{3}^{(3,1)},H_{\mathbb{C}}^{(3,3)}\}$, where
\begin{equation}\label{montagna}
F_3^{(3,1)}=\sum_{\substack{ j_1,j_2\in S,j_3\in S^c\\ 
\s_1 j_1+\s_2 j_2+\s_3 j_3=0\\ 
\s_i=\pm}}C_{j_1,j_2,j_3}^{\s_1,\s_2,\s_{3}}u^{\s_1}_{j_1}u^{\s_2}_{j_2}u^{\s_3}_{j_3}\,,
\qquad
H_{\mathbb{C}}^{(3,3)}=\sum_{\substack {j_1,j_2,j_3 \in S^c \\
\s_1 j_1+\s_2 j_2+\s_{3} j_3=0 \\
\s_i=\pm}} \widetilde{C}^{\s_1,\s_{2},\s_{3}}_{j_1,j_2,j_3}u^{\s_1}_{j_1}u^{\s_2}_{j_2}u^{\s_3}_{j_3}\,,
\end{equation}
for some coefficients ${C}^{\s_1,\s_{2},\s_{3}}_{j_1,j_2,j_3}$, 
$\widetilde{C}^{\s_1,\s_{2},\s_{3}}_{j_1,j_2,j_3}\in\mathbb{C}$. 
Using \eqref{poissonBraComp} one gets
\begin{equation*}%\label{trivtriv}
\{F_3^{(3,1)},H_{\mathbb{C}}^{(3,3)}\}=
\sum_{\substack{j_1,j_2\in S, \, k_1,k_2\in S^{c} \\ \s_1 j_1+\s_2 j_2+\s_3 k_1+\s_4 k_2=0\\
\s_i=\pm}} 
P_{j_1,j_2,k_1,k_2}^{\s_1,\s_2,\s_3,\s_4}u^{\s_1}_{j_1}u^{\s_2}_{j_2}u^{\s_3}_{k_1}u^{\s_4}_{k_2}\,
\end{equation*}
for some $P_{j_1,j_2,k_1,k_2}^{\s_1,\s_2,\s_3,\s_4}\in \mathbb{C}$.
A monomial of degree $4$ $u^{\s_1}_{j_1}u^{\s_2}_{j_2}u^{\s_3}_{k_1}u^{\s_4}_{k_2}$ 
is supported on trivial resonances if $\s_1=-\s_2$, $\s_3=-\s_{4}$, 
 $j_1=j_2$ and $k_1=k_2$, since $j_i\in S$, $k_i\in S^{c}$, $i=1,2$. 
{We note that by applying \eqref{poissonBraComp}
the coefficients  $P_{j_1,j_1,k_1,k_1}^{\s_1,-\s_1,\s_3,-\s_3}$ 
should be  given by  
combinations of coefficients of the form
${C}^{\s,-\s,\pm}_{j,j,0}\widetilde{C}^{\s',-\s',\mp}_{k,k,0}$ (up to permutation of the indexes). 
However, due to the restrictions in  the sums \eqref{montagna}
and the fact that $0\notin S\cup S^{c}$
(see \eqref{TangentialSitesDP}) such 
coefficients do not appear.}
% On the other hand, by the momentum condition, we have that
%$j_1-j_2=-(k_1-k_2)$, {which is a contradiction since} $0\notin S\cup S^{c}$
%(see \eqref{TangentialSitesDP}). 
Hence the \eqref{occhio1} holds. 
The others equalities in \eqref{occhio1} follow by reasoning in the same way.
\end{proof}

Now consider the Hamiltonian $H_{LB}^{(4)}$ in \eqref{larubia2Linear}
obtained by two steps of Weak normal form
plus two steps of Linear Birkhoff normal form 
(see Propositions \ref{LemmaBelloWeak}, \ref{LemmaBelloLinear}).
Then, recalling the explicit formula \eqref{F3F4Linear}
we get
\begin{equation*}%\label{resLINBNF}
\begin{aligned}
%&\Pi^{d_z=2}\Pi_{{\mbox{Ker}}(H^{(2)})}
%( \frac12 \{\mathfrak{F}^{(3,2)}, H_2^{(3,\ge 2)}\}+ H_2^{(4,\geq2)})
%= \Pi^{d_z=2}\Pi_{\mathrm{triv}} (\frac12 \{\mathfrak{F}^{(3,2)},H_2^{(3,\ge 2)}\} 
%+ H_2^{(4,\geq2)})\\
Z_2^{(4,2)}&\stackrel{Rmk. \ref{conclusioneTrivial}}{=}
\Pi^{d_z=2}\Pi_{\mathrm{triv}} (\frac12 \{{F}^{(3,2)},H_{\mathbb{C}}^{(3, 2)}\}
+ \frac{1}{2} \{{F}^{(3,\leq1)}, H_{\mathbb{C}}^{(3, \leq 1)}\}
+\{  {F}^{(3,1)}, H_{\mathbb{C}}^{(3, 3)}\}
+H^{(4,2)}_{\mathbb{C}})\\
&\stackrel{\eqref{occhio1}}{=} 
 \Pi^{d_z=2}\Pi_{\mathrm{triv}} 
(H^{(4,2)}_{\mathbb{C}}+ \frac12\{F^{(3,2)}, H_2^{(3, 2)}\}
+ \frac{1}{2} \{F^{(3,\leq1)}, H^{(3, \leq 1)}\})
\\ &\stackrel{\eqref{rambo5},\eqref{megliotrivial}}{=} 
 \Pi^{d_z=2}\Pi_{{\mathrm{triv}}}
 \Big(H^{(4)}_{\mathbb{C}}+ \frac12\{{F}^{(3)}, H^{(3)}\}\Big)\,.
\end{aligned}
\end{equation*}
This implies, 
recalling the \eqref{HamZakDyaCraig}, the second equality in the 
claim \eqref{equivalenzaZZZ}. The first one in \eqref{equivalenzaZZZ}
follows by using the projector $\Pi^{d_z=0}$ and reasoning in the same way.
\end{proof}

\section{Weak Birkhoff Normal Form}\label{sec:weak}
The aim of this section is to construct a $\zeta$-parameter family of approximately invariant, finite dimensional tori supporting quasi-periodic motions 
with frequency $\omega(\zeta)$. We will impose the 
map $\zeta\mapsto \omega(\zeta)$ to be a diffeomorphism 
and we will consider such approximate solutions as the 
starting point for the Nash-Moser algorithm. For this purpose
we apply a Weak Birkhoff normal form procedure as explained in 
section \ref{compaBNFpro}.
%In order to get a ``good'' approximation we perform several steps of such procedure.
%So, in order to state the main result of this section, 
%we need some preliminary definitions.

\begin{definition}{\bf (Resonance).}\label{risonanzaWW}
Let $n\geq 3$ be an integer. We say that $\left\{\big(j_i, \sigma_i \big) \right\}_{i=1}^n\subset (\mathbb{Z}\times\{\pm \})^n$ is a resonance of order $n$ if
\begin{equation}\label{por}
\mathcal{M}_{\vec{\sigma}}(j_1, \dots, j_n):=\sum_{i=1}^n \sigma_i j_i=0, \qquad \mathcal{R}_{\vec{\sigma}}(j_1, \dots, j_n):=\sum_{i=1}^n \sigma_i \sqrt{\lvert j_i \rvert}=0
\end{equation}
where $\vec{\sigma}:=(\sigma_1, \dots, \sigma_n).$
Let us denote by $n_+$ the number of $\sigma_i=+$ and, up to reordering in the index $i$, let us assume that $\sigma_i=+$ for $i=1, \dots, n_+$. \\
%Note that if $n_+=0$ or $n_+=n$ then $\mathcal{R}_{\sigma}$ in \eqref{por} is never zero.\\ 
A resonance of order $n$ is said to be \emph{trivial} if $n$ is even, $n_+=n/2$ and
\begin{equation}\label{coppiette}
\big\{ j_1, \dots, j_{n_+}\big\}=\big\{ j_{n_++1}, \dots, j_n\big\}.
\end{equation}
Otherwise we say that it is non-trivial. 
\end{definition}

\begin{remark}\label{oddresonances}
Note that resonances of odd order are non-trivial.
We also remark that if $\left\{\big(j_i, \sigma_i \big) \right\}_{i=1}^n$ is a 
trivial resonance then the functions 
$\mathcal{M}_{\vec{\sigma}}$ and $\mathcal{R}_{\vec{\sigma}}$ 
defined in \eqref{por} are identically zero.
\end{remark}

%Given a homogeneous Hamiltonian $G_{\mathbb{C}}$ we denote the adjoint action of $G_{\mathbb{C}}$ by $\mbox{ad}_{G_{\mathbb{C}}}:=\{ G_{\mathbb{C}}, \cdot \}$ and by $\Pi_{\mbox{Ker}(G_{\mathbb{C}})}$, $\Pi_{\mbox{Rg}(G_{\mathbb{C}})}$ the projection on the kernel and the Range of $\mbox{ad}_{G_{\mathbb{C}}}$.\\
%We define $\Pi_{\mathrm{triv}}$ the projection on the monomials supported on trivial resonances of $\mbox{ad}_{G_{\mathbb{C}}}$.

\noindent
Consider a monomial of the form $u^{\sigma_{1}}_{j_1}\dots u^{\sigma_n}_{j_{n}}$ and the Hamiltonian $H^{(2)}_\C$ in \eqref{Rinco}. 
Then we have (recall \eqref{HamMom2}, \eqref{poissonBraComp})
\begin{equation}\label{momcons}
\{ M_{\mathbb{C}} , u^{\sigma_{1}}_{j_1}\dots u^{\sigma_n}_{j_{n}}\}
=-\ii \mathcal{M}_{\vec{\sigma}}(j_1, \dots, j_n) 
u^{\sigma_{1}}_{j_1}\dots u^{\sigma_n}_{j_{n}},
\end{equation}
\begin{equation}\label{adjoint}
\{ H^{(2)}_{\mathbb{C}} , u^{\sigma_{1}}_{j_1}\dots u^{\sigma_n}_{j_{n}}\}
=-\ii \mathcal{R}_{\vec{\sigma}}(j_1, \dots, j_n) u^{\sigma_{1}}_{j_1}\dots u^{\sigma_n}_{j_{n}}.
\end{equation}
We point out that the Hamiltonian \eqref{Rinco} preserves the momentum and, in particular, by \eqref{momcons}, each term $H_{\mathbb{C}}^{(p)}$ is sum of monomials Fourier supported on $j_1, \dots, j_n$ such that $\mathcal{M}_{\vec{\sigma}}(j_1, \dots, j_n)=0$.\\
We say that a monomial $u^{\sigma_{1}}_{j_1}\dots u^{\sigma_n}_{j_{n}}$ is \emph{resonant} if it belongs to the $\mbox{Ker}(H^{(2)}_{\mathbb{C}})$, namely, by \eqref{adjoint}, it is Fourier supported on $j_1, \dots, j_n$ such that $\mathcal{R}_{\vec{\sigma}}(j_1, \dots, j_n)=0$. \\
For a finite dimensional subspace of 
$H_0^1(\mathbb{T})\times H_0^1(\mathbb{T})$ 
\begin{equation}\label{FinitedimensionalSubspaceDP}
E:=E_C:=\mbox{span}\left\{ e^{\mathrm{i}\,j\,x} : 0<\lvert j \rvert\le C \right\}\,, 
\quad C>0\,,
\end{equation}
let $\Pi_E$ denotes the corresponding $L^2$-projector on $E$.

The aim of this section is to prove the following.

\begin{proposition}\label{WBNFWW}
Let $s\geq 0$.
For any \emph{generic}  
choice of the set $S$ of tangential sites of the form
\eqref{TangentialSitesDP}
there exist $r>0$, depending on  $S$, 
and an analytic  symplectic change of coordinates 
\begin{equation}\label{FBI}
\Phi_B\colon \mathcal B_{r}(0,H_0^s(\T)\times H_0^s(\T))\to H_0^s(\T)\times H_0^s(\T)\,, 
\qquad \Phi_B= \mathrm{I} + \Psi\,,\quad 
 \Psi= \Pi_E \circ  \Psi \circ \Pi_E 
\end{equation}
where $E$ is a finite dimensional space 
as in \eqref{FinitedimensionalSubspaceDP}, 
such that 
$M_{\mathbb{C}}\circ\Phi_B=M_{\mathbb{C}}$ where 
$M_{\mathbb{C}}$ is the momentum Hamiltonian in 
\eqref{HamMom2} and 
the Hamiltonian $H_{\mathbb{C}}$ in \eqref{Rinco} transforms into
\begin{equation}\label{piotta}
\begin{aligned}
\mathcal{H}_{\mathbb{C}}:= H_{\mathbb{C}}\circ \Phi_B 
&= H_{\mathbb{C}}^{(2)} +\mathcal{H}_{\mathbb{C}}^{(4,0)} +
\sum_{k=3}^{7}
\mathcal{H}_{\mathbb{C}}^{(2k,0)} 
+\mathcal{H}_{\mathbb{C}}^{(\geq 16, \le 1)}
+ \mathcal{H}_{\mathbb{C}}^{(\geq 3,\geq 2)}\,,
\end{aligned}
\end{equation}
where 
\begin{equation}\label{theoBirH}
\begin{aligned}
\mathcal{H}_{\mathbb{C}}^{(4,0)} &:=
\frac{1}{4 \pi} \sum_{k \in S} |k|^3  
  |u_k|^4 
  +  \frac{1}{\pi} \sum_{\substack{k_1, k_2 \in S, \, 
  \sign(k_1) = \sign( k_2 )  \\ |k_2| < |k_1|}} |k_1| |k_2|^2  
 |u_{k_1}|^2  |u_{k_2}|^2 \,
\end{aligned}
\end{equation}
%\begin{equation}\label{theoBirH}
%\begin{aligned}
%\mathcal{H}_{\mathbb{C}}^{(4,0)} &:=
%\frac{1}{4 \pi} \sum_{k \in S} |k|^3  
%\big(  |u_k|^4  - 2 |u_{k}|^2 |u_{-k}|^2   \big)
%\\
%&+  \frac{1}{\pi} \sum_{\substack{k_1, k_2 \in S, \, \sign(k_1) = \sign( k_2 )  \\ |k_2| < |k_1|}} |k_1| |k_2|^2  
%  \big( - |u_{-k_1}|^2 |u_{k_2} |^2 + |u_{k_1}|^2  |u_{k_2}|^2 \big)\,.
%\end{aligned}
%\end{equation}
 is the fourth order Hamiltonian
provided by 
Zakharov-Dyachenko in \cite{Zak2} and 
Craig-Worfolk in  \cite{CW} restricted to momomials supported on the
 tangential set $S$, 
and $\mathcal{H}_{\mathbb{C}}^{(2k, 0)}
=\Pi_{\mbox{Ker}(H_{\mathbb{C}}^{(2)})} 
\mathcal{H}_{\mathbb{C}}^{(2k, 0)}$ with $k=3, \ldots, 7$ 
are supported only on trivial resonances
.
\end{proposition}

In the following subsection we show that one can choose ``generically''
the set of indexes $S$  in such a way that
in the Hamiltonian in \eqref{Rinco}
there are no monomials $u_{j_1}^{\s_1} \cdots u_{j_p}^{\s_p}$,
with at most one index $j_{i}\in S^{c}$, 
supported on a non-trivial resonance. In turn this implies the existence  of a finite dimensional submanifold that is invariant for the truncated normal form. The integrability of the normalized Hamiltonian gives that such manifold is foliated by families of tori. 
In subsection \ref{escoesco} we prove Proposition \ref{WBNFWW}.

\subsection{Resonances}\label{algresoreso}

The aim of this subsection is to prove the following result.

\begin{proposition}\label{resonances}
If $\left\{\big(j_i, \sigma_i \big) \right\}_{i=1}^n$ is a non-trivial resonance of order $n$ with at most one index $j_i$ outside $S$ then there exists a non-identically zero polynomial $\mathcal{P}\colon \mathbb{R}^{n}\to \mathbb{R}$ such that $\mathcal{P}(j_1, \dots, j_n)=0$.
\end{proposition}
We remark that there are just a finite number of non-trivial resonances with one index out of $S$, since $S$ is finite and the momentum is preserved. Then by Proposition \ref{resonances} there exists a generic choice of tangential sites such that in the Hamiltonian \eqref{Rinco}
there are no non-trivial resonant  monomials $u_{j_1}^{\s_1} \cdots u_{j_p}^{\s_p}$,
with at most one index $j_{i}\in S^{c}$.
%supported on a non-trivial resonance.
%\comment{enunciato modificato, pensarci}

Before giving the proof of the proposition above 
we need the following algebraic lemma.
\begin{lemma}\label{algebraiclemma}
Let $f\colon \mathbb{R}^{n-1}\to \mathbb{R}$ be 
an algebraic function and let $(\bar{x}_1, \dots, \bar{x}_{n-1})\in\mathbb{R}^{n-1}$ 
be a zero of $f$. If
\begin{equation}\label{okc0}
\nabla f (\bar{x}_1, \dots, \bar{x}_{n-1})\neq 0
\end{equation}
then there exists a non identically zero polynomial 
$Q\colon \mathbb{R}^{n-1}\to \mathbb{R}$ 
such that $Q(\bar{x}_1, \dots, \bar{x}_{n-1})=0$. 
\end{lemma}
\begin{proof}
Since $f$ is an algebraic function there exists 
a non-identically zero polynomial 
$P\colon \mathbb{R}^n\to \mathbb{R}$ 
such that
\begin{equation}\label{okc}
P(x_1, \dots, x_{n-1}, f(x_1, \dots, x_{n-1}))=0 
\qquad \forall (x_1, \dots, x_{n-1})\in\mathbb{R}^{n-1}\,.
\end{equation}
We can write the polynomial $P$ in the following form
\begin{equation}\label{okc2}
P(x_1, \dots, x_n)=p_0(x_1, \dots, x_{n-1})+\sum_{k\geq m} 
p_k(x_1, \dots, x_{n-1}) x_n^k
\end{equation}
where $m\geq 1$ and $p_i$ are polynomials of $n-1$ real variables. Without loss of generality we can assume that $p_m$ is non identically zero. By \eqref{okc} we have that
\[
P(\bar{x}_1, \dots, \bar{x}_{n-1}, 0)=0.
\]
If $p_0$ is non identically zero 
then we conclude by considering $Q:=p_0$. 
Otherwise, by \eqref{okc0} there exists $j\in\{1, \dots, n-1\}$ 
such that $\partial_{x_j} f(\bar{x}_1, \dots, \bar{x}_{n-1})\neq 0$. 
By \eqref{okc}, \eqref{okc2} we have that for all 
$(x_1, \dots, x_{n-1})\in\mathbb{R}^{n-1}$
\begin{align}
&0 =\partial_{x_j}^m P(x_1, \dots, x_{n-1}, f(x_1, \dots, x_{n-1}))\nonumber\\ 
&=\sum_{k\geq m} \sum_{\alpha=0}^m 
\begin{pmatrix} m\\ \alpha \end{pmatrix} 
\partial_{x_j}^{m-\alpha}  p_k(x_1, \dots, x_{n-1}) 
\partial_{x_j}^{\alpha} (f(x_1, \dots, x_{n-1})^k)\,.\label{okc3}
\end{align}
Now we consider the above relation when it is evaluated at $(\bar{x}_1, \dots, \bar{x}_{n-1})$. Since $k\geq m$ there is just one term in \eqref{okc3} which is not zero, the one involving just derivatives of $f$, but not $f$ itself 
(recall that $f(\bar{x}_1, \dots, \bar{x}_{n-1})=0$). 
Such term corresponds to $k=m=\alpha$ and gives
\[
0=\partial_{x_j}^m P(\bar{x}_1, \dots, \bar{x}_{n-1}, 0)=m!\, 
\Big(\partial_{x_j} f (\bar{x}_1, \dots, \bar{x}_{n-1})\Big)^m 
\,p_m(\bar{x}_1, \dots, \bar{x}_{n-1})\,.
\]
Hence we deduce that 
$p_m(\bar{x}_1, \dots, \bar{x}_{n-1})=0$ and 
we can choose $Q:=p_m$.
\end{proof}

\begin{proof}[{\bf Proof of Proposition \ref{resonances}}]
Note that if $\left\{\big(j_i, \sigma_i \big) \right\}_{i=1}^n$ 
is a non-trivial resonance and all the $j_i$'s belong to $S$ then, {thanks to the momentum conservation (recall $\mathcal{M}_{\vec{\sigma}}$ in \eqref{por}), }
we can consider
\[
\mathcal{P}(x_1, \dots, x_n):=\sum_{i=1}^n \sigma_i x_i\,.
\]
Hence let us consider the case in which there is one 
$j_i\notin S$ and, to simplify the notation, suppose that $i=n$.
We can assume that $0<n_+<n$ otherwise 
$\mathcal{R}(j_1, \dots, j_n)$ in \eqref{por} is strictly positive.\\
By using the first equation in \eqref{por} we have that 
$\sigma_n j_n=-\sum_{i\neq n}\sigma_i j_i$. 
Then, by using the second equation in \eqref{por}, 
we have that $(j_1, \dots, j_{n-1})$ is a zero of the algebraic function
\[
f(x_1, \dots, x_{n-1}):=\sum_{i=1}^{n-1} \sigma_i \sqrt{\lvert x_i \rvert}
+\sigma_n \sqrt{\lvert \sum_{i\neq n} \sigma_i x_i \rvert}\,.
\]
We have
\[
\partial_{x_1} f(j_1, \dots, j_{n-1})
=\frac{\sigma_1\,\mbox{sgn}(j_1)}{2\sqrt{\lvert j_1 \rvert}}
+\frac{\sigma_1 \sigma_n \mbox{sgn}
(\sum_{i\neq n} \sigma_i j_i)}{2 \sqrt{\lvert j_n \rvert}}\,.
\]
Thus $\partial_{x_1} f(j_1, \dots, j_{n-1})= 0$ if and only if 
$\lvert j_1 \rvert=\lvert j_n \rvert$ and 
$\mbox{sgn}(j_1)=\sigma_n \mbox{sgn}(\sigma_n j_n)$, 
which implies $j_1=j_n$. Since we assumed that 
$j_n\in S^c$ and $j_1\in S$ we have that 
$\partial_{x_1} f(j_1, \dots, j_{n-1})\neq 0$. 
By Lemma \ref{algebraiclemma} this means that $(j_1, \dots, j_{n-1})$ 
is the zero of some non-identically zero polynomial $Q(x_1, \dots, x_{n-1})$. 
We conclude by considering 
$\mathcal{P}(x_1, \dots, x_n)=x_n \,Q(x_1, \dots, x_{n-1})$.
\end{proof}

\subsection{Proof of Proposition \ref{WBNFWW}}\label{escoesco}
The change of coordinates $\Phi_B$ is constructed 
via a weak version of the Birkhoff normal form 
algorithm as described in section \ref{compaBNFpro}. 
Such method has been implemented in several 
papers (see for instance \cite{KdVAut}, \cite{FGP1}), 
so here we just resume the scheme that we follow 
to obtain the result in Proposition \ref{WBNFWW}.\\
The map $\Phi_B$ is defined as the composition of several symplectic 
changes of variables $\Phi_k$
%:=\Phi_{k{|{t=1}}}^{t}$, 
that are time-one flow maps of homogeneous Hamiltonians 
$F^{(k+2, \le 1)}$ such that $\Pi^{d_z \le 1} F^{(k+2, \le 1)}=F^{(k+2, \le 1)}$, 
namely $F^{(k+2, \le 1)}$ is Fourier supported on set 
of indexes with at most one index out of $S$. 
Since the Hamiltonian $H_{\mathbb{C}}$ 
Poisson commutes with the momentum $\mathcal{M}_{\mathbb{C}}$ 
we look for generators $F^{(k+2, \le 1)}$ with the same property. %\\
%%%%%%%%%%%%%%%%%%%%%%%%%%%%%%%%%%%%%
%We point out that a monomial $u_{j_1}\dots u_{j_{k}} \,\bar{u}_{j_{k+1}}\,\dots \bar{u}_{j_n}$ 
%of degree $n$ commutes with momentum if and only if 
%\begin{equation}\label{campana}
%\mathcal{M}_{\sigma}(j_1, \dots, j_n)=0
%\end{equation}
%(see \eqref{por}), where $\sigma_i=1$ if $1\le i\le k$
%and $\sigma_i=-1$ if $k+1\le i\le n$.
%%%%%%%%%%%%%%%%%%%%%%%%%%%%%%%%%%%%%
By momentum conservation and the fact that 
$\Pi^{d \le 1} F^{(k+2, \le 1)}=F^{(k+2, \le 1)}$, 
we have that the Fourier support of $F^{(k+2, \le 1)}$ 
is given by a finite set of integer indexes. So $\Phi_{k}$ it is a well defined, 
analytic and invertible map of the phase space as time-one flow map 
of an ordinary differential equation with polynomial vector field. %\\

\noindent
Now we describe a generic step of weak 
Birkhoff normal form and we see how we choose the generators 
$F^{(k+2, \le 1)}$.\\
Let $k\geq 1$, then after $k$-steps of weak Birkhoff normal form the Hamiltonian has the form
\begin{equation}\label{weakform}
\begin{cases}
H_k:=H_{k-1}\circ \Phi_k^{-1}=H_{\mathbb{C}}^{(2)}+Z^{(\le k+2, \le 1)}_k+R^{(\geq k+3, \le 1)}_k+H_{k-1}^{(\geq 3, \geq 2)}\\
H_0=H_{\mathbb{C}},\qquad Z_0^{(\le 2, \le 1)}=0, \quad R_0^{( 3, \le 1)}=H_{\mathbb{C}}^{(3, \le 1)}\,,
\end{cases}
\end{equation}
where $\Pi_{Ker(H^{(2)})} Z_k^{(\le k+2, \le 1)}=Z_k^{(\le k+2, \le 1)}$ is given by
\[
Z_k^{(\le k+2, \le 1)}:=Z_{k-1}^{(\le k+1, \le 1)}+\Pi_{Ker(H^{(2)})} R_{k-1}^{( k+2, \le 1)}\,.
\]
By Proposition \ref{resonances}, for a generic choice of the set $S$ we have
\begin{itemize}
 \item $Z^{(\le k+2, 1)}_k=0$ for any $k$,
 \item $Z_k^{(k+2, \le 1)}=0$ if $k$ is odd,
 \item $Z^{(k+2, 0)}_k$ depends only on the actions 
 $\lvert u_j \rvert^2$ with $j\in\mathbb{Z}$, $j\neq 0$ if $k$ is even.
 \end{itemize}
The Hamiltonian $H_{k}$ in  \eqref{weakform} is obtained by choosing $F^{(k+2, \le 1)}$ 
as the solution of the following equation {
\begin{equation}\label{homological}
\{ H_{\mathbb{C}}^{(2)}, F^{(k+2, \le 1)}\}=
\Pi_{\mbox{Rg}(H^{(2)})} R_{k-1}^{(k+2,\le 1)}\,,
\end{equation}
}
and Taylor expanding $H_{k-1}\circ \Phi_k^{-1}$.
We perform $N=13$ steps of the Weak BNF procedure described above.
Then we consider 
$\Phi_B=\Phi^{-1}_{13}\circ \dots \circ \Phi_{1}^{-1}$ and we set
\[
\mathcal{H}_{\mathbb{C}}^{(2 k, 0)}
=Z^{( 2k, 0)}_{13} \quad 2\le k\le 7\,, 
\qquad 
\mathcal{H}_{\mathbb{C}}^{(\geq 15, \le 1)}=R_{13}^{(\le 15, \le 1)}\,, 
\quad  \mathcal H_{\mathbb{C}}^{(\geq 3,\geq 2)}=H_{13}^{(\geq 3,\geq 2)}\,.
\]
By the discussion of section \ref{sec:integrability} we conclude that 
 the expression of the 
normalized Hamiltonian $\mathcal{H}^{(4, 0)}$ 
is the one obtained by (\cite{BFP}, \cite{CW}, \cite{Zak2}) 
through a formal procedure and then restricted 
to monomials supported on the
 tangential sites $S$.  Indeed
 the map $\Phi_{B}$ coincides with the map $\Phi$
  in 
 Proposition \ref{LemmaBelloWeak} up to order  $5$, and hence
 $\mathcal{H}^{(4,0)}$ is equal to $H^{(4)}_{WB}$
 appearing in \eqref{wnormalform1}.
By
Proposition
 \ref{fineIdentif}
 (see the first equality in \eqref{equivalenzaZZZ})
  and the expression \eqref{HamZakDyaCraig}
  we get the \eqref{theoBirH}.

 \section{The nonlinear functional setting}\label{sec:5}
 In this section we consider the Hamiltonian 
 $\mathcal{H}_{\mathbb{C}}$ in \eqref{piotta}
 obtained after the weak Birkhoff normal form 
 procedure of section \ref{sec:weak}.
 We show that actually  $\mathcal{H}_{\mathbb{C}}$
 possesses \emph{approximately} invariant manifolds foliated by tori supporting quasi-periodic motions.
 To do this, we first introduce
 action-angle coordinates in the neighborhood of such manifolds.
 Lemma \ref{Twist1} guarantees that 
 the 
dynamics on  the manifold is non-isochronous. 
In subsection \ref{sezioneNonlinearFunct}  we introduce a 
nonlinear functional 
%(see \eqref{NonlinearFunctionalDP})
whose
zeros are quasi-periodic solutions for the whole Hamiltonian 
 $\mathcal{H}_{\mathbb{C}}$.
 This is stated in Theorem \ref{IlTeoremaDP} which will imply our main result.

\subsection{Action-angle variables}\label{sezioneActionAngle}

On the submanifold $\{ z=0\}$ we put the following action-angle variables
\begin{equation}\label{aa0dp}
\begin{aligned}
\mathbb{T}^{\nu}&\times \mathbb{R}_{+}^{\nu} \longrightarrow \{ z=0\}\,,
\qquad
(\theta, I) \longmapsto v=\sum_{j\in S} \sqrt{I_j} \,e^{-\mathrm{i} \theta_j}\,
e^{\mathrm{i} j x}\,,
\end{aligned}
\end{equation}
with $\mathbb{R}^{\nu}_{+}:=(0,\infty)^{\nu}$.
The symplectic form in \eqref{symReal} (or equivalently  \eqref{symComp} 
in the complex coordinates)
restricted to the subspace $H_S$ (see \eqref{decomposition})  
transforms 
into the $2$-form
$\sum_{j\in S} d{I_{j}}\wedge d\theta_j$.
Recalling \eqref{piotta},
we have that the Hamiltonian 
$\mathcal{H}_{\mathbb{C}}^{(\leq 14,0)}(\theta, I, 0)
=\sum_{j\in S} |j|^{\frac{1}{2}} I_j
+\mathcal{H}_{\mathbb{C}}^{(4, 0)}(I)
+\sum_{k=3}^{7}\mathcal{H}_{\mathbb{C}}^{(2k, 0)}(I)$
depends only by the actions $I$ and its equations of motion read as
\begin{equation}\label{HamiltonianSisteminoDP}
\left\{
\begin{aligned}
\dot{\theta}_j&=\partial_{I_j} 
\mathcal{H}_{\mathbb{C}}^{(\leq 14,0)}(\theta, I, 0)
=|j|^{\frac{1}{2}}+\sum_{k\in S}\mathbb{A}_{j}^{k}I_{j}+O(I^2)\,, \\
\dot{I}_j&=-\partial_{\theta_j} \mathcal{H}_{\mathbb{C}}^{(\leq 14,0)}(\theta, I, 0)=0\,, 
\end{aligned}\right.\qquad j\in S\,,
\end{equation}
%where
%\begin{equation*}%\label{omeghinoDP}
%\pa_{I_j}\mathcal{H}_{\mathbb{C}}^{(\leq 14,0)}(\theta,I,0)=
%\og(j)+\sum_{k\in S}\mathbb{A}_{j}^{k}I_{j}+O(I^2)\,,\quad j\in S\,,
%\end{equation*}
%and 
where the matrix $\mathbb{A}$ is the $\nu\times\nu$ symmetric matrix
associated to the quadratic form (see \eqref{theoBirH}) 
\begin{equation}\label{parco1}
\begin{aligned}
\frac{1}{2}\mathbb{A}I\cdot I:=\mathcal{H}_{\mathbb{C}}^{(4,0)}(I) &:=
\frac{1}{4 \pi} \sum_{k \in S} |k|^3  
  I_{k}^2  
+  \frac{1}{\pi}\sum_{\substack{k_1, k_2 \in S,  \, \sign(k_1) = \sign( k_2 )  \\ |k_2| < |k_1|}} 
|k_1| |k_2|^2  
  I_{k_1}  I_{k_2}\,.
\end{aligned}
\end{equation}
Let 
$0<\varepsilon\ll1$ and  rescale 
$I\mapsto \varepsilon^2 I$, 
so that the \emph{frequency-amplitude map} 
can be written as
\begin{equation}\label{FreqAmplMapDP}
 \alpha(I)=\overline{\omega}+\e^2\mathbb{A}\,I+O(\e^4)\,,
\qquad 
O(\e^4)= \sum_{k=2}^{6} \e^{2k}\mathtt{b}_{k}
\big(\underbrace{I,\ldots,I}_{k-{\rm times}}\big)\,,
\end{equation}
where $\overline{\omega}$ is the linear 
frequencies vector in  \eqref{LinearFreqDP}, 
and where $\mathtt{b}_{k}$ are $k$-linear functions
of the variables $I$.

In order to work in a small neighborhood of the 
prefixed  torus $\{ I\equiv \z \}$, with $\zeta\in \mathbb{R}_{+}^{\nu}$, 
 it is advantageous to introduce a set of coordinates 
$(\theta, y, z,\bar{z})\in\mathbb{T}^{\nu}\times
 \mathbb{R}^{\nu}\times H_S^{\perp}\times H_S^{\perp}$ 
adapted to it, defined by
\begin{equation}\label{AepsilonDP}
u=A_{\varepsilon}(\theta, y, z)=
\e v_\e(\theta,y)+\e^{b}z \quad \Longleftrightarrow \quad
\begin{cases}
u_j:=\varepsilon\sqrt{ \z_j+\varepsilon^{2b-2} y_j }\,
e^{-\mathrm{i}\theta_j}\,e^{\mathrm{i}\,j\,x}\,, \,\,\quad j\in S\,,\\[2mm]
u_j:=\varepsilon^b z_j\,,  \qquad \qquad \qquad \qquad\qquad j\in S^c\,,
\end{cases}
\end{equation}
with $b>1$ chosen as
\begin{equation}\label{donnapia}
b:=1+\frac{a}{2}\,,\qquad 0<a\ll1\,.
%a := 2b -2\,,
\end{equation}
%and fix $a>0$ appropriately small.
We define
\begin{equation}\label{AepsilonDP2}
\vect{u}{\bar{u}}:={\bf A}_{\e}(\theta,y,z,\bar{z}):=
\left(
\begin{matrix}A_{\varepsilon}(\theta, y, z)\vspace{0.2em}\\
\ov{A_{\varepsilon}(\theta, y, z)}
\end{matrix}
\right)\,.
\end{equation}
For the tangential sites 
$S:=\{ \overline{\jmath}_1,\dots, \overline{\jmath}_{\nu}\}$  
(see \eqref{TangentialSitesDP})
we will also denote 
$\theta_{\overline{\jmath}_i}:=\theta_i$, $y_{\overline{\jmath}_i}:=y_{i}$, 
 $\zeta_{\overline{\jmath}_i}:=\zeta_i$, 
$ i=1,\dots, \nu$.
{After the time-rescaling $t\mapsto \e^{2b} t$} we can consider the symplectic $2$-form 
%(recall $\Omega$ in \eqref{symComp}) 
\begin{equation}\label{NewSimplFormDP}
\mathcal{W}:=\sum_{i=1}^{\nu} d y_i\wedge d\theta_i
-\ii \sum_{j\in S^c}\,d z_j\wedge d \ov{z_{j}}
=\Big( \sum_{i=1}^{\nu} d y_i\wedge d \theta_i \Big) 
\oplus \Omega_{S^{\perp}}\,,
%=d \mathcal{V}\,,
\end{equation}
where $\Omega_{S^{\perp}}$ is the symplectic form
 $\Omega$ %in \eqref{symComp}
restricted to the subspace $H_{S}^{\perp}$ in 
\eqref{decomposition}.

\begin{remark}
Note that $\mathcal{W}$ in \eqref{NewSimplFormDP}
has the form \eqref{formaestesa} 
when $(h,\bar{h})$ is restricted to the subspace 
$H^{\perp}_{S}$.
Moreover, in the real coordinates, the symplectic form 
$\mathcal{W}$ in \eqref{NewSimplFormDP} reads 
\begin{equation}\label{NewSimplFormDPreal}
\widetilde{\mathcal{W}}:=\sum_{i=1}^{\nu} d y_i\wedge d\theta_i
+ \sum_{j\in S^c}\,d \psi_j\wedge d \eta_{-j}
=\Big( \sum_{i=1}^{\nu} d y_i\wedge d \theta_i \Big) 
\oplus \widetilde{\Omega}_{S^{\perp}}=d \mathcal{V}\,,
\end{equation}
where $\widetilde{\Omega}_{S^{\perp}}$ is the symplectic form
 $\widetilde{\Omega}$ in \eqref{symReal}
restricted to the subspace $H_{S}^{\perp}$ in 
\eqref{decomposition}
and $\mathcal{V}$ is the contact $1$-form on 
$\mathbb{T}^{\nu}\times\mathbb{R}^{\nu}\times H_S^{\perp}$ 
defined by $\mathcal{V}_{(\theta, y, W)}\colon 
\mathbb{R}^{\nu}\times\mathbb{R}^{\nu}\times 
H_S^{\perp}\to \mathbb{R}$,
\begin{equation}\label{1FormDP}
\mathcal{V}_{(\theta, y, W)} [\hat{\theta}, \hat{y}, \hat{W}]
:=y\cdot \hat{\theta}+\frac{1}{2} (J W,\hat{W})_{L^2}.
\end{equation}
 \end{remark}
 
 \noindent
The Hamiltonian  $\mathcal{H}_{\mathbb{C}}$ in 
\eqref{piotta} becomes
\begin{equation}\label{HamiltonianaRiscalataDP}
H_{\mathbb{C},\varepsilon}:=
\varepsilon^{-2 b}\,\mathcal{H}_{\mathbb{C}}\circ {\bf A}_{\varepsilon}\,.
\end{equation}

In the following lemma we prove that, 
under an appropriate choice of the tangential set 
\eqref{TangentialSitesDP}, the function \eqref{FreqAmplMapDP} 
is a diffeomorphism for $\varepsilon$ 
small enough  and then the system 
\eqref{HamiltonianSisteminoDP} 
is integrable and \emph{non-isochronous}.

\begin{lemma}{\bf (Twist condition).}\label{Twist1}
For a generic choice of the set $S$ 
 one has
$\lvert \det \mathbb{A} \rvert\geq \frac{1}{(4\pi)^{\nu}}$.
\end{lemma}

\begin{proof}
%The matrix  $\mathbb{A}$ is
%the $\nu\times\nu$ matrix
%associated to the quadratic form 
%\eqref{parco1} with $I=(I_j)_{j\in S}$, $I_j:=\lvert u_j \rvert^2$
%and recall the form of $S$ in 
%\eqref{TangentialSitesDP}, \eqref{TangentialSitesDPbis}.
Recall \eqref{parco1}. The entries of $\mathbb{A}$ are homogeneous 
polynomials of the tangential sites, 
and so it is the $\det \mathbb{A}$. 
We claim that $\det \mathbb{A}$ is a
\emph{not identically} zero polynomial. Hence the result follows since $4\pi \mathbb{A}$ is an integer matrix and then
$(4\pi)^{\nu}\det \mathbb{A}$ is a non-zero integer number.
To prove the claim we consider the matrix 
$\mathbb{A}(\overline{\jmath}_i)$ 
restricted at $\overline{\jmath}_i=\lambda$, $i=1,\dots, \nu$, 
for some $\lambda>0$. Such matrix has the form
\begin{equation}\label{georgia}
\mathbb{A}(\lambda, \dots, \lambda)=
\lambda^{3} \pi^{-1}D
\end{equation}
where $D$ is a $\nu\times\nu$ matrix
whose diagonal elements are all $1/4$
and the appearing out of diagonal are just $0$ and $1/2$.
%Consider the elementary matrix
%$M:=4 {\rm Id}$ where ${\rm Id}$ is the  
%identity matrix of $\mathbb{R}^{\nu\times \nu}$. 
Then the matrix
\begin{equation}\label{georgia2}
4\,\pi \lambda^{-3} \,\mathbb{A}(\lambda, \dots, \lambda)=4D\,
\end{equation}
is the identity in $\mathbb{Z}/2\mathbb{Z}$ 
and in particular its determinant is $1$. 
This means that the determinant of the matrix in \eqref{georgia2}
is an odd number, in particular it is different from zero. This concludes the proof.
%Since $M$ is invertible we proved that the determinant of $\mathbb{A}$
%is different form zero.
%Therefore $\det \mathbb{A}$ is not an identically zero polynomial.
\end{proof}
As a consequence of Lemma \ref{Twist1}
the map in \eqref{FreqAmplMapDP} is invertible and we denote
\begin{equation}\label{xiomega}
\z:=\z(\omega) := \alpha^{(-1)}(\omega) = 
\e^{-2}\mathbb A^{-1}(\omega-\bar\omega) +O( \e^2)\,.
\end{equation}

\noindent
In the following it will be convenient to consider the Hamiltonian 
$H_{\mathbb{C},\e}$ in \eqref{HamiltonianaRiscalataDP} 
expressed in terms of  real variables.
Recall the splitting in \eqref{splitto2}. Then the variables $(\eta,\psi)$ in the 
action-angle variables in \eqref{AepsilonDP}, \eqref{AepsilonDP2} read
\begin{equation}\label{AepsilonDPreal}
\left(\begin{matrix}
\eta \\ \psi
\end{matrix}\right):=\Lambda^{-1}{\bf A}_{\e}\Big(\theta,y,
\Lambda\vect{\widetilde{\eta}}{\widetilde{\psi}}\Big)
\end{equation}

\noindent
Using this notation, the Hamiltonian 
$H_{\mathbb{C},\e}(\theta,y,z,\bar{z})$
in \eqref{HamiltonianaRiscalataDP} 
expressed in terms
of real variables is
\begin{equation}\label{HamiltonianaRiscalataDPreal}
H_{\e}(\theta,y,\widetilde{\eta},\widetilde{\psi}):=
\varepsilon^{-2 b}\,\mathcal{H}_{\mathbb{C}}\circ 
{\bf A}_{\varepsilon}\big(\theta,y,
\Lambda\vect{\widetilde{\eta}}{\widetilde{\psi}}\big)
=
\varepsilon^{-2 b}\,{H}\circ\Lambda^{-1}
\circ \Phi_{B}\circ
{\bf A}_{\varepsilon}\big(\theta,y,\Lambda\vect{\widetilde{\eta}}{\widetilde{\psi}}\big)
\end{equation}
where we used also 
\eqref{piotta} and \eqref{Rinco}.
In the following we shall need a homogeneous expansion of the 
Hamiltonian $H_{\e}(\theta,y,\widetilde{\eta},\widetilde{\psi})$ 
in \eqref{HamiltonianaRiscalataDPreal} in the variables 
$y, \widetilde{\eta},\widetilde{\psi}$.
Notice that the map $\Lambda$ in \eqref{CVWW}
does not affect the homogeneity degree in the normal variables. For simplicity we expand the Hamiltonian 
$H_{\mathbb{C},\varepsilon}$ in degree of homogeneity in $y$ and in the complex
variables $z,\bar{z}$ (see \eqref{splitto2}).
In order to lighten the notation \eqref{notaTRA}
we shall write
\[
R(v^{k-q}z^q)=R_k(v^{\alpha_1}\, \bar{v}^{\beta_1}\, z^{\alpha_2}\,\bar{z}^{\beta_2})\,,
\qquad \alpha_1+\beta_1=k-q\,,\quad \alpha_2+\beta_2=q\,.
\]
We also define (see \eqref{FreqAmplMapDP})
\begin{equation}\label{Bi}
\mathbb{B}(\zeta)y:=\sum_{k=2}^{7}\e^{2(k-1)}\mathtt{b}_{k}
\big(\underbrace{\zeta,\ldots,\zeta}_{(k-1)-{\rm times}},y\big)\,.
\end{equation}
Hence, by \eqref{Rinco}, \eqref{piotta}, 
and writing $v_{\varepsilon}:=v_{\varepsilon}(\theta, y)$, we have that
\begin{equation}\label{HamiltonianaRiscalataDP2}
\begin{aligned}
&H_{\mathbb{C},\varepsilon}(\theta, y, z,\bar{z})=e(\zeta)+\alpha(\zeta)\cdot y+\int_{\mathbb{T}} |D|^{\frac{1}{2}}z\cdot \bar{z}\,dx
+\frac{\varepsilon^{2\,b}}{2}\,\big(\mathbb{A}+\mathbb{B}(\zeta)\big)\, y\cdot y
+\varepsilon^{4 b} O(y^3)
\\&\qquad+
\sum_{k=1}^{15} \e^{k} R(v_{\e}^{k}z^2)
+
\sum_{k=3}^{15}
\varepsilon^{k-3+b}\sum_{q=3}^{k} \varepsilon^{(b-1)(q-3)}R(v^{k-q}z^q)
+\varepsilon^{-2 b} \mathcal{H}^{(\geq 16)} (\varepsilon v_{\varepsilon}
+\varepsilon^b z)
\end{aligned}
\end{equation}
where $e(\zeta)$ is a constant, $\alpha(\zeta)$ is the rescaled 
frequency-amplitude map \eqref{FreqAmplMapDP} and $\mathbb{A}$ is 
in \eqref{parco1}. 
%defined as 
%\begin{equation*} %\label{FreqAmplMapDPcorre}
%\mathbb{B}(\zeta)y:=\sum_{k=2}^{7}\e^{2(k-1)}\mathtt{b}_{k}
%\big(\underbrace{\zeta,\ldots,\zeta}_{(k-1)-{\rm times}},y\big)\,.
%\end{equation*}

\subsection{The functional equation 
%The nonlinear functional setting
}\label{sezioneNonlinearFunct}

We write the Hamiltonian in \eqref{HamiltonianaRiscalataDPreal}
(possibly eliminating  constant terms depending only on $\zeta$ 
which are irrelevant for the dynamics) as
\begin{equation} \label{HepsilonDP}
H_{\varepsilon}=\mathcal{N}+P\,, \qquad 
W:=\begin{pmatrix} \widetilde{\eta}\\ \widetilde{\psi} \end{pmatrix}
\end{equation}
\begin{equation*}
\mathcal{N}(\theta, y, W)
=\alpha(\zeta)\cdot y+\frac{1}{2} (N(\theta) W, W)_{L^2}\,,
\quad 
\frac{1}{2}(N(\theta) W, W)_{L^2}:=
\frac{1}{2}((\mathtt{d} \nabla_{W} H_{\varepsilon}) (\theta, 0, 0)[W], W)_{L^2}\,,
\end{equation*}
where $\mathcal{N}$ describes the 
linear dynamics normal to the torus, 
and $P:=H_{\varepsilon}-\mathcal{N}$ collects the nonlinear perturbative effects.
Note that both $\mathcal{N}$ and $P$ depend on $\omega$ through the map $\omega\mapsto \zeta(\omega)$.\\
We consider $H_{\varepsilon}$ as a $(\omega, \varepsilon)$-parameter 
family of Hamiltonians and we note that, for 
$P=0$, $H_{\varepsilon}$ possess 
an invariant torus at the origin with frequency $\omega$, 
which we want to continue to an invariant 
torus for the full system.

\noindent
We will select the frequency parameters from the 
following set (recall \eqref{xiomega})
\begin{equation}\label{OmegaEpsilonDP}
\Omega_{\varepsilon}:=\{\omega\in \R^\nu\,: \,\, \zeta(\omega)\in [1, 2]^{\nu}\}\,.
\end{equation}
Set (see \eqref{donnapia})
%(see \eqref{AepsilonDP})
\begin{equation}\label{gammaDP}
\gamma=\varepsilon^{2b}\,, \quad \tau:=3\nu+7\,
%\quad b=1+a\,, 
%\quad 
%0<a\ll1\,,
\end{equation}
and consider two constants $\mathtt{C}_1, \mathtt{C}_2>0$ depending on $S$.
 We define the non-resonant sets
 \begin{equation}\label{ZEROMEL}
 \begin{aligned}
\mathcal{G}^{(0)}_0&:=
\big\{ \omega \in \Omega_{\varepsilon} : 
\lvert \omega\cdot \ell \rvert\geq \gamma\,\langle \ell \rangle^{-\tau}, 
\,\, \forall \ell \in\mathbb{Z}^{\nu}\setminus\{0\}\big\}\,,\\
\mathcal{G}_0^{(1)}&:=
\Big\{ \omega \in \Omega_{\varepsilon} : 
\lvert (\overline{\omega}-\varepsilon^2\,m_1\,\mathtt{v})\cdot \ell +\varepsilon^2 \mathbb{A} \zeta(\omega) \cdot \ell\rvert\geq \gamma\,\langle \ell \rangle^{-\tau}, 
\,\, \forall \ell \in\mathbb{Z}^{\nu}\setminus\{0\}\Big\}\,,
\end{aligned}
\end{equation}
\begin{equation}\label{G02}
\begin{aligned}
\mathcal{G}_0^{(2)}(\mathtt{C}_1, \mathtt{C}_2)
:=\Big\{ &\omega\in \Omega_{\e} :
\lvert (\overline{\omega}-\varepsilon^2 m_1\,\mathtt{v}
+\varepsilon^2 \mathbb{A}\zeta(\omega))\cdot \ell+\sigma \sqrt{|j|}
-\sigma'\sqrt{|k|}\rvert\geq \,\frac{\gamma}{\langle \ell \rangle^{\tau}},\\
& \,\,\, \lvert \ell \rvert\le \mathtt{C}_1,\,\,\max\{ \lvert j \rvert, \lvert k \rvert \}\le \mathtt{C}_2, 
\,\,\, \mathtt{v} \cdot \ell+j-k=0,\,\, \s, \s'=\pm \Big\}\,
\end{aligned}
\end{equation}
where $\mathbb{A}$ is defined in \eqref{parco1}
and
\begin{equation}\label{def:m1}
m_1=m_1(\omega):=w\cdot \zeta(\omega), \qquad w:=(w_j)_{j\in S},\,\,\,w_j:=\lvert j \rvert j\,,
\qquad
\mathtt{v}:=(\mathtt{v}_j)_{j\in S},\,\,\,\mathtt{v}_j:= j\,.
\end{equation}
We require that
\begin{equation}\label{0di0Melnikov}
\omega\in\mathcal{G}_0:=\mathcal{G}^{(0)}_0\cap \mathcal{G}^{(1)}_0\cap \mathcal{G}_0^{(2)}\,.
\end{equation}

\begin{lemma}\label{measG0}
For generic choices of $S$ and any constants $\mathtt{C}_1, \mathtt{C}_2>0$ depending only on $S$,
we have that 
$\lvert \Omega_{\varepsilon}\setminus\mathcal{G}_0 \rvert
\leq C_{*}\varepsilon^{2(\nu-1)} \gamma\,$ 
for some $C_*=C_*(S)>0\,$.
\end{lemma}

\begin{proof}
The estimate for $\Omega_{\varepsilon}\setminus \mathcal{G}^{(0)}_0$
follows using \eqref{FreqAmplMapDP}, \eqref{xiomega} and Lemma \ref{Twist1}.
Consider now the set $\mathcal{G}^{(1)}_0$.
We have that
\begin{equation}\label{dimaio}
\Omega_{\varepsilon}\setminus \mathcal{G}^{(1)}_0
=\bigcup_{\ell\in\mathbb{Z}^{\nu}} T_{\ell}\,, 
\quad T_{\ell}:=
\{ \omega\in \Omega_{\varepsilon} : \lvert \phi(\omega, \ell)\rvert
< \gamma\,\langle \ell \rangle^{-\tau} \}\,, 
\quad \phi(\omega, \ell):=(\overline{\omega}-\varepsilon^2\,m_1\,\mathtt{v})\cdot \ell 
+\varepsilon^2 \mathbb{A} \zeta(\omega) \cdot \ell\,.
\end{equation}
First we write
\begin{equation}\label{matriceV}
m_1\,\mathtt{v}=\mathbb{V} \zeta, \quad \mathbb{V}:=\mathtt{v}\,w^T.
\end{equation}
Then $
\nabla_{\zeta} \phi(\omega, \ell)=  (\mathbb{A}^T-\mathbb{V}^T)\ell$. 
Now we prove that, for generic choices of $S$, 
we have 
\begin{equation}\label{twistcondition2}
\det(\mathbb{A}-\mathbb{V})\neq 0.
\end{equation}
Since $\mathbb{V}$ and $\mathbb{A}$ are integer matrices 
and $\ell\in\mathbb{Z}^{\nu}$ we will have 
that $\min \{\lvert \partial_{\zeta_i} \phi(\omega, \ell) \rvert : 
i=1, \dots, \nu\}\geq 1$. 
Now we compute the matrix 
$\mathbb{A}-{\mathbb{V}}$ evaluated at the point 
$(\lambda, \dots, \lambda)$ for some $\lambda\in\mathbb{N}$. 
At such point
$\mathbb{V}$
 has all entries equal to $\lambda^3$.
We write
%\[
%4\,\pi \lambda^{-3}  (\mathbb{A}-{\mathbb{V}})=\mathtt{A}-\pi \mathtt{V}\,,
%\qquad
%\mathtt{V}:=4\pi \lambda^{-3}  {\mathbb{V}}\,,
%\quad 
% \mathtt{A}:=4\pi \lambda^{-3}  D\,.
%\]
\[
{ 4\,\pi \lambda^{-3}  (\mathbb{A}-{\mathbb{V}}) =\mathtt{A}-\pi \mathtt{V}, \qquad \mathtt{A}:=4\,\pi \lambda^{-3}  \mathbb{A}, \quad \mathtt{V}:= 4\lambda^{-3}  {\mathbb{V}}}.
\]
Notice that $\mathtt{V}$, $\mathtt{A}$ are
integer matrices (recall \eqref{georgia}, \eqref{georgia2}  and the definition of $D$ ) and $\mathtt{A}$ is invertible (see proof of Lemma \ref{Twist1}).
Then $\mathtt{A}-\pi \mathtt{V}=\mathtt{A}(\mathrm{I}_N-\pi \mathtt{A}^{-1}\mathtt{V})$. The eigenvalues of $\mathtt{A}^{-1}\mathtt{V}$ are algebraic numbers, hence $1$ cannot be an eigenvalue of $\pi \mathtt{A}^{-1}\mathtt{V}$. This proves that also $\mathtt{A}-\pi \mathtt{V}$ is invertible.
Therefore the determinant of $\mathtt{A}-\pi \mathtt{V}$
is not an identically zero polynomial.
Hence
\[
|\Omega_{\varepsilon}\setminus \mathcal{G}^{(1)}_0|\lesssim
\sum_{\ell\in \mathbb{Z}^{\nu}}|T_{\ell}|\lesssim
\gamma\sum_{\ell\in \mathbb{Z}^{\nu}}\langle \ell\rangle^{-\tau}\lesssim\gamma\,,
\]
since $\tau> \nu/2$ (see \eqref{gammaDP}).
Now consider the set in \eqref{G02} and write
\begin{equation}\label{dimaio2}
\begin{aligned}
&\Omega_{\varepsilon}\setminus\mathcal{G}_0^{(2)}(\mathtt{C}_1, \mathtt{C}_2)=\bigcup_{\substack{\ell\in \mathbb{Z}^{\nu},\s,\s'=\pm,\\
|\ell|<\mathtt{C}_1}}
\bigcup_{\substack{j,k\in S^{c},\\
\max\{ \lvert j \rvert, \lvert k \rvert \}\le \mathtt{C}_2}}R_{\ell,j,k}^{\s,\s'}\,,\\
R_{\ell,j,k}^{\s,\s'}&:=\{
\omega\in \Omega_{\varepsilon} : \lvert \psi(\omega, \ell,j,k)\rvert
< \gamma\,\langle \ell \rangle^{-\tau}\,,
\mathtt{v} \cdot \ell+j-k=0
\}\,,\\
%&\quad 
\psi(\omega, \ell,j,k)&:=
(\overline{\omega}-\varepsilon^2 m_1\,\mathtt{v}
+\varepsilon^2 \mathbb{A}\zeta)\cdot \ell+\sigma \sqrt{|j|}
-\sigma'\sqrt{|k|}\,.
\end{aligned}
\end{equation}
Reasoning as done for the function $\phi$ in \eqref{dimaio}
we deduce that 
$\min \{\lvert \partial_{\zeta_i} \psi(\omega, \ell,j,k) \rvert : 
i=1, \dots, \nu\}\geq 1$. Hence the 
measure of a single bad set is bounded as $|R_{\ell,j,k}^{\s,\s'}|
\lesssim\gamma\langle\ell\rangle^{-\tau}$. By \eqref{dimaio2} 
we deduce
\[
|\Omega_{\varepsilon}\setminus\mathcal{G}_0^{(2)}(\mathtt{C}_1, \mathtt{C}_2)|\lesssim
\mathtt{C}_2\sum_{\ell\in\mathbb{Z}^{\nu}}\gamma
\langle \ell\rangle^{-\tau}\lesssim \mathtt{C}_2\gamma\,.
\]
This implies the thesis.
\end{proof}

We look for an embedded invariant torus
\begin{equation}\label{iDP}
i \colon \mathbb{T}^{\nu}\to 
\mathbb{T}^{\nu}\times \mathbb{R}^{\nu}\times H_S^{\perp}\,, 
\quad \varphi \mapsto i(\varphi):=(\theta(\varphi), y(\varphi), W(\varphi))
\end{equation}
\[
\mathfrak{I}(\varphi):=i(\varphi)-(\varphi, 0, 0)
:=(\Theta(\varphi), y(\varphi), W(\varphi))
\]
of the Hamiltonian vector field $X_{H_{\varepsilon}}$ 
(see \eqref{HepsilonDP}, \eqref{HamiltonianaRiscalataDPreal}) 
supporting quasi-periodic solutions with 
diophantine frequency $\omega\in \mathcal{G}_0$ 
(recall \eqref{0di0Melnikov}),
satisfying the following condition:
%\comment{non capisco questa condizione}
\begin{itemize}
\item {\bf (Travelling wave):} 
the torus embedding $i(\vphi)=(\varphi, 0, 0)+\mathfrak{I}(\varphi)$ satisfies 
\begin{equation}\label{trawaves}
\mathtt{v}\cdot\pa_{\vphi}\Theta(\vphi)=0\,,\quad 
\mathtt{v}\cdot\pa_{\vphi}y(\vphi)=0\,,\quad 
(\mathtt{v}\cdot\pa_{\vphi}+\pa_{x})W(\vphi, x)=0\,,
\end{equation}
where %$\Lambda $ is defined in \eqref{CVWW} and 
$\mathtt{v}$ is the vector velocity given in \eqref{def:m1}.
\end{itemize}
We remark that the embedding $i(\vphi)$ % as in \eqref{iDP}, 
in the original coordinates (see \eqref{CVWW}, \eqref{FBI}, \eqref{AepsilonDP2}),  
reads
\begin{equation}\label{trawaves2}
(\eta(t, x), \psi(t, x))=\Big(\Lambda^{-1}
 \Phi_{B}
{\bf A}_{\varepsilon}\Big) (\theta(\omega t), y(\omega t), \Lambda W(\omega t)) \,.
%\in S_{\mathtt{v}}
\end{equation}
\begin{remark}
The condition \eqref{trawaves} on the embedding $i(\vphi)$
is equivalent to require that 
%By the assumption \eqref{trawaves}
%it is easy to note that 
$(\eta,\psi)$ in \eqref{trawaves2} belongs to
the subspace
$S_{\mathtt{v}}$ defined in 
\eqref{funzmom}. \\
%This follows by an 
%explicit computation
%using \eqref{aa0dp}. 
Consider the momentum Hamiltonian 
in \eqref{Hammomento}
expressed in the action-angle variables \eqref{aa0dp} 
(see also \eqref{splitto2})
\begin{equation}\label{trawavesmom}
M=M(\theta,y,W)=-\mathtt{v}\cdot y
+\int_{\mathbb{T}}\tilde{\eta}_{x}\tilde{\psi}dx\,,
\quad 
W=\left(\begin{matrix}\tilde{\eta} \\ \tilde{\psi} 
\end{matrix}\right)\,.
\end{equation}
Then the condition \eqref{trawaves}
 can be written as
 \begin{equation}\label{cifrato}
 \mathtt{v}\cdot\pa_{\vphi}i(\vphi)+X_{M}(i(\vphi))=0\,,
 \qquad
 X_{M}(i(\varphi))=(-\mathtt{v},0, W_x)\,,
 \end{equation}
 or equivalently
 \begin{equation}\label{momfrakI}
 \mathtt{v}\cdot\pa_{\vphi}\mathfrak{I}(\vphi)+(0, 0, W_x)=0 .
 \end{equation}
 \end{remark}
 
\noindent
For technical reason, it is useful to consider the modified Hamiltonian
\begin{equation}\label{HepsilonZetaDP}
H_{\varepsilon, \Xi}(\theta, y, W):=H_{\varepsilon}(\theta, y, W)
+\Xi\cdot\theta, \quad \Xi\in\mathbb{R}^{\nu}\,.
\end{equation}
More precisely, we introduce $\Xi$ in order to 
control the average in the $y$-component 
in our Nash Moser scheme.
The vector $\Xi$ has no dynamical 
consequences since an invariant torus 
for the Hamiltonian vector field 
$X_{H_{\varepsilon, \Xi}}$ is actually 
invariant for $X_{H_{\varepsilon}}$ itself (see Lemma \ref{Lemma6.1DP}).

%\begin{remark}
%We remark that the Hamiltonian $H_{\varepsilon}(\theta, y, W)
%+\Xi\cdot\theta$ commutes with the momentum \eqref{trawavesmom}
%\end{remark}

Thus, we look for zeros of the nonlinear operator 
$\mathcal{F}(i,\Xi)\equiv \mathcal{F}(i, \Xi, \omega, \varepsilon):=
\omega\cdot\partial_{\varphi} i(\varphi)-X_{\mathcal{N}}(i(\varphi))-X_P(i(\varphi))+(0, \Xi, 0)$ 
defined as
\begin{align}\label{NonlinearFunctionalDP}
\mathcal{F}(i, \Xi)
=\begin{pmatrix}
\omega\cdot\partial_{\varphi} \Theta(\varphi)-\partial_y P(i(\varphi))\\
\omega\cdot\partial_{\varphi} y(\varphi)
+\frac{1}{2} \partial_{\theta} (N(\theta(\varphi))W(\varphi),W(\vphi))_{L^2(\mathbb{T})}
+\partial_{\theta} P(i(\varphi))+\Xi\\
\omega\cdot\partial_{\varphi} W(\varphi)
-J N(\theta(\varphi))\,W(\varphi)-J\nabla_WP(i(\varphi))
\end{pmatrix}  
\end{align}
where $\Theta(\varphi):=\theta(\varphi)-\varphi$ is $(2\pi)^{\nu}$-periodic
and $J$ is defined in \eqref{symReal}. 

\begin{lemma}\label{Fpreservatrav}
If $i(\varphi)$ is traveling then $\mathcal{F}(i(\varphi), \Xi)$ is traveling.
\end{lemma}
\begin{proof}
We observe that $\mathcal{F}(i, \Xi)=\omega\cdot \partial_{\varphi} i-X_{H_{\varepsilon}}(i)+X_{\Xi\cdot \theta}(i)$. We show that $\mathcal{F}(i, \Xi)$ solves \eqref{momfrakI}. We have
\[
\mathtt{v} \cdot \partial_{\varphi} X_{H_{\varepsilon}}(i)=- d X_{H_{\varepsilon}}(i)[(-\mathtt{v}, 0, W_x)]=- d X_{H_{\varepsilon}}(i)[(0, 0, W_x)]+\mathtt{v}\cdot\partial_{\theta} X_{H_{\varepsilon}}(i)=- d_W X_{H_{\varepsilon}}(i)[W_x]+\mathtt{v}\cdot\partial_{\theta} X_{H_{\varepsilon}}(i)
\]
\[
 (0, 0, \partial_x)(X_{H_{\varepsilon}}(i))=(0, 0, d_W X^{(W)}_{H_{\varepsilon}}(i)[W_x])
\]
then
\[
(\mathtt{v}\cdot \partial_{\varphi}+(0, 0, \partial_x)) X_{H_{\varepsilon}}(i)=
dX_{M}(i)[X_{H_\e}(i)]-d X_{H_\e}(i)[X_{M}(i)]=
[X_M,X_{H_\e}](i)
=X_{\{M, {H_{\varepsilon}}\}}(i)=0\,,
\]
where the Poisson brackets are with respect to the symplectic form 
\eqref{NewSimplFormDPreal}.
Eventually it is easy to see that
\[
(\mathtt{v} \cdot \partial_{\varphi}+ (0, 0, \partial_x))(X_{\Xi\cdot \theta}(i))=
X_{\{M,\Xi\cdot\theta\}}=X_{-\mathtt{v}\cdot\Xi}=0\,.
\]
Since $\omega\cdot \partial_{\varphi}$ commutes with $\mathtt{v}\cdot \partial_{\varphi}$ and $X_M$, then $\omega\cdot \partial_{\varphi}$ satisfies the \eqref{momfrakI}.

\end{proof}

We define the Sobolev norm of the periodic component of the embedded torus
\begin{equation}\label{frakkiIDP}
 \lVert \mathfrak{I} \rVert_{s}:=\lVert \Theta \rVert_{s}
+\lVert y \rVert_{s}+\lVert W \rVert_s\,,
\end{equation}
where $ W \in H^s_{S^{\perp}}:= H^s \cap H_{S}^\perp$ 
(recall \eqref{decomposition}) with norm defined in \eqref{space} 
and  with abuse of notation, we are denoting by 
$\|\cdot\|_{s}$ the Sobolev norms of functions in 
$H^{s}(\T^{\nu};\R^{\nu})$. 
From now on we fix 
\[
s_0:=[\nu/2]+4.
\]
We say that $\mathfrak{I}$ is traveling if $\mathfrak{I}+(\varphi, 0, 0)$ satisfies the \eqref{trawaves}.\\
Notice that in the coordinates \eqref{AepsilonDP}, 
 a quasi-periodic solution corresponds to 
 an embedded invariant torus \eqref{iDP}. 
 Therefore we can reformulate the main Theorem \ref{thm:main} as follows.
\begin{theorem}\label{IlTeoremaDP}
Let $\nu\geq 1$.
For any generic choice of $S$ (see \eqref{TangentialSitesDP})
%and Def. \ref{Generic})
there exists $\varepsilon_0>0$  small enough, 
such that the following holds.
For all $\varepsilon\in (0, \varepsilon_0)$ 
there exist positive constants 
$C=C(\nu)$, $\mu_1=\mu_1(\nu)$ and 
a Cantor-like set 
$\mathcal{C}_{\varepsilon}\subseteq\Omega_{\varepsilon}$ 
(see \eqref{OmegaEpsilonDP}), 
with asymptotically full measure as 
$\varepsilon\to 0$, namely
\begin{equation}\label{frazionemisureDP}
\lim_{\varepsilon \to 0}
\dfrac{\lvert \mathcal{C}_{\varepsilon} \rvert}{\lvert \Omega_{\varepsilon} \rvert}=1\,,
\end{equation}
such that, for all $\omega\in\mathcal{C}_{\varepsilon}$, 
there exists a solution 
$i_{\infty}(\varphi):=i_{\infty}(\omega, \varepsilon)(\varphi)$  
%(see \eqref{iDP})
of the equation 
$\mathcal{F}(i_{\infty}, 0, \omega, \varepsilon)=0$.
%(see \eqref{NonlinearFunctionalDP}). 
Hence the embedded torus $\varphi\mapsto i_{\infty}(\varphi)$ 
is invariant for the Hamiltonian vector field 
$X_{H_{\varepsilon}}$, and it is filled by quasi-periodic 
solutions with frequency $\omega$. The torus $i_{\infty}$ satisfies
\begin{equation*}
\lVert i_{\infty}(\varphi)-(\varphi, 0, 0) 
\rVert_{s_0+\mu_1}^{\gamma, \mathcal{C}_{\varepsilon}}\le 
C\,\varepsilon^{16-2 b}\,\gamma^{-1}\,.
\end{equation*}
Moreover the torus $i_{\infty}$ satisfies 
condition \eqref{trawaves} and it is linearly stable.
\end{theorem}
\noindent
We can deduce Theorem \ref{thm:main} from Theorem \ref{IlTeoremaDP}, indeed
the quasi-periodic solution $(\eta,\psi)$ in \eqref{SoluzioneEsplicitaWW} is 
\[
\big(\eta(t, x), \psi(t, x)\big)=\Big(\Lambda^{-1}
\circ \Phi_{B}\circ
{\bf A}_{\varepsilon}\Big) (\theta_{\infty}(\omega t), y_{\infty}(\omega t), \Lambda W_{\infty}(\omega t)) 
\]
for $\omega=\omega(\zeta)\in \mathcal{C}_{\e}$, where $\omega(\zeta)$ is the frequency amplitude map \eqref{FreqAmplMapDP}.
The rest of the paper is devoted to the proof of Theorem \ref{IlTeoremaDP}.

\subsubsection{Tame estimates of the nonlinear vector field}

We give tame estimates for the composition operator induced by the Hamiltonian vector fields $X_{\mathcal{N}}$ and $X_{P}$ in \eqref{NonlinearFunctionalDP}.
Since the functions $y\to \sqrt{\zeta+\varepsilon^{2(b-1)} y}$
and  $\theta\to e^{\mathrm{i}\,\theta}$ 
 are analytic for $\varepsilon$ small enough and $\lvert y \rvert\le C$, 
 classical  composition results 
 (see for instance Lemma $2.2$ in \cite{KdVAut}) 
 imply that, 
 for all $\lVert \mathfrak{I} \rVert_{s_0}^{\gamma, \calO}\le 1$,
\begin{equation*}
\lVert A_{\varepsilon}(\theta(\varphi), y(\varphi), W(\varphi)) 
\rVert_s^{\gamma, \calO}\lesssim_s 
\varepsilon (1+\lVert \mathfrak{I} 
\rVert_s^{\gamma, \calO})\,.
\end{equation*}
In the following lemma we collect tame 
estimates for the Hamiltonian vector fields 
$X_{\mathcal{N}}, X_{P}, X_{H_{\varepsilon}}$ 
(see \eqref{HepsilonDP}). 
\begin{lemma}\label{TameEstimatesforVectorfieldsDP}
Let $\calO\Subset \Omega_{\varepsilon}$ and $\mathfrak{I}(\varphi)$ in \eqref{frakkiIDP} satisfying 
$\lVert \mathfrak{I} \rVert_{s_0+1}^{\gamma, \calO}\lesssim
\,\varepsilon^{16-2 b}\gamma^{-1}$. 
Then we have (recall \eqref{FreqAmplMapDP}, \eqref{Bi})
\begin{equation*}%\label{EstimatesVecFieldDP}
\begin{aligned}
\lVert \partial_y P(i) \rVert_s^{\gamma, \calO}&\lesssim_s \varepsilon^{14}
+\varepsilon^{2 b} \lVert \mathfrak{I} \rVert_{s+1}^{\gamma, \calO}\,, 
\qquad  \qquad \quad
\lVert \partial_{\theta} P(i) \rVert_s^{\gamma, \calO}\lesssim_s 
\varepsilon^{16-2 b}(1+\lVert \mathfrak{I} 
\rVert_{s+1}^{\gamma, \calO})\,,
\\
\lVert \nabla_W P(i) \rVert_s^{\gamma, \calO}&\lesssim_s 
\varepsilon^{15-b}+\varepsilon^{16-b}\gamma^{-1} 
\lVert \mathfrak{I} \rVert_{s+1}^{\gamma, \calO}\,, 
\qquad
\lVert X_P (i) \rVert_s^{\gamma, \calO}\lesssim_s \varepsilon^{16-2 b}+\varepsilon^{2 b} 
\lVert \mathfrak{I} \rVert_{s+1}^{\gamma, \calO}\,, \\
\lVert \partial_{\theta}\partial_y P(i) \rVert_s^{\gamma, \calO}
&\lesssim_s \varepsilon^{14}+\varepsilon^{15}\gamma^{-1} 
\lVert \mathfrak{I} \rVert_{s+1}^{\gamma, \calO}\,, 
\qquad \; \lVert \partial_y\nabla_W P(i) \rVert_s^{\gamma, \calO}
\lesssim_s \varepsilon^{13+b}+\varepsilon^{2 b-1}
\lVert \mathfrak{I} \rVert_{s+1}^{\gamma, \calO}\,,\\
\lVert \partial_{y y} P(i)-\frac{\varepsilon^{2 b}}{2} &\big(\mathbb{A}+\mathbb{B}(\zeta)\big) \rVert_s^{\gamma, \calO}
\lesssim_s \varepsilon^{12+2 b}+\varepsilon^{13+2 b}\gamma^{-1} \lVert \mathfrak{I} \rVert_{s+1}^{\gamma, \calO}\,,
\end{aligned}
\end{equation*}
and for all $\hat{\imath}:=(\hat{\Theta}, \hat{y}, \hat{W})$,
\begin{equation*}%\label{stimavfy}
\begin{aligned}
\lVert \partial_y \td_i X_P(i)[\hat{\imath}] \rVert_s^{\gamma, \calO}
&\lesssim_s \varepsilon^{2 b-1} 
(\lVert \hat{\imath} \rVert_{s+1}^{\gamma, \calO}
+\lVert \mathfrak{I} \rVert_{s+1}^{\gamma, \calO}
\lVert \hat{\imath} \rVert^{\g, \calO}_{s_0+1})\,,\\
\lVert \td_i X_{H_{\varepsilon}}(i)[\hat{\imath}]
+(0, 0, J\, \hat{W})\rVert_s^{\gamma, \calO}
&\lesssim_s \varepsilon (\lVert \hat{\imath} \rVert_{s+1}^{\gamma, \calO}
+\lVert \mathfrak{I} \rVert_{s+1}^{\gamma, \calO}
\lVert \hat{\imath} \rVert^{\g, \calO}_{s_0+1})\,,\\
\lVert \td_i^2 X_{H_{\varepsilon}} (i) [\hat{\imath}, \hat{\imath}]\rVert_s^{\gamma, \calO}
&\lesssim_s \varepsilon (\lVert \hat{\imath} 
\rVert_{s+1}^{\gamma, \calO}\lVert \hat{\imath} \rVert_{s_0+1}^{\gamma, \calO}
+\lVert \mathfrak{I} \rVert_{s+1}^{\gamma, \calO}
(\lVert \hat{\imath} \rVert^{\g, \calO}_{s_0+1})^2)\,.
\end{aligned}
\end{equation*}
\end{lemma}

\noindent
In the sequel we will use that, 
by the diophantine condition \eqref{0di0Melnikov}, 
the operator $(\omega\cdot\pa_{\f})^{-1}$ 
 is defined for all functions $u$ with zero $\varphi$-average, 
 and satisfies
\begin{equation*}\label{stimaDomegaDP}
\lVert (\omega\cdot\partial_{\varphi})^{-1} u \rVert_s
\lesssim_s \gamma^{-1}\,\lVert u \rVert_{s+\tau}, 
\quad \lVert (\omega\cdot\partial_{\varphi})^{-1} u \rVert_s^{\gamma, \calO}\lesssim_s \gamma^{-1} \lVert u \rVert^{\gamma, \calO}_{s+2\tau+1}\,.
\end{equation*}

\section{Approximate inverse}\label{sezione6DP}

We want to solve the nonlinear functional equation %(see \eqref{NonlinearFunctionalDP})
\begin{equation}\label{EquazioneFunzionaleDP}
\mathcal{F}(i, \Xi)=0
\end{equation}
by applying a Nash-Moser scheme. It is well 
known that the main issue in implementing this algorithm concerns the approximate inversion of the linearized operator of $\mathcal{F}$ 
at any approximate solution $(i_n, \Xi_n)$, 
namely $d \mathcal{F}(i_n, \Xi_n)$. 
Note that $d \mathcal{F}(i_n, \Xi_n)$ is independent  of $\Xi_n$ (
see \eqref{NonlinearFunctionalDP}).
One of the main problems is that the 
$(\theta, y, W)$-components of $d\mathcal{F}(i_n, \Xi_n)$ 
are coupled and then the linear system 
\begin{equation}\label{coupled}
d \mathcal{F}(i_n, \Xi_n)[\hat\imath, \hat{\Xi}]=
\omega\cdot \partial_{\varphi} \hat\imath
-d_{i} X_{H_\varepsilon} (i_n)[\hat\imath]
-(0, \hat{\Xi}, 0)=g=(g^{(\theta)}, g^{(y)}, g^{(W)})
\end{equation}
is quite involved.
In order to approximately solve \eqref{coupled} we follow 
the scheme developed by Berti-Bolle in 
\cite{BertiBolle} which describes a way to 
approximately triangularize \eqref{coupled}. 
 %This method has been applied in \cite{KdVAut}, \cite{Giuliani}. 
 We recall here the main steps of the strategy.

\smallskip
We  study the solvability of equation \eqref{coupled} 
at an approximate solution, 
which we denote by 
$(i_0, \Xi_0)$, $i_0(\varphi)=(\theta_0(\varphi), y_0(\varphi), W_0(\varphi))$ .
%in order to keep the notations of \cite{KdVAut},  \cite{Giuliani} . 
We assume the following hypothesis, which we shall 
verify at any step of the Nash-Moser iteration:
\begin{itemize}
\item \textbf{Assumption}. The map $\omega \mapsto i_0(\omega)$ 
is a Lipschitz function of the form \eqref{iDP}, satisfying \eqref{trawaves}, 
defined on some subset 
$\calO_0\Subset \mathcal{G}_0 \subseteq\Omega_{\varepsilon}$ 
(recall \eqref{0di0Melnikov},\eqref{OmegaEpsilonDP}) 
and, for some $\mathfrak{p}_0:=\mathfrak{p}_0( \nu)>0$,
\begin{equation}\label{AssumptionDP}
\lVert \mathfrak{I}_0 \rVert_{s_0+\mathfrak{p}_0}^{\gamma, \calO_0}
\le \varepsilon^{16-2b} \gamma^{-1}\,, 
\quad \lVert Z \rVert_{s_0+\mathfrak{p}_0}^{\gamma, \calO_0}
\le \varepsilon^{16-2 b}\,,
%\quad \gamma=\varepsilon^{2b}, %\quad a\ll 1,
\end{equation}
where $\mathfrak{I}_0(\varphi):=i_0(\varphi)-(\varphi, 0, 0)$ 
and $Z$ is the error function
\begin{equation*} %\label{errorfunction}
Z(\varphi):=(Z_1, Z_2, Z_3)(\varphi)
:=\mathcal{F}(i_0, \Xi_0)(\varphi)
=\omega\cdot \partial_{\varphi} i_0 (\varphi)
-X_{H_{\varepsilon, \Xi_0}}(i_0(\varphi))\,.
\end{equation*}
\end{itemize}
{This assumption will be verified inductively in Theorem \ref{NashMoserDP}, see \eqref{ConvergenzaDP}.\\}
By estimating the Sobolev norm of the function $Z$ 
we can measure how the embedding 
$i_0$ is close to being invariant for 
$X_{H_{\varepsilon, \Xi_0}}$. 
It is well known that an invariant torus $i_0$ 
with diophantine flow is isotropic 
(see e.g.\cite{BertiBolle}), 
namely the pull-back $1$-form $i_0^*\mathcal{V}$ 
is closed, where $\mathcal{V}$ is the 
Liouville $1$-form in \eqref{1FormDP}. 
This is tantamount to say that the 
$2$-form $\widetilde{\mathcal{W}}$ in  \eqref{NewSimplFormDPreal} 
vanishes on the torus $i_0(\T^{\nu})$, because 
$i_0^* \widetilde{\mathcal{W}}=i_0^* d\mathcal{V}=d\,i_0^*\mathcal{V}$. 
For an ``approximately invariant'' embedded torus $i_0$ the 
$1$-form $i_0^*\mathcal{V}$ 
is only ``approximately closed''. 
In order to make this statement 
quantitative we consider
\begin{equation*} %\label{pullBackLambdaDP}
i_0^*\mathcal{V}=\sum_{k=1}^{\nu} a_k(\varphi)\,d\varphi_k\,, 
\quad a_k(\varphi):=([\partial_{\varphi} \theta_0(\varphi)]^T y_0(\varphi))_k
+\frac{1}{2} ( \partial_{\varphi_k} W_0(\varphi)\,, 
J W_0(\varphi))_{L^2}
\end{equation*}
and we quantify how small is
\begin{equation*} %\label{pullBackWdp}
i_0^* \widetilde{\mathcal{W}}=d\,i_0^* \mathcal{V}
=\sum_{1\le k<j\le \nu} A_{k\,j}(\varphi)\,d\varphi_k\wedge d \varphi_j\,, 
\quad A_{k\,j}(\varphi)
:=\partial_{\varphi_k} a_j(\varphi)
-\partial_{\varphi_j} a_k(\varphi)\,.
\end{equation*}
The next lemma proves that if $i_0$ is a solution of the equation \eqref{EquazioneFunzionaleDP}, 
then the parameter $ \Xi$ has to be naught, 
hence the embedded torus $i_0$ 
supports a quasi-periodic solution of the 
``original'' system with Hamiltonian 
$H_{\varepsilon}$. 
\begin{lemma}{(Lemma $6.1$ in \cite{KdVAut})}\label{Lemma6.1DP}
We have
$\lvert  \Xi_0 \rvert^{\gamma, \calO_0}\lesssim \lVert Z \rVert_{s_0}^{\gamma, \calO_0}$.
In particular, if $\mathcal{F}(i_0, \Xi_0)=0$ then 
$\Xi_0=0$ and the torus 
$i_0(\varphi)$ is invariant for the vector field $X_{H_{\varepsilon}}$.
\end{lemma}
\noindent
%Now we estimate the size of 
%$i_0^*\mathcal{W}$ in terms of the error function $Z$.
%By \eqref{pullBackLambdaDP}, \eqref{pullBackWdp} we get
%\[
%\lVert A_{k\,j} \rVert_s^{\gamma, \calO_0}\le_s \lVert \mathfrak{I}_0 \rVert^{\gamma, \calO_0}_{s+2}.
%\]
%Moreover, we have the following bound.
%\begin{lemma}{(Lemma $6.2$ in \cite{KdVAut})}
%The coefficients $A_{k\,j}(\varphi)$ in \eqref{pullBackWdp} satisfy
%\begin{equation*}
%\lVert A_{k\,j} \rVert_s^{\gamma, \calO_0}\le_s \gamma^{-1} (\lVert Z \rVert_{s+2\tau+2}^{\gamma, \calO_0}+\lVert Z \rVert_{s_0+1}^{\gamma, \calO_0}\lVert \mathfrak{I}_0\rVert_{s+2\tau+2}^{\gamma, \calO_0}).
%\end{equation*}
%\end{lemma}
%If $Z=0$ then $i_0$ is a  solution. In general we say that $i_0$ is "approximately invariant" up to order $O(Z)$.
%We observe that by Lemma $6.1$ in \cite{KdVAut} we have that if $i_0$ is a solution, then the parameter $\Xi_0$ has to be naught, hence the embedded torus $i_0$ supports a quasi-periodic solution of the ``original'' system with Hamiltonian $H_{\varepsilon}$ (see \eqref{HepsilonDP}). 
%\smallskip

By \cite{BertiBolle} we know that it is possible to construct an embedded torus 
$i_{\delta}(\varphi)=(\theta_0(\varphi), y_{\delta}(\varphi), W_0(\varphi))$, 
which differs from $i_0$ only for a small modification of the $y$-component, 
such that the $2$-form $\widetilde{\mathcal{W}}$ (recall  \eqref{NewSimplFormDPreal}) 
vanishes on the torus $i_{\delta}(\T^{\nu})$, namely $i_{\delta}$ is isotropic. 
In particular $i_{\delta}(\varphi)$ 
is approximately invariant up to order 
$O(Z)$ (see Lemma $7$ in \cite{BertiBolle}).
More precisely we have the following.
In the paper we denote equivalently the differential $\partial_i$ or $d_i$. We use the notation $\Delta_{\varphi}:=\sum_{k=1}^{\nu}\partial^2_{\varphi_{k}}$.

\begin{lemma}{\textbf{(Isotropic torus).}}{(Lemma $6.3$ in \cite{KdVAut})}
\label{isotropictorus}
The torus $i_{\delta}=(\theta_0(\varphi), y_{\delta}(\varphi), z_0(\varphi))$ 
defined by 
\begin{equation*}
y_{\delta}:=y_0+[\partial_{\varphi}\theta_0(\varphi)]^{-T}\rho(\varphi)\,, 
\quad \rho_j(\varphi):=\Delta^{-1}_{\varphi} 
\sum_{k=1}^{\nu} \partial_{\varphi_j} A_{k\,j}(\varphi)\,,
\end{equation*}
is isotropic. If \eqref{AssumptionDP} holds, then, for some 
$\mathfrak{d}:=\mathfrak{d}(\nu, \tau)$,
\begin{align*}
&\lVert y_{\delta}-y_0 \rVert_s^{\gamma, \calO_0}\leq_s \|\mathfrak{I}_0\|^{\gamma,\calO_0}_{s+\mathfrak{d}}\,,\\
&\lVert y_{\delta}-y_0 \rVert_s^{\gamma, \calO_0}\leq_s 
\gamma^{-1} (\lVert Z \rVert_{s+\mathfrak{d}}^{\gamma, \calO_0}
\lVert \mathfrak{I}_0 \rVert_{s_0+\mathfrak{d}}^{\gamma, \calO_0}
+\lVert Z \rVert_{s_0+\mathfrak{d}}^{\gamma, \calO_0}
\lVert \mathfrak{I}_0 
\rVert_{s+\mathfrak{d}}^{\gamma, \calO_0})\,,\nonumber\\
&\lVert \mathcal{F}(i_{\delta}, \zeta_0) \rVert_s^{\gamma, \calO_0}\leq_s 
\lVert Z \rVert_{s+\mathfrak{d}}^{\gamma, \calO_0}
+\lVert Z \rVert_{s_0+\mathfrak{d}}^{\gamma, \calO_0}
\lVert \mathfrak{I}_0 \rVert_{s+\mathfrak{d}}^{\gamma, \calO_0}\,, 
%\label{6.10DP}\\& 
\qquad
\lVert \partial_i i_{\delta} [\hat{\imath}] \rVert_s\leq_s
 \lVert \hat{\imath} \rVert_{s}
+\lVert \mathfrak{I}_0 \rVert_{s+\mathfrak{d}}
\lVert \hat{\imath} \rVert_{s}\,.
\end{align*}
\end{lemma}
 The strategy now is to construct an approximate inverse for 
 $d \mathcal{F}(i_0, \Xi_0)$ 
 by starting from an approximate inverse for the linear operator 
 $d \mathcal{F}(i_\delta, \Xi_0)$. 
 The advantage of analyzing the linearized problem 
 at $i_{\delta}$ is that, thanks 
 to the isotropicity of $i_{\delta}$, it is possible to construct a 
 \emph{symplectic} change of variables which 
 approximately triangularizes the linear system . 
 %(see Lemma $6.3$ in \cite{BertiBolle}).  
 %For   the details we refer to \cite{BertiBolle} 
 %and \cite{KdVAut}, here we  only give the relevant 
 %definitions and state the main result.
We define the symplectic change of coordinates
\begin{equation}\label{GdeltaDP}
\begin{pmatrix}
\theta \\ y \\ W
\end{pmatrix}
:=G_{\delta}\begin{pmatrix}
\vartheta\\ Y \\ U
\end{pmatrix}
:=\begin{pmatrix}
\theta_0(\vartheta) \\
y_{\delta}(\vartheta)+[\partial_{\vartheta}\theta_0(\vartheta)]^{-T}Y
+[(\partial_{\theta} \tilde{W}_0)(\theta_0(\vartheta))]^T\,J^{-1}U\\
W_0(\vartheta)+U
\end{pmatrix}\,,
\end{equation}
where $\tilde{W}_0:=W_0 (\theta_0^{-1} (\theta))$.  
% then $i_{\delta}(\varphi)$ is the trivial embedded torus $(\varphi, 0, 0)$.
We show that the map $G_{\delta}$ ``preserves'' 
the subspace of traveling embeddings, namely satisfying \eqref{trawaves}.
\begin{lemma}\label{subspaceMomento}
Assume that %the embedding 
$i_{\delta}(\vphi)$ 
satisfies the condition \eqref{trawaves}.
Then $\mathcal{I}(\vphi):=(\vartheta(\vphi),Y(\vphi), U(\vphi))$
satisfies \eqref{trawaves} if and only if 
 $i(\vphi)=(\theta(\vphi),y(\vphi),W(\vphi)):=G_{\delta}(\mathcal{I}(\vphi))$ 
satisfies \eqref{trawaves}.
\end{lemma}

\begin{proof}
We use the equivalent condition \eqref{cifrato}.
We note that
\[
\mathtt{v}\cdot\pa_{\vphi}i(\vphi)+X_{M}(i(\vphi))=0\qquad \Leftrightarrow
\qquad \mathtt{v}\cdot\pa_{\vphi}\mathcal{I}(\vphi)
+(d G_{\delta})^{-1}X_{M}G_{\delta}(\mathcal{I}(\vphi))=0\,.
\]
Then, since $G_{\delta}$ is symplectic  one has 
$( d G_{\delta})^{-1}X_{M}G_{\delta}=X_{M\circ G_{\delta}}$.
Therefore
it is sufficient to prove that $M\circ G_{\delta}=M$.
For the variables $W_0,W,U$ in $H_{S}^{\perp}$ (see \eqref{decomposition})
we shall use the notation 
$W_0=(\eta_0,\psi_0)^{T}$, $W=(\tilde{\eta},\tilde{\psi})^{T}$,
$U=({\eta}_1,{\psi}_1)^{T}$. Hence, using \eqref{GdeltaDP} and 
\eqref{trawavesmom},  
we have
\[
M(G_{\delta}(\vartheta,Y, U))=
-\mathtt{v}\cdot y_{\delta}(\vartheta)-[\pa_{\vartheta}\theta_{0}(\vartheta)]^{-1}\mathtt{v}\cdot Y
-\mathtt{v}\cdot BJ^{-1} U
+\int_{\mathbb{T}}(\eta_0+\eta_1)_{x}(\psi_0+\psi_1)dx\,,
\]
where $B:=[(\partial_{\theta} \tilde{W}_0)(\theta_0(\vartheta))]^T$.
Using $\tilde{W}_0:=W_0 (\theta_0^{-1} (\theta))$ and that $i_{\delta}(\vphi)$
satisfies \eqref{trawaves} we deduce 
\[
\begin{aligned}
&-\mathtt{v}\cdot BJ^{-1} U=\int_{\mathbb{T}}
(d\eta_0)[\mathtt{v}]\psi_1 dx-
\int_{\mathbb{T}}
(d\psi_0)[\mathtt{v}]\eta_1 dx\,,\\
&(d\eta_0)[\mathtt{v}]=-(\eta_0)_{x}\,,\;\;\;
(d\psi_0)[\mathtt{v}]=-(\psi_0)_{x}\,,
\;\;\;[d\theta_{0}(\vartheta)]^{-1}\mathtt{v}=\mathtt{v}\,.
\end{aligned}
\]
Therefore, by integrating by parts, we get
\[
\begin{aligned}
M(G_{\delta}(\vartheta,Y,U))&=
-\mathtt{v}\cdot Y-\mathtt{v}\cdot y_{\delta}(\vartheta)
+
\int_{\mathbb{T}}\Big[
(\psi_0)_{x}\eta_1-(\eta_0)_{x}\psi_1 \big]dx
\\&\qquad+\int_{\mathbb{T}}\Big[(\eta_0)_x\psi_0 +
(\eta_0)_x\psi_1+(\eta_1)_{x}\psi_0+(\eta_1)_x\psi_1\Big]dx\\
&=-\mathtt{v}\cdot Y+\int_{\mathbb{T}}(\eta_1)_x\psi_1dx
-\mathtt{v}\cdot y_{\delta}(\vartheta)+\int_{\mathbb{T}}(\eta_0)_x\psi_0dx\,.
\end{aligned}
\]
Since $i_{\delta}$ satisfies \eqref{trawaves}
the term $-\mathtt{v}\cdot y_{\delta}(\vartheta)+\int_{\mathbb{T}}(\eta_0(\vartheta))_x
\psi_0(\vartheta)dx$
is independent of $\vartheta$, hence it does not contributes  to the vector field
$X_{M\circ G_{\delta}}$. Then we have the thesis.
\end{proof}

The transformed Hamiltonian 
$K:=K(\vartheta, Y, U, \Xi_0)$ is (recall \eqref{HepsilonZetaDP})
\begin{equation}\label{Kdp}
\begin{aligned}
K:=H_{\varepsilon, \Xi_0}\circ G_{\delta}&=\theta_0(\vartheta)\cdot \Xi_0
+K_{00}(\vartheta)+K_{10}(\vartheta)\cdot Y
+(K_{01}(\vartheta), U)_{L^2}+\frac{1}{2} K_{20}(\vartheta) Y \cdot Y+\\
&+(K_{11}(\vartheta) Y, U)_{L^2}
+\frac{1}{2} (K_{02}(\vartheta) U, U)_{L^2}+K_{\geq 3}(\vartheta, Y, U)
\end{aligned}
\end{equation}
where $K_{\geq 3}$ collects the terms at least cubic in the variables $(Y, U)$. 
At any fixed $\vartheta$, the Taylor coefficient 
$K_{00}(\vartheta)\in\mathbb{R}, 
K_{10}(\vartheta)\in\mathbb{R}^{\nu}, 
K_{01}(\vartheta)\in H_S^{\perp}, K_{20}(\vartheta)$ 
is a $\nu\times\nu$ real matrix, 
$K_{02}(\vartheta)$ is a linear self-adjoint 
operator of $H_S^{\perp}$ and finally
$K_{11}(\vartheta)\colon\mathbb{R}^{\nu} \to H_S^{\perp}$.
The above Taylor coefficients do not depend on the parameter $\Xi_0$.
The Hamilton equations associated to \eqref{Kdp} are
\begin{equation}\label{VectorFieldKdp}
\begin{cases}
\dot{\vartheta}=K_{10}(\vartheta)+K_{20}(\vartheta) Y
+K_{11}^T(\vartheta) U
+\partial_{\eta} K_{\geq 3}(\vartheta, Y, U)\\
\begin{aligned}
\dot{Y}=&-[\partial_{\vartheta} \theta_0(\vartheta)]^T \Xi_0
-\partial_{\vartheta} K_{00}(\vartheta)-[\partial_{\vartheta} K_{10}(\vartheta)]^{T}Y
-[\partial_{\vartheta} K_{01}(\vartheta)]^{T} U-\\
&-\partial_{\vartheta} \left(\frac{1}{2} K_{20}(\vartheta)Y\cdot Y
+(K_{11}(\vartheta) Y, U)_{L^2}
+\frac{1}{2} (K_{02}(\vartheta) U, U)_{L^2}
+K_{\geq 3}(\vartheta, Y, U)\right)
\end{aligned}
\\
\dot{U}=J (K_{01}(\vartheta)+K_{11}(\vartheta) Y
+K_{02}(\vartheta) U+\nabla_{U} K_{\geq 3}(\vartheta, Y, U))
\end{cases}
\end{equation}
where $[\partial_{\vartheta} K_{10}(\vartheta)]^T$ is the $\nu\times \nu$ transposed matrix and $[\partial_{\vartheta} K_{01}(\vartheta)]^T, K_{11}^T(\vartheta)\colon H_S^{\perp} \to \mathbb{R}^{\nu}$ are defined by the duality relation 
\[
(\partial_{\vartheta} K_{01} (\vartheta)[\hat{\vartheta}], U)_{L^2}
=\hat{\vartheta}\cdot [\partial_{\vartheta} K_{01}(\vartheta)]^T U\,, 
\quad \forall \hat{\vartheta}\in\mathbb{R}^{\nu}, U\in H_S^{\perp}\,,
\]
and similarly for $K_{11}$.
Explicitly, for all $U\in H_S^{\perp}$, 
and denoting $\underline{e}_k$ the 
$k$-th versor of $\mathbb{R}^{\nu}$,
\begin{equation*}
K_{11}^T(\vartheta) U=\sum_{k=1}^{\nu} 
(K_{11}^T(\vartheta) U\cdot \underline{e}_k)\,
\underline{e}_k=\sum_{k=1}^{\nu} (U, K_{11}(\vartheta) 
\underline{e}_k)_{L^2} \underline{e}_k\in\mathbb{R}^{\nu}.
\end{equation*}
In the next lemma we estimate the coefficients 
$K_{00}, K_{10}, K_{01}$ in the Taylor expansion \eqref{Kdp}. 
%The term $K_{10}$ describes how the tangential 
%frequencies vary with respect to $\omega$.
Note that on an exact solution $(i_0, \Xi_0)$ 
we have 
$K_{00}(\vartheta)=\mbox{const}$, $K_{10}=\omega$ and $K_{01}=0$.

\begin{lemma}{(Lemma $6.4$ in \cite{KdVAut})}\label{Lemma6.4DP}
Assume \eqref{AssumptionDP}. Then there is $\mu:=\mu(\nu)$ 
such that
\[
\lVert \partial_{\vartheta} K_{00} \rVert_s^{\gamma, \calO_0}
+\lVert K_{10}-\omega \rVert_s^{\gamma, \calO_0}
+\lVert K_{01} \rVert_s^{\gamma, \calO_0}
\lesssim_s
\lVert Z \rVert_{s+\mu}^{\gamma, \calO_0}
+\lVert Z \rVert_{s_0+\mu}^{\gamma, \calO_0} 
\lVert \mathfrak{I}_0 \rVert_{s+\mu}^{\gamma, \calO_0}\,.
\]
\end{lemma}
In the next Lemma we 
estimate $K_{20}, K_{11}$ in \eqref{Kdp}.  This result can be obtained 
reasoning as in Lemma $6.6$ in \cite{KdVAut} 
taking into account the bounds
in Lemma \ref{TameEstimatesforVectorfieldsDP}.
%The norm of $K_{20}$ is the sum of the norms of its matrix entries.
\begin{lemma}\label{Lemma6.6dp}
Assume \eqref{AssumptionDP}. Then for some $\mu:=\mu(\nu)$ we have
\begin{align*}
&\lVert K_{20}-\frac{\varepsilon^{2 b}}{2} \mathbb{A}\rVert_s^{^{\gamma, \calO_0}}
\lesssim_s 
\varepsilon^{2b+2}+
\varepsilon^{2 b}\lVert \mathfrak{I}_0 \rVert_{s+\mu}^{^{\gamma, \calO_0}}\,,
%\label{PFM1}
\\
&\lVert K_{11} Y\rVert_s^{^{\gamma, \calO_0}}
\lesssim_s 
\varepsilon^{15} \gamma^{-1} 
\lVert Y \rVert_s^{^{\gamma, \calO_0}}
+\varepsilon^{2 b-1} \lVert \mathfrak{I}_0 \rVert_{s+\mu}^{\gamma, \calO_0}
\lVert Y \rVert_{s_0}^{^{\gamma, \calO_0}}\,,
\\
&\lVert K_{11}^{T} U \rVert_{s}^{^{\gamma, \calO_0}} 
\lesssim_s \varepsilon^{15} \gamma^{-1} 
\lVert U \rVert_{s+2}^{^{\gamma, \calO_0}}
+\varepsilon^{2 b-1} \lVert \mathfrak{I}_0 \rVert_{s+\mu}^{^{\gamma, \calO_0}}
\lVert U \rVert_{s_0+2}^{^{\gamma, \calO_0}}\,.
\end{align*}
In particular 
\begin{align*}
&\lVert K_{20}-\frac{\varepsilon^{2 b}}{2} 
\mathbb{A}\rVert_{s_0}^{^{\gamma, \calO_0}}\lesssim 
\varepsilon^{2b+2}\,, 
\quad 
\lVert K_{11} Y \rVert_{s_0}^{^{\gamma, \calO_0}}\lesssim
\varepsilon^{15} \gamma^{-1}\lVert Y \rVert_{s_0}^{^{\gamma, \calO_0}}\,, 
\quad 
\lVert K_{11}^{T} U \rVert_{s_0}^{^{\gamma, \calO_0}}\lesssim 
\varepsilon^{15} \gamma^{-1} \lVert U \rVert_{s_0+2}^{^{\gamma, \calO_0}}\,.
\end{align*}
\end{lemma}
%\begin{proof}
%The $\varepsilon^{2b+2}$ comes from the term 
%$\Pi_{Ker} R(v^6)$ (rescaled by $\varepsilon^{6-2b}$), 
%which gives the $\varepsilon^3$ term of the expansion 
%of the frequency, $\mathbb{B}[\xi, \xi, \xi]$, 
%but also $\mathbb{B}[\xi, \xi, y^2 ]$. 
%The $y$'s give $\varepsilon^{4 b-4}$, 
%then $6-2b+4b-4=2b+2$ (In KdV some argument for $v^6$ 
%which was not normalized).\\
%We recall that we look at $\partial_y^2 P$, 
%where we have not just $v=\bar{v}$, 
%but also $v-\bar{v}$, which has the 
%estimate with $\mathfrak{I}$.
%The term 
%$\varepsilon^{2 b}\lVert \mathfrak{I}_0 \rVert_{s+\sigma}^{^{\gamma, \calO_0}}$ 
%comes from the expansion 
%$\varepsilon^{4-2 b}\Pi_{Ker} v_{\delta}^4
%=\varepsilon^{4-2 b}\Pi_{Ker} (v-v_{\delta}) 
%v_{\delta}^3+hot$, 
%indeed $v-v_{\delta}\sim \mathfrak{I}_{\delta}$ 
%and from $v_{\delta^3}$ we expand in 
%$y$ in order to get $y^2$ which is scaled 
%with $\varepsilon^{4 b-4}$. Then $4-2 b+4b-4=2b$.\\
%\end{proof}

\begin{remark}\label{paganini}
Notice that, if $(\vartheta,Y, U)$ is such that
\[
\mathtt{v}\cdot\pa_{\vphi}\vartheta=\mathtt{v}\,,\quad \mathtt{v}\cdot\pa_{\vphi} Y=0\,,
\qquad (\mathtt{v}\cdot\pa_{\vphi}+\pa_{x}) U=0\,,
\]
then, using the equations \eqref{VectorFieldKdp}, we deduce that
\begin{align}
&\mathtt{v}\cdot\partial_{\varphi} K_{10}(\vartheta)=0, \qquad
(\mathtt{v}\cdot\pa_{\vphi}+\pa_{x})(J K_{11}(\vartheta)Y)=0\,,
\qquad 
\mathtt{v}\cdot\pa_{\vphi}K_{20}(\vartheta)Y=0\,,\nonumber\\
&\mathtt{v}\cdot\pa_{\vphi} K_{11}^T(\vartheta) U=0\,,\qquad
(\mathtt{v}\cdot\pa_{\vphi}+\pa_{x})(K_{02}(\vartheta)U)=0\,.\label{K02xtran}
\end{align}
\end{remark}

\noindent
We apply the linear change of variables
\begin{equation}\label{DGdp}
dG_{\delta}(\varphi, 0, 0) \begin{pmatrix}
\hat{\vartheta}\\ \hat{Y}\\ \hat{U} 
\end{pmatrix} :=\begin{pmatrix}
\partial_{\vartheta} \theta_0(\varphi) & 0 & 0\\
\partial_{\vartheta} y_{\delta}(\varphi) & 
[\partial_{\vartheta} \theta_0(\varphi)]^{-T} & 
[(\partial_{\theta} \tilde{W}_0)(\theta_0 (\varphi))]^T J^{-1}\\
\partial_{\vartheta} W_0(\varphi) & 0 & \mathrm{Id}
\end{pmatrix}
\begin{pmatrix}
\hat{\vartheta} \\ \hat{Y} \\ \hat{U}
\end{pmatrix}\,.
\end{equation}
In these new coordinates the linearized operator 
$d_{i, \Xi}\mathcal{F}(i_{\delta}, \Xi_0)$ 
is ``approximately'' the operator obtained 
linearizing \eqref{VectorFieldKdp} at 
$(\vartheta, Y, U, \Xi)=(\varphi, 0, 0, \zeta_0)$ 
with $\omega\cdot\pa_{\vphi}$
%$\mathcal{D}_{\omega}$ 
instead of $\partial_t$, namely
\begin{equation}\label{6.26DP}
\begin{pmatrix}
\omega\cdot\pa_{\vphi} \hat{\vartheta}
-\partial_{\vartheta} K_{10}(\varphi)[\hat{\vartheta}]
-K_{20}(\varphi) \hat{Y}
-K_{11}^T(\varphi) \hat{U}\\
\omega\cdot\pa_{\vphi} \hat{Y}
+[\partial_{\vartheta} \theta_0(\varphi)]^T \hat{\Xi}
+\partial_{\vartheta} [\partial_{\vartheta} \theta_0(\varphi)]^T[\hat{\vartheta}, \Xi_0]
+\partial_{\vartheta\vartheta} K_{00}(\varphi)[\hat{\vartheta}]
+[\partial_{\vartheta} K_{10}(\varphi)]^T\hat{Y}
+[\partial_{\vartheta} K_{01}(\varphi)]^T \hat{U}\\
\omega\cdot\pa_{\vphi} \hat{U}
-J\{ \partial_{\vartheta} K_{01}(\varphi)[\hat{\vartheta}]
+K_{11}(\varphi)\hat{Y}+K_{02}(\varphi)\hat{U}\}
\end{pmatrix}\,.
\end{equation}
We give estimates on the composition operator 
induced by the transformation \eqref{DGdp}.

\begin{lemma}{(Lemma $6.7$ in \cite{KdVAut})}\label{LemmaDG}
Assume \eqref{AssumptionDP} and let $\hat{\imath}:=(\hat{\vartheta}, \hat{Y}, \hat{U})$. 
Then, for some $\mu:=\mu(\nu)$, we have
\begin{equation*}
\begin{aligned}
&\lVert dG_{\delta}(\varphi, 0, 0)[\hat{\imath}] \rVert_{s}
+\lVert dG_{\delta}(\varphi, 0, 0)^{-1}[\hat{\imath}] \rVert_s\leq_s 
\lVert \hat{\imath} \rVert_s+\lVert \mathfrak{I}_0 \rVert_{s+\mu}
\lVert \hat{\imath} \rVert_{s_0}\\
&\lVert d^2G_{\delta}(\varphi, 0, 0)[\hat{\imath}_1, \hat{\imath}_2]\rVert_s\leq_s 
\lVert \hat{\imath}_1\rVert_s \lVert \hat{\imath}_2 \rVert_{s_0}
+\lVert \hat{\imath}_1 \rVert_{s_0} \lVert \hat{\imath} \rVert_s
+\lVert \mathfrak{I}_0 \rVert_{s+\mu}
\lVert \hat{\imath} \rVert_{s_0} 
\lVert \hat{\imath}_2 \rVert_{s_0}\,.
\end{aligned}
\end{equation*}
Moreover the same estimates hold if 
we replace $\lVert \cdot \rVert_s$ with 
$\Vert \cdot \rVert_s^{^{\gamma, \calO_0}}$.
\end{lemma}

\noindent
In order to construct an approximate inverse of \eqref{6.26DP} 
it is sufficient to solve the system of equations
\begin{equation}\label{Ddp}
\mathbb{D}[\hat{\vartheta}, \hat{Y}, \hat{U}, \hat{\Xi}]:=
\begin{pmatrix}
\omega\cdot\pa_{\vphi}\hat{\vartheta}-K_{20} (\varphi) \hat{Y}
-K_{11}^T(\varphi) \hat{U}\\
\omega\cdot\pa_{\vphi} \hat{Y}
+[\partial_{\vartheta} \theta_0(\varphi)]^T \hat{\Xi}\\
\omega\cdot\pa_{\vphi}\hat{U}-J K_{11}(\varphi)\hat{Y}
-J K_{02} (\varphi) \hat{U}
\end{pmatrix}=\begin{pmatrix}
g_1 \\ g_2 \\ g_3
\end{pmatrix}
\end{equation}
which is obtained by \eqref{6.26DP} neglecting the terms 
that are naught at a solution, namely, 
by Lemmata \ref{Lemma6.1DP} and 
\ref{Lemma6.4DP}, 
$\partial_{\vartheta} K_{10}, \partial_{\vartheta \vartheta} K_{00}, 
\partial_{\vartheta} K_{00}, \partial_{\vartheta} K_{01}$ 
and $\partial_{\vartheta} [\partial_{\vartheta} \theta_0(\varphi)]^T[\cdot , \zeta_0]$.
The term $(g_1,g_2,g_3)^{T}$ is assumed to satisfies the \eqref{trawaves}
and we look for a solution $(\hat{\vartheta}, \hat{Y}, \hat{U})$ of \eqref{Ddp} 
with the same property.

First, we solve the second equation, namely
\begin{equation}\label{secondeqDP}
\omega\cdot\pa_{\vphi}\hat{Y}=g_2-[\partial_{\vartheta} \theta_0(\varphi)] \hat{\Xi}\,.
\end{equation}
We choose $\hat{\Xi}:=\langle g_2 \rangle_{\varphi}$ (where 
$\langle \cdot\rangle_{\vphi}$ denotes the $\vphi$-average)
so that the 
$\varphi$-average of the right hand side of 
\eqref{secondeqDP} is zero.
%, namely
%\begin{equation}\label{ValoreperZetaDP}
%\hat{\Xi}=\langle g_2 \rangle_{\varphi}.
%\end{equation}
Note that the $\varphi$-averaged matrix 
$\langle(\partial_{\vartheta} \theta_0)^T \rangle_{\varphi}
=\langle \mathrm{Id}+(\partial_{\vartheta} \Theta_0)^T \rangle_{\varphi}
=\mathrm{Id}$ because 
$\theta_0(\varphi)=\varphi
+\Theta_0(\varphi)$ and 
$\Theta_0(\varphi)$ is periodic in each component. 
Therefore
\begin{equation}\label{etaDP}
\hat{Y}=(\omega\cdot\pa_{\vphi})^{-1} (g_2-[\partial_{\vartheta} \theta_0(\varphi)]^T 
\langle g_2 \rangle_{\varphi})
+\langle\hat{Y} \rangle_{\varphi}\,, 
\qquad \langle\hat{Y} \rangle_{\varphi}\in\mathbb{R}^{\nu}\,,
\end{equation}
where the average $\langle\hat{Y} \rangle_{\varphi}$ 
will be fix when we deal with the first equation.
We remark that, by assumption on $i_0$ and $(g_1,g_2,g_3)^{T}$
one has (recall $\theta_0(\vphi)=\vphi+\Theta_0(\vphi)$)
\begin{equation}\label{paganini1}
\mathtt{v}\cdot\pa_{\vphi}\Theta_0(\vphi)=\mathtt{v}\cdot\pa_{\vphi}g_2(\vphi) =0
\;\;\;\;\Rightarrow\;\;\;\;
\mathtt{v}\cdot\pa_{\vphi}\hat{Y}(\vphi)=0\,.
\end{equation}

We now analyze the third equation, namely
\begin{equation}\label{thirdeqDP}
\mathcal{L}_{\omega} \hat{U}=g_3+J K_{11}(\varphi) \hat{Y}\,, 
\qquad \mathcal{L}_{\omega}:=\omega\cdot \partial_{\varphi}-J K_{02} (\varphi)\,.
\end{equation}
Since $\hat{Y}$ has been fixed in \eqref{etaDP} (up to a constant),
%If we fix $\hat{\eta}$, 
then solving the equation \eqref{thirdeqDP} 
is tantamount to invert the operator $\mathcal{L}_{\omega}$.
For the moment we make the following assumption 
(that will be proved in section \ref{sez:KAMredu}):

\begin{itemize}
\item
\textbf{Inversion Assumption:} 
There exist $\gotp_1:=\gotp_1(\nu)>0$ 
and a set 
$\Omega_{\infty}\subset \mathcal{O}_0\subseteq \Omega_{\varepsilon}$ 
such that for all $\omega\in \Omega_{\infty}$ 
and every $G\in \Big(H_{S^{\perp}}^{s+2\tau+1}\cap S_{\mathtt{v}}\Big)^2$ 
(see \eqref{funzmom}), 
there exists a solution $h:=\mathcal{L}_{\omega}^{-1} G\in \Big(H_{S^{\perp}}^{s}\cap S_{\mathtt{v}}\Big)^2$ 
of the linear equation $\mathcal{L}_{\omega} h=G$ and satisfies
\begin{equation}\label{InversionAssumptionDP}
\lVert \mathcal{L}_{\omega}^{-1} G \rVert_s^{\gamma, \Omega_{\infty}}
\lesssim_s
\gamma^{-1} (\lVert G \rVert_{s+2\tau+1}^{\gamma, \Omega_{\infty}}
+\varepsilon \gamma^{-6}
\lVert \mathfrak{I}_\delta \rVert_{s+\gotp_1}^{\gamma, \calO_0}
\lVert G \rVert_{s_0}^{\gamma, \Omega_{\infty}})\,.
\end{equation}	
\end{itemize}
By the above assumption, there exists a solution of \eqref{thirdeqDP}
\begin{equation}\label{wDP}
\hat{U}=\mathcal{L}_{\omega}^{-1}[g_3+J\, K_{11}(\varphi)\hat{Y}]\,.
\end{equation}
Moreover, recalling the assumption on $g_3$ and using Remark \ref{paganini},
we have
 \begin{equation}\label{paganini2}
(\mathtt{v}\cdot\pa_{\vphi}+\pa_{x})\hat{U}=0\,.
%\qquad {\rm if}\qquad 
%(\mathtt{v}\cdot\pa_{\vphi}+\pa_{x})(g_3+J\, K_{11}(\varphi)\hat{Y})=0\,.
\end{equation}

Now consider the first equation
\begin{equation}\label{firsteqDP}
\omega\cdot\pa_{\vphi}\hat{\vartheta}=g_1+K_{20} \hat{Y}-K_{11}^T(\varphi) \hat{U}\,.
\end{equation}
 Substituting \eqref{etaDP}, \eqref{wDP} in the equation \eqref{firsteqDP}, we get 
 \begin{equation}\label{firsteq2DP}
\omega\cdot\pa_{\vphi} \hat{\vartheta}= g_1+M_1(\varphi)\langle \hat{Y} \rangle_{\varphi}+M_2(\varphi) g_2+M_3(\varphi) g_3-M_2(\varphi)[\partial_{\vartheta} \theta_0]^T \langle g_2 \rangle_{\varphi}\,,
 \end{equation}
 where
 \begin{equation}
 M_1(\varphi):=K_{20}(\varphi)
 -K_{11}^T(\varphi)\mathcal{L}_{\omega}^{-1}J K_{11}(\varphi)\,, 
 \quad 
 M_2(\varphi):=M_1(\varphi)(\omega\cdot\pa_{\vphi})^{-1}\,, 
 \quad 
 M_3(\varphi):=-K_{11}^T(\varphi)\mathcal{L}_{\omega}^{-1}\,.
 \end{equation}
 In order to solve the equation \eqref{firsteq2DP} 
 we have to choose $\langle\hat{Y}\rangle_{\varphi}$ 
 such that the right hand side in \eqref{firsteq2DP} 
 has zero $\varphi$-average.
 By Lemma \ref{Lemma6.6dp} and \eqref{AssumptionDP}, 
 the $\varphi$-averaged matrix $\langle M_1 \rangle_{\varphi}$ is invertible, 
 for $\varepsilon$ small, and 
 $\langle M_1 \rangle_{\varphi}^{-1}=
 O(\varepsilon^{-2b})=O(\gamma^{-1})$.
 Thus we define
 \begin{equation*}
\langle\hat{Y}\rangle_{\varphi}=
-(\langle M_1 \rangle_{\varphi})^{-1}\{\langle g_1 \rangle_{\varphi}
+\langle M_2 g_2 \rangle_{\varphi}
+\langle M_3 g_3  \rangle_{\varphi}
-\langle M_2(\partial_{\vartheta} \theta_0)^T \rangle_{\varphi}\,
\langle g_2  \rangle_{\varphi}\}\,.
 \end{equation*}
 With this choice of $\langle \hat{Y}  \rangle_{\varphi}$ 
 the equation \eqref{firsteq2DP} has the solution
 \begin{equation*} %\label{psiDP}
 \hat{\vartheta}:=\big(\omega\cdot \partial_{\varphi} \big)^{-1}\Big(g_1
 +M_1(\varphi) \langle\hat{\eta} \rangle_{\varphi}
 +M_2(\varphi) g_2+M_3(\varphi) g_3
 -M_2(\varphi)[\partial_{\vartheta} \theta_0]^T 
 \langle g_2 \rangle_{\varphi} \Big)\,.
 \end{equation*}
Again, using Remark \ref{paganini} and \eqref{firsteqDP}, we get
 \begin{equation}\label{paganini3}
 \mathtt{v}\cdot\pa_{\vphi}\hat{\vartheta}=0\,.
 \end{equation}
In conclusion, we have constructed a solution 
$(\hat{\vartheta}, \hat{Y}, \hat{U}, \hat{\Xi})$ 
of the linear system \eqref{Ddp} which satisfies \eqref{trawaves}.  
Consider the operator 
\begin{equation*} %\label{ApproxInverseDP}
\mathbf{T}_0:=(d \tilde{G}_{\delta})(\varphi, 0, 0)\circ
\mathbb{D}^{-1}\circ (d G_{\delta}(\varphi, 0, 0))^{-1}\,,
\end{equation*}
 where $\tilde{G}_{\delta}((\vartheta, Y, U), \Xi)$ 
 is the identity on the $\Xi$-component.
By \eqref{paganini1}, \eqref{paganini2}, \eqref{paganini3} 
and using Lemma \ref{subspaceMomento}
one can check that ${\bf T}_0g$ satisfies \eqref{trawaves} 
if $g$ satisfies \eqref{trawaves}.
We denote the norm 
$\lVert (\vartheta, Y, U, \Xi)\rVert_s^{\gamma, \calO}
:=\max \{ \lVert (\vartheta, Y, U) \rVert_s, \lvert \Xi \rvert^{\gamma, \calO} \}$.
 In \cite{BertiBolle} (see also \cite{KdVAut},\cite{Giuliani}) 
 the following result is proved.
\begin{theorem}\label{TeoApproxInvDP}
Assume \eqref{AssumptionDP} and the \textbf{inversion assumption}.
Then there exists $\mu_1:=\mu_1( \nu)$  such that, 
for all $\omega\in \Omega_{\infty}$, the following holds:

\noindent
1.  for all $g:=(g^{(\theta)}, g^{(y)}, g^{(W)})$ satisfying 
\eqref{trawaves} ${\bf T}_0g$ satisfies \eqref{trawaves} and 
\begin{equation*} %\label{TameEstimateApproxInvDP}
\lVert \mathbf{T}_0 g \rVert_s^{\gamma, \Omega_{\infty}}
\lesssim_s\gamma^{-1}(\lVert g \rVert_{s+\mu_1}^{\gamma, \Omega_{\infty}}
+\varepsilon\gamma^{-6} 
\lVert \mathfrak{I}_0 \rVert_{s+\mu_1}^{\gamma, \calO_0}
 \lVert g \rVert_{s_0+\mu_1}^{\gamma, \Omega_{\infty}})\,.
\end{equation*}
\noindent
2. $\mathbf{T}_0$  is an approximate inverse of $\td \mathcal{F}(i_0)$, 
namely
\begin{equation*} %\label{6.41DP}
\begin{aligned}
&\lVert (d \mathcal{F}(i_0)\circ \mathbf{T}_0-\mathrm{I}) 
g \rVert_s^{\gamma, \Omega_{\infty}}
\lesssim_s \varepsilon^{2 b-1}\g^{-2} 
\Big( \lVert \mathcal{F}(i_0, \Xi_0) \rVert_{s_0+\mu_1}^{\gamma, \calO_0}
\lVert g \rVert_{s+\mu_1}^{\gamma, \Omega_{\infty}}\\
&\qquad\qquad\qquad 
+\{\lVert \mathcal{F}(i_0, \Xi_0)\rVert_{s+\mu_1}^{\gamma, \calO_0}
+\varepsilon \gamma^{-6}
\lVert \mathcal{F}(i_0, \Xi_0) \rVert_{s_0+\mu_1}^{\gamma, \calO_0}
\lVert \mathfrak{I}_0 \rVert_{s+\mu_1}^{\gamma, \calO_0}\} 
\lVert g \rVert_{s_0+\mu_1}^{\gamma, \Omega_{\infty}} \Big)\,.
\end{aligned}
\end{equation*}
\end{theorem}

The proof of the above theorem follows word by 
word the one of Theorem $6.10$ in \cite{KdVAut}.
It is based on tame bounds for the map $G_{\delta}$ 
and the coefficients of the Taylor expansion of the 
Hamiltonian $K$ at the trivial torus $(\varphi, 0, 0)$, 
see Lemmata \ref{Lemma6.4DP}, 
\ref{Lemma6.6dp}, \ref{isotropictorus}, \ref{LemmaDG}.

%%%%%%%%%%%%%%%%%%%%%%%%%%%%%%%%%%%%%%%%%%%%%%%%%%%

\section{The linearized operator in the normal directions}\label{sec:Linopnormal}

Recalling the assumption \eqref{AssumptionDP}, 
in the sequel we assume that 
$\mathfrak{I}_{\delta}:=\mathfrak{I}_{\delta}(\varphi; \omega)
=i_{\delta}(\varphi;\, \omega)-(\varphi,\, 0,\, 0)$ 
satisfies, for some $\gotp_1=\gotp_1(\nu)>0$,
\begin{equation}\label{IpotesiPiccolezzaIdeltaDP}
\lVert \mathfrak{I}_{\delta} \rVert_{s_0+\gotp_1}^{\g,\calO_0}
\lesssim\,\varepsilon^{16-2 b}\gamma^{-1}\,,
\end{equation}
and the momentum condition \eqref{trawaves}.
We note that $G_\delta$ in \eqref{GdeltaDP} 
is the identity plus a translation plus a finite rank linear 
operator and 
it is
 $O(\varepsilon^{16-2 b}\gamma^{-1})$-close to the identity in low norm
 $s_0+\gotp_1$.
Returning to the initial variables 
we define  (see \eqref{AepsilonDP}, \eqref{GdeltaDP}) 
$T_{\delta}:=T_{\delta}(\vphi)$ as
\begin{equation}\label{TdeltaDP}
T_{\delta}
:={\bf A}_{\varepsilon}(\Lambda G_{\delta}(\varphi, 0, 0))
= \e \vect{v_\delta}{\ov{v_\delta}} +\e^b W_0\,,
\quad v_\delta = \sum_{j\in S} \sqrt{\zeta_j + \e^{2b-2}y_{\delta j}(\varphi)} 
e^{\mathrm{i} (j  x - \theta_{0 j}(\varphi))  }
\end{equation}
and we have, for some $\mu:=\mu(\nu)>0$,
\begin{equation*} %\label{balaban}
\lVert \Phi_B(T_{\delta}) \rVert_s^{\gamma, \calO_0}
\lesssim_s \varepsilon\, (1+\lVert \mathfrak{I}_{\delta} 
\rVert^{\gamma, \calO_0}_{s+\mu})\,, 
\qquad
\lVert d_i \Phi_B(T_{\delta}) [\hat{\imath}] \rVert_s
\lesssim_s \varepsilon(\lVert \hat{\imath}\rVert_{s+\mu}
+ \lVert \mathfrak{I}_{\delta} \rVert_{s+\mu}
\lVert \hat{\imath} \rVert_{s_0+\mu})\,.
\end{equation*}
By following section $7$ in \cite{KdVAut} (see Lemma $7.1$),  
$K_{02}$ in \eqref{thirdeqDP} %\eqref{K020} 
has rather explicit estimates 
(see also Proposition 6.2 in 
\cite{FGP1}).
In the next proposition we will give a more 
explicit formulation of $K_{0 2}(\vphi)$.
%in \eqref{eq3},\eqref{K020}.
Notice that,
by the shape derivative formula \eqref{shapeDer},  
the linearized operator of \eqref{eq:113}
at $(\eta,\psi)(\vphi,x)$ is given by 
\begin{equation}\label{linWW}
\mathcal{L}:=\mathcal{L}(\vphi):=\omega\cdot\pa_{\vphi}+
\left(
\begin{matrix}
\pa_{x}V+G(\eta)B & -G(\eta) \\
(1+BV_x)+BG(\eta)B & V\pa_{x}-BG(\eta)
\end{matrix}
\right)\,.
\end{equation}
where $V$, $B$  are given in  \eqref{def:V} and \eqref{form-of-B}.
Hence we have the following.

\begin{proposition}\label{LinearizzatoIniziale}
Assume \eqref{IpotesiPiccolezzaIdeltaDP}. Then there 
exists $\mu_0=\mu_0(\nu)>0$ such that the 
following holds. The Hamiltonian 
operator $\mathcal{L}_{\omega}$ in \eqref{thirdeqDP} %\eqref{eq3}
has the form
\begin{equation}\label{BoraMaledetta}
\mathcal{L}_{\omega}=
\Pi_S^{\perp}\big(\omega\cdot\partial_{\varphi}
-J\pa_{(\eta,\psi)}\nabla_{(\eta,\psi)}H\big(\Lambda^{-1}
\Phi_{B}(T_{\delta})\big)
+\mathcal{Q}_0\big)
=\Pi_S^{\perp}\big( \mathcal{L}+\mathcal{Q}_0\big) \,,
\end{equation}
where $T_{\delta}$ is defined in \eqref{TdeltaDP}, $\Phi_B$ is the 
Birkhoff map given in Proposition \ref{WBNFWW},  
$H$ is the Hamiltonian in \eqref{Hamiltonian}
and $\mathcal{L}$ is the operator in \eqref{linWW}.
The operator $\mathcal{Q}_0$ is finite rank and has the form
\begin{equation}\label{FiniteDimFormDP}
\mathcal{Q}_0(\varphi) w=\sum_{\lvert j \rvert\le C} 
\int_0^1 (w, g_j(\tau, \varphi))_{L^2}\,
\chi_j(\tau, \varphi)\,d\tau\,,
\end{equation}
for some functions 
$g_{j}(\tau,\cdot), \chi_{j}(\tau,\cdot)\in H^{s}\cap H_{S}^{\perp}$. 
In particular we consider the expansion
$\mathcal{Q}_0= \sum_{i=1}^{12} \e^i\mathcal{R}_i +\mathcal{R}_{>12}$,
where the remainders $\mathcal{R}_i$, $i=1,\ldots,12$ have 
the form \eqref{FiniteDimFormDP}, they
do not depend 
on $\mathfrak{I}_\delta$ and satisfy
\begin{equation}\label{DecayR2dp}
\lVert g^{(i)}_j \rVert_s^{\gamma, \calO_0}
+\lVert \chi_j^{(i)} \rVert_s^{\gamma, \calO_0}
\lesssim_s 1\,,\qquad i=1,\ldots,12\,,
\end{equation}
while $\mathcal{R}_{>12}$ satisfies
\begin{equation}\label{DecayR*dp}
\begin{aligned}
 \lVert g^{>12}_j \rVert_s^{\gamma, \calO_0}
 \lVert \chi_j^{>12} \rVert_{s_0}^{\gamma, \calO_0}&
 +\lVert g^{>12}_j \rVert_{s_0}^{\gamma, \calO_0}
 \lVert \chi_j^{>12} \rVert_{s}^{\gamma, \calO_0}\lesssim_s 
\varepsilon^{13}+\varepsilon^{2}\lVert \mathfrak{I}_{\delta} 
\rVert_{s+\mu_0}^{\gamma, \calO_0}\,, %\label{DecayR*dp}
\\[2mm]
\lVert d_i g^{>12}_j [\hat{\imath}] \rVert_s\lVert \chi_j^{>12} \rVert_{s_0}
+\lVert d_i g^{>12}_j [\hat{\imath}] \rVert_{s_0}\lVert \chi_j^{>12} \rVert_{s}
&+\lVert g^{>12}_j \rVert_{s_0} \lVert d_i \chi^{>12}_j[\hat{\imath}] \rVert_{s}
+\lVert g^{>12}_j \rVert_{s} \lVert d_i \chi^{>12}_j [\hat{\imath}]\rVert_{s_0} 
%\label{DecayR*dp2}
\\ 
&\qquad \qquad \qquad \qquad
\lesssim_s \varepsilon^{2} \lVert \hat{\imath} \rVert_{s+\mu_0}
+\varepsilon^{b}\lVert \mathfrak{I}_{\delta} \rVert_{s+\mu_0}
 \lVert \hat{\imath} \rVert_{s_0+\mu_0}\,.
 %\nonumber
\end{aligned}
\end{equation}
Finally, recalling the Definition \ref{LipTameConstants},
 we have
 \begin{equation}\label{dude}
\begin{aligned}
\mathfrak{M}^{\gamma}_{\mathcal{Q}_0}(0,s)
&\lesssim_s \varepsilon^2 
(1+\lVert \mathfrak{I}_{\delta}
\rVert^{\gamma, \calO_0}_{s+\mu_0})\,,\\
\mathfrak{M}_{ d_i \mathcal{Q}_0 [\hat{\imath}] }(0,s)&\lesssim_s 
\varepsilon^2 \lVert \hat{\imath} \rVert_{s+\sigma_0}+\varepsilon^{2b-1} 
\lVert \mathfrak{I}_{\delta}\rVert_{s+\mu_0}
\lVert \hat{\imath} \rVert_{s_0+\mu_0}\,.
%\label{dude2}
\end{aligned}
\end{equation}
\end{proposition}

\begin{proof}
The expression \eqref{BoraMaledetta} follows from the 
Definition \eqref{thirdeqDP} and \eqref{Kdp} %\eqref{K020}  
by remarking that $G_\delta$ (up to a translation)
and the weak BNF transformation $\Phi_B$ are the identity plus a 
finite rank operator, while the action angle change 
of coordinates is a rescaling plus 
a finite rank operator.
Then by chain rule  we get  
\begin{equation*}
\begin{aligned}
d_{\widetilde{W}} \nabla_{\widetilde{W}} (H_\e\circ G_\delta ) &\stackrel{\eqref{HamiltonianaRiscalataDPreal}}{=}
 \e^{-2b} (d_{W} \nabla_W (\mathcal{H}_{\mathbb{C}}
 \circ {\bf A}_\e))\circ G_\delta 
 + R_1 
=   (d_{z} \nabla_z \mathcal{H}_{\mathbb{C}})
\circ {\bf A}_\e\circ G_\delta + R_1   \\ 
&\stackrel{\eqref{piotta}}{=}
(d_{W} \nabla_{W} (H\circ\Lambda^{-1}\circ\Phi_{B}))
\circ {\bf A}_\e\circ G_\delta + R_1 
\\&
=  (d_{W} \nabla_{W} H)\circ\Lambda^{-1}
\circ\Phi_{B}\circ {\bf A}_\e\circ G_\delta
 + R_1 + R_2
\end{aligned}
\end{equation*}
%\comment{questa proof forse si puo' scrivere piu' corta usando lemma7.1 in KdVAut}
 where the finite rank part contains 
all the terms where a derivative falls 
on $\Phi_{B}\circ {\bf A}_\e\circ G_\delta$.
Then \eqref{BoraMaledetta} follows from 
the definition of $H$ in \eqref{Hamiltonian} (see \eqref{linWW}).
%Regarding the estimates, \eqref{unasuaziaricca} 
%follows from \eqref{balaban}; r
Regarding the 
bounds \eqref{DecayR2dp}, \eqref{DecayR*dp}, 
we split the finite rank part $R_1+R_2$ as follows. 
The operator $R_1$ contains all terms arising from 
derivatives of $G_\delta$. By tame estimates 
on the map $G_{\delta}$ 
(see for instance Lemma $6.7$ in \cite{KdVAut}), 
it satisfies the bounds \eqref{DecayR*dp} and 
 we put it in $\mathcal{R}_{>12}$. The finite rank term
$R_2$   comes from the Birkhoff map.  
This is an analytic map so we consider
 the Taylor expansion 
\begin{equation*} %\label{luchino}
\Phi_B(u,\bar{u})=\vect{u}{\bar{u}}+\sum_{i=2}^{12}\Psi_i(u,\bar{u})
+\Psi_{\geq 13}(u,\bar{u})\,,
\end{equation*}
where each  $\Psi_i(u)$
is homogeneous of degree $i$ in $u$, while   
$\Psi_{\geq 13}=O(u^{13})$ and they all map 
$H^1_0(\mathbb{T})\times H_{0}^{1}(\mathbb{T})$ in itself. 
We have to evaluate $\Phi_B$ and 
its derivatives (up to order two) at  
$u=T_{\delta}$
%=\varepsilon v_{\delta}+\varepsilon^b z_0$
given in \eqref{TdeltaDP}.
We denote by ${v}_{I}$ the traveling wave
\begin{equation}\label{vsegnatoDP}
{v}_{I}(\varphi, x):=\sum_{j\in S} \sqrt{\zeta_j} 
e^{\mathrm{i} (j x+\mathtt{l}(j)\cdot \varphi)}=\sum_{j\in S} \sqrt{\zeta_j} 
e^{\mathrm{i} (\varphi-\mathtt{v} x)\cdot \mathtt{l}(j)}= A_\e(\f,0,0)
\end{equation}
where $\mathtt{l}(\bar\jmath_i):=-e_i$, $e_i $ being the $i$-th vector of the canonical 
basis of $\mathbb{Z}^{\nu}$. 
%\comment{
%Non credo che serva altro. La condizione
%%and is such that 
%$\mathtt{l}(-\bar{\jmath}_i)=-\mathtt{l}(\bar{\jmath}_i)$
%serve solo se stai facendo le standing wave, soluzioni pari. 
%Per questo gli altri hanno sempre sta condizione.
%Oppure forse per dire che la $v_{I}$ e' reale. Infatti nella DP c'e'.
%Qui NO! Caso generale
%}
\noindent
We observe that\footnote{The function 
$\e\Lambda^{-1}\vect{v_{I}}{\ov{v_{I}}}$
%$\e\bar v$ 
represents a torus supporting a 
quasi-periodic motion which is invariant for the system 
\eqref{HepsilonDP} with $P=0$, namely it is 
the approximate solution from which we bifurcate. }
\begin{equation*}\label{differenzaV}
\lVert v_{\delta}-{v}_{I}\rVert^{\gamma, \calO_0}_s
\lesssim \lVert \mathfrak{I}_{\delta}\rVert^{\gamma, \calO_0}_{s}\,,
\end{equation*}
and hence we can expand 
\begin{equation}\label{PhiBdp}
\Phi_B(T_{\delta})=\e \vect{v_{I}}{\ov{v_{I}}}+\sum_{i=2}^{13}\e^i\Psi_i(v_{I}, \ov{v_{I}})
+\tilde{q}= \Phi_B^{\le 12} + \tilde q \,,
\end{equation}
where $\tilde{q}$ is a remainder which satisfies
\begin{equation}\label{QtildaDP}
\lVert \tilde{q} \rVert_s^{\gamma, \calO_0}\lesssim_s 
\varepsilon^{13}+
\varepsilon\lVert \mathfrak{I}_{\delta} \rVert_s^{\gamma, \calO_0}\,, 
\quad \lVert d_i \tilde{q} [\hat{\imath}] \rVert_s
\lesssim_s \varepsilon (\lVert \hat{\imath} 
\rVert_s+\lVert \mathfrak{I}_{\delta} \rVert_s 
\lVert \hat{\imath} \rVert_{s_0})\,. 
\end{equation}
Then in $\mathcal{R}_i$ we include all the terms 
homogeneous of degree $i$  coming from 
derivatives of $\Phi_B-\mathrm{I}$ , evaluated at $\tilde q=0$; we put in 
$\mathcal{R}_{>12}$ all the rest.
The \eqref{dude} %, \eqref{dude2} 
follows by \eqref{DecayR2dp}, 
\eqref{DecayR*dp}.
%and \eqref{DecayR*dp2}.
\end{proof}

\begin{remark}\label{ordinivari}
The motivation for separating the $\mathcal{R}_i$ 
and $\mathcal{R}_{>12}$ 
is the following. Consider the Hamiltonian $H_\e$ 
as a function of $\zeta$ instead of $\omega$. 
Then in all our expressions we can, and shall, 
evidence a {\it purely polynomial} term 
$\sum_{i=0}^{12} \e^i f_i(\zeta)$ (where the $f_i$ are 
$\e$ independent)  plus a {\it remainder}, which is not analytic in $\e$, of size 
$\varepsilon^{13}+\varepsilon\lVert \mathfrak{I}_{\delta} 
\rVert_s^{\gamma, \calO_0}$. 
By the assumption \eqref{AssumptionDP}, 
this means that in {\it low norm} $s=s_0+\gotp_1$ 
all these remainders are negligible 
w.r.t. terms of order $\e^{12}$. 
In this framework $\mathcal{R}_{>12}$ is purely 
a remainder, while the $\mathcal{R}_i$  are 
homogeneous polynomial terms. 
\end{remark}

\subsection{Homogeneity expansions}
In the following we shall assume that 
the assumption \eqref{IpotesiPiccolezzaIdeltaDP} holds true for
some $\gotp_1$ large enough. The constant $\gotp_1$
represents the \emph{loss of derivatives} accumulated along 
the reduction procedure of
subsequent sections. 
In order to estimate the variation of the eigenvalues 
with respect to the approximate invariant torus, 
we need also to estimate the variation with respect to 
the torus embedding 
$i(\vphi)$ in a low norm $\|\cdot\|_{p}$
%for all Sobolev indexes $p$
such that
\begin{equation}\label{condireP}
s_0\le p\ll s_0+\gotp_1\,. %\quad {\rm for\; some}\,\;\; \mu_0=\mu_0(\tau,\nu)>0\,.
\end{equation}
From now on we denote with $\mu$ an increasing running index that represents the loss of derivatives appearing at each step of the reducibility procedure.

We need the following definition. 
\begin{definition}\label{espOmogenea}
%{\bf ($(p,k)$-Homogenenous symbols)}
 Let $1\leq k\leq 6$.
Assume that \eqref{IpotesiPiccolezzaIdeltaDP} holds and let $a\in S^{m}$ 
for some $m\in \mathbb{R}$, $\mathtt{R}\in \mathfrak{L}_{\rho,p}$
 (see Def. \ref{pseudoR}, \ref{ellerho}), 
 depending in a Lipschitz way on $\omega\in\mathcal{O}\subseteq \mathcal{O}_0$
and $\mathfrak{I}_{\delta}$.
We say that $a\in S_{k}^{m}$ if it can be written as
\begin{equation}\label{espandoenonmifermo}
a(\vphi,x,\x)=\sum_{i=1}^{14-2k}\e^{i} a_{i}(\varphi, x, \xi)+\widetilde{a}(\vphi,x,\x)
\end{equation}
where $a_{i}$ is a $i$-homogeneous symbol of the form
\begin{equation}\label{simboOMO2DEF}
a_{i}(\vphi,x,\x):=\frac{1}{(\sqrt{2\pi})^{i}}\sum_{\substack{j_1,\ldots,j_{i}\in S \\ 
\s_i=\pm}}(\mathtt{a}_i)_{j_1,\ldots,j_{i}}^{\s_1,\ldots,\s_{i}}(\x)
\sqrt{\zeta_{j_1}\cdots\zeta_{j_i}}e^{\ii(\s_1 j_1+\ldots+\s_{i}j_{i})x}
e^{\ii(\s_1\mathtt{l}(j_1)+\ldots+\s_{i}\mathtt{l}(j_i))\cdot\f}
\end{equation}
for some Fourier multipliers $(\mathtt{a}_i)_{j_1,\ldots,j_{i}}^{\s_1,\ldots,\s_{i}}(\x)\in S^{m}$,
the symbol $\widetilde{a}\in S^{m}$ satisfies 
\begin{equation}\label{AVBDEF}
\begin{aligned}
|\widetilde{a}|_{m,s,\alpha}^{\gamma,\calO}&\lesssim_{m,s,\alpha} 
\gamma^{1-k}(\e^{13}+\e\|\mathfrak{I}_{\delta}\|_{s+\mu}^{\gamma,\calO_0})\,,
%\label{AVBDEF}
\\
|\Delta_{12}\widetilde{a}|_{m,p,\alpha}&\lesssim_{m,p,\alpha}
\e\gamma^{1-k}
(1+\|\mathfrak{I}_{\delta}\|_{p+\mu})\|i_{1}-i_{2}\|_{p+\mu}\,,
%\label{AVB2DEF}
\end{aligned}
\end{equation}
for some $\mu=\mu(\nu)>0$.
Similarly, we write 
$\mathtt{R}\in \mathfrak{L}_{\rho,p}^{k}\otimes\mathcal{M}_2(\mathbb{C})$ 
to denote a remainder 
$\mathtt{R}$ which has the form
\begin{equation}\label{espandoenonmifermo2}
\mathtt{R}=\sum_{i=1}^{14-2k}\e^{i} \mathtt{R}_{i}+\widetilde{\mathtt{R}}
\end{equation}
where $\mathtt{R}_i$ is a $i$-homogeneous smoothing operator  of the form
\begin{equation}\label{simboOMO2DEF2}
\begin{aligned}
\mathtt{R}_i
& =\left(
\begin{matrix}
(\mathtt{R}_i)_{+}^{+} & (\mathtt{R}_i)_{+}^{-}\\
(\mathtt{R}_i)_{-}^{+} & (\mathtt{R}_i)_{-}^{-}
\end{matrix}
\right)\,,  
%\quad 
%(\mathtt{R}_i)_{\s}^{\s'}
%\in \mathfrak{L}_{\rho,p}\,,
%\widetilde{\mathcal{R}}^{-\rho}_i \, ,  
\quad 
 (\mathtt{R}_{i})_{\s}^{\s'}=\ov{ (\mathtt{R}_{i})_{-\s}^{-\s'}} \, , \\
 (\mathtt{R}_i)_{\s}^{\s'}z^{\s'}&:=\frac{1}{\sqrt{2\pi}}
 \sum_{j\in \mathbb{Z}\setminus\{0\}}e^{\ii \s jx}\big( \sum_{k\in \mathbb{Z}} 
 (\mathtt{R}_i)_{\s,j}^{\s',k}z^{\s'}_{k}\big)\,,\\
(\mathtt{R}_{i})_{\s,j}^{\s',k}&
:=\frac{1}{(\sqrt{2\pi})^{i}}\sum_{\substack{
\sum_{q=1}^{i}\s_{q}j_{q}=\s j-\s' k
%j_1,\ldots,j_{i}\in S \\ 
%\s_i=\pm
}}
 \big((\mathtt{R}_i)_{\s,j}^{\s',k}\big)_{j_1,\ldots,j_{i}}^{\s_1,\ldots,\s_{i}}
%(\mathtt{R}_i)_{j_1,\ldots,j_{i}}^{\s_1,\ldots,\s_{i}}
\sqrt{\zeta_{j_1}\cdots\zeta_{j_i}}
%e^{\ii(\s_1 j_1+\ldots+\s_{i}j_{i})x}
e^{\ii(\s_1\mathtt{l}(j_1)+\ldots+\s_{i}\mathtt{l}(j_i))\cdot\f}
\end{aligned}
\end{equation}
for some coefficients 
$((\mathtt{R}_i)_{\s,j}^{\s',k})_{j_1,\ldots,j_{i}}^{\s_1,\ldots,\s_{i}}\in \mathbb{C}$
and $(\mathtt{R}_i)_{\s}^{\s'}
\in \mathfrak{L}_{\rho,p}$,
the operator $\widetilde{\mathtt{R}}\in \mathfrak{L}_{\rho,p}(\calO)$ satisfies 
\begin{equation}\label{AVBDEF2}
\begin{aligned}
\mathbb{M}^{\gamma}_{\widetilde{\mathtt{R}}}(s, \mathtt{b})
&\lesssim_{s,\rho}
\gamma^{1-k}(\e^{13}+\e\|\mathfrak{I}_{\delta}\|_{s+\mu}^{\gamma,\calO_0})\,,
\quad \quad \quad \qquad 0\leq \mathtt{b}\leq\rho-2\,,
%\label{AVBDEF2}
\\
\mathbb{M}_{\Delta_{12}\widetilde{\mathtt{R}}}(s, \mathtt{b})&\lesssim_{p,\rho}
\e\gamma^{1-k}
(1+\|\mathfrak{I}_{\delta}\|_{p+\mu})\|i_{1}-i_{2}\|_{p+\mu}
\,,\quad 0\leq \mathtt{b}\leq \rho-3\,,
%\label{AVB2DEF2}
\end{aligned}
\end{equation}
for some $\mu=\mu(\nu)>0$.
\end{definition}

We now prove that the results of section \ref{sec:pseudopseudo} extends 
to the classes of symbols introduced above.
\begin{lemma}\label{lem:escobarhomo}
Fix $\rho,p$ as in Definition \ref{ellerho}. 
Let $m\in\mathbb{R}$, $N:=m+\rho$, $1\leq k\leq6$, 
and consider a symbol $b(\vphi,x,\x)$ in $S_{k}^{m}$.
Assume the \eqref{IpotesiPiccolezzaIdeltaDP}.
 Then there exists a remainder $R_{\rho}\in \gotL^{k}_{\rho,p}$ such that 
\begin{equation}\label{passohomo}
\op(b)=\opw(c)+R_{\rho}\,, 
\end{equation} 
where $c$ has the  form \eqref{passo}, belongs to $S_{k}^{m}$ 
and satisfies \eqref{AVBDEF} with $\mu\rightsquigarrow \mu+N$.
\end{lemma}

\begin{proof}
One deduce the lemma on the symbol $c$ 
by using the explicit expression 
\eqref{passo} which is linear in $b$. So the homogeneous expansion of $c$
follows from the one of $b$.
The remainder $R_{\rho}$ is equal to $\opw(\tilde{a})$ where $\tilde{a}$
is in Lemma \ref{weyl/standard}. The expression of $\tilde{a}$ can be computed explicitly
as one can deduce from the proof of Lemma $3.5$ in \cite{BD}.
In particular it is given in formul\ae\, $(3.2.11), (3.2.12), (3.2.13)$ therein.
Such expressions are linear in $b$. Therefore $\tilde{a}$ admits the 
expansion as in \eqref{espandoenonmifermo}.
The non homogeneous terms in the expansions of $c$ and $R_{\rho}$
satisfy estimates \eqref{AVBDEF}, \eqref{AVBDEF2} using the \eqref{passo2}, 
\eqref{passo3} and the fact that,  by hypothesis, the symbol $b$ belongs to $S_{k}^{m}$
and hence satisfies  \eqref{AVBDEF}, \eqref{AVBDEF2}.
\end{proof}

\begin{lemma}{\bf (Composition).}\label{Jameshomo}
 %$\rho\in \mathbb{N}$, 
Let $m,m'\in \mathbb{R}$, $k,k'\in \mathbb{N}$, $k'\leq k$.
Fix  $\rho,p$ as in Definition \ref{ellerho},
such that $\rho \geq \max\{-(m+m'+1), 3\}$
and  define $N:=m+m'+\rho\geq1$.
Consider two symbols
$a(\vphi,x,\x)\in S_{k}^m$, $b(\vphi,x,\x)\in S_{k'}^{m'}$.
Assume the \eqref{IpotesiPiccolezzaIdeltaDP}.
There exist an operator $R_{\rho}\in \gotL^{k'}_{\rho, p}$ 
and a constant $\mu=\mu(N)\sim N$
such that (recall Def. \ref{cancelletti}) 
\begin{equation}\label{escobar12homo}
\opw(a)\circ\opw(b)
=\opw(c)+R_{\rho}, \qquad c:=a\#^{W}_{N} b\in S_{k'}^{m+m'}\,.
\end{equation}
\end{lemma}

\begin{proof}
In Lemma \ref{James}
the symbol $c$ is constructed by using 
 the expansion $a\#^{W}_{N}b$ 
 (defined in \eqref{espansione2}). Hence, by linearity, 
 this expansion is a sum of 
 $a_i\#^{W}_{N} b_j$, where $a_i$ and $b_j$ 
 are the homogeneous terms of $a$ and $b$, plus the term
 $\tilde{a}\#^{W}_{N} \tilde{b}$ where  
 $\tilde{a}$ and $\tilde{b}$  are respectively 
 the non homogeneous symbols in the
 expansions for $a$ and $b$.
If $i+j\geq 15-2k'$ then $\e^{i+j}a_i\#^{W}_{N} b_j$ satisfies the \eqref{AVBDEF}
with $k\rightsquigarrow k'$, and then can be considered in $\tilde{c}$.
The term  $\tilde{a}\#^{W}_{N} \tilde{b}$ satisfies \eqref{AVBDEF}
since $\tilde{a},\tilde{b}$
satisfies the same estimates.
One can deduce from the proof of Lemma \ref{James} that 
the remainder $R_{\rho}$
is the sum of three operators $L_{\rho}, \tilde{R}_{\rho}, Q_{\rho}$ 
belonging to $\mathfrak{L}_{\rho,p}^{k'}$.
Indeed $\tilde{R}_{\rho}, Q_{\rho}$ are obtained by applying Lemma 
\ref{lem:escobarhomo} (which is the counterpart of Lemma \ref{lem:escobar}).
The operator $L_{\rho}$ has an explicit expression (in terms of $a,b$) given by formula
$(2.30)$ in \cite{BM1}.
 The estimates \eqref{AVBDEF2} on the non-homogeneous term of $R_{\rho}$ 
 follow by the \eqref{jamal}, \eqref{jamal2}
 and the estimates on the symbols $a,b$.
\end{proof}

Recall the function $v_{I}(\vphi,x)$ 
defined in \eqref{vsegnatoDP}.
Notice that, by setting
\begin{equation}\label{coeffvI}
v_{j}:=(v_{I})_{j}(\vphi):= \sqrt{\zeta_{j}}e^{\ii \mathtt{l}(j)\cdot\vphi}\,, \qquad j\in S\,,
\end{equation}
we can write
\begin{equation}\label{vsegnatoDP2}
v_{I}(\vphi,x)=\sum_{j\in S} \sqrt{\zeta_j} 
e^{\mathrm{i} (j x+\mathtt{l}(j)\cdot \varphi)}=\sum_{j\in S}v_{j}\frac{e^{\ii jx}}{\sqrt{2\pi}}.
\end{equation}
We shall also use the notation $v_{j}^{\s}$, $\s=\pm$ where
$v_{j}^{+}=v_{j}$, $v_{j}^{-}=\ov{v_{j}}$, $j\in S$.
Let $a_{k}(\vphi,x,\x)$ be a $k$-homogeneous symbol of the form
\eqref{simboOMO2DEF}. Of course such symbol depends only on the function $v_{I}$ in \eqref{vsegnatoDP2}. In particular we can write
\begin{equation}\label{iena}
\begin{aligned}
a_{k}(\vphi,x,\x)&=
%\sum_{\substack{j_1,\ldots,j_{k}\in S \\ 
%\s_i=\pm}}(\mathtt{a}_k)_{j_1,\ldots,j_{k}}^{\s_1,\ldots,\s_{k}}(\x)
%\sqrt{\x_{j_1}\cdots\x_{j_i}}e^{\ii(\s_1 j_1+\ldots+\s_{k}j_{k})x}
%e^{\ii(\s_1\mathtt{l}(j_1)+\ldots+\s_{k}\mathtt{l}(j_k))\cdot\f}\\
%&=
\sum_{\substack{j_1,\ldots,j_{k}\in S \\ 
\s_i=\pm}}(\mathtt{a}_k)_{j_1,\ldots,j_{k}}^{\s_1,\ldots,\s_{k}}(\x)
v_{j_1}^{\s_1}v_{j_2}^{\s_2}\cdots v_{j_k}^{\s_k}
e^{\ii (\s_1 j_1+\ldots+\s_{k}j_{k})x}\,.
\end{aligned}
\end{equation}
Notice that 
homogeneous symbols and operators of Definition \ref{espOmogenea}
do not depend on $\mathfrak{I}_{\delta}$.

\begin{lemma}\label{iena10}
%Assume that the condition $(i),(ii),(iii)$ above hold. T
Consider  a homogeneous symbol $a_{k}(\vphi,x,\x)$ with
coefficients
$(\mathtt{a}_k)_{j_1,\ldots,j_{i}}^{\s_1,\ldots,\s_{k}}(\x)$ as in \eqref{iena}.
Then

\noindent
{\bf Reality:}% {for any $v_{I}(\vphi,x)$ }
 the symbol $a_{k}(\vphi,x,\x)$ is \emph{real valued}
if and only if
\begin{equation}\label{iena2}
 \ov{(\mathtt{a}_k)_{j_1,\ldots,j_{k}}^{\s_1,\ldots,\s_{k}}(\x)}=
(\mathtt{a}_k)_{j_1,\ldots,j_{k}}^{-\s_1,\ldots,-\s_{k}}(\x)\,.
\end{equation}

\vspace{0.5em}
\noindent
{\bf Anti-reality:} %for any $v_{I}(\vphi,x)$ 
the symbol $a_{k}(\vphi,x,\x)$ is purely imaginary 
if and only if
\begin{equation}\label{iena2PureIm}
 \ov{(\mathtt{a}_k)_{j_1,\ldots,j_{k}}^{\s_1,\ldots,\s_{k}}(\x)}=-
(\mathtt{a}_k)_{j_1,\ldots,j_{k}}^{-\s_1,\ldots,-\s_{k}}(\x)\,.
\end{equation}
Moreover $a_{k}\in S_{k}^{m}$ and $\mathtt{R}_{k}\in \mathfrak{L}_{\rho,p}^{k}$
are $x$-translation invariant, i.e. satisfy respectively the \eqref{space3Xinv} 
and \eqref{constcoeffMomentum}.
\end{lemma}
\begin{proof}
It follows by explicit computations.
\end{proof}
In view of Remark \ref{ordinivari} we show that the operator $\mathcal{L}$
in \eqref{linWW} can be written in terms of symbols and operators belonging to the 
classes introduced in Definition \ref{espOmogenea}.

\begin{lemma}\label{Doo}
The functions  $V$, $B$ in \eqref{def:V} and \eqref{form-of-B} belongs to the class
$S_1^{0}$ of Definition \ref{espOmogenea}. In particular
\begin{equation}\label{expTaylorFunz}
V(\Lambda^{-1}\Phi_{B}\circ T_{\delta}\vect{v_I}{\ov{v_{I}}})=\sum_{k=1}^{12}\e^{k}
{\mathtt{V}}_{k}(\vect{v_{I}}{\ov{v_{I}}})+{\mathtt{V}}_{\geq13}
\end{equation}
with
\begin{equation}\label{simboOMO2}
\mathtt{V}_{k}(\vect{v_{I}}{\ov{v_{I}}}):=\sum_{\substack{j_1,\ldots,j_{p}\in S \\ 
\s_i=\pm}}(\mathtt{V}_k)_{j_1,\ldots,j_{p}}^{\s_1,\ldots,\s_{k}}
\sqrt{\zeta_{j_1}\cdots\zeta_{j_k}}e^{\ii(\s_1 j_1+\ldots+\s_{k}j_{k})x}
e^{\ii(\s_1\mathtt{l}(j_1)+\ldots+\s_{k}\mathtt{l}(j_k))\cdot\f}\,,
\end{equation}
and
\begin{equation}\label{mediediagonali}
( \mathtt{V}_1)^{+}_n = ( \mathtt{V}_1)^{-}_n = \frac{1}{\sqrt{2}} n |n|^{-1/4} \, , 
\qquad 
(\mathtt{V}_2)^{+-}_{n, n} =(\mathtt{V}_2)^{-+}_{n, n} = \frac{1}{2}n | n | \,. 
\end{equation}
The Dirichlet-Neumann   operator has the form $G(\eta)=\opw(|\x|)+\mathcal{R}$
for some $\mathcal{R}\in \mathfrak{L}_{\rho,p}^{1}$, for any $\rho\geq 3$. 
\end{lemma}

\begin{proof}
%In particular, the functions $V,B$ in \eqref{def:V}, 
%\eqref{form-of-B}, \eqref{linWW},
%
In view of \eqref{DNexp3bis}, \eqref{DNexp3}, and using \eqref{def:V}, \eqref{form-of-B}, 
the function $V$
admits the expansion
\begin{equation}\label{expTaylorFunz1000}
V(\eta,\psi)=V(\Lambda^{-1}\vect{u}{\bar{u}})=\sum_{k=1}^{12}\e^{k}\,
\widetilde{\mathtt{V}}_{k} (\vect{u}{\ov{u}})+\widetilde{\mathtt{V}}_{\geq13} (\vect{u}{\ov{u}})\,.
%\qquad
%B=\sum_{k=1}^{14}\e^{k}\mathtt{B}_{k}+\mathtt{B}_{\geq15}\,,
\end{equation}
The function $\widetilde{\mathtt{V}}_{\geq13}$
satisfies \eqref{QtildaDP} thanks to Proposition \ref{DNDNDN2},
 and so it satisfies \eqref{AVBDEF}.
The functions $\widetilde{\mathtt{V}}_{k}$ 
is a homogeneous 
function of $u, \ov{u}$ (of degree $k$). The functions $\widetilde{\mathtt{V}}_{1}$, $\widetilde{\mathtt{V}}_{2}$ are given in Lemma \ref{siespandeV}.
Moreover, since the Hamiltonian in \eqref{Hamiltonian} poisson commutes
with the momentum in \eqref{Hammomento} 
(i.e. $G(\eta)$ depends on $x$ only through the variable $\eta$)
we deduce that  $\widetilde{\mathtt{V}}_{k}$  is supported, in Fourier,
on monomials $u_{j_1}^{\s_1}\cdots u_{j_{k}}^{\s_k}$
such that
\[
\sum_{p=1}^{k}\s_{p}j_{p}=0\,.
\]
The \eqref{expTaylorFunz}, \eqref{simboOMO2} follow by  \eqref{coeffvI}, \eqref{vsegnatoDP2}, \eqref{PhiBdp} and by using the \eqref{expTaylorFunz1000}. 
The estimates \eqref{AVBDEF}, $k=1$,  for $\mathtt{V}_{\geq13}$ follow by the 
estimates on $\widetilde{\mathtt{V}}_{\geq13}$, the estimates on 
$\Phi_{B}$ and $G_{\delta}$.
For the function $B$ one can reason similarly.\\
The function $\mathtt{V}_1=\widetilde{\mathtt{V}}_1$ and $\mathtt{V}_2=\widetilde{\mathtt{V}}_2+\widetilde{\mathtt{V}}_1\circ \Psi_2$ (recall \eqref{PhiBdp}). Then \eqref{mediediagonali} follows by Lemma \ref{lem:V1} and the fact that the $x$-average of $\widetilde{\mathtt{V}}_1\circ \Psi_2$ is zero.\\
By Proposition \ref{DNDNDN} we have that  
$G(\eta):=\opw(|\x|)+R_{G}(\eta)$ with $R_{G}(\eta)$
a pseudo differential operator satisfying \eqref{DN2}.
By Lemma $B.2$ in \cite{FGP} we have that $R_{G}(\eta)$ belongs to 
$\mathfrak{L}_{\rho,p}$. By Taylor expanding $R_{G}(\eta)$ in $\eta$
one can deduce that actually $R_{G}(\eta)\in \mathfrak{L}_{\rho,p}^{1}$.
\end{proof}

\subsection{Hamiltonian structure of the linearized operator}\label{struttHAMlomega}

Following Remark \ref{ordinivari}, we  evidence the homogeneous 
terms 
 in the Hamiltonian of 
$\mathcal{L}_{\omega}$ %, let us call it $\mathsf{H}$, 
whose Hamiltonian vector fields have  degree $\leq 13$, 
since they are NOT perturbative. 
As explained in \eqref{PhiBdp} this entails expanding 
the map $\Phi_B(T_{\delta})$ in powers of $\e$ 
up to order five plus a {\it small remainder $\tilde q$}.

\noindent
We consider the  symplectic form in the extended phase space 
$(\f,Q,W)\in \mathbb{T}^{\nu}\times\mathbb{R}^{\nu}\times H_S^{\perp}$,
(recall the definition of $W$ in \eqref{HepsilonDP}, \eqref{splitto2})
\begin{equation}\label{formasimpletticaestesa}
\Omega_{e}(\varphi, Q, W):=dQ\wedge d\varphi+d \widetilde{\psi}\wedge d\widetilde{\eta}
%\sum_{j\in S^c} 
%d \widetilde{\eta}_j\wedge d \widetilde{\psi}_{-j}
\end{equation}
with the  Poisson brackets 
(recalling  $\{ \cdot, \cdot \}$  defined in \eqref{poissonBra22})
\begin{equation}\label{PoissonEstesa}
\begin{aligned}
\{  F, G \}_{e}&:=\partial_{Q} F \partial_{\vphi} G
-\partial_{\vphi} F \partial_{Q} G+\{ F, G\}\\&
%\{  F, G \}_{e}:
=\partial_{Q} F \partial_{\vphi} G
-\partial_{\vphi} F \partial_{Q} G+
\int_{\T}\big(\nabla_{\widetilde{\eta}}G\nabla_{\widetilde{\psi}}F
-\nabla_{\widetilde{\psi}}G\nabla_{\widetilde{\eta}}F\big)dx\,.
\end{aligned}
\end{equation}
Passing to the complex variables 
$(\f,Q,z,\bar{z})$ in \eqref{splitto2}, we have that,
with abuse of notation, the extended symplectic form \eqref{formasimpletticaestesa}
and the Poisson brackets in \eqref{PoissonEstesa}
reads (recalling \eqref{poissonBraComp})
\begin{align}
\Omega_{e}(\varphi, Q, W)&:=d Q\wedge d\varphi -\ii d z\wedge d\bar{z}\,,\nonumber\\
\{F_{\mathbb{C}}, G_{\mathbb{C}}\}_{e}&:=
\partial_{Q} F_{\mathbb{C}} \partial_{\vphi} G_{\mathbb{C}}
-\partial_{\vphi} F_{\mathbb{C}} \partial_{Q} G_{\mathbb{C}}+
\{F_{\mathbb{C}},G_{\mathbb{C}}\}\label{estesaPois}
\\
&\;=
\partial_{\varphi} F_{\mathbb{C}} \partial_{\eta} G_{\mathbb{C}}
-\partial_{\eta} F_{\mathbb{C}} \partial_{\varphi} G_{\mathbb{C}}-\ii
\int_{\mathbb{T}}
(\nabla_{z}G_{\mathbb{C}}\nabla_{\bar{z}}F_{\mathbb{C}}
-\nabla_{\bar{z}}G_{\mathbb{C}}\nabla_{{z}}F_{\mathbb{C}})dx
\nonumber
\end{align}
where $F_{\mathbb{C}}=F\circ\Lambda^{-1}$ and $G_{\mathbb{C}}=G\circ\Lambda^{-1}$.

\noindent
We denote by $\mathsf{H}$ the Hamiltonian of 
the operator \eqref{BoraMaledetta} 
with respect to the symplectic 
form \eqref{formasimpletticaestesa}. % is (see \eqref{fosco})
%In particular, i
In the complex variables, we have
\begin{equation}\label{hamiltonianalinearizzata}
 \mathsf{H}\circ\Lambda^{-1}
 :=\mathsf{H}_0+\sum_{i=1}^{12}\varepsilon^i\mathsf{H}_i+ \mathsf{H}_{>12}
 +\sum_{i=2}^{12}\varepsilon^i\mathsf{H}_{\mathcal{R}_i}
 +\mathsf{H}_{\mathcal{R}_{>12}}\,,
\end{equation}
\begin{equation}\label{leHamiltonianeLin}
\begin{aligned}
&\mathsf{H}_0=\overline{\omega}\cdot Q
+\int_{\T} |D|^{\frac{1}{2}}z\cdot \bar{z}\,dx\,, \quad
{\rm and} \;\;\;
\|X_{\mathsf{H}_{>12}}
 \rVert^{\gamma, \calO_0}_s,
  \|X_{\mathsf{H}_{\mathcal{R}_{>12}}}
 \rVert^{\gamma, \calO_0}_s
 \lesssim_s 
 \varepsilon^{13}+\varepsilon\lVert \mathfrak{I}_{\delta}
 \rVert^{\gamma, \calO_0}_{s+\mu}\,,
\end{aligned}
\end{equation}
for some $\mu>0$
and the Hamiltonians $\mathsf{H}_i$ are homogeneous in the variables
$v_{I},\ov{v_{I}}$ given in \eqref{vsegnatoDP}. 
The functions 
$\mathsf{H}_{\mathcal R_i}$, $\mathsf{H}_{\mathcal{R}_{>12}}$ 
are the quadratic forms associated 
to the  linear operators $\mathcal{R}_{i}$, $\mathcal{R}_{>12}$, 
thus the estimates on the Hamiltonian 
vector fields can be deduced from \eqref{DecayR2dp},
\eqref{DecayR*dp}. 

\begin{remark}\label{rmk:corazon}
We note that 
 $\mathcal{L}_{\omega}$
is the linearized operator in the normal directions of the 
Hamiltonian $\mathcal{H}_{\mathbb{C}}=H_{\mathbb{C}}\circ\Phi_{B}$ 
given in Proposition \ref{WBNFWW}
written in the \emph{real} variables.
Since the map $\Phi_{B}$ coincides, up to  degree $2$,
 with $\Phi_{WB}$, 
 the  Taylor expansion
 (up to degree $2$ in $\varepsilon$) of $H_{\mathbb{C}}\circ\Phi_{B}$
 coincides with the Hamiltonian $H_{\mathbb{C}}\circ\Phi_{WB}$
 constructed in section \ref{compaBNFpro}. Therefore 
 the 
 $\e^{2}$-terms of $\mathcal{L}_{\omega}$ are given by the Hamiltonian 
 vector field of
 \[
 \e^{2}(\mathsf{H}_2+\mathsf{H}_{\mathcal{R}_2})=(\Pi^{d_z=2}
 \widehat{H}_{1}^{(4)}\circ {\bf A}_{\e|y=0, \theta=\varphi})
 \]
 where $\widehat{H}_1^{(4)}$ is in \eqref{F3F4Weak}.
\end{remark}

\subsection{Algebraic properties of the linearized operator}\label{sec:linstruct}
The linearized operator $\mathcal{L}$ in \eqref{linWW}
satisfies several algebraic properties which are consequence 
of the symmetries 
of the water waves vector field.
The next lemma is fundamental for our scope.

\begin{lemma}\label{algStrucHamlinear}
Consider functions $(\eta,\psi)\in S_{\mathtt{v}}$ (see \eqref{funzmom})
and the 
linearized operator 
$\mathcal{L}$ in \eqref{linWW}.
Then $\mathcal{L}$ 
is \emph{Hamiltonian} and \emph{$x$-translation invariant}.
The same holds true for the operator $\mathcal{L}_{\omega}$ in \eqref{thirdeqDP}.
%, i.e.
%$\mathcal{C}\mathcal{L}\mathcal{C}^{-1}$
%satisfies \eqref{Ham1comp}-\eqref{self-adjT} 
%and \eqref{traInvcomp}.
\end{lemma}

\begin{proof}
The operator $\omega\cdot\pa_{\vphi}-\mathcal{L}$ 
is the linearized hamiltonian vector field \eqref{eq:113} 
at $(\eta,\psi)$. Hence $J^{-1}(\omega\cdot\pa_{\vphi}-\mathcal{L})$
 is symmetric.
%
%Using \eqref{linWW} we write
%$\mathcal{Q}:=\omega\cdot\pa_{\vphi}-\mathcal{L}$.
%Moreover by denoting 
%by $X$ the vector field in the r.h.s. of \eqref{eq:113}, %\eqref{HS},
%namely
%\begin{equation}\label{XHAM}
%X=X(\eta,\psi)=J \nabla H(\eta,\psi)\,,
%\end{equation}
%where have that
%\[
%\mathcal{Q}(\eta,\psi)[\cdot]:=(dX)(\eta,\psi)[\cdot]
%\]
%where $d$ denotes the differential in the variables $(\eta,\psi)$.
%Then $J^{-1}\mathcal{Q}$ is symmetric
%thanks to the Hamiltonian structure of \eqref{XHAM}.
An explicit computation shows that 
$\mathcal{C}J^{-1}(\omega\cdot\pa_{\vphi}-\mathcal{L})\mathcal{C}^{-1}$
is self-adjoint, i.e. satisfies \eqref{Ham1comp}-\eqref{self-adjT}.
Now the Hamiltonian of the operator $(\omega\cdot\pa_{\vphi}-\mathcal{L})$
is given by
\[
\mathtt{Q}(\vphi,\hat{\eta},\hat{\psi})=\omega\cdot Q+
\frac{1}{2}\int_{\mathbb{T}} J^{-1}\big(\omega\cdot\pa_{\vphi}-\mathcal{L}\big)\vect{\hat\eta}{\hat\psi}
 \cdot\vect{\hat\eta}{\hat\psi}dx\,.
\]
Consider also the Momentum Hamiltonian 
\begin{equation*} %\label{extMom0}
\mathcal{M}_{real}(\vphi,\hat{\eta},\hat{\psi})=-\mathtt{v}\cdot Q+
\int_{\mathbb{T}} \hat{\eta}_x \cdot\hat{\psi}dx\,.
\end{equation*}
By an explicit computation
(recall \eqref{PoissonEstesa}) we have
\begin{equation*} %\label{PoissonMom-QQQ}
\big\{
\mathcal{M}_{real}, \mathtt{Q}
\big\}_{e}=\frac{1}{2}\int_{\mathbb{T}}A_1(\vphi)\hat{\psi}\cdot\hat{\psi}dx+
\frac{1}{2}\int_{\mathbb{T}}A_2(\vphi)\hat{\eta}\cdot\hat{\eta}dx-
\frac{1}{2}\int_{\mathbb{T}}A_3(\vphi)\hat{\eta}\cdot\hat{\psi}dx\,,
\end{equation*}
where $A_{i}(\vphi)$, $i=1,2,3$ are defined as
\begin{align*}
A_1(\vphi)&:= 
-(\mathtt{v}\cdot\pa_{\vphi}G)(\eta)+G(\eta)\pa_{x}-\pa_{x}G(\eta)\,,
%\label{def:A1}
\\
A_2(\vphi)&:= -\Big(\mathtt{v}\cdot\pa_{\vphi} (BG(\eta)B+BV_{x})\Big)
+ (BG(\eta)B+BV_{x})\pa_{x}-\pa_{x} (BG(\eta)B+BV_{x})\,,
%\label{def:A2}
\\
A_3(\vphi)&:= 
-\Big( \mathtt{v}\cdot\pa_{\vphi}(\pa_{x}V+G(\eta)B)\Big)
+(\pa_{x}V+G(\eta)B)\pa_{x}-\pa_{x}(\pa_{x}V+G(\eta)B)\,.
%\label{def:A3}
\end{align*}
By Lemma \ref{MomentumandDN} 
we have that $A_1(\vphi)\equiv0$. This is true since $\eta\in S_{\mathtt{v}}$.
Consider the operator $A_{3}(\vphi)$. % in \eqref{def:A3}.
First of all we note that
\[
\begin{aligned}
-(\mathtt{v}\cdot\pa_{\vphi}(\pa_xV))&+\pa_{x}V\pa_{x}-\pa_{xx}V=
-\big(\mathtt{v}\cdot\pa_{\vphi}V+V_{x}\big)\pa_{x}-\mathtt{v}\cdot\pa_{\vphi}V_{x}-V_{xx}=0
\end{aligned}
\]
since the function $V(\eta,\psi)\in S_{\mathtt{v}}$ (see item $(ii)$ of Lemma \ref{MomentumandDN}). Moreover (recall \eqref{shapeDer})
\[
-(\mathtt{v}\cdot\pa_{\vphi}(G(\eta)B))+G(\eta)B\pa_{x}-\pa_{x}G(\eta)B=
-G'(\eta)\big[\mathtt{v}\cdot\pa_{\vphi}\eta+\eta_x\big]B-G(\eta)\big( 
\mathtt{v}\cdot\pa_{\vphi}B+B_x\big)=0\,,
\]
since (recall Lemmata \ref{funzioniSVVV}, \ref{MomentumandDN})
$\eta,B(\eta,\psi)\in S_{\mathtt{v}}$.
Then we have $A_{3}(\vphi)\equiv0$. Reasoning similarly
one can check that also the operator $A_2(\vphi)$ is identically zero.
Then we proved that $\big\{
\mathcal{M}_{real}, \mathtt{Q}
\big\}_{e}=0$, which implies (writing the Hamiltonians
$\mathcal{M}_{real}, \mathtt{Q}$ in complex variables)
the \eqref{extMomcomp}.
%\eqref{PoissonMom-QQQ}=0$.
Hence, by Lemma \ref{lem:momentocomp},
we have that $\omega\cdot\pa_{\vphi}-\mathcal{L}$ is $x$-translation invariant, i.e.
$\mathcal{C}(\omega\cdot\pa_{\vphi}-\mathcal{L})\mathcal{C}^{-1}$
satisfies \eqref{traInvcomp}.
The operator $\mathcal{L}_{\omega}$ in \eqref{thirdeqDP}
is Hamiltonian by the construction of section \ref{sezione6DP}.
Moreover, by Remark \ref{paganini} (see equation \eqref{K02xtran}), 
we deduce that
\[
\mathtt{v}\cdot\pa_{\vphi}K_{02}(\vphi)+\pa_{x}K_{02}(\vphi)-K_{02}(\vphi)\pa_{x}=0\,.
\]
This implies that the coefficients of $K_{02}(\vphi)$
satisfy \eqref{constcoeffMomentum}. By Lemma \ref{lem:momentocomp}
we have that the operator $K_{02}(\vphi)$, and hence $\mathcal{L}_{\omega}$,
is  $x$-translation invariant.
\end{proof}
%\begin{remark}
%Notice that, by condition \eqref{trawaves}, 
%we have linearized on functions $\eta,\psi$ (given as in \eqref{trawaves2}) which
%belongs to $S_{\mathtt{v}}$. 
%%satisfy \eqref{trawaves2}.
%Hence
%Lemma \ref{algStrucHamlinear} applies.
%In particular the Hamiltonian $ \mathsf{H}\circ\Lambda^{-1}$
%in \eqref{leHamiltonianeLin} satisfies the \eqref{extMomcomp}.
%\end{remark}

%%%%%%%%%%%%%%%%%%%%%%%%%%%%%%%%%%%%%
%%%%%%%%%%%%%%%%%%%%%%%%%%%%%%%%%%%%%
%%%%%%%%%%%%%%%%%%%%%%%%%%%%%%%%%%%%%
%%%%%%%%%%%%%%%%%%%%%%%%%%%%%%%%%%%%%
%%%%%%%%%%%%%%%%%%%%%%%%%%%%%%%%%%%%%

%\section{Good unknown of Alinhac and complex variables}\label{sec:Good}
\section{Symmetrization of the linearized operator at the highest order}\label{sec:Good}
 The aim of the following sections is to 
conjugate the linearized operator $\mathcal{L}_{\omega}$ in \eqref{BoraMaledetta}
to a constant coefficients operator, up to a regularizing 
remainder.
This will be achieved by applying several transformations
which clearly depends nonlinearly on the point $(\eta,\psi)$ in \eqref{trawaves2} 
on which  we linearized.

\subsection{Good unknown of Alinhac}
The aim of this section is to rewrite 
the operator $\mathcal{L}_{\omega}$ 
in \eqref{BoraMaledetta} in terms of the so called ``good unknown`` of Alinhac 
(more precisely its symplectic correction). This will be done in 
Proposition \ref{lemma:7.4}. As we will see, these 
coordinates are the correct ones
in order to diagonalize, at the highest order, the operator $\mathcal{L}_{\omega}$.
We shall first prove some preliminary results.
Following \cite{AB1}, \cite{BM1}
we conjugate the linearized operator $\mathcal{L}$ in \eqref{linWW}
by the operator
\begin{equation}\label{flussoG}
\begin{aligned}
\Phi_{\mathbb{B}}=\Pi_{S}^{\perp}\mathcal{G}\Pi_{S}^{\perp}={\rm Id}+\mathbb{B}\,,\qquad 
\mathbb{B}:=\Pi_{S}^{\perp}\left( 
\begin{matrix} 0 & 0\\ B & 0
\end{matrix}
\right)\Pi_{S}^{\perp}\,,
\qquad
\mathcal{G} :=\left( 
\begin{matrix} 1 & 0\\ B & 1
\end{matrix}
\right)\,,
%\qquad 
%\Phi_{\mathbb{B}}^{-1}={\rm Id}-\mathbb{B}+\mathcal{R}_{2}\,,
\end{aligned}
\end{equation}
%
%\vspace{3em}
%multiplication operator 
%\begin{equation}\label{goodAl}
%\mathcal{G} :=\left( 
%\begin{matrix} 1 & 0\\ B & 1
%\end{matrix}
%\right)\,,\quad \mathcal{G}^{-1} :=\left( 
%\begin{matrix} 1 & 0\\ -B & 1
%\end{matrix}
%\right)
%\end{equation}
where $B$ is the real valued function in \eqref{form-of-B}.
Define the function
\begin{equation}\label{func:a}
a=a(\vphi,x):=(\omega\cdot\pa_{\vphi}B)+VB_x\,.
\end{equation}
%By an explicit computation we have 
%\begin{equation}\label{elle0}
%\tilde{\mathcal{L}}_0=\mathcal{G}^{-1}\mathcal{L}
%\mathcal{G}=
%\omega\cdot\pa_{\vphi}+
%\left(
%\begin{matrix}
%\pa_{x}V & -G(\eta) \\
%1+a & V\pa_{x}
%\end{matrix}
%\right)
%\end{equation}
%where $V$ is in \eqref{def:V} and $a$ is the function 
In the following lemma we study some properties of the map $\Phi_{\mathbb{B}}$.
%$\tilde{\mathcal{L}}_{0}$.
\begin{lemma}\label{mappa1}
%There is $\s=\s(\nu)$ such that the functions $a,V,B$
%have expansions as in \eqref{expTaylorFunz}, \eqref{simboOMO2},
%satisfy \eqref{space3Xinv} and
%\begin{equation}\label{AVB}
%\begin{aligned}
%\|a-1\|_{s}^{\gamma,\calO_0}+\|V\|_{s}^{\gamma,\calO_0}+
%\|B\|_{s}^{\gamma,\calO_0}&\lesssim_s 
%\e(1+\|\mathfrak{I}_{\delta}\|_{s+\s}^{\gamma,\calO_0})\,,\\
%\|\Delta_{12}a\|_{p}^{\gamma,\calO_0}+\|\Delta_{12}V\|_{p}^{\gamma,\calO_0}+
%\|\Delta_{12}B\|_{p}^{\gamma,\calO_0}&\lesssim_p 
%\e(1+\|\mathfrak{I}_{\delta}\|_{p+\s}^{\gamma,\calO_0})\|i_{1}-i_{2}\|_{p+\s}\,,
%%\label{AVB2}
%\end{aligned}
%\end{equation}
%with $p$ as in \eqref{condireP}. 
The function $a(\varphi, x)\in S_1^0$ and it satisfies \eqref{space3Xinv}.
Moreover (recall Definition \ref{LipTameConstants})
the maps $\mathcal{G}^{\pm1}-{\rm Id}$ and $(\mathcal{G}^{\pm1}-{\rm Id})^{*}$
satisfy
\begin{equation}\label{AVB3}
\begin{aligned}
\mathfrak{M}^{\gamma}_{(\mathcal{G}^{\pm1}-{\rm Id})}(s)+
\mathfrak{M}^{\gamma}_{(\mathcal{G}^{\pm1}-{\rm Id})^{*}}(s)&\lesssim_{s}
\e(1+\|\mathfrak{I}_{\delta}\|_{s+\mu}^{\gamma,\calO_0})\,,
\\
\|\Delta_{12}(\mathcal{G}^{\pm1}-{\rm Id})h\|_{p}+
\|\Delta_{12}(\mathcal{G}^{\pm1}-{\rm Id})^{*}h\|_{p}&\lesssim_{p}
\e\|h\|_{p}\|i_1-i_{2}\|_{p+\mu}\,.
%\label{AVB4}
\end{aligned}
\end{equation}
The map $\Phi_{\mathbb{B}}$
is symplectic w.r.t  the symplectic form
\eqref{formasimpletticaestesa} and $x$-translation invariant, 
i.e. is in $\mathfrak{T}_1$ (see Def. \ref{goodMulti}).
One has that $\Phi_{\mathbb{B}}^{-1}={\rm Id}-\mathbb{B}$.
Moreover $\Phi_{\mathbb{B}}^{-1}$, $\Phi_{\mathbb{B}}$ 
satisfy the estimates \eqref{AVB3}. 
\end{lemma}

\begin{proof}
%The proof of \eqref{AVB}, \eqref{AVB3} is the same as the one of Lemma 6.3 in \cite{BM1}.
The homogeneity expansion of $a$ (with estimates \eqref{AVBDEF}) 
follows from the ones of 
$V$ and $B$ (given in Lemma \ref{Doo}) and
by formula \eqref{func:a}.
The functions $a$, $V$, $B$ satisfy the \eqref{space3Xinv} by
Lemma \ref{MomentumandDN}.
% and \ref{funzioniSVVV}.
%by \eqref{condire10}.
Let us now consider the Hamiltonian
\begin{equation*}
H_{\mathbb{B}}=\frac{1}{2}\int_{\mathbb{T}}J^{-1}\mathbb{B}\vect{\tilde{\eta}}{\tilde{\psi}}
\cdot \vect{\tilde{\eta}}{\tilde{\psi}}dx\,,
\end{equation*}
where $J$ is in \eqref{symReal} 
and denote by $\Phi_{\mathbb{B}}^{\tau}$, $\tau\in[0,1]$, the flow generated by the
Hamiltonian $H_{\mathbb{B}}$.
We have that 
\[
\Phi_{\mathbb{B}}=\Phi_{\mathbb{B}}^{1}
=\exp(\mathbb{B})=\sum_{n\geq 0}\frac{1}{n!}\mathbb{B}^{n}=
{\rm Id }+\mathbb{B}+\mathcal{R}_1\,,\qquad 
\mathcal{R}_1:=\sum_{n =2}^{\infty}\frac{1}{n!}\mathbb{B}^{n}\,.
\]
Using that the matrix $\sm{0}{0}{B}{0}$ is nilpotent, 
we deduce that
\[
\mathbb{B}^{2}=
-\Pi_{S}^{\perp}\sm{0}{0}{B}{0}\Pi_{S}\sm{0}{0}{B}{0}\Pi_{S}^{\perp}\,,
\qquad 
\mathbb{B}^{n}=-(-1)^{n}\Pi_{S}^{\perp}\sm{0}{0}{B}{0}
\big(\Pi_{S}\sm{0}{0}{B}{0}\big)^{n-1}\Pi_{S}^{\perp}\,.
\]
%By the formula above w
We note that, given 
$W=\Pi_{S}^{\perp}W=\vect{\tilde{\eta}}{\tilde{\psi}}$,
we have $\mathbb{B}^{2}W\equiv0$. 
Therefore $\mathcal{R}_1=0$ and formula \eqref{flussoG} follows.
%Hence $\Phi_{\mathbb{B}}$ in \eqref{flussoG}
%is the flow at time one of the Hamiltonian $H_{\mathbb{B}}$.
%Then it is symplectic.
By  Lemma \ref{MomentumandDN},
the function $B$ 
belongs to $S_{\mathtt{v}}$ hence 
$\mathcal{G}$ is $x$-translation invariant.
\end{proof}

We now study the conjugate of the operator $\mathcal{L}_{\omega}$.
\begin{proposition}{\bf (Symplectic good unknown).}\label{lemma:7.4}
The conjugate the operator $\mathcal{L}_{\omega}$ in \eqref{BoraMaledetta}
under the map $\Phi_{\mathbb{B}}$ in \eqref{flussoG} has the form
\begin{equation}\label{primamappa}
\mathcal{L}_0=\Phi_{\mathbb{B}}^{-1}\mathcal{L}_{\omega}
\Phi_{\mathbb{B}}=
\Pi_{S}^{\perp} \Big(\omega\cdot\pa_{\vphi}+\opw\left(
\begin{matrix}
\ii\x V+\frac{V_x}{2} & -|\x| \\
1+a & \ii \x V-\frac{V_x}{2}
\end{matrix}
\right)+\widetilde{\mathcal{Q}}_0\Big)\,,
\end{equation}
where $\widetilde{\mathcal{Q}}_0\in \gotL^{1}_{\rho,p}\otimes\mathcal{M}_{2}(\mathbb{C})$ 
(see Def. \ref{ellerho}, \ref{espOmogenea}) for any $\rho\geq3$.
%$\widetilde{\mathcal{Q}}_0$ is finite rank of the form 
%\eqref{FiniteDimFormDP}. 
%Moreover, for $s_0\leq s\leq \mathcal{S}$,
%%it is in $\mathcal{L}_{\rho,p}$ and 
%\begin{equation}\label{resto1}
%\begin{aligned}
%\mathbb{M}_{\widetilde{\mathcal{Q}}_0}^{\gamma}(s,\mathtt{b})&\lesssim_{s,\rho}
%\e(1+\|\mathfrak{I}_{\delta}\|_{s+\s}^{\gamma,\calO_0})
%\,,\qquad
%0\leq \mathtt{b}\leq \rho-2\,,
%\\
%\mathbb{M}_{\Delta_{12}\widetilde{\mathcal{Q}}_0}(p,b)&\lesssim_{p,\rho}
%\e(1+\|\mathfrak{I}_{\delta}\|_{\rho+\s}^{\gamma,\calO_0})\|i_{1}-i_{2}\|_{p+\s}
%\,,\qquad
%0\leq \mathtt{b}\leq \rho-3\,.
%\end{aligned}
%\end{equation}
Finally the operator $\mathcal{L}_0$ is real-to-real,  
Hamiltonian and $x$-translation invariant.
\end{proposition}

\begin{proof}
Notice that
\[
\begin{aligned}
&\Phi_{\mathbb{B}}^{-1}\mathcal{L}_{\omega}\Phi_{\mathbb{B}}=
\Pi_{S}^{\perp}\mathcal{G}^{-1}\mathcal{L}\mathcal{G}\Pi_{S}^{\perp}
+\widetilde{\mathcal{Q}}_0\,,\\
&\widetilde{\mathcal{Q}}_0:=\Pi_{S}^{\perp}\mathcal{G}^{-1}
\mathcal{Q}_0\mathcal{G}\Pi_{S}^{\perp}
-\Pi_{S}^{\perp}\mathcal{G}^{-1}\Pi_{S}(\mathcal{L}+\mathcal{Q}_0)
\Pi_{S}^{\perp}\mathcal{G}\Pi_{S}^{\perp}
-\Pi_{S}^{\perp}\mathcal{G}^{-1}(\mathcal{L}
+\mathcal{Q}_0)\Pi_{S}\mathcal{G}\Pi_{S}^{\perp}\,.
\end{aligned}
\]
Recalling \eqref{linWW} we have that
\begin{equation}\label{elle0}
\mathcal{G}^{-1}\mathcal{L}
\mathcal{G}=
\omega\cdot\pa_{\vphi}+
\left(
\begin{matrix}
\pa_{x}V & -G(\eta) \\
1+a & V\pa_{x}
\end{matrix}
\right)
\end{equation}
where $a$ is the function defined in \eqref{func:a}
By applying 
Lemma \ref{Jameshomo} 
we have
\[
\begin{aligned}
\pa_{x}V&=\opw(\ii \x)\circ\opw(V)
=
\opw\big(\ii \x V+\frac{1}{2\ii }\{\ii \x,V\}\big)=
\opw\big(\ii\x V+\frac{V_x}{2}\big)\,,\\
V\pa_{x}&=\opw(V)\circ\opw(\ii \x)
=
\opw\big(\ii \x V+\frac{1}{2\ii }\{V, \ii \x\}\big)=
\opw\big(\ii\x V-\frac{V_x}{2}\big)\,,
\end{aligned}
\]
up to smoothing terms in $\gotL_{\rho,p}^{1}$.
 %satisfying \eqref{resto1}.
%Moreover, by \eqref{IpotesiPiccolezzaIdeltaDP}, we apply 
%Proposition \ref{DNDNDN}. Hence by
%Lemma B.2 in \cite{FGP},
%we also have that
Moreover, by Lemma \ref{Doo}, we also have that
$G(\eta)[\cdot]:=|D|=\opw(|\x|)$
up to smoothing term $\gotL_{\rho,p}^{1}$.
% satisfying \eqref{resto1}.
Then the formula for the pseudo differential operator in 
\eqref{primamappa} follows.
The remainder $\widetilde{\mathcal{Q}}_0$ is finite rank of
the form \eqref{FiniteDimFormDP}.
Notice that the remainder $\mathcal{Q}_0$ in Proposition 
\ref{LinearizzatoIniziale}  admits expansions in homogeneous remainders, i.e.
belongs to $\mathfrak{L}_{\rho,p}^{1}$ also using Lemma 
$C.7$ in \cite{FGP1}.
Using Lemma \ref{Doo} we conclude that 
$\widetilde{\mathcal{Q}}_0\in \mathfrak{L}^{1}_{\rho,p}$.
%The estimates on $\widetilde{\mathcal{Q}}_0$ 
%follows b
%By  Lemma \ref{mappa1}
%and Proposition \ref{LinearizzatoIniziale} we also have that
%$\widetilde{\mathcal{Q}}_0\in\mathfrak{L}_{\rho,p}^{1}$.
The operator $\mathcal{L}_0$ is 
Hamiltonian and $x$-translation invariant thanks to the properties of 
$\Phi_{\mathbb{B}}$ and Lemma \ref{algStrucHamlinear}.
\end{proof}

\subsection{Complex formulation of Water waves}
We want to rewrite the operator 
$\mathcal{L}_0(\vphi)$ in \eqref{primamappa} in the complex coordinates
\eqref{CVWW}.
Following  the strategy used in Proposition $3.3$ in \cite{BFP} 
we prove the following result.
\begin{proposition}{\bf (Linearized operator in complex variables).}\label{lem:compVar}
There is $\mu=\mu(\nu)>0$ and a matrix of symbols 
$A_{-1}\in S_1^{-1}\otimes\mathcal{M}_{2}(\mathbb{C})$ (see Def. \ref{espOmogenea}), 
satisfying \eqref{space3Xinv}, such that (recall \eqref{CVWW}, \eqref{primamappa})
\begin{equation}\label{elle1}
\begin{aligned}
\mathcal{L}_1(\vphi)&:= \Lambda \mathcal{L}_0(\vphi)\Lambda^{-1}
:=\\
&\Pi_{S}^{\perp}\Big(\omega\cdot\pa_{\vphi}+
\opw\Big(\ii A_1(\vphi,x)\x+\ii A_{1/2}(\vphi,x)|\x|^{\frac{1}{2}}+A_0(\vphi,x)+
A_{-1}(\vphi,x,\x)\Big)
%\\&\qquad\qquad
%+\opw\Big(
%A_{-1}(\vphi,x,\x)\Big)
+
\mathcal{R}^{(1)}\Big)\,,
\end{aligned}
\end{equation}
where
$\mathcal{R}^{(1)}=\mathcal{R}^{(1)}(\vphi)$ is in 
$\gotL_{\rho,p}^{1}\otimes\mathcal{M}_{2}(\mathbb{C})$ and
\begin{align}\label{matriciinizio}
& A_{1}(\vphi,x) := 
\left(\begin{matrix} V(\vphi,x) & 0 \\
0 & V(\vphi,x) \end{matrix}\right) 
\\
& A_{1/2}(\vphi,x)  := 
\left(\begin{matrix}1+\tilde{a}(\vphi,x) & \tilde{a}(\vphi,x) \\
 -\tilde{a}(\vphi,x) & -(1+\tilde{a}(\vphi,x)) \end{matrix}\right) \, , 
 \qquad \tilde{a}:= 
 \frac12 (  \omega\cdot\pa_{\vphi} B + V B_x ) \, ,  \label{function:a}
 \\
&  A_{0}(\vphi,x) :=  \frac{1}{4}\left(\begin{matrix}0 & 1\\1&0\end{matrix}\right)V_x (\vphi, x) \, .
 \label{function:c}
\end{align} 
%Moreover one has
%\begin{equation}\label{Napalmi}
%\begin{aligned}
%\lvert A_{-1} \rvert^{\gamma, \calO_0}_{-1, s, \alpha} & 
%\lesssim_{s, \alpha, \rho} 
%\varepsilon (1+\lVert \mathfrak{I}_{\delta}\rVert^{\gamma, \calO_0}_{s+\s})\,, 
%\quad  s\geq s_0\,,\\\
%\lvert \Delta_{12} A_{-1}   \rvert_{-1, p, \alpha}  &\lesssim_{p, \alpha, \rho} \varepsilon (1+\lVert \mathfrak{I}_{\delta} \rVert_{p+\sigma})\|i_1-i_2\|_{p+\s},
%\end{aligned}
%\end{equation}
%and
%\begin{equation}\label{Napalmii}
%\begin{aligned}
%\mathbb{M}^{\gamma}_{{\mathcal{R}^{(1)}}}(s, \tb) & \lesssim_{s, \rho} \varepsilon 
%(1+\lVert \mathfrak{I}_{\delta}\rVert^{\gamma, \calO_0}_{s+\sigma})\,,  
%\quad s_0\le s\le \mathcal{S}\, ,
%\qquad 0\le \tb\le \rho-2\,,\\
%\mathbb{M}_{\Delta_{12} \mathcal{R}^{(1)}  }(p, \tb) & \lesssim_{p, \rho} 
%\varepsilon (1+\lVert \mathfrak{I}_{\delta} \rVert_{p+\sigma})
%\|i_1-i_2\|_{p+\s}\,,\qquad 0\le \tb\le \rho-3\,.
%\end{aligned}
%\end{equation}
\end{proposition}

\begin{proof}
We start by applying the change of variables $\mathfrak{F}$ in 
\eqref{matriciCompVar}.
Using Lemma \ref{Jameshomo}
and the \eqref{primamappa} we get
\[
\mathfrak{F}\mathcal{L}_0\mathfrak{F}^{-1}=\Pi_{S}^{\perp}\Big(
\omega\cdot\pa_{\vphi}+
\opw{\left(  
\sm{|  \xi |^{-1/4}}{0}{0}{|  \xi  |^{1/4}}  \#^{W}_{\rho} 
\sm{ \ii V  \xi + \frac{V_x}{2} }{-| \xi |}{ 1 +  a  }{ \ii V \xi - \frac{V_x}{2}}  
\#^{W}_\rho 
\sm{| \xi  |^{1/4}}{0}{0}{| \xi |^{-1/4}} \right)}    
\vect{\tilde \eta}{ \tilde \omega} \Big)
\]
up to a remainder in $\gotL^{1}_{\rho,p}$, i.e. a remainder which 
admits an expansion as \eqref{espandoenonmifermo2}, \eqref{simboOMO2DEF2} 
with estimates \eqref{AVBDEF2}.
By expanding the symbols on the diagonal, using formula \eqref{espansione2},
we get
\begin{equation*}%\label{rf2}
\begin{aligned}
& | \xi |^{-1/4 }  \#^{W}_\rho( \ii V \xi +\frac{V_x}{2} )  \#^{W}_\rho |\xi  |^{1/4} 
 =     \ii V \xi +\frac{V_x}{4} \, , 
 %\\&  
\qquad | \xi |^{1/4 }  \#^{W}_\rho ( \ii V \xi - \frac{V_x}{2} )  \#^{W}_\rho |\xi  |^{- 1/4} 
 =   \ii V \xi - \frac{V_x}{4}   \, ,
\end{aligned}
\end{equation*}
up to symbols in $S_1^{-1}$. 
Similarly we get
\begin{align*}%\label{rf3}
| \xi  |^{-1/4} \#^{W}_\rho ( -| \xi |) \#^{W}_\rho |\xi  |^{-1/4} =- | \xi |^{1/2} \,,
 %\\
 \qquad 
 | \xi  |^{1/4} \#^{W}_\rho (1 +  a_0 ) \#^{W}_\rho | \xi  |^{1/4} 
%&
=   (1 +  a) | \xi  |^{1/2}  \,, %\label{rf4}
\end{align*}
up to symbols in $S_1^{-3/2}$.
All this new symbols of order less or equal $-1$ 
and belongs to $S_{1}^{-1}$.
%satisfy the bounds \eqref{Napalmi}
%by using \eqref{crawford}, \eqref{crawford2} and the bounds 
%\eqref{AVB}.
%%, \eqref{AVB2}.
By applying the change of coordinates $\mathcal{C}$ in \eqref{matriciCompVar}
we get that
$\mathcal{C}\mathfrak{F}\mathcal{L}_0\mathfrak{F}^{-1}\mathcal{C}^{-1}=\Lambda\mathcal{L}_0\Lambda^{-1}$
has the form \eqref{elle1} with matrices of symbols in 
\eqref{matriciinizio}-\eqref{function:c}.
\end{proof}

\begin{lemma}\label{LambdaAdmiss}
The operator $\mathcal{L}_1$ in \eqref{elle1} belongs to
the class $\mathfrak{S}_1$ given in  Definition \ref{compaMulti}.
\end{lemma}

\begin{proof}
First notice that, by Proposition
 \ref{lemma:7.4},
the operator $\mathcal{L}_0$ 
in \eqref{primamappa} is
is real-to-real,  Hamiltonian and $x$-translation invariant.
Moreover 
the map
$\Lambda$ in \eqref{CVWW} is
\emph{symplectic}, $x$-\emph{translation invariant}. 
Hence also $\mathcal{L}_1$ satisfies the same properties of 
$\mathcal{L}_0$.
To prove that $\mathcal{L}_1\in \mathfrak{S}_1$ we need to prove that 
the pseudo differential operator in \eqref{elle1} is itself Hamiltonian
and $x$-translation invariant (it is clearly real-to-real).
By Lemmata \ref{lem:momentoSimbolo}, 
\ref{realHamMtrixSimbo} 
we just have to  show that the matrix of symbols
\begin{equation*} %\label{giungla}
A(\vphi,x,\x):=\ii A_1(\vphi,x)\x+\ii A_{1/2}(\vphi,x)|\x|^{\frac{1}{2}}
+A_0(\vphi,x)+A_{-1}(\vphi,x,\x)
\end{equation*}
satisfies \eqref{selfSimbo2}  and \eqref{space3Xinv}.
The condition \eqref{space3Xinv} follows trivially by the fact that
the functions $V,B$ belongs to $S_{\mathtt{v}}$ and that 
the expansion \eqref{espansione2} preserves this property.
The  \eqref{selfSimbo2} follows by the explicit computations performed in
Proposition \ref{lem:compVar} since the functions $V,B$ in
 \eqref{def:V}, \eqref{form-of-B}
and $\widetilde{a}$ in \eqref{function:a} are real valued.
Therefore $\mathcal{L}_1\in\mathfrak{S}_1$. 
\end{proof}

%%%%%%%%%%%%%%%%%%%%%%%%%%%%%%%%%%%%%
%%%%%%%%%%%%%%%%%%%%%%%%%%%%%%%%%%%%%
%%%%%%%%%%%%%%%%%%%%%%%%%%%%%%%%%%%%%
%%%%%%%%%%%%%%%%%%%%%%%%%%%%%%%%%%%%%

\section{Block-diagonalization}\label{sec:block}
In this section we block-diagonalize the operator
$\mathcal{L}_1$ in \eqref{elle1} up to a $\rho$-smoothing remainders in 
$\mathfrak{L}_{\rho,p}$.

\subsection{Block-diagonalization at order 1/2}
In this subsection we diagonalize the matrix %of symbols
$A_{1/2}(\vphi,x)  |\xi|^{1/2} $
in \eqref{function:a}.
%up to a matrix of symbols of order $ 0 $.
We consider the multiplication operator $\mathcal{M}$ defined as
\begin{equation}\label{TRA}
\begin{aligned}
&\mathcal{M}:=\mathcal{M}(\vphi)
:=\left(\begin{matrix}
f & g \\
g & f
\end{matrix}\right) \, ,\\
&f:=f(\vphi, x) :=\frac{1+\tilde{a}+\lambda_{+}}{\sqrt{(1+\tilde{a}
+\lambda_{+})^2-\tilde{a}^2}}\,, \qquad 
g:=g(\vphi, x) :=\frac{-\tilde{a}}{\sqrt{(1+\tilde{a}+\lambda_{+})^2-\tilde{a}^2}} \, , \\
&\lambda_{\pm}:=\lambda_{\pm}(\vphi,x):=\pm \sqrt{(1+\tilde{a})^{2}-\tilde{a}^{2}} 
\end{aligned}
\end{equation}
where the functions $\lambda_{\pm}$
%\begin{equation}\label{eigenvalues}
%\lambda_{\pm}=\lambda_{\pm}(\vphi,x)
%:=\pm \sqrt{(1+\tilde{a})^{2}-\tilde{a}^{2}} 
%\end{equation} 
are the eigenvalues of $ A_{1/2} $ in \eqref{function:a}. 
Notice that $f\geq0$. 
Notice also that
\begin{equation}\label{invC}
{\rm det}(\mathcal{M})=f^2-g^2=1  \, , \quad 
\mathcal{M}^{-1} =
 \left(\begin{matrix}
f & - g \\
- g & f
\end{matrix}\right)  \, , 
\end{equation}
%and 
\begin{equation}\label{nuovocoef}
A_{1/2}^{(2)}:=\mathcal{M}^{-1}A_{1/2}\;\mathcal{M}=  
\left(\begin{matrix}
\lambda_+ & 0 \\
0 & -\lambda_+
\end{matrix}\right)
= 
\left(\begin{matrix}
1+a^{(0)} & 0 \\
0 & -(1+a^{(0)})
\end{matrix}\right) \, , \quad 
a^{(0)}:=\lambda_{+}-1\,. %\in S_1^{0} \, . 
\end{equation}
The main result of this section is the following.

\begin{proposition}{\bf (Block-diagonalization at order 1/2).}\label{lem:mappa2}
There exists an invertible 
map ${\bf \Phi}_{\mathbb{M}}(\vphi)\colon H^s_{S^{\perp}}(\mathbb{T})\times H^s_{S^{\perp}}(\mathbb{T})\to H^s_{S^{\perp}}(\mathbb{T})\times H^s_{S^{\perp}}(\mathbb{T})$ such that (recall \eqref{elle1})
\begin{equation}\label{elle2}
\begin{aligned}
\mathcal{L}_2&:={\bf \Phi}_{\mathbb{M}}\mathcal{L}_1{\bf \Phi}^{-1}_{\mathbb{M}}
:=\Pi_{S}^{\perp}\omega\cdot\pa_{\vphi}\\
&+\Pi_{S}^{\perp}\Big(
\opw\Big(\ii A_1(\vphi,x)\x+\ii A_{1/2}^{(2)}(\vphi,x)|\x|^{\frac{1}{2}}
+A_0^{(2)}(\vphi,x)+A_{-1/2}^{(2)}(\vphi,x,\x)\Big)+\mathcal{R}^{(2)}\Big)
\end{aligned}
\end{equation}
where $\mathcal{R}^{(2)}=\mathcal{R}^{(2)}(\vphi)$ is 
 in    $\mathfrak{L}^{1}_{\rho,p}\otimes\mathcal{M}_2(\mathbb{C})$,
$A_{-1/2}^{(2)}$ is in 
 $S_1^{-\frac{1}{2}}\otimes\mathcal{M}_2(\mathbb{C})$
and satisfies \eqref{space3Xinv},
$A_1$ is the matrix in \eqref{matriciinizio}, $A_{1/2}^{(2)}$ 
is in \eqref{nuovocoef} and
\begin{align}
&  A_{0}^{(2)}(\vphi,x) 
:=  
\left(\begin{matrix}0 &  b^{(2)}\\ 
b^{(2)}&0 \end{matrix}\right) \, ,\qquad 
b^{(2)}:=\frac{V_x}{4}-V(f_x g-g_x f)\,, \label{function:c2}
\end{align} 
where the functions %$a^{(0)},
$f,g$
 are given in %\eqref{nuovocoef}, 
 \eqref{TRA}.
% In particular,
% there is $\s=\s(\tau,\nu)>0$
%% (possibly larger than the one of Propisition \ref{lemma:7.4})
% %Lemma \ref{lem:mappa1} ) 
% such that
% \begin{equation}\label{stimaAmeno12}
% \begin{aligned}
% \|a^{(0)}\|_{s}^{\gamma,\calO_0}+
%\|A_0^{(2)}\|_{s}^{\gamma,\calO_0}+|A_{-1/2}^{(2)}|_{-\frac{1}{2},s,\alpha}^{\gamma,\calO_0}
%&\lesssim_{s,\alpha}
%\e(1+\|\mathfrak{I}_{\delta}\|_{s+\s}^{\gamma,\calO_0})\,,
%\\
%\|\Delta_{12}a^{(0)}\|_{p}+
%\|\Delta_{12}A_0^{(2)}\|_{p}
%+|\Delta_{12}A_{-1/2}^{(2)}|_{-\frac{1}{2},p,\alpha}&\lesssim_{p,\alpha}
%\e(1+\|\mathfrak{I}_{\delta}\|_{p+\s}^{\gamma,\calO_0})\|i_{1}-i_{2}\|_{p+\s}\,,
%\end{aligned}
%\end{equation}
%and,  
%for $s_0\leq s\leq \mathcal{S}$,
%\begin{equation}\label{resto1mappa2}
%\begin{aligned}
%\mathbb{M}_{\mathcal{R}^{(2)}}^{\gamma}(s,\mathtt{b})&\lesssim_{s,\rho}
%\e(1+\|\mathfrak{I}_{\delta}\|_{s+\s}^{\gamma,\calO_0})
%\,,\qquad
%0\leq \mathtt{b}\leq \rho-2\,,\\
%\mathbb{M}_{\Delta_{12}\mathcal{R}^{(2)}}(p,\mathtt{b})&\lesssim_{p,\rho}
%\e(1+\|\mathfrak{I}_{\delta}\|_{p+\s}^{\gamma,\calO_0})\|i_{1}-i_{2}\|_{p+\s}
%\,,\qquad
%0\leq \mathtt{b}\leq \rho-3\,.
%\end{aligned}
%\end{equation}
Moreover
\begin{equation}\label{AVB3mappa2PRO}
\begin{aligned}
&\mathfrak{M}^{\gamma}_{({\bf \Phi}_{\mathbb{M}}^{\pm}-{\rm Id})}(s)+
\mathfrak{M}^{\gamma}_{({\bf \Phi}_{\mathbb{M}}^{\pm}-{\rm Id})^{*}}(s)\lesssim_{s}
\e(1+\|\mathfrak{I}_{\delta}\|_{s+\mu}^{\gamma,\calO_0})\,,\\
&\|\Delta_{12}({\bf \Phi}_{\mathbb{M}}^{\pm}-{\rm Id})h\|_{p}+
\|\Delta_{12}({\bf \Phi}_{\mathbb{M}}^{\pm}-{\rm Id})^{*}h\|_{p}\lesssim_{p}
\e\|h\|_{p}\|i_1-i_{2}\|_{p+\mu}\,.
\end{aligned}
\end{equation}
\end{proposition}
\noindent
We divide the proof of the proposition above into several steps.
We have the following.

\begin{lemma}\label{mappa2}
The functions $a^{(0)}$ in \eqref{nuovocoef}, $f,g$ in \eqref{TRA}
belong to the class $S_1^{0}$ of Def. \ref{espOmogenea}.
%There is $\s=\s(\tau,\nu)$ such that
%\begin{equation}\label{AVBmappa2}
%\begin{aligned}
%\|a^{(0)}\|_{s}^{\gamma,\calO_0}+\|f-1\|_{s}^{\gamma,\calO_0}+
%\|g\|_{s}^{\gamma,\calO_0}&\lesssim_s 
%\e(1+\|\mathfrak{I}_{\delta}\|_{s+\s}^{\gamma,\calO_0})\,,\\
%\|\Delta_{12}a^{(0)}\|_{p}^{\gamma,\calO_0}+\|\Delta_{12}f\|_{p}^{\gamma,\calO_0}+
%\|\Delta_{12}g\|_{p}^{\gamma,\calO_0}&\lesssim_p
%\e(1+\|\mathfrak{I}_{\delta}\|_{p+\s}^{\gamma,\calO_0})\|i_{1}-i_{2}\|_{p+\s}\,,
%%\label{AVB2mappa2}
%\end{aligned}
%\end{equation}
%with $p$ as in \eqref{condireP}. 
In particular the functions $\tilde{a},f,g$
are $x$-translation invariant, i.e. satisfies \eqref{space3Xinv}.
Moreover (recall Definition \ref{LipTameConstants})
the maps $\mathcal{M}^{\pm}-{\rm Id}$ and $(\mathcal{M}^{\pm}-{\rm Id})^{*}$
(see \eqref{TRA})
satisfy
\begin{equation}\label{AVB3mappa2}
\begin{aligned}
&\mathfrak{M}^{\gamma}_{(\mathcal{M}^{\pm}-{\rm Id})}(s)+
\mathfrak{M}^{\gamma}_{(\mathcal{M}^{\pm}-{\rm Id})^{*}}(s)\lesssim_{s}
\e(1+\|\mathfrak{I}_{\delta}\|_{s+\mu}^{\gamma,\calO_0})\,,\\
&\|\Delta_{12}(\mathcal{M}^{\pm}-{\rm Id})h\|_{p}+
\|\Delta_{12}(\mathcal{M}^{\pm}-{\rm Id})^{*}h\|_{p}\lesssim_{p}
\e\|h\|_{p}\|i_1-i_{2}\|_{p+\mu}\,.
\end{aligned}
\end{equation}
\end{lemma}
\begin{proof}
One has $a^{(0)}, f,g\in S_1^{0}$ since they are regular functions of 
$\tilde{a}$ (see \eqref{nuovocoef}, \eqref{TRA})
and the fact that $\tilde{a}$ in \eqref{function:a} belongs to $S_1^{0}$ 
by Lemma \ref{mappa1}.
The  \eqref{AVB3mappa2} 
follows by Lemma 2.13 in \cite{BM1}.
\end{proof}

In order to conjugate the operator 
$\mathcal{L}_1$ in \eqref{elle1} 
we shall consider the map
$\mathcal{M}^{\perp}:=\Pi_{S}^{\perp}\mathcal{M}\Pi_{S}^{\perp}$.
Unfortunately the map $\mathcal{M}^{\perp}$ is not symplectic 
with respect to the symplectic form \eqref{formasimpletticaestesa}.
We first construct a symplectic correction. % for $\mathcal{M}^{\perp}$.
\begin{lemma}\label{lem:defC}
Consider the map $\mathcal{M}$ in \eqref{TRA}. 
There is a function $c=c(\vphi,x)\in S_1^{0}$
such that
$\mathcal{M}^{-1}:=\Phi_{C}^{1}$ where $\Phi_{C}^{\tau}$ is the flow
of 
\begin{equation}\label{flussoC}
\pa_{\tau}\Phi_{C}^{\tau}=C \Phi_{C}^{\tau}\,,\quad \Phi_{C}^{0}={\rm Id}\,,
\qquad C:=C(\vphi,x):=\left( \begin{matrix}0 & c(\vphi,x) \\
\ov{c(\vphi,x)} & 0
\end{matrix}
\right)\,.
\end{equation}
%Moreover one has
%\begin{equation}\label{definizioneC3}
%\begin{aligned}
%\|c\|_{s}^{\gamma,\calO_0}&\lesssim_{s} \e(1+\|\mathfrak{I}_{\delta}\|_{s+\s}^{\gamma,\calO_0})\,,\\
%\|\Delta_{12}c\|_{p}&\lesssim_{p}
%\e(1+\|\mathfrak{I}_{\delta}\|_{p+\s}^{\gamma,\calO_0})\|i_{1}-i_{2}\|_{p+\s}\,.
%\end{aligned}
%\end{equation}
\end{lemma}

\begin{proof}
The proof is based on ideas used in Lemma $5.2$ in \cite{FIloc} (see also \cite{FI2}).
One has that
\begin{equation*} %\label{iperbolici2}
\Phi_{C}^{1}:=S(\vphi):=\exp\{C(\vphi,x)\}:=\left(
\begin{matrix}
c_{1}(\vphi,x) & c_{2}(\vphi,x)\vspace{0.2em}\\
\ov{c_{2}(\vphi,x)} & c_{1}(\vphi,x)
\end{matrix}
\right)\,,
\qquad
\begin{aligned}
&c_1 =c_1(\vphi,x):=\cosh(|c(\vphi,x)|)\,,\\
&c_2=c_{2}(\vphi,x):=\frac{c(\vphi,x)}{|c(\vphi,x)|}\sinh(|c(\vphi,x)|)\,.
\end{aligned}
\end{equation*}
%where
%\begin{equation}\label{iperbolici2}
%c_1=c_1(\vphi,x):=\cosh(|c(\vphi,x)|) , \qquad c_{2}=c_{2}(\vphi,x):=\frac{c(\vphi,x)}{|c(\vphi,x)|}\sinh(|c(\vphi,x)|).
%\end{equation}
Note that the function $c_2(\vphi,x)$ above is not singular indeed
\begin{equation*}
\begin{aligned}
c_{2}&=\frac{c}{|c|}\sinh(|c|)=\frac{c}{|c|}
\sum_{k=0}^{\infty}\frac{(|c|)^{2k+1}}{(2k+1)!}=
c\sum_{k=0}^{\infty}\frac{\big(c\ov{c}\big)^k}{(2k+1)!}\,.
\end{aligned}
\end{equation*} 
We note moreover that for any $x\in \mathbb{T}$ one has $\det(S(\vphi,x))=1$, 
hence its inverse $S^{-1}(\vphi,x)$ is well defined. 
We choose $c(\vphi,x)$ in such a way that 
$S^{-1}(\vphi,x):=\mathcal{M}$ (see \eqref{TRA}). 
Therefore we have to solve the following equations 
\begin{equation*} %\label{losappaimorisolvere???}
\cosh(|c(\vphi,x)|)=f(\vphi,x)\,, 
\quad 
\frac{c(\vphi,x)}{|c(\vphi,x)|}\sinh(|c(\vphi,x)|)=-g(\vphi,x)\,.
\end{equation*}
Concerning the first one we note that
$
f^{2}-1=\frac{|\tilde{a}|^{2}}{2\lambda_{+}(1+\tilde{a}+\lambda_{+})}\geq0\,.
 $
Therefore 
 \begin{equation*} %\label{definizioneC2}
|c(\vphi,x)|:= {\rm arccosh}(f(\vphi,x))
=\ln\Big(f(\vphi,x)+\sqrt{(f(\vphi,x))^{2}-1}\Big)\,,
\end{equation*}
is well-defined. 
For the second equation one observes that
the function 
\[
\frac{\sinh(|c(\vphi,x)|)}{|c(\vphi,x)|}=1+\sum_{k\geq0}
\frac{(c(\vphi,x)\bar{c}(\vphi,x))^{k}}{(2k+1)!}\geq 1\,,
\]
hence we set
\begin{equation}\label{definizioneC}
c(\vphi,x):=g(\vphi,x)\frac{|c(\vphi,x)|}{\sinh(|c(\vphi,x)|)}\,.
\end{equation}
Using that $f,g\in S_1^{0}$ (see Lemma \ref{mappa2}), and using the 
\eqref{definizioneC} we deduce that $c\in S_1^{0}$.
\end{proof}
%We are now in position to conclude the proof of Proposition \ref{lem:mappa2}.
\begin{proof}[{\bf Proof of Proposition \ref{lem:mappa2}}]
Let us define ${\bf \Phi}_{\mathbb{M}}={\bf \Phi}_{\mathbb{M}}^{1}$ where 
${\bf \Phi}_{\mathbb{M}}^{\tau}$ is the flow of
\begin{equation}\label{flussoCpro}
\pa_{\tau}{\bf \Phi}_{\mathbb{M}}^{\tau}=\Pi_{S}^{\perp}C\Pi_{S}^{\perp}
{\bf \Phi}_{\mathbb{M}}^{\tau}\,,\quad {\bf \Phi}_{\mathbb{M}}^{0}={\rm Id}\,,
\end{equation}
where $C$ is the matrix in \eqref{flussoC}
with symbol $c(\vphi,x)$ given by Lemma \ref{lem:defC}.
By Lemma \ref{differenzaFlussi} (see also Remark \ref{differenzaFlussiRMK})
we have that
\[
{\bf \Phi}_{\mathbb{M}}^{1}= 
\Pi_S^\perp\Phi^{1}_{C}\Pi_S^\perp\circ(\rm{Id} +\mathcal{R})
\]
where $\Phi^{1}_{C}=\mathcal{M}^{-1}$ is the flow of \eqref{flussoC} (see also \eqref{TRA})
and $\mathcal{R}$ is a matrix of finite rank operators of the form \eqref{FiniteDimFormDP10001}
satisfying 
estimates \eqref{giallo2}. Hence ${\bf \Phi}_{\mathbb{M}}$
is well-defined and satisfies bounds \eqref{AVB3mappa2PRO}  
thanks to \eqref{AVB3mappa2}. %, \eqref{AVB4mappa2}.
By Lemma \ref{georgiaLem}
we have that the conjugate 
$ \mathcal{L}_2:={\bf \Phi}_{\mathbb{M}}\mathcal{L}_1{\bf \Phi}^{-1}_{\mathbb{M}}$
is given by 
$\Pi_{S}^{\perp}\mathcal{M}^{-1}\mathcal{L}_1\mathcal{M}\Pi_{S}^{\perp}$
up to finite rank remainders 
belonging to $\mathfrak{L}_{\rho,p}^{1}\otimes\mathcal{M}_2(\mathbb{C})$.
%This is true by using  estimates \eqref{giallo2Rank},
%Proposition \ref{lem:compVar} to estimates 
%the symbols of the operator $\mathcal{L}_1$ in 
%\eqref{elle1} and Remark \ref{yumayuma}.
We have that 
\[
\mathcal{M}^{-1}\mathcal{L}_1\mathcal{M}=
\mathcal{M}^{-1}\omega\cdot\pa_{\vphi}\mathcal{M} 
+
\mathcal{M}^{-1}\opw\big( 
\ii A_1 \x+\ii A_{1/2} |\x|^{\frac{1}{2}}+A_0+A_{-1}
\big)\mathcal{M}+\mathcal{M}^{-1}\mathcal{R}^{(1)}\mathcal{M}\,.
\]
Hence Lemma \ref{Jameshomo} %and Remark \ref{yumayuma} 
implies that
\begin{equation*} %\label{eq:41555}
\mathcal{M}^{-1}\mathcal{L}_1\mathcal{M}=
\opw{ \Big( \mathcal{M}^{-1}\#^{W}_\rho \omega\cdot\pa_{\vphi}\mathcal{M} + 
\mathcal{M}^{-1} \#^{W}_\rho \big( 
\ii A_1 \x+\ii A_{1/2} |\x|^{\frac{1}{2}}+A_0+A_{-1}
\big) \#^{W}_\rho \mathcal{M}\Big) } 
\end{equation*}
up to terms in  $\mathfrak{L}^{1}_{\rho,p}\otimes\mathcal{M}_2(\mathbb{C})$.
% which satisfy
%\eqref{resto1mappa2} by \eqref{jamal}, \eqref{jamal2},
%\eqref{Napalmi}, \eqref{Napalmii} and \eqref{AVB}. % and \eqref{AVB2}.
In the following we apply 
several times Lemma \ref{Jameshomo} %and Remark \ref{yumayuma} 
on the composition
of pseudo differential operators.
By \eqref{TRA}, \eqref{invC} 
we deduce that $  (\omega\cdot\pa_{\vphi}f)f-(\omega\cdot\pa_{\vphi}g) g = 0 $. Hence
we have 
 \begin{equation*}%\label{TRA9}
 \begin{aligned}
\mathcal{M}^{-1} \#^{W}_\rho \omega\cdot\pa_{\vphi}\mathcal{M} 
%&= \left(
%\begin{matrix}
%(\omega\cdot\pa_{\vphi}f)f-(\omega\cdot\pa_{\vphi}g)g 
%& -(\omega\cdot\pa_{\vphi}f)g+(\omega\cdot\pa_{\vphi} g)f\\ 
%-(\omega\cdot\pa_{\vphi}f)g+(\omega\cdot\pa_{\vphi} g)f
% & (\omega\cdot\pa_{\vphi}f)f-(\omega\cdot\pa_{\vphi}g)g 
%\end{matrix}
%\right) \\
&= \left(
\begin{matrix}
0 &   -(\omega\cdot\pa_{\vphi}f)g+(\omega\cdot\pa_{\vphi} g)f
\\  -(\omega\cdot\pa_{\vphi}f)g+(\omega\cdot\pa_{\vphi} g)f
 &  0 
\end{matrix}
\right) \,.
\end{aligned}
\end{equation*}
%because  differentiating $ f^2-g^2=1 $ we get 
%$  (\omega\cdot\pa_{\vphi}f)f-(\omega\cdot\pa_{\vphi}g) g = 0 $.
By \eqref{matriciinizio},  using symbolic calculus 
and $ f^2-g^2=1 $ (see \eqref{invC}),  we obtain the exact expansion  
\begin{equation*}%\label{TRA10}
 \mathcal{M}^{-1} \#^{W}_{\rho}( \ii A_{1} \x)\#^{W}_{\rho} \mathcal{M}
= \left(\begin{matrix}
  \ii V\x  & -V(f_x g-g_x f)  \\
-V(f_x g-g_x f) &  \ii V\x  
\end{matrix}\right)\,.
\end{equation*}
By \eqref{nuovocoef}  we have
\begin{equation*}%\label{TRA6}
\mathcal{M}^{-1} \#^{W}_{\rho} (\ii A_{1/2}|\x|^{\frac{1}{2}}) \#^{W}_{\rho}  
\mathcal{M}
=\ii \left( \begin{matrix} 
(1+a^{(0)})|\x|^{\frac{1}{2}}  & 0 \\
0 & -(1+a^{(0)})|\x|^{\frac{1}{2}}  
\end{matrix}\right)+r_{-1/2}\,,
\end{equation*}
for some  $r_{-1/2}\in S_1^{-\frac{1}{2}}\otimes\mathcal{M}_2(\mathbb{C})$.
Moreover, recalling \eqref{function:c} and that 
$ A_{-1} \in S_1^{-1}\otimes\mathcal{M}_{2}(\mathbb{C})$, 
we have 
\begin{equation*}%\label{TRA7}
\begin{aligned}
& \mathcal{M}^{-1} \#^{W}_{\rho}   A_{0} 
\#^{W}_{\rho} \mathcal{M}= A_0 = \frac{1}{4}\sm{0}{1}{1}{0}V_{x}\,,
\qquad 
\mathcal{M}^{-1} \#_\rho A_{-1}  \#^{W}_\rho \mathcal{M}=:r_{-1} \,,
\end{aligned}
\end{equation*}
with $r_{-1}\in 
S^{-1}_1\otimes\mathcal{M}_{2}(\mathbb{C}) $.
By the discussion above we obtain the 
\eqref{elle2} and \eqref{function:c2}.
\end{proof}

We conclude this section with the following result.
\begin{lemma}\label{GoodAmmissi2}
The operator $\mathcal{L}_2$ in \eqref{elle2} is in $\mathfrak{S}_1$ 
and the map ${\bf \Phi}_{\mathbb{M}}$ in Proposition \ref{lem:mappa2} 
is in $\mathfrak{T}_1$ (see Def. \ref{compaMulti}, \ref{goodMulti}).
\end{lemma}

\begin{proof}
We first note the following. The functions $f$, $g$ in \eqref{TRA}
are real valued and satisfy \eqref{space3Xinv}
since, by Lemma \ref{LambdaAdmiss}, the operator 
$\mathcal{L}_1$ is in $\mathfrak{S}_1$.
 Therefore the
 function $c(\vphi,x)$ given by Lemma \ref{lem:defC}
 is in $S_{\mathtt{v}}$. As a consequence the flow of \eqref{flussoCpro}
 is in 
$\mathfrak{T}_1$, i.e. is symplectic and $x$-translation invariant. 
Then also the operator $\mathcal{L}_2$ is 
Hamiltonian and $x$-translation invariant.
The matrices $\ii A_1$, $\ii A_{1/2}$,  $A_{j}^{(2)}$, $j=0,-1/2$, in \eqref{elle2}
satisfy \eqref{selfSimbo2}, \eqref{space3Xinv} by construction. 
%Hence $\mathcal{L}_{2}\in \mathfrak{S}_1$.
\end{proof}

\subsection{Block-diagonalization at negative orders}
In this subsection we iteratively 
 block-diagonalize the operator 
  \eqref{elle2} 
  (which is already block-diagonal at the orders $ 1$ and $ 1/ 2 $) 
up to smoothing remainders. The main result of the section is the following.

\begin{proposition}{\bf (Block-diagonalization at lower orders).}\label{blockTotale}
There exists an invertible 
map ${\bf \Psi}(\vphi)\colon H^s_{S^{\perp}}(\mathbb{T})
\times H^s_{S^{\perp}}(\mathbb{T})
\to H^s_{S^{\perp}}(\mathbb{T})\times H^s_{S^{\perp}}(\mathbb{T})$ 
such that (recall \eqref{elle2})
\begin{equation}\label{elle3}
\mathcal{L}_{3}:={\bf \Psi}
\mathcal{L}_2{\bf \Psi}^{-1}
=\Pi_{S}^{\perp}\Big(\omega\cdot\pa_{\vphi}+
\opw\left(\begin{matrix} d(\vphi,x,\x) & 0 \\
0 & \ov{d(\vphi,x,-\x)}  
\end{matrix}\right)
+\mathcal{R}^{(3)}\Big)\,,
\end{equation}
where $\mathcal{R}^{(3)}=\mathcal{R}^{(3)}(\vphi)$ is a smoothing remainder in 
$\mathfrak{L}_{\rho,p}^{1}\otimes\mathcal{M}_2(\mathbb{C})$ and 
\begin{equation}\label{diagpezzo}
d(\vphi,x,\x):=
\ii V(\vphi,x)\x+\ii (1+a^{(0)}(\vphi,x))|\x|^{\frac{1}{2}}+c^{(3)}(\vphi,x,\x)\,,
\end{equation}
for some $c^{(3)}\in S_1^{-\frac{1}{2}}$.
%Moreover we have
%\begin{equation}\label{motociclo2totale}
%  \begin{aligned}
%|c^{(3)}|_{-\frac{1}{2},s,\alpha}^{\gamma,\calO_0}
%&\lesssim_{s,\alpha}
%\e(1+\|\mathfrak{I}_{\delta}\|_{s+\s}^{\gamma,\calO_0})\,, 
%\qquad
%|\Delta_{12}c^{(3)}|_{-\frac{1}{2},p,\alpha}
%\lesssim_{p,\alpha}
%\e(1+\|\mathfrak{I}_{\delta}\|_{p+\s}^{\gamma,\calO_0})
%\|i_{1}-i_{2}\|_{p+\s}\,,
%\end{aligned}
%\end{equation}
%and,  
%for $s_0\leq s\leq \mathcal{S}$,
%\begin{equation}\label{motociclo4totale}
%\begin{aligned}
%\mathbb{M}_{\mathcal{R}^{(3)}}^{\gamma}(s,\mathtt{b})&\lesssim_{s,\rho}
%\e(1+\|\mathfrak{I}_{\delta}\|_{s+\s}^{\gamma,\calO_0})
%\,,\qquad
%0\leq \mathtt{b}\leq \rho-2\,,\\
%\mathbb{M}_{\Delta_{12}\mathcal{R}^{(3)}}(p,\mathtt{b})&\lesssim_{p,\rho}
%\e(1+\|\mathfrak{I}_{\delta}\|_{p+\s}^{\gamma,\calO_0})
%\|i_{1}-i_{2}\|_{p+\s}\,,\qquad
%0\leq \mathtt{b}\leq \rho-3\,,
%\end{aligned}
%\end{equation}
%for some $\s=\s(\tau,\nu)>0$.
Finally, the map ${\bf \Psi}$ satisfies 
 \begin{equation}\label{stimaMappaPsi}
\|{\bf \Psi}^{\pm1}u\|_{s}^{\gamma,\calO_0}
\lesssim_{s}\|u\|_{s}^{\gamma, \calO_0}+
\e\|\mathfrak{I}_{\delta}\|^{\gamma, \calO_0}_{s+\mu}
\|u\|_{s_0}^{\gamma, \calO_0}\,.
\end{equation}
%for some $\s>0$ depending on $\nu,\tau$.
\end{proposition}
The rest of the section is devoted to the proof of the proposition above.
We shall prove inductively that
for 
 $ j = 0, \ldots, 2 \rho $, there are 
 
 \vspace{0.5em}
 \noindent
 $\bullet$ symbols $c_{j}\in S_1^{-1/2}$, $b_{j}\in S_1^{-j/2}$;
 
  \vspace{0.5em}
 \noindent
$\bullet$ smoothing operators 
$\mathcal{Q}^{(j)}\in \gotL_{\rho,p}^{1}\otimes\mathcal{M}_2(\mathbb{C})$,
% symbols $c_{j}\in S^{-1/2}$, $b_{j}\in S^{-j/2}$
% and smoothing operators 
% $\mathcal{Q}^{(j)}\in \gotL_{\rho,p}\times\mathcal{M}_2(\mathbb{C})$
%with estimates
%\begin{equation}\label{motociclo}
%\begin{aligned}
%|c_{j}|_{-\frac{1}{2},s,\alpha}^{\gamma,\calO_0}
%+|b_{j}|_{-\frac{j}{2},s,\alpha}^{\gamma,\calO_0}
%&\lesssim_{s,j,\alpha}
%\e(1+\|\mathfrak{I}_{\delta}\|_{s+\s_{j}}^{\gamma,\calO_0})\,,\\
%|\Delta_{12}c_{j}|_{-\frac{1}{2},p,\alpha}
%+|\Delta_{12}b_{j}|_{-\frac{j}{2},p,\alpha}&\lesssim_{j,p,\alpha}
%\e(1+\|\mathfrak{I}_{\delta}\|_{p+\s_{j}}^{\gamma,\calO_0})
%\|i_{1}-i_{2}\|_{p+\s_{j}}\,,
%\end{aligned}
%\end{equation}
%and,  
%for $s_0\leq s\leq \mathcal{S}$,
%\begin{equation}\label{motociclo4}
%\begin{aligned}
%\mathbb{M}_{\mathcal{Q}^{(j)}}^{\gamma}(s,\mathtt{b})&\lesssim_{s,\rho}
%\e(1+\|\mathfrak{I}_{\delta}\|_{s+\s_{j}}^{\gamma,\calO_0})
%\,,\qquad
%0\leq \mathtt{b}\leq \rho-2\,,\\
%\mathbb{M}_{\Delta_{12}\mathcal{Q}^{(j)}}(p,\mathtt{b})&\lesssim_{p,\rho}
%\e(1+\|\mathfrak{I}_{\delta}\|_{p+\s_{j}}^{\gamma,\calO_0})\|i_{1}-i_{2}\|_{p+\s_{j}}
%\,,\qquad
%0\leq \mathtt{b}\leq \rho-3\,,
%\end{aligned}
%\end{equation}
%for some $\s_{j}=\s_{j}(\tau,\nu)>0$

\vspace{0.5em}
\noindent
 such that the following holds true. 
Let  ${\bf \Psi}_{j}^{\theta}:={\bf \Psi}_{j}^{\theta}(\vphi)$ be the flow at time 
$\theta\in [0,1]$ of 
\begin{equation}\label{flow-jth}
\partial_{\theta} {\bf \Psi}^{\theta}_{j} (\vphi) = 
 \Pi_{S}^{\perp}\opw{( \ii M_j (\vphi, x,\xi) )} 
\Pi_{S}^{\perp}{\bf \Psi}_{j}^{\theta}(\vphi) \, ,
 \quad {\bf \Psi}_{j}^{0}(\vphi) = {\rm Id} \, ,  
\end{equation}
with
\begin{equation}\label{generatore-jth}
M_j(\vphi,x,\xi):=\left(\begin{matrix}
0 & \!\! \!\!  -\ii m_j(\vphi,x,\xi) \\ 
-\ii\ov{m_{j}(\vphi,x,-\xi)}
& \!\! \!\! 0
\end{matrix}\right)\,, \qquad
m_j =\frac{\chi(\x)b_{j}(\vphi,x,\x)}{2\ii(1+a^{(0)}(\vphi,x))|\x|^{\frac{1}{2}}}\in 
S_1^{-\frac{j+1}{2}} \, ,
\end{equation}
where
$\chi\in {\cal C}^{\infty}(\R;\R)$ is an even and 
positive cut-off function
as in \eqref{cutofffunct}, and
 \begin{equation}\label{sist2 j-th}
\mathcal{Y}^{(j)} :=\mathcal{Y}^{(j)} (\vphi):=\omega\cdot\pa_{\vphi}+\opw\big(
\mathcal{D}(\vphi,x,\x)+B^{(j)}
\big)+\mathcal{Q}^{(j)}(\vphi)
\end{equation}
where  (recall \eqref{matriciinizio}, \eqref{nuovocoef})
\begin{equation}\label{matricDiag}
\mathcal{D}(\vphi,x,\x):=\ii A_{1}(\vphi,x)\x+\ii A_{1/2}^{(2)}(\vphi,x)
|\x|^{\frac{1}{2}}\,,
\end{equation}
and $B^{(j)}=B^{(j)}(\vphi,x,\x)$ is 
 the matrix of symbols of the form
\begin{equation}\label{bjbjbj}
B^{(j)} =\left(
\begin{matrix}
 c_j(\vphi,x,\x)& b_j(\vphi,x,\x)\vspace{0.2em}\\
 \ov{b_j(\vphi,x,-\x)} & \ov{c_j(\vphi,x,-\x)}
\end{matrix}
\right), \qquad c_j\in S_1^{-\frac{1}{2}},\;\;  
b_j\in S_1^{-\frac{j}{2}} \, .
\end{equation}
Then one has that
\begin{equation}\label{JJ+1}
\mathcal{Y}^{(j+1)}:=
\left({\bf \Psi}_{j}^{\theta} \mathcal{Y}^{(j)}{\bf \Psi}_{j}^{-\theta}\right)_{|\theta=1}
\end{equation}
has the form
 \eqref{sist2 j-th} with $ j + 1 $ instead of $ j $
 and $B^{(j+1)}$, $\mathcal{Q}^{(j+1)}$.
 %satisfying \eqref{motociclo}, \eqref{motociclo4} 
 %with $\s_{j}\rightsquigarrow\s_{j+1}$.

\vspace{0.9em}
{\bf Inizialization.} 
The operator  $\mathcal{L}_{2}$ in \eqref{elle2} has the form
of $\mathcal{Y}^{(j)}$ in \eqref{sist2 j-th}
for $j=0$
with matrix of symbols $B^{(0)}\rightsquigarrow A_{0}^{(2)}+A_{-1/2}^{(2)}$
and smoothing operator $\mathcal{Q}^{(0)}\rightsquigarrow \mathcal{R}^{(2)}$.

\vspace{0.9em}
{\bf Iteration.} 
We now argue by induction. 
First of all notice that the symbols of the matrix $ M_j $ defined 
in \eqref{generatore-jth} have negative order for any $ j \geq 0 $.
We prove that the flow map ${\bf \Psi}_{j}^{\theta}$
in \eqref{flow-jth} is well-posed.
\begin{lemma}\label{water1}
Assume that $b_{j}$ in \eqref{generatore-jth}
is in $S_1^{-j/2}$.
%satisfies \eqref{motociclo}. 
%and \eqref{motociclo2}. 
Then $m_j\in S_1^{-(j+1)/2}$.
%
%the seminorms $|M_j|_{-(j+1)/2,s,\alpha}^{\gamma,\calO_0}$,
%$|\Delta_{12}M_j|_{-(j+1)/2,p,\alpha}$ satisfy
%%the matrix of symbols $M_{j}$ in \eqref{generatore-jth}
%%satisfies 
%estimates like \eqref{motociclo}.
%%\eqref{motociclo2} 
%%for $ b_{j}$. 
Moreover
the map ${\bf \Psi}_{j}^{\theta}$ has the form
\begin{equation}\label{formazione}
{\bf \Psi}_{j}^{\theta}=\Pi_{S}^{\perp}\widetilde{{\bf \Psi}}_{j}^{\theta}\Pi_{S}^{\perp}
\circ({\rm Id}+\mathcal{R}_{j}(\theta))\,,\quad \forall \, \theta\in[0,1]\,,
\end{equation}
with
$\mathcal{R}_{j}\in\gotL^{1}_{\rho,p}\otimes\mathcal{M}_{2}(\mathbb{C})$ 
and $\widetilde{\Psi}_{j}^{\theta}$ is the flow of 
\begin{equation}\label{flow-jthsenzapro}
\partial_{\theta} \widetilde{\bf \Psi}^{\theta}_{j} (\vphi) = 
  \opw{(\ii M_j (\vphi, x,\xi) )} \widetilde{\bf \Psi}_{j}^{\theta}(\vphi) \, ,
 \quad \widetilde{\bf \Psi}_{j}^{0}(\vphi) = {\rm Id} \, .
\end{equation}
Moreover $\widetilde{\bf \Psi}^{\theta}_{j} (\vphi)$ has the form
\begin{equation}\label{formazione2}
\widetilde{\bf \Psi}^{\theta}_{j} (\vphi)={\rm Id}
+\opw(\widetilde{M}_{j}(\theta;\vphi,x,\x))+\widetilde{\mathcal{R}}_{j}(\theta)
\end{equation}
where $\widetilde{M}_{j}\in S_1^{-\frac{j+1}{2}}\otimes\mathcal{M}_{2}(\mathbb{C})$  and
$\widetilde{\mathcal{R}}_{j}\in\gotL^{1}_{\rho,p}\otimes\mathcal{M}_{2}(\mathbb{C})$.
In particular the non homogeneous terms of 
$\widetilde{M}_{j}$, $\mathcal{R}_{j}$  and $\widetilde{\mathcal{R}}_{j}$
satisfy the bounds \eqref{AVBDEF}, \eqref{AVBDEF2}
uniformly in $\theta\in[0,1]$.
%the same bounds (uniformly in $\theta\in [0,1]$)
%\eqref{motociclo}-\eqref{motociclo4}
% as $b_{j}$ and $\mathcal{Q}^{(j)}$ (respectively) with a possibly larger
% $\s_{j}>0$.
\end{lemma}

\begin{proof}
The first assertion of the matrix $M_{j}$ follows 
by the \eqref{generatore-jth} and by the fact that
$b_{j}\in S_1^{-j/2}$ 
and $a^{(0)}\in S_1^{0}$ by Lemma \ref{mappa2}.
The \eqref{formazione} follows by Lemma \ref{differenzaFlussi} 
(see also Remark \ref{differenzaFlussiRMK}). 
%and reasoning as in 
%Remark \ref{yumayuma}.
The \eqref{formazione2} follows by
applying Lemma B.6  in the appendix of 
\cite{FGP} and by  Lemma \ref{lem:escobarhomo}
%Lemmata \ref{weyl/standard}, \ref{lem:escobar}
%and Remark \ref{standard/Weyl} 
to pass to the Weyl quantization.
\end{proof}

\noindent
By Lemma \ref{georgiaLem}
the conjugated operator $\mathcal{Y}^{(j+1)}$ in \eqref{JJ+1}
under the map  $\Psi^{\theta}_{j}$ in \eqref{formazione}
 is given by
\begin{equation}\label{sis:j+1}
{\bf \Psi}_{j}^{1} \mathcal{Y}^{(j)}{\bf \Psi}_{j}^{-1}=
\widetilde{\bf \Psi}_{j}^{1}
  \big( \omega\cdot\pa_{\vphi}\widetilde{\bf \Psi}_{j}^{-1}\big)+
 \widetilde{\bf \Psi}_{j}^{1}  
 \opw\big(
\mathcal{D}(\vphi,x,\x)+B^{(j)}
\big) 
\widetilde{\bf \Psi}_{j}^{-1}  
+ \widetilde{\bf \Psi}_{j}^{1} \mathcal{Q}_{j}\widetilde{\bf \Psi}_{j}^{-1}\,,
\end{equation}
up to finite rank remainders in $\mathfrak{L}_{\rho,p}^{1}\otimes\mathcal{M}_2(\mathbb{C})$.
%(again recalling Remark \ref{yumayuma}).
 %satisfying bounds like  \eqref{motociclo4}.
The third summand in \eqref{sis:j+1} is in $\mathfrak{L}^{1}_{\rho,p}$
and satisfies the required estimates by Lemma \ref{water1} 
and the estimates on the smoothing remainder $\mathcal{Q}_{j}$.

Consider now the first summand in \eqref{sis:j+1}.
We have that the time contribution admits the Lie expansion
\begin{equation*}%\label{Lie2Vec}
\begin{aligned}
\widetilde{\bf \Psi}_{j}^{1}
\big( \omega\cdot\pa_{\vphi}\widetilde{\bf \Psi}_{j}^{-1}\big)
&=-\sum_{q=1}^{L}\frac{1}{q!}{\rm ad}^{q-1}_{\ii {\bf A}}
[\ii \omega\cdot\pa_{\vphi}{\bf A}]
%\\&
-
\frac{1}{L!}\int_{0}^{1} (1-\theta)^{L} 
\widetilde{\bf \Psi}^{\theta} {\rm ad}^{L}_{\ii {\bf A}}
[\ii \omega\cdot\pa_{\vphi}{\bf A}](\widetilde{\bf \Psi}^{\theta})^{-1}d\theta
\end{aligned}
\end{equation*}
specified for $ {\bf A}:=\opw(M_j (U) )$
and where ${\rm ad}_{\ii {\bf A}}[X]=[\ii {\bf A}, X]$.
We recall  (see \eqref{espansione2})  that
$$
M_{j}\#^{W}_{\rho}\omega\cdot\pa_{\vphi}M_{j}
-\omega\cdot\pa_{\vphi}M_{j}\#^{W}_{\rho}M_{j}=
\{M_{j},\omega\cdot\pa_{\vphi}M_{j}\}
\in S^{-({j+1})-1}
$$ 
up to a symbol in 
$S^{-({j+1})-3}_1$. 
By Lemma \ref{Jameshomo}
%Lemmata \ref{James}, \ref{James2} (and Remark \ref{yumayuma}) 
we have that 
${\rm ad}_{\ii \opw(M_{j})}[\ii \opw(\omega\cdot\pa_{\vphi}M_{j})]$ 
is a pseudo differential operator 
with symbol in $S^{-({j+1})-1}_1\otimes\mathcal{M}_2(\mathbb{C})$ 
plus a smoothing remainder 
in $\mathfrak{L}_{\rho,p}^{1}\otimes\mathcal{M}_2(\mathbb{C})$.
As a consequence we deduce 
(using also Lemma $B.3$ in \cite{FGP}), for $q\geq2$,
\[
{\rm ad}^{q}_{\ii \opw(M_{j})}[\ii \opw(\omega\cdot\pa_{\vphi}M_{j})]
=\opw(B_{q})+R_{q}, \qquad 
B_{q}\in S_{1}^{-\frac{j+1}{2}(q+1)-q}\otimes\mathcal{M}_2(\mathbb{C}) \, ,
\]
and $R_{q}\in \mathfrak{L}_{\rho,p}^{1}\otimes\mathcal{M}_2(\mathbb{C})$.
By  taking   $ L $ large enough with respect to $ \rho $, we obtain that 
$\widetilde{\bf \Psi}_{j}^{1}
\big( \omega\cdot\pa_{\vphi}\widetilde{\bf \Psi}_{j}^{-1}\big)$
is a 
 pseudo differential operator with symbol in 
 $ S_1^{-(j+1)/2}\otimes\mathcal{M}_2(\mathbb{C})$ plus a 
 smoothing operator in 
 $\mathfrak{L}^{1}_{\rho,p}\otimes\mathcal{M}_2(\mathbb{C})$. 
%The estimates \eqref{motociclo}, \eqref{motociclo4} 
%with $j\rightsquigarrow j+1$
%are a consequence of the bounds \eqref{crawford}-\eqref{jamal2} 
%in Lemma \ref{James} and of
%Lemma \ref{water1} to estimate the symbols in $M_{j}$.
%
% \eqref{motociclo}-\eqref{motociclo4}
%at the  $j$-th step.

We now study the space contribution, 
which is the second summand in \eqref{sis:j+1}. It admits 
the Lie expansion
\begin{equation}\label{Lie1Vec}
  \widetilde{\bf \Psi}_{j}^{1}  X (\widetilde{\bf  \Psi}_{j}^{1})^{-1}  = X
  +\sum_{q=1}^{L}\frac{1}{q!}{\rm ad}_{\ii {\bf A}}^{q}[X]
 +  \frac{1}{L!} \int_{0}^{1}  (1- \theta)^{L} \widetilde{\bf \Psi}^{\theta}
  {\rm ad}_{\ii{\bf A}}^{L+1}[X] (\widetilde{\bf \Psi}^{\theta})^{-1} 
d \theta
\end{equation}
where $X= \opw\big(
\mathcal{D}(\vphi,x,\x)+B^{(j)}
\big) $.
We claim that
\begin{equation}\label{eq:428}
 \begin{aligned}  
  \widetilde{\bf \Psi}_{j}^{1}  \opw\big(
\mathcal{D}(\vphi,x,\x)+B^{(j)}
\big) \widetilde{\bf \Psi}_{j}^{-1}  
 &= \opw\big(
\mathcal{D}(\vphi,x,\x)+B^{(j)}
\big)\\
&\qquad\qquad+ 
\big[ \opw(\ii M_j  ) , \opw\big(
\mathcal{D}(\vphi,x,\x)+B^{(j)}
\big) \big] 
\end{aligned} 
\end{equation}
plus a pseudo differential operator with symbol in 
$ S_1^{-\frac{j+1}{2}}\otimes\mathcal{M}_{2}(\mathbb{C})$ 
and a 
 smoothing operator in $\mathfrak{L}^{1}_{\rho,p}\otimes\mathcal{M}_{2}(\mathbb{C})$.
 %which satisfy estimates \eqref{motociclo}, \eqref{motociclo4}.
We first give the expansion of the commutator terms in \eqref{eq:428}
using the expression of $\mathcal{Y}^{(j)}   $ in  \eqref{sist2 j-th}. 
By Lemma \ref{Jameshomo} 
%Lemmata \ref{James}, \ref{James2}
we have 
\begin{equation}\label{esp2-jth}
\begin{aligned}
&\Big[
\opw(\ii M_j), \opw(\mathcal{D}(\vphi,x,\x))
\Big]:=\opw\left(\sm{0}{p_{j}(\vphi,x,\x)}{\ov{p_{j}(\vphi,x,-\x)}}{0}\right)
+\opw(B)+R
\\
&\qquad p_{j}:=-2\ii m_{j}(\vphi,x,\xi)(1+a^{(0)}(\vphi,x))|\xi|^{\frac{1}{2}}
\end{aligned}
\end{equation}
for some matrices of symbols  
$B\in S^{-\frac{j+1}{2}}_1\otimes\mathcal{M}_{2}(\mathbb{C})$
and remainder $R\in \gotL_{\rho,p}^1\otimes\mathcal{M}_{2}(\mathbb{C})$.
Moreover, since $ B^{(j)} $, is a matrix of symbols of order $ - 1/ 2 $, 
for $ j \geq 1 $,  
respectively $ 0 $ for $ j = 0 $
(see \eqref{bjbjbj}), we have that 
\begin{equation*} %\label{esp1-jth}
\big[
\opw(\ii M_j ) , \opw(B^{(j)}) 
\big] \in 
\left\{
\begin{aligned}
& OPS^{-\frac{j+2}{2}}\otimes\mathcal{M}_{2}(\mathbb{C}) \quad {\rm for } \ j \geq 1 \\
& OPS^{-\frac{1}{2}}\otimes\mathcal{M}_{2}(\mathbb{C})
\qquad {\rm for } \ j =0 
\end{aligned}\right.
\end{equation*}
up to a smoothing operator in 
$\mathfrak{L}^{1}_{\rho,p}\otimes\mathcal{M}_{2}(\mathbb{C})$. 
It follows that the off-diagonal symbols of order $ - j / 2 $ 
in \eqref{eq:428} are 
of the form  $\sm{0}{q_{j}(\vphi,x,\x)}{\ov{q_{j}(\vphi,x,-\x)}}{0}$ with
\begin{equation*} %\label{equa-jthquatuor}
q_{j}(\vphi,x,\x)\stackrel{\eqref{esp2-jth}}{:=}b_{j}(\vphi,x,\x)
-2\ii m_{j}(\vphi,x,\xi)(1+a^{(0)}(\vphi,x))|\xi|^{\frac{1}{2}} \, . 
\end{equation*}
By the definition of $\chi$ in \eqref{cutofffunct} 
we have that
the operator 
$\opw((1-\chi(\x))b_{j}(\vphi,x,\x))$ is in 
$\mathfrak{L}^{1}_{\rho,p}$ for any 
$\rho\geq3$. 
This  follows by applying Lemma $B.2$ in \cite{FGP}
and by 
the fact that the symbol $(1-\chi(\x))b_{j}(\vphi,x,\x)$
is in $S_1^{-m}$ for any $m>0$.
%satisfies estimates like \eqref{motociclo} 
%%, \eqref{motociclo2}
%with $-j/2$ replaced by any $m<0$.
Moreover, by  
  the choice of $m_{j}(\vphi,x,\x)$ in \eqref{generatore-jth} we have that
  \[
 \chi(\x) b_{j}(\vphi,x,\x)+p_{j}(\vphi,x,\x)=0\,.
%-2\ii m_{j}(\vphi,x,\xi)(1+a^{(0)}(\vphi,x))|\xi|^{\frac{1}{2}}=0 \, .
  \]
 This implies that $[\ii \opw(M_{j}) ,\opw\big(
\mathcal{D}(\vphi,x,\x)+B^{(j)}
\big)]$
 is a pseudo differential operator with symbol in
  $S_1^{-\frac{j+1}{2}}\otimes\mathcal{M}_{2}(\mathbb{C})$ plus a remainder in 
  $\mathfrak{L}^{1}_{\rho,p}\otimes\mathcal{M}_{2}(\mathbb{C})$.
  Now, using again Lemma \ref{Jameshomo},
  %Lemmata \ref{James}, \ref{James2}, 
  we deduce, for $q\geq2$, 
  \[
{\rm ad}^{q}_{\ii \opw(M_{j})}[
\opw\big(
\mathcal{D}(\vphi,x,\x)+B^{(j)}
\big)
]=\opw(\widetilde{B}_{q})+\widetilde{R}_{q}, \qquad 
\widetilde{B}_{q}\in S_1^{-\frac{j+1}{2}q}\otimes\mathcal{M}_{2}(\mathbb{C}) \, , 
\]
where $\widetilde{R}_{q} $ is in 
$ \mathcal{L}^{1}_{\rho,p}\otimes\mathcal{M}_{2}(\mathbb{C})$.
  Using formula \eqref{Lie1Vec} with $L$ large enough (w.r.t. $\rho$)
 one obtains the claim  in  \eqref{eq:428}.

We conclude that 
the operator 
$\mathcal{Y}^{(j+1)}$ in \eqref{sis:j+1} has the form
\eqref{sist2 j-th}
for some matrix of symbol 
$B^{(j+1)}$ 
of the form \eqref{bjbjbj} with $j\rightsquigarrow j+1$ 
and smoothing operators $\mathcal{Q}^{(j+1)} $ 
in $\mathfrak{L}^{1}_{\rho,p}\otimes\mathcal{M}_2(\mathbb{C})$. 
%We remark that bounds 
%\eqref{motociclo}, \eqref{motociclo4}
%on $B^{(j+1)}$ and $\mathcal{Q}^{(j+1)}$
%follow since we just applied several times the composition
%Lemmata \ref{James} and \ref{James2}
%which provides precise estimates on the symbols and remainders 
%appearing in the composition
%of pseudo differential operators (see bounds \eqref{crawford}-\eqref{jamal2}).

We are in position to conclude the proof of Proposition \ref{blockTotale}.
\begin{proof}[{\bf Proof of Proposition \ref{blockTotale}}]
We set 
\begin{equation}\label{mappaRho}
{\bf \Psi}:=
({\bf \Psi}_{2\rho-1}^{\theta}\circ\cdots \circ{\bf \Psi}^{\theta}_0)_{|\theta=1}\,.
\end{equation}
Then, recalling  \eqref{elle2}, \eqref{nuovocoef}, \eqref{matriciinizio},
%\eqref{sis:j+1}, 
\eqref{sist2 j-th}, \eqref{matricDiag}, 
we have that $\mathcal{L}_{3}:=\mathcal{Y}^{(2\rho)}$
has the form \eqref{elle3}, \eqref{diagpezzo}.
%The estimates \eqref{motociclo2totale}, \eqref{motociclo4totale} 
%follow by \eqref{motociclo}, \eqref{motociclo4} for $j=2\rho$.
\end{proof}

The following result regards the algebraic properties of the map 
${{\bf \Psi}}$
and of the new operator $\mathcal{L}_3$ in \eqref{elle3}.

\begin{lemma}\label{GoodAmmissi3}
The operator  $\mathcal{L}_3$
%defined in \eqref{elle3} 
is in $\mathfrak{S}_1$ 
and the map ${\bf \Psi}$
%in \eqref{mappaRho} 
is in $\mathfrak{T}_1$ 
(recall Def. \ref{compaMulti}, 
\ref{goodMulti}). 
\end{lemma}

\begin{proof}
We prove inductively that  
$\mathcal{Y}^{(j)}\in \mathfrak{S}_1$ for 
$j\geq0$. 
%and map ${\bf \Psi}^{\theta}_{j}$
%for any $j\geq0$.
Let us start by the case $j=0$.
For $j=0$ we have that $\mathcal{Y}^{(0)}=\mathcal{L}_{2}\in \mathfrak{G}_1$ 
by  Lemma \ref{GoodAmmissi2}.
Assume that $\mathcal{Y}^{(j)}\in \mathfrak{S}_1$. We show that 
$\mathcal{Y}^{(j+1)}\in \mathfrak{S}_1$. First of all
the symbol
$m_j$ in \eqref{generatore-jth} satisfies \eqref{space3Xinv}
as well as  $a^{(0)}$ and $b_j$.
Moreover the matrix $M_j$ has the form \eqref{retoreMat} with symbols satisfying
\eqref{selfSimbo2} and \eqref{space3Xinv}. 
Hence, by Lemmata \ref{realHamMtrixSimbo}, 
\ref{lem:momentoSimbolo}
${\bf \Psi}^{\theta}_{j}\in  \mathfrak{T}_1$.
As a consequence $\mathcal{Y}^{(j+1)}$  is Hamiltonian 
and $x-$translation invariant
and the matrix of symbols $B^{(j+1)}$ in \eqref{bjbjbj} satisfies
\eqref{selfSimbo2} and \eqref{space3Xinv} by construction. Thus
 $\mathcal{Y}^{(j+1)}$ is in $\mathfrak{S}_1$. 
 By the fact that   
 $\mathcal{Y}^{(j)}\in \mathfrak{S}_1$ for $j=0,\ldots,2\rho$, implies that  
   ${\bf \Psi}^{\theta}_{j}\in  \mathfrak{T}_1$ for $j=0,\ldots, 2\rho-1$, we deduce that
   $\Psi\in \mathfrak{T}_1$.
\end{proof}

%%%%%%%%%%%%%%%%%%%%%%%%%%%%%%%%%%%%%
%%%%%%%%%%%%%%%%%%%%%%%%%%%%%%%%%%%%%
%%%%%%%%%%%%%%%%%%%%%%%%%%%%%%%%%%%%%
%%%%%%%%%%%%%%%%%%%%%%%%%%%%%%%%%%%%%

\section{Reduction at the highest orders}\label{ordiniPositivi}
In this section we want to eliminate the $(\vphi,x)$-dependence
from the 
unbounded symbols  in \eqref{diagpezzo}.
In particular 
we shall prove that the symbols transforms into \emph{integrable} ones according
to the following definition.
\begin{definition}{\bf (Integrable symbols).}\label{integroSimbo}
Let $k=2p$ for some $p\in\mathbb{N}$.
We say that a $k$-homogeneous symbol $a_k\in S^{m}$ as in \eqref{simboOMO2DEF}
is \emph{integrable} if 
it does not depend on $(\vphi,x)$ and 
can be written as
\begin{equation*}%\label{simboOMO2DEFIntegro}
a_{k}(\x):=\frac{1}{(\sqrt{2\pi})^{k}}
\sum_{\substack{j_{1},\ldots,j_{p}\in S
} }(\mathtt{a}_k)_{j_1,\ldots,j_{p}}(\x)
{\zeta_{j_1}\cdots\zeta_{j_p}}\,.
%e^{\ii(\s_1\mathtt{l}(j_1)+\ldots+\s_{i}\mathtt{l}(j_i))\cdot\f}
\end{equation*}
Equivalently a symbol $a_{k}\in S^{m}$ is integrable 
if \,$(\mathtt{a}_{k})_{j_1,\ldots,j_{k}}^{\s_1,\ldots,\s_{k}}\neq0$
implies that  $\left\{\big(j_i, \sigma_i \big) \right\}_{i=1}^k \in S\times\{\pm\}$
is such that, up to permutation of the indexes,
\begin{equation}\label{iena9}
\s_1=\cdots=\s_{p}=+\,, \quad \s_{p+1}=\cdots=\s_{2p}=-\,, 
\qquad \{j_1,\ldots, j_p\}=\{j_{p+1},\ldots, j_{2p}\}\,.
\end{equation}
We denote by $\mathcal{S}_{k}$ the set of such 
$\left\{\big(j_i, \sigma_i \big) \right\}_{i=1}^k$.
We adopt the convention to write $\mathcal{S}_{k}=\emptyset$ if $k=2p+1$ (recall Remark \ref{oddresonances}).
\end{definition}
\begin{remark}\label{accendo}
It is easy to note that, if $a_k\in S^m$ is integrable,
then it is a Fourier multiplier.
%Moreover
%notice that symbols $a_k$ with $k=2p$, as in \eqref{simboOMO2DEF} 
%restricted to $\mathcal{S}_{k}$ are integrable according to Definition \ref{integroSimbo}.
We remark that the function $\mathcal{R}_{\vec{\s}}(j_1,\ldots,j_{k})$ in \eqref{por}
is identically zero only when it is restricted to the set $\mathcal{S}_{k}$.
Indeed if $j\in S$ then $-j\notin S$ (see formula 
\eqref{TangentialSitesDP}, \eqref{TangentialSitesDPbis}).
\noindent
Notice that, for generic choices of $S$, we have that
$\left\{\big(j_i, \sigma_i \big) \right\}_{i=1}^k\notin\mathcal{S}_{k}$
implies $|\mathcal{R}_{\vec{\s}}(j_1,\ldots,j_{k})|\geq \mathtt{c}>0$
with $c$ depending only on $S$ and $k$.
\end{remark}

The key result of this section is the following.
\begin{proposition}{\bf (Reduction to constant coefficients).}\label{riduzSum}
Assume \eqref{IpotesiPiccolezzaIdeltaDP}.
For generic choices of the set of tangential sites $S$ 
%(see  Def. \ref{Generic})
there exist constants 
$\mathfrak{m}_{i} : \Omega_{\varepsilon}\to \mathbb{R}$, $i=1,1/2,0$
of the form
\begin{align}
\mathfrak{m}_i&=\mathfrak{m}_i(\omega)
=\varepsilon^2 m_i(\omega)+\mathtt{m}_{i}^{(\geq4)}(\omega)\,,
\quad 
\mathtt{m}_{i}^{(\geq4)}:=\sum_{k=2}^{4+2i}\e^{2k}m_{i}^{(2k)}(\omega)
+\mathtt{r}_i(\omega)\,, \label{mer}
\end{align}
where $m_i, m_{i}^{(2k)}, k\geq 2$ are integrable (and $2k$-homogenenous) 
according to Definition \ref{integroSimbo}, and %(recall \eqref{def:m1})
\begin{equation}\label{constM1}
m_1:=\frac{1}{\pi}\sum_{n\in S} n|n| \zeta_{n}\,,
\end{equation}
\begin{equation}\label{merstima}
\lvert \mathtt{r}_{i} \rvert^{\gamma, \Omega_{\varepsilon}}
\lesssim\e^{17-2b}\gamma^{-3+2i}\,,
\qquad |\Delta_{12} \mathtt{r}_{i}|\lesssim_{p}
\e(1+\|\mathfrak{I}_{\delta}\|_{p+\mu})\|i_{1}-i_{2}\|_{p+\mu}\,, \qquad i=1, 1/2, 0
\end{equation}
such that the following holds.
For all $\omega$ belonging to the set
\begin{equation}\label{omegaduegamma}
\Omega^{2\gamma}_{\infty}:=\big\{ \omega\in \mathcal{O}_0 : 
\lvert \omega \cdot \ell+ \mathfrak{m}_1 j \rvert\geq 
\frac{\gamma}{\langle \ell \rangle^{\tau}}, 
\,\,\forall \ell\in\mathbb{Z}^{\nu},\,\,\mathtt{v}\cdot \ell+j=0,\,\,\, j\in S^{c}, \,\,(\ell, j)\neq (0, 0)  \big\}\,,
\end{equation}
there is a map ${\bf \Phi}={\bf \Phi}(\omega, \varphi) \colon H^s_{S^{\perp}}(\mathbb{T})\times H^s_{S^{\perp}}(\mathbb{T})\to H^s_{S^{\perp}}(\mathbb{T})\times H^s_{S^{\perp}}(\mathbb{T})$
 such that  (recall \eqref{elle3})
\begin{equation}\label{elle6}
\mathcal{L}_6:={\bf \Phi} \,\mathcal{L}_3\,{\bf \Phi}^{-1}=
\Pi_{S}^{\perp}\Big(\omega\cdot\pa_{\vphi}+
\mathtt{D}
+\mathtt{Q}\Big)\,,
\end{equation}
where 
\begin{equation}\label{simboM}
\begin{aligned}
\mathtt{D}:=\opw\left(\begin{matrix} \ii \mathtt{M}(\x) & 0 \\
0 & -\ii \mathtt{M}(-\x)  \end{matrix}\right)\,,\qquad 
\ii \mathtt{M}(\x):=
\ii \mathfrak{m}_1\,\x+\ii (1+\mathfrak{m}_\frac{1}{2})|\x|^{\frac{1}{2}}
+\ii \mathfrak{m}_0\sign(\x)\,.
\end{aligned}
\end{equation}
The remainder $\mathtt{Q}$ has the form
\begin{equation}\label{formaRestoQQ}
\mathtt{Q}:=\opw\left( \begin{matrix}
\ii \mathtt{q}(\vphi,x,\x) & 0\\
0& -\ii \ov{\mathtt{q}(\vphi,x,-\x)}
\end{matrix}\right)+\mathcal{Q}(\vphi)\,,
\end{equation}
where $\mathtt{q}\in S^{-1/2}_{4}$ and $\mathcal{Q}\in \mathfrak{L}_{\rho,p}^{4}\otimes
\mathcal{M}_{2}(\mathbb{C})$ (see Def. \ref{espOmogenea}).
Moreover
the operator $\mathcal{L}_6$ is in 
$\mathfrak{S}_1$ and 
the map ${\bf \Phi}$ is in $\mathfrak{T}_1$ and  satisfies
 \begin{equation}\label{stimaMappaPhi}
\|({\bf \Phi})^{\pm1}u\|_{s}^{\gamma,\Omega^{2\gamma}_{\infty}}
\lesssim_{s}\|u\|_{s}^{\gamma, \Omega^{2\gamma}_{\infty}}+
\gamma^{-3}\big(\e^{13}+\e\|\mathfrak{I}_{\delta}\|^{\gamma, \calO_0}_{s+\mu}\big)
\|u\|_{s_0}^{\gamma, \Omega^{2\gamma}_{\infty}}\,,
\end{equation}
for some $\m>0$ depending on $\nu$.
\end{proposition}

The proof of Proposition \ref{riduzSum} is divided into several steps performed in subsections
\ref{redord1}, \ref{redord12}, \ref{ordiniNegativi}.

\subsection{Integrability at order 1}\label{redord1}
The aim of this subsection is to eliminate the $(\vphi,x)$-dependence from
the symbol $V(\vphi,x)\x$ in \eqref{diagpezzo}.
More precisely we shall prove the following result.

\begin{proposition}\label{proporduno}
Let $\rho\geq 3$, $p\geq s_0$ and assume that \eqref{IpotesiPiccolezzaIdeltaDP} holds. 
Then there exist $\mu=\mu(\nu)$ and 
$\mathfrak{m}_1\colon \Omega_{\varepsilon}\to \mathbb{R}$ of the form
\eqref{mer} with $i=1$, and $m_1$ defined as in \eqref{constM1}
such that, 
for all $\omega\in\Omega^{2\gamma}_{\infty} $
(see \eqref{omegaduegamma})
there exists a map ${\bf \Phi}_1(\omega, \varphi)\colon H^s_{S^{\perp}}(\mathbb{T})\times H^s_{S^{\perp}}(\mathbb{T})\to H^s_{S^{\perp}}(\mathbb{T})\times H^s_{S^{\perp}}(\mathbb{T})$ such that
(recall \eqref{elle3})
\begin{equation}\label{elle4}
\mathcal{L}_4:={\bf \Phi}_1 \,\mathcal{L}_3\,{\bf \Phi}_1^{-1}=\Pi_{S}^{\perp}\Big(\omega\cdot\pa_{\vphi}+
\opw\left(\begin{matrix} d^{(4)}(\vphi,x,\x) & 0 \\
0 & \ov{d^{(4)}(\vphi,x,-\x)}  \end{matrix}\right)+\mathcal{R}^{(4)}\Big)\,,
\end{equation}
where $\mathcal{R}^{(4)}\in\mathfrak{L}^{2}_{\rho,p}\otimes\mathcal{M}_2(\mathbb{C})$,
the symbol $d^{(4)}(\vphi,x,\x)$ has the form %(recall \eqref{diagpezzo})
\begin{equation}\label{diagpezzo2}
d^{(4)}(\vphi,x,\x):=
\ii \mathfrak{m}_1\,\x+\ii (1+a^{(4)}(\vphi,x))|\x|^{\frac{1}{2}}+c^{(4)}(\vphi,x,\x)\,,
\end{equation}
with $a^{(4)}\in S_2^{0}$, 
$c^{(4)}\in S_2^{-1/2}$ (recall Def. \ref{espOmogenea}).
Moreover, for any $\omega\in \Omega^{2\gamma}_{\infty}$, we have
\begin{equation}\label{stimaMapp1}
\|({\bf \Phi}_1)^{\pm1}u\|_{s}^{\gamma,\Omega^{2\gamma}_{\infty}}
\lesssim_{s}\|u\|_{s}^{\gamma, \Omega^{2\gamma}_{\infty}}+
\gamma^{-1}\big(\e^{13}+\e\|\mathfrak{I}_{\delta}\|^{\gamma,\calO_0}_{s+\mu}\big)
\|u\|_{s_0}^{\gamma, \Omega^{2\gamma}_{\infty}}\,.
\end{equation}
Finally the operator  $\mathcal{L}_4$
%defined in \eqref{elle4} 
is in $\mathfrak{S}_1$ and the map ${\bf \Phi}_1$
%of Proposition \ref{proporduno} 
is in $\mathfrak{T}_1$ (recall Def. \ref{compaMulti}, 
\ref{goodMulti}). 
\end{proposition}

The rest of the section \ref{redord1} is 
devoted to the proof of Proposition \ref{proporduno}. It  is based on two steps
which we perform in subsections \ref{sec:lemmapre} and \ref{sec:straight}.

\subsubsection{Preliminary steps}\label{sec:lemmapre}
By Lemma \ref{Doo}
we know that the function $V(\vphi,x)$ in \eqref{diagpezzo} 
admits an expansion in $\e$ as in \eqref{espandoenonmifermo} with 
$k=1$. In this subsection we show 
how to reduce the terms of order less or equal to $\e^{12}$.

\begin{lemma}{\bf ({Preliminary steps}).}\label{lemmapre}
There exists a function $\beta^{(1)}(\varphi, x)$ of the form
\begin{equation}\label{genPre1}
\beta^{(1)}(\vphi,x)=\sum_{i=1}^{12}\e^{i}\beta^{(1)}_i(\vphi,x)\,,
\end{equation}
where $\beta^{(1)}_i$ is a real valued $i$-homogeneous symbol 
(as in \eqref{simboOMO2DEF}) independent of
$\x\in\mathbb{R}$
for $i=1,\ldots,12$ satisfying 
\begin{equation}\label{stimaSoloe}
\|\beta^{(1)}\|_{s}^{\gamma,\mathcal{O}_0}\lesssim_{s}\e\,,
\end{equation}
such that the following holds.
Let
\begin{equation}\label{mappapasso1}
{\bf T}_1^{\tau}:=\left(
\begin{matrix}
\mathcal{T}_{\beta^{(1)}}^{\tau}(\vphi) & 0 \\
0 & \ov{\mathcal{T}_{\beta^{(1)}}^{\tau}(\vphi)}
\end{matrix}
\right)\,,\qquad \tau\in[0,1]\,,
\end{equation}
where $\mathcal{T}_{\beta^{(1)}}^{\tau}$ is the flow of \eqref{linprobpro}
with $f$ as in \eqref{sim1} and $\beta\rightsquigarrow \beta^{(1)}$, 
then
\begin{equation}\label{elle3star}
\mathcal{L}_{3, *}:={\bf T}^1_1 \,\mathcal{L}_3 \, ({\bf T}^1_1)^{-1}=
\Pi_{S}^{\perp}\Big(\omega\cdot\pa_{\vphi}+
\opw\left(\begin{matrix} d_*^{(3)}(\vphi,x,\x) & 0 \\
0 & \ov{d_*^{(3)}(\vphi,x,-\x)}  \end{matrix}\right)+\mathcal{R}_*^{(3)}\Big)
\end{equation}
where 
$\mathcal{R}_*^{(3)}\in\mathfrak{L}_{\rho, p}^{1}\otimes\mathcal{M}_2(\mathbb{C})$,  
the symbols $d_{*}^{(3)}$
has the form
\begin{align}
d^{(3)}_*(\vphi,x,\x)&:=
\ii V^{+}(\vphi,x)\x+\ii (1+a^{(3)}_*(\vphi,x))|\x|^{\frac{1}{2}}+c^{(3)}_*(\vphi,x,\x)\,,
\label{diagpezzo2star}\\
V^{+}(\vphi,x)&:=\e^{2}m_{1}+\sum_{k=1}^{6}\e^{2k}m_{1}^{(2k)}
+V^{+}_{\geq13}(\vphi,x)\,,
\label{termV1finale}
\\
1+a^{(3)}_*(\vphi,x)&:=
(1+a^{(3)}(\vphi,x+\beta^{(1)}(x)))
\big|(1+\tilde{\beta}^{(1)}_{y}(1,y))_{|y=x+\beta^{(1)}(x)}\big|^{\frac{1}{2}}
\label{termA1}\,,
\end{align}
where $m_1$ is in \eqref{constM1} and $m_{1}^{(2k)}$ are integrable 
(according to Def. \ref{integroSimbo})
and
\begin{equation}\label{iena40}
\begin{aligned}
\|V^{+}_{\geq13}\|_{s}^{\gamma,\calO_0}&\lesssim_s 
(\e^{13}+\e\|\mathfrak{I}_{\delta}\|_{s+\mu}^{\gamma,\calO_0})\,,
%\label{iena40}
\\
\|\Delta_{12}V^{+}_{\geq13}\|_{p}&\lesssim_p 
\e(1+\|\mathfrak{I}_{\delta}\|_{p+\mu})\|i_{1}-i_{2}\|_{p+\mu}\,,
%\label{iena41}
\end{aligned}
\end{equation}
for some $\mu=\mu(\nu)>0$,
and $c^{(3)}_*$ is in $S_1^{-1/2}$.
The operator $\mathcal{L}_{3,*}$ is in $\mathfrak{S}_1$, 
the map
$\mathcal{T}_{\beta^{(1)}}^{\tau}$ is in $\mathfrak{T}_1$ and satisfies
\begin{equation}\label{stimaMapp1bis}
\|(\mathcal{T}_{\beta^{(1)}})^{\pm1}u\|_{s}^{\gamma,\calO_0}
\lesssim_{s}\|u\|_{s}^{\gamma, \calO_0}+
\e
\|u\|_{s_0}^{\gamma, \calO_0}\,.
\end{equation}
\end{lemma}
%\begin{remark}\label{ienaNantes}
%We underline that, since $\beta^{(1)}$ is real valued, the
%symbols $\ii \x \beta^{(1)}(\vphi,x)$ and 
%$\ii \x b(\tau,\vphi,x)$ (with $b(\tau,\vphi,x)$ defined as in \eqref{sim1}) 
%satisfy the condition
%\eqref{iena2PureIm}. 
%Hence the pseudo differential operator $\opw(\ii \x b(\tau,\vphi,x))$ 
%is \emph{skew-self-adjoint}.
%\end{remark}
\begin{proof}
We start by proving the
last assertions on the map ${\bf T}_1^{\tau}$.
%$\mathcal{T}_{\beta^{(1)}}^{\tau}$.
If $\beta^{(1)}$  is real valued then 
we have that
the symbol $b(\tau,\vphi,x)\x$ defined in \eqref{sim1} with 
$\beta\rightsquigarrow \beta^{(1)}$ in \eqref{genPre1}
is real and $x$-translation invariant.
Then the flow map generated by 
$\Pi_{S}^{\perp}\opw(\ii b(\tau,\vphi,x)\x)\Pi_{S}^{\perp}$ (see \eqref{linprobpro})
is symplectic and $x$-translation invariant, 
i.e. it belongs to $\mathfrak{T}_1$.
%Then $\mathcal{T}_{\beta^{(1)}}^{\tau}$ is in $\mathfrak{T}_1$ because
%$\mathcal{T}_{\beta^{(1)}}^{\tau}$ is the flow map generated by 
%$\Pi_{S}^{\perp}\opw(\ii b(\tau,\vphi,x)\x)\Pi_{S}^{\perp}$ (see \eqref{linprobpro}),
%which is Hamiltonian and $x$-translation invariant.
%\eqref{diffeotot}).
Moreover, if $\beta^{(1)}$ satisfies \eqref{stimaSoloe}, 
by Lemma \ref{differenzaFlussi}, 
we can write
\[
\mathcal{T}^{1}_{\beta^{(1)}}=
\Pi_{S}^{\perp} \Phi_{\beta^{(1)}}^{\tau}\Pi_{S}^{\perp}\circ({\rm Id}+\mathcal{R})\,,
\]
where $\Phi_{\beta^{(1)}}^{\tau}$ is the flow of \eqref{linprob}
 with $f\rightsquigarrow b\x$ with $b$ as in \eqref{sim1}
with $\beta\rightsquigarrow \beta^{(1)}$
and where $\mathcal{R}$ is a finite rank operator 
of the form \eqref{FiniteDimFormDP10001}
satisfying \eqref{giallo2}.
Therefore (see Lemma \ref{CoroDPdiffeo})
we have that
the flow $\mathcal{T}_{\beta^{(1)}}^{\tau}$ is well posed and satisfies estimates like
\eqref{flow2}-\eqref{flow222}. 
This implies, using also \eqref{stimaSoloe}, the \eqref{stimaMapp1bis}.

\medskip
%We now conjugate the operator $\mathcal{L}_{3}$ in 
%\eqref{elle3} with the map $\mathcal{T}_{\beta^{(1)}}^{\tau}$.
By Lemma \ref{georgiaLem} (see item $(ii)$) we have that the conjugate of 
$\mathcal{L}_3$ in \eqref{elle3} under the map ${\bf T}_{1}^{1}$
is given by
\begin{equation}\label{formazione5}
\Pi_{S}^{\perp}\widetilde{{\bf T}}_1^{1}\mathcal{L}_{3}(\widetilde{{\bf T}}_1^{1})^{-1}\Pi_{S}^{\perp}\,,
\qquad
\widetilde{{\bf T}}_1^{\tau}:=\left(
\begin{matrix}
{\Phi^{\tau}_{\beta^{(1)}}}(\vphi) & 0 \\
0 & \ov{\Phi_{\beta^{(1)}}^{\tau}(\vphi)}
\end{matrix}
\right)\,,\qquad \tau\in[0,1]\,,
\end{equation}
up to finite rank operators belonging to $\mathfrak{L}_{\rho, p}^{1}$. 
%This follows by estimates \eqref{giallo2Rank}
%and Lemma \ref{Doo} to estimates the symbols of the operator $\mathcal{L}_3$ in 
%\eqref{elle3}.
In order to study the operator in \eqref{formazione5}, 
we apply Theorem \ref{EgorovQuantitativo}
and Lemmata \ref{EgorovTempo}, \ref{preparailsugo} to the operator 
$\mathcal{L}_{3}$ in \eqref{elle3}.
Then formul\ae \,\eqref{elle3star}-\eqref{diagpezzo2star} follow
for some 
$\mathcal{R}^{(3)}_{*}\in \mathfrak{L}_{\rho,p}^{1}\otimes\mathcal{M}_2(\mathbb{C})$, 
$c^{(3)}_{*}\in S^{-1/2}_{1}$.
Formula \eqref{termA1} on $a_{*}^{(3)}$ follows by \eqref{formaEsplic2}
applied with $w(x,\x)\rightsquigarrow (1+a^{(0)}(x))|\x|^{1/2}$.
Similarly, by a direct computation, we have that
the new symbol at order $1$ is given by
\begin{equation}\label{termV1}
V^{+}(\vphi,x):=-g(\vphi,x)+ 
V(\vphi,x+\beta^{(1)}(x))(1+\tilde{\beta}^{(1)}_{y}(1,y))_{|y=x+\beta^{(1)}(x)}\,,
\end{equation}
where $g(\vphi,x)$ is given by Lemma \ref{EgorovTempo}.
We need to prove that, for some suitable $\beta^{(1)}$ the symbol in \eqref{termV1}
has the form \eqref{termV1finale}.

By  formula
 \eqref{eq:intq0} in Lemma \ref{EgorovTempo} 
 we get the following 
  Taylor expansion for 
 $g(\vphi,x)$:
\begin{equation}\label{Doo2}
\begin{aligned}
g(\vphi,x)&=\e\omega\cdot\pa_{\vphi}\beta^{(1)}_1+\e^{2}\omega\cdot\pa_{\vphi}\beta^{(1)}_2-\e^{2}(\beta^{(1)}_1)_x\omega\cdot\pa_{\vphi}\beta^{(1)}_1
+
\sum_{i=3}^{12}\e^{i}\big(\omega\cdot\pa_{\vphi}\beta^{(1)}_i+\mathtt{g}_i\big)+
\widetilde{g}(\vphi,x)\\
&\stackrel{\eqref{FreqAmplMapDP}, \eqref{LinearFreqDP}}{=}
\e\bar{\omega}\cdot\pa_{\vphi}\beta^{(1)}_1+
\e^{2}\bar{\omega}\cdot\pa_{\vphi}\beta^{(1)}_2-\e^{2}(\beta^{(1)}_1)_x\bar{\omega}\cdot\pa_{\vphi}\beta^{(1)}_1+
\sum_{i=3}^{12}\e^{i}\big(\bar{\omega}\cdot\pa_{\vphi}\beta^{(1)}_i
+\widetilde{\mathtt{g}}_i\big)+
\widetilde{g}(\vphi,x)\,,
\end{aligned}
\end{equation}
where $\widetilde{g}(\vphi,x)$ satisfies \eqref{AVBDEF}
%-\eqref{AVB2DEF} 
(actually is independent of
$\mathfrak{I}_{\delta}$)
and $\mathtt{g}_i$, $\widetilde{\mathtt{g}_i}(\vphi,x)$ are 
$i$-homogeneous symbols of the form \eqref{simboOMO2DEF}
whose coefficients depend, for any $i=1,\ldots,12$, 
only on $\beta^{(1)}_{j}$ with $j<i$.

\noindent
By using formula \eqref{dthetaa0} in 
Theorem \ref{EgorovQuantitativo} and recalling that the 
symbol $V$ has an expansion as in 
\eqref{espandoenonmifermo}-\eqref{simboOMO2DEF}
for some homogeneous symbols $\mathtt{V}_i$ 
(see Lemma \ref{Doo}), we
can Taylor expand 
the second summand in \eqref{termV1}. We have
\begin{equation}\label{Doo33}
\begin{aligned}
V(\vphi,x+\beta^{(1)}(x))(1+\tilde{\beta}^{(1)}_{y}(1,y))_{|y=x+\beta^{(1)}(x)}&=
\e\mathtt{V}_1+\e^{2}\Big(\mathtt{V}_2
+(\mathtt{V}_1)_x\beta^{(1)}_1-\mathtt{V}_1(\beta^{(1)}_1)_{x}
\Big)
%\\&
+\sum_{i=3}^{12}\e^{i}\mathtt{r}_{i}+\widetilde{V}(\vphi,x)\,,
\end{aligned}
\end{equation}
where $\widetilde{V}(\vphi,x)$ satisfies \eqref{AVBDEF}
%-\eqref{AVB2DEF}
and $\mathtt{r}_i$ are $i$-homogeneous symbols 
of the form \eqref{simboOMO2DEF}
whose coefficients depend, for any $i=1,\ldots,12$, 
only on $\beta^{(1)}_{j}$ with $j<i$.
Hence, more precisely, we have
\begin{align}
V^{+}(\vphi,x)&=\sum_{i=1}^{12}\e^{i} \mathtt{V}^{+}_i(\vphi,x)
+\widetilde{V}^{+}(\vphi,x)\,,\label{Doo3}\\
\mathtt{V}^{+}_1&=-\bar{\omega}\cdot\pa_{\vphi}\beta^{(1)}_1+\mathtt{V}_1\,,\label{Doo4}\\
\mathtt{V}_2^{+}&=-\bar{\omega}\cdot\pa_{\vphi}\beta^{(1)}_2+
F_{2}\,, \quad F_{2}:=\mathtt{V}_2+\mathtt{f}_2\,,\label{Doo5}\\ 
& \qquad\mathtt{f}_2 =\mathtt{r}_2-\widetilde{g}_{2}=(\mathtt{V}_1)_x\beta^{(1)}_1-\mathtt{V}_1(\beta^{(1)}_1)_{x}
+(\beta^{(1)}_1)_x\bar{\omega}\cdot\pa_{\vphi}\beta^{(1)}_1
\nonumber\\
\mathtt{V}_i^{+}&=-\bar{\omega}\cdot\pa_{\vphi}\beta^{(1)}_i+
F_{i}\,,\quad F_{i}:=
\mathtt{V}_i+
\mathtt{f}_i\,,\quad i=3,\ldots,12\,,
\label{Doo55}\\
\widetilde{V}^{+}&:=\widetilde{V}-\widetilde{g},\label{Doo56}
\end{align}
where $\mathtt{f}_i=\mathtt{r}_i-\widetilde{g}_{i}$
are $i$-homogeneous symbols of the form \eqref{simboOMO2DEF}
whose coefficients are sums and products 
of derivatives of 
%depend, for any 
%$i=3,\ldots,12$, only on 
$\beta^{(1)}_{j}$, $\mathtt{V}_j$ 
with $j<i$. The function $\widetilde{V}^{+}$ satisfies \eqref{AVBDEF}
%-\eqref{AVB2DEF}
with $k=1$.
We shall verify that 
the coefficients $(F_{i})^{\s_1\cdots\s_{i}}_{j_1,\ldots,j_{i}}$, $i=2,\ldots,12$
satisfy the conditions \eqref{iena2}
by showing, iteratively, that the $\beta^{(1)}_{j}$ are real valued.
%, \eqref{iena32} and \eqref{iena42}. 
Indeed we observe that  to construct the coefficients $\mathtt{f}_{i}$ we just used the 
equations  \eqref{eq:intq0} and  \eqref{dthetaa0}
to Taylor expand the new symbol. 
%Hence the result follows by applying Lemma \ref{ienaNantes2}.

%\medskip
We now show how to choose the function $\beta^{(1)}$ in \eqref{genPre1}
in such a way that the function $V^{+}$ in \eqref{Doo3}
is \emph{integrable}
(according to Def. \ref{integroSimbo}) up to terms of ``high'' degree of homogeneity.

All the functions $\mathtt{f}_i$ can be actually  computed explicitly but this is not necessary for our scope.
Our aim is to find, iteratively, functions $\beta^{(1)}_i$ in order to reduce the functions $V_{i}^{+}$
to integrable symbols (see Def. \ref{integroSimbo}).
We shall denote
by 
\[
(\beta^{(1)}_{i})_{j_1,\ldots,j_i}^{\s_1\ldots \s_{i}}\,,\quad
(\mathtt{V}_{i})_{j_1,\ldots,j_i}^{\s_1\ldots \s_{i}}\,,\quad 
(F_{i})_{j_1,\ldots,j_i}^{\s_1\ldots \s_{i}}\,,\quad i=1,\ldots,12
\]
respectively the coefficients in the expansion \eqref{simboOMO2DEF}
of the functions $\beta^{(1)}_i$, $\mathtt{V}_{i}$ and $F_{i}$ in \eqref{Doo4}-\eqref{Doo55}.
By Lemma \ref{Doo}, Lemma \ref{iena10}
applied on the symbols
$\ii\x \mathtt{V}_i$ %and formula \eqref{simboOmoxixi},
the coefficients
of the functions $\mathtt{V}_{i}$ satisfy the conditions 
\eqref{iena2}.
%, \eqref{iena32} and \eqref{iena42}
%in Lemma \ref{ienaiena10}. 

\vspace{0.5em}
\noindent
{\bf Step $\e$.}
By \eqref{simboOMO2DEF} we have the expansion
\[
\mathtt{V}_1:=\sum_{\s=\pm, n\in S}(\mathtt{V}_1)_{n}^{\s}\sqrt{\zeta_{n}}e^{\ii \s\mathtt{l}(n)\cdot\vphi}
e^{\ii nx}\,,\qquad (\mathtt{V}_1)_{n}^{\s}\in \mathbb{C}\,.
\]
Since the function $\mathtt{V}_1$ is real-valued, the coefficients $(\mathtt{V}_1)_{n}^{\s}$
satisfy
\eqref{iena2} with $k=1$.
We set
\begin{equation}\label{def:beta1}
(\beta^{(1)}_1)_{n}^{\s}:= - \frac{\s({\mathtt V}_1)^{\s}_{n}}{\ii\sqrt{|n|}}\,,\quad \s=\pm\,, n\in S\,,\;\;n\neq0\,.
\end{equation}
One can check that 
\begin{equation}\label{detodounpoco}
V_{1}^{+}\stackrel{\eqref{Doo4}}{=}-\bar{\omega}\cdot\pa_{\vphi}\beta^{(1)}_1+\mathtt{V}_1
\stackrel{\eqref{def:beta1}}{\equiv}0\,
\end{equation}
and that the coefficients of $\beta^{(1)}_1$ satisfy \eqref{iena2},
%, \eqref{iena32} and \eqref{iena4}, 
i.e. $\beta^{(1)}_1$ is real valued.
%verifies Hypothesis  $(\bf H)$.

\vspace{0.5em}
\noindent
{\bf Step $\e^{2}$.}
Now we want to eliminate the function in \eqref{Doo5}.
%We write $F_2:=\mathtt{V}_2+(\mathtt{V}_1)_x\beta_1-\mathtt{V}_1(\beta_1)_{x}
%+(\beta_1)_x\bar{\omega}\cdot\pa_{\vphi}\beta_1$ which has an expansion as in \eqref{simboOMO2DEF}
%with $k=2$ with some coefficients $(F_2)_{n_1,n_2}^{\s_1,\s_2}$
%again satisfying the reality condition \eqref{iena2}.
We  define, for $n_1,n_2\in S$,
\begin{equation}\label{beta-2}
(\beta^{(1)}_{2})^{\s\s}_{n_1, n_2} := 
\frac{-(F_2)^{\s\s}_{n_1, n_2}}{\ii \s (\omega_{n_1} + \omega_{n_2})} \, , \, 
\;\; \s=\pm \, ,
\qquad 
(\beta^{(1)}_{2})^{\s (-\s)}_{n_1, n_2} :=  
\frac{-\s(F_{2 })^{\s(-\s)}_{n_1, n_2}}{\ii (\omega_{n_1} - \omega_{n_2})} \, , 
\;\; n_1 \neq \pm n_2 \, , 
\end{equation}
and 
$(\beta^{(1)}_{2})_{0,0}^{\s \s} := 0 $, 
$(\beta^{(1)}_2)^{\s(-\s)}_{n,\pm n}:=0$.
In this way we have 
\begin{equation}\label{nuovcoeffBeta2}
\begin{aligned}
\mathtt{V}_2^{+}&=m_1
+\sum_{n\in S^{+}\cup(-S^{+})}\big((F_2)^{+-}_{n,-n}+(F_2)^{-+}_{-n,n}\big)\,e^{\mathrm{i} 2 n x}\,\zeta_{n}\,,\\
m_1&:=\frac{1}{2\pi}\sum_{n\in S}\big((F_2)^{+-}_{n,n}+(F_2)^{-+}_{n,n}\big)\zeta_{n}\,.
\end{aligned}
\end{equation}
By the form of $S$ in \eqref{TangentialSitesDP} we note 
that the second sum in \eqref{nuovcoeffBeta2}
is actually zero since, if $n\in S$ then $-n\in S^{c}$.
Hence $\mathtt{V}_2^{+}$ is integrable according to Definition \ref{integroSimbo}.
Again we can check explicitly that
$\beta^{(1)}_2$, defined by \eqref{beta-2}, 
satisfies \eqref{iena2}, i.e. $\beta^{(1)}_2$ is real.
This is a consequence of the fact that $F_2$ in \eqref{Doo5} is real
since $\beta_{1}^{(1)}$ is real.
It remains to show that $m_1$ in \eqref{nuovcoeffBeta2}
has the form \eqref{constM1}.
To do this we compute the coefficients of the function $F_{2}$
in \eqref{Doo5}.
Using the expansion \eqref{iena} with $k=1$ for $\mathtt{V}_1$ and $\beta^{(1)}_1$, 
and recalling \eqref{vsegnatoDP2}, 
we get
\begin{equation}\label{ienaNantes1000}
\begin{aligned}
F_{2}&\stackrel{\eqref{detodounpoco}, \eqref{Doo5}}{=} \mathtt{V}_2+
(\mathtt{V}_1)_x\beta^{(1)}_1=\sum_{j_1,j_2\in S,\s=\pm}(\mathtt{F}_2)^{\s\s}_{j_1,j_2}v_{j_1}^{\s}v_{j_2}^{\s} e^{\ii \s(j_1+j_2)x}\\
&+\frac{1}{2\pi}\sum_{j_1,j_2\in S}v_{j_1}^{+}v_{2}^{-}\big[
({\mathtt{V}}_2)^{+-}_{j_1,j_2}+({\mathtt{V}}_2)^{-+}_{j_2,j_1}+
(\mathtt{V}_1)^{+}_{j_1}(\beta^{(1)}_1)^{-}_{j_2}\ii j_1-
(\mathtt{V}_1)^{-}_{j_2}(\beta^{(1)}_1)^{+}_{j_1}\ii j_2
\big]e^{\ii (j_1-j_2)x}\,.
\end{aligned}
\end{equation}
We are not interested in computing the coefficients 
$(\mathtt{F}_2)^{\s\s}_{j_1,j_2}$.
By  Lemma \ref{Doo} we have
\[
\begin{aligned}
(F_2)^{+-}_{n,n}+(F_2)^{-+}_{n,n}&\stackrel{\eqref{ienaNantes1000}, 
\eqref{def:beta1}}{=}
({\mathtt{V}}_2)^{+-}_{n,n}+({\mathtt{V}}_2)^{-+}_{n,n}+
(\mathtt{V}_1)^{+}_{n}(\mathtt{V}_1)^{-}_{n}\frac{2n}{\sqrt{|n|}}
%\\&
\stackrel{\eqref{mediediagonali}}{=}n|n|+
 \frac{1}{2} n^{2} |n|^{-1/2}2n|n|^{-1/2}=2n|n|
\end{aligned}
\]
which implies the \eqref{constM1}.

\vspace{0.5em}
\noindent
{\bf Step $\e^{\geq3}$.} %, $3\leq p\leq 14$.}
Consider now the functions \eqref{Doo55}.
We recall that for any $3\leq i\leq 12$, 
the functions $V_{i}^{+}$ depends only on $\beta^{(1)}_{j}$, $\mathtt{V}_j$, 
$\mathtt{V}_i$ 
with $j<i$. 
For $i=3,\ldots, 12$ we define 
\begin{equation}\label{partenorm}
[\mathtt{V}_i^{+}]:=\sum_{\mathcal{S}_{i} }
(\mathtt{V}^{+}_i)_{j_1,\ldots,j_{i}}^{\s_1,\ldots,\s_{i}}(\x)
\sqrt{\x_{j_1}\cdots\x_{j_i}}
e^{\ii (\sum_{k=1}^{i}\s_{k} j_{k})x}
e^{\ii (\sum_{k=1}^{i}\s_{k}\mathtt{l}(j_{k}))\cdot\vphi}\,,
\end{equation}
where $\sum_{\mathcal{S}_i}$ denotes the sum over indexes restricted to the set $\mathcal{S}_i$
defined in Definition \ref{integroSimbo}.
We also set , for $\vec{\s}=(\s_{1},\ldots,\s_{i})$, $\vec{j}=(j_1,\ldots,j_{i})$,
\begin{equation}\label{def:betai}
(\beta^{(1)}_{i})_{\vec{j}}^{\vec{\s}}:=
-\frac{(\mathtt{F}_{i})^{\vec{\s}}_{\vec{j}}}{\ii \mathcal{R}_{\vec{\sigma}}(\vec{j})}\,,
\quad %(\vec{\s},\vec{j})
\left\{\big(j_k, \sigma_k \big) \right\}_{k=1}^i \notin \mathcal{S}_{i}\,,
\qquad
(\beta^{(1)}_{i})_{\vec{j}}^{\vec{\s}}=0\,, \quad  %(\vec{\s},\vec{j})
\left\{\big(j_k, \sigma_k \big) \right\}_{k=1}^i \in \mathcal{S}_{i}\,,
\end{equation}
and $\mathcal{R}_{\vec{\sigma}}(\vec{j})$ is the function defined in \eqref{por}.
Again, by induction,  the coefficients in \eqref{def:betai}
satisfy \eqref{iena2}, hence 
the functions $\beta_{i}^{(1)}$ are real valued.
Moreover, by using the estimates on $\mathcal{R}_{\vec{\sigma}}$ 
in Remark \ref{accendo},
%Lemma \ref{nondegPreliminary}, 
the function $\beta^{(1)}$
in \eqref{genPre1}
satisfies \eqref{stimaSoloe}.
By Remark \ref{accendo} we have
that the symbols in \eqref{partenorm}
are integrable according to Definition \ref{integroSimbo}.
Therefore the symbol $V^{+}$ in \eqref{diagpezzo2star} 
has the form \eqref{termV1finale}
by setting $m_1^{(2k)}:=[\mathtt{V}_{2k}^{+}]$ (see \eqref{partenorm})
and $V_{\geq13}^{+}:=\widetilde{V}^{+}$ in \eqref{Doo56}.
The estimates \eqref{iena40} %, \eqref{iena41} 
follow by
 \eqref{AVBDEF}
 %-\eqref{AVB2DEF}
with $k=1$ on $\widetilde{V}^{+}$.
Since $\beta^{(1)}$ is $x$-translation invariant (i.e. satisfies \eqref{space3Xinv})
it is easy to check that the symbols $V^{+}_{\geq13}$, $a^{(3)}_{*}$, $c^{(3)}_*$
constructed above satisfy the same condition. Hence, since the map 
$\mathcal{T}_{\beta^{(1)}}^{\tau}$ is in $\mathfrak{T}_1$, one has that
$\mathcal{L}_{3,*}$ belong to $\mathfrak{S}_1$. 
\end{proof}

\subsubsection{Straightening theorem}\label{sec:straight}
In this subsection we conclude the proof of Proposition \ref{proporduno}.
More precisely we eliminate the dependence on $(\vphi,x)$ in the symbol
$V^{+}_{\geq13}$ in \eqref{termV1finale}. 
We first need a preliminary result.

\begin{lemma}{\bf ({Straightening Lemma}).}\label{drizzotutto}
For all $\omega\in \Omega^{2 \gamma}_{\infty}$ defined in 
\eqref{omegaduegamma} there 
exists 
$\beta^{(2)}(\omega):=\beta^{(2)}(\omega, \cdot)\in S_{\mathtt{v}}$, 
 (see \eqref{funzmom}), 
satisfying 
\begin{equation}\label{stimabeta2}
\begin{aligned}
\lVert \beta^{(2)} \rVert_s^{\gamma, \Omega_{\infty}^{2 \gamma}}\,
&\lesssim_s
 \gamma^{-1} (\e^{13}
 +\e\|\mathfrak{I}_{\delta}\|_{s+\mu}^{\gamma,\calO_0})\\
\|\Delta_{12}\beta^{(2)}\|_{p}\,
&\lesssim_p 
\e\gamma^{-1}(1+\|\mathfrak{I}_{\delta}\|_{s+\mu})\|i_{1}-i_{2}\|_{p+\mu}\,
\end{aligned}
\end{equation}
for some $\mu=\mu(\nu)>0$,
such that 
the vector field on $\mathbb{T}^{\nu+1}$ 
\begin{equation}\label{vecfield1}
Y:=\omega \cdot \frac{\partial}{\partial \varphi}
+\big(\mathtt{c}+V_{\geq13}^{+}(\varphi, x) \big)\frac{\partial}{\partial x}\,,
\qquad 
\mathtt{c}:=\mathtt{c}(\omega):= \e^{2}m_{1}+\sum_{k=1}^{6}\e^{2k}m_{1}^{(2k)}\,,
\end{equation}
transforms,
under the diffeomorphism 
$\Gamma\colon(\varphi, x)\mapsto (\varphi, y),\,y:= x+\beta^{(2)}
(\omega; \varphi,  x)$ of $\mathbb{T}^{\nu+1}$,  into
\begin{equation}\label{vecfield2}
\Gamma_* Y=
%\omega \cdot \frac{\partial}{\partial \varphi}
%+\Big((f\circ\Psi^{-1})(\widetilde{\Theta})
%+\mathtt{w}\cdot \partial_{\widetilde{\Theta}} 
%\gamma(\widetilde{\Theta}) \Big)\frac{\partial}{\partial y}
%=
\omega \cdot \frac{\partial}{\partial \varphi}
+\mathfrak{m}_1\,\frac{\partial}{\partial y}\,,
\end{equation}
where $\mathfrak{m}_{1}=\mathfrak{m}_{1}(\omega)$ 
is defined for $\omega\in \Omega_{\e}$ (see \eqref{OmegaEpsilonDP}),
and $\mathtt{r}_1:=\mathfrak{m}_1-\mathtt{c}$ satisfies
%\[
%m=\mathtt{c}+\int_{\mathbb{T}^{\nu}}  V^+_{\geq 15}\circ \Psi^{-1}\,d  \widetilde{\Theta}.
%\]
\begin{equation}\label{vecfield3}
\begin{aligned}
\lvert \mathtt{r}_1 \rvert^{\gamma, \Omega_\e}&\lesssim 
\lVert   V^+_{\geq 13}  \rVert^{\gamma, \calO_0}_{s_0}
(1+\gamma^{-1}\lVert V^{+}_{\geq13} \rVert^{\gamma, \calO_0}_{s_0+2\tau+5})\,,\\
|\Delta_{12} \mathtt{r}_1|&\lesssim_{p}
\e(1+\|\mathfrak{I}_{\delta}\|_{p+\mu})\|i_{1}-i_{2}\|_{p+\mu}.
\end{aligned}
\end{equation}
\end{lemma}
\begin{proof}
%We shall  apply Theorem $3.1$ in \cite{FGMP}
%(see also Proposition $3.6$ and Lemma $3.7$ in \cite{FGP}).
First of all we recall that, by Lemma \ref{lemmapre}, 
the function $V^{+}_{\geq13}$ is in $S_{\mathtt{v}}$. Then 
we write $\mathtt{c}+V^{+}_{\geq13}(\vphi,x)=f(\Theta)$ 
with  $\Theta=\vphi-\mathtt{v}x$
for some smooth function 
$f : \mathbb{T}^{\nu}\to\mathbb{R}$.
Then the vector field in \eqref{vecfield1} can be written as a non-degenerate vector field on $\mathbb{T}^{\nu}$
\[
X:=\big(\omega-\mathtt{v}\,f( \Theta)\big)\, 
\cdot\frac{\partial}{\partial \Theta}\,.
\]
Theorem $3.1$ in \cite{FGMP} (see also Proposition $3.6$ and Lemma $3.7$ in \cite{FGP})
provides the existence of a Cantor set 
$\Omega_{\infty}^{2 \gamma}\subset \mathbb{R}^{\nu}$ 
(of the form  \eqref{omegaduegamma})
such that for $\omega\in \Omega_{\infty}^{2 \gamma}$ 
the following holds true.
There exists a diffeomorphism 
$\Psi:=\Psi(\omega)\colon \mathbb{T}^{\nu}\to\mathbb{T}^{\nu}$ 
of the form 
$\Theta\mapsto \widetilde{\Theta}:= \Theta+\alpha(\omega; \Theta)$,
with inverse 
$\Psi^{-1}=\Psi^{-1}(\omega)\colon \mathbb{T}^{\nu}\to\mathbb{T}^{\nu}$,
$\Theta=\widetilde{\Theta}+\widetilde{\alpha}(\widetilde{\Theta})$
such that
\begin{equation}\label{manhattan}
\Psi_* X=\Psi^{-1}\Big(
\omega-\mathtt{v}f(\Theta)
+(\omega-\mathtt{v}\,f(\Theta))\cdot 
\partial_{\Theta}\alpha\Big)
\cdot 
\frac{\partial}{\partial\widetilde{ \Theta}}
=\Big(\omega-\mathtt{v}\big(\mathtt{c}+\mathtt{r}_1\big)\Big)\cdot 
\frac{\partial}{\partial\widetilde{ \Theta}}
\end{equation}
for some $\mathtt{r}_1=\mathtt{r}_1(\omega)\in\mathbb{R}$
%diophantine vector $\mathtt{w}=\mathtt{w}(\omega)\in\mathbb{R}^{\nu}$ 
such that
\[
%\lvert \mathtt{w}-(\omega+\mathtt{v}\,\mathtt{c} )
%\rvert^{\gamma, \calO}_{\mathbb{R}^{\nu}}
\lvert\mathtt{r}_1\rvert^{\gamma, \calO_0}\leq \lVert 
V^+_{\geq 13} \rVert_{s_1}^{\gamma, \calO_0}
\stackrel{\eqref{iena40}, \eqref{IpotesiPiccolezzaIdeltaDP}}{\lesssim} 
\varepsilon^{17-2 b}\gamma^{-1}\,.
\]
 The assumptions of this theorem are satisfied if we consider $\gotp_1$ in \eqref{IpotesiPiccolezzaIdeltaDP} large enough 
 (more precisely $s_0+\gotp_1\geq s_1$ where $s_1$ is an index provided by the theorem) and $\varepsilon$ small enough (in particular such that $\varepsilon^{13}\le \eta_*$ where $\eta_*$ is provided by the theorem).
 Moreover the following estimates hold:
 \begin{equation}\label{stimaalpha}
\begin{aligned}
\lVert \alpha \rVert_s^{\gamma, \Omega_{\infty}^{2 \gamma}}
&\lesssim_s \gamma^{-1} \lVert V_{\geq 13}^+ 
\rVert_{s+2\tau+4}^{\gamma, \Omega_{\infty}^{2 \gamma}}\,,\qquad 
 \lVert \Delta_{12}\alpha   \rVert_{p} \lesssim_{p} 
\varepsilon \gamma^{-1}\|\Delta_{12}V^{+}_{\geq13}\|_{p+\mu}\,.
%(1+\lVert \mathfrak{I}_{\delta} \rVert_{p+\hat{\sigma}})\|i_1-i_2\|_{p+{\s}}\,,
\end{aligned}
\end{equation}
Similar for the inverse $\widetilde{\alpha}$.
The \eqref{manhattan} implies that the function $\alpha(\Theta)$
 solves the equation
 \begin{equation}\label{manhattan2}
(\omega-\mathtt{v}\,f(\Theta))\cdot 
\partial_{\Theta}\alpha=\mathtt{v} \big(f(\Theta)-\mathfrak{m}_1\big)\,,
\qquad \mathfrak{m}_1=\mathtt{c}+\mathtt{r}_1\,.
 \end{equation}
 Writing \eqref{manhattan2} for the components of the vector $\alpha(\Theta)$,
  we observe that  $\alpha_{i}(\Theta)/\ov{\jmath}_{i}$ 
  solves the same equation for any $i=1,\ldots,\nu$.
Hence there exists a smooth  $F : \mathbb{T}^{\nu}\to \mathbb{R}$
such that   $\alpha(\Theta)$ is of the form $-\mathtt{v} F(\Theta)$ where
\begin{equation}\label{manhattan4}
(\omega-\mathtt{v}\,f(\Theta))\cdot 
\partial_{\Theta}F(\Theta)+f(\Theta)=\mathfrak{m}_1\,.
\end{equation}
Using that
$\alpha(\Theta)+\widetilde{\alpha}\big(\Theta+\alpha(\Theta)\big)=0$,
 we deduce that the function $G(\widetilde{\Theta})$, defined
 by the relation $-\mathtt{v}G(\widetilde{\Theta})=\widetilde{\alpha}(\widetilde{\Theta}) $
 is such that
 \begin{equation}\label{manhattan3}
 F(\Theta)+G(\Theta-\mathtt{v} F(\Theta))=0\,.
 \end{equation}

We consider 
$\beta^{(2)}:=\beta^{(2)}(\omega; \vphi,x)
:=G(\vphi-\mathtt{v}x)$ 
and 
$\widetilde{\beta}^{(2)}:=\widetilde{\beta}^{(2)}(\omega; \vphi,y)
:=F(\vphi-\mathtt{v}y)$ .
By using \eqref{manhattan3} one can check that
the map $y\mapsto x=y+\widetilde{\beta}^{(2)}(\omega; \vphi,y)$ 
is the inverse of 
the diffeomorphism  
$x\mapsto y:= x+\beta^{(2)}
(\omega; \varphi,  x)$ of $\mathbb{T}^{\nu+1}$.
By \eqref{iena40}, \eqref{stimaalpha} and 
\eqref{IpotesiPiccolezzaIdeltaDP} we deduce the
\eqref{stimabeta2}.
By \eqref{manhattan4} we also deduce that
\begin{equation}\label{gotham}
\omega\cdot\pa_{\vphi}\widetilde{\beta}^{(2)}(\vphi,y)+(\mathtt{c}+V^{+}_{\geq13}(\vphi,y))
(1+\pa_{y}\widetilde{\beta}^{(2)}(\vphi,y))
=\mathfrak{m}_1\,,\qquad \forall\,(\vphi,y)\in \mathbb{T}^{\nu+1}\,.
\end{equation}
This implies the \eqref{vecfield2}.
%where
%\[
%\gamma:=-(\mathtt{w}\cdot \partial_{\widetilde{\Theta}})^{-1} 
%\Big(  f\circ \Psi^{-1}- \mathfrak{m}_1\Big)\,, 
%\quad \mathfrak{m}_1:=\int_{\mathbb{T}^{\nu}}  f\circ \Psi^{-1}\,d  \widetilde{\Theta}\,.
%\]
%
%Then the diffeomorphism $\Gamma\colon(\varphi, x)\mapsto 
%(\varphi, y),\,y:= x+\beta^{(2)}(\omega; \varphi+\mathtt{v} x)$ 
%of $\mathbb{T}^{\nu+1}$ transforms the vector field
%\eqref{vecfield1} into \eqref{vecfield2}.
%We define
%\[ 
%\mathfrak{m}_1=\mathtt{c}+\mathtt{r}\,,
%%\int_{\mathbb{T}^{\nu}}  
%%V^+_{\geq 13}\circ \Psi^{-1}\,d  \widetilde{\Theta}\,,
%\qquad \mathtt{r}= \int_{\mathbb{T}^{\nu}}
%V^+_{\geq 13}\circ \Psi^{-1}\,d  \widetilde{\Theta}\,.
%\]
Finally, by 
$\lvert \mathtt{r}_1 \rvert^{\gamma, \calO_0}\lesssim 
\lVert V^+_{\geq 13}  \rVert^{\gamma, \calO_0}_{s_0}
(1+\lVert \alpha \rVert^{\gamma, \Omega_{\infty}^{2 \gamma}}_{s_0+1})$ and 
\eqref{stimaalpha},
we get
%which implies, together with \eqref{stimaalpha}, 
the first bound in 
\eqref{vecfield3}.
 Indeed by Kirszbraun Theorem we can extend the function 
 $\mathtt{r}(\omega)$ to $\Omega_{\varepsilon}$ with the same bound on the norm.
 The bound on $\Delta_{12}\mathtt{r}_1$ in \eqref{vecfield3}
 can be deduced using Lemma $3.7$ in \cite{FGP} and the estimates on
$\Delta_{12}V^{+}_{\geq13}$ in  \eqref{iena40}.
 \end{proof}

\begin{proof}[{\bf Proof of Proposition \ref{proporduno}}]
We consider the map
\begin{equation}\label{MappaT2}
{\bf T}_2^{1}:=\left(
\begin{matrix}
\mathcal{T}_{\beta^{(2)}}^{1}(\vphi) & 0 \\
0 & \ov{\mathcal{T}_{\beta^{(2)}}^{1}(\vphi)}
\end{matrix}
\right)
\end{equation}
where $\mathcal{T}_{\beta^{(2)}}^{\tau}(\vphi) $, $\tau\in[0,1]$ is the flow given by \eqref{linprobpro}
with $f$ as in \eqref{sim1}
with $\beta\rightsquigarrow \beta^{(2)}$ given by Lemma \ref{drizzotutto}.
First of all by Lemma \ref{differenzaFlussi}
we can write
\[
\mathcal{T}^{1}_{\beta^{(2)}}=
\Pi_{S}^{\perp} \Phi_{\beta^{(2)}}^{\tau}\Pi_{S}^{\perp}\circ({\rm Id}+\mathcal{R})
\]
where $\Phi_{\beta^{(2)}}^{\tau}$ is the flow of \eqref{linprob}
 with $f\rightsquigarrow b\x$ with $b$ as in \eqref{sim1}
with $\beta\rightsquigarrow \beta^{(2)}$
and where $\mathcal{R}$ is a finite rank operator of the form \eqref{FiniteDimFormDP10001}
satisfying \eqref{giallo2}.
Therefore (see Lemma \ref{CoroDPdiffeo})
we have that
the flow $\mathcal{T}_{\beta^{(2)}}^{\tau}$ is well posed and satisfies estimates like
\eqref{flow2}-\eqref{flow222} 
This implies, using also \eqref{stimabeta2}, the \eqref{stimaMapp1}.
Notice that, by estimates \eqref{stimabeta2}, we could say that
$\beta^{(2)}$ is in $S_{2}^{0}$. Actually all the homogeneous
terms in the expansion \eqref{espandoenonmifermo} of $\beta^{(2)}$ are zero.

We now conjugate the operator 
 $\mathcal{L}_{3, *}$ in \eqref{elle3star} 
  with the map ${\bf T}^1_2$.
We follow the strategy of  Lemma \ref{lemmapre}.
By Lemma \ref{georgiaLem} we have that the conjugate of  $\mathcal{L}_{3, *}$ 
is given by
\[
\Pi_{S}^{\perp}\widetilde{{\bf T}}_2^{1}\mathcal{L}_{3,*}(\widetilde{{\bf T}}_2^{1})^{-1}\Pi_{S}^{\perp}\,,
\qquad
\widetilde{{\bf T}}_2^{\tau}:=\left(
\begin{matrix}
{\Phi^{\tau}_{\beta^{(2)}}}(\vphi) & 0 \\
0 & \ov{\Phi_{\beta^{(2)}}^{\tau}(\vphi)}
\end{matrix}
\right)\,,\qquad \tau\in[0,1]\,,
\]
up to finite rank operators belonging to $\mathfrak{L}_{\rho, p}^{2}\otimes\mathcal{M}_2(\mathbb{C})$. 
%The latter claim follows by estimates \eqref{giallo2Rank}, 
%Lemma \ref{lemmapre} to estimates the symbols of 
%the operator $\mathcal{L}_{3,*}$ in 
%\eqref{elle3star} and \eqref{stimabeta2}.
Then we shall apply Theorem \ref{EgorovQuantitativo}
and Lemmata \ref{EgorovTempo}, \ref{preparailsugo}.

%
%We use the following relations:
%\begin{align*}
%\tau \beta^{(2)}(\varphi, y)&=
%-\tilde{\beta}^{(2)}(\tau, \varphi, y+\tau \beta^{(2)}(\varphi, y))\\
%\Big(1+\tau \beta_y^{(2)}(\varphi, y) \Big) 
%\Big(1+\tilde{\beta}_y^{(2)}(\tau, \varphi, y+\tau \beta^{(2)}(\varphi, y)) \Big)&=1\\[2mm]
%\omega\cdot \partial_{\varphi} 
%\tilde{\beta}^{(2)}(\tau, \varphi, y+\tau \beta^{(2)}(\varphi, y))
%&=-\tau\,\omega\cdot \partial_{\varphi} \beta^{(2)}(\varphi, y) 
%\Big(1+\tilde{\beta}_y^{(2)}(\tau, \varphi, y+\tau\,\beta^{(2)}(\varphi, y)) \Big)\\[2mm]
% \omega\cdot \partial_{\varphi} 
%\tilde{\beta}_y^{(2)}(\tau, \varphi, y+\tau \beta^{(2)}(\varphi, y))\, 
%(1+\tau \tilde{\beta}_y (\varphi, y))&=
%-\tau\,\omega\cdot \partial_{\varphi} \beta_y^{(2)}(\varphi, y) 
%\Big(1+\tilde{\beta}_y^{(2)}(\tau, \varphi, y+\tau\,\beta^{(2)}(\varphi, y)) \Big) \\
%&\quad-\tau\,\omega\cdot \partial_{\varphi} \beta^{(2)}(\varphi, y)\,
%\tilde{\beta}_{y y}^{(2)}(\tau, \varphi, y+\tau\,\beta^{(2)}(\varphi, y))\,(1+\tau\,\beta_y^{(2)}(\varphi, y)) \\[2mm]
%(1+\tilde{\beta}^{(2)}_y (\tau, \varphi, y+\tau\beta^{(2)}(\varphi, y)))
%\tau \beta_{yy}^{(2)}&=
%-\tilde{\beta}^{(2)}_{yy}(\tau, \varphi, y+\tau \beta^{(2)}(\varphi, y)) 
%\,(1+\tau \beta_y^{(2)}(\varphi, y))^2\,.
%\end{align*}
Recall that
the map
$\Phi_{\beta^{(2)}}^{\tau}$
 has the explicit 
expressions in \eqref{ignobelSymp}.
By a direct computation we have that
\begin{equation}\label{fede1}
\begin{aligned}
\Phi_{\beta^{(2)}}^1\,\, \omega \cdot \partial_{\varphi}\,\,\, 
\big(\Phi_{\beta^{(2)}}^1 \big)^{-1}  &=
\omega\cdot \partial_{\varphi}\, 
+\frac{\omega\cdot \partial_{\varphi} \tilde{\beta}_y^{(2)}
(1, \varphi, y+\beta^{(2)}(\varphi, y))}{2\,(1+\tilde{\beta}_y^{(2)} 
(1, \varphi, y+ \beta^{(2)}(\varphi, y)) )}\, 
%\\&\quad
+ \omega\cdot \partial_{\varphi} \tilde{\beta}^{(2)}
(1, \varphi, y+\,\beta^{(2)}(\varphi, y))\,\,\partial_y \,,\\[2mm]
\Phi_{\beta^{(2)}}^1\,\,V^{+}(\varphi, y)\,\, \partial_y\,\,\, 
\big(\Phi_{\beta^{(2)}}^1 \big)^{-1}  &=
\frac{V^{+}(\varphi, y+ \beta^{(2)}(\varphi, y))\,\tilde{\beta}_{yy}^{(2)}
(1, \varphi, y+\,\beta^{(2)}(\varphi, y))\, }{2\, 
(1+\tilde{\beta}_y^{(2)}(1, \varphi, y+ \beta^{(2)}(\varphi, y)))}\,\\
&\quad+V^{+}(\varphi, y+ \beta^{(2)}(\varphi, y))\,\,
(1+\tilde{\beta}^{(2)}_y(1, \varphi, y+\tau\,\beta^{(2)}(\varphi, y)))\,\partial_y \,.
\end{aligned}
\end{equation}
Moreover,
 by Lemma \ref{lem:escobarhomo} % and Remark \ref{standard/Weyl}
 we can write (up to smoothing reminders in $\mathfrak{L}_{\rho,p}^{2}$)
 %satisfying estimates 
 %\eqref{passo2}, \eqref{passo3})
 \begin{equation}\label{fede3}
 \opw(\ii V^{+}(\vphi)\x)=\ii V^{+}(\vphi,x)\pa_{x}+\frac{1}{2}\ii V_{x}^{+}(\vphi,x)\,.
 \end{equation}
 Now we consider the unbounded part of the operator $\mathcal{L}_{3, *}$, namely
  \begin{equation}\label{fede4}
  \begin{aligned}
 \omega\cdot\pa_{\vphi}&+\ii \opw(V^{+}(\vphi,x)\x+(1+a_{*}^{(3)}(\vphi,x))|\x|^{\frac{1}{2}}
 +c^{(3)}_{*}(\vphi,x,\x))\\
 \stackrel{\eqref{fede3}}{=}
 & \omega\cdot\pa_{\vphi}+\ii V^{+}(\vphi,x)\pa_{x}+\frac{\ii}{2}V_{x}^{+}(\vphi,x)+
  \opw(\ii(1+a_{*}^{(3)}(\vphi,x))|\x|^{\frac{1}{2}}
 +c^{(3)}_{*}(\vphi,x,\x))\,.
 \end{aligned}
 \end{equation}
Using the explicit computation in  \eqref{fede1} and Theorem \ref{EgorovQuantitativo},
we have that the conjugate of \eqref{fede4} is,  
up to smoothing operators in $\mathfrak{L}^{2}_{\rho,p}$,
  \[
 \omega\cdot\pa_{\vphi}+\ii \opw(V^{++}(\vphi,x)\x+(1+a^{(4)}(\vphi,x))|\x|^{\frac{1}{2}}
 +c^{(4)}(\vphi,x,\x)+r_0((\vphi,x)))\,,
 \]
 where
 \begin{align}
 V^{++}(\vphi,x)&:=\omega\cdot\pa_{\vphi}\tilde{\beta}^{(2)}(\vphi,x)+V^{+}(\vphi,x)(1+\tilde\beta_{y}^{(2)}(\vphi,x))_{|x=y+\beta^{(2)}(\vphi,y)}\,,\label{fede5}\\
 r_0(\vphi,x)&:=
 \frac{\pa_{y}( V^{++}(\vphi,x)_{|x=y+\beta^{(2)}(\vphi,y)})}{2(1+\tilde{\beta}_y^{(2)}
 (1, \varphi, y+ \beta^{(2)}(\varphi, y)))}\,,\nonumber
 %\label{fede6}
 \\
 (1+a^{(4)}(\vphi,x))&:=(1+a^{(3)}_{*}(\vphi,x+\beta^{(2)}(x)))
\big|(1+\tilde{\beta}^{(2)}_{y}(1,y))_{|y=x+\beta^{(2)}(x)}\big|^{\frac{1}{2}}\,,
\nonumber
%\label{fede7}
 \end{align}
 and $c^{(4)}(\vphi,x,\x)$ is some symbol in $S_2^{-1/2}$.
 Recall that the function $V^{+}$ constructed in Lemma \ref{lemmapre}
 is in $S_{\mathtt{v}}$ (see \eqref{funzmom}). Therefore,
by using Lemma \ref{drizzotutto} (see eq. \eqref{gotham}), we deduce that
$V^{++}=\mathfrak{m}_{1}$ and  $\mathfrak{m}_{1}$  
has the form  \eqref{mer} thanks to \eqref{vecfield1}, \eqref{manhattan2}.
%satisfies \eqref{mer}
% thanks to the  \eqref{iena40}, \eqref{stimaalpha} 
% and \eqref{IpotesiPiccolezzaIdeltaDP}.
As a consequence $r_0(\vphi,x)$ %in \eqref{fede6} 
is identically zero.
By item $(ii)$ of Theorem \ref{EgorovQuantitativo} we have that $a^{(4)}\in S_2^0$. 
The discussion above implies that the conjugate of 
$\mathcal{L}_{3,*}$ \eqref{elle3star} has the form \eqref{elle4} with
$\mathcal{R}^{(4)}\in \mathfrak{L}^{2}_{\rho,p}\otimes\mathcal{M}_2(\mathbb{C})$.
%that  the symbol 
%$c^{(4)}$ and the remainder $\mathcal{R}^{(4)}$ admit the expansions in 
%\eqref{espandoenonmifermo}, \eqref{espandoenonmifermo2} with $k=2$. In particular
%the estimates \eqref{AVBDEF}, %\eqref{AVB2DEF}, 
%\eqref{AVBDEF2} %, \eqref{AVB2DEF2}
%with $k=2$ on the symbol $c^{(4)}$ and the remainder $\mathcal{R}^{(4)}$ follows
%again by the estimates on symbols and remainders 
%given by Theorem \ref{EgorovQuantitativo}, Lemma \ref{preparailsugo}
%using \eqref{stimabeta2},  \eqref{IpotesiPiccolezzaIdeltaDP}
%and   \eqref{AVBDEF} %, \eqref{AVB2DEF} 
%with $k=1$ on the symbols and operators of 
%$\mathcal{L}_{3,*}$.
%Hence $c^{(4)}\in S^{-1/2}_{2}$, 
%$\mathcal{R}^{(4)}\in \mathcal{L}^{2}_{\rho,p}$.
The bounds in \eqref{merstima} with $i=1$ on the constant $\mathtt{r}_1$
follows by \eqref{vecfield3}, \eqref{iena40} %, \eqref{iena41} 
and 
using \eqref{IpotesiPiccolezzaIdeltaDP}.
To conclude we set (recall \eqref{mappapasso1}, \eqref{MappaT2})
\begin{equation}\label{mappaPhi11}
{\bf \Phi}_1:={\bf T}^{1}_2\circ{\bf T}^{1}_1\,.
\end{equation}
Lemma \ref{CoroDPdiffeo} with $\beta\rightsquigarrow\beta^{(2)}$, estimate \eqref{stimaMapp1} and 
 \eqref{stimaMapp1bis} imply the \eqref{stimaMapp1}.
%applying Lemma \ref{CoroDPdiffeo}
%and using \eqref{genPre1} and \eqref{stimabeta2}.
To conclude the proof of Proposition \ref{proporduno}
it remains to show that $\mathcal{L}_{4}\in \mathfrak{S}_1$
and $\Phi_1\in \mathfrak{T}_1$.

First of all recall that 
the functions $\beta^{(i)}$, $i=1,2$ are real and belongs to $S_{\mathtt{v}}$ 
(see \eqref{funzmom}).
The map $\Phi_1\in\mathfrak{T}_1$ because $\mathbf{T}^{1}_1, 
\mathbf{T}^{1}_2\in\mathfrak{T}_1$. %of Proposition \eqref{proporduno} 

Hence $\mathcal{L}_4$
defined in \eqref{elle4} is Hamiltonian and $x$-translation invariant.
The symbol
$d^{(4)}(\vphi,x,\x)$ satisfies 
\eqref{retoreMat}, \eqref{selfSimbo2}, \eqref{space3Xinv}
by construction. Hence, by Lemmata 
\ref{lem:momentoSimbolo}, \ref{realHamMtrixSimbo},
we deduce that $\mathcal{L}_{4}$ is in $\mathfrak{S}_1$.
%the operator 
%\[
%\opw\left(\begin{matrix} d^{(4)}(\vphi,x,\x) & 0 \\
%0 & \ov{d^{(4)}(\vphi,x,-\x)}  \end{matrix}\right)\,,
%\]
%is Hamiltonian and $x$-translation invariant, i.e. 
%$\mathcal{L}_{4}$ is in $\mathfrak{S}_1$.
\end{proof}

\subsection{Integrability at order 1/2}\label{redord12}
The aim of this section is to eliminate the $(\vphi,x)$ dependence from the 
symbol $a^{(4)}$ in \eqref{diagpezzo2}
appearing in the operator $\mathcal{L}_{4}$ in \eqref{elle4}.

\begin{proposition}\label{propordmezzo}
Let $\rho\geq 3$, $p\geq s_0$ and assume that 
\eqref{IpotesiPiccolezzaIdeltaDP} holds. 
Then there exist $\mu=\mu(\nu)$ and 
$\mathfrak{m}_{1/2}\colon \Omega_{\varepsilon}\to \mathbb{R}$ of the form
\eqref{mer} with $i=1/2$
such that, for all $\omega\in \Omega^{2\gamma}_{\infty}$
(see \eqref{omegaduegamma}), 
there exists a map 
${\bf \Phi}_2(\omega)\colon H^s_{S^{\perp}}(\mathbb{T})\times 
H^s_{S^{\perp}}(\mathbb{T})\to 
H^s_{S^{\perp}}(\mathbb{T})\times H^s_{S^{\perp}}(\mathbb{T})$ such that
(recall \eqref{elle4})
\begin{equation}\label{elle5}
\mathcal{L}_5:={\bf \Phi}_2 \,\mathcal{L}_4\,{\bf \Phi}_2^{-1}=
\Pi_{S}^{\perp}\Big(\omega\cdot\pa_{\vphi}+
\opw\left(\begin{matrix} d^{(5)}(\vphi,x,\x) & 0 \\
0 & \ov{d^{(5)}(\vphi,x,-\x)}  \end{matrix}\right)+\mathcal{R}^{(5)}\Big)
\end{equation}
where $\mathcal{R}^{(5)}\in\mathfrak{L}^{3}_{\rho,p}\otimes\mathcal{M}_2(\mathbb{C})$,
the symbol $d^{(5)}(\vphi,x,\x)$ has the form (recall \eqref{diagpezzo})
\begin{equation}\label{diagpezzo2or12}
d^{(5)}(\vphi,x,\x):=
\ii \mathfrak{m}_1\,\x+\ii (1+\mathfrak{m}_\frac{1}{2})|\x|^{\frac{1}{2}}
+\ii b^{(5)}(\vphi,x)\sign(\x)
+c^{(5)}(\vphi,x,\x)\,,
\end{equation}
with $b^{(5)}\in S_3^{0}$ independent of $\x$, 
$c^{(5)}\in S_3^{-\frac{1}{2}}$.
%In particular the symbols $b^{(5)}$, $c^{(5)}$ satisfy the estimates \eqref{AVBDEF}, 
 %%\eqref{AVB2DEF} 
 %with $k=3$ and the remainder $\mathcal{R}^{(5)}$ satisfies
 %\eqref{AVBDEF2}
% %, \eqref{AVB2DEF2} 
 %with $k=3$.
Moreover
\begin{equation}\label{stimaMapp2}
\|({\bf \Phi}_2)^{\pm1}u\|_{s}^{\gamma,\Omega^{2\gamma}_{\infty}}
\lesssim_{s}\|u\|_{s}^{\gamma, \Omega^{2\gamma}_{\infty}}+
\gamma^{-2}\big(\e^{13}+\e\|\mathfrak{I}_{\delta}\|^{\gamma, \calO_0}_{s+\mu}\big)
\|u\|_{s_0}^{\gamma, \Omega^{2\gamma}_{\infty}}\,.
\end{equation}
Finally
the operator  $\mathcal{L}_5$
%defined in \eqref{elle5} 
is in $\mathfrak{S}_1$ and the map ${\bf \Phi}_2$
%of Proposition \ref{propordmezzo} 
is in $\mathfrak{T}_1$ (recall Def. \ref{compaMulti}, 
\ref{goodMulti}). 
\end{proposition}

%Moreover we shall prove the following.
%\begin{lemma}\label{GoodAmmissi5}
%The operator  $\mathcal{L}_5$
%defined in \eqref{elle5} is in $\mathfrak{S}_1$ and the map ${\bf \Phi}_2$
%of Proposition \ref{propordmezzo} is in $\mathfrak{T}_1$ (recall Def. \ref{compaMulti}, 
%\ref{goodMulti}). 
%%$(ii)$ If $\mathcal{L}_2$ in \eqref{elle3} is in $\mathfrak{S}_2$ then  
%%${\bf \Phi}_1\in \mathfrak{T}_2$ and $\mathcal{L}_4\in \mathfrak{S}_2$.
%% $(iii)$ If $\mathcal{L}_3\in \mathfrak{S}_3$ then  
%%${\bf \Phi}_1\in \mathfrak{T}_3$ and $\mathcal{L}_4\in \mathfrak{S}_3$.
%\end{lemma}

Following the strategy used in section 
 \ref{redord1} 
 we divide the proof of 
 of Proposition \ref{propordmezzo} into two steps.

\subsubsection{Preliminary steps}\label{preliminaremezzo}

\begin{lemma}{\bf ({Preliminary steps}).}\label{lemmapremezzo}
There exists a function $\beta^{(3)}(\varphi, x)$ of the form
\begin{equation}\label{genPre1mezzo}
\beta^{(3)}(\vphi,x)=\sum_{i=1}^{10}\e^{i}\beta^{(3)}_i(\vphi,x)\,,
\end{equation}
where $\beta^{(3)}_i$ is a real valued 
$i$-homogeneous symbol as in \eqref{simboOMO2DEF}
for $i=1,\ldots,10$ independent of $\x\in \mathbb{R}$ 
satisfying

\begin{equation}\label{stimaSoloe2}
\|\beta^{(3)}\|_{s}^{\gamma,\mathcal{O}_0}\lesssim_{s}\e\,,
\end{equation}
such that the following holds.
If 
\begin{equation}\label{mappapasso3}
{\bf T}_3^{\tau}:=\left(
\begin{matrix}
\mathcal{T}_{\beta^{(3)}}^{\tau}(\vphi) & 0 \\
0 & \ov{\mathcal{T}_{\beta^{(3)}}^{\tau}(\vphi)}
\end{matrix}
\right)\,,\qquad \tau\in[0,1]\,,
\end{equation}
where $\mathcal{T}_{\beta^{(3)}}^{\tau}$ is the flow of \eqref{linprobpro}
with generator as in \eqref{sim2} with $\beta\rightsquigarrow \beta^{(3)}$, 
then
\begin{equation}\label{elle4star}
\mathcal{L}_{4, *}:={\bf T}^1_3 \,\mathcal{L}_4 \, ({\bf T}^1_3)^{-1}=
\Pi_{S}^{\perp}\Big(\omega\cdot\pa_{\vphi}+
\opw\left(\begin{matrix} d_*^{(4)}(\vphi,x,\x) & 0 \\
0 & \ov{d_*^{(4)}(\vphi,x,-\x)}  \end{matrix}\right)+\mathcal{R}_*^{(4)}\Big)
\end{equation}
where 
$\mathcal{R}_*^{(4)}\in\mathfrak{L}_{\rho, p}^{2}\otimes\mathcal{M}_2(\mathbb{C})$ 
and
\begin{align}
d^{(4)}_*(\vphi,x,\x)&:=
\ii \mathfrak{m}_1\x
+\ii (1+a^{(4)}_*(\vphi,x))|\x|^{\frac{1}{2}}
+\ii b^{(4)}_{*}(\vphi,x)\sign(\x)+c^{(4)}_*(\vphi,x,\x)\,,
\label{diagpezzoord12}\\
a^{(4)}_*(\vphi,x)&:=\e^{2}m_{\frac{1}{2}}
+\sum_{k=2}^{5}\e^{2k}m_{\frac{1}{2}}^{(2k)}+
a^{(4,*)}_{\geq11}(\vphi,x)\,,
\label{termA1finale12}
\end{align}
where $\mathfrak{m}_1$ is in \eqref{mer}, $m_{1/2}$, 
$m_{1}^{(2k)}$ are respectively $2$, $2k$-homogenenous and   integrable (according to Def. \ref{integroSimbo}).
Moreover one has
$c^{(4)}_*\in S_2^{-1/2}$ and $b^{(4)}_*\in S_2^{0}$
(recall Def. \ref{espOmogenea})
independent of $\x\in \mathbb{R}$ and 
\begin{align}
\|a^{(4,*)}_{\geq11}\|_{s}^{\gamma,\Omega_{\infty}^{2\gamma}}&\lesssim_s \gamma^{-1}
(\e^{13}+\e\|\mathfrak{I}_{\delta}\|^{\gamma,\calO_0}_{s+\mu})\,,
\label{iena40a4}\\
\|\Delta_{12}a^{(4,*)}_{\geq11}\|_{p}
&\lesssim_p \gamma^{-1}
\e(1+\|\mathfrak{I}_{\delta}\|_{p+\mu})
\|i_{1}-i_{2}\|_{p+\mu}\,.\label{iena41a4}
\end{align}
Finally, the operator $\mathcal{L}_{4, *}$ is in $\mathfrak{S}_1$, the map
$\mathcal{T}_{\beta^{(3)}}^{\tau}$ is in $\mathfrak{T}_1$ and satisfies
\begin{equation}\label{stimaMapp2bis}
\|(\mathcal{T}_{\beta^{(3)}})^{\pm1}u\|_{s}^{\gamma,\Omega^{2\gamma}_{\infty}}
\lesssim_{s}\|u\|_{s}^{\gamma, \Omega^{2\gamma}_{\infty}}+
\e
\|u\|_{s_0}^{\gamma, \Omega^{2\gamma}_{\infty}}\,.
\end{equation}
%and in $\mathfrak{T}_3$ (recall Def. \ref{goodMulti}). Moreover if the function 
%$v_{I}(\vphi,x)$ in \eqref{vsegnatoDP} is even in $x$, then the map
%$\mathcal{T}_{\beta^{(1)}}^{\tau}$ is in $\mathfrak{T}_2$. 
\end{lemma}

\begin{proof}
By Lemma \ref{CoroDPdiffeo2} we have that, if $\beta^{(3)}$ is real valued
and satisfies \eqref{stimaSoloe2}, then
the flow $\mathcal{T}^{\tau}_{\beta^{(3)}}$
is well posed and satisfies bounds like  \eqref{flow2}-\eqref{flow222}.
By Lemma \ref{differenzaFlussi} we have
\[
\mathcal{T}^{1}_{\beta^{(3)}}=
\Pi_{S}^{\perp} \Phi_{\beta^{(3)}}^{\tau}\Pi_{S}^{\perp}\circ({\rm Id}+\mathcal{R})
\]
where $\Phi_{\beta^{(3)}}^{\tau}$ is the flow of \eqref{linprob}
 with $f$  as in \eqref{sim2}
with $\beta\rightsquigarrow \beta^{(3)}$
and where $\mathcal{R}$ is a finite rank operator of the form \eqref{FiniteDimFormDP10001}
satisfying \eqref{giallo2}.
By Lemma \ref{georgiaLem} we have that the conjugate of 
$\mathcal{L}_4$ in \eqref{elle4} under the map ${\bf T}_{3}^{1}$ is given by 
\[
\Pi_{S}^{\perp}\widetilde{{\bf T}}_3^{1}\mathcal{L}_{4}(\widetilde{{\bf T}}_2^{1})^{-1}\Pi_{S}^{\perp}\,,
\qquad
\widetilde{{\bf T}}_3^{\tau}:=\left(
\begin{matrix}
{\Phi^{\tau}_{\beta^{(3)}}}(\vphi) & 0 \\
0 & \ov{\Phi_{\beta^{(3)}}^{\tau}(\vphi)}
\end{matrix}
\right)\,,\qquad \tau\in[0,1]\,,
\]
up to finite rank operators belonging to 
$\mathfrak{L}_{\rho, p}^{2}\otimes\mathcal{M}_2(\mathbb{C})$. 
%The latter claim follows by estimates \eqref{giallo2Rank}, 
%Proposition \ref{proporduno} to estimates the 
%symbols of the operator $\mathcal{L}_{4}$ in 
%\eqref{elle4} and \eqref{stimaSoloe2}.
%$\Pi_{S}^{\perp} \Phi_{\beta^{(3)}}^{\tau}\mathcal{L}_4(\Phi_{\beta^{(3)}}^{\tau})^{-1}\Pi_{S}^{\perp}$.
We apply Lemmata \ref{flusso12basso}, \ref{flusso12basso2}.
We start by studying the time contribution. By Lemma \ref{flusso12basso2}
(see in particular item $(ii)$)
we have that (see \eqref{coniugato1002}, \eqref{castello1tempo})
\begin{equation}\label{tempomezzo}
\begin{aligned}
\Phi^{1}_{\beta^{(3)}}&\omega\cdot\pa_{\vphi}(\Phi^{1}_{\beta^{(3)}})^{-1}=
\omega\cdot\pa_{\vphi}
-\ii\opw\Big( \omega\cdot\pa_{\vphi}\beta^{(3)}|\x|^{\frac{1}{2}}
+\frac{1}{2}\{ \beta^{(3)}|\x|^{\frac{1}{2}}, \omega\cdot\pa_{\vphi}\beta^{(3)}|\x|^{\frac{1}{2}}
\}+r_{1}
\Big)\\
&%\stackrel{\eqref{poissonSimbo}}{=}
=\omega\cdot\pa_{\vphi}
-\ii\opw\Big( \omega\cdot\pa_{\vphi}\beta^{(3)}|\x|^{\frac{1}{2}}
+\frac{1}{4}(\beta^{(3)}\omega\cdot\pa_{\vphi}\beta^{(3)}_x
-\beta^{(3)}_x\omega\cdot\pa_{\vphi}\beta^{(3)})\sign(\x)
%\{\ii \beta^{(3)}|\x|^{\frac{1}{2}},\ii \omega\cdot\pa_{\vphi}\beta^{(3)}|\x|^{\frac{1}{2}}\}
+r_{1}
\Big)\,,
\end{aligned}
\end{equation}
up to smoothing remainder $\mathcal{R}_1$ in the class $\mathfrak{L}^{2}_{\rho,p}$
and where $r_1$ is some symbol in $S_2^{-3/2}$.
We remark that the symbol of the operator above is not singular at $\x=0$.
Indeed, 
recalling the notation \eqref{notasignXi},
 we are writing  
$\sign(\x)$ instead of  $\chi(\x)\sign(\x)$ (with $\chi$ as in \eqref{cutofffunct}).
In the following we will use systematically such notation.
%By Taylor expanding the symbols above in $\e$, by using \eqref{stimaSoloe2}
%and estimates \eqref{crawford100tempo}-\eqref{jamal2100tempo}
%we actually deduce that $\mathcal{R}_1\in \mathcal{L}^{2}_{\rho,p}$, 
%$r_1\in S^{-3/2}_{2}$ and that the symbol of order zero appearing in \eqref{tempomezzo}
%is in $S_2^{0}$ (recall Def. \ref{espOmogenea}).

We now consider the contribution coming from the conjugation of the spatial operator
by using Lemma \ref{flusso12basso} applied to the pseudo differential operator
with symbol in \eqref{diagpezzo2}.
First we notice that (see \eqref{coniugato1000}, \eqref{castello1})
\begin{equation}\label{ord112}
\begin{aligned}
\Phi^{1}_{\beta^{(3)}}\opw(\ii \mathfrak{m}_1\x)(\Phi^{1}_{\beta^{(3)}})^{-1}&=
\ii\opw\big(
 \mathfrak{m}_1\x+\{\beta^{(3)}|\x|^{\frac{1}{2}},\mathfrak{m}_{1}\x\}
\big)\\
&+\ii\opw\big(\frac{1}{2}
\{\beta^{(3)}|\x|^{\frac{1}{2}},\{ \beta^{(3)}|\x|^{\frac{1}{2}},\mathfrak{m}_{1}\x\}\}+r_2\big)\\
&=
\ii\opw\big(
 \mathfrak{m}_1\x-\mathfrak{m}_1\beta^{(3)}_x|\x|^{\frac{1}{2}}
-\frac{\mathfrak{m}_1\sign(\x)}{4}(\beta^{(3)}\beta^{(3)}_{xx}-(\beta^{(3)}_{x})^{2})+r_2
\big)
\end{aligned}
\end{equation}
up to smoothing remainder $\mathcal{R}_2\in \mathfrak{L}^{2}_{\rho,p}$
and $r_2\in S_2^{-1/2}$.
%By Taylor expanding the symbols above in $\e$, by using \eqref{stimaSoloe2}
%and estimates \eqref{crawford100}-\eqref{jamal2100}
%we actually deduce that $\mathcal{R}_2\in \mathcal{L}^{2}_{\rho,p}$, 
%$r_2\in S^{-1/2}_{2}$ and that the symbol of order zero appearing in \eqref{ord112}
%is in $S_2^{0}$ (recall Def. \ref{espOmogenea}).
Moreover we have
\begin{equation}\label{ord1122}
\begin{aligned}
\Phi^{1}_{\beta^{(3)}}&\opw(\ii (1+a^{(4)})|\x|^{\frac{1}{2}})(\Phi^{1}_{\beta^{(3)}})^{-1}=
\ii\opw\big((1+a^{(4)})|\x|^{\frac{1}{2}}+\{\beta^{(3)}|\x|^{\frac{1}{2}},(1+a^{(4)})|\x|^{\frac{1}{2}}\}
+r_3
\big)\\
&=\ii\opw\big(
(1+a^{(4)})|\x|^{\frac{1}{2}}+
\frac{\sign(\x)}{2}(\beta^{(3)}a^{(4)}_{x}-(1+a^{(4)})\beta^{(3)}_x)+r_3
\big)
\end{aligned}
\end{equation}
up to smoothing remainder $\mathcal{R}_3\in \mathfrak{L}^{2}_{\rho,p}$
and where $r_3 \in S_2^{-1/2}$.
%By Taylor expanding the symbols above in $\e$, by using \eqref{stimaSoloe2}
%and estimates \eqref{crawford100}-\eqref{jamal2100}
%we actually deduce that $\mathcal{R}_3\in \mathcal{L}^{2}_{\rho,p}$, 
%$r_3\in S^{-1/2}_{2}$ and that the symbol of order 
%zero appearing in \eqref{ord1122}
%is in $S_2^{0}$ (recall Def. \ref{espOmogenea}).
Similarly (using Lemmata \ref{flusso12basso}, 
\ref{CoroDPdiffeo200}) we conclude
\begin{equation}\label{ord1123}
\begin{aligned}
\Phi^{1}_{\beta^{(3)}}&\big[\opw(c^{(4)})+\mathcal{R}^{(4)}\big]
(\Phi^{1}_{\beta^{(3)}})^{-1}=
\opw\big(c^{(4)}+r_4\big)
\end{aligned}
\end{equation}
up to smoothing remainder $\mathcal{R}_4\in \mathfrak{L}^{2}_{\rho,p}$
and where $r_4 \in S_2^{-1}$.
By collecting the \eqref{tempomezzo}-\eqref{ord1123}
we conclude that the conjugate of  $\mathcal{L}_4$ in \eqref{elle4} under the map in \eqref{mappapasso3} has the form \eqref{elle4star}
with symbol $d^{(4)}_*(\vphi,x,\x)$ as in \eqref{diagpezzoord12} where
$b^{(4)}_*\in S_{2}^{0}$, $c^{(4)}_*\in S^{-1/2}_2$ and
\begin{align*}
a^{(4)}_{*}&=a^{(4)}-\omega\cdot\pa_{\vphi}\beta^{(3)}
-\mathfrak{m}_1\beta^{(3)}_{x}\,,
%\label{alive1}
\\
b^{(4)}_{*}&=-\frac{1}{4}\big(
\beta^{(3)}\omega\cdot\pa_{\vphi}\beta^{(3)}_x-\beta^{(3)}_x\omega\cdot\pa_{\vphi}\beta^{(3)}
\big)-\frac{\mathfrak{m}_1}{4}\big( \beta^{(3)}\beta^{(3)}_{xx}-(\beta^{(3)}_x)^{2}\big)%\label{alive2}
%\\&
+\frac{1}{2}\big(
\beta^{(3)}a^{(4)}_x-(1+a^{(4)})\beta^{(3)}_x
\big)\,.\nonumber
\end{align*}
We shall prove that it is possible to choose $\beta^{(3)}$ in \eqref{genPre1mezzo}
in such a way that
the symbol $a^{(4)}_{*}$ %in \eqref{alive1} 
has the form \eqref{termA1finale12}.
Recall that $a^{(4)}$ is a symbol (independent of $\x\in \mathbb{R}$)
in the class $S_2^0$ and hence it
admits and expansion as \eqref{espandoenonmifermo} with $k=2$
for some $i$-homogeneous symbols $a^{(4)}_i$ and a non-homogeneous symbol
$\tilde{a}^{(4)}$. Therefore, recalling the expansion
of $\mathfrak{m}_1$ in \eqref{mer}
and of 
$\omega$ in \eqref{FreqAmplMapDP}, 
we have
\begin{equation}\label{ventrale}
\begin{aligned}
a^{(4)}_{*}&=\sum_{i=1}^{10}\e^{i}a_{i}^{(4)}
-\sum_{i=1}^{10}\e^{i}\bar{\omega}\cdot\pa_{\vphi}\beta^{(3)}_{i}
-\sum_{i=1}^{10}\e^{i}\left(
\e^{2}\mathbb{A}\x+\sum_{k=2}^{6}\e^{2k}\mathtt{b}_{k}(\x,\ldots,\x)
\right)\cdot\pa_{\vphi}\beta^{(3)}_{i}\\
&-\left(\e^{2}m_1+\sum_{k=2}^{6}\e^{2k}m_{1}^{2k}+\mathtt{r}\right)
\sum_{i=1}^{10}\e^{i}(\beta^{(3)}_i)_x+\tilde{a}^{(4)}\,.
\end{aligned}
\end{equation}
This implies that the symbol $a^{(4)}_*$ has an expansion as in 
 \eqref{espandoenonmifermo} with $k=2$
 for some non-homogeneous symbol $\tilde{a}^{(4)}_{*}$ satisfying 
 \eqref{AVBDEF}, 
 %\eqref{AVB2DEF}
 with $k=2$, and some $i$-homogeneous symbols $a_{i,*}^{(4)}$, $i=1,\ldots,10$.
 In particular we have
 \begin{align}
 a^{(4)}_{1,*}&:= a^{(4)}_{1}-\bar{\omega}\cdot\pa_{\vphi}\beta^{(3)}_1\,,\label{astar1}\\
 a^{(4)}_{2,*}&:=a^{(4)}_{2}-\bar{\omega}\cdot\pa_{\vphi}\beta^{(3)}_2\,,\label{astar2}\\
   a^{(4)}_{3,*}&:=\mathtt{F}_{3}-\bar{\omega}\cdot\pa_{\vphi}\beta^{(3)}_3\,,
   \quad \mathtt{F}_3:=a^{(4)}_{3}+\mathtt{f}_{3}\,,
  \quad\mathtt{f}_{3}:=-\mathbb{A}\x\cdot\pa_{\vphi}\beta^{(3)}_1
  -{m}_1(\beta^{(3)}_1)_{x}
    \,,\label{astar3}\\
   a^{(4)}_{j,*}&:=\mathtt{F}_j-\bar{\omega}\cdot\pa_{\vphi}\beta^{(3)}_j\,,
  \quad\mathtt{F}_j:=+ a^{(4)}_{j}+ \mathtt{f}_{j}\,,\quad 4\leq j\leq 10\,,\label{astarj}
 \end{align}
 where $\mathtt{f}_{j}$ are some $j$-homogeneous symbols of the form \eqref{simboOMO2DEF}
 independent of $\x\in \mathbb{R}$
 and depending only on $\beta^{(3)}_{k}$ with $1\leq k\leq j-1$.

 \vspace{0.5em}
\noindent
{\bf Step $\e$.}
By \eqref{simboOMO2DEF} we have the expansion
\[
a^{(4)}_1:=\sum_{\s=\pm, n\in S}(\mathtt{a}_1^{(4)})_{n}^{\s}\sqrt{\zeta_{n}}e^{\ii \s\mathtt{l}(n)\cdot\vphi}
e^{\ii nx}\,,\qquad (\mathtt{a}^{(4)}_1)_{n}^{\s}\in \mathbb{C}\,.
\]
Since the function ${a}^{(4)}_1$ is real-valued, the coefficients $(\mathtt{a}^{(4)}_1)_{n}^{\s}$
satisfy
\eqref{iena2} with $k=1$.
We set
\begin{equation*}%\label{def:beta3}
(\beta^{(3)}_1)_{n}^{\s}:= \frac{-\s({\mathtt V}_1)^{\s}_{n}}{\ii\sqrt{|n|}}\,,
\quad \s=\pm\,, n\in S\,,\;\;n\neq0\,.
\end{equation*}
One can check that  (see \eqref{astar1}) $a^{(4)}_{*,1}\equiv0$
%\begin{equation*}%\label{detodounpoco3}
%a^{(4)}_{*,1}\equiv0\,.
%\end{equation*}
and that the function $\beta^{(3)}_1$ satisfies \eqref{iena2},
i.e. it is real valued. In particular the estimate \eqref{stimaSoloe2} holds.

\vspace{0.5em}
\noindent
{\bf Step $\e^{2}$.}
Now we want to eliminate the function in \eqref{astar2}.
We  define, for $n_1,n_2\in S$,
\begin{equation}\label{beta-23}
\begin{aligned}
&(\beta^{(3)}_{2})^{\s\s}_{n_1, n_2} 
:= \frac{-(\mathtt{a}^{(4)}_2)^{\s\s}_{n_1, n_2}}{\ii \s (\omega_{n_1} + \omega_{n_2})} \, , \, 
\;\; \s=\pm \, ,
\qquad 
(\beta^{(3)}_{2})^{\s (-\s)}_{n_1, n_2} :=  
\frac{-\s(\mathtt{a}^{(4)}_{2 })^{\s(-\s)}_{n_1, n_2}}{\ii (\omega_{n_1} - \omega_{n_2})} \, , 
\;\; n_1 \neq \pm n_2 \, , \\
&(\beta^{(3)}_{2})_{0,0}^{\s \s} := 0 \,,\;\;\; 
(\beta^{(3)}_2)^{\s(-\s)}_{n,\pm n}:=0\,,  \;\; \,
\end{aligned}
\end{equation}
%and $(\beta^{(3)}_{2})_{0,0}^{\s \s} := 0 $, 
%$(\beta^{(3)}_2)^{\s(-\s)}_{n,\s' n}:=0$,  $\s' = \pm $,
where $(\mathtt{a}^{(4)}_{2 })^{\s_1\s_2}_{n_1, n_2}$
are the coefficients of ${a}^{(4)}_2$ in the corresponding 
expansion \eqref{simboOMO2DEF}.
Notice that, by the form of $S$ in \eqref{TangentialSitesDP}, 
the condition
$n_1 \neq \pm n_2$ in \eqref{beta-23} reduces to $n_1 \neq n_2$.
Using that $a^{(4)}_{2}$ is real valued one has that 
$(\beta^{(3)}_2)^{\s_1\s_2}_{n_1,n_2}$
satisfy \eqref{iena2}, hence $\beta^{(3)}_{2}$ is real valued. In particular is satisfies the estimate
\eqref{stimaSoloe2}.
Moreover, recalling   Definition \ref{integroSimbo},
we define $\mathfrak{m}_{\frac{1}{2}}$
 as the restriction of $a^{(4)}_{2}$ to the set $\mathcal{S}_{2}$,
 i.e.
 \begin{equation*}%\label{Mgoti2}
{m}_{\frac{1}{2}}:=(a^{(4)}_{2})_{|\mathcal{S}_2}:=
 \frac{1}{2\pi}\sum_{n\in S}\big((a^{(4)}_{2})^{+-}_{n,n}+(a^{(4)}_{2})^{-+}_{n,n}\big)\zeta_{n}\,,
 \end{equation*}
 which is integrable according to Definition \ref{integroSimbo} (recall Remark \ref{accendo}).
 By \eqref{beta-23}, one can check that 
 $a^{(4)}_{2,*}\equiv {m}_{1/2}$.
%\begin{equation*}%\label{detodounpoco4}
%a^{(4)}_{*,2}\equiv {m}_{\frac{1}{2}}\,.
%\end{equation*}

\vspace{0.5em}
\noindent
{\bf Step $\e^{\geq3}$.}
We prove inductively that, for
 $j\geq 3$, there are functions $\beta^{(3)}_j$ of the form \eqref{simboOMO2DEF}
 with $i=j$ and coefficients independent of $\x\in \mathbb{R}$
 which are real valued and satisfying \eqref{stimaSoloe2} such that 
 (recall \eqref{astar3},\eqref{astarj})
 \begin{equation}\label{induco}
 a^{(4)}_{j,*}\equiv\left\{\begin{aligned}
 &m_{\frac{1}{2}}^{(j)}\,\quad j\; {\rm even}\,,\\
 &0\,\qquad \,\, j \; {\rm odd}\,,
 \end{aligned}
 \right.
 \end{equation}
with
 \begin{equation*}%\label{constM1212}
 m_{\frac{1}{2}}^{(j)}:=\sum_{\mathcal{S}_{j}}
 (\mathtt{F}_j)^{\s_1,\ldots,\s_{j}}_{n_1,\ldots,n_{j}}
 \sqrt{\zeta_{j_1}\cdots\zeta_{j_j}}e^{\ii(\s_1 j_1+\ldots+\s_{j}j_{j})x}
e^{\ii(\s_1\mathtt{l}(j_1)+\ldots+\s_{j}\mathtt{l}(j_j))\cdot\f}
\end{equation*}
where $\sum_{\mathcal{S}_{j}}$ denotes the sum restricted over indexes in $\mathcal{S}_{j}$.
Notice that $m_{1/2}^{(j)}$
 is integrable according to Definition \ref{integroSimbo} (recall Remark \ref{accendo}).
 
 Assume by induction that the assertion above is true for $3\leq k\leq j$.
 By construction (see the \eqref{ventrale}) the function $\mathtt{F}_{j+1}$ 
 depends only on 
 $\beta^{(3)}_{k}$, $k\leq j$, hence, by the inductive hypothesis it is real valued.
We  set, for $\vec{\s}=(\s_{1},\ldots,\s_{j+1})$, $\vec{j}=(j_1,\ldots,j_{j+1})$,
\begin{equation}\label{def:betai3}
(\beta^{(3)}_{j+1})_{\vec{j}}^{\vec{\s}}:=
\frac{-(\mathtt{F}_{j+1})^{\vec{\s}}_{\vec{j}}}{\ii \mathcal{R}_{\vec{\s}}(\vec{j})}\,,
\quad %(\vec{\s},\vec{j})
\left\{\big(j_k, \sigma_k \big) \right\}_{k=1}^{j+1}\notin \mathcal{S}_{j+1}\,,
\qquad
(\beta^{(3)}_{j+1})_{\vec{j}}^{\vec{\s}}=0\,, \quad  
%(\vec{\s},\vec{j})
\left\{\big(j_k, \sigma_k \big) \right\}_{k=1}^{j+1}\in \mathcal{S}_{j+1}\,,
\end{equation}
and $\mathcal{R}_{\vec{\s}}(\vec{j})$ is the function defined in \eqref{por}.
Since $\mathtt{F}_{j+1}$ is real, then   the coefficients 
$ (\mathtt{F}_{j+1})^{\s_1,\ldots,\s_{j+1}}_{n_1,\ldots,n_{j+1}}$ satisfies
 \eqref{iena2}. By an explicit computation one can 
 check that also $(\beta^{(3)}_{j+1})_{\vec{j}}^{\vec{\s}}$
 satisfies \eqref{iena2}, so that
$\beta^{(3)}_{j+1}$ is real valued.
The estimate
\eqref{stimaSoloe2} follows by using the bound on $\mathcal{R}_{\vec{\s}}$ in 
Remark \ref{accendo}.
The \eqref{stimaMapp2bis} follows by %\eqref{smallPrel2} and 
Lemma \ref{CoroDPdiffeo2}.
Using \eqref{def:betai3} one verifies the \eqref{induco}.
By the construction above, one can check that
the operator $\mathcal{L}_{4, *}$ is in $\mathfrak{S}_1$ and the map $\mathcal{T}_{\beta^{(3)}}^{\tau}$ is in $\mathfrak{T}_1$
(i.e. is symplectic and $x$-translation invariant).
.
%This conclude the proof of the lemma.
\end{proof}

\subsubsection{Reduction to constant coefficients at order 1/2}\label{rambo1010}
The aim of this section is to eliminate the $(\vphi,x)$-dependence
in the symbol $a^{(4)}_{*}$ in \eqref{termA1finale12} appearing in the operator
$\mathcal{L}_{4,*}$ in \eqref{elle4star}. 
Then we 
 conclude the proof of Proposition \ref{propordmezzoj}.
We first need a preliminary result.
\begin{lemma}\label{wish1}
For any $\omega\in \Omega_{\infty}^{2\gamma}$ (see \eqref{omegaduegamma})
there exists a real valued function $\beta^{(4)}(\varphi, x)\in S_{\mathtt{v}}$ (see \eqref{funzmom})
such that 
\begin{equation}\label{equabeta4}
a^{(4)}_{*}(\vphi,x)-(\omega\cdot\pa_{\vphi}+\mathfrak{m}_1\pa_{x})\beta^{(4)}(\vphi,x)
=\mathfrak{m}_{\frac{1}{2}}\,,
\end{equation}
where $a^{(4)}_{*}$ is in \eqref{termA1finale12},
the constant $\mathfrak{m}_{\frac{1}{2}}$ is in \eqref{mer}, with $i=1/2$ 
and
\begin{equation}\label{def:remezzo}
\mathtt{r}_{\frac{1}{2}}:=\int_{\mathbb{T}^{\nu}}a^{(4,*)}_{\geq11}(\vphi,x)d\vphi\,.
\end{equation}
The function  $\beta^{(4)}(\varphi, x)$ satisfies 
\begin{align}
\|\beta^{(4)}\|_{s}^{\gamma,\Omega_{\infty}^{2\gamma}}&\lesssim_s \gamma^{-2}
(\e^{13}+\e\|\mathfrak{I}_{\delta}\|_{s+\mu}^{\gamma,\calO_0})\,,\label{iena40beta4}\\
\|\Delta_{12}\beta^{(4)}\|_{p}&\lesssim_p
\e\gamma^{-2}(1+\|\mathfrak{I}_{\delta}\|_{p+\mu})
\|i_{1}-i_{2}\|_{p+\mu}\,,\label{iena41beta4}
\end{align}
for some $\mu>0$ depending on $\nu$.
Moreover the map $\mathcal{T}_{\beta^{(4)}}^{\tau}(\vphi)$, $\tau\in[0,1]$, 
defined as the flow
of \eqref{linprobpro} with generator as in 
\eqref{sim2} with $\beta\rightsquigarrow\beta^{(4)}$
is well posed and satisfies
 \begin{equation}\label{stimaMapp4}
\|(\mathcal{T}^{\tau}_{\beta^{(4)}})^{\pm1}u\|_{s}^{\gamma,\Omega^{2\gamma}_{\infty}}
\lesssim_{s}\|u\|_{s}^{\gamma, \Omega^{2\gamma}_{\infty}}+
\gamma^{-2}\big(\e^{13}+\e\|\mathfrak{I}_{\delta}\|^{\gamma, \calO_0}_{s+\mu}\big)
\|u\|_{s_0}^{\gamma, \Omega^{2\gamma}_{\infty}}\,,
\end{equation}
uniformly in  $\tau\in[0,1]$. Finally $\mathcal{T}^{\tau}_{\beta^{(4)}}$ belongs to
$\mathfrak{T}_1$.
\end{lemma}

\begin{proof}
Passing to the Fourier basis (recall \eqref{realfunctions})
we have that equation \eqref{equabeta4} reads (recall \eqref{termA1finale12})
\begin{equation}\label{equabeta4bis}
(a^{(4,*)}_{\geq11})_{lj}-\ii ({\omega}\cdot l+\mathfrak{m}_1j)(\beta^{(4)})_{lj}=0\,,\qquad \forall \; l\in \mathbb{Z}^{\nu}\,, j\in \mathbb{Z}\setminus\{0\}\,,\quad (l,j)\neq(0,0)\,.
\end{equation}
Using the diophantine condition in \eqref{omegaduegamma}
one can prove
\begin{equation*}
\|\beta^{(4)}\|_{s}^{\gamma,\Omega_{\infty}^{2\gamma}}\lesssim_{s}\gamma^{-1}
\|a^{(4,*)}_{\geq11}\|_{s+2\tau+1}^{\gamma,\Omega_{\infty}^{2\gamma}}
\stackrel{\eqref{iena40a4}}{\lesssim_s}
\gamma^{-2}
(\e^{13}+\e\|\mathfrak{I}_{\delta}\|_{s+\mu}^{\gamma,\calO_0})
\end{equation*}
which is the \eqref{iena40beta4}. The \eqref{iena41beta4}
follows similarly using \eqref{iena41a4}.
Since $a^{(4,*)}_{\geq11}$ is in $S_{\mathtt{v}}$ 
(recall that 
$\mathcal{L}_{4,*}\in\mathfrak{S}_1$)
 %since $\mathcal{L}_4\in \mathfrak{S}_1$
%by Prop. \ref{proporduno} and 
%by Lemma \ref{GoodAmmissi4} and the fact that 
%$\mathcal{T}_{\beta^{(3)}}^{\tau}\in\mathfrak{T}_1$ %is in $\mathfrak{T}_1$ 
%by Lemma \ref{lemmapre}) 
then, by \eqref{equabeta4bis}, we deduce that also $\beta^{(4)}\in S_{\mathtt{v}}$. 
By Lemma \ref{differenzaFlussi} we have
\begin{equation}\label{georgiaaa}
\mathcal{T}^{1}_{\beta^{(4)}}=
\Pi_{S}^{\perp} \Phi_{\beta^{(4)}}^{\tau}\Pi_{S}^{\perp}\circ({\rm Id}+\mathcal{R})
\end{equation}
where $\Phi_{\beta^{(4)}}^{\tau}$ is the flow of \eqref{linprob}
 with $f$  as in \eqref{sim2}
with $\beta\rightsquigarrow \beta^{(4)}$
and where $\mathcal{R}$ is a finite rank operator of the form \eqref{FiniteDimFormDP10001}
satisfying \eqref{giallo2}.
The estimate \eqref{stimaMapp4} follows by Lemma \ref{CoroDPdiffeo2}
and estimates \eqref{flow2}-\eqref{flow222}.
Moreover the map $\mathcal{T}_{\beta^{(4)}}^{\tau}(\vphi)$
is in $\mathfrak{T}_1$. 
\end{proof}

\begin{proof}[{\bf Proof of Proposition \ref{propordmezzo}}]
We consider the map
\begin{equation}\label{MappaT4}
{\bf T}_4:=\left(
\begin{matrix}
\mathcal{T}_{\beta^{(4)}}^{1}(\vphi) & 0 \\
0 & \ov{\mathcal{T}_{\beta^{(4)}}^{1}(\vphi)}
\end{matrix}
\right)
\end{equation}
where $\mathcal{T}_{\beta^{(4)}}^{\tau}(\vphi) $, $\tau\in[0,1]$ 
is given by Lemma \ref{wish1}.
We now conjugate the operator 
 $\mathcal{L}_{4, *}$ in \eqref{elle4star} 
  with the map ${\bf T}_4$.
First of all we remark that
the symbol
\begin{equation}\label{wish2}
f(\vphi,x,\x):=\beta^{(4)}(\vphi,x)|\x|^{\frac{1}{2}}
\end{equation}
belongs to the class $S_{3}^{1/2}$ (see Def. \ref{espOmogenea}). This is true
because estimates
\eqref{iena40beta4}, \eqref{iena41beta4} imply the \eqref{AVBDEF}
%, \eqref{AVB2DEF}
with $k=3$.
%Following the strategy of section \ref{preliminaremezzo}
%we first note that, by Lemma \ref{differenzaFlussi},
%\[
%\mathcal{T}^{1}_{\beta^{(4)}}=
%\Pi_{S}^{\perp} \Phi_{\beta^{(4)}}^{\tau}\Pi_{S}^{\perp}\circ(1+\mathcal{R})
%\]
%where $\Phi_{\beta^{(4)}}^{\tau}$ is the flow of \eqref{linprob}
% with $f$  as in \eqref{sim2}
%with $\beta\rightsquigarrow \beta^{(4)}$
%and where $\mathcal{R}$ is a finite rank operator of the form \eqref{FiniteDimFormDP10001}
%satisfying \eqref{giallo2}.
Moreover (recall \eqref{georgiaaa}), by Lemma \ref{georgiaLem} 
we have that
the conjugate of $\mathcal{L}_{4,*}$ under 
the map ${\bf T}_{4}$ is given by 
\[
\Pi_{S}^{\perp}\widetilde{{\bf T}}_4^{1}\mathcal{L}_{4,*}
(\widetilde{{\bf T}}_4^{1})^{-1}\Pi_{S}^{\perp}\,,
\qquad
\widetilde{{\bf T}}_4^{\tau}:=\left(
\begin{matrix}
{\Phi^{\tau}_{\beta^{(4)}}}(\vphi) & 0 \\
0 & \ov{\Phi_{\beta^{(4)}}^{\tau}(\vphi)}
\end{matrix}
\right)\,,\qquad \tau\in[0,1]\,,
\]
up to finite rank operators belonging to 
$\mathfrak{L}_{\rho, p}^{3}\otimes\mathcal{M}_2(\mathbb{C})$. 
%The latter claim follows by estimates \eqref{giallo2Rank}, 
%Lemma \ref{lemmapremezzo} to estimates the symbols 
%of the operator $\mathcal{L}_{4,*}$ in 
%\eqref{elle4star} and \eqref{iena40beta4}, \eqref{iena41beta4}.
%$\Pi_{S}^{\perp} \Phi_{\beta^{(4)}}^{\tau}\mathcal{L}_{4,*}(\Phi_{\beta^{(4)}}^{\tau})^{-1}\Pi_{S}^{\perp}$
 Hence
we apply Lemmata \ref{flusso12basso}, \ref{flusso12basso2}
with $f(\vphi,x,\x)$ as in \eqref{wish2}.
By explicit computations using \eqref{coniugato1000}, \eqref{castello1}, \eqref{coniugato1002}, \eqref{castello1tempo} (as done in \eqref{tempomezzo}-\eqref{ord1123})
we obtain that the operator 
\[
\mathcal{L}_{5}:={\bf T}^1_4 \,\mathcal{L}_{4,*} \, ({\bf T}^1_4)^{-1}
\]
has the form \eqref{elle5} with some 
remainder 
$\mathcal{R}^{(5)}\in\mathfrak{L}^{3}_{\rho,p}\otimes\mathcal{M}_2(\mathbb{C})$, 
some
symbol $d^{(5)}(\vphi,x,\x)$ of the form
\[
d^{(5)}(\vphi,x,\x):=\ii \mathfrak{m}_1\,\x+\ii (1+A(\vphi,x))|\x|^{\frac{1}{2}}
+\ii b^{(5)}(\vphi,x)\sign(\x)
+c^{(5)}(\vphi,x,\x)\,,
\]
where
\begin{align*}
A&=a^{(4)}_{*}-(\omega\cdot\pa_{\vphi}+\mathfrak{m}_1\pa_{x})\beta^{(4)}\,,
%\label{alive12}
\\
b^{(5)}&=b^{(4)}_{*}-\frac{1}{4}\big(
\beta^{(4)}\omega\cdot\pa_{\vphi}\beta^{(4)}_x
-\beta^{(4)}_x\omega\cdot\pa_{\vphi}\beta^{(4)}
\big)-\frac{\mathfrak{m}_1}{4}
\big( \beta^{(4)}\beta^{(4)}_{xx}-(\beta^{(4)}_x)^{2}\big)
%\label{alive22}
%\\&
+\frac{1}{2}\big(
\beta^{(4)}(a^{(4)}_{*})_x-(1+a^{(4)}_{*})\beta^{(4)}_x
\big)\,
%\nonumber
\end{align*}
and $c^{(5)}\in S_3^{-1/2}$.
By equation \eqref{equabeta4} in Lemma \ref{wish1} we actually have that
$A(\vphi,x)\equiv\mathfrak{m}_{1/2}$
in \eqref{mer}, $i=1/2$\,.
%By estimates \eqref{iena40beta4}, \eqref{iena41beta4}, 
%and the fact that $b^{(4)}_*\in S_2^{0}$
%we have that $b^{(5)}(\vphi,x)$ is in $S^{0}_{3}$ because it satisfies estimates like 
%\eqref{AVBDEF} %, \eqref{AVB2DEF} 
%with $k=3$.
%In the same way we deduce that $\mathcal{R}^{(5)}\in \mathcal{L}^{3}_{\rho,p}$, 
%$c^{(5)}\in S^{-1/2}_{3}$ (see Def. \ref{espOmogenea}) 
%using estimates \eqref{crawford100}-\eqref{jamal2100}, 
%\eqref{crawford100tempo}-\eqref{jamal2100tempo}
%in  Lemmata \ref{flusso12basso}, \ref{flusso12basso2}.
The constant $\mathtt{r}_{1/2}$ in \eqref{def:remezzo} satisfies the bounds
\eqref{merstima} for $i=1/2$ by \eqref{iena40a4}, \eqref{iena41a4} and \eqref{IpotesiPiccolezzaIdeltaDP}.
To conclude we set (recall \eqref{mappapasso3}, \eqref{MappaT4})
\[
{\bf \Phi}_2:={\bf T}_4\circ{\bf T}_3\,.
\]
The estimate \eqref{stimaMapp2} follows by composition
using \eqref{stimaMapp2bis} and \eqref{stimaMapp4}.
We conclude the proof by studying the algebraic properties of
${\bf \Phi}_2$ and $\mathcal{L}_{5}$.

The map ${\bf \Phi}_2$ is in $\mathfrak{T}_1$ (recall Def. \ref{compaMulti}) 
since $\mathbf{T}_3, \mathbf{T}_4\in \mathfrak{T}_1$ (see Lemmata \ref{lemmapre}, \ref{wish1}).
Since, by Prop. \ref{proporduno}, 
$\mathcal{L}_{4}$ is in $\mathfrak{S}_1$ we have that $\mathcal{L}_{5}$ in \eqref{elle5}
is Hamiltonian and $x$-translation invariant. The symbol $d^{(5)}(\vphi,x,\x)$
satisfies \eqref{space3Xinv}, \eqref{selfSimbo2} by construction, 
and so the operator 
$\mathcal{L}_{5}\in\mathfrak{S}_1$.
\end{proof}

%\begin{proof}[{\bf Proof of Lemma \ref{GoodAmmissi5}}]
%The map ${\bf \Phi}_2$ is in $\mathfrak{T}_1$ (recall Def. \ref{compaMulti}) 
%since it is composition 
%of maps in $\mathfrak{T}_1$ (see Lemmata \ref{lemmapre}, \ref{wish1}).
%Since, by Lemma \ref{GoodAmmissi4}, 
%$\mathcal{L}_{4}$ is in $\mathfrak{S}_1$ we have that $\mathcal{L}_{5}$ in \eqref{elle5}
%is Hamiltonian and $x$-translation invariant. The symbol $d^{(5)}(\vphi,x,\x)$
%satisfies \eqref{space3Xinv}, \eqref{selfSimbo2}, and so the operator 
%$\mathcal{L}_{5}\in\mathfrak{S}_1$.
%\end{proof}

\subsection{Integrability at order zero}\label{ordiniNegativi}
The aim of this section is to conjugate the operator $\mathcal{L}_{5}$
in \eqref{elle5} to constant coefficients up to a smoothing remainder
of order $-1/2$.
The key result of the section is the following.

\begin{proposition}\label{propordmezzoj}
Let $\rho\geq 3$, $p\geq s_0$ and assume that 
\eqref{IpotesiPiccolezzaIdeltaDP} holds. Then there exist 
$\mu=\mu(\nu)$ and 
$\mathfrak{m}_{0}\colon \Omega_{\varepsilon}\to \mathbb{R}$ 
of the form
\eqref{mer} with $i=0$ such that, 
for all $\omega\in \Omega^{2\gamma}_{\infty}$
(see \eqref{omegaduegamma}), 
there exists a map 
${\bf \Phi}_3(\omega)\colon H^s_{S^{\perp}}(\mathbb{T})
\times H^s_{S^{\perp}}(\mathbb{T})\to 
H^s_{S^{\perp}}(\mathbb{T})\times H^s_{S^{\perp}}(\mathbb{T})$ such that
(recall \eqref{elle5})
\begin{equation}\label{ellej}
\mathcal{L}_6:={\bf \Phi}_3 \,\mathcal{L}_5\,{\bf \Phi}_3^{-1}=
\Pi_{S}^{\perp}\Big(
\omega\cdot\pa_{\vphi}+
\opw\left(\begin{matrix} \mathtt{d}^{(6)}(\vphi,x,\x) & 0 \\
0 & \ov{\mathtt{d}^{(6)}(\vphi,x,-\x)}  \end{matrix}\right)+\mathcal{Q}^{(6)}\Big)
\end{equation}
where
the remainder
$\mathcal{Q}^{(6)}\in \mathfrak{L}^{4}_{\rho,p}\otimes\mathcal{M}_2(\mathbb{C})$, 
%(recall Def. \ref{espOmogenea}),
the symbol $\mathtt{d}^{(6)}(\vphi,x,\x)$ has the form 
\begin{align}
\mathtt{d}^{(6)}(\vphi,x,\x)&:=
\ii \mathfrak{m}_1\,\x+\ii (1+\mathfrak{m}_\frac{1}{2})|\x|^{\frac{1}{2}}
+\ii 
\mathfrak{m}_{0}\sign(\x)
+\ii q^{(6)}(\vphi,x,\x)\,,\label{diagpezzoJe}
\end{align}
with $\mathfrak{m}_1$, $\mathfrak{m}_{1/2}$ in \eqref{mer} with $i=1,1/2$,
and $q^{(6)}(\vphi,x,\x)\in S_{4}^{-1/2}$.
Moreover for any $\omega\in \Omega_{\infty}^{2\gamma}$ we have
\begin{equation}\label{stimaMappj}
\|({\bf \Phi}_3)^{\pm1}u\|_{s}^{\gamma,\Omega^{2\gamma}_{\infty}}
\lesssim_{s}\|u\|_{s}^{\gamma, \Omega^{2\gamma}_{\infty}}+
\gamma^{-3}\big(\e^{13}+\e\|\mathfrak{I}_{\delta}\|^{\gamma, \calO_0}_{s+\mu}\big)
\|u\|_{s_0}^{\gamma, \Omega^{2\gamma}_{\infty}}\,.
\end{equation}
Finally the map $\Phi_{3}$ are in $\mathfrak{T}_1$ and the operator
$\mathcal{L}_{6}$ are in $\mathfrak{S}_1$ (see Def. \ref{compaMulti}, \ref{goodMulti}).
\end{proposition}

As done  in sections \ref{redord1}, \ref{redord12}, 
we divide the construction of the map $\Phi_{j}$ into two steps.
%analyzed in the following sections.

\subsubsection{Preliminary steps}\label{figaro00}
%Consider the operator $\mathcal{Z}_{j}$, $3\leq j\leq 6$ in \eqref{ellej}. 
We prove the following result.

\begin{lemma}{\bf (Preliminary steps).}\label{lemmaprej}
There exists a real valued symbol $\beta^{(5)}(\varphi, x,\x)\in S^{0}$ of the form
\begin{equation}\label{genPre1j}
\beta^{(5)}(\vphi,x,\x)
=\sum_{i=1}^{8}\e^{i}\beta^{(5)}_i(\vphi,x)\sign(\x)\,,
\end{equation}
(recall the notation \eqref{notasignXi})
where $\beta^{(5)}_i$ is a real valued 
$i$-homogeneous functions in $S^{0}$ as in \eqref{simboOMO2DEF}
for $i=1,\ldots,8$
satisfying
\begin{equation}\label{stimaSoloe2j}
\|   \beta^{(5)}   \|_{s}^{\gamma,\mathcal{O}_0}\lesssim_{s}\e\,,
\end{equation}
such that the following holds.
If 
\begin{equation}\label{mappapasso3j}
\mathbf{T}_{5}^{\tau}:=\left(
\begin{matrix}
\mathcal{T}_{ \beta^{(5)}}^{\tau}(\vphi) & 0 \\
0 & \ov{\mathcal{T}_{ \beta^{(5)}}^{\tau}(\vphi)}
\end{matrix}
\right)\,,\qquad \tau\in[0,1]\,,
\end{equation}
where $\mathcal{T}_{ \beta^{(5)}}^{\tau}$ is the flow of \eqref{linprobpro}
with generator as in \eqref{sim3} with $f(\vphi,x,\x)\rightsquigarrow  \beta^{(5)}$, 
then
\begin{equation}\label{elle4starj}
\mathcal{L}_{5, *}:=\mathbf{T}^1_{5} \,\mathcal{L}_{5} \, (\mathbf{T}^1_{5})^{-1}
=\Pi_{S}^{\perp}\Big(\omega\cdot\pa_{\vphi}+
\opw\left(\begin{matrix} \mathtt{d}_*^{(5)}(\vphi,x,\x) & 0 \\
0 & \ov{\mathtt{d}_*^{(5)}(\vphi,x,-\x)}  \end{matrix}\right)+\mathcal{Q}_*^{(5)}\Big)
\end{equation}
where 
$\mathcal{Q}_*^{(5)}\in\mathfrak{L}_{\rho, p}^{3}\otimes\mathcal{M}_2(\mathbb{C})$ 
and (recall \eqref{diagpezzoJe})
\begin{align}
\mathtt{d}^{(5)}_*(\vphi,x,\x)&:=
\ii \mathfrak{m}_1\,\x+\ii (1+\mathfrak{m}_\frac{1}{2})|\x|^{\frac{1}{2}}
+\ii q_*(\vphi,x)\sign(\x)+\ii q^{(5)}_*(\vphi,x,\x)\,,
\label{diagpezzoord12j}\\
q_*(\vphi,x)&:=
\Big(\varepsilon^2 m_{0}+\sum_{k=1}^{4}\e^{2k}m_{0}^{(2k)}\Big)
+
q^{(*)}_{\geq 9}(\vphi,x)\,,
\label{termA1finale12j}
\end{align}
where  $m_{0}$, 
$m_{0}^{(2k)}$ are respectively 
$2$, $2k$-homogenenous in $S^{0}$ 
and  integrable (according to Def. \ref{integroSimbo}).
Moreover 
$q_*\in S_3^{0}$ 
(recall Def. \ref{espOmogenea}), independent of $\x\in\mathbb{R}$ and 
\begin{align}
\|q^{(*)}_{\geq 9}\|_{s}^{\gamma,\Omega_{\infty}^{2\gamma}}&
\lesssim_s \gamma^{-2}
(\e^{13}+\e\|\mathfrak{I}_{\delta}\|_{s+\mu}^{\gamma,\calO_0})\,,\label{iena40a4j}\\
\|\Delta_{12}q^{(*)}_{\geq 9}\|_{p}
&\lesssim_p \gamma^{-2}
\e(1+\|\mathfrak{I}_{\delta}\|_{p+\mu})
\|i_{1}-i_{2}\|_{p+\mu}\,,\label{iena41a4j}
\end{align}
for some $\mu=\mu(\nu)>0$. The symbol $q^{(5)}_*(\vphi,x,\x)$
is in $S_3^{-\frac{1}{2}}$.
Finally the map
$\mathcal{T}_{ \beta^{(5)}}^{\tau}$ is in $\mathfrak{T}_1$ and satisfies
\begin{equation}\label{stimaMapp2bisj}
\|(\mathcal{T}^{\tau}_{\beta^{(5)}})^{\pm1}u\|_{s}^{\gamma,\Omega^{2\gamma}_{\infty}}
\lesssim_{s}\|u\|_{s}^{\gamma, \Omega^{2\gamma}_{\infty}}+
\e
\|u\|_{s_0}^{\gamma, \Omega^{2\gamma}_{\infty}}\,,\qquad \forall\,\tau\in[0,1]\,.
\end{equation}
The operator $\mathcal{L}_{5,*}$ is in $\mathfrak{S}_1$.
%and in $\mathfrak{T}_3$ (recall Def. \ref{goodMulti}). Moreover if the function 
%$v_{I}(\vphi,x)$ in \eqref{vsegnatoDP} is even in $x$, then the map
%$\mathcal{T}_{\beta^{(1)}}^{\tau}$ is in $\mathfrak{T}_2$. 
\end{lemma}

\begin{proof}
We follow the strategy used in the proof of Lemma \ref{lemmapremezzo}.
By Lemma \ref{differenzaFlussi} we have
\[
\mathcal{T}^{1}_{\beta^{(5)}}=
\Pi_{S}^{\perp} \Phi_{\beta^{(5)}}^{\tau}\Pi_{S}^{\perp}\circ({\rm Id}+\mathcal{R})
\]
where $\Phi_{\beta^{(5)}}^{\tau}$ is the flow of \eqref{linprob}
 with $f$  as in \eqref{sim3}
with $\beta\rightsquigarrow \beta^{(5)}$
and where $\mathcal{R}$ is a finite rank operator of the form \eqref{FiniteDimFormDP10001}
satisfying \eqref{giallo2}.
Then by Lemma \ref{CoroDPdiffeo2} we have that
%, if $\gamma_1$ is real valued
%and satisfies \eqref{stimaSoloe2j}, then
the flow $\mathcal{T}^{\tau}_{\beta^{(5)}}$
is well posed and satisfies bounds like  \eqref{flow2}-\eqref{flow222}.

By Lemma \ref{georgiaLem} we have that
the conjugate of $\mathcal{L}_{5}$ in \eqref{elle5} 
under the map $\mathbf{T}_5^{1}$
is given by 
\[
\Pi_{S}^{\perp}\widetilde{\mathbf{T}}_5^{1}\mathcal{L}_{5}
(\widetilde{\mathbf{T}}_5^{1})^{-1}\Pi_{S}^{\perp}\,,
\qquad
\widetilde{\mathbf{T}}_5^{\tau}:=\left(
\begin{matrix}
\Phi^{\tau}_{\beta^{(5)}}(\vphi) & 0 \\
0 & \ov{\Phi_{\beta^{(5)}}^{\tau}(\vphi)}
\end{matrix}
\right)\,,\qquad \tau\in[0,1]\,,
\]
up to finite rank operators belonging to 
$\mathfrak{L}_{\rho, p}^{3}\otimes\mathcal{M}_2(\mathbb{C})$. 
 Hence
we shall apply Lemmata \ref{flusso12basso}, \ref{flusso12basso2}.
We start by studying the time contribution. By Lemma \ref{flusso12basso2}
we have that (see \eqref{coniugato1002}, \eqref{castello1tempo})
\begin{equation}\label{tempomezzoj}
\begin{aligned}
\Phi^{1}_{\beta^{(5)}}&\omega\cdot\pa_{\vphi}
(\Phi^{1}_{\beta^{(5)}})^{-1}=
\omega\cdot\pa_{\vphi}
-\ii\opw\Big( \omega\cdot\pa_{\vphi}\beta^{(5)}(\vphi,x,\x)+r_{1}(\vphi,x,\x)
\Big)\,,
\end{aligned}
\end{equation}
up to smoothing remainder $\mathcal{R}_1$ in the class $\mathfrak{L}^3_{\rho,p}$
and where $r_1$ is some symbol in $S_3^{-1}$.

We now consider the contribution coming form the conjugation of the spatial operator
by using Lemma \ref{flusso12basso} applied to the pseudo differential operator
with symbol in \eqref{diagpezzoJe}.
First we notice that (see \eqref{coniugato1000}, \eqref{castello1})
\begin{equation}\label{ord112j}
\begin{aligned}
\Phi^{1}_{\beta^{(5)}}\opw(\ii \mathfrak{m}_1\x)(\Phi^{1}_{\beta^{(5)}})^{-1}&=
\ii\opw\big(
 \mathfrak{m}_1\x+\{\beta^{(5)},\mathfrak{m}_{1}\x\}+r_2
\big)
%\\&
=
\opw\big(
\ii \mathfrak{m}_1\x-\mathfrak{m}_1\pa_{x}(\beta^{(5)})+r_2
\big)
\end{aligned}
\end{equation}
up to smoothing remainder $\mathcal{R}_2$ in the class $\mathfrak{L}^3_{\rho,p}$
and where $r_2$ is some symbol in $S_3^{-1}$.

Finally we have (recall \eqref{elle5}) %\eqref{diagpezzoJe})
\begin{equation}\label{ord1122j}
\begin{aligned}
\Phi^{1}_{\beta^{(5)}}&(\mathtt{d}^{(j)}-\ii m_1\x)(\Phi^{1}_{\beta^{(5)}})^{-1}=
\ii\opw\big({d}^{(5)}-\ii m_1\x
+r_3+r_4
\big)\\
&r_{3}:=\{\beta^{(5)},{d}^{(5)}-\ii m_1\x\}\in S_3^{-\frac{1}{2}}\,,
\end{aligned}
\end{equation}
up to smoothing remainder $\mathcal{R}_3$ in the class $\mathfrak{L}^3_{\rho,p}$
and where $r_4$ is some symbol in $S_3^{-3/2}$.
By collecting the \eqref{tempomezzoj}-\eqref{ord1122j}
we conclude that the conjugate of  $\mathcal{L}_5$ in \eqref{elle5}
(see also \eqref{diagpezzo2or12})
 under the map in \eqref{mappapasso3j} has the form \eqref{elle4starj}
with symbol $\mathtt{d}^{(5)}_*(\vphi,x,\x)$ as in \eqref{diagpezzoord12j} where
$q^{(5)}_*\in S_{3}^{-\frac{1}{2}}$,  and
\begin{align}
q_{*}\sign(\x)&=b^{(5)}(\vphi,x)\sign(\x)
-\big(\omega\cdot\pa_{\vphi}+\mathfrak{m}_1\pa_{x}\big)\beta^{(5)}\,.
\label{alive1j}
\end{align}
We shall prove that it is possible to choose $\beta^{(5)}$ in \eqref{genPre1j}
in such a way
the symbol $q_{*}$ in \eqref{alive1j} has the form \eqref{termA1finale12j}.
Recall that $b^{(5)}$ is a symbol
in the class $S_3^{0}$ (see Prop. \ref{propordmezzo}) and hence it
admits and expansion as \eqref{espandoenonmifermo} with $k=3$
for some $i$-homogeneous symbols $b^{(5)}_i$ and a non-homogeneopus symbol
$\tilde{b}^{(5)}$. Therefore, recalling the expansion
of $\mathfrak{m}_{1}$ in \eqref{mer}
and of 
$\omega$ in \eqref{FreqAmplMapDP}, 
we have
\begin{equation}\label{ventralej}
\begin{aligned}
q_{*}&=\sum_{i=1}^{8}\e^{i}b_{i}^{(5)}
-\sum_{i=1}^{8}\e^{i}\bar{\omega}\cdot\pa_{\vphi} \beta_i^{(5)}
-\sum_{i=1}^{8}\e^{i}\left(
\e^{2}\mathbb{A}\x+\sum_{k=2}^{6}\e^{2k}\mathtt{b}_{k}(\x,\ldots,\x)
\right)\cdot\pa_{\vphi} \beta_i^{(5)}\\
&-\left(\e^{2}m_1+\sum_{k=2}^{6}\e^{2k}m_{1}^{2k}+\mathtt{r}\right)
\sum_{i=1}^{8}\e^{i}(\beta_i^{(5)})_x+\tilde{b}^{(5)}\,.
\end{aligned}
\end{equation}
This implies that the symbol $q_*$ has an expansion as in 
 \eqref{espandoenonmifermo} with $k=3$
 for some non-homogeneous symbol $\tilde{q}_{*}$ satisfying \eqref{AVBDEF}, 
 %\eqref{AVB2DEF}
 with $k=3$, and some $i$-homogeneous symbols $q_{i,*}$, $i=1,\ldots,8$.
 In particular we have
 \begin{align}
 q_{1,*}&:= \mathtt{F}_1-\bar{\omega}\cdot\pa_{\vphi} \beta_1^{(5)}\,,\quad 
 \mathtt{F}_1:=b^{(5)}_{1}\label{astar1j}\\
 q_{2,*}&:=\mathtt{F}_2-\bar{\omega}\cdot\pa_{\vphi} \beta_2^{(5)} \,,\quad 
 \mathtt{F}_2:=b^{(5)}_{2}\label{astar2j}\\
   q_{3,*}&:=\mathtt{F}_{3}-\bar{\omega}\cdot\pa_{\vphi} \beta_3^{(5)}\,,
   \quad \mathtt{F}_3:=b^{(5)}_{3}+\mathtt{f}_{3}\,,
  \quad\mathtt{f}_{3}:=-\mathbb{A}\x\cdot\pa_{\vphi} \beta_1^{(5)}
  -{m}_1( \beta_1^{(5)})_{x}
    \,,\label{astar3j}\\
   q_{i,*}&:=\mathtt{F}_i-\bar{\omega}\cdot\pa_{\vphi} \beta_i^{(5)}\,,
  \quad\mathtt{F}_i:= b^{(5)}_{i}+ \mathtt{f}_{i}\,,\quad 4\leq i\leq 8\,,\label{astarjj}
 \end{align}
 where $\mathtt{f}_{i}$ are some $i$-homogeneous symbols in $S^{0}$,
 independent of $\x\in \mathbb{R}$,
 of the form \eqref{simboOMO2DEF}
 and depending only on $\beta_k^{(5)}$ with $1\leq k\leq i-1$.

We now prove inductively that, for
 $i\geq 1$, there are functions $\beta_i^{(5)}$ of the form \eqref{simboOMO2DEF}
 which are real valued and satisfying \eqref{stimaSoloe2j} such that 
 (recall \eqref{astar3j},\eqref{astarjj})
 \begin{equation}\label{inducoj}
 q_{i,*}\equiv\left\{\begin{aligned}
 &m_{0}^{(i)}(\x)\,\quad i\; {\rm even}\,,\\
 &0\,\qquad \quad\,\, i \; {\rm odd}\,,
 \end{aligned}
 \right.
 \end{equation}
with
 \begin{equation*}%\label{constM1212j}
 m_{0}^{(i)}:=\sum_{\mathcal{S}_{i}}
 (\mathtt{F}_i)^{\s_1,\ldots,\s_{i}}_{n_1,\ldots,n_{i}}(\x)
 \sqrt{\x_{j_1}\cdots\x_{j_i}}e^{\ii(\s_1 j_1+\ldots+\s_{i}j_{i})x}
e^{\ii(\s_1\mathtt{l}(j_1)+\ldots+\s_{i}\mathtt{l}(j_i))\cdot\f}
\end{equation*}
where $\sum_{\mathcal{S}_{i}}$ denotes the sum restricted over indexes in $\mathcal{S}_{i}$.
Notice that $m_{0}^{(i)}$
 is integrable according to Definition \ref{integroSimbo} (recall Remark \ref{accendo}).
 
 Assume by induction that the assertion above is true for $1\leq k\leq j$.
 By construction (see the \eqref{ventralej}) the function $\mathtt{F}_{j+1}$ 
 depends only on $b^{(5)}_{k}\sign(\x)$ (which is real) and 
 $\beta_k^{(5)}$, $k\leq j$, hence, by the inductive hypothesis it is real valued.
We  set, for $\vec{\s}=(\s_{1},\ldots,\s_{i+1})$, $\vec{j}=(j_1,\ldots,j_{i+1})$,
\begin{equation}\label{def:gammai3}
(\beta_{i+1}^{(5)})_{\vec{j}}^{\vec{\s}}:=
\frac{-(\mathtt{F}_{i+1})^{\vec{\s}}_{\vec{j}}}{\ii \mathcal{R}_{\vec{\s}}(\vec{j})}\,,
\quad %(\vec{\s},\vec{j})
\left\{\big(j_k, \sigma_k \big) \right\}_{k=1}^{i+1}\notin \mathcal{S}_{i+1}\,,
\qquad
(\beta_{i+1}^{(5)})_{\vec{j}}^{\vec{\s}}=0\,, \quad  
%(\vec{\s},\vec{j})
\left\{\big(j_k, \sigma_k \big) \right\}_{k=1}^{i+1}\in \mathcal{S}_{i+1}\,,
\end{equation}
and $\mathcal{R}_{\vec{\s}}(\vec{j})$ is the function defined in \eqref{por}.
Since $\mathtt{F}_{j+1}$ is real, then   the coefficients 
$ (\mathtt{F}_{i+1})^{\s_1,\ldots,\s_{i+1}}_{n_1,\ldots,n_{i+1}}(\x)$ satisfies
 \eqref{iena2}. By an explicit computation one can 
 check that also $(\beta_{i+1}^{(5)})_{\vec{j}}^{\vec{\s}}$
 satisfies \eqref{iena2}, so that
$\beta_{i+1}^{(5)}$ is real valued.
The estimate
\eqref{stimaSoloe2j} follows by using Remark \ref{accendo}.
The \eqref{stimaMapp2bisj} follows by Lemma \ref{CoroDPdiffeo2}.
Using \eqref{def:gammai3} one verifies the \eqref{inducoj}.
By construction $\mathcal{L}_{5, *}\in\mathfrak{S}_1$ and the map $\mathcal{T}_{\beta^{(5)}}^{\tau}$ is in $\mathfrak{T}_1$
(i.e. is symplectic and $x$-translation invariant).
This conclude the proof of the lemma.
\end{proof}

\subsubsection{Reduction to constant coefficients at order zero}
The aim of this section is to eliminate the $(\vphi,x)$ dependence
in the symbol $q_{*}$ in \eqref{termA1finale12j} appearing in the operator
$\mathcal{L}_{5,*}$ in \eqref{elle4starj}. 
Then we 
 conclude the proof of Proposition \ref{propordmezzo}.
We first need a preliminary result.
\begin{lemma}\label{wish1j}
For any $\omega\in \Omega_{\infty}^{2\gamma}$ (see \eqref{omegaduegamma})
there exists a real valued symbol $\beta^{(6)}(\varphi, x,\x)\in S_4^{0}$ 
satisfying \eqref{space3Xinv} and 
such that 
\begin{equation}\label{equabeta4j}
q_{*}(\vphi,x)\sign(\x)-(\omega\cdot\pa_{\vphi}+\mathfrak{m}_1\pa_{x})\beta^{(6)}(\vphi,x,\x)
=\mathfrak{m}_{0}\sign(\x)\,,
\end{equation}
where $q_{*}$ is in \eqref{termA1finale12j},
the constant $\mathfrak{m}_{0}$ is in \eqref{mer}, $i=0$, with
\begin{equation}\label{def:remezzoj}
\mathtt{r}_{0}:=\int_{\mathbb{T}^{\nu}}q_{*}(\vphi,x)d\vphi\,.
\end{equation} 
The symbol  $\beta^{(6)}(\varphi, x,\x)$ satisfies 
\begin{align}
|\beta^{(6)}|_{0,s,\alpha}^{\gamma,\Omega_{\infty}^{2\gamma}}&
\lesssim_s \gamma^{-3}
(\e^{13}+\e\|\mathfrak{I}_{\delta}\|_{s+\mu}^{\gamma,\calO_0})\,,\label{iena40beta4j}\\
|\Delta_{12}\beta^{(6)}|_{0,p,\alpha}&\lesssim_p
\e\gamma^{-3}(1+\|\mathfrak{I}_{\delta}\|_{p+\mu})
\|i_{1}-i_{2}\|_{p+\mu}\,,\label{iena41beta4j}
\end{align}
for some $\mu>0$ depending on $\nu,\alpha$.
Moreover the map $\mathcal{T}_{\beta^{(6)}}^{\tau}(\vphi)$, $\tau\in[0,1]$, defined as the flow
of \eqref{linprobpro} with generator as in 
\eqref{sim3} with $f\rightsquigarrow\beta^{(6)}$
is well posed and satisfies
 \begin{equation}\label{stimaMapp4j}
\|(\mathcal{T}^{\tau}_{\beta^{(6)}})^{\pm1}u\|_{s}^{\gamma,\Omega^{2\gamma}_{\infty}}
\lesssim_{s}\|u\|_{s}^{\gamma, \Omega^{2\gamma}_{\infty}}+
\gamma^{-3}\big(\e^{13}+\e\|\mathfrak{I}_{\delta}\|^{\gamma, \calO_0}_{s+\mu}\big)
\|u\|_{s_0}^{\gamma, \Omega^{2\gamma}_{\infty}}\,,
\end{equation}
uniformly in  $\tau\in[0,1]$. Finally $\mathcal{T}^{\tau}_{\beta^{(6)}}$ belongs to
$\mathfrak{T}_1$.
\end{lemma}

\begin{proof}
We reason as done in Lemma \ref{wish1}.
Due to the special form of the symbol of order zero in \eqref{diagpezzoord12j}-\eqref{termA1finale12j}
we look for a symbol $\beta^{(6)}(\vphi,x,\x)$ of the form
\[
\beta^{(6)}(\vphi,x,\x):=\sum_{l\in \mathbb{Z}^{\nu}, j\in \mathbb{Z}}
(\beta^{(6)})_{lj}e^{\ii( l \cdot\vphi+jx)}\sign(\x)\,.
\]
Therefore,
passing to the Fourier basis (recall \eqref{realfunctions}),
we have that equation \eqref{equabeta4j} reads (recall \eqref{termA1finale12j})
\begin{equation}\label{equabeta4bisj}
(q^{(*)}_{\geq 9})_{l j}-\ii ({\omega}\cdot l+\mathfrak{m}_1l)(\beta^{(6)})_{l j}=0\,,\qquad \forall \; l\in \mathbb{Z}^{\nu}\,, j\in \mathbb{Z}\setminus\{0\}\,,\quad (l,j)\neq(0,0)\,.
\end{equation}
Using the diophantine condition in \eqref{omegaduegamma}
one can prove
\begin{equation*}
|\beta^{(6)}|_{0,s,\alpha}^{\gamma,\Omega_{\infty}^{2\gamma}}\lesssim_{s}\gamma^{-1}
|q^{(*)}_{\geq 9}|_{0, s+2\tau+1,\alpha}^{\gamma,\Omega_{\infty}^{2\gamma}}
\stackrel{\eqref{iena40a4j}}{\lesssim_s}
\gamma^{-3}
(\e^{13}+\e\|\mathfrak{I}_{\delta}\|_{s+\mu}^{\gamma,\calO})
\end{equation*}
which is the \eqref{iena40beta4j}. The \eqref{iena41beta4j}
follows similarly using \eqref{iena41a4j}.
Since $q^{(*)}_{\geq 9}$ is in $S_{\mathtt{v}}$ (recall that $\mathcal{L}_{5}$ in \eqref{elle5} 
is in $\mathfrak{S}_1$
and the fact that 
$\mathcal{T}_{\beta^{(5)}}^{\tau}$ is in $\mathfrak{T}_1$ is in $\mathfrak{T}_1$ 
by Lemma \ref{lemmaprej}) 
then, by \eqref{equabeta4bis}, we deduce that also $\beta^{(6)}\in S_{\mathtt{v}}$. 
By Lemma \ref{differenzaFlussi} we have
\[
\mathcal{T}^{1}_{\beta^{(6)}}=
\Pi_{S}^{\perp} \Phi_{\beta^{(6)}}^{\tau}\Pi_{S}^{\perp}\circ({\rm Id}+\mathcal{R})
\]
where $\Phi_{\beta^{(6)}}^{\tau}$ is the flow of \eqref{linprob}
 with $f$  as in \eqref{sim3}
with $f\rightsquigarrow \beta^{(6)}$
and where $\mathcal{R}$ is a finite rank operator of the form \eqref{FiniteDimFormDP10001}
satisfying \eqref{giallo2}.
Then by Lemma \ref{CoroDPdiffeo2} we have that
%, if $\gamma_1$ is real valued
%and satisfies \eqref{stimaSoloe2j}, then
the flow $\mathcal{T}^{\tau}_{\beta^{(6)}}$
is well posed and 
the estimate \eqref{stimaMapp4j} follows by estimates \eqref{flow2}-\eqref{flow222}.
Moreover we have that the map $\mathcal{T}_{\beta^{(6)}}^{\tau}$
is in $\mathfrak{T}_1$. 
\end{proof}

\begin{proof}[{\bf Proof of Proposition \ref{propordmezzoj}}]
We consider the map
\begin{equation}\label{MappaT4j}
\mathbf{T}_6^{\tau}:=\left(
\begin{matrix}
\mathcal{T}_{\beta^{(6)}}^{\tau}(\vphi) & 0 \\
0 & \ov{\mathcal{T}_{\beta^{(6)}}^{\tau}(\vphi)}
\end{matrix}
\right)\,,\qquad \tau\in[0,1]\,,
\end{equation}
where $\mathcal{T}_{\beta^{(6)}}^{\tau}(\vphi) $, $\tau\in[0,1]$ 
is given by Lemma \ref{wish1j}.
We now conjugate the operator 
 $\mathcal{L}_{5, *}$ in \eqref{elle4starj} 
  with the map $\mathbf{T}_6^{1}$.
First of all we remark that
the symbol $\beta^{(6)}(\vphi,x,\x)$
%\begin{equation}\label{wish2}
%f(\vphi,x,\x):=\beta^{(4)}(\vphi,x)|\x|^{\frac{1}{2}}
%\end{equation}
belongs to the class $S_{4}^{0}$ (see Def. \ref{espOmogenea}). This is true
because estimates
\eqref{iena40beta4j}, \eqref{iena41beta4j} imply the \eqref{AVBDEF}
%, \eqref{AVB2DEF}
with $k=4$.
We follow the strategy of section \ref{preliminaremezzo}.
By Lemma \ref{georgiaLem} we have that
the conjugate of $\mathcal{L}_{5,*}$ in \eqref{elle4starj} is given by 
\[
\Pi_{S}^{\perp}\widetilde{\mathbf{T}}_6^{1}\mathcal{L}_{5,*}
(\widetilde{\mathbf{T}}_6^{1})^{-1}\Pi_{S}^{\perp}\,,
\qquad
\widetilde{\mathbf{T}}_6^{\tau}:=\left(
\begin{matrix}
{\Phi^{\tau}_{\beta^{(6)}}}(\vphi) & 0 \\
0 & \ov{\Phi_{\beta^{(6)}}^{\tau}(\vphi)}
\end{matrix}
\right)\,,\qquad \tau\in[0,1]\,,
\]
up to finite rank operators belonging to 
$\mathfrak{L}_{\rho, p}^{4}\otimes \mathcal{M}_2(\mathbb{C})$. 
Hence
we shall apply Lemmata \ref{flusso12basso}, \ref{flusso12basso2}
with $f(\vphi,x,\x)$ as in \eqref{wish2}.
By explicit computations using \eqref{coniugato1000}, \eqref{castello1}, \eqref{coniugato1002}, \eqref{castello1tempo} (as done in \eqref{tempomezzo}-\eqref{ord1123})
we obtain that the operator 
\[
\mathcal{L}_{6}:=\mathbf{T}_6^1 \,\mathcal{L}_{5,*} \, (\mathbf{T}_6^1)^{-1}
\]
has the form \eqref{ellej}, %with $j\rightsquigarrow j+1$, 
with some 
remainder $\mathcal{Q}^{(6)}\in\mathfrak{L}^4_{\rho,p}$, some
symbol $\mathtt{d}^{(6)}(\vphi,x,\x)$ of the form
\[
\mathtt{d}^{(6)}(\vphi,x,\x):=
\ii \mathfrak{m}_1\,\x+\ii (1+\mathfrak{m}_\frac{1}{2})|\x|^{\frac{1}{2}}
%+\ii 
%\sum_{k=3}^{j-1}\mathfrak{m}_{\frac{3-k}{2}}(\x)
+\ii A(\vphi,x,\x)+\ii q^{(6)}(\vphi,x,\x)
\,,
\]
where
\begin{align}
A&=q_{*}\sign(\x)-(\omega\cdot\pa_{\vphi}+\mathfrak{m}_1\pa_{x})\beta^{(6)}\,,
\label{alive12j}
\end{align}
and $q^{(6)}\in S_4^{-1/2}$.
By equation \eqref{equabeta4j} in Lemma \ref{wish1j} we actually have that
$A(\vphi,x,\x)$ in \eqref{alive12j} is equal to $\mathfrak{m}_{0}$
in \eqref{mer}\,.
By estimates \eqref{iena40beta4j}, \eqref{iena41beta4j}, and the fact that 
$q_*\in S_3^{0}$
we have that $q^{(6)}(\vphi,x,\x)$ is in $S^{-1/2}_{4}$ because it satisfies estimates like 
\eqref{AVBDEF} %, \eqref{AVB2DEF} 
with $k=4$.

The constant $\mathtt{r}_0$ in \eqref{def:remezzoj}
satisfies the bounds \eqref{merstima}, $i=0$, by \eqref{iena40a4j}, \eqref{iena41a4j}
and \eqref{IpotesiPiccolezzaIdeltaDP}.
To conclude we set (recall \eqref{mappapasso3j}, \eqref{MappaT4j})
\[
{\bf \Phi}_3:=\mathbf{T}_6^{1}\circ\mathbf{T}_5^{1}\,.
\]
The estimate \eqref{stimaMappj} follows by composition
using \eqref{stimaMapp2bisj} and \eqref{stimaMapp4j}. The map ${\bf \Phi}_3\in\mathfrak{T}_1$ because $\mathbf{T}_6^{1}, \mathbf{T}_5^{1}\in\mathfrak{T}_1$.
\end{proof}

\begin{proof}[{\bf Proof of Proposition \ref{riduzSum}}]
We set
\begin{equation}\label{phiphiphi}
{\bf \Phi}:={\bf \Phi}_{3}\circ{\bf \Phi}_{2}\circ{\bf \Phi}_{1}
\end{equation}
where the maps ${\bf \Phi}_{j}$, $j=1,2,3$, are given by Propositions
\ref{proporduno}, \ref{propordmezzo} and \ref{propordmezzoj}. 
The map ${\bf \Phi}$ is defined for $\omega\in \Omega_{\infty}^{2\gamma}$
and belongs to $\mathfrak{T}_1$ since ${\bf \Phi}_{j}\in \mathfrak{T}_1$, $j=1,2,3$.
The \eqref{stimaMappaPhi} follows by \eqref{IpotesiPiccolezzaIdeltaDP}
and \eqref{stimaMapp1}, \eqref{stimaMapp2}, \eqref{stimaMappj}.
By construction we have that
(see \eqref{ellej}, \eqref{diagpezzoJe})
$\mathcal{L}_{6}:={\bf \Phi}\mathcal{L}_3{\bf \Phi}^{-1}$ has the form
\begin{equation}
\mathcal{L}_{6}
=\Pi_{S}^{\perp}\Big(\omega\cdot\pa_{\vphi}+
\opw\left(\begin{matrix} \ii \mathtt{M}(\x)+\ii q^{(6)}(\vphi,x,\x) & 0 \\
0 & -\ii \mathtt{M}(\x)-\ii\ov{q^{(6)}(\vphi,x,-\x)}  \end{matrix}\right)+\mathcal{Q}^{(6)}\Big)
\end{equation}
with $\mathtt{M}(\x)$ as in \eqref{simboM}, and
where $q^{(6)}(\vphi,x,\x)\in S_{4}^{-\frac{1}{2}}$ and 
$\mathcal{Q}^{(6)}\in \mathfrak{L}^{4}_{\rho,p}\otimes\mathcal{M}_2(\mathbb{C})$ 
(recall Def. \ref{espOmogenea}).
The \eqref{formaRestoQQ} follows by setting $\mathtt{q}=q^{(6)}$
and $\mathcal{Q}=\mathcal{Q}^{(6)}$.
%In particular by  \eqref{mer}, \eqref{mer12}, \eqref{mer12j}
%we deduce the \eqref{espandoSimboM}.
%In particular by  \eqref{mer}, \eqref{mer12}, \eqref{mer12j}
%we deduce the \eqref{espandoSimboM}. Moreover
%(see  \eqref{espandoenonmifermo}, \eqref{espandoenonmifermo2}) 
%we have
%\[
%q^{(6)}=\sum_{i=1}^{6}\e^i q^{(6)}_i+q^{(6)}_{\geq7}\,,
%\qquad
%\mathcal{Q}^{(6)}=\sum_{i=1}^{6}\e^i \mathcal{Q}^{(6)}_1+\mathcal{Q}^{(6)}_{\geq7}
%\]
%where $q^{(6)}_i$, $\mathcal{Q}^{(6)}_i$, $i=1,\ldots,6$, 
%are respectively $i$-homogeneous symbols
%and operators as in \eqref{simboOMO2DEF}, \eqref{simboOMO2DEF2}
%By Lemma $C.7$ in \cite{FGP1}
%we have that 
%\[
%\mathtt{Q}_{\geq7}:=\opw(q^{(6)}_{\geq7})+\mathcal{Q}^{(6)}_{\geq 7}\in \mathfrak{C}_{-\frac{1}{2}}
%\]
%and 
%\begin{align}
%&\mathbb{B}^{\g}_{\mathtt{Q}_{\geq7}}(s)\lesssim_{s}
%|q^{(6)}_{\geq3}|_{-\frac{1}{2},s+\s,\alpha}^{\gamma,\Omega_{\infty}^{2\gamma}}+
%\mathbb{M}^{\gamma}_{\mathcal{Q}^{(6)}_{\geq3}}(s, \mathtt{b})
%\,,\label{scossa3}\\
%&\mathbb{B}_{\Delta_{12} \mathtt{Q}_{\geq7}}(p)\lesssim_{p}
%|q^{(6)}_{\geq3}|_{-\frac{1}{2},s+\s,\alpha}^{\gamma,\Omega_{\infty}^{2\gamma}}+
%\mathbb{M}^{\gamma}_{\Delta_{12}\mathcal{Q}^{(6)}_{\geq7}}(s, \mathtt{b})
%\|i_{1}-i_{2}\|_{p+\s}\,,\label{scossa4}
%\end{align}
%for some $\s=\s(\tau,\nu)>0$ and $\mathtt{b}$ defined in Def. \ref{Cuno}.
%By estimates
%\eqref{AVBDEF}, \eqref{AVB2DEF}, \eqref{AVBDEF2}, \eqref{AVB2DEF2}
%with $k=4$
%we have that the \eqref{scossa3}, \eqref{scossa4}
%imply the \eqref{scossa1}, \eqref{scossa2}.
\end{proof}

%%%%%%%%%%%%%%%%%%%%%%%%%%%%%%%%%%%%%%%%%%%%%%%%%%%%%%%%%%%

\section{Linear Birkhoff normal form}\label{sec:linBNF}

In this section the main goal is to normalize all the terms in the remainder 
$\mathtt{Q}$ in \eqref{elle6} which are not ``\emph{perturbative}''
for the KAM procedure of section  \ref{sez:KAMredu}.
This will be obtained by applying several changes of coordinates involving 
``small divisors'' (see subsections \ref{stepsLower}, \ref{stepsHigher}).
In particular, thanks to the abstract result stated in Proposition \ref{identificoBNF}, we will compute explicitly the corrections
of size $O(\e^2)$ to the eigenvalues of the operator $\mathcal{L}_6$ in \eqref{elle6}. Such corrections will play a fundamental role in the estimates of the measure of the resonant sets of frequencies.

The key result of the section 
is the following.

\begin{proposition}{\bf (Linear Birkhoff normal form).}\label{ConcluLinear}
There are constants $r_j: \Omega_{\varepsilon}\to \mathbb{R}$, independent of $\mathfrak{I}_{\delta}$, 
and constants $\mathtt{C}_{1},\mathtt{C}_{2}>0$ depending only on $S$ such that
for any $\omega\in\mathcal{G}_0^{(2)}(\mathtt{C}_1, \mathtt{C}_2)$ 
(see \eqref{G02})
there exists a symplectic and $x$-translation invariant  
map $\Upsilon(\varphi)\colon H^s_{S^{\perp}}(\T)\times H^s_{S^{\perp}}(\T) 
\to H^s_{S^{\perp}}(\T)\times H^s_{S^{\perp}}(\T)$ such that (recall \ref{elle6})
\begin{equation}\label{elle11fine}
\mathcal{L}_{14}:=\Upsilon \mathcal{L}_{6} \Upsilon ^{-1} 
=\Pi_{S}^{\perp}\big(
\omega\cdot \partial_{\varphi}
+\mathfrak{D}
+\mathfrak{R}\big)\,,
\end{equation}
where
%$\mathfrak{D}:=\mbox{diag}_{j\in S^{c}}(\mathrm{i} \,d_j)$ with
\begin{align}
&\mathfrak{D}:=(\mathfrak{D}_{\s}^{\s'})_{\s,\s'=\pm}\,, \quad \mathfrak{D}_{\s}^{-\s}\equiv0\,,\;\; 
\ov{\mathfrak{D}_{+}^{+}}=\mathfrak{D}_{-}^{-}
%\left(
%\begin{matrix}
%\mathfrak{D}_{+}^{+} & 0\\ 0 & \mathfrak{D}_{-}^{-}
%\end{matrix}
%\right)
\,, \qquad \mathfrak{D}_{+}^{+}:=\mbox{diag}_{j\in S^{c}}(\mathrm{i} \,d_j)\,,   \label{diagD}  \\ 
d_j&:=\mathfrak{m}_1(\omega)\,j+(1+\mathtt{m}_{\frac{1}{2}}^{(\geq4)}(\omega))
\sqrt{\lvert j \rvert}+\mathtt{m}_{0}^{(\geq4)}(\omega)\sign(j)
+r_j(\omega)\,,\quad d_{j}\in \mathbb{R}\,, \label{finaleigen}\\
&\sup_{j\in S^{c}}\langle j \rangle^{1/2}\lvert r_j \rvert^{\gamma, \Omega_{\varepsilon}}\lesssim \varepsilon^3\,, \quad r_j\in\mathbb{R} \label{accendo2}
\end{align}
where $\mathfrak{m}_1$, $\mathtt{m}_{1/2}^{(\geq4)}$, 
$\mathtt{m}_{0}^{(\geq4)}$ are given in \eqref{mer}.
%The eigenvalues $r_j$ are independent of $\mathfrak{I}_{\delta}$. Finally 
The operator  $\mathfrak{R}$ is in $\mathfrak{S}_0$ (see Def. \ref{compaMulti})
%Hamiltonian, $x-$translation invariant 
and 
is $Lip$-$(-1/2)$-modulo tame  (see Def. 
\ref{menounomodulotamesenzaLip}, \ref{menounomodulotame}) 
with the following estimates
\begin{equation}\label{scossa1RR5}
\begin{aligned}
&  {\mathfrak M}_{\mathfrak{R}}^{{\sharp, \g}} (-1/2,s,\mathtt{b}_0) 
\lesssim_{s}\gamma^{-3}(\e^{13}
+\e\|\mathfrak{I}_{\delta}\|^{\gamma,\calO_0}_{s+\mu})\,,\\
&{\mathfrak M}_{\Delta_{12}\mathfrak{R}}^{{\sharp}} (-1/2,p,\mathtt{b}_0) 
%\|\lD^{1/2}\underline{\Delta_{12} \mathfrak{R} }\lD^{1/2}\|_{\mathcal L(H^{s_0})}, \,\, \|\lD^{1/2}\underline{\Delta_{12} \langle \partial_{\f}\rangle^{{\mathtt b_0}} 
% \mathfrak{R}  }\lD^{1/2}\|_{\mathcal L(H^{s_0})} 
 \lesssim_p  \e\gamma^{-3}(1+\|\mathfrak{I}_{\delta}\|_{p+\mu})
\|i_{1}-i_{2}\|_{p+\mu}\,,
\end{aligned}
\end{equation}
for some $\mu=\mu(\nu)>0$.
The map $\Upsilon$ does not depend on $\mathfrak{I}_{\delta}$
and satisfies 
\begin{equation}\label{boundUpsilonTotale}
\lVert \Upsilon^{\pm 1}u \rVert_s\lesssim_{s,S} 
\lVert u \rVert_{s}\,.%+\e\lVert u \rVert_{s_0}
\end{equation}
\end{proposition}

The proof of the proposition above involves many arguments and will be divided into
several steps.
The basic idea is to implement a ``linear'' Birkhoff procedure 
as explained in subsection \ref{compaBNFpro}.
This will require some ``non-resonance'' conditions
which are studied in subsection \ref{nonRescondLINEAR}.
A fundamental argument in the proof is the identification of normal forms
result contained in Proposition \ref{identificoBNF}.
This result implies that the corrections at order $\e^{2}$
in the eigenvalues \eqref{finaleigen} depend only 
on the constant $m_1$ in \eqref{constM1}.

\subsection{Preliminary technical results}\label{preliminareLinearBNF}

In the following subsection we shall prove some technical results which we need to implement the 
\emph{Linear Birkhoff normal form} procedure.

We shall study the ``non-resonance'' conditions
required on the small divisors (see subsection \ref{nonRescondLINEAR});
then we will introduce the 
class of \emph{almost-diagonal} operators 
(see subsection \ref{almostOPERATORS}).

\subsubsection{Non-resonance conditions}\label{nonRescondLINEAR}
Recalling the vectors  $\bar{\omega}$  in \eqref{LinearFreqDP}
and $\mathtt{v}$ is in \eqref{def:m1} we define for $p=1, \dots, 6$
\begin{align}\label{deltaUNO}
& \delta^{(p)}_{\sigma, \sigma', j, k}(\ell)
:=\overline{\omega}\cdot \ell+\sigma \sqrt{\lvert j \rvert}-\sigma' \sqrt{\lvert k\rvert}\,,
\qquad \forall\; \s,\s'=\pm\,,\; j,k\in S^{c}\,,\;\; \ell\in \mathbb{Z}^{\nu}\,,
\\
&{\rm with}\qquad
\mathtt{v}\cdot \ell+j- k=0\,, \quad |\ell|\leq p\,.\label{deltaUNO2}
\end{align}

\begin{remark}\label{max12}
Given a set of indexes $j_1, \dots, j_N$, we denote by $\max(1)$ and $\max(2)$ the first and the second largest values among the moduli of the $j_i$'s.\\  
If \,$\sum_{i=1}^N \sigma_i j_i=0$ 
for some $N>0$, $j_i\in\mathbb{Z}\setminus\{0\}$, 
$\sigma_i\in \{ \pm 1\}$, then 
$\max(2)\geq N^{-1} \max(1)$. 
Indeed, suppose that 
$\lvert j_1 \rvert=\max(1)$ and $\lvert j_2 \rvert=\max(2)<\max(1)\,/\,N$, 
then 
\[
\max(1)=\lvert \sigma_1 j_1\rvert\leq 
\sum_{i=2}^N \lvert \sigma_i j_i \rvert\le \frac{(N-1)}{N}\,\max(1)\,,
\]
which is a contradiction.
\end{remark}

\begin{lemma}{\bf (Lower bounds).}\label{lowerboundDel1}
$(i)$
There exist  constants $C_1(S),C_2(S)>0$ such that
\begin{equation}\label{lowerboundDel1BIS}
\lvert \delta^{(1)}_{\sigma, \sigma', j, k}(\ell) \rvert \geq C_1(S)\,, 
\qquad \forall \sigma, \sigma'\in\{ \pm 1\}\,, 
\,\,j, k\in\mathbb{Z}\setminus\{0\},\,\,\ell\in\mathbb{Z}^{\nu}, \lvert \ell \rvert=1\,,
\end{equation}
except for $\sigma=\sigma'$, $j=k$, $\ell=0$
and 
\begin{equation}\label{lowerboundDel2BIS}
\lvert \delta^{(2)}_{\sigma, \sigma', j, k}(\ell) \rvert \geq C_2(S)\,, 
\qquad \forall \sigma, \sigma'\in\{ \pm 1\}\,, 
\,\,j, k\in\mathbb{Z}\setminus\{0\},\,\,\ell\in\mathbb{Z}^{\nu}, \lvert \ell \rvert=2\,,
\end{equation}
except for $\sigma=\sigma'$, $j=k$, $\ell=0$.

\noindent
$(ii)$ For $3\le p\leq 6$ and  a ``\emph{generic}'' choice of the tangential sites 
$S$ in \eqref{TangentialSitesDP}
there exist constants $C(S)\gg1$, $0<K(S)\ll1$ 
depending only on the $\max(S)$ (see \eqref{maxS})
such that (recall \eqref{deltaUNO})
\begin{equation}\label{smallDivLinearBNF2}
|\delta_{\s,\s',j,k}^{(p)}(\ell)|\geq K(S)\,,
\qquad 0< |\ell|\leq p\,, \;\;|j|, |k|\geq C(S)\,,
\end{equation}
except for $\sigma=\sigma'$, $p$ even, 
 $j=k$, $\ell=0$.
\end{lemma}
\begin{proof}
See Appendix \ref{10.3proofLem}.
\end{proof}

\subsubsection{Almost-diagonal operators}\label{almostOPERATORS}
We need further definitions.
\begin{definition}{\bf (Almost diagonal operators).}\label{def:AlmostDiag}
We say that ${\bf A}(\vphi)\colon H_{S^{\perp}}^s(\mathbb{T})
\times H_{S^{\perp}}^s(\mathbb{T})
\to H_{S^{\perp}}^s(\mathbb{T})\times H_{S^{\perp}}^s(\mathbb{T})$ 
is  $N$-\emph{almost-diagonal} 
if,  $\forall j, k\in S^{c}$, $\s,\s'=\pm$, one has
\begin{equation*}%\label{almostDD}
(\mathbf{A})_{\sigma, j}^{\sigma', k}(\ell)\neq 0
\qquad \Rightarrow \qquad  \lvert \ell \rvert\le N\,.
\end{equation*}
 We say that ${\bf A}(\vphi)$ is  \emph{almost-diagonal} 
 if it is $N$-almost-diagonal for some $N>0$. 
\end{definition}

\begin{remark}\label{homogalmostdiag}
Notice that
the following facts hold:

\noindent
$\bullet$  An  $i$-homogeneous operator of the form \eqref{simboOMO2DEF2} 
is $i$-almost-diagonal and $x$-translation invariant;

\noindent
$\bullet$ if ${\bf A}(\vphi)$ %of the form \eqref{almostEsempio}
is $x$-translation invariant (according to Def. \ref{admiOpComp}) and 
almost-diagonal, then
(recall Lemma \ref{lem:momentocomp})
\[
{\rm if} \qquad 
(\mathbf{A})_{\sigma, j}^{\sigma', k}(\ell)\neq0 \qquad
\Rightarrow \;\;\exists\, C=C(S)>0 \quad {\rm such\; that} 
\quad |\ell|\leq C\,, |j-k|\leq C\,;
\]

\noindent
$\bullet$ if ${\bf A}(\vphi)$ is
%of the form \eqref{almostEsempio}
%is $x$-translation invariant, 
almost-diagonal and bounded on 
$H_{S^{\perp}}^s(\mathbb{T})\times H_{S^{\perp}}^s(\mathbb{T})$,
then it is tame according with Def. \ref{TameConstants}.
This follows since ${\bf A}$ has ``decay'' off-diagonal.
% (see Lemma $2.2$ in \cite{Airy}).
\end{remark}

%
%\begin{lemma}\label{homogalmostdiag}
%A $i$-homogeneous operator of the form \eqref{simboOMO2DEF2} is almost-diagonal.
%\end{lemma}
%
%\comment{da rivedere la def 9.1}

\begin{lemma}\label{almostC12}
Let $m\geq 0$ and  $\mathbf{A}$ be an almost-diagonal, 
$x$-translational invariant operator.
Then the operator 
$\mathbf{B}:=\langle D \rangle^{m}\,\mathbf{A}\,\langle D \rangle^{m}$
is bounded in the majorant norm (see Def. \ref{opeMagg}) topology if and only if 
% belonging to $\mathfrak{C}_{-(1/2)}$. 
% Then there exists a constant $C>0$ depending on $S$ such that
\begin{equation}\label{stimaCoeffalmost}
\lvert (\mathbf{A})_{\sigma, j}^{\sigma', k}(\ell) \rvert \leq 
\frac{C}{\langle j, k\rangle^{2m}}\,,\quad \ell\in \mathbb{Z}^{\nu}\,,\;\; j,k\in S^{c}\,,
\end{equation}
for some constant $C>0$ depending only on $S$.
\end{lemma}
\begin{proof}
See Appendix \ref{10.6prooflemma}.
\end{proof}

\begin{lemma}\label{LemmaAggancio}
Fix $\mathtt{b}_0>0$.
If $\rho\gg s_0+\mathtt{b}_0$ then
the operator $\mathtt{Q}$ in \eqref{formaRestoQQ} has the form
\begin{equation}\label{formaResto}
\mathtt{Q}=\sum_{i=1}^{6}\e^{i} \mathtt{Q}_i+\mathtt{Q}_{\geq7}
\end{equation}
with $\mathtt{Q}_i$, $i=1,\ldots,6$,  $i$-homogeneous as in \eqref{simboOMO2DEF2}
(see Def. \ref{espOmogenea}). In particular
one has
\begin{equation}\label{formaResto2}
|(\mathtt{Q}_i)_{\s,j}^{\s,k}(\ell)|\leq \frac{C}{\langle j,k\rangle^{1/2} }\,,
\qquad |(\mathtt{Q}_i)_{\s,j}^{-\s,k}(\ell)|\leq \frac{C}{\langle j,k\rangle^{\rho} }\,,
\end{equation}
for some $C>0$ dependent on the tangential sites $S$. 
Moreover, for any $\omega\in \Omega_{\infty}^{2\gamma}$ 
the operator $\mathtt{Q}_{\geq 7}$
is Lip-$(-1/2)$-modulo tame (see Def. \ref{def:op-tame})
\begin{equation}\label{scossa1}
\begin{aligned}
&{\mathfrak M}_{\mathtt{Q}_{\geq7}}^{{\sharp, \g}} (-1/2,s,\mathtt{b}_0) 
\lesssim_{s}\gamma^{-3}(\e^{13}
+\e\|\mathfrak{I}_{\delta}\|^{\gamma,\calO_0}_{s+\mu})\,,\\
&
%\|\lD^{1/2}\underline{\Delta_{12}\mathtt{Q}_{\geq7} }\lD^{1/2}\|_{\mathcal L(H^{s_0})}, \,\, \|\lD^{1/2}\underline{\Delta_{12} \langle \partial_{\f}\rangle^{{\mathtt b_0}} 
%\mathtt{Q}_{\geq7}  }\lD^{1/2}\|_{\mathcal L(H^{s_0})} 
{\mathfrak M}_{\Delta_{12}\mathtt{Q}_{\geq7}}^{\sharp}(-1/2,p,\mathtt{b}_0)
\lesssim_p \e\gamma^{-3}
(1+\|\mathfrak{I}_{\delta}\|_{p+\mu})
\|i_{1}-i_{2}\|_{p+\mu}\,,
%\label{scossa2}
\end{aligned}
\end{equation}
for some $\mu=\mu(\nu)>0$.
\end{lemma}
\begin{proof}
See Appendix \ref{10.7prooflemma}.
\end{proof}

%We give the following definition.

\begin{definition}\label{ProiettoCCC}
Let ${\bf M}={\bf M}(\vphi)$ be a matrix of operators.
$(i)$ Given a constant $C>0$ we set
\begin{equation*}%\label{splittoOperatori}
{\bf M}=\pi_C\,{\bf M}+\pi_C^{\perp}{\bf M}\,,
\end{equation*}
where, for any $\ell\in\mathbb{Z}^{\nu}$, $ j,k\in S^{c}$, we define
\begin{equation*}%\label{splittoOperatori2}
(\pi_{C}^{\perp}{\bf M})_{\s,j}^{\s',k}(\ell):=\left\{
\begin{aligned}
&{\bf M}_{\s,j}^{\s',k}(\ell)\,,\quad \max\{ |j|,|k|\}\geq C\\
&0 \qquad \qquad {\rm otherwise}\,.
\end{aligned}\right.
\end{equation*}
The operator $\pi_{C}{\bf M}$ is defined by difference.

\noindent
$(ii)$ We define the operator $\bral \bf M \brar$ as
\begin{equation}\label{doppiebra}
(\bral {\bf M} \brar)_{\s, j}^{\s', k}(\ell):=\begin{cases}
({\bf M})_{\s, j}^{\s, j}(0) \quad \mbox{if}\,\,\s=\s', \,\,\,j=k,\,\,\,\ell=0,\\
0 \qquad \qquad \qquad \mbox{otherwise}\,.
\end{cases}
\end{equation}
\end{definition}

\begin{remark}\label{propAlgebricheSplitto}
Notice that, if the operator ${\bf M}$  is in $\mathfrak{S}_0$ 
(see Def. \ref{compaMulti}),
then 
$\pi_{C}{\bf M}$, $\pi_{C}^{\perp}{\bf M}$ belong to  $\mathfrak{S}_0$.
\end{remark}
%\begin{proof}
%It is sufficient to prove that $\pi_C {\bf M}\in \mathfrak{S}_0$. First we study the reality of $\pi_C {\bf M}$. We use Lemma \ref{coeffFouHamilto}, \eqref{forma-complessa}.
%\[
%(\pi_C {\bf M})_{+, j}^{+, k}(\ell)=\begin{cases}
%{\bf M}_{+, j}^{+, k}(\ell) \qquad \mbox{if}\,\,\, \max\{ |j|, |k|  \}\le C\\
%0
%\end{cases}
%\]
%Since ${\bf M}$ is real we have that ${\bf M}_{+, j}^{+, k}(\ell)=\ov{\mathbf{M}_{-, -j}^{-, -k}(-\ell)}$
%\end{proof}

We shall use the following notation:
given some operator $M=M(\vphi)$ we write
\begin{equation}\label{ADAD}
{\rm ad}_{{\bf A}_i}[M]:=[{\bf A}_i, M]={\bf A}_i M- M{\bf A}_i\,.
\end{equation}
The following result is fundamental for our scope.
\begin{lemma}\label{lemmaeqomologica}
Let ${\bf B}$ be a smoothing operator 
in $\mathfrak{S}_0$ and 
$p$-almost-diagonal  satisfying
 estimates like \eqref{formaResto2}.
%belonging to the class $\mathfrak{C}_{-1/2}$.
Consider the function $\delta_{\s,\s',j,k}^{(p)}(\ell)$ in 
 \eqref{deltaUNO} and let 
 \begin{equation}\label{costantefond}
 \mathtt{C}:=\mathtt{C}(S):=
 \max\{12\,\lvert \mathtt{v}\rvert, 2\, C(S)\,, 
 2 \,\max(S) \}>0\,,
 \end{equation}
where $C(S)$ is the constant given by item $(ii)$ of Lemma \ref{lowerboundDel1}.
We define
the operator ${\bf A}={\bf A}(\vphi)=\pi_{\mathtt{C}}^{\perp}{\bf A}(\vphi)$ as (see Def. \ref{ProiettoCCC})
\begin{equation}\label{omoeqCCC}
(\pi_{\mathtt{C}}^{\perp}{\bf A})_{\s,j}^{\s',k}(\ell):=\left\{
\begin{aligned}
&{\bf B}_{\s,j}^{\s',k}(\ell)/ \ii \delta^{(p)}_{\s,\s',j,k}(\ell)\,,\quad \max\{|j|,|k|\}\geq \mathtt{C}\,, \ell\neq0\\
&0 \qquad \qquad \qquad \qquad \qquad  {\rm otherwise}\,.
\end{aligned}\right.
\end{equation}
Then ${\bf A}(\vphi)$ 
is in $\mathfrak{S}_0$ and 
$p$-almost-diagonal
 and solves the equation 
 %(see 
 %\eqref{espandoSimboM}, \eqref{espandoSimboM22})
 \begin{equation}\label{omoeqCCC2}
 -{\rm ad}_{ \bar{\omega}\cdot\pa_{\vphi}
 +\mathtt{D}_0} \big[ \mathbf{A}\big] +{\bf B}=
 \bral \pi_{\mathtt{C}}^{\perp}{\bf B}\brar+\pi_{\mathtt{C}}{\bf B}\,,\qquad
 \mathtt{D}_0:= \opw\left(\begin{matrix} \ii \lvert \xi \rvert^{1/2} & 0 \\
0 & -\ii \lvert \xi \rvert^{1/2} \end{matrix}\right)\,.
  \end{equation}
 % where $\mathtt{D}_0=\opw()$
%recall \eqref{doppiebra}.
Moreover one has %(see Def. \ref{Cuno})
\begin{equation}\label{omoeqCCC4}
|{\bf A}_{\s,j}^{\s,k}(\ell)|\leq (K(S))^{-1} \frac{C}{\langle j, k \rangle^{1/2}}\,, \qquad |{\bf A}_{\s,j}^{-\s,k}(\ell)|\leq (K(S))^{-1} \frac{C}{\langle j, k \rangle^{\rho}}
%\mathbb{B}^{\g}_{{\bf A}}(s)\lesssim_{s,S} 
%\mathbb{B}^{\g}_{{\bf B}}(s)\,.
\end{equation}
where $K(S)$ is in \eqref{smallDivLinearBNF2} and $C>0$ is a constant depending on the tangential sites.
\end{lemma}
\begin{proof}
See Appendix \ref{10.10prooflemma}.
\end{proof}

\begin{lemma}\label{compoC12}
Let ${\bf A}\in \mathfrak{S}_0$ be   
$i$-almost-diagonal   satisfying 
\begin{equation}\label{formaResto2AA}
|({\bf A})_{\s,j}^{\s,k}(\ell)|\leq \frac{C({\bf A})}{\langle j,k\rangle^{1/2} }\,,
\qquad |({\bf A})_{\s,j}^{-\s,k}(\ell)|\leq \frac{C({\bf A})}{\langle j,k\rangle^{\rho} }\,,
\end{equation}
for some constant $C({\bf A})>0$ and some $\rho>0$.
Then the following holds.

\noindent
$(i)$ The operator ${\bf A}$ is Lip-$(-1/2)$-modulo tame with 
\begin{equation}\label{UPC1}
\mathfrak{M}_{{\bf A}}^{\sharp,\gamma}(-1/2, s,\mathtt{b_0})\lesssim_{s} C({\bf A})\,.
\end{equation}

\noindent
$(ii)$ Let ${\bf B}\in \mathfrak{S}_0$ be 
a  Lip-$(-1/2)$-modulo tame operator 
(see Def. \ref{def:op-tame}). The operator ${\bf C}_{k}:={\rm ad}_{{\bf A}}^{k}[{\bf B}]$, $k\geq 1$ 
is in $\mathfrak{S}_0$,  Lip-$(-1/2)$-modulo tame and 
\begin{align}
\mathfrak{M}_{{\bf C}_{k}}^{\sharp,\gamma}(-1/2, s,\mathtt{b_0})
&\lesssim_{s} \mathfrak{M}_{{\bf A}}^{\sharp,\gamma}(-1/2, s,\mathtt{b_0}) \,
(\mathfrak{M}_{{\bf A}}^{\sharp,\gamma}(-1/2, s,\mathtt{b_0}))^{k-1}\,
\mathfrak{M}_{{\bf B}}^{\sharp,\gamma}(-1/2, s,\mathtt{b_0})\nonumber
\\&\qquad\qquad+( \mathfrak{M}_{{\bf A}}^{\sharp,\gamma}(-1/2, s,\mathtt{b_0}))^{k}\,
\mathfrak{M}_{{\bf B}}^{\sharp,\gamma}(-1/2, s,\mathtt{b_0})\,,\nonumber
%\label{chegioia1}
\\
\mathfrak{M}_{\Delta_{12}{\bf C}_{k}}^{\sharp}(-1/2, p,\mathtt{b_0})
&\lesssim_{p} \mathfrak{M}^{\sharp}_{{\bf B}(i_1)}(-1/2, p,\mathtt{b_0})\times
\nonumber\\
&\times
\sum_{j_1+j_2=k-1} 
(\mathfrak{M}_{{\bf A}(i_1)}^{\sharp,\gamma}(-1/2, p,\mathtt{b_0}))^{j_1}
\mathfrak{M}_{\Delta_{12}{\bf A}}^{\sharp}(-1/2, p,\mathtt{b_0})
(\mathfrak{M}_{{\bf A}(i_2)}^{\sharp,\gamma}(-1/2, p,\mathtt{b_0} ))^{j_2}\nonumber
%\label{chegioia}
\\&\qquad\qquad  +
(\mathfrak{M}_{{\bf A}(i_2)}^{\sharp,\gamma}(-1/2, p,\mathtt{b_0}))^{k}
\mathfrak{M}_{\Delta_{12}{\bf B}}^{\sharp}(-1/2, p,\mathtt{b_0})\,.\nonumber
\end{align}
\noindent
$(iii)$ Let $n\geq 1$ and consider ${\bf F}:=\sum_{k\geq n} (\e^k/ k!)\,\, {\bf C}_k$. For $\varepsilon$ small enough we have that ${\bf F}\in \mathfrak{S}_0$ and Lip-$(-1/2)$-modulo tame with the following bounds
\begin{equation}
\begin{aligned}\label{stimaserielie}
\mathfrak{M}^{ \sharp, \gamma}_{\bf F}(-1/2, s,\mathtt{b_0}) &
\lesssim_s  \mathfrak{M}_{{\bf A}}^{\sharp,\gamma}(-1/2, s,\mathtt{b_0}) 
\,\mathfrak{M}_{{\bf B}}^{\sharp,\gamma}(-1/2, s,\mathtt{b_0})
\varepsilon^n (\mathfrak{M}_{{\bf A}}^{\sharp,\gamma}
(-1/2, s,\mathtt{b_0}))^{n-1} \\
&
+\mathfrak{M}_{{\bf B}}^{\sharp,\gamma}(s,\mathtt{b_0})
\varepsilon^n (\mathfrak{M}_{{\bf A}}^{\sharp,\gamma}(-1/2, s,\mathtt{b_0}))^{n} ,
\end{aligned}
\end{equation}
\begin{equation}\label{stimaserieliedelta12}
\begin{aligned}
\mathfrak{M}_{\Delta_{12}{\bf F}}^{\sharp}(-1/2, p, \mathtt{b_0})&\lesssim_{p}
 \mathfrak{M}^{\sharp}_{{\bf B}(i_1)}(-1/2, p, \mathtt{b_0})
\mathfrak{M}_{\Delta_{12}{\bf A}}^{\sharp}(-1/2, p, \mathtt{b_0}) +
\mathfrak{M}_{\Delta_{12}{\bf B}}^{\sharp}(-1/2, p, \mathtt{b_0}).
\end{aligned}
\end{equation}
\end{lemma}

\begin{proof}
See Appendix \ref{10.11prooflemma}.
\end{proof}

Now consider the operator $\mathtt{D}$ in \eqref{simboM}. 
We write (recall \eqref{mer})
\begin{equation}\label{espandoSimboM}
\begin{aligned}
&\mathtt{D}:=\mathtt{D}_0+\e^{2}\mathtt{D}_2+\e^{4}\mathtt{D}_{4}+\e^{6}\mathtt{D}_{6}+
\mathtt{D}_{\geq8}\,,
\qquad \mathtt{D}_0:=\opw\left(\begin{matrix} \ii \lvert \xi \rvert^{1/2} & 0 \\
0 & -\ii \lvert \xi \rvert^{1/2} \end{matrix}\right)\,,
\\
&\mathtt{D}_2:=\opw\left(\begin{matrix} \ii \mathtt{M}_{2}(\x)& 0 \\
0 & -\ii \mathtt{M}_2(-\x) \end{matrix}\right)\,,\quad
\ii \mathtt{M}_2(\x):=
\ii {m}_1\,\x+\ii {m}_{\frac{1}{2}}|\x|^{\frac{1}{2}}
+\ii {m}_0\sign(\x)\,,\\
&\mathtt{D}_{2k}:=\opw\left(\begin{matrix} \ii \mathtt{M}_{2k}(\x)& 0 \\
0 & -\ii \mathtt{M}_{2k}(-\x) \end{matrix}\right)\,,\quad
\ii \mathtt{M}_{2k}(\x):=
\ii {m}_1^{(2k)}\,\x+\ii {m}_{\frac{1}{2}}^{(2k)}|\x|^{\frac{1}{2}}
+\ii {m}_0^{(2k)}\sign(\x)\,,
\end{aligned}
\end{equation}
with $k=2,3$, 
and $\mathtt{D}_{\geq8}:=\mathtt{D}-\mathtt{D}_0-\e^{2}\mathtt{D}_2-\e^{4}\mathtt{D}_4
-\e^{6}\mathtt{D}_6$.
In the following Lemma we study how the operator $\mathtt{D}$
conjugates
under a map generated by an almost-diagonal operator.

\begin{lemma}\label{mercato}
Let ${\bf A},{\bf B}\in \mathfrak{S}_0$ and assume that
${\bf A}$ is $i_1$-almost-diagonal, 
${\bf B}$ is $i_2$-almost-diagonal and 
that they satisfy estimates like \eqref{formaResto2AA}.
Then we have:

\noindent
$(i)$ the operator ${\bf C}:={\rm ad}_{{\bf A}}[{\bf B}]$ is in $\mathfrak{S}_0$,
is $(i_1+i_2)$-almost-diagonal and
\begin{equation}\label{UPC2}
|{\bf C}_{\s,j}^{\s,k}(\ell)|\lesssim\frac{C({\bf C})}{\langle j,k\rangle^{\frac{1}{2}}}\,,
\qquad
|{\bf C}_{\s,j}^{-\s,k}(\ell)|\lesssim\frac{C({\bf C})}{\langle j,k\rangle^{\rho}}\,,
\end{equation}
for any $j,k\in S^c$, $\ell\in \mathbb{Z}^{\nu}$, $\s=\pm$
for some constant $C({\bf C})$ depending on $S$ and the constant 
$C({\bf A})$, $C({\bf B})$.

\noindent
(ii) Let $\mathcal{D}:=\mathtt{D}-\mathtt{D}_0-\mathtt{D}_{\geq8}$ 
.
Define ${\bf M}:={\rm ad}_{{\bf A}}[\mathcal{D}]$.
%and 
% ${\bf N}:={\rm ad}^{2}_{{\bf A}}[\mathcal{D}]={\rm ad}_{{\bf A}}[{\bf M}]$.
 Then ${\bf M}$ belongs to $\mathfrak{S}_0$ and 
\begin{equation}\label{formaResto2AAbis}
|({\bf M})_{\s,j}^{\s,k}(\ell)|\leq 
\frac{\e^{2}C}{\langle j,k\rangle^{1/2} }\,,
\quad |({\bf M})_{\s,j}^{-\s,k}(\ell)|\leq \frac{\e^{2}C}{\langle j,k\rangle^{\rho-1} }\,.
\end{equation}

\noindent
$(iii)$ Let ${\bf N}:={\rm ad}_{{\bf A}}[\mathtt{D}_{\geq8}]$.
Then ${\bf N}\in \mathfrak{S}_0$ is Lip-$(-1/2)$-modulo tame with
\begin{equation*}%\label{chegioia10}
\begin{aligned}
\mathfrak{M}_{{\bf N}}^{\sharp,\gamma}(-1/2, s,\mathtt{b_0})&\lesssim_{s}
\e^{17-2b}\gamma^{-3}\,,\\
\mathfrak{M}_{\Delta_{12}{\bf N}}^{\sharp}(-1/2, p, \mathtt{b}_0)
&\lesssim_{p}
\e^{17-2b}\gamma^{-3}+C({\bf A})
\e\gamma^{-2}(1+\e\|\mathfrak{I}_{\delta}\|_{p+\mu})
\|i_1-i_{2}\|_{p+\mu}\,.
\end{aligned}
\end{equation*}
\end{lemma}
\begin{proof}
See Appendix \ref{10.12prooflemma}.
\end{proof}

\subsection{Terms \texorpdfstring{$O(\varepsilon), O(\varepsilon^2)$}{O(e),O(e2)} and identification of normal forms}\label{stepsLower}
In this section we perform the first two steps of linear Birkhoff normal form
in order to normalize the terms $O(\e)$ and $O(\e^{2})$
in \eqref{elle6}.
%We have the following.
%\begin{proposition}\label{propLBNF}
%Assume the \eqref{IpotesiPiccolezzaIdeltaDP}.
%For any $\omega\in \Omega_{\infty}^{2\gamma}$ (see \eqref{omegaduegamma})
%there is a map ${\Upsilon}(\omega t)\colon H^s_{S^{\perp}}(\mathbb{T})\times H^s_{S^{\perp}}(\mathbb{T})\to H^s_{S^{\perp}}(\mathbb{T})\times H^s_{S^{\perp}}(\mathbb{T})$
%satisfying (recall \eqref{espandoSimboM}, \eqref{espandoSimboM22})
%\begin{equation}
%\Upsilon \mathcal{L}_6 \Upsilon^{-1}=\omega\cdot \partial_{\varphi}+\mathtt{D}_0+\varepsilon^2 \widehat{\mathtt{D}}_2+\mathtt{D}_{\geq 3}+\sum_{i=3}^6\varepsilon^i \tilde{\mathtt{Q}}_i +\varepsilon^4 \mathtt{Z}_4+\varepsilon^6 \mathtt{Z}_6+\mathfrak{R}_6
%\end{equation}
%$\tilde{\mathtt{Q}}_i =\pi_\mathtt{C} \tilde{\mathtt{Q}}_i\in\mathfrak{C}_{-1/2}$ 
%are almost-diagonal and independent of 
%$\mathfrak{I}_{\delta}$, 
%$\mathtt{Z}_i=\pi_\mathtt{C}^{\perp} \mathtt{Z}_i\in
%\mathfrak{C}_{-1/2}$ 
%are diagonal and independent of 
%$\mathfrak{I}_{\delta}$ (here $\mathtt{C}=\mathtt{C}(S)$ given by Lemma \ref{smallDivLinearBNF}), 
%$\mathfrak{R}_6\in\mathfrak{C}_{-1/2}$ 
%and the following estimates hold
%\begin{align}
%&\mathbb{B}^{\g}_{\mathfrak{R}_6}(s)\lesssim_{s}\gamma^{-3}(\e^{13}
%+\e\|\mathfrak{I}_{\delta}\|^{\gamma,\calO}_{s+\s})\,,\label{scossa1TRIS}\\
%&\mathbb{B}_{\Delta_{12} \mathfrak{R}_6}(p)\leq \e\gamma^{-3}
%(1+\|\mathfrak{I}_{\delta}\|_{p+\s}^{\gamma,\calO})\|i_{1}-i_{2}\|_{p+\s}\,,\label{scossa2TRIS}
%\end{align}
%for some $\s=\s(\tau,\nu)>0$.
%\end{proposition}

We look for symplectic linear changes of variable 
$\Upsilon_i(\varphi)\colon H^s_{S^{\perp}}(\T)\times H^s_{S^{\perp}}(\T) \to H^s_{S^{\perp}}(\T)\times H^s_{S^{\perp}}(\T)$, 
$i=1, \dots, 6$ of the form
%\begin{equation}\label{HAMa1}
%H_{\mathtt{A}_i}(u):=\varepsilon^i\sum_{j, j'\in S^c} 
%\begin{pmatrix}
%(\mathtt{A}_i)_j^{j'}(\varphi)\,  & 0\\
%0 & \overline{(\mathtt{A}_i)_j^{j'}}(\varphi)
%\end{pmatrix}
%\begin{pmatrix} z_{j'}\\
%\bar{z}_{j'}
%\end{pmatrix}
%\cdot
%\begin{pmatrix} \bar{z}_{j}\\
%z_{j}
%\end{pmatrix}
%\end{equation}
%where $\mathtt{A}_i(\varphi)$ is a self-adjoint operator 
%$\forall \varphi\in\T^{\nu}$ and thus 
\begin{equation}\label{Fi1dp}
\Upsilon_i:=
\exp(\varepsilon^i\,  \mathbf{A}_i)
\end{equation}
where $\mathbf{A}_i(\varphi)$
is a Hamiltonian operator (see Definition \ref{admiOpComp}). 
%
%=\mathrm{i} \,E \,\mathtt{A}_i(\varphi)$ with $\mathtt{A}_i=\mathtt{A}_i(\varphi)$ of the form \eqref{forma-complessa} and satisfying \eqref{self-adjT}. We note that $\mathbf{A}_i$ is then a Hamiltonian operator (see Definition \ref{admiOpComp}). 
\begin{remark}
Notice that if the operator ${\bf A}_{i}\in \mathfrak{S}_0$ then $\Upsilon_i$
belongs to $\mathfrak{T}_1$  (see Def. \ref{goodMulti}).
\end{remark}
Recalling \eqref{ADAD}, 
we have that the conjugate of $M$ under the map $\Upsilon_i$ is given (formally) by
the Lie series
\begin{equation}\label{LieAD}
\Upsilon_i M\Upsilon_{i}^{-1}=\sum_{k\geq0}\frac{1}{k!}{\rm ad}_{{\bf A}_i}^{k}M\,,\qquad 
{\rm ad}_{{\bf A}_i}^{k}[M]={\rm ad}_{{\bf A}_i}[{\rm ad}_{{\bf A}_i}^{k-1}{M}]\,,\;
k\geq 1\,,
\quad 
{\rm ad}_{{\bf A}_i}^{0}M=M\,.
\end{equation}
We define $\Pi_{Ker({\bf A})}$ and $\Pi_{Rg({\bf A})}$ the projection on the kernel and the Range of $\rm{ad}_{\bf A}$ respectively.

\medskip

\noindent
{\bf Step one (order $\varepsilon$).}  
Let us define the matrix ${\bf A}_1$ as
\begin{equation}\label{coeffA1}
(\mathbf{A}_1)_{\sigma, j}^{\sigma', k}(\ell)=\dfrac{(\mathtt{Q}_1)_{\sigma, j}^{\sigma', k}(\ell)}{\mathrm{i} \delta^{(1)}_{\sigma, \sigma', j, k}(\ell) } \qquad  j, k\in S^c,\,\,\ell\in\mathbb{Z}^{\nu}
\end{equation}
where $\mathtt{Q}_1$ is in \eqref{formaResto} and $\delta^{(1)}$ is defined in \eqref{deltaUNO}. Note that by Lemma \ref{lowerboundDel1} the denominator in \eqref{coeffA1} never vanishes.
We have the following result.
%The operator $\mathbf{A}_1$ is almost-diagonal, indeed if $\lvert \ell \rvert>1$ then $\delta^{(1)}_{\sigma, \sigma', j, k, \ell}=0$. Moreover $\mathbf{A}_1$ is $x$-translation invariant because $\mathtt{Q}_1$ is it.
\begin{lemma}\label{LBNFeq1}
The operator $\mathbf{A}_1$ is
in $\mathfrak{S}_0$ and 
$1$-almost-diagonal 
and satisfies
\begin{equation}\label{stimaAA1}
|({\bf A}_1)_{\s,j}^{\s',k}(\ell)|\leq (C_1(S))^{-1}|(\mathtt{Q}_{1})_{\s,j}^{\s',k}(\ell)|\,,
\end{equation}
with $C_1(S)>0$ in \eqref{lowerboundDel1BIS}. % depending also on the tangential site $S$.
Moreover it solves the equation 
\begin{equation}\label{omoeqAA1}
ad_{\overline{\omega}\cdot\partial_{\varphi} +\mathtt{D}_0}[\mathbf{A}_1]=\mathtt{Q}_1.
\end{equation}
\end{lemma}

\begin{proof}
Estimate \eqref{stimaAA1} follows by Lemma \ref{lowerboundDel1}. 
Thanks to the definition of ${\bf A}_1$ in \eqref{coeffA1} one can check that
 equation \eqref{omoeqAA1}  is satisfied.
The operator $\mathbf{A}_1$ is almost-diagonal by definition.
Recall 
that $\mathtt{Q}_1$  is Hamiltonian and $x$-translation invariant. 
By using the expression of 
$\delta^{(1)}_{\sigma, \sigma', j, k}(\ell)$ in \eqref{deltaUNO} we have
\begin{equation}\label{deltaalgebra}
\overline{\delta^{(1)}_{\sigma, \sigma', j, k}(\ell)}=\delta^{(1)}_{\sigma, \sigma', k, j}(-\ell), \qquad \overline{\delta^{(1)}_{\sigma, \sigma',  -j, -k}(-\ell)}=\delta^{(1)}_{-\sigma, -\sigma', k, j}(\ell).
\end{equation}
Thus by Lemmata
 \ref{coeffFouHamilto}, 
\ref{lem:momentocomp} ${\bf A}_1$ is Hamiltonian and $x$-translation invariant.
\end{proof}

\begin{lemma}\label{stepLBNF1}
 The conjugate of the operator $\mathcal{L}_{6}$ in \eqref{elle6}
 has the form
\begin{equation}\label{elle7}
\mathcal{L}_7:=\Upsilon_1 \mathcal{L}_6 \Upsilon_1^{-1} 
=\omega\cdot \partial_{\varphi}+\mathtt{D}
+\sum_{i=2}^6 \varepsilon^i\mathtt{Q}_i^{(1)}+\Pi_S^{\perp} \mathfrak{R}_1
\end{equation}
where $\mathtt{Q}_{i}^{(1)}$, $i=2,\ldots,6$, 
are 
in $\mathfrak{S}_0$, are  $i$-almost-diagonal and 
satisfy 
\begin{equation}\label{formaResto2QQ1}
|(\mathtt{Q}_i^{(1)})_{\s,j}^{\s,k}(\ell)|\leq \frac{C}{\langle j,k\rangle^{1/2} }\,,
\qquad |(\mathtt{Q}^{(1)}_i)_{\s,j}^{-\s,k}(\ell)|\leq \frac{C}{\langle j,k\rangle^{\rho-1}}\,,
\qquad \mathtt{Q}_2^{(1)}=\mathtt{Q}_2+\frac{1}{2}ad_{\mathbf{A}_1}[\mathtt{Q}_1]\,
\end{equation}
and $\mathfrak{R}_1$ is in $\mathfrak{S}_0$
and 
is $Lip$-$(-1/2)$-modulo tame  with the following estimates
\begin{equation}\label{scossa1RR} 
\begin{aligned}
 {\mathfrak M}_{\mathfrak{R}_1}^{{\sharp, \g}} (-1/2, s,\mathtt{b}_0) 
 &\lesssim_{s}\gamma^{-3}(\e^{13}
+\e\|\mathfrak{I}_{\delta}\|^{\gamma,\calO_0}_{s+\mu})\,,\\
  {\mathfrak M}_{\Delta_{12}\mathfrak{R}_{1}}^{{\sharp}} (-1/2, p,\mathtt{b}_0) 
&\lesssim_{p} \e\gamma^{-3}
(1+\|\mathfrak{I}_{\delta}\|_{p+\mu})\|i_{1}-i_{2}\|_{p+\mu}\,,
\end{aligned}
\end{equation}
for some $\mu=\mu(\nu)>0$.
\end{lemma}
\begin{proof}
We Taylor expand in $\e$ the conjugate $\mathcal{L}_{7}$
by applying formula \eqref{LieAD}
(recall %the expansion in $\e$ in 
\eqref{formaResto},\eqref{espandoSimboM}).
Thus formula
\eqref{elle7}, with $\mathtt{Q}_2^{(1)}$ as in \eqref{formaResto2QQ1}
%\eqref{elle7Q2}  
follows by Lemma \ref{LBNFeq1}.
Estimates \eqref{scossa1RR}, %\eqref{scossa2RR}, 
\eqref{formaResto2QQ1} follow 
by using Lemmata \ref{LBNFeq1}, 
\ref{compoC12}, \ref{mercato}.
\end{proof}

\noindent
{\bf Step two (order $\varepsilon^2$).}  
Let us define the matrix ${\bf A}_2$ as
\begin{equation}\label{coeffA2}
(\mathbf{A}_2)_{\sigma, j}^{\sigma', k}(\ell)=\begin{cases}
\dfrac{(\mathtt{Q}_2^{(1)})_{\sigma, j}^{\sigma', k}(\ell)}{\mathrm{i} \delta^{(2)}_{\sigma, \sigma', j, k}(\ell) }\qquad \mbox{if}\,\,\, \delta^{(2)}_{\sigma, \sigma', j, k} (\ell) \neq 0\\
0 \qquad \qquad  \qquad\,\,\,\, \mbox{if}\,\,\, \delta^{(2)}_{\sigma, \sigma', j, k}(\ell) = 0
\end{cases}
\end{equation}
where $\mathtt{Q}_2^{(1)}$ is in \eqref{formaResto2QQ1} and $\delta^{(2)}$ is defined in \eqref{deltaUNO}.
We have the following result.
\begin{lemma}\label{LBNFeq2}
The operator $\mathbf{A}_2$ is
in $\mathfrak{S}_0$ and 
$2$-almost-diagonal 
and satisfies
\begin{equation}\label{stimaAA2}
|({{\bf A}_2})_{\s,j}^{\s',k}(\ell)|\leq (C_2(S))^{-1}|(\mathtt{Q}_{2}^{(1)})_{\s,j}^{\s',k}(\ell)|\,,
\end{equation}
with $C_2(S)>0$ in \eqref{lowerboundDel2BIS}. % depending also on the tangential site $S$.
Moreover it solves the equation 
\begin{equation}\label{omoeqAA2}
ad_{\overline{\omega}\cdot\partial_{\varphi} +\mathtt{D}_0}[\mathbf{A}_2]=\Pi_{Rg(\overline{\omega}\cdot\partial_{\varphi}+\mathtt{D_0})}\mathtt{Q}_1^{(2)}.
\end{equation}
\end{lemma}

\begin{proof}
Estimate \eqref{stimaAA2} follows by Lemma \ref{lowerboundDel1}. 
Thanks to the definition of ${\bf A}_2$ in \eqref{coeffA2} one can check that
 equation \eqref{omoeqAA2}  is satisfied.
The operator $\mathbf{A}_2$ is almost-diagonal by definition.
Recall 
that $\mathtt{Q}^{(1)}_2$  is $\mathfrak{S}_0$. 
Thus ${\bf A}_2$ is  in $\mathfrak{S}_0$ because 
the denominator 
$\delta^{(2)}_{\sigma, \sigma', j, k}(\ell)$ satisfies the algebraic relations \eqref{deltaalgebra} (using Lemmata
Lemma \ref{coeffFouHamilto}, 
\ref{lem:momentocomp}).
\end{proof}

\begin{lemma}\label{stepLBNF2}
 The conjugate of the operator $\mathcal{L}_{7}$ in \eqref{elle7}
 has the form
\begin{equation}\label{elle8}
\mathcal{L}_8:=\Upsilon_2 \mathcal{L}_7 \Upsilon_2^{-1} =
\omega\cdot \partial_{\varphi}+\mathtt{D}
+\varepsilon^2\bral \mathtt{Q}_2^{(1)} \brar
+\sum_{i=3}^6 \varepsilon^i\mathtt{Q}_i^{(2)}+\Pi_S^{\perp}\mathfrak{R}_2
\end{equation}
where $\mathtt{Q}_{i}^{(2)}$, $i=3,\ldots,6$, 
are 
in $\mathfrak{S}_0$, are  $i$-almost-diagonal and 
satisfy 
\begin{equation*}%\label{formaResto2QQ2}
|(\mathtt{Q}_i^{(2)})_{\s,j}^{\s,k}(\ell)|\leq \frac{C}{\langle j,k\rangle^{1/2} }\,,
\qquad |(\mathtt{Q}^{(2)}_i)_{\s,j}^{-\s,k}(\ell)|\leq \frac{C}{\langle j,k\rangle^{\rho-2}}\,.
\end{equation*}
Moreover $\mathfrak{R}_2$ is in $\mathfrak{S}_0$
%Hamiltonian, $x-$translation invariant 
and 
is $Lip$-$(-1/2)$-modulo tame  with the following estimates
\begin{equation}\label{scossa1RR2}
\begin{aligned}
{\mathfrak M}_{\mathfrak{R}_2}^{{\sharp, \g}} (-1/2, s,\mathtt{b}_0) 
&\lesssim_{s}\gamma^{-3}(\e^{13}
+\e\|\mathfrak{I}_{\delta}\|^{\gamma,\calO_0}_{s+\mu})\,,\\
 {\mathfrak M}_{\Delta_{12}\mathfrak{R}_{2}}^{{\sharp}} (-1/2, p,\mathtt{b}_0) 
&\lesssim_{p}\e\gamma^{-3}
(1+\|\mathfrak{I}_{\delta}\|_{p+\mu})\|i_{1}-i_{2}\|_{p+\mu}\,,
\end{aligned}
\end{equation}
for some $\mu=\mu(\nu)>0$ possibly larger than the one in Lemma \ref{stepLBNF1}.
\end{lemma}

\begin{proof}
One can  follows word by word the proof of Lemma \ref{stepLBNF1}
using that, by Lemma \ref{lowerboundDel1} (recall \eqref{doppiebra}), 
 \[
 \Pi_{Ker(\overline{\omega}\cdot \partial_{\varphi}
 +\mathtt{D}_0)} \mathtt{Q}_2^{(1)}=\bral \mathtt{Q}_2^{(1)} \brar.
 \] 
\end{proof}

The following proposition allows to compute explicitly the first order corrections of the eigenvalues
\[
\mathtt{D}_2+ \bral \mathtt{Q}^{(1)}_2 \brar.
\] 

\begin{proposition}{\bf (Identification of normal forms).}\label{identificoBNF}
Recall \eqref{espandoSimboM}, \eqref{HamZakDyaCraig}.
Consider the symplectic change of coordinates 
 ${\bf A}_{\varepsilon}$ in  \eqref{AepsilonDP2} and let us denote by  ${\bf A}_1:={\bf A}_{{\varepsilon}_{|_{\varepsilon=1}}}$. 
Then
\begin{equation}\label{mare}
%\varepsilon^2\Big(\frac{c(\omega)}{2} \int_{\T} z^2\,dx
%+\Pi_{\mbox{\rm Ker}(\mathsf{H}_0)}\mathcal{K}^{(1)}_2\Big)
\mathtt{D}_2+ \bral \mathtt{Q}^{(1)}_2 \brar=
X_{\mathtt{H}}\,,
\qquad
\mathtt{H}:={\Big[\Pi_{{\rm triv}}\Pi^{d_z=2} \Big(H^{(4)}_{FB} \Big)\Big]\circ {\bf A}_{{1{|_{ \substack{y=0\\ \theta=\f}}}}}}.
\end{equation}
%where
%\begin{equation}\label{theoBirH}
%\begin{aligned}
%H^{(4)}_{ZD} (z,\bar{z}) &:=
%\frac{1}{4 \pi} \sum_{k \in \Z} |k|^3  \big(  |z_k|^4  - 2 |z_{k}|^2 |z_{-k}|^2   \big)
%\\
%&+  \frac{1}{\pi} \sum_{\substack{k_1, k_2 \in \Z, \, \sign(k_1) = \sign( k_2 )  \\ |k_2| < |k_1|}} |k_1| |k_2|^2  
%  \big( - |z_{-k_1}|^2 |z_{k_2} |^2 + |z_{k_1}|^2  |z_{k_2}|^2 \big).
%\end{aligned}
%\end{equation}
 
\end{proposition}

\begin{proof}
By Proposition \ref{fineIdentif} 
the Hamiltonian $\mathtt{H}$ in \eqref{mare} 
is equal to
$H^{(4)}_{LB} \circ {\bf A}_{1}$ at $y=0$, $\theta=\vphi$, which in turn
%$H^{(4)}_{LB} \circ {\bf A}_{{1{|_{ \substack{y=0\\ \theta=\f}}}}}= 
%\mathtt{H}$. 
 by \eqref{larubia2Linear}, 
 \eqref{equivalenzaZZZ} is equivalent to
 \[
 K:=\Pi_{\mathrm{triv}} \Pi^{d_z=2}  
\mathcal{H}^{(4)}\circ {\bf A}_{{1{|_{ \substack{y=0\\ \theta=\f}}}}}
 \]
 with $\mathcal{H}^{(4)}$ in \eqref{F3F4Linear}.
Hence we have to prove that
\begin{equation}\label{conjecture}
\mathtt{D}_2+ \bral \mathtt{Q}^{(1)}_2 \brar =
\begin{pmatrix}
\pa_{\bar{z}} K\\
-\pa_{z} K\end{pmatrix}.
\end{equation}
 We recall that $\mathcal{H}^{(4)}$ is the $4$-degree Hamiltonian, quadratic in the normal variables, obtained by the \emph{formal} ``Weak'' plus ``Linear'' BNF procedure; on the other hand $\mathtt{D}_2+ \bral \mathtt{Q}^{(1)}_2 \brar$ are the terms of size $\varepsilon^2$ of $\mathcal{L}_8$ in \eqref{elle8}. In turn $\mathcal{L}_8$ is the Hamiltonian operator associated to the Hamiltonian, quadratic in the normal variables, obtained by the \emph{rigorous} procedure: weak BNF plus the regularization procedure of section \ref{ordiniPositivi} plus two steps of linear BNF.
 We now show that the two procedures provide the same output, namely the \eqref{conjecture} holds.\\
In other words we prove that 
the Taylor expansion up to degree 2 in $\e$
 of the operator $\mathcal{L}_8$ in \eqref{elle8} coincides with the 
 vector field of the Hamiltonian $H^{(2)}_{\mathbb{C}}+H^{(4)}_{LB}$
 (see \eqref{Rinco}, \eqref{larubia2Linear})
 written in the action-angle variables defined in \eqref{AepsilonDP2}.

We recall that, by construction,
the operator $\mathcal{L}_8$ in \eqref{elle8}
has the form
\begin{equation}\label{elle8Completo}
\mathcal{L}_{8}=\mathcal{T}\circ \mathcal{L}_{\omega}\circ\mathcal{T}^{-1}\,,
\end{equation}
where $\mathcal{L}_{\omega}$ is in \eqref{BoraMaledetta}
and the map $\mathcal{T}$ is the composition
\begin{equation}\label{mappaCompleta}
\mathcal{T}:=\Upsilon_2 \Upsilon_1 {\bf \Phi} {\bf \Psi} {\bf \Phi}_{\mathbb{M}} \Lambda \Phi_{\mathbb{B}}^{-1},
%{\bf \Psi}^{1}{\bf \Phi}_1{\bf \Phi}_2{\bf \Phi}_3\,,
\end{equation}
where $\Lambda$ is in \eqref{CVWW}
and the other maps above are given in 
Propositions \ref{lemma:7.4}, \ref{lem:compVar}, 
 \ref{lem:mappa2}, \ref{blockTotale}, \ref{riduzSum} and Lemmata \ref{stepLBNF1}, \ref{stepLBNF2}
%Proposition \ref{lem:mappa2}, 
%\ref{water1} (see also \eqref{mappaRho}),  Proposition \ref{blockTotale}
%and \ref{riduzSum}
.
Now we compute the Taylor expansion of $\mathcal{L}_8$ in terms of the generators of the maps composing $\mathcal{T}$ and in terms of the map $\Phi_B$ of the weak BNF.\\
First of all, recalling the arguments in 
section \ref{sec:Linopnormal}, we note that 
 $\mathcal{L}_{\omega}$
is the linearized operator in the normal directions of the 
Hamiltonian $\mathcal{H}_{\mathbb{C}}=H_{\mathbb{C}}\circ\Phi_{B}$ 
(see Proposition \ref{WBNFWW})
written in the \emph{real} variables.
In particular, recalling subsection \ref{struttHAMlomega} (see \eqref{hamiltonianalinearizzata}), 
we have in the complex variables
\begin{equation}\label{corazon2}
\begin{aligned}
\Lambda\mathcal{L}_{\omega}\Lambda^{-1}&=
\bar{\omega}\cdot\pa_{\vphi}\vect{z}{\bar{z}}
+\ii E|D|^{\frac{1}{2}}\vect{z}{\bar{z}}
+\e \mathtt{S}_1\vect{z}{\bar{z}}
+\e^{2}\mathtt{S}_2\vect{z}{\bar{z}}+
\e^{2}(\mathbb{A}\zeta)\cdot\pa_{\vphi}\vect{z}{\bar{z}}
+O(\e^{3})\,
\end{aligned}
\end{equation}
%\begin{equation}\label{corazon2}
%\begin{aligned}
%\Lambda\mathcal{L}_{\omega}\Lambda^{-1}&=\bar{\omega}\cdot\pa_{\vphi}\vect{z}{\bar{z}}
%+\ii \sm{|D|^{\frac{1}{2}}}{0}{0}{-|D|^{\frac{1}{2}}}\vect{z}{\bar{z}}+
%\ii \e\left(
%\begin{matrix}
%\pa_{\bar{z}}\mathtt{H}_1\\
%-\pa_{z}\mathtt{H}_1\end{matrix}\right)+O(\e^{2})\\
%&=\bar{\omega}\cdot\pa_{\vphi}\vect{z}{\bar{z}}
%-\ii \sm{|D|^{\frac{1}{2}}}{0}{0}{-|D|^{\frac{1}{2}}}\vect{z}{\bar{z}}+
%\e \mathtt{S}\vect{z}{\bar{z}}+O(\e^{2})
%\end{aligned}
%\end{equation}
where
\[
\e^{2}(\mathbb{A}\zeta)=\pa_{Q}(\mathsf{H}_2+\mathsf{H}_{\mathcal{R}_2})_{|z=0}\,,
\qquad \e\mathtt{S}_1\vect{z}{\bar{z}}=\ii \e\left(
\begin{matrix}
\pa_{\bar{z}}\mathsf{H}_1\\
-\pa_{z}\mathsf{H}_1\end{matrix}\right)\,,
\qquad
\e^{2}\mathtt{S}_2\vect{z}{\bar{z}}=\ii \e^2\left(
\begin{matrix}
\pa_{\bar{z}}(\mathsf{H}_2+\mathsf{H}_{\mathcal{R}_2})\\
-\pa_{z}(\mathsf{H}_2+\mathsf{H}_{\mathcal{R}_2})\end{matrix}\right)
\]
and $\mathtt{S}_i=\mathtt{S}_{i}(v_I)$, $i=1,2$ are 
 some $2\times2$ matrix of  operators of the form
\eqref{simboOMO2DEF2} with, respectively, $i=1$, $i=2$.
%\begin{equation}\label{simboOMO2DEF200}
%\begin{aligned}
%\mathtt{S}&:=((\mathtt{S})_{\s}^{\s'})_{\s,\s'=\pm}
%\quad 
% (\mathtt{S})_{\s}^{\s'}=\ov{ (\mathtt{S}_{1})_{-\s}^{-\s'}} \, , \quad
% (\mathtt{S})_{\s}^{\s'}z^{\s'}:=\frac{1}{\sqrt{2\pi}}
% \sum_{j\in S^c}e^{\ii \s jx}\big( \sum_{k\in S^c} 
% (\mathtt{S})_{\s,j}^{\s',k}z^{\s'}_{k}\big)\,,\\
%(\mathtt{S})_{\s,j}^{\s',k}&
%:=\frac{1}{\sqrt{2\pi}}\sum_{\substack{
%\s_{1}j_{1}=\s j-\s' k}}
% \big((\mathtt{S})_{\s,j}^{\s',k}\big)_{j_1}^{\s_1}(v_{I})^{\s_1}_{j_1}
%\stackrel{\eqref{coeffvI}}{=}
%\frac{1}{\sqrt{2\pi}}\sum_{\substack{
%\s_{1}j_{1}=\s j-\s' k}}
% \big((\mathtt{S})_{\s,j}^{\s',k}\big)_{j_1}^{\s_1}
%\sqrt{\zeta_{j_1}}
%e^{\ii(\s_1\mathtt{l}(j_1))\cdot\f}
%\end{aligned}
%\end{equation}
%The second equality in \eqref{corazon2} follows since 
%the Hamiltonian $\mathtt{H}_1$ in \eqref{hamiltonianalinearizzata} 
%is $1$-homogeneous in the variables
%$v_{I},\ov{v_{I}}$ given in \eqref{vsegnatoDP}. 
In \eqref{corazon2} we denoted by $O(\e^{3})$ all the terms which are at least cubic 
in $v_{I}, \ov{v_{I}}$.

%Notice that
%the Hamiltonian $\mathcal{H}_{\mathbb{C}}$ is nothing but the 
%water waves Hamiltonian 
%$H$  in 
% \eqref{Hamiltonian} composed with the \emph{weak} Birkhoff map $\Phi_{B}$.
Since the map $\Phi_{B}$ coincides, up to  degree $2$,
 with $\Phi_{WB}$, 
 the  Taylor expansion
 (up to degree $2$ in $\varepsilon$) of $\mathcal{L}_{\omega}$
 coincides with the Hamiltonian vector field, in the normal directions, 
 of the Hamiltonian $H_{\mathbb{C}}\circ\Phi_{WB}$. 
 This implies, recalling \eqref{F3F4Weak}, \eqref{larubiaWeak}, that
 \begin{equation}\label{corazon3}
 \e \mathtt{S}_1\vect{z}{\bar{z}}\equiv \ii\e J\nabla_{(z,\bar{z})}
\big(H^{(3,2)}_{\mathbb{C}}\circ{\bf A}_{{1{|_{ \substack{y=0, \theta=\f}}}}}\big)
% \left(
% \begin{matrix}
% \pa_{\bar{z}}H_{\mathbb{C}}^{(3,2)} 
% \\ 
%  -\pa_{{z}}H_{\mathbb{C}}^{(3,2)}
% \end{matrix}
% \right)\circ {\bf A}_{{1{|_{ \substack{y=0\\ \theta=\f}}}}}
 \,,
 \qquad
 \e^{2}\mathtt{S}_2\vect{z}{\bar{z}}
% \ii \e^2\left(
%\begin{matrix}
%\pa_{\bar{z}}(\mathsf{H}_2+\mathsf{H}_{\mathcal{R}_2})\\
%-\pa_{z}(\mathsf{H}_2+\mathsf{H}_{\mathcal{R}_2})\end{matrix}\right)
:=
\ii\e^2 J\nabla_{(z,\bar{z})} \big(\widehat{H}_1^{(4)}\circ
{\bf A}_{{1{|_{ \substack{y=0, \theta=\f}}}}}\big)\,.
%\left(
% \begin{matrix}
% \pa_{\bar{z}}\widehat{H}_{1}^{(4)} \\ 
%  -\pa_{{z}}\widehat{H}_{1}^{(4)}
% \end{matrix}
% \right)\circ {\bf A}_{{1{|_{ \substack{y=0\\ \theta=\f}}}}}\,.
 \end{equation}
Secondly we note that the map $\mathcal{T}$ in \eqref{mappaCompleta}
admits a Taylor expansion in $\e$.
More precisely we claim that
\begin{equation}\label{claimTTT}
\mathcal{T}={\rm Id}+\e \mathtt{T}_1+\e^{2} \mathtt{T}_2+O(\e^{3})\,,
\end{equation}
where $\mathtt{T}_i=\mathtt{T}_{i}(v_I)$ are $2\times2$ matrix of  operators of the form
\eqref{simboOMO2DEF2} with, respectively, $i=1$ and $i=2$.
Indeed the  maps  
$\Phi_{\mathbb{B}}$, ${\bf \Phi}_{\mathbb{M}}$, ${\bf \Psi}$,
 ${\bf \Phi}$ appearing in \eqref{mappaCompleta}
has been constructed as the time one flow map of a Hamiltonian PDE
whose generator is a pseudo differential operator with symbol
in the class $S_{k}^{m}$ for some $m\in \mathbb{R}$
and $1\leq k\leq 6$. The maps $\Upsilon_i$, $i=1,2,$ are given by \eqref{Fi1dp}.
Hence the \eqref{claimTTT} follows.

By \eqref{corazon2}
and \eqref{claimTTT} (see also \eqref{espandoSimboM}) 
we get
\begin{equation}\label{corazon8}
\begin{aligned}
\mathcal{L}_{8}&=\bar{\omega}\cdot\pa_{\vphi}\vect{z}{\bar{z}}
+\mathtt{D}_0\vect{z}{\bar{z}}+\e\mathtt{S}_1\vect{z}{\bar{z}}
+\e \Big[\mathtt{T}_1, \bar{\omega}\cdot\pa_{\vphi}+\mathtt{D}_0\Big]\vect{z}{\bar{z}}
+\e^{2}[\mathtt{T}_1,\mathtt{S}_1 ]\vect{z}{\bar{z}}\\&
+\frac{\e^{2}}{2}\Big[\mathtt{T}_1, [\mathtt{T}_1, 
\bar{\omega}\cdot\pa_{\vphi}+\mathtt{D}_0
]\Big]\vect{z}{\bar{z}}
+\e^{2}\mathtt{S}_2\vect{z}{\bar{z}}+
\e^{2}(\mathbb{A}\zeta)\cdot\pa_{\vphi}\vect{z}{\bar{z}}+\e^{2}\Big[\mathtt{T}_2, \bar{\omega}\cdot\pa_{\vphi}+\mathtt{D}_0\Big]\vect{z}{\bar{z}}+O(\e^{3})\,.
\end{aligned}
\end{equation}
In section \ref{ordiniPositivi} and Lemma \ref{stepLBNF1}
we proved that 
%By construction (see \eqref{elle8}) we have
\begin{equation}\label{corazon4}
\mathtt{S}_1\vect{z}{\bar{z}}
+ \Big[\mathtt{T}_1, \bar{\omega}\cdot\pa_{\vphi}+
\mathtt{D}_0\Big]\vect{z}{\bar{z}}
%=
%\e\mathtt{S}\vect{z}{\bar{z}}
%+\e \Big[\mathtt{T}\vect{z}{\bar{z}}, X_{\mathtt{H}_0}\Big]
\equiv0\,.
\end{equation}
Our aim is to express the operator $\mathtt{T}$ in terms of the vector field of the Hamiltonian 
$F_{3}^{(3,2)}$ appearing in \eqref{F3F4Linear}.
First of all notice that $F_{3}^{(3,2)}$ solves (recall  \eqref{compvecWW})
%, \eqref{HSuk})
 \begin{equation}\label{corazon10}
H_{{\mathbb{C}}}^{(3,2)}+\{F_{3}^{(3,2)},H_{\mathbb{C}}^{(2)}\}=0\qquad \Leftrightarrow
\qquad
X_{H_{{\mathbb{C}}}^{(3,2)}}+
X_{\{F_{3}^{(3,2)}, H_{\mathbb{C}}^{(2)} \}}
=X_{H_{{\mathbb{C}}}^{(3,2)}}+[X_{{F_{3}^{(3,2)}}}, X_{H_{\mathbb{C}}^{(2)}}]=0\,,
\end{equation}
 where $[\cdot,\cdot]$ denotes the \emph{nonlinear} commutator  between two vector fields defined 
 as
\[
\begin{aligned}
&X=\sum_{\s=\pm}X^{(u^{\s})}\pa_{u^{\s}}\,, \quad Y=\sum_{\s=\pm}Y^{(u^{\s})}\pa_{u^{\s}}\,,\\
[X, Y]&=
\sum_{\s=\pm}\Big(\sum_{\s'}d_{u^{\s'}}Y^{(u^{\s})}[X^{(u^{\s'})}]-
d_{u^{\s'}}X^{(u^{\s})}[Y^{(u^{\s'})}]
\Big)\pa_{u^{\s}}\,.
%\sum_{\s=\pm}d_{u^{\s}}Y[X^{\s}]-d_{g}X[Y]\,.
\end{aligned}
\]
Define the vector field (recall that $\mathtt{T}$ has the form \eqref{simboOMO2DEF2})
\[
R:=\sum_{\s=\pm}R^{\s}\pa_{z^{\s}}\,, \qquad R^{\s}:=\sum_{\s'}\mathtt{T}_{\s}^{\s'}z^{\s'}\,.
\]
We define
\begin{equation}\label{corazon901}
G^{(3, 2)}:=H^{(3,2)}_\mathbb{C}  \circ{\bf A}_{1{|_{ \substack{y=0\\ \theta=\f}}} } , \qquad \mathcal{F}^{(3, 2)}:=F_3^{(3,2)}  \circ{\bf A}_{1{|_{ \substack{y=0\\ \theta=\f}}} } .
\end{equation}
Then, using that
$\mathsf{H}_0=H^{(2)}_{\mathbb{C}}\circ {\bf A}_1$ at $y=0$, $\theta=\vphi$
%$X_{\mathtt{H}_0}=X_{H^{(2)}_{\mathbb{C}}}\circ{\bf A}_1$ 
(recall \eqref{leHamiltonianeLin})
and the \eqref{corazon3},
we have that equations \eqref{corazon4}, \eqref{corazon10} read
\begin{equation}\label{corazon70}
\begin{aligned}
&  X_{G^{(3,2)}  }+ [R,X_{\mathsf{H}_0}]=0\,,\\
&X_{G^{(3,2)}}
+[X_{\mathcal{F}^{(3,2)}}, X_{\mathsf{H}_0}]=0,
\end{aligned}
\end{equation}
where
the vector fields are computed
with respect to the extended symplectic form 
\eqref{estesaPois}.

Due to the absence of $3$-waves interactions we have that the operator
$[\cdot,  X_{\mathsf{H}_0}]$
is invertible. Therefore
\begin{equation}\label{corazon9}
R\equiv X_{\mathcal{F}^{(3,2)}}\,.
\end{equation}
The terms $O(\e^{2})$ of $\mathcal{L}_{8}$ in 
\eqref{corazon8} are given by the $O(\e^{2})$-terms of $\mathcal{L}_{\omega}$,
given by \eqref{corazon3}, plus 
\begin{equation}\label{corazon900}
\frac{\e^{2}}{2}[\mathtt{T}_1,\mathtt{S}_1]\vect{z}{\bar{z}}+
\e^{2}\Big[\mathtt{T}_2, \bar{\omega}\cdot\pa_{\vphi}+\mathtt{D}_0\Big]\vect{z}{\bar{z}}\,,
\end{equation}
where we used \eqref{corazon4}. Notice that the second summand
belongs to the range of 
the adjoint action of $X_{\mathsf{H}_0}$.
Using \eqref{corazon3}, \eqref{corazon901} and \eqref{corazon9} we note also that
the first summand in \eqref{corazon900} is equal to 
\begin{equation}\label{corazon902}
\frac{\e^{2}}{2}[X_{\mathcal{F}^{(3,2)}}, X_{G^{(3,2)}}]\,.
\end{equation}
By \eqref{F3F4Weak}, \eqref{F3F4Linear} we have
\[
\Pi^{d_{z}=2}\mathcal{H}^{(4)}=\Pi^{d_z=2} \widehat{H}_1^{(4)}+\frac{1}{2}\{
F_3^{(3,2)},H^{(3,2)}_{\mathbb{C}}\}\,.
\]
%some corrections coming from the 
%adjoint action of $\mathtt{T}$.
By \eqref{corazon901} we have that the vector field 
of the Hamiltonian 
$\frac{1}{2}\{
F_3^{(3,2)},H^{(3,2)}_{\mathbb{C}}\}$ 
is equal to the one in \eqref{corazon902}.
Hence, 
by Remark \ref{rmk:corazon}, we get the \eqref{conjecture}.
% depend only on the field $\mathtt{T}$ which is
%the field $\widetilde{T}$ written in the action angle variables. 
%Then by \eqref{corazon9} we have that they coincide with the vector field 
%of the fourth order Hamiltonian in \eqref{F3F4Linear}. 
%Then \eqref{conjecture}
%follows.
\end{proof}

By Lemma $7.1$ in \cite{BFP} (see also fomul\ae\, $(7.21)$, $(7.22)$ therein)
we have that the Hamiltonian vector field corresponding to the Hamiltonian function
$H_{FB}:= H^{(2)}_{\mathbb{C}}+H_{FB}^{(4)}$ (see \eqref{Rinco} 
and \eqref{theoBirH}) is given by  (see \eqref{compvecWW})
\begin{equation}\label{CSFIN}
\dot z_n = - \ii \omega_n z_n 
- \frac{\ii}{\pi} \Big( \sum_{|j| <  |n|}  j |j|  |z_{j}|^2 \Big)   n z_n 
+ [ R (z)]_n\,,
\end{equation}
\begin{align}\label{primRn}
[R(z)]_n := -  \frac{\ii}{ 2 \pi}  |n|^3 \big( | z_n |^2 - 2  | z_{-n} |^2 \big) z_n 
 & +  \frac{ \ii }{\pi}  \!\!\!  \sum_{\substack{|n| < |k_1|, \\ \sign(k_1) = \sign( n )}} \!\!\! |k_1| |n|^2 z_{n} \big(|z_{-k_1}|^2- |z_{k_1}|^2\big) \, .
   \end{align}

Then it is easy to note that Proposition \ref{identificoBNF} implies
\begin{equation}\label{formaNormeps2}
\mathtt{D}_2+ \bral \mathtt{Q}^{(1)}_2 \brar=\sm{\mathtt{Z}}{0}{0}{\ov{ \mathtt{Z}}}
%\left(\begin{matrix} \mathtt{Z} & 0 \\
%0 & \overline{\mathtt{Z}} \end{matrix}\right)\,,  
\qquad 
\mathtt{Z}:=\mbox{diag}_{j\in S^c}(\ii \kappa_j)\,,
\qquad
\kappa_j=\big(m_1(\omega)+c_j(\omega)\big) j,
\end{equation}
where (recall \eqref{maxS})
\begin{equation}\label{cj}
\begin{aligned}
%\kappa_j&=\big(m_1(\omega)+c_j(\omega)\big) j\\
c_j&:=\left\{
\begin{aligned}
&\frac{1}{\pi}\sum_{\substack{ k\in S \\ |k|>|j|}}k(|j|-|k|)\zeta_{k}\,,\qquad 
|j|<\max(S), \\
&0\,,\qquad \qquad \qquad\,\, \qquad\qquad |j|\geq\max(S)\,.
\end{aligned}
\right.
\end{aligned}
\end{equation}
Note that $\mathtt{Z}=\ov{\mathtt{Z}}$ since $c_{-j}=c_j$ for all $j\in S^c$.
The presence of the finitely many corrections $c_{j}\,j$
makes hard to impose the second Melnikov conditions.\\
{Indeed instead of considering $\omega\in \mathcal{G}^{(0)}_1$ in \eqref{ZEROMEL} we should consider more complicated sets like
\[
\widetilde{\mathcal{G}}_{0, j, k}^{(1)}:=
\Big\{ \omega \in \Omega_{\varepsilon} : 
\lvert (\overline{\omega}-\varepsilon^2\,m_1(\omega)\,\mathtt{v})\cdot \ell +\varepsilon^2 \mathbb{A} \zeta(\omega) \cdot \ell+c_j(\omega)\,j-c_k(\omega)\,k\rvert\geq \gamma\,\langle \ell \rangle^{-\tau}, 
\,\, \forall \ell \in\mathbb{Z}^{\nu}\setminus\{0\}\Big\}\,
\] 
for $j, k$ such that $|j|, |k|<\max(S)$.
We remark that in contrast to the correction $m_1\,(j-k)$, appearing in the second order Melnikov conditions, we cannot write $c_j\,j-c_k\,k$ as a function of the index $\ell$ by using the momentum law $\mathtt{v}\cdot \ell+j-k=0$. Then if one wants to prove a twist condition like \eqref{twistcondition2} in Lemma \ref{measG0} to prove that the sets $\widetilde{\mathcal{G}}_{0, j, k}^{(1)}$ have large measure, one should have precise information on the $c_j$'s. The definition of the $c_j$'s depends strongly on the form of the set $S$. Since we want to provide the existence of quasi-periodic solutions with an arbitrary number of frequencies (the cardinality of $S$ is arbitrary) and we want the set $S$ to have the most general form we prefer to approach this problem in a different way.\\}
 To overcome this problem
we introduce \emph{rotating coordinates} for a finite number of modes.
We consider the map 
$\Upsilon_*(\varphi)\colon H^s_{S^{\perp}}(\T)\times H^s_{S^{\perp}}(\T) 
\to H^s_{S^{\perp}}(\T)\times H^s_{S^{\perp}}(\T)$
defined as $\Upsilon_*=\exp(\varepsilon^2\mathbf{A}_*)$ where
\begin{equation}\label{alphaj}
\mathbf{A}_*(\varphi)=\sm{\mathtt{A}_*}{0}{0}{\ov{ \mathtt{A}_*}}, \quad \mathtt{A}_*(\varphi):= \mbox{diag}_{j\in S^c} (\mathrm{i}\alpha_j\cdot \varphi), \quad \alpha_j:=
\frac{\mathtt{v}^{\perp}}{\omega\cdot \mathtt{v}^{\perp}}\,c_j\,j\,,
\end{equation}
where $\mathtt{v}^{\perp}\in\mathbb{Z}^{\nu}$ 
is orthogonal to $\mathtt{v}$ and note that $\mathtt{A}_*=\ov{\mathtt{A}_*}$.
For concreteness we choose (recall \eqref{velocityvec})
\[
\mathtt{v}^{\perp}=(\ov{\jmath}_2, -\ov{\jmath}_1,0,\ldots,0)\,.
\]
In this way it is easy to note that $|\omega\cdot \mathtt{v}^{\perp}|\geq C(S) $
form some $C(S)>0$ depending only on the set $S$.
We observe that such change of coordinates differs from the identity for a finite rank operator. Moreover $\Upsilon_*$ is symplectic 
and $x$-translation invariant thanks to 
Lemma \ref{lem:momentocomp} and the fact 
that $\mathtt{v}\cdot \alpha_j=0$ for all $j\in S^c$. 
Since 
\[
\exp(\varepsilon^2\mathbf{A}_*)=  \sm{\exp(\varepsilon^2\mathtt{A}_*)}{0}{0}{\ov{\exp(\varepsilon^2\mathtt{A}_*)}}, \quad \exp(\varepsilon^2\mathtt{A}_*) =\mbox{diag}_{j\in S^c } (e^{\mathrm{i}  \varepsilon^2\alpha_j\cdot \varphi})
\]
we have that
\begin{equation}\label{boundUpsilon}
\lVert \Upsilon_*^{\pm 1}u-u \rVert_s\lesssim_s \varepsilon^2  \lVert u \rVert_{s}.
\end{equation}

Recalling the definition of $\mathtt{D}$ in \eqref{simboM} and \eqref{espandoSimboM}
we define
\begin{equation}\label{espandoSimboM22}
\widehat{\mathtt{D}}_2:=\opw\left(\begin{matrix}
\ii m_1\x & 0\\ 0 & \ii m_1\,\x
\end{matrix}\right)\,,\qquad
\mathcal{D}_{\geq4}:=\e^{4}\mathtt{D}_{4}+\e^{6}\mathtt{D}_{6}+
\mathtt{D}_{\geq8}
\end{equation}
We have the following.

\begin{lemma}{\bf (Rotating coordinates).}\label{phaseshift}
Consider the operator $\mathcal{L}_8$ in \eqref{elle8}. We have that
\begin{equation}\label{elle8tilde}
\tilde{\mathcal{L}}_8:=\Upsilon_* \mathcal{L}_8 \Upsilon^{-1}_*=
\omega\cdot \partial_{\varphi}
+\mathtt{D}_0+\varepsilon^2 \widehat{\mathtt{D}}_2
+\mathcal{D}_{\geq 4}
+\sum_{i=3}^6 \varepsilon^i (\pi_{\mathtt{C}} \tilde{\mathtt{Q}}_{i}^{(2)}
+\pi^{\perp}_{\mathtt{C}} \mathtt{Q}_{i}^{(2)})+\Pi_S^{\perp}\tilde{\mathfrak{R}}_2
\end{equation}
where $\mathtt{C}$ is given in \eqref{costantefond}, $\pi^{\perp}_{\mathtt{C}}$ is defined in Lemma \ref{lemmaeqomologica}, $\tilde{\mathtt{Q}}_i^{(2)}$, 
$i=3, \dots, 6$ are in $\mathfrak{S}_0$ (see Def. \ref{compaMulti}), almost-diagonal 
(see Def. \ref{def:AlmostDiag}) and
satisfy
\begin{equation}\label{formaResto2QQ2tilde}
\tilde{\mathtt{Q}}_i^{(2)}=\pi_{\mathtt{C}} \tilde{\mathtt{Q}}_i^{(2)}, \quad 
|(\tilde{\mathtt{Q}}_i^{(2)})_{\s,j}^{\s,k}(\ell)|\leq \frac{C_i}{\langle j,k\rangle^{1/2} }\,,
\qquad |(\tilde{\mathtt{Q}}^{(2)}_i)_{\s,j}^{-\s,k}(\ell)|\leq \frac{C_i}{\langle j,k\rangle^{\rho-2}}\,,
\quad (\tilde{\mathtt{Q}}_i^{(2)})_{\s,j}^{\s,k}(\ell)\neq0\;\; \Rightarrow \;\;|\ell|\le C\,,
\end{equation} 
for some constants $C_{i},C$, $i=3,\ldots,6$, depending on $S$.
The operator 
$\tilde{\mathfrak{R}}_2$ is in $\mathfrak{S}_{0}$ and 
is Lip-$(-1/2)$-modulo tame 
with  estimates
\begin{equation}\label{scossa1RR2tilde}
\mathfrak{M}_{\tilde{\mathfrak{R}}_2}^{\sharp,\gamma}(s,\mathtt{b}_0)
\lesssim_{s} \gamma^{-3}(\varepsilon^{13}+
\varepsilon \lVert \mathfrak{I}_{\delta}\rVert_{s+\mu}^{\gamma, \calO_0} )\,, 
\quad \mathfrak{M}_{\Delta_{12} \tilde{\mathfrak{R}}_2}^{\sharp}(p) 
\lesssim_{p}  \varepsilon \gamma^{-3} 
(1+\lVert \mathfrak{I}_{\delta}\rVert_{p+\mu}) \lVert i_1-i_2 \rVert_{p+\mu}
\end{equation}
for some $\mu=\mu(\nu)>0$ possibly larger than the one in Lemma \ref{stepLBNF2}.
\end{lemma}
\begin{proof}
We have that (recall \eqref{formaNormeps2}, \eqref{espandoSimboM22})
\begin{equation}
\Upsilon_* \omega\cdot \partial_{\varphi} \Upsilon_*^{-1}
=\omega\cdot \partial_{\varphi}-\varepsilon^2\mathbf{A}_*(\omega)\stackrel{\eqref{alphaj}}{=}\omega\cdot \partial_{\varphi}-\varepsilon^2\mathtt{D}_2-\varepsilon^2 \bral \mathtt{Q}^{(1)}_2 \brar+\varepsilon^2\widehat{\mathtt{D}}_2,
\end{equation}
and, for any  $i=2,\ldots,6$,
%\begin{equation}\label{coeffConjQi}
%(\Upsilon_* {\bf A}\Upsilon_*^{-1})_{\s, j}^{\s', k}(\ell)
%=\sum_{\ell_1,\ell_2} (\Upsilon_*)_{\s, j}^{\s, j}(\ell_1)\,
%(\Upsilon^{-1}_*)_{\s', k}^{\s', k}(\ell_2)\, ({\bf A})_{\s, j}^{\s', k}(\ell-\ell_1-\ell_2)
%\end{equation}
%\begin{equation}
\begin{equation}\label{coeffConjQi}
(\Upsilon_* \mathtt{Q}_{i}^{(2)}\Upsilon_*^{-1})_{\s, j}^{\s', k}(\ell)
=\sum_{\ell_1,\ell_2} (\Upsilon_*)_{\s, j}^{\s, j}(\ell_1)\,
(\Upsilon^{-1}_*)_{\s', k}^{\s', k}(\ell_2)\, (\mathtt{Q}_{i}^{(2)})_{\s, j}^{\s', k}
(\ell-\ell_1-\ell_2)\,.
\end{equation}
Using  \eqref{coeffConjQi}, 
 \eqref{alphaj}, \eqref{cj}, $\mathtt{C}\gg \max(S)$ and the fact that $\Upsilon_*, \mathtt{Q}_{i}^{(2)}$ are $x$-translation invariant, 
 we deduce 
\begin{equation}\label{proiettIDE}
\pi_{\mathtt{C}}^{\perp} (\Upsilon_* \mathtt{Q}_{i}^{(2)} \Upsilon_*^{-1})
= \pi_{\mathtt{C}}^{\perp}  \mathtt{Q}_{i}^{(2)}\,.
\end{equation}
In order to prove that 
$\Upsilon_{*}\mathfrak{R}_{2}\Upsilon_{*}^{-1}$ is $(-1/2)$-modulo tame
we shall prove that (recall Definition \ref{menounomodulotame})
the operator $R:=
\langle\pa_{\vphi}\rangle^{\mathtt{b}_0}\langle D_{x}\rangle^{\frac{1}{4}} 
\Upsilon_{*}\mathfrak{R}_{2}\Upsilon_{*}^{-1}\langle D_{x}\rangle^{\frac{1}{4}}$
is modulo tame (see Def. \ref{def:op-tame}).
Recalling \eqref{funtore} and reasoning as in \eqref{coeffConjQi}
we get
\begin{align}
|(R)_{\s,j}^{\s',k}(\ell)|
%&\leq
%\sum_{\ell_1, \ell_{2}}
%(\Upsilon_*)_{\s, j}^{\s, j}(\ell_1)\,
%(\Upsilon^{-1}_*)_{\s', k}^{\s', k}(\ell_2)\, (\mathfrak{R}_{2})_{\s, j}^{\s', k}
%(\ell-\ell_1-\ell_2)
%\langle j\rangle^{\frac{1}{4}}
%\langle k\rangle^{\frac{1}{4}}\langle \ell\rangle^{\mathtt{b}_0}\nonumber\\
&\lesssim_{S}
\sum_{\ell_1, \ell_{2}}|
(\Upsilon_*)_{\s, j}^{\s, j}(\ell_1)\,
(\Upsilon^{-1}_*)_{\s', k}^{\s', k}(\ell_2)\, (\mathfrak{R}_{2})_{\s, j}^{\s', k}
(\ell-\ell_1-\ell_2)|
\langle j\rangle^{\frac{1}{4}}
\langle k \rangle^{\frac{1}{4}}\langle \ell-\ell_1-\ell_2\rangle^{\mathtt{b}_0}\,,
\label{oralegale}
\end{align}
where we used
that, since $\Upsilon_{*}$ is almost diagonal, $|\ell|\lesssim_{S}|\ell-\ell_1-\ell_2|$\,.
The coefficients in \eqref{oralegale}
are the ones of the operator $\underline{\widetilde{R}}$ (see Definition \ref{opeMagg})
where
\[
\widetilde{R}:=
\Upsilon_{*}\langle\pa_{\vphi} \rangle^{\mathtt{b}_0} \langle D_{x}\rangle^{\frac{1}{4}} 
\mathfrak{R}_{2}
\langle D_{x}\rangle^{\frac{1}{4}} 
\Upsilon_{*}^{-1}\,.
\]
Therefore, using \eqref{scossa1RR2} %, \eqref{scossa2RR2} 
and \eqref{boundUpsilon}
we deduce the \eqref{scossa1RR2tilde}.
Using formula \eqref{coeffConjQi} it is easy to check
that
\[
\pi_{\mathtt{C}}(\Upsilon_* \mathtt{Q}_{i}^{(2)} \Upsilon_*^{-1})
\]
satisfies 
 the \eqref{formaResto2QQ2tilde}
by using that $\mathtt{Q}^{(2)}_i$ is $i$-almost-diagonal and it satisfies the bounds given in Lemma \ref{stepLBNF2}.
All the operators appearing in \eqref{elle8tilde} are in $\mathfrak{S}_0$
since the map $\Upsilon_{*}$ is symplectic and $x$-translation invariant.
\end{proof}

\subsection{Terms \texorpdfstring{$O(\varepsilon^{\geq 3})$}{O(e3)}: analysis of high modes}\label{stepsHigher}
In this section we shall normalize the terms $\pi_{\mathtt{C}}^{\perp}{\mathtt{Q}}_{i}^{(2)}$
appearing in \eqref{elle8tilde}.

\noindent
{\bf Step three (order $\varepsilon^3$).}  
Let us define the matrix ${\bf A}_3=\pi_{\mathtt{C}}^{\perp}{\bf A}_{3}$ as
\begin{equation}\label{coeffA3}
({\bf A}_{3})_{\s,j}^{\s',k}(\ell):=\left\{
\begin{aligned}
&(\mathtt{Q}_3^{(2)})_{\s,j}^{\s',k}(\ell)/ \ii \delta^{(3)}_{\s,\s',j,k}(\ell)\,,\quad \max\{|j|,|k|\}\geq \mathtt{C}\,, \ell\neq0\,,
|\ell|\leq 3\\
&0 \qquad \qquad {\rm otherwise}\,
\end{aligned}\right.
\end{equation}
where $\mathtt{Q}_3^{(2)}$ is in \eqref{elle8tilde} and $\delta^{(3)}_{\s,\s',j,k}(\ell)$ 
is defined in \eqref{deltaUNO}.
We have the following result.
\begin{lemma}\label{LBNFeq3}
The operator $\mathbf{A}_3$ is
in $\mathfrak{S}_0$ and 
$3$-almost-diagonal 
and satisfies
\begin{equation}\label{stimaAA3}
|({{\bf A}_3})_{\s,j}^{\s',k}(\ell)|\leq (K(S))^{-1}|(\mathtt{Q}_{3}^{(2)})_{\s,j}^{\s',k}(\ell)|\,,
\end{equation}
with $K(S)>0$ in \eqref{smallDivLinearBNF2}. 
Moreover it solves the equation 
\begin{equation}\label{omoeqAA3}
ad_{\overline{\omega}\cdot\partial_{\varphi} +\mathtt{D}_0}[\mathbf{A}_3]=
\pi_{\mathtt{C}}^{\perp}\mathtt{Q}_3^{(2)}\,.
\end{equation}
\end{lemma}

\begin{proof}
Estimate \eqref{stimaAA3} follows by item $(ii)$ in Lemma \ref{lowerboundDel1} 
Thanks to the definition of ${\bf A}_3$ in \eqref{coeffA3} one can check that the
 equation \eqref{omoeqAA3}  is satisfied.
Recall 
that $\mathtt{Q}_2^{(3)}$  is $\mathfrak{S}_0$. 
By using the expression of 
$\delta^{(3)}_{\sigma, \sigma', j, k}(\ell)$ in \eqref{deltaUNO}
and Lemmata
Lemma \ref{coeffFouHamilto}, 
\ref{lem:momentocomp}, one can check 
that ${\bf A}_3$ is $\mathfrak{S}_0$.
\end{proof}

\begin{lemma}\label{stepLBNF3}
%The map $\Upsilon_1$  in \eqref{Fi1dp} with 
% $\mathbf{A}_1$ in \eqref{coeffA1} is in $\mathfrak{T}_1$ (see Def. \ref{goodMulti}). 
% Moreover the 
 The conjugate of the operator $\tilde{\mathcal{L}}_{8}$ in \eqref{elle8tilde}
 has the form
\begin{equation}\label{elle9}
\mathcal{L}_9:=\Upsilon_3 \tilde{\mathcal{L}}_8 \Upsilon_3^{-1} 
=
\omega\cdot \partial_{\varphi}
+\mathtt{D}_0+\varepsilon^2 \widehat{\mathtt{D}}_2
+\mathcal{D}_{\geq 4}
+\sum_{i=3}^6 \varepsilon^i   \pi_{\mathtt{C}} \tilde{\mathtt{Q}}_{i}^{(3)}
+
\sum_{i=4}^{6}\e^{i}\pi^{\perp}_{\mathtt{C}}  \mathtt{Q}_{i}^{(3)}+\Pi_S^{\perp}\mathfrak{R}_{3}
\end{equation}
where $\Upsilon_{3}$ is the map in \eqref{Fi1dp}
with ${\bf A}_{3}$ in \eqref{coeffA3}, $\hat{\mathtt{D}}_{2}$, $\mathcal{D}_{\geq4}$ 
are in \eqref{espandoSimboM22},
 $\tilde{\mathtt{Q}}_i^{(3)}=\pi_{\mathtt{C}} \tilde{\mathtt{Q}}_i^{(3)}$, 
$i=3, \dots, 6$ are in $\mathfrak{S}_0$ (see Def. \ref{compaMulti}), almost-diagonal 
(see Def. \ref{def:AlmostDiag}) and
satisfy
\begin{equation}\label{formaResto2QQ2tilde3}
|(\tilde{\mathtt{Q}}_i^{(3)})_{\s,j}^{\s,k}(\ell)|\leq \frac{C_i}{\langle j,k\rangle^{1/2} }\,,
\qquad |(\tilde{\mathtt{Q}}^{(3)}_i)_{\s,j}^{-\s,k}(\ell)|\leq \frac{C_i}{\langle j,k\rangle^{\rho-3}}\,,
\quad (\tilde{\mathtt{Q}}_i^{(3)})_{\s,j}^{\s,k}(\ell)\neq0\;\; \Rightarrow \;\;|\ell| \le C\,,
\end{equation} 
for some constants $C_{i},C$, $i=3,\ldots,6$, depending on $S$.
Moreover the operators $\mathtt{Q}_{i}^{(3)}$, $i=4,\ldots,6$, 
are 
in $\mathfrak{S}_0$, are  $i$-almost-diagonal and 
satisfy the estimates \eqref{formaResto2QQ2tilde3}.
Finally 
the operator  $\mathfrak{R}_3$ is in $\mathfrak{S}_0$
%Hamiltonian, $x-$translation invariant 
and 
is $Lip$-$(-1/2)$-modulo tame  with the following estimates
\begin{align}
&  {\mathfrak M}_{\mathfrak{R}_3}^{{\sharp, \g}} (-1/2,s,\mathtt{b}_0) 
\lesssim_{s}\gamma^{-3}(\e^{13}
+\e\|\mathfrak{I}_{\delta}\|^{\gamma,\calO_0}_{s+\mu})\,,\label{scossa1RR3}\\
&
  {\mathfrak M}_{\Delta_{12}\mathfrak{R}_{3}}^{{\sharp}} (-1/2,p,\mathtt{b}_0) 
\lesssim_p \e\gamma^{-3}
(1+\|\mathfrak{I}_{\delta}\|_{p+\mu})\|i_{1}-i_{2}\|_{p+\mu}\,,\label{scossa2RR3}
\end{align}
for some $\mu=\mu(\nu)>0$ possibly larger than the one in Lemma \ref{phaseshift}.
\end{lemma}

\begin{proof}
We Taylor expand in $\e$ the conjugate $\mathcal{L}_{9}$
by applying formula \eqref{LieAD}
(recall the expansion in $\e$ in \eqref{elle8tilde}).
Thus formula
\eqref{elle9} follows by Lemma \ref{LBNFeq3}, \eqref{omoeqAA3}.
Estimates \eqref{scossa1RR3}, \eqref{scossa2RR3}, \eqref{formaResto2QQ2tilde3} follow 
by using Lemmata \ref{LBNFeq3}, 
\ref{compoC12}, \ref{mercato} and bounds \eqref{scossa1RR2tilde}.
\end{proof}

In the following lemma we collect the result of conjugating the operator $\mathcal{L}_9$ in \eqref{elle9} by three maps of the same form of $\Upsilon_3$ whose aim is to normalize the terms of order $\varepsilon^i$ with $i=4, 5, 6$ such that $\pi_{\mathtt{C}}^{\perp}$ leaves them invariant.

\begin{lemma}\label{stepLBNF456}
For any $\omega\in \Omega^{2\gamma}_{\infty}$ there exists a map $\Upsilon_{\geq 4}(\omega)\colon H^s_{S^{\perp}}(\T)\times H^s_{S^{\perp}}(\T) 
\to H^s_{S^{\perp}}(\T)\times H^s_{S^{\perp}}(\T)$ such that 
\begin{equation}\label{elle10}
\mathcal{L}_{10}:=\Upsilon_{\geq 4} \mathcal{L}_9 \Upsilon_{\geq 4}^{-1} 
=
\omega\cdot \partial_{\varphi}
+\mathtt{D}_0+\varepsilon^2 \widehat{\mathtt{D}}_2
+\mathcal{D}_{\geq 4}+\pi_{\mathtt{C}}\mathfrak{Q}
+\varepsilon^4\bral\pi^{\perp}_{\mathtt{C}} \mathtt{Q}_{4}^{(4)} \brar+\varepsilon^6\bral\pi^{\perp}_{\mathtt{C}} \mathtt{Q}_{6}^{(4)} \brar
+\Pi_S^{\perp}\mathfrak{R}_{4}
\end{equation}
where $\hat{\mathtt{D}}_{2}$, $\mathcal{D}_{\geq4}$ 
are in \eqref{espandoSimboM22},
 $\mathfrak{Q}=\pi_{\mathtt{C}} \mathfrak{Q}$ is in $\mathfrak{S}_0$ (see Def. \ref{compaMulti}), almost-diagonal 
(see Def. \ref{def:AlmostDiag}) and
satisfies
\begin{equation}\label{formaResto2QQ2tilde4}
|(\mathfrak{Q})_{\s,j}^{\s,k}(\ell)|\leq \frac{C_i\,\varepsilon^3}{\langle j,k\rangle^{1/2} }\,,
\qquad |(\mathfrak{Q})_{\s,j}^{-\s,k}(\ell)|\leq \frac{C_i\,\varepsilon^3}{\langle j,k\rangle^{\rho-6}}\,,
\quad (\mathfrak{Q})_{\s,j}^{\s,k}(\ell)\neq0\;\; \Rightarrow \;\;|\ell|\le C_{\mathfrak{Q}}\,,
\end{equation} 
for some constants $C_{i},C_{\mathfrak{Q}}$, $i=3,\ldots,6$, depending on $S$.
%Moreover the operators $\mathtt{Q}_{i}^{(4)}$, $i=4,\ldots,6$, 
%are 
%in $\mathfrak{S}_0$, are  $i$-almost-diagonal and 
%satisfy the estimates \eqref{formaResto2QQ2tilde3}.
Finally 
the operator  $\mathfrak{R}_4$ is in $\mathfrak{S}_0$
%Hamiltonian, $x-$translation invariant 
and 
is $Lip$-$(-1/2)$-modulo tame  with the following estimates
\begin{align}
&  {\mathfrak M}_{\mathfrak{R}_4}^{{\sharp, \g}} (-1/2,s,\mathtt{b}_0) \lesssim_{s}\gamma^{-3}(\e^{13}
+\e\|\mathfrak{I}_{\delta}\|^{\gamma,\calO_0}_{s+\mu})\,,\label{scossa1RR4}\\
&
  {\mathfrak M}_{\Delta_{12}\mathfrak{R}_{4}}^{{\sharp}} (-1/2,p,\mathtt{b}_0) 
\lesssim_p \e\gamma^{-3}
(1+\|\mathfrak{I}_{\delta}\|_{p+\mu})\|i_{1}-i_{2}\|_{p+\mu}\,,\label{scossa2RR4}
\end{align}
for some $\mu=\mu(\nu)>0$ possibly large than the one in Lemma \ref{phaseshift}.
\end{lemma}

\begin{proof}
The map $\Upsilon_{\geq 4}$ is the composition of three maps $\Upsilon_4, \Upsilon_5, \Upsilon_6$ defined in the following way:\\
the $\Upsilon_i$, $i=4, 5, 6$, are the time one flow map of a Hamiltonian system with generator
\begin{equation}\label{coeffA4}
({\bf A}_{i})_{\s,j}^{\s',k}(\ell):=\left\{
\begin{aligned}
&(\mathtt{Q}_i^{(i-1)})_{\s,j}^{\s',k}(\ell)/ \ii \delta^{(i)}_{\s,\s',j,k}(\ell)\,,\quad \max\{|j|,|k|\}\geq \mathtt{C}\,, \ell\neq0\,,
|\ell|\leq i\\
&0 \qquad \qquad \qquad \qquad \qquad \qquad {\rm otherwise}\,,
\end{aligned}\right.
\end{equation}
where $\mathtt{Q}_4^{(3)}$ is in \eqref{elle9}, $\mathtt{Q}_5^{(4)}=\pi_{\mathtt{C}^{\perp}} \mathtt{Q}_5^{(4)}$ is the term of order $\varepsilon^5$ in the Taylor expansion of $\Upsilon_4 \mathcal{L}_9 \Upsilon_4^{-1}$  and $\mathtt{Q}_6^{(5)}=\pi_{\mathtt{C}^{\perp}} \mathtt{Q}_6^{(5)}$ is the term of order $\varepsilon^6$ in the Taylor expansion of $\Upsilon_5\Upsilon_4 \mathcal{L}_9 \Upsilon_4^{-1} \Upsilon_5^{-1}$ .\\
Then the proof follows by following the same proof of Lemma \ref{stepLBNF3} for each conjugation through the maps $\Upsilon_i$, $i=4, 5, 6$.
\end{proof}

\subsection{Terms \texorpdfstring{$O(\varepsilon^{\geq 3})$}{O(e3)}: analysis of low modes}\label{prepreKAMKAM}
In this subsection we conclude the proof of Proposition \ref{ConcluLinear}.
In particular  
we normalize the term $\pi_{\mathtt{C}}\mathfrak{Q}$ 
in \eqref{elle10}, 
which is finite rank \emph{but}  not perturbative for the KAM scheme.

\begin{lemma}\label{prelKAMKAM}
Recall the constant $\mathtt{C}$ in Lemma \ref{lemmaeqomologica} and set $\mathtt{C}_2:=\mathtt{C}$.
There exists a constant $\mathtt{C}_1>0$ such that 
for any $\omega\in\mathcal{G}_0^{(2)}(\mathtt{C}_1, \mathtt{C}_2)$ (see \eqref{G02})
there exists a map $\Psi(\varphi)\colon H^s_{S^{\perp}}(\T)\times H^s_{S^{\perp}}(\T) 
\to H^s_{S^{\perp}}(\T)\times H^s_{S^{\perp}}(\T)$ such that
\begin{equation}\label{elle13}
\mathcal{L}_{14}:=\Psi \mathcal{L}_{10} \Psi^{-1} 
=
\omega\cdot \partial_{\varphi}
+\mathfrak{D}
+\Pi_S^{\perp}\mathfrak{R}
\end{equation}
where $\mathfrak{D}$ is given in \eqref{diagD}.
The operator  $\mathfrak{R}$ is in $\mathfrak{S}_0$
and 
is $Lip$-$(-1/2)$-modulo tame. Moreover $\mathfrak{R}$
satisfies the estimates  \eqref{scossa1RR5}.
%, \eqref{scossa2RR5}.
%\begin{align}
%&  {\mathfrak M}_{\mathfrak{R}_5}^{{\sharp, \g}} (s,\mathtt{b}_0) \lesssim_{s}\gamma^{-3}(\e^{13}
%+\e\|\mathfrak{I}_{\delta}\|^{\gamma,\calO}_{s+\s})\,,\label{scossa1RR5}\\
%&
%  {\mathfrak M}_{\Delta_{12}\mathfrak{R}_{5}}^{{\sharp}} (p,\mathtt{b}_0) 
%\leq \e\gamma^{-3}
%(1+\|\mathfrak{I}_{\delta}\|_{p+\s}^{\gamma,\calO})\|i_{1}-i_{2}\|_{p+\s}\,,\label{scossa2RR5}
%\end{align}
%for some $\s=\s(\tau,\nu)>0$ possibly large than the one in Lemma \ref{stepLBNF456}.
\end{lemma}
\begin{proof}
The aim of the procedure is to reduce the size of the remainder $\mathfrak{Q}$
which has the property $\pi_{\mathtt{C}}\mathfrak{Q}=\mathfrak{Q}$
up to diagonal (in space and time) terms. 
This is done iteratively by applying a finite number (five) of changes of coordinates.
%The map $\Psi$ is constructed as  the composition of five 
%symplectic  transformations $\Psi_i$, $i=1,\ldots,5$.
%
%To do this we shall apply changes of coordinates 
%generated 
Such maps will be constructed as flows 
of operators $\mathfrak{A}_i$
 such that $\pi_{\mathtt{C}}\mathfrak{A}_i=\mathfrak{A}_i$.
 As a consequence the flow map $\Psi_i=\exp(\mathfrak{A}_i)$ leaves
 invariant the diagonal operators $R$ such that $\pi_{\mathtt{C}}^{\perp}R=R$ (see \eqref{bertier}).
 Then we set $\Psi=\Psi_{5}\circ\cdots \circ\Psi_1$.
 The choice of $\mathfrak{A}_i$ requires to solve an homological 
 equation involving, at each step,  only 
 the adjoint action of the operator
\[
(\bar{\omega}+\e^{2}\mathbb{A}\zeta)\cdot\partial_{\varphi} 
+\mathtt{D}_0+\e^{2}\hat{\mathtt{D}}_2
\] 
(see \eqref{omoeqAA1Frak}). We remark that the adjoint action of the above 
operator acting on 
$x$-translation invariant operators is diagonal with 
eigenvalues
\[
\ii(\bar{\omega}+\e^{2}\mathbb{A}\zeta\cdot\ell)+
\ii \s\sqrt{|j|}-\ii \s'\sqrt{|k|}+\ii m_1(j-k)=
\ii\big((\bar{\omega}+\varepsilon^2 (\mathbb{A}\zeta-m_1\,\mathtt{v}))\cdot 
\ell+\s\sqrt{|j|}-\s'\sqrt{|k|} \big)
\]
for $\ell\in \mathbb{Z}^{\nu}$, $j,k\in S^{c}$ with $\mathtt{v}\cdot \ell+ j- k=0$.
Since $\omega$ is chosen in $\mathcal{G}_0^{(2)}(\mathtt{C}_1, \mathtt{C}_2)$ 
(see \eqref{G02})
the kernel of the adjoint action, restricted to $\mathtt{C}_1$-almost 
diagonal operators, contains just diagonal matrices.
By the \eqref{formaResto2QQ2tilde4} the size of 
these new diagonal corrections will be $O(\e^{3})$.
As a consequence, at each step, the size of the remainder 
we want to cancel out
will be $O({\e^{3+3k}\gamma^{-k}})$.
It is easy to check that for $k=5$
the size of the last remainder is less than $\e^{13}\gamma^{-3}$, i.e.
it is comparable with the remainder $\mathfrak{R}_4$
which is already \emph{perturbative} for the KAM scheme.

 We now provide the proof of the first step. The others will follow in the same way.
The first map is defined as the time one flow map of the Hamiltonian generated by (recall \eqref{elle10})
\begin{equation}\label{mathfrakA1}
(\mathfrak{A}_1)_{\s, j}^{\s', k}(\ell):=\begin{cases}
\dfrac{(\mathfrak{Q})_{\s, j}^{\s', k}(\ell)}{\mathrm{i} 
\big((\bar{\omega}+\varepsilon^2 (\mathbb{A}\zeta-m_1\,\mathtt{v}))\cdot 
\ell+\s\sqrt{|j|}-\s'\sqrt{|k|} \big)} \qquad \mbox{if}\,\,(\ell,j, k)\neq (0, j, j),\\[2mm]
0 \qquad \qquad \qquad \qquad \qquad \qquad \qquad \qquad \qquad \qquad\qquad \mbox{otherwise}.
\end{cases}
\end{equation}
We observe that $\mathfrak{A}_1$ shares the same properties of $\mathfrak{Q}$, namely it is $C_{\mathfrak{Q}}$-almost-diagonal (recall \eqref{formaResto2QQ2tilde4}) and it belongs to $\mathfrak{S}_0$.
By \eqref{formaResto2QQ2tilde4} and by the fact that $\omega\in \mathcal{G}_0^{(2)}(\mathtt{C}_1, \mathtt{C}_2)$, by choosing $\mathtt{C}_1>C_{\mathfrak{Q}}$, we have that
\begin{equation}\label{formaResto2QQ2tilde4KAM}
|(\mathfrak{A}_1)_{\s,j}^{\s,k}(\ell)|\leq \frac{C_i\,\varepsilon^3\gamma^{-1}}{\langle j,k\rangle^{1/2} }\,,
\qquad |(\mathfrak{A}_1)_{\s,j}^{-\s,k}(\ell)|\leq \frac{C_i\,\varepsilon^3\gamma^{-1}}{\langle j,k\rangle^{\rho-6}}\,.
%\quad (\mathfrak{A}_1)_{\s,j}^{\s,k}(\ell)\neq0\;\; \Rightarrow \;\;|\ell|>C\,,
\end{equation} 
Then by item $(i)$ of Lemma \ref{compoC12} we have that $\mathfrak{A}_1$ is $(-1/2)$-Lip-modulo tame with
\begin{equation}\label{stimafrakA}
\mathfrak{M}_{\mathfrak{A}_1}^{\sharp,\gamma}(-1/2,s,\mathtt{b_0})\lesssim_{s} 
\varepsilon^3 \gamma^{-1}\stackrel{\eqref{gammaDP}, \eqref{donnapia}}{\lesssim_{s}}
\e^{1-a}\,.
\end{equation}
Note that, thanks to \eqref{mathfrakA1}, 
the operator $\mathfrak{A}_1$ satisfies (recall \eqref{doppiebra})
\begin{equation}\label{omoeqAA1Frak}
\pi_{\mathtt{C}}\mathfrak{A}_1=\mathfrak{A}_1\,,\qquad 
{\rm ad}_{(\bar{\omega}+\e^{2}\mathbb{A}\zeta)\cdot\partial_{\varphi} 
+\mathtt{D}_0+\e^{2}\hat{\mathtt{D}}_2}[\mathfrak{A}_1]=\mathfrak{Q}-\bral \mathfrak{Q}\brar\,.
\end{equation}
By Definition \ref{ProiettoCCC} we also note that, given an operator $R$, 
\begin{equation}\label{bertier}
%\pi_{\mathtt{C}}^{\perp}R\in {\rm Ker}(\mathfrak{A}_1)\,,\quad 
[\mathfrak{A}_1,\pi_{\mathtt{C}}R]=\pi_{\mathtt{C}}[\mathfrak{A}_1,\pi_{\mathtt{C}}R]\,,
\qquad
\big[\mathfrak{A}_1,\bral\pi_{\mathtt{C}}^{\perp}R\brar\big]=0\,.
\end{equation}
Setting $\Psi_1=\exp(\mathfrak{A}_1)$,
we remark that
the conjugate $\Psi_1 \mathfrak{R}_4\Psi_1^{-1}$ is Lip-$(-1/2)$-modulo tame
and satisfies estimates like
\eqref{scossa1RR5} %\eqref{scossa2RR5}
by \eqref{formaResto2QQ2tilde4KAM}, \eqref{stimafrakA}, \eqref{scossa1RR4}, \eqref{scossa2RR4}
and using Lemma \ref{compoC12}.
Reasoning as in item $(ii)$ of Lemma \ref{mercato}
we deduce that (recall \eqref{elle10}, \eqref{espandoSimboM22})
\[
{\rm ad}_{\mathfrak{A}_1}(\mathcal{D}_{\geq4}+\mathfrak{Q})=B
\]
is almost-diagonal and 
\begin{equation}\label{formaResto2QQ2tilde4KAM2}
|B_{\s,j}^{\s,k}(\ell)|\leq \frac{C\,\varepsilon^6\gamma^{-1}}{\langle j,k\rangle^{1/2} }\,,
\qquad |B_{\s,j}^{-\s,k}(\ell)|\leq \frac{C\,\varepsilon^6\gamma^{-1}}{\langle j,k\rangle^{\rho-7}}\,,
%\quad (\mathfrak{A}_1)_{\s,j}^{\s,k}(\ell)\neq0\;\; \Rightarrow \;\;|\ell|>C\,,
\end{equation} 
where we used estimates \eqref{formaResto2QQ2tilde4KAM}, \eqref{formaResto2QQ2tilde4} and Lemma \ref{mercato}. 
Here $C>0$ is some constant depending on $S$
possibly larger than $C_{\mathfrak{Q}}$.
Therefore, applying Lemmata \ref{compoC12}, \ref{mercato} and  
using the Lie expansion \eqref{LieAD}, the \eqref{omoeqAA1Frak} and the 
\eqref{bertier},
we get
\begin{equation}\label{elle11Quad}
\mathcal{L}_{11}:=\Psi_1 \mathcal{L}_{10} \Psi_1^{-1} 
=
\omega\cdot \partial_{\varphi}
+\mathtt{D}_0+\varepsilon^2 \widehat{\mathtt{D}}_2
+\mathcal{D}_{\geq 4}+\bral\pi_{\mathtt{C}}\mathfrak{Q}\brar+\pi_{\mathtt{C}}\mathfrak{Q}_2
+\varepsilon^4\bral\pi^{\perp}_{\mathtt{C}} \mathtt{Q}_{4}^{(4)} \brar+\varepsilon^6\bral\pi^{\perp}_{\mathtt{C}} \mathtt{Q}_{6}^{(4)} \brar
+\mathfrak{R}_{5}
\end{equation}
where
\[
\pi_{\mathtt{C}}\mathfrak{Q}_2=
\mathfrak{Q}_2:=\sum_{k=1}^{4}\frac{1}{k!}{\rm ad}^{k}_{\mathfrak{A}_1}(\mathcal{D}_{\geq4}+\mathfrak{Q})+
\sum_{k=2}^{4} \frac{1}{k!}{\rm ad}^{k}_{\mathfrak{A}_1}
(\omega\cdot\pa_{\vphi}+\mathtt{D}_0+
\e^{2}\hat{\mathtt{D}}_2)\,.
\]
Moreover the remainder $\mathfrak{R}_5$ is Lip-$(-1/2)$-modulo tame
and satisfies estimates like
\eqref{scossa1RR5}, %\eqref{scossa2RR5}, 
the operator $\mathfrak{Q}_2$
%is almost-diagonal with 
%$\pi_{\mathtt{C}}{\mathfrak{Q}}_2=\mathfrak{Q}_2$  and 
satisfies estimates
 \eqref{formaResto2QQ2tilde4KAM2}.
 In particular $\mathfrak{Q}_2$, $\mathfrak{R}_{5}$
are Hamiltonian and $x$-translation invariant.
The remainder $\mathfrak{Q}_2$ is the one that will be normalized in the second step.
One iterates this procedure by reasoning as in the step above,
taking $\omega\in \mathcal{G}^{(0)}_2(\mathtt{C}_1,\mathtt{C}_2)$
where $\mathtt{C}_1$ is chosen large enough
in such a way that the terms we have to normalize are again $\mathtt{C}_1$-almost 
diagonal. After five steps one gets the \eqref{elle13} for some 
remainder $\mathfrak{R}$, with the same properties of $\mathfrak{R}_{5}$, and where
\[
\mathfrak{D}:=
\mathtt{D}_0+\varepsilon^2 \widehat{\mathtt{D}}_2
+\mathcal{D}_{\geq 4}+\bral\pi_{\mathtt{C}}\widetilde{\mathfrak{Q}}\brar
+\varepsilon^4\bral\pi^{\perp}_{\mathtt{C}} \mathtt{Q}_{4}^{(4)} \brar+\varepsilon^6\bral\pi^{\perp}_{\mathtt{C}} \mathtt{Q}_{6}^{(4)} \brar
\]
for some $\widetilde{\mathfrak{Q}}$ $\mathtt{C}_1$-almost diagonal and satisfying \eqref{formaResto2QQ2tilde4}.
By Definition \ref{ProiettoCCC}, \eqref{espandoSimboM22} and \eqref{espandoSimboM}
one has that $\mathfrak{D}={\rm diag}(\ii d_j)$ where $d_j$ have the form \eqref{finaleigen}.
\end{proof}

\begin{proof}[{\bf Proof of Proposition \ref{ConcluLinear}}]
Let us define (recall \eqref{Fi1dp})
\[
\Upsilon:=\Psi\circ\Upsilon_{6}\circ\Upsilon_{5}\circ\Upsilon_{4}\circ\Upsilon_{3}\circ\Upsilon_{*}\circ\Upsilon_{2}\circ\Upsilon_{1}
\]
Then the result follows by Lemmata \ref{stepLBNF1}, \ref{stepLBNF2},
Proposition \ref{identificoBNF}, Lemmata \ref{phaseshift}, \ref{stepLBNF3}, 
\ref{stepLBNF456}
and \ref{prelKAMKAM}.
\end{proof}

%%%%%%%%%%%%%%%%%%%%%%%%%%%%%%%%%%%%%%%%%%%%%%%%%

\section{Inversion of the linearized operator 
%and inversion of the linearized operator
}\label{sez:KAMredu}

In this section we conclude the inversion of the linearized operator in the normal directions.
We diagonalize the operator $\mathcal{L}_{14}$ in \eqref{elle11fine}
through a KAM reducibility scheme.
Then we will prove the
inversion assumption \ref{InversionAssumptionDP}.

%\subsection{Diagonalization}\label{diagdiagonalizzo}
We state the KAM reducibility theorem. We remark that the small divisors in \eqref{OmegoneInfinitoDP} do not lose derivatives as in \cite{BBHM}.
Hence the proof shall follow the same line of Theorem $7.13$ in \cite{FGP1}. The only difference relies on the fact that the KAM scheme has to run into the class of $x$-translation invariant matrices. However it is easy to see that this is a direct consequence of the fact that the remainder $\Pi_{S}^{\perp} \mathfrak{R}\in\mathfrak{S}_0$ is $x$-translation invariant since, by induction, each KAM transformation turns out to be $x$-translation invariant.

\begin{theorem}{{\bf (Reducibility).}}\label{ReducibilityWW}
Fix $\tilde{\g}\in [\gamma^{3}/4, 4\gamma^{3}]$ and $\tau= 3\nu+7, \tb_0:=6\tau+6$.
Assume that $\omega \mapsto i_{\delta}(\omega)$ is a Lipschitz function defined on  $\calO_0\Subset \Omega_{\varepsilon}$ (recall   \eqref{AssumptionDP}), satisfying \eqref{IpotesiPiccolezzaIdeltaDP} with $\gotp_1 \geq\mu$ where $\mu=\mu(\nu)$ 
is given in Proposition \ref{ConcluLinear}. %and $\tb_0:=6 \tau+5$. 
There exist $\delta_0\in (0, 1)$, $N_0>0$, $C_0>0$, such that, if 
(recall that by \eqref{donnapia}, \eqref{gammaDP} $\g= \e^{2+a}$)
\begin{equation}\label{PiccolezzaperKamredDP}
N_0^{C_0} \varepsilon^{13} \gamma^{-6}%=N_0^{C_0} \varepsilon^{1-(9/2) a}
\le \delta_0, 
\end{equation}
then the following holds.
\begin{itemize}
\item[(i)] \textbf{(Eigenvalues)}. For all $\omega\in \Omega_{\varepsilon}$ there exists a sequence \footnote{Whenever it is not strictly necessary, we shall drop the dependence on $i_{\delta}$.}
\begin{align}\label{FinalEigenvaluesDP}
&d_j^{\infty}(\omega):=d_{j}(\omega, i_{\delta}(\omega))+
%d_j^{\infty}(\omega, i_{\delta}(\omega)):=m(\omega, i_{\delta}(\omega))\,\og(j)
%+\varepsilon^2 \kappa_j (\omega) +
 r_j^{(\infty)}(\omega, i_{\delta}(\omega)), \quad j\in S^c,
\end{align}
with $d_{j}$ given in \eqref{finaleigen}
%and $\kappa_j$ in \eqref{clinica} and \eqref{diagonalop} respectively.
 Furthermore, for all $j\in S^c$
\begin{equation}\label{stimeautovalfinaliDP}
 \sup_j \langle j\rangle^{\frac{1}{2}} |r_j^{(\infty)}|^{\gamma^{3}} 
 \lesssim \e^{13}\gamma^{-3}, \qquad r_j^{(\infty)}\in \mathbb{R}\,.
 %r_j^{\infty}=-r^{\infty}_{-j}.
\end{equation}
 %All the eigenvalues $\mathrm{i} d_j^{\infty}$ are purely imaginary. 
%We define, for convenience, $d_0^{\infty}(\omega):=0$.
\item[(ii)] \textbf{(Conjugacy)}. For all $\omega$ in the set (see \eqref{omegaduegamma}) 
%$\Omega_{\infty}^{2\gamma}\cap S_{\infty}^{2\tilde{\gamma}}$

\begin{equation}\label{OmegoneInfinitoDP}
\begin{aligned}
\begin{aligned}
S^{2 {{\gamma}^{3}}}_{\infty}:=S^{2 {\gamma}^3}_{\infty}(i_{\delta}):=
&\big\{ \omega\in\Omega^{2\gamma}_{\infty} 
: \lvert \omega\cdot \ell+\s d_{\s j}^{\infty}(\omega)-\s' d_{\s' k}^{\infty}(\omega)\rvert
\geq \frac{2 {{\gamma}^{3}}}{\langle \ell \rangle^{\tau}},\\
& \quad
\,\,\mathtt{v}\cdot\ell+ j-k=0 \,, \forall \ell\in\mathbb{Z}^{\nu}, \,\,\forall j, k\in S^c,\, j\neq k\,,\s,\s'=\pm  \big\}
\end{aligned}
\end{aligned}
\end{equation}
there is a real-to-real, bounded, invertible, linear operator 
$\Phi_{\infty}(\omega)\colon H^s_{S^{\perp}}(\T)\times H^s_{S^{\perp}}(\T)\to
H^s_{S^{\perp}}(\T)\times H^s_{S^{\perp}}(\T)$, 
with bounded inverse $\Phi_{\infty}^{-1}(\omega)$, that conjugates $\mathcal{L}_{14}$ 
in \eqref{elle11fine} 
to constant coefficients, namely 
\begin{equation}\label{Linfinito}
\begin{aligned}
&\mathcal{L}_{\infty}(\omega):=\Phi_{\infty} (\omega) \circ \mathcal{L}_{14} \circ \Phi_{\infty}^{-1}(\omega)=\omega\cdot \partial_{\varphi}+\mathfrak{D}_{\infty}(\omega),\\
& \mathfrak{D}_{\infty}:=((\mathfrak{D}_{\infty})_{\s}^{\s'})_{\s,\s'=\pm}\,, \quad 
(\mathfrak{D}_{\infty})_{\s}^{-\s}\equiv0\,,\;\; 
\ov{(\mathfrak{D}_{\infty})_{+}^{+}}=(\mathfrak{D}_{\infty})_{-}^{-}\,,\quad
(\mathfrak{D}_{\infty})_{+}^{+}:=\mathrm{diag}_{j\in S^c} \{ \mathrm{i} d_j^{\infty}(\omega) \}.
\end{aligned}
\end{equation}
The transformations $\Phi_{\infty}, \Phi_{\infty}^{-1}$ are tame and they satisfy for $s_0\le s\le \mathcal{S}$ 
\begin{equation}\label{grano}
\lVert (\Phi^{\pm 1}_{\infty}-\mathrm{I}) h \rVert^{\gamma^{3}, S_{\infty}^{2 \gamma^3}}_{s}
\lesssim_s \big(\varepsilon^{13} \gamma^{-6}+\varepsilon \gamma^{-6} 
\lVert \mathfrak{I}_{\delta} \rVert_{s+\mu}^{\gamma, \calO_0}\big) 
\lVert h \rVert^{\gamma^{3}, S_{\infty}^{2 \gamma^3}}_{s_0}
+\varepsilon^{13} \gamma^{-6} \lVert h \rVert^{\gamma^{3}, S_{\infty}^{2\gamma^3}}_s.
\end{equation}
Moreover $\Phi_{\infty}, \Phi_{\infty}^{-1}$ are symplectic
and $x$-translation invariant, and $\mathcal{L}_{\infty}$ is a Hamiltonian 
and $x$-translation invariant
operator.
\item[(iii)] \textbf{(Dependence on $i_{\delta}(\omega)$)}. Let $i_1(\omega)$ and $i_2(\omega)$ be two Lipschitz maps satisfying \eqref{IpotesiPiccolezzaIdeltaDP} with $\mathfrak{I}_{\delta}\rightsquigarrow i_k(\varphi)-(\varphi, 0, 0)$, $k=1, 2$, and such that
\begin{equation}
\lVert i_1-i_2 \rVert_{s_0+\mu}\lesssim \rho N^{-(\tau+1)}
\end{equation}
for $N$ sufficiently large and $0<\rho<\gamma^{3}/4$. Fix $\g_1\in [\gamma^{3}/2, 2 \gamma^{3}]$ 
and $\g_2:=\g_1-\rho$. 
Let $r_j^{(\infty)}(\omega, i_k(\omega))$ 
be the sequence in \eqref{FinalEigenvaluesDP} 
with $\tilde{\g}\rightsquigarrow \g_k$ for $k=1, 2$. 
Then for all $\omega\in S_{\infty}^{\g_1}(i_1)$ 
we have, for some $\kappa_*>(3/2) \tau$, 
\begin{equation}\label{r12}
%\g^{-1}\lvert \Delta_{12} m \rvert+ 
\sup_j \langle j\rangle^{1/2} |\Delta_{12}r_j^{(\infty)}| \le 
\e \g^{-3} \| i_1-i_2 \|_{s_0+\mu}+\varepsilon^{13}\gamma^{-3} N^{-\kappa_*}.
\end{equation}
\end{itemize}
\end{theorem}

%\subsection{Inversion of the linearized operator}\label{InversionLIN}
%{Proof of the inversion assumption \eqref{InversionAssumptionDP}}

We are in position to give estimates on the inverse of the operator $\mathcal{L}_{\omega}$ 
in \eqref{BoraMaledetta} and give the proof
of  the inversion assumption \eqref{InversionAssumptionDP}
Let us  now define (recall \eqref{gammaDP})
\begin{equation}\label{primediMelnikov}
\mathcal{F}_{\infty}^{2\gamma}(i_{\delta})
:=\big\{ \omega\in \calO_0 
: \lvert \omega\cdot \ell+\s\, d_{\s j}^{\infty}(\omega) 
\rvert\geq \frac{2\gamma}{\langle \ell \rangle^{\tau}} , \,\,\mathtt{v}\cdot \ell+ j=0,
\,\, \forall \ell\in\mathbb{Z}^{\nu} , \forall j\in S^c, \s=\pm\big\}\,.
\end{equation}
We deduce the inversion assumption \eqref{InversionAssumptionDP} by the following result.

\begin{proposition}{\bf (Invertibility).}\label{InversionLomegaDP}
Assume the hypotheses of Theorem \ref{ReducibilityWW}, 
\eqref{IpotesiPiccolezzaIdeltaDP} with $\gotp_1\geq \mu+2\tau+1$, where $\mu$ is given in 
Proposition \ref{ConcluLinear}. 
Then for all $\omega\in \calO_{\infty}:=S_{\infty}^{2 \gamma^{3}}(i_{\delta}) \cap 
\mathcal{F}_{\infty}^{2\gamma}(i_{\delta})$ (see \eqref{OmegoneInfinitoDP}), 
for any 
$G:=(g, \bar{g})\in \Big(H_{S^{\perp}}^{s+2\tau+1}\cap S_{\mathtt{v}}\Big)^2$ 
(see \eqref{funzmom}) 
the equation $\mathcal{L}_{\omega} h=G$ has a solution 
$h=\mathcal{L}_{\omega}^{-1} G\in \Big(H_{S^{\perp}}^{s}\cap S_{\mathtt{v}}\Big)^2$, satisfying
\begin{equation}\label{TameEstimateLomegaDP}
\begin{aligned}
\lVert \mathcal{L}_{\omega}^{-1} G \rVert_s^{\gamma, \calO_{\infty}} 
&\lesssim_s \gamma^{-1} 
(\lVert G \rVert_{s+2\tau+1}^{\gamma, \calO_{\infty}}
+\varepsilon \gamma^{-6}
\lVert \mathfrak{I}_{\delta} \rVert_{s+\gotp_1}^{\gamma, \calO_{0}}
\lVert G \rVert_{s_0}^{\gamma, \calO_{\infty}}).
\end{aligned}
\end{equation}
\end{proposition}

\begin{proof}
By Propositions \ref{lemma:7.4}, \ref{lem:compVar}, \ref{lem:mappa2}, \ref{blockTotale}, \ref{riduzSum}, 
\ref{ConcluLinear} and Theorem \ref{ReducibilityWW}  we have that,
recall \eqref{BoraMaledetta},
%the operator $\mathcal{L}_{\omega}$ in \eqref{BoraMaledetta} 
%transforms in to the diagonal opear
%to a diagonal operator 
%$\mathcal{L}_{\infty}=\chi \mathcal{L}_{\omega} \chi^{-1}$, 
\begin{equation}\label{chichichi}
\mathcal{L}_{\infty}=\chi \mathcal{L}_{\omega} \chi^{-1}\,,\qquad 
\chi:=\Phi_{\infty}\circ \Upsilon\circ{\bf \Phi}\circ{\bf \Psi}\circ
\Phi_{\mathbb{M}}\circ\Lambda\circ\Phi_{\mathbb{B}}^{-1}\,.
\end{equation}
Moreover, by \eqref{AVB3}, \eqref{AVB3mappa2PRO}, 
\eqref{stimaMappaPsi}, \eqref{stimaMappaPhi}, 
\eqref{boundUpsilonTotale} 
and \eqref{grano} %and Lemma \ref{isotropictorus} 
we have the following estimates
\begin{equation}\label{nebbia}
\lVert \chi^{\pm 1} h \rVert_s^{\gamma, \calO_{\infty}}
\lesssim_s \lVert h \rVert^{\gamma, \calO_{\infty}}_s
+\varepsilon \gamma^{-6}\lVert \mathfrak{I}_{\delta} \rVert_{s+\mu}^{\gamma, \calO_0}   
\lVert h \rVert^{\gamma, \calO_{\infty}}_{s_0}\,.
\end{equation}
We have
\begin{equation}\label{pinkpink}
\mathcal{L}_{\infty}^{-1} G=\begin{pmatrix}
\sum_{j\in S^c, \ell\in \mathbb{Z}^{\nu}} \frac{g_{\ell j}}{\mathrm{i}\big(\omega\cdot \ell+d_j^{\infty}(\omega)\big)}\,e^{\mathrm{i}(\ell \cdot \varphi+j x)}\\
\sum_{j\in S^c, \ell\in \mathbb{Z}^{\nu}} \frac{\bar{g_{\ell j}}}{\mathrm{i}\big(\omega\cdot \ell-d_{-j}^{\infty}(\omega)\big)}\,e^{-\mathrm{i}(\ell \cdot \varphi+j x)}
\end{pmatrix}
\qquad
\stackrel{\eqref{primediMelnikov}}{\Rightarrow}\qquad
\lVert \mathcal{L}_{\infty}^{-1} G \rVert^{\gamma, \calO_{\infty}}_s
\lesssim_{s} \gamma^{-1} \lVert G \rVert_{s+2\tau+1}^{\gamma, \calO_{\infty}}\,.
\end{equation}
Thus, by \eqref{nebbia},  
we get the estimate \eqref{TameEstimateLomegaDP}.
Recalling \eqref{vascoblasco31}
we have that if $g\in S_{\mathtt{v}}$ then $g_{\ell j}\neq 0$ implies 
$\mathtt{v}\cdot \ell+j=0$. Hence $\mathcal{L}_{\infty}^{-1} G$
belongs to $S_{\mathtt{v}}$. Finally we deduce that $h$ is in $S_{\mathtt{v}}$
since the map $\chi$ in \eqref{nebbia} is $x$-translation invariant (see \eqref{traInvcomp}).
\end{proof}

%%%%%%%%%%%%%%%%%%%%%%%%%%%%%%%%%%%%%%%%%%%%%%%%%%%%

\section{The Nash-Moser nonlinear iteration}\label{sezione9WW}

In this section we prove Theorem \ref{IlTeoremaDP}. 
It will be a consequence of the Nash-Moser 
theorem \ref{NashMoserDP}.\\
Consider the finite-dimensional subspaces
\[
E_n:=\{ \mathfrak{I}(\varphi)=(\Theta(\varphi), y(\varphi), W(\varphi, x)) \in H^{s_0}(\mathbb{T}^{\nu} ; \mathbb{R}^{\nu}) \times H^{s_0}(\mathbb{T}^{\nu} ; \mathbb{R}^{\nu}) \times H^{s_0}_{S^{\perp}} :  
\Theta=\Pi_n\Theta, y=\Pi_n y, W=\Pi_n W \}
\]
where 
\begin{equation*}%\label{Nn}
N_n:=N_0^{\chi^n}\,, \quad n=0, 1, 2, \dots\,, \quad \chi:=3/2\,, \quad N_0>0\,
\end{equation*}
and $\Pi_n$ are the projectors 
\begin{equation*}
\Pi_n \Theta(\varphi):=\sum_{\lvert \ell \rvert<N_n} \Theta_{\ell}\, 
e^{\mathrm{i} \ell\cdot \varphi}\,, \,\,
\Pi_n y(\varphi):=\sum_{\lvert \ell \rvert<N_n} y_{\ell}\,
e^{\mathrm{i} \ell\cdot \varphi}\,,\,\,
\Pi_n W(\varphi, x):=\sum_{\lvert (\ell, j) \rvert<N_n} 
W_{\ell j}\,e^{\mathrm{i}(\ell\cdot \varphi+j x)}\,, 
\end{equation*}
where
\begin{equation*}
\Theta(\varphi)=\sum_{\ell\in\mathbb{Z}^{\nu}} 
\Theta_{\ell}\,e^{\mathrm{i} \ell\cdot \varphi}\,, \,\quad
y(\varphi)=\sum_{\ell\in \mathbb{Z}^{\nu}} y_{\ell} \,
e^{\mathrm{i} \ell\cdot\varphi}\,,\quad W(\varphi, x)=
\sum_{\ell\in\mathbb{Z}^{\nu}, j\in S^c} W_{\ell j}\,e^{\mathrm{i} (\ell\cdot \varphi+j x)}\,.
\end{equation*}
\begin{remark}\label{projtrav}
We point out that if the embedding $\mathfrak{I}(\varphi):=(\Theta(\varphi), y(\varphi), W(\varphi))$ is traveling, i.e. it satisfies \eqref{trawaves}, then $\Pi_n \mathfrak{I}:=(\Pi_n \Theta, \Pi_n y, \Pi_n W)$ is traveling too. Indeed the \eqref{trawaves} is equivalent to say that each component of $\mathfrak{I}$ satisfies a linear PDE with constant coefficients (both in space and time). The action of these linear operators is diagonal on the monomials.
Hence the projection on some Fourier modes of a solution of such PDEs is still a solution of the same.
\end{remark}
We define $\Pi_n^{\perp}:=\mathrm{I}-\Pi_n$. 
The classical smoothing properties hold, namely, 
for all $\alpha, s\geq 0$,
\begin{equation*}%\label{SmoothingDP}
\lVert \Pi_n \mathfrak{I} \rVert_{s+\alpha}^{\g, \calO}\le N_n^{\alpha} \lVert \mathfrak{I}_{\delta} \rVert_s^{\g, \calO}, \quad \forall \,\mathfrak{I}(\omega)\in H^s\,, 
\quad \lVert \Pi_n^{\perp} \mathfrak{I} \rVert_s^{\g, \calO}
\le N_n^{-\alpha} \lVert \mathfrak{I} \rVert_{s+\alpha}^{\g, \calO}\,, 
\quad \forall\, \mathfrak{I}(\omega)\in H^{s+\alpha}\,.
\end{equation*}

\noindent
Recall \eqref{gammaDP}, \eqref{donnapia} for the definition of $b$ we set $a:=2b-2$.
We define the following constants 
\begin{equation}\label{parametriNMdp}
\begin{aligned}
&\alpha_0:=3\mu_1+3\,, \qquad \qquad  \qquad\alpha:=3\alpha_0+1\,, 
\qquad \qquad \qquad \alpha_1:=(\alpha-3 \mu_1)/2\,,\\
&k:=3(\alpha_0+\rho^{-1})+1\,, 
\qquad \beta_1:=6\alpha_0+3 \rho^{-1} +3\,, 
\qquad \frac{1}{2}\left(  \frac{1-8a}{C_1 (1+a)}\right)<\rho<\frac{1-8a}{C_1 (1+a)}\,
\end{aligned}
\end{equation}
where $\mu_1:=\mu_1( \nu)>0$ is the 
``loss of regularity'' given by the 
Theorem \ref{TeoApproxInvDP} and $C_1$ is fixed below. 
%\comment{controllare i parametri in \eqref{parametriNMdp}}

\begin{theorem}{{\bf (Nash-Moser).}}\label{NashMoserDP}
Let $\tau:=3\nu+7$ and recall the functional in \eqref{NonlinearFunctionalDP}.
Then there exist $C_1>\max\{ \alpha_0+\alpha, C_0 \}$ 
(where $C_0:=C_0( \nu)$ is the one in 
Theorem \ref{ReducibilityWW}), 
$\delta_0:=\delta_0( \nu)>0$ such that, if
\begin{equation}\label{SmallnessConditionNMdp}
N_0^{C_1} \varepsilon^{b_*+1} \gamma^{-7}<\delta_0\,, 
\quad \gamma:=\varepsilon^{2+a}=\varepsilon^{2 b}\,, 
\quad N_0:=(\varepsilon\gamma^{-1})^{\rho}\,, 
\quad b_*=16- 2 b\,,
\end{equation}
then there exists $C_{*}=C_{*}(S)>0$ such that 
for all $n\geq 0$ the following holds:
\begin{itemize}
\item[$(\mathcal{P}1)_n$] there exists a function 
$(\mathfrak{I}_n, \Xi_n)\colon \mathcal{G}_n 
\subseteq \Omega_{\varepsilon} \to E_{n-1}\times \mathbb{R}^{\nu}, 
\omega\mapsto (\mathfrak{I}_n(\omega), \Xi_n(\omega)), 
(\mathfrak{I}_0, \Xi_0):=(0, 0), E_{-1}:=\{0\}$,
where the set $\mathcal{G}_0$ is defined 
in \eqref{0di0Melnikov} (with $\mathtt{C}_1,\mathtt{C}_2$ 
given in Proposition \ref{ConcluLinear}) 
and the sets 
$\mathcal{G}_n$ for $n\geq 1$ 
are defined inductively by:
\begin{align}
&\mathcal{G}_{n+1}:=
\Lambda_{n+1}^{(0)}\bigcap \Lambda_{n+1}^{(1)}\bigcap \Lambda_{n+1}^{(2,+)}
\bigcap \Lambda_{n+1}^{(2,-)}\,, \nonumber\\%\;\;{\rm with}\;\;\;
&\Lambda^{(0)}_{n+1}:=\left\{ \omega\in \mathcal{G}_n : 
\lvert\omega\cdot \ell+\mathfrak{m}_1(i_n) j \rvert\geq 
\frac{2\,\gamma_n}{\langle \ell \rangle^{\tau}}, \,\,\,\mathtt{v}\cdot \ell+j=0,\,\,
\,\,\forall j\in S^c , \ell\in\mathbb{Z}^{\nu} \right\}\,, \nonumber\\ \label{GnDP}
&\Lambda^{(1)}_{n+1}:=\left\{ \omega\in \mathcal{G}_n : 
\lvert\omega\cdot \ell+ \s d_{\s j}^{\infty}(i_n) \rvert\geq 
\frac{2\,\gamma_n}{\langle \ell \rangle^{\tau}}, \,\,
\forall j\in S^c , \ell\in\mathbb{Z}^{\nu}\,, \; \mathtt{v}\cdot\ell+j=0\,, \, \s=\pm \right\}\,,\\
&\Lambda^{(2,+)}_{n+1}:=\Big\{ \omega\in \mathcal{G}_n : 
\lvert \omega\cdot \ell+\s(d_{\s j}^{\infty}(i_n)+d_{-\s k}^{\infty}(i_n)) \rvert\geq 
\frac{2\,\gamma_n\,}{\langle \ell \rangle^{\tau}}, \,\,\; \mathtt{v}\cdot\ell+j-k=0\,,\,
\forall j, k\in S^c, \ell\in\mathbb{Z}^{\nu}\,,\;\s=\pm \Big\}\,, \nonumber\\
&\Lambda^{(2,-)}_{n+1}:=\Big\{ \omega\in \mathcal{G}_n : 
\lvert \omega\cdot \ell+\s(d_{\s j}^{\infty}(i_n)-d_{\s k}^{\infty}(i_n)) \rvert\geq 
\frac{2\,\gamma^{*}_n\,}{\langle \ell \rangle^{\tau}}, \,\,
\forall j, k\in S^c, j\neq k , \ell\in\mathbb{Z}^{\nu}\,,\;
\mathtt{v}\cdot\ell+j-k=0\,, \s=\pm \Big\}\,, \nonumber
\end{align}
where $\gamma_n:=\gamma (1+2^{-n})$, 
$\gamma^{*}_n:=\gamma^{3}(1+2^{-n})$ and 
$d_j^{\infty}(\omega):=d_j^{\infty}(\omega, i_n(\omega))$ 
are defined in \eqref{FinalEigenvaluesDP}. 
%(and $d_0^{\infty}(\omega)=0$).
Moreover 
$\lvert \zeta_n \rvert^{\gamma, \mathcal{G}_n}
\lesssim \lVert \mathcal{F}(U_n) \rVert_{s_0}^{\gamma, \mathcal{G}_n}$ 
and
\begin{equation}\label{ConvergenzaDP}
\lVert \mathfrak{I}_n \rVert_{s_0+\mu_1}^{\gamma, \mathcal{G}_n}
\le C_* \varepsilon^{b_*} \gamma^{-1}\,, 
\quad \lVert \mathcal{F}(U_n) \rVert_{s_0+\mu_1+3}^{\gamma, \mathcal{G}_n}
\le C_* \varepsilon^{b_*}\,,
\end{equation}
where $U_n:=(i_n, \Xi_n)$ with 
$i_n(\varphi)=(\varphi, 0, 0)+\mathfrak{I}_n(\varphi)$. The embedding $\mathfrak{I}_n$ is traveling, namely it satisfies \eqref{trawaves}.\\
The differences 
$\hat{\mathfrak{I}}_n:=\mathfrak{I}_n-\mathfrak{I}_{n-1}$ 
(where we set $\hat{\mathfrak{I}}_0:=0$) 
is defined on $\mathcal{G}_n$, and satisfy
\begin{equation}\label{FrakHatDP}
\lVert \hat{\mathfrak{I}}_1 \rVert_{s_0+\mu_1}^{\gamma, \mathcal{G}_1}
\le C_* \varepsilon^{b_*} \gamma^{-1}\,, 
\quad \lVert \hat{\mathfrak{I}}_n \rVert_{s_0+\mu_1}^{\gamma, \mathcal{G}_{n}}
\le C_* \varepsilon^{b_*} \gamma^{-1} N_{n-1}^{-\alpha}\,, 
\quad \forall n\geq 2\,.
\end{equation}
\item[$(\mathcal{P} 2)_n$] 
$\lVert \mathcal{F}(U_n) \rVert_{s_0}^{\gamma, \mathcal{G}_n}
\le C_* \varepsilon^{b_*} N_{n-1}^{-\alpha}$ 
where we set $N_{-1}:=1$.
\item[$(\mathcal{P}3)_n$]{(High Norms)}. 
$\lVert \mathfrak{I}_n \rVert_{s_0+\beta_1}^{\gamma, \mathcal{G}_n}
\le C_* \varepsilon^{b_*} \gamma^{-1} N_{n-1}^k$ 
and $\lVert \mathcal{F}(U_n) \rVert_{s_0+\beta_1}^{\gamma, \mathcal{G}_n}
\le C_* \varepsilon^{b*} N_{n-1}^k$.
\item[$(\mathcal{P} 4)_n$]{(Measure)}. 
The measure of the ``Cantor-like'' sets $\mathcal{G}_n$ 
satisfies
\begin{equation}\label{MisureDP}
\lvert \Omega_{\varepsilon}\setminus \mathcal{G}_0 \rvert
\le C_* \varepsilon^{2 (\nu-1)} \gamma\,, 
\quad \lvert \mathcal{G}_n \setminus \mathcal{G}_{n+1} \rvert
\le C_* \varepsilon^{2 (\nu-1)} \gamma N_{n-1}^{-1}\,.
\end{equation}  
\end{itemize}
\end{theorem}

\begin{proof}
Recalling \eqref{NonlinearFunctionalDP} and Lemma \ref{TameEstimatesforVectorfieldsDP} we have
\[
\lVert \mathcal{F}(U_0)\rVert_s
=\lVert \mathcal{F}((\varphi, 0, 0), 0) \rVert_s
=\lVert X_P(\varphi, 0, 0) \rVert_s\lesssim_s \varepsilon^{16-2 b}\,.
\]
Hence the smallness conditions in 
$(\mathcal{P}_1)_0, (\mathcal{P}_2)_0, (\mathcal{P}_3)_0$ 
hold taking $C_*$ large enough.
Assume that $(\mathcal{P}_1)_n, (\mathcal{P}_2)_n, (\mathcal{P}_3)_n$ 
hold for some $n\geq 0$.
By \eqref{parametriNMdp} and \eqref{SmallnessConditionNMdp} we have that $N_0^{C_1} \varepsilon^{b_*+1} \gamma^{-7}$ is arbitrarily small with $\varepsilon$.\\
If we take $C_1$ bigger than $C_0$ in Theorem \ref{ReducibilityWW} then \eqref{PiccolezzaperKamredDP} holds for $\varepsilon$ small enough. 	
In \eqref{AssumptionDP} we consider $\mathfrak{p}_0=\gotp_1+{\mathfrak{d}}$, where $\gotp_1:=\mu+2\tau+1$ (where $\mu$ is given in 
Proposition \ref{ConcluLinear}) and ${\mathfrak{d}}$ appears in Lemma \ref{isotropictorus}. Since $\mu_1\gg \mathfrak{p}_0$ \eqref{ConvergenzaDP} implies \eqref{AssumptionDP} 
and so \eqref{IpotesiPiccolezzaIdeltaDP}, 
and Proposition \ref{InversionLomegaDP} applies. 
Hence the operator 
$\mathcal{L}_{\omega}:=\mathcal{L}_{\omega}(\omega, i_n (\omega))$  
in \eqref{BoraMaledetta} is defined on 
$\calO_0= \mathcal{G}_n$  and is invertible for all 
$\omega\in \mathcal{G}_{n+1}$ since 
$\mathcal{G}_{n+1}\subseteq \calO_{\infty}$ 
and the %last estimate in 
\eqref{TameEstimateLomegaDP} holds. 
This means that the assumption \eqref{InversionAssumptionDP} 
of Theorem  \ref{TeoApproxInvDP} is verified with 
$\Omega_{\infty}=\mathcal{G}_{n+1}$. 
By Theorem \ref{TeoApproxInvDP} 
there exists an approximate inverse 
$\textbf{T}_n(\omega):=\textbf{T}_0(\omega, i_n(\omega))$ 
of the linearized operator 
$L_n(\omega):=d \mathcal{F}(\omega, i_n(\omega),\Xi_n)
\equiv d \mathcal{F}(\omega, i_n(\omega),0)$.

%satisfying \eqref{TameEstimateApproxInvDP}. 
%By \eqref{SmallnessConditionNMdp}, \eqref{ConvergenzaDP}
%\begin{align}
%&\lVert \textbf{T}_n g\rVert_s\lesssim_s 
%\gamma^{-1} (\lVert g \rVert_{s+\mu}
%+\varepsilon\gamma^{-5/2} \{ \lVert \mathfrak{I}_n \rVert_{s+\mu}
%+\gamma^{-1} \lVert \mathfrak{I}_n \rVert_{s_0+\mu} 
%\lVert \mathcal{F}(U_n) \rVert_{s+\mu} \} \lVert g \rVert_{s_0+\mu})\,,\label{9.10dp}\\
%&\lVert \textbf{T}_n g \rVert_{s_0}\lesssim \gamma^{-1} \lVert g \rVert_{s_0+\mu}
%\end{align}
%and, by \eqref{6.41DP}, using also \eqref{SmallnessConditionNMdp}, \eqref{ConvergenzaDP}, \eqref{SmoothingDP},
%\begin{align}
%\lVert (L_n\circ \textbf{T}_n-\mathrm{I}) g \rVert_s 
%\lesssim_s & \varepsilon^{2 b-1}\gamma^{-2} 
%(\lVert \mathcal{F}(U_n) \rVert_{s_0+\mu}\lVert g \rVert_{s+\mu}
%+\lVert \mathcal{F}(U_n) \rVert_{s+\mu}\lVert g \rVert_{s_0+\mu}\notag\\
%& + \varepsilon\gamma^{-5/2}\lVert \mathfrak{I}_n \rVert_{s+\mu}
%\lVert \mathcal{F}(U_n) \rVert_{s_0+\mu}\lVert g \rVert_{s_0+\mu})\\
%\lVert (L_n\circ \textbf{T}_n-\mathrm{I}) g \rVert_{s_0} 
%%\lesssim 
%%& \varepsilon^{2 b-1}\gamma^{-2} 
%%\lVert \mathcal{F}(U_n) \rVert_{s_0+\mu} 
%%\lVert g \rVert_{s_0+\mu}\notag\\\notag
%\lesssim & \varepsilon^{2 b-1}\gamma^{-2} 
%(\lVert \Pi_n \mathcal{F}(U_n) \rVert_{s_0+\mu}
%+\lVert \Pi_n^{\perp} \mathcal{F}(U_n) \rVert_{s_0+\mu})
%\lVert g \rVert_{s_0+\mu}\\
%\lesssim &\varepsilon^{2 b-1}\gamma^{-2} N_n^{\mu}   \Big(
%\lVert \mathcal{F}(U_n) \rVert_{s_0}+N_n^{-\beta_1}
%\lVert \mathcal{F}(U_n) \rVert_{s_0+\beta_1} \Big) \lVert g \rVert_{s_0+\mu}\,.
%\end{align}

Now, for all $\omega\in\mathcal{G}_{n+1}$, we can define, for $n\geq 0$,
\begin{equation}\label{HnDefDP}
U_{n+1}:=U_n+H_{n+1}, \quad H_{n+1}
:=(\hat{\mathfrak{I}}_{n+1}, \hat{\Xi}_{n+1})
:=-\tilde{\Pi}_n \textbf{T}_n \Pi_n \mathcal{F}(U_n)\in 
E_n\times\mathbb{R}^{\nu}\,,
\end{equation}
where 
$\tilde{\Pi}_n(\mathfrak{I}, \Xi):=(\Pi_n \mathfrak{I}, \Xi)$ . By the inductive hypothesis $\mathfrak{I}_n$ is traveling, namely it satisfies \eqref{trawaves}. We claim that $\hat{\mathfrak{I}}_{n+1}$ satisfies \eqref{trawaves}. Indeed by Lemma \ref{Fpreservatrav} $\mathcal{F}(U_n)$ is traveling, by Remark \ref{projtrav} the projection $\Pi_n$ of a traveling embedding is a traveling embedding,  by Theorem  \ref{TeoApproxInvDP} the approximate inverse $\mathbf{T}_n$ maps traveling embeddings into traveling embeddings.
By construction we have
\begin{align*}
&\mathcal{F}(U_{n+1})=\mathcal{F}(U_n)+L_n H_{n+1}+Q_n\,,\\
Q_n:=&Q(U_n, H_{n+1})\,, 
\quad Q(U_n, H):=\mathcal{F}(U_n+H)-\mathcal{F}(U_n)-L_n H\,, 
\quad H\in E_n\times\mathbb{R}^{\nu}\,.
\end{align*}
%\[
%\mathcal{F}(U_{n+1})=\mathcal{F}(U_n)+L_n H_{n+1}+Q_n\,,
%\]
%where
%\begin{equation}
%Q_n:=Q(U_n, H_{n+1}), \quad Q(U_n, H):=\mathcal{F}(U_n+H)-\mathcal{F}(U_n)-L_n H, \quad H\in E_n\times\mathbb{R}^{\nu}.
%\end{equation}
Then, by the definition of $H_{n+1}$ in \eqref{HnDefDP}, 
using $[L_n, \Pi_n]$ and writing 
$\tilde{\Pi}_n^{\perp}(\mathfrak{I}, \Xi):=(\Pi_n^{\perp} \mathfrak{I}, 0)$ 
we have
\begin{equation*}
\begin{aligned}
\mathcal{F}(U_{n+1})&=
\mathcal{F}(U_n)-L_n \tilde{\Pi}_n \textbf{T}_n \Pi_n \mathcal{F}(U_n)+Q_n
=\mathcal{F}(U_n)-L_n \textbf{T}_n \Pi_n \mathcal{F}(U_n)
+L_n \tilde{\Pi}_n^{\perp} \textbf{T}_n \Pi_n \mathcal{F}(U_n)+Q_n\\
&=\mathcal{F}(U_n)-\Pi_n L_n \textbf{T}_n \Pi_n \mathcal{F}(U_n)
+(L_n \tilde{\Pi}_n^{\perp}-\Pi_n^{\perp}L_n)
\textbf{T}_n \Pi_n \mathcal{F}(U_n)+Q_n
%\\&
=\Pi_n^{\perp} \mathcal{F}(U_n)+R_n+Q_n+Q'_n
\end{aligned}
\end{equation*}
where
\begin{equation}
R_n:=(L_n \tilde{\Pi}_n^{\perp}-\Pi_n^{\perp}L_n)
\textbf{T}_n \Pi_n \mathcal{F}(U_n)\,, 
\quad Q'_n:=-\Pi_n (L_n \textbf{T}_n-\mathrm{I})
\Pi_n \mathcal{F}(U_n)\,.
\end{equation}

The estimates in $(\mathcal{P} 1)_n$, $(\mathcal{P} 2)_n$, $(\mathcal{P} 3)_n$
follow word by word as in 
section $8$ in \cite{FGP1}. The measure estimates $(\mathcal{P} 4)_n$, which are different respect to \cite{FGP1}, are proved in the next section.
\end{proof}

\subsection{Measure estimates}\label{sezioneStimeMisura}
In this section we give the proof of the bound \eqref{MisureDP}.
Recall \eqref{FinalEigenvaluesDP}, \eqref{finaleigen}.
Let us define for $0<\eta\lesssim\,\sqrt{\varepsilon}$, 
$\tau_0\geq 1$ and $n\in \mathbb{N}$
\begin{align*}
R^{(-)}_{\ell j k}(\eta, \tau_0):=R^{(-)}_{\ell j k}(i_{n},\eta, \tau_0)
&:=\{ \omega\in\mathcal{G}_n : \lvert \omega\cdot \ell
+ \s(d_{\s j}^{\infty}-  d_{\s k}^{\infty}) \rvert\le 
{2\eta}\langle \ell \rangle^{-\tau_0}\,,\; \mathtt{v}\cdot\ell+j-k=0\,,\; \s=\pm\}\\
R^{(+)}_{\ell j k}(\eta, \tau_0):=R^{(+)}_{\ell j k}(i_{n},\eta, \tau_0)
&:=\{ \omega\in\mathcal{G}_n : \lvert \omega\cdot \ell
+\s( d_{\s j}^{\infty}+  d_{-\s k}^{\infty}) \rvert\le 
{2\eta}\langle \ell \rangle^{-\tau_0}\,, \; \mathtt{v}\cdot\ell+j-k=0\,,\, \s=\pm\}\,,\\
Q_{\ell j }(\eta, \tau_0):=Q_{\ell j }(i_n,\eta, \tau_0)
&:=\{ \omega\in\mathcal{G}_n : \lvert \omega\cdot \ell+\mathfrak{m}_1(\omega)\, j \rvert\le 
{2\eta}\langle \ell \rangle^{-\tau_0}\,, \; \mathtt{v}\cdot\ell+j=0\}\,,\\
P_{\ell j }(\eta, \tau_0):=P_{\ell j }(i_n,\eta, \tau_0)
&:=\{ \omega\in \mathcal{G}_n : \lvert \omega\cdot \ell+\s d_{\s j}^{\infty} \rvert\le 
{2\eta}\langle \ell \rangle^{-\tau_0}\,,\; \mathtt{v}\cdot\ell+j=0\,,\s=\pm\}\,.
\end{align*}
Recalling \eqref{GnDP} we can write, setting 
$\eta\rightsquigarrow \gamma_n$ for the sets  
$Q_{\ell j }(\eta,\tau_0)$, $P_{\ell j }(\eta,\tau_0)$ and  $R^{(+)}_{\ell jk}(\eta,\tau_0)$,
 $\eta\rightsquigarrow \gamma^{*}_n$ for the set $R^{(-)}_{\ell jk}(\eta,\tau_0)$, 
 and $\tau_0\rightsquigarrow \tau$, 
\begin{equation}\label{UnionDP}
\mathcal{G}_n \setminus \mathcal{G}_{n+1}=
\bigcup_{\ell\in\mathbb{Z}^{\nu}, j, k\in S^c} 
\Big( R^{+}_{\ell j k}(i_n,\gamma_{n},\tau) \cup 
R^{-}_{\ell j k}(i_n,\gamma^{*}_{n},\tau)\cup 
Q_{\ell j }(i_n,\gamma_{n},\tau) \cup P_{\ell j}(i_n,\gamma_{n},\tau)\Big)\,.
\end{equation}
We notice that
\begin{itemize}
\item by \eqref{0di0Melnikov} and $\gamma>\gamma^{3}$ 
(see \eqref{SmallnessConditionNMdp}), 
$R^{(-)}_{\ell j k}(i_n)=\emptyset$ for $j=k$, $\ell\neq 0$.
\item If $\ell=0$, by the momentum conservation, we have that $ j= k$. 
Hence $\mathcal{R}^{(+)}_{0 j j}=\emptyset$.
\end{itemize}

We start with a preliminary lemma, 
which gives a first relation between $\ell,j,k$ which must be 
satisfied in order to have non empty resonant sets.
\begin{lemma}\label{BrexitDP}
Let $n \geq 0$.
There is a constant $C>0$  dependent 
of the tangential set and 
independent of $\ell, j, k, n, i_n, \omega$
such that the following holds:
\begin{itemize}
\item if $R^{(+)}_{\ell j k}(i_n,\eta,\tau_0)\neq \emptyset$ 
then 
 $\lvert \ell \rvert\geq C ( \sqrt{\vert j \rvert}+\sqrt{\lvert k \rvert} )$;
\item if $Q_{\ell j}(i_n,\eta,\tau_0)\neq \emptyset$  
then 
$\lvert \ell \rvert\geq C \lvert j \rvert$;
\item if $P_{\ell j}(i_n,\eta,\tau_0)\neq \emptyset$  
then  $\lvert \ell \rvert\geq C \lvert j \rvert$.
\end{itemize}
\end{lemma}

\begin{proof}
If $R^{(+)}_{\ell j k}(i_n,\eta,\tau_0)\neq \emptyset$, 
then there exists $\omega$ such that 
\begin{equation*}
\lvert \s (d_{\s j}^{\infty}(\omega, i_n(\omega))+d_{-\s k}^{\infty}(\omega, i_n(\omega)))-\mathfrak{m}_1\,\mathtt{v}\cdot \ell \rvert
< {2 \eta}{\langle \ell \rangle^{-\tau_0}}
+ \lvert ({\omega}-\mathfrak{m}_1\,\mathtt{v})\cdot \ell \rvert\,.
\end{equation*}
Moreover, using \eqref{FinalEigenvaluesDP}, 
\eqref{stimeautovalfinaliDP}, we get
$\lvert \s(d_{\s j}^{\infty}(\omega, i_n(\omega))+d_{-\s k}^{\infty}(\omega, i_n(\omega)))-\mathfrak{m}_1\,\mathtt{v}\cdot \ell \rvert 
\geq \frac{1}{3} \big( \sqrt{\lvert j \rvert}+\sqrt{\lvert k \rvert} \big).
$
Thus, for $\varepsilon$ small enough 
\[
\lvert \ell \rvert
\gtrsim_S  \lvert ({\omega}-\mathfrak{m}_1\,\mathtt{v})\cdot \ell \rvert
\gtrsim_S \left( \frac{1}{3}-\frac{2 \eta}{\langle \ell \rangle^{\tau_0}
\big( \sqrt{\lvert j \rvert}+\sqrt{\lvert k \rvert} \big)} \right)\big( \sqrt{\lvert j \rvert}+\sqrt{\lvert k \rvert} \big)\gtrsim_S \frac{1}{4}\big( \sqrt{\lvert j \rvert}+\sqrt{\lvert k \rvert} \big)
\]
and this proves the first claim on $R^{(+)}_{\ell j k}(i_{n}, \eta,\tau_0)$.
The others are similar.
\end{proof}

\begin{remark}\label{infinitiJK}
The above lemma 
implies that for any $\ell\in \mathbb{Z}^{\nu}$ there are only finitely many
indexes $j,k\in S^{c}$ such that the sets 
 ${R}^{(+)}_{\ell,j,k}(\eta,\tau_0), Q_{\ell,j}(\eta,\tau_0), P_{\ell,j}(\eta,\tau_0)$ 
 are not empty. 
 Unfortunately
 this is not true for the set $R^{(-)}_{\ell,j,k}(\eta,\tau_0)$.
% , we still have that for 
% %For $\mathcal{R}^{(-)}$ we have that 
% fixed $\ell$ there are infinitely many $j, k$ such that $R^{-}_{\ell,j,k}$ is not empty. 
\end{remark}

\subsubsection{Measure of a resonant set}\label{singleSET}
The aim of this subsection is to prove the following lemma.
\begin{lemma}\label{singolo}
We have that 
\begin{equation}\label{stimaBadERRE}
\lvert R^{(-)}_{\ell j k}(\eta, \sigma) \rvert\le 
K \varepsilon^{2(\nu-1)} \eta \langle \ell \rangle^{-\tau_0}\,,
\end{equation}
for some $K=K(S)$.
The same holds for $R^{(+)}_{\ell j k}(\eta, \tau_0) $,
$Q_{\ell j}(\eta,\tau_0)$ and $P_{\ell j}(\eta,\tau_0)$.
\end{lemma}

\begin{lemma}\label{stimaBAsso}
We have
\begin{equation}\label{stimaBassoderome}
\lvert\nabla_{\omega} \big(\omega-\mathfrak{m}_1(\omega)\,\mathtt{v} \big) \cdot \ell\rvert\gtrsim_S\,\lvert \ell \rvert
\end{equation}
\end{lemma}
\begin{proof}
By \eqref{mer},  estimate \eqref{merstima},  \eqref{xiomega} and \eqref{matriceV} we have that
\begin{equation}\label{stimaBassoderome2}
\partial_{\omega_i}m_1(\omega) \mathtt{v}\cdot \ell= \e^{-2}\sum_{j} (\mathbb{V}\mathbb{A}^{-1})_i^j\,\ell_j+O(\e^{2})\quad \Rightarrow\quad
\pa_{\omega_i}\big[ \big(\omega-\mathfrak{m}_1(\omega)\,\mathtt{v} \big) \cdot \ell \big]=
\big[({\rm Id}-\mathbb{V}\mathbb{A}^{-1})\ell\big]_{i}+O(\e^{2}).
\end{equation}
By Lemma \ref{measG0}
we have that ${\rm det}(\mathbb{A}-\mathbb{V})\neq0$, hence the 
\eqref{stimaBassoderome2} implies the \eqref{stimaBassoderome}.
\end{proof}

\begin{proof}[{\bf Proof of Lemma \ref{singolo}.}]
We prove the lemma for the set $R^{(-)}_{\ell j k}(\eta, \tau_0)$ 
that is the most difficult case.
The estimates for the other sets follow in the same way.
We have that $\lvert \mbox{sign}(j)-\mbox{sign}(k) \rvert\le 2$ and
\begin{equation}\label{alba2}
\lvert \sqrt{\lvert j \rvert}-\sqrt{\lvert k \rvert} \rvert \le 
\frac{|\lvert j \rvert-\lvert k \rvert|}{\sqrt{\lvert j \rvert}
+\sqrt{\lvert k \rvert} }\le \frac{\lvert \mathtt{v}\rvert\lvert \ell \rvert}{\sqrt{\lvert j \rvert}
+\sqrt{\lvert k \rvert} }.
\end{equation}
Then by Lemma \ref{stimaBAsso}, \eqref{merstima}, 
\eqref{accendo2}, \eqref{stimeautovalfinaliDP}
we have that
\begin{align*}
\lvert\nabla_{\omega} \big(\omega\cdot \ell
+ \s(d_{\s j}^{\infty}-  d_{\s k}^{\infty}) \big) \rvert &\geq {C}(S)\lvert \ell \rvert-
\varepsilon^2 \lvert \nabla_{\omega} \mathtt{m}_{1/2}^{(\geq 4)}(\omega ) \rvert \lvert \mathtt{v} \rvert \lvert \ell \rvert-2 \varepsilon^2 \lvert \nabla_{\omega}\mathtt{m}_0(\omega)\rvert\\
&-\lvert \nabla_{\omega} r_{j} \rvert-\lvert \nabla_{\omega} r_{j}^{\infty} \rvert -\lvert \nabla_{\omega} r_{k} \rvert-\lvert \nabla_{\omega} r_{k}^{\infty} \rvert
%\\&
\gtrsim_S (1-\varepsilon^2 )|\ell|\,.
\end{align*}
Then the \eqref{stimaBadERRE} follows
by Fubini theorem. 
\end{proof}

\subsubsection{Summability}\label{siSomma}

\begin{lemma}\label{InclusionideiBadSetsDP}
For $n\geq 1, \lvert \ell \rvert\le N_{n-1}$, one has  
$R^{(-)}_{\ell j k}(i_n,\gamma_{n}^{*},\tau), R^{(+)}_{\ell j k}(i_n,\gamma_n,\tau),
Q_{\ell j}(i_n,\gamma_n,\tau), P_{\ell j}(i_n,\gamma_n,\tau)=\emptyset$.
\end{lemma}
\begin{proof}
Consider the set $R^{(-)}_{\ell j k}(i_n,\gamma_n^{*},\tau)$.
By the momentum condition $\mathtt{v}\cdot\ell+j-k=0$
we have to consider the sets $R^{(-)}_{\ell j k}(i_n,\gamma_{n}^{*},\tau)$
only for
\[
|j-k|\leq |\mathtt{v}| |\ell|\lesssim_S N_{n-1}\,.
\]

By \eqref{merstima}, \eqref{r12} (with $i_\d^{(1)}\rightsquigarrow i_n$ and $i_\d^{(2)}\rightsquigarrow i_{n-1}$, $N\rightsquigarrow N_{n-1}$) 
and \eqref{FrakHatDP}  we have for all $j, k\in S^c$ 
\begin{equation}\label{marathon1}
\lvert \s(d_{\s j}^{\infty}-d_{\s k}^{\infty})(i_n)-\s(d_{\s j}^{\infty}-d_{\s k}^{\infty})(i_{n-1})\rvert\le \varepsilon\gamma^{-3}
N_{n-1}^{-\mathtt{a}}\qquad\qquad \forall\omega\in \mathcal{G}_n,
\end{equation}
where $\mathtt{a}:=\min\{\kappa_*,\alpha\}$ (recall $\alpha$ in \eqref{parametriNMdp} and $\kappa_*$ in \eqref{r12}).
Now for all $j\neq k$, $\lvert \ell \rvert\le N_{n-1}$, $\omega\in\mathcal{G}_n$ by \eqref{marathon1} 
\begin{equation}
\begin{aligned}
\lvert \omega\cdot \ell +\s (d_{\s j}^{\infty}(i_n)-d_{\s k}^{\infty}(i_{n}))\rvert&\geq \lvert \omega\cdot \ell +\s (d_{\s j}^{\infty}(i_{n-1})-d_{\s k}^{\infty}(i_{n-1}))\rvert
-\lvert \s(d_{\s j}^{\infty}-d_{\s k}^{\infty})(i_n)-\s(d_{\s j}^{\infty}-d_{\s k}^{\infty})(i_{n-1})\rvert\\
&\geq 2\gamma^{*}_{n-1}\langle \ell \rangle^{-\tau}-\varepsilon\gamma^{-3 }N_n^{-\mathtt{a}}\geq 2\gamma^{*}_{n}\langle \ell \rangle^{-\tau}
\end{aligned}
\end{equation}
since $\varepsilon\gamma^{-6}N_n^{\tau-(2/3)\mathtt{a}}2^{n+1}\le 1$. 
Since  $R_{\ell j k}^{(-)}(i_n,\gamma_{n}^{*},\tau)\subseteq \mathcal G_n$ then 
$R^{(-)}_{\ell j k}(i_n,\gamma_n^{*},\tau)=\emptyset$ .
The others are similar.
\end{proof}

\noindent
We have proved that, recall \eqref{UnionDP},
\begin{equation}\label{differenzeInsiemiMisuraDP}
\mathcal{G}_n \setminus \mathcal{G}_{n+1}\subseteq 
\bigcup_{\substack{j, k\in S^c\\ \lvert \ell \rvert> N_{n-1}}}
\Big( R_{\ell j k}^{(+)}(i_n,\gamma_{n},\tau)\cup 
R^{(-)}_{\ell j k}(i_n,\gamma^{*}_{n},\tau)\cup 
Q_{\ell j}(i_n,\gamma_{n},\tau)\cup 
P_{\ell j}(i_n,\gamma_{n},\tau)\Big)\,,  \quad \forall n\geq 1\,.
\end{equation}
By Lemmata \ref{BrexitDP}, \ref{InclusionideiBadSetsDP}, \ref{singolo}
we deduce  that
\begin{equation}\label{quasifine}
\begin{aligned}
\left\lvert \bigcup_{\ell\in\mathbb{Z}^{\nu}, j, k\in S^c} 
R^{(+)}_{\ell j k}(\gamma_{n},\tau)
\cup Q_{\ell j}(\gamma_{n},\tau)\cup 
P_{\ell j}(\gamma_{n},\tau)  \right\rvert
&\le \sum_{\substack{\lvert\ell\rvert >N_{n-1},\\ 
\lvert j \rvert, \lvert k \rvert\leq C|\ell|^{2}}}
\Big(|R^{(+)}_{\ell j k}(\gamma_{n},\tau)|
+|Q_{\ell j}(\gamma_{n},\tau)|
+|P_{\ell j }(\gamma_{n},\tau)|\Big)\\
&\lesssim_S \e^{2(\nu-1)}\sum_{|\ell|>N_{n-1}}
\frac{\gamma |\ell|^{2}}{\langle \ell\rangle^{\tau}}
\lesssim_S\e^{2(\nu-1)}\gamma N_{n-1}^{-1}\,.
\end{aligned}
\end{equation}
Consider now indexes  $\ell,j,k$ such that $\mathtt{v}\cdot\ell+j-k=0$ 
and $\sign(j)=-\sign(k)$.
Therefore we deduce that $|j-k|=|j|+|k|\leq |\mathtt{v}||\ell|$. Then
by Lemmata  \ref{InclusionideiBadSetsDP}, \ref{singolo} we get
\begin{equation}\label{bomba}
\begin{aligned}
\left\vert\bigcup_{\substack{\ell\in\mathbb{Z}^{\nu}, j, k\in S^c \\
\sign(j)=-\sign(k)}} R^{(-)}_{\ell j k}(\gamma_{n}^{*},\tau)\right\vert
&\leq
 \sum_{\substack{\lvert\ell\rvert >N_{n-1},\\
|j|+|k|\leq |\mathtt{v}||\ell|}}|R^{(-)}_{\ell j k}(\gamma_{n}^{*},\tau)|
 &\lesssim_S \e^{2(\nu-1)}\sum_{|\ell|>N_{n-1}}\frac{\gamma |\ell|}{\langle \ell\rangle^{\tau}}
\lesssim_S\e^{2(\nu-1)}\gamma^{3} N_{n-1}^{-1}\,.
\end{aligned}
\end{equation}
It remains to study the measure of the union of the sets
$R_{\ell j k}^{(-)}(i_n)$ over 
for indexes such that $\sign(j)=\sign(k)$, $|\ell|>N_{n-1}$.
This is the most difficult case, since by
Remark
\ref{infinitiJK}, for any fixed $\ell\in \mathbb{Z}^{\nu}$ there are infinite many 
indexes $j,k$ such that $R_{\ell j k}^{-}(i_n)\neq\emptyset$. So we cannot reason as in 
\eqref{quasifine}, \eqref{bomba} and we need a more refined argument.

\begin{lemma}\label{delpiero}
There exists $\mathtt{C}>0$ such that
if  $\sign(j)=\sign(k)$ and 
$\lvert j \rvert, \lvert k \rvert\geq \mathtt{C} \langle \ell \rangle^{2\nu+6}\gamma^{-2}$ 
then 
\begin{equation}\label{alba11}
R_{\ell j k}^{(-)}(i_n,\gamma_{n}^{*}, \tau)\subseteq Q_{\ell, j-k}(\gamma, \nu+2).
\end{equation}
\end{lemma}
\begin{proof}
Recalling \eqref{FinalEigenvaluesDP}, \eqref{finaleigen}, 
\eqref{mer}, using that 
$\sign(j)=\sign(k)$ 
we write $\psi_{\ell ,j,k}=\omega\cdot\ell+\s (d_{\s j}^{\infty}-d_{\s k}^{\infty})$ with
\[
\psi_{\ell,j,k}=\omega\cdot \ell+\mathfrak{m}_1(j-k)+\s\big(
(1+\mathtt{m}^{(\geq4)}_{1/2})(\sqrt{|j|}-\sqrt{|k|})
+r_{\s j}-r_{\s k}+r^{\infty}_{\s j}-r^{\infty}_{\s k}\big)\,.
\]
We show that if $\omega\in (Q_{\ell, j-k}(\gamma, \nu+2))^{c}$
then $\omega\in (R^{(-)}_{\ell, j,k}(\gamma_{n}^{*}, \tau))^{c}$.
If $\omega\in (Q_{\ell, j-k}(\gamma, \nu+2))^{c}$ one has
\begin{equation*}%\label{alba5}
|\omega\cdot \ell+\mathfrak{m}_1(j-k)|\geq \gamma\langle \ell\rangle^{-(\nu+2)}\,.
\end{equation*}
Moreover since $\lvert j \rvert, \lvert k \rvert\geq \mathtt{C} \langle \ell \rangle^{2(\nu+3)}\gamma^{-2}$
then we have
\[
|\psi_{\ell ,j,k}|\gtrsim \frac{\gamma}{\langle \ell \rangle^{\nu+2}}
-\frac{\langle \ell \rangle}{\sqrt{| k |}+\sqrt{| j |}}
-\frac{\e^3\gamma}{\sqrt{| j |}}+\frac{\e^{3}}{\sqrt{| k |}}\gtrsim
 \frac{\gamma}{\langle \ell \rangle^{\nu+2}}(1-\frac{1+\e}{\mathtt{C}})
\gtrsim  \frac{\gamma}{2\langle \ell \rangle^{\nu+2}}\,,
\]
for $\mathtt{C}>0$ large enough. This also implies 
$|\psi_{\ell,j,k}|\gtrsim\gamma^{3}\langle \ell\rangle^{-\tau}$, since 
$\tau>\nu+2$ and $\gamma\geq \gamma^{3}$.
This means $\omega\in (R^{-}_{\ell,j,k}(i_n))^{c}$ and hence we get the thesis.
\end{proof}
As a consequence of Lemma \ref{delpiero}
we have
\begin{equation*}
\left\lvert \bigcup_{\substack{\ell\in\mathbb{Z}^{\nu}, j, k\in S^c\\
\sign(j)=\sign(k)}} R^{(-)}_{\ell j k}(i_n,\gamma_{n}^{*},\tau)  \right\rvert\le 
\sum_{\substack{\lvert\ell\rvert >N_{n-1}, \\ \lvert j \rvert, \lvert k \rvert
\geq \mathtt{C}\langle \ell \rangle^{2(\nu+3)}\gamma^{-2} \\
\sign(j)=\sign(k)}}|R^{(-)}_{\ell j k}(i_n,\gamma_{n}^{*},\tau)|
+\sum_{\substack{\lvert\ell\rvert >N_{n-1}, \\ \lvert j \rvert, \lvert k \rvert\leq2 \mathtt{C} 
\langle \ell \rangle^{2(\nu+3)}\gamma^{-2}\\
\sign(j)=\sign(k)}} 
|R^{(-)}_{\ell j k}(i_n,\gamma_{n}^{*},\tau)|\,.
\end{equation*}
On one hand, using Lemmata \ref{singolo} and  \ref{delpiero}, we have that 
\begin{equation*}
\begin{aligned}
\sum_{\substack{\lvert\ell\rvert >N_{n-1},\\ \lvert j \rvert, \lvert k \rvert\geq \mathtt{C} 
\langle \ell \rangle^{2(\nu+3)}\gamma^{-2} \\
\sign(j)=\sign(k)}}
|R^{(-)}_{\ell j k}(i_n,\gamma_{n}^{*},\tau)|
%\\&\qquad \qquad 
&\stackrel{\eqref{alba11}}{\lesssim} 
\sum_{\substack{|\ell|>N_{n-1}, j-k=h, \\
\lvert h \rvert
\leq C \lvert \ell \rvert}}|Q_{\ell h}(i_n,\gamma,\nu+2)|
%\\&\qquad \qquad \qquad
\lesssim_S \sum_{\substack{|\ell|>N_{n-1}, j-k=h, \\
\lvert h \rvert
\leq C \lvert \ell \rvert}} \varepsilon^{2(\nu-1)}
\gamma \langle \ell \rangle^{-\nu-2}
\\&\qquad \qquad 
\lesssim_S  \varepsilon^{2(\nu-1)}\gamma\sum_{\lvert \ell \rvert\geq N_{n-1}}  \langle \ell \rangle^{-(\nu+1)}
%\\&\qquad \qquad \qquad 
\lesssim_S K \e^{2(\nu-1)}\g N_{n-1}^{-1}\,.
\end{aligned}
\end{equation*}
On the other hand
\begin{equation*}
\begin{aligned}
\sum_{\substack{\lvert\ell\rvert >N_{n-1},  \lvert j-k \rvert 
\le C\lvert \ell \rvert \\ \lvert j \rvert, \lvert k \rvert\leq 2\mathtt{C} 
\langle \ell \rangle^{2(\nu+3)}\gamma^{-2}
\\ \sign(j)=\sign(k)}} |R^{(-)}_{\ell j k} (i_n,\gamma_{n}^{*},\tau)| 
&
\stackrel{\eqref{stimaBadERRE}}{\lesssim}_S  \varepsilon^{2(\nu-1)}  
\sum_{\lvert \ell \rvert\geq N_{n-1}} 
\frac{\gamma^{3}\langle \ell \rangle^{2\nu+6}}{{\gamma}^{2}\langle \ell \rangle^{\tau}} 
\\
&\lesssim_S  \gamma \varepsilon^{2(\nu-1)} 
\sum_{\lvert \ell \rvert\geq N_{n-1}}  \langle \ell \rangle^{-(\tau-2\nu-6)}\lesssim_S  \gamma \varepsilon^{2(\nu-1)} N_{n-1}^{-1}.
\end{aligned}
\end{equation*}
By the discussion above and \eqref{quasifine}, \eqref{bomba}
we obtain the estimates \eqref{MisureDP}.
%\end{proof}

%\textbf{Conclusion of the Proof of Theorem \ref{IlTeoremaDP}}. 
\subsection{Conclusion of the Proof of Theorem \ref{IlTeoremaDP}}
Theorem \ref{NashMoserDP} implies that the sequence $(\mathfrak{I}_n, \Xi_n)$ 
is well defined for $\omega\in \mathcal{G}_{\infty}:=\cap_{n\geq 0} \mathcal{G}_n$, 
$\mathfrak{I}_n$ is a Cauchy sequence in 
$\lVert \cdot \rVert_{s_0+\mu_1}^{\g,\mathcal{G}_{\infty}}$ 
(see \eqref{FrakHatDP}) and $\lvert \Xi_n \rvert^{\gamma}\to 0$. 
Therefore $\mathfrak{I}_n$ converges to a limit 
$\mathfrak{I}_{\infty}$ in norm 
$\lVert \cdot \rVert_{s_0+\mu_1}^{\gamma,\mathcal{G}_{\infty}}$ 
and, by $(\mathcal{P} 2)_n$, for all 
$\omega\in\mathcal{G}_{\infty}, i_{\infty}(\varphi):=(\varphi, 0, 0)+\mathfrak{I}_{\infty}(\varphi)$ 
is a solution of
\begin{equation*}%\label{eccolasol}
\mathcal{F}(i_{\infty}, 0)=0 \qquad \mbox{with} \qquad 
\lVert \mathfrak{I}_{\infty} \rVert_{s_0+\mu_1}^{\gamma,\mathcal{G}_{\infty}}
\lesssim\,\varepsilon^{16- 2 b} \gamma^{-1}
\end{equation*}
by \eqref{ConvergenzaDP}, \eqref{SmallnessConditionNMdp}. 
Therefore $\varphi \mapsto i_{\infty}(\varphi)$ is an invariant 
torus for the Hamiltonian vector field $X_{H_{\varepsilon}}$ (recall \eqref{HepsilonDP}). 
We point out that $\mathfrak{I}_{\infty}$ is traveling, i.e. satisfies \eqref{trawaves}, since it is limit (in the $H^{s_0+\mu_1}$ topology) of solutions of the linear constant coefficients PDE $\mathtt{v}\cdot \partial_{\varphi}+(0, 0, \partial_x)=0$.
By \eqref{MisureDP},
\[
\lvert \Omega_{\varepsilon}\setminus \mathcal{G}_{\infty} \rvert 
\le \lvert \Omega_{\varepsilon} \setminus \mathcal{G}_0 \rvert
+\sum_{n\geq 0} \lvert \mathcal{G}_n \setminus \mathcal{G}_{n+1} \rvert
\le 2\,C_* \varepsilon^{2 (\nu-1)} \gamma
+C_* \varepsilon^{2(\nu-1)}\gamma 
\sum_{n\geq 1} N_{n-1}^{-1}
\lesssim C_{*}\varepsilon^{2(\nu-1)} \gamma\,.
\]
The set $\Omega_{\varepsilon}$ in \eqref{OmegaEpsilonDP} 
has measure $\lvert \Omega_{\varepsilon} \rvert=O(\varepsilon^{2 \nu})$. 
Hence 
$\lvert \Omega_{\varepsilon}\setminus 
\mathcal{G}_{\infty} \rvert/\lvert \Omega_{\varepsilon} \rvert\to 0$ 
as $\varepsilon\to 0$ because $\gamma=o(\varepsilon^2)$, 
and therefore the measure of 
$\mathcal{C}_{\varepsilon}:=\mathcal{G}_{\infty}$ 
satisfies \eqref{frazionemisureDP}.

%\comment{la prova della stabilita' io la vorrei scritta alla Rob, come in riducibilita}

It remains to show the linear stability of the embedding $i_{\infty}(\f)$. Recall that in the original coordinates the solution reads as in \eqref{trawaves2}.
By the discussion of section \ref{sezione6DP} 
(see also \cite{KdVAut} for further details) 
and section \ref{sez:KAMredu},
%\ref{regularization},
since $i_{\infty}(\f)$ is isotropic
and solves the equation $\mathcal{F}(i_{\infty}, 0)=0$,
%\eqref{eccolasol}, 
it is possible to find a change of coordinates
$G_{\infty}$ (of the form \eqref{GdeltaDP}),
so that in the linearized system at $i_{\infty}$ of the Hamiltonian 
$H_{\e}\circ G_{\infty}(\psi, Y, U)$ 
the equation for the 
actions is given by $\dot{Y}=0$. 
Moreover, by section \ref{sez:KAMredu}
%\ref{regularization} 
the linear equation for the normal variables $U$ is conjugated, 
by setting (see \eqref{chichichi}, \eqref{CVWW}) $Z=(\chi\circ \Lambda)(U)$,
%$w=\Upsilon \circ \Phi_{\infty}(z)$ 
to the 
diagonal system $\dot{Z}_j+\ii d_j^{\infty}(\omega)Z_j=f_{j}(\omega t), j\in S^{c}$, 
where $f(\omega t)$ is a forcing term.
Since $d_j^{\infty}\in\mathbb{R}$ a standard argument 
shows that the Sobolev norms of $Z$ do not increase in time. 

%%%%%%%%%%%%%%%%%%%%%%%%%%%%%%%%%%%%%%%%%%%%%%

\appendix

\section{Flows and conjugations}

\subsection{Flows of pseudo differential PDEs}
Here we state some preliminary results
which are the counterpart of sections $3.1$, $3.2$
in \cite{FGP}.
In this paper we shall consider the flow of the equation
 \begin{equation}\label{linprobpro}
 \left\{\begin{aligned}
 &\pa_{\theta} \Psi^\theta u = \ii\Pi_{S}^{\perp}\big[\opw(f(\tau,\vphi,x,  \x)) 
 [\Pi_{S}^{\perp}\Psi^\theta u]\big]\,,\\
 & \Psi^0u = u \, , 
 \end{aligned}\right.
\end{equation}
where $ f $ is a symbol assumed to have one of the following forms: 
\begin{align}
& f(\tau,\vphi,x, \x):=
b(\tau,\vphi,x)\x\,,\qquad
b(\tau,\vphi,x):=\frac{\beta(\vphi,x)}{1+\tau \beta_{x}(\vphi,x)}\,, \quad
\beta(\vphi,x) \in H^{s}(\mathbb{T}^{\nu+1};\mathbb{R})\label{sim1}\\
& f(\tau,\vphi,x, \x):=
\beta(\vphi,x)|\x|^{\frac{1}{2}}, \qquad \quad \  \beta(\vphi,x) \in H^{s}(\mathbb{T}^{\nu+1};\mathbb{R}) \, , \label{sim2}\\
&  f(\tau,\vphi,x, \x) \in S^{m} \, \, , \quad 
 \qquad m \leq 0 \,,\;\; f \;\; {\rm real \; valued} \label{sim3}.
\end{align}
Notice that the map $\Psi^{\tau}$ in  \eqref{linprobpro} (if well posed)
is symplectic. Indeed it is the flow associated to  the Hamiltonian 
\begin{equation*}%\label{pseudo}
S(\tau,\vphi,z)=\frac{1}{2} \int_{\mathbb{T}}\opw(f(\tau;\vphi,x,\x)) 
z\cdot\bar{z} \,dx\,,\qquad
z\in H^{\perp}_{S}\,.
\end{equation*}	

In the following we study the properties of the flow in \eqref{linprobpro}. In particular 
we shall compare the map $\Psi^{\tau}$, $\tau\in[0,1]$ with 
the flow $\Phi^{\tau}$ of the equation
 \begin{equation}\label{linprob}
 \pa_{\theta} \Phi^\theta = \ii\opw(f(\vphi,x,  \x)) \Phi^\theta, 
 \quad  \Phi^0 = {\rm Id} \, ,
\end{equation}
with $f(\tau,\vphi,x,\x)$ as in \eqref{sim1}-\eqref{sim3}.
We start by considering the case $f$ as in \eqref{sim1}.

\noindent
For $\tau\in [0,1]$ we consider
the map
\begin{equation}\label{ignobelSymp}
\begin{aligned}
&\mathcal{A}^{\tau}h(\varphi, x):=\sqrt{1+\tau\beta_{x}(\vphi,x)}
h(\varphi, x+\tau\beta(\varphi, x)), 
\qquad \quad \,\, \varphi\in \T^{\nu},\, x\in \T,\\
&(\mathcal{A}^{\tau})^{-1}h(\varphi, y):=
\sqrt{1+\tilde{\beta}_{y}(\tau,\vphi,y)}
h(\varphi, y+\tilde{\beta}(\tau, \varphi, y)), \quad \varphi\in \T^{\nu},\, y\in \T,
\end{aligned}
\end{equation}
where
 $\beta$ is some smooth function and 
 $\tilde{\beta}(\tau;x,\x)$ is such that
\begin{equation}\label{diffeoInv}
x\mapsto y=x+\tau\beta(\vphi,x) \;\;\; \Leftrightarrow \;\;\; y\mapsto x=y+\tilde{\beta}(\tau,\vphi,x), \;\; \tau\in[0,1]\,,
\end{equation}
It is easy to check that the map in \eqref{ignobelSymp}
is the flow of the equation
 \begin{equation}\label{diffeotot}
 \left\{\begin{aligned}
&\pa_{\tau}\mathcal{A}^{\tau}(u)=\opw(\ii b(\tau;\vphi,x)\x)
\mathcal{A}^{\tau}(u)\\
& \mathcal{A}^{0}u=u\,,
\end{aligned}\right.
\qquad 
b(\tau,\vphi,x):=\frac{\beta(\vphi,x)}{1+\tau \beta_{x}(\vphi,x)}\,.
\end{equation}

\begin{remark}\label{weyl-nonweyl}
By an explicit computation it is easy to  check that
\begin{equation*}%\label{fear}
\pa_{\tau}\mathcal{A}^{\tau}u=b\pa_{x}\mathcal{A}^{\tau}u+
\frac{b_x}{2}\mathcal{A}^{\tau}u
=:
\op(a(\tau,\vphi,x,\x))\mathcal{A}^{\tau}u\,,
\quad a(\tau,\vphi,x,\x):=\ii b(\tau,\vphi,x)\x+\frac{1}{2}b_x(\tau,\vphi,x)\,,
\end{equation*}
where 
$\op(a)$ is defined in \eqref{standardquanti}.
By using \eqref{bambola} one deduces the \eqref{diffeotot}.
\end{remark}
We first need to show that $\mathcal{A}^{\tau}$ is well defined 
as map on $H^{s}$.

\begin{lemma}\label{CoroDPdiffeo}
Fix $n \in \mathbb{N}$.
There exists $\s=\s(\rho,\nu)$ such that, if $\| \beta\|^{\gamma, \calO}_{s_0+\s}< 1$, then
 the flow $\mathcal{A}^{\tau}(\vphi)$ of \eqref{diffeotot} 
satisfies 
for $s\in [s_0, \mathcal{S}]$,
\begin{equation}\label{flow2}
\begin{aligned}
&\sup_{\tau\in[0,1]} \lVert \mathcal{A}^{\tau} u \rVert_s^{\gamma, \calO}
+\sup_{\tau\in[0,1]} \lVert ({\mathcal{A}}^{\tau})^{*} u \rVert_s^{\gamma, \calO}
\lesssim_s \left(
\|u\|_{s}^{\gamma, \calO}+\| b\|_{s+\s}^{\gamma, \calO}\|u\|^{\gamma, \calO}_{s_0}
\right),
\end{aligned}
\end{equation}
\begin{equation*}
\sup_{\tau\in[0,1]} \lVert (\mathcal{A}^{\tau}-\mathrm{Id})u \rVert_s^{\gamma, \calO}+
\sup_{\tau\in[0,1]} \lVert (({\mathcal{A}}^{\tau})^{*}-\mathrm{Id})u \rVert_s^{\gamma, \calO}
\lesssim_s \left(
\lVert \beta \rVert_{s_0+\s}^{\gamma, \calO} \|u\|^{\gamma, \calO}_{s+1}
+\| \beta\|_{s+\s}^{\gamma, \calO}\|u\|^{\gamma, \calO}_{s_0+1}
\right).
\end{equation*}
For any $|\al |\leq n$, $m_1,m_{2}\in \mathbb{R}$ such that $m_1+m_{2}=|\al|$, for any 
$s\geq s_0$ there exist $\mu_*$, $\s_{*}$, depending on 
$|\al|,m_1,m_2$,
%$\mu_*=\mu_*(|\al|,m_1,m_2)$, $\s_*=\s_*(|\al|,m_1,m_2)$ 
and $\delta=\delta(m_1,s)$
such that if $
\|\beta\|^{\gamma,\calO}_{s_0+\mu_*}\leq\delta,
$ and $\|\Delta_{12}\beta\|^{\gamma,\calO}_{p+\s_*}
\leq1$ for $p+\s_{*}\leq s_0+\mu_{*} $,
then one has
\begin{equation}\label{flow22}
\sup_{\tau\in[0,1]}\|\langle D_{x}\rangle^{-m_1}\pa_{\vphi}^{\al}\mathcal{A}^{\tau}(\vphi)
\langle D_{x}\rangle^{-m_2}u\|^{\gamma,\calO}_{s}
\lesssim_{s,\mathtt{b},m_1,m_2}
\| u \|^{\gamma,\calO}_{s}+\|\beta\|^{\gamma,\calO}_{s+\mu_*}\| u\|^{\gamma,\calO}_{s_0},
\end{equation}
%and for $m_{1}+m_2=|\al|+1$,
\begin{equation}\label{flow222}
\sup_{\tau\in[0,1]}  \| \langle D_{x}\rangle^{-m_1}\pa_{\vphi}^{\al}\Delta_{12}\mathcal{A}^{\tau}(\vphi) 
\langle D_{x}\rangle^{-m_2}u\|_{p} \lesssim_{p,\mathtt{b},m_1,m_2}
\lVert u \rVert_{p}\lVert \Delta_{12} \beta   \rVert_{p+\mu_*}, \quad m_{1}+m_2=|\al|+1.
\end{equation}
\end{lemma}

\begin{proof}
It follows by Lemmata $B.7$, $B.8$, $B.9$ in \cite{FGP}.
\end{proof}
We have the following.
\begin{lemma}\label{CoroDPdiffeo2}
Fix $n \in \mathbb{N}$.
There exists $\s=\s(\rho)$ such that, 
if  \,$|f|^{\gamma, \calO}_{m,s_0+\s,\alpha}\leq 1$, 
with $f$
in \eqref{sim2} or \eqref{sim3}, 
then
 the flow $\Phi^{\tau}(\vphi)$ of \eqref{linprob} 
satisfies 
for $s\in [s_0, \mathcal{S}]$,
bounds like \eqref{flow2}-\eqref{flow222}.
\end{lemma}

\begin{proof}
The result for $f(\vphi,x,\x)$ as in \eqref{sim2} follows by the results in 
the appendix of \cite{BM1}. For the case \eqref{sim3}
one can follow almost word by word the proof of 
Proposition $3.1$
in \cite{FGP} and use Lemmata \ref{weyl/standard}, \ref{lem:escobar}
and Remark \ref{standard/Weyl}
to pass from the standard quantization \eqref{standardquanti}
to the Weyl quantization in \eqref{weylquanti}.
\end{proof}

\begin{lemma}\label{differenzaFlussi}
Fix $\rho\geq 3$ and $p\geq s_0$ then the following holds.

\vspace{0.5em}
\noindent 
$(i)$ There exist $\delta\ll 1$, $\su:=\su(\rho,p)$  such that if
\begin{equation}\label{flow1}
|f|_{m,s_0+\s_1,\alpha}^{\gamma,\mathcal{O}} \le \delta
\end{equation} 
then the following holds. 
Let $\Phi^{\tau}$ be the flow of the system 
\eqref{linprob}, then the flow of 
\begin{equation}\label{rosso}
\partial_{\tau} \Psi^{\tau} u=
\Pi_S^{\perp}[\opw(\ii f(\tau;\vphi,x,\x)) \Pi_S^{\perp}[\Psi^{\tau}u]]\,,\qquad
\Psi^0 u=u\,,
\end{equation}
 is well defined 
for $|\tau|\le 1$ 
and one has
$\Psi^1= \Pi_S^\perp\Phi^{1}\Pi_S^\perp\circ(\rm{Id} +\mathcal{R})$
where $\mathcal{R}$ is an operator 
with the form %\eqref{FiniteDimFormDP}.
\begin{equation}\label{FiniteDimFormDP10001}
\mathcal{R}(\varphi) w=\sum_{\lvert j \rvert\le C} \int_0^1 (w, g_j(\tau, \varphi))_{L^2(\mathbb{T})}\,\chi_j(\tau, \varphi)\,d\tau\,,
\end{equation}
with $\lVert g^{(i)}_j \rVert_s^{\gamma, \calO}
+\lVert \chi_j^{(i)} \rVert_s^{\gamma, \calO}\lesssim_s 1$.
Moreover $\mathcal{R}$ belongs to $\gotL_{\rho,p}(\calO)$ 
and satisfies
\begin{equation}\label{giallo2}
\mathbb{M}^{\gamma}_{\mathcal{R}}(s, \tb)
\lesssim_s  |f |^{\gamma, \calO}_{m,s+\su,\alpha}\,, 
\qquad
\mathbb{M}_{\Delta_{12}\mathcal{R} }(p, \tb)\lesssim _{p}
|\Delta_{12}f|_{m,p+\su,\alpha}\,.
\end{equation}

\noindent
$(ii)$ Assume now that $f\in S_{k}^{m}$ (recall Def. \ref{espOmogenea})
and \eqref{IpotesiPiccolezzaIdeltaDP} holds.
Then, for $\e>0$ small enough, the result of item $(i)$ holds with an  
operator $\mathcal{R}$ in 
$\mathfrak{L}_{\rho,p}^{k}$.
\end{lemma}

\begin{proof}
$(i)$
By using  Lemmata \ref{weyl/standard}, \ref{lem:escobar}
and Remark \ref{standard/Weyl} we can 
 follow the same strategy of the proof of Lemma
$C.1$ in \cite{FGP1}.
%Proposition $3.1$ in \cite{FGP1} provides 
For 
$\delta, \sigma_1$ such that the 
smallness condition \eqref{flow1} 
is satisfied , 
we have that the flow $\Phi^{\tau}$ is well-defined 
for $\lvert \tau \rvert\le 1$. This follows by
Lemmata \ref{CoroDPdiffeo}, \ref{CoroDPdiffeo2} (see also Lemmata 
$B.7$, $B.8$, $B.9$ in \cite{FGP}
for more details).

Let us define $\Upsilon^{\tau}$ as the flow of the following Cauchy problem 
\begin{equation}\label{rosso2}
\partial_{\tau} \Upsilon^{\tau} u=
-\big(\Psi^{\tau}_* \mathtt{Z} \big) \Upsilon^{\tau} u\,,\qquad 
\Upsilon^0 u=u
\end{equation}
with 
\begin{equation}\label{rizla}
\begin{aligned}
\mathtt{Z}u&:=\opw(\ii f(\tau;\vphi,x,\x)) \Pi_S[ u]
+\Pi_S\opw(\ii f(\tau;\vphi,x,\x))\Pi_S^{\perp}[ u]
%\\&\qquad\qquad
=
\sum_{j\in S} \big(g_j(\tau), u \big)_{L^2}\,
\chi_j(\tau)+\sum_{j\in S} \big(\tilde{g}_j(\tau), u \big)_{L^2}\,
\tilde{\chi}_j(\tau)\,,\\
g_j&=\tilde{\chi}_{j}:=e^{\mathrm{i} j x}\,, 
\quad \chi_j=\opw(\ii f(\tau;\vphi,x,\x)) \,e^{\mathrm{i} j x}\,, 
\quad \tilde{g}_j:=\Pi_{S}^{\perp}\opw(-\ii \ov{f(\tau;\vphi,x,-\x)})e^{\ii j x}\,.
%\og(j)\Pi_S^{\perp}[b(\tau)\,e^{\mathrm{i} j x}]\,.
\end{aligned}
\end{equation}
Equation \eqref{rosso2} is well posed on 
$H^{s}$ since its vector field is finite rank. 
By an explicit computation, using \eqref{linprob} and \eqref{rosso2},
we deduce that  $\Psi^{\tau}=\Pi_S^{\perp}\circ\Phi^{\tau}\circ \Upsilon^{\tau}$ 
is well-defined on $H^s$ and solves \eqref{rosso}.
\noindent
In order to
show that $\Upsilon^{\tau}-\mathrm{I}$ is
of the form \eqref{FiniteDimFormDP10001} one can follows
almost word by word
the proof of Lemma $C.1$ in \cite{FGP1} using the estimates 
\eqref{flow2}-\eqref{flow222} on the flow $\Phi^{\tau}$.

\noindent
$(ii)$ By the estimates on the symbol $f\in S_{k}^{m}$ we deduce 
that \eqref{flow1} 
holds. The operator $\mathtt{Z}$ in \eqref{rizla} is in $\mathfrak{L}_{\rho,p}^{k}$
(see also Lemma $C.7$ in \cite{FGP1})
by using its explicit expression. By Taylor expanding the flow $\Psi^{\tau}$
in \eqref{rosso} one can also deduce that the generator of the flow in \eqref{rosso2}
is in $\mathfrak{L}_{\rho,p}^{k}$. Then the result follows.
\end{proof}

As a consequence of Lemma \ref{differenzaFlussi} we have the following
\begin{lemma}\label{georgiaLem}
$(i)$ Consider the operator
\[
A:=\Pi_{S}^{\perp}\big(\omega\cdot\pa_{\vphi}+\opw(d(\vphi,x,\x))+\mathcal{Q}\big)\,,\quad
d\in S^{m'}\,,\quad \mathcal{Q}\in \mathfrak{L}_{\rho,p}\,,
\]
and let $\Psi^{\tau}$ be the flow of \eqref{rosso}. 
There exist $\delta\ll 1$, $\su:=\su(\rho,p)$  such that if
\begin{equation*}%\label{flow1bertier}
|d|_{m',s_0+\s_1,\alpha}^{\gamma,\mathcal{O}}+
\mathbb{M}^{\gamma}_{\mathcal{Q}}(s_0+\s_1, \tb) \le \delta
\end{equation*} 
then the following holds. 
One has
\begin{equation}\label{georgia100}
\Psi^{1}A(\Psi^{1})^{-1}=\Pi_{S}^{\perp}\Phi^{1}A(\Phi^{1})^{-1}\Pi_{S}^{\perp}
+\widetilde{\mathcal{R}}(\vphi)\,,
\end{equation}
where $\widetilde{\mathcal{R}}$ is a finite rank operator of the form \eqref{FiniteDimFormDP10001}
satisfying 
\begin{equation*}%\label{giallo2Rank}
\begin{aligned}
\mathbb{M}^{\gamma}_{\widetilde{\mathcal{R}}}(s, \tb)
&\lesssim_s  |f |^{\gamma, \calO}_{m,s+\su,\alpha}+
|f |^{\gamma, \calO}_{m,s_0+\su,\alpha}
(|d|^{\gamma, \calO}_{m',s+\su,\alpha}+\mathbb{M}^{\gamma}_{\mathcal{Q}}(s, \tb))\,, \\
\mathbb{M}_{\Delta_{12}\widetilde{\mathcal{R}} }(p, \tb)&\lesssim _{p}
|\Delta_{12}f|_{m,p+\su,\alpha}(1+C_{p})\,,
%\\
\qquad
C_p:=|\Delta_{12}f|_{m,p+\su,\alpha}+\mathbb{M}_{\Delta_{12}\mathcal{Q} }(p, \tb)\,.
\end{aligned}
\end{equation*}
$(ii)$ Assume that the generator $f$ of the flow 
$\Psi^{\tau}$ of \eqref{rosso} is in $S_{k}^{m}$, and that $d\in S_{k'}^{m'}$, 
$\mathcal{Q}\in \mathfrak{L}_{\rho,p}^{k'}$. 
Assume also that  \eqref{IpotesiPiccolezzaIdeltaDP} holds.
Then, for $\e>0$ small enough, the result of item $(i)$ holds true
with a remainder $\widetilde{\mathcal{R}}$ in the class 
$\mathfrak{L}_{\rho,p}^{\max\{k,k'\}}$.
\end{lemma}
\begin{proof}
$(i)$
Formula \eqref{georgia100} follows by Lemma \ref{differenzaFlussi}
to write $\Psi^1= \Pi_S^\perp\Phi^{1}\Pi_S^\perp\circ(\rm{I} +\mathcal{R})$.
The estimates on the remainder $\widetilde{\mathcal{R}}$ follow by reasoning as in 
the proof of Lemma $C.2$ in \cite{FGP1}
and by using Lemmata \ref{weyl/standard}, \ref{lem:escobar}
and Remark \ref{standard/Weyl}
to pass from the standard quantization \eqref{standardquanti}
to the Weyl quantization in \eqref{weylquanti}.

\noindent
$(ii)$ This item follows by using item $(ii)$ in Lemma \ref{differenzaFlussi},
reasoning as in the proof of Lemma $C.2$ in \cite{FGP1}
and using Lemma \ref{lem:escobarhomo} (instead of Lemma \ref{lem:escobar})
to pass to the Weyl quantization.
\end{proof}

\begin{remark}\label{differenzaFlussiRMK}
The same result of Lemmata \ref{differenzaFlussi}, \ref{georgiaLem} holds also in the case
$d(\vphi,x,\x)$ is a matrix of symbols in $S^{m'}\otimes\mathcal{M}_2(\mathbb{C})$,
$m'\leq 0$
and $\mathcal{Q}$ is a matrix of operators in 
 $\mathfrak{L}_{\rho,p}\otimes\mathcal{M}_2(\mathbb{C})$.
\end{remark}

\subsection{Egorov Theory and conjugation rules}
In this section we study how pseudo differential operators conjugates 
under the flows in \eqref{diffeotot} and \eqref{linprob}.

\noindent
{\bf Notation.} Consider an integer $n\in \N$.
To simplify the notation from now on we shall write, 
$\Sigma^{*}_{n}$ the sum over indexes $k_1,k_2,k_3\in\N$ such that 
$k_1<n$, $k_1+k_2+k_3=n$ and $k_1+k_2\geq 1$.
\begin{theorem}{\bf (Egorov).}\label{EgorovQuantitativo}
Fix $\rho \geq 3$, $p\geq s_0$, $m\in\mathbb{R}$ 
with $\rho+m> 0$. Let $w(\varphi, x, \xi)\in S^m$ with 
$w=w(\omega, \mathfrak{I}(\omega))$, Lipschitz in 
$\omega\in \calO\Subset\R^\nu$ and in the variable $\mathfrak{I}$. 
Let $\mathcal{A}^{\tau}$ be the flow of the system \eqref{diffeotot}.
$(i)$ There exist $\su:=\su(m, \rho)$ and $\delta:=\delta(m, \rho)$ such that, if
\begin{equation}\label{coin1}
\lVert \beta \rVert^{\gamma, \calO}_{s_0+\su}<\delta,
\end{equation}
then $\mathcal{A}^{\tau} \opw(w) (\mathcal{A}^{\tau})^{-1}=\opw(q(\varphi, x, \xi))+R$ where $q\in S^m$ and 
$R\in\gotL_{\rho,p}(\calO)$.
In particular 
\begin{equation}\label{formaEsplic}
q(\varphi, x, \xi)=q_0(\varphi, x,\x)+q_1(\varphi, x,\x)\,,
\end{equation}
where $q_1\in S^{m-2}$ and  (recall \eqref{diffeoInv})
\begin{equation}\label{formaEsplic2}
q_0(\varphi, x,\x)=w\big(x+\beta(\varphi, x), \x(1+\tilde{\beta}_{y}(1, \varphi, y))_{|y=x+\beta(\varphi, x)}  \big)\,.
\end{equation}
The symbol $q_0$ has the form $q_0(\varphi, x,\x)=p_0(\tau, \varphi, x,\x)_{|\tau=1}$
where $p_0$ solves the equation
\be\label{dthetaa0}
\frac{d}{d \tau} p_0 (\tau) = \{ b(\tau, x) \xi, \varphi, p_0 (\tau)  \} \, , \quad p_0 (0) = w(\varphi, x,\x) \, . 
\ee
 Moreover,  one has that 
 the following estimates hold:
 \begin{equation}\label{zeppelin}
\begin{aligned}
\lvert q \rvert^{\gamma, \calO}_{m, s, \alpha}&\le_{m, s, \alpha, \rho}  
\lvert w \rvert^{\gamma, \calO}_{m, s, \alpha+\su}+
\sum_{s}^{*} 
\lvert w \rvert^{\gamma, \calO}_{m, k_1, \alpha+k_2+\su} 
\lVert \beta \rVert^{\gamma, \calO}_{k_3+\su}\,,\\
\lvert \Delta_{12} q   \rvert_{m, p, \alpha} 
&\le_{m, p, \alpha, \rho} 
\lvert w \rvert_{m, p+1, \alpha+\su}\lVert \Delta_{12} \beta   \rVert_{p+1}
%+\lvert w \rvert_{m, s_0+1, \alpha+\su}\lVert \Delta_{12} \beta   \rVert_{s+1}
+\lvert \Delta_{12} w   \rvert_{m, p, \alpha+\su}
\\&+\sum_{p+1}^{*}  
\lvert w \rvert_{m, k_1, \alpha+k_2+\su}\lVert \beta \rVert_{k_3+\su} 
\lVert \Delta_{12} \beta   \rVert_{s_0+1}
+\sum_{p}^{*}  
\lvert \Delta_{12} w   \rvert_{m, k_1, k_2+\alpha+\su}
\lVert \beta \rVert_{k_3+\su}\,.
%\label{deeppurple}
\end{aligned}
\end{equation}
Furthermore for any %$\tb\le \rho-2$ and 
$s_0\le s\le \mathcal{S}$
\begin{equation}\label{CostanteTameRdei}
\begin{aligned}
\mathbb{M}^{\g}_{R}(s, \tb) &\le_{s, m, \rho}   
\lvert w \rvert^{\gamma, \calO}_{m, s+\rho, \su}+\sum_{s+\rho}^{*} 
\lvert w \rvert^{\gamma, \calO}_{m, k_1, k_2+\su}
\lVert \beta \rVert^{\gamma, \calO}_{k_3+\su}\,,\quad 0\le \tb\le \rho-2\,,\\
\mathbb{M}_{\Delta_{12} R  }(p, \tb) &\le_{m, p, \rho}  
%\Big( 
\lvert w \rvert_{m, p+\rho, \su}\lVert \Delta_{12} \beta   \rVert_{p+\su}
%+\lVert \Delta_{12} \beta   \rVert_{s+\su}\lvert w \rvert_{m, s_0+\rho, \su}
+\lvert \Delta_{12} w   \rvert_{m, s+\rho, \su}\\
&+\sum_{p+\rho}^{*}  
\lvert w \rvert_{m, k_1, k_2+\su} \lVert \beta \rVert_{k_3+\su}\lVert \Delta_{12} \beta   \rVert_{s_0+\su}+\sum_{p+\rho}^{*} 
\lvert \Delta_{12} w   \rvert_{m, k_1, k_2+\su}\lVert \beta \rVert_{k_3+\su}\,,
\quad 0\le \tb\le \rho-3\,.
\end{aligned}
\end{equation}
\noindent
$(ii)$ Assume that $\beta\in S_{k}^{0}$, $w\in S_{k'}^{m}$
and that \eqref{IpotesiPiccolezzaIdeltaDP} holds. 
Then, for $\e>0$ small enough, the result of item $(i)$
holds for $q\in S_{\max\{k,k'\}}^{m}$ and $R\in \mathfrak{L}_{\rho,p}^{\max\{k,k'\}}$.
\end{theorem}

\begin{proof}
$(i)$ First of all we write 
$\mathcal{A}^{\tau}:=\mathcal{A}_2^{\tau}\circ \mathcal{A}^{\tau}_1$
(recall \eqref{ignobelSymp}) where
\begin{equation*}%\label{ignobelSymp100}
\begin{aligned}
&\mathcal{A}_1^{\tau}h(\varphi, x):=%\sqrt{1+\tau\beta_{x}(\vphi,x)}
h(\varphi, x+\tau\beta(\varphi, x)), 
\qquad \quad \,\, \varphi\in \T^{\nu},\, x\in \T,\\
&\mathcal{A}_2^{\tau}h(\varphi, x):=
\sqrt{1+\tau\beta_{x}(\vphi,x)}h(\vphi,x)
\quad \varphi\in \T^{\nu},\, x\in \T\,.
\end{aligned}
\end{equation*}
The operator $\mathcal{A}_1^{\tau} \opw(w) (\mathcal{A}_1^{\tau})^{-1}$
satisfies \eqref{formaEsplic}-\eqref{CostanteTameRdei}
by Theorem $3.4$ in \cite{FGP}.
The same bounds are satisfied by 
 $\mathcal{A}_2^{\tau} \opw(w) (\mathcal{A}_2^{\tau})^{-1}$
 by explicit computations since $\mathcal{A}_2$ is a multiplication operator.
 Equation \eqref{dthetaa0} follows by the first step in the proof
 of Theorem $3.4$ in \cite{FGP}.
 
 \noindent
 $(ii)$ The symbol $q_0$ can be expanded in $\e$ by using the equation 
 \eqref{dthetaa0} and the expansion on $w$ and $\beta$.
  By the proof of Theorem $3.4$ in \cite{FGP}
 one can see that the symbol $q_1$ is constructed iteratively 
 by solving transport equations similar to \eqref{dthetaa0}
 and hence one can conclude $q_1\in S_{\max\{k,k'\}}^{m-2}$.
 Again by Theorem $3.4$ in \cite{FGP} one sees that
 the remainder $R$ has the form
\begin{equation}\label{pirati}
R=R(\tau)
=\int_0^{\tau} \mathcal{A}_1^{\tau} (\mathcal{A}_1^s)^{-1} \mathcal{M} \mathcal{A}_1^{s} (\mathcal{A}_1^{\tau})^{-1}\,ds\,,
\end{equation}
for some $\mathcal{M}\in \mathfrak{L}_{\rho,p}^{\max\{k,k'\}}$.
Notice that the flow $\mathcal{A}_1^{\tau}$ satisfies 
$\pa_{\tau}\mathcal{A}_1^{\tau}=\pa_{x}b(\tau;x)\mathcal{A}_1^{\tau}$.
Then we have $R \in \mathfrak{L}_{\rho,p}^{\min\{k,k'\}}$ 
by Taylor expanding the
flow $\mathcal{A}_1^{\tau}$
and reasoning as in the proof of Lemma $B.10$ in \cite{FGP}.
 \end{proof}

\begin{lemma}{\bf (Conjugation of $\omega\cdot\pa_{\vphi}$).}\label{EgorovTempo}
Let $\mathcal{A}^{\tau}$ be the flow of the system \eqref{diffeotot}.
$(i)$ 
There exist $\su:=\su(m, \rho)$ and $\delta:=\delta(m, \rho)$ such that, if
\eqref{coin1} holds, then
\begin{equation*}%\label{coniugato1001}
\mathcal{A}^{1}\circ\omega\cdot\pa_{\vphi}\circ(\mathcal{A}^{1})^{-1}
 = \omega\cdot\pa_{\vphi}-
\ii  \opw(g(\vphi,x)\x)  
\end{equation*}
where $ g \in S^{0}$, independent of $\x$.  
In particular one has $g(\vphi,x)\x=p_0(\tau,\vphi,x,\x)_{|\tau=1}$ where
$p_0$ solves the equation
\begin{equation}\label{eq:intq0}
\frac{d}{d \tau } 
p_0(\tau,\vphi, x, \xi) = \{ b(\tau,\vphi,  x) \xi,  p_0(\tau, \vphi,x, \xi) \} 
+ \ii \omega\cdot\pa_{\vphi} b(\tau,\vphi,x) \xi \, ,
\quad p_0 (0) = 0 \,  . 
\end{equation}
Finally one has 
\begin{equation}\label{crawford1001}
\begin{aligned}
\lvert g \rvert^{\g, \calO}_{0, s, \alpha}
\le_{s, \rho, \alpha} & 
\lvert \beta \rvert^{\g, \calO}_{0, s+\s_1,\alpha+\s_1}
\end{aligned}
\end{equation}
for all %$0\le \tb\le \rho-2$ and 
$s_0\le s\le \mathcal{S}$.
Moreover one has
\begin{equation}\label{jamal2101}
\begin{aligned}
\lvert \Delta_{12} g   \rvert_{0, p, \alpha} \le_{p, \alpha, \rho} 
&\lvert \Delta_{12} \beta  \rvert_{0, p+\s_1, \alpha+\s_1}
(1+\lvert \beta \rvert_{0, p+\s_1, \alpha+\s_1})\,,
%\\
%\mathbb{M}_{\Delta_{12} R_{\rho} } (p, \tb) \le_{p, \rho} 
%&\lvert \Delta_{12} \beta   \rvert_{0, p+\s_1, \s_1}
%\lvert \beta \rvert_{0, p+\s_1, \s_1}\,,
\end{aligned}
\end{equation}
%for all $0\le \tb \le \rho-3$ and 
where $p$ is the constant 
given in Definition \ref{ellerho}.

\noindent
$(ii)$ Assume that $\beta\in S_{k}^{0}$
and that \eqref{IpotesiPiccolezzaIdeltaDP} holds. Then, for $\e>0$ small enough, 
the result of item $(i)$ holds
with $g\in S_{k}^{0}$.% and $R\in \mathfrak{L}_{\rho,p}^{k}$.
\end{lemma}
\begin{proof}
$(i)$ One can reason essentially as in Lemma $A.5$ in \cite{BFP}.
$(ii)$ One gets that $g\in S_k^{0}$ by using the equation \eqref{eq:intq0}.
\end{proof}

\begin{lemma}\label{flusso12basso}
{\bf (Conjugation of a pseudo differential operator).}
Let $\Phi^{\theta}(\vphi)$ be the flow of \eqref{linprob} 
with symbol $ f(\vphi, x, \xi ) $ in $  S^{m} $ with $ m \leq 1/ 2  $, 
of the form  \eqref{sim2} or \eqref{sim3}. 
Let $a=a(\omega, \mathfrak{I}(\omega))$ 
in $S^{m'}$
depending on   $\omega\in \calO\Subset 
\mathbb{R}^{\nu}$ and on $\mathfrak{I}$  
in a Lipschitz way for some $m'\in \R $. 
$(i)$ There exist $\su:=\su(m, m', \rho)$
and $\delta:=\delta(m, m', \rho)$ such that, if
\begin{equation}\label{ipoipo}
\lvert f \rvert^{\gamma, \calO}_{m, s_0+\s_1, \alpha+\su}<\delta\,,
\end{equation}
then
\begin{equation}\label{coniugato1000}
\begin{aligned}
& \Phi^{1}(\vphi)\opw(a(\vphi,x,\x))(\Phi^{1}(\vphi))^{-1} = 
\opw(c(\vphi,x,\x))+R(\vphi)\,,
\end{aligned}
\end{equation}
 where
 \begin{equation}\label{crawford100}
\begin{aligned}
\lvert c \rvert^{\g, \calO}_{m', s, \alpha}\le_{s, \rho, \alpha, m, m'} & 
\lvert a \rvert^{\g, \calO}_{m', s+\s_1,\alpha+\s_1}
(1+\lvert f \rvert^{\g, \calO}_{m, s_0+\s_1, \alpha+\s_1})
+\lvert a \rvert^{\g, \calO}_{m', s_0+\s_1, \alpha+\s_1}
\lvert f \rvert^{\g, \calO}_{m, s+\s_1, \alpha+\s_1}\,,\\
\mathbb{M}^{\g}_{R}(s, \tb)
&\le_{s, \rho,  m , m'} \lvert a \rvert^{\g, \calO}_{m',s+\s_1, \s_1}
\lvert f \rvert^{\g, \calO}_{m, s_0+\s_1, \s_1}
+\lvert a \rvert^{\g, \calO}_{m, s_0, \s_1}
\lvert b \rvert^{\g, \calO}_{m', s+\s_1, \s_1}\,,
\end{aligned}
\end{equation}
for all $0\le \tb\le \rho-2$ and $s_0\le s\le \mathcal{S}$.
Moreover one has
\begin{equation}\label{jamal2100}
\begin{aligned}
\lvert \Delta_{12} c   \rvert_{m', p, \alpha} \le_{p, \alpha, \rho, m, m'} 
&\lvert \Delta_{12} a   \rvert_{m, p+\s_1, \alpha+\s_1}
(1+\lvert f \rvert_{m, p+\s_1, \alpha+\s_1})
\\&+\lvert a \rvert_{m', p+\s_1, \alpha+\s_1}
\lvert \Delta_{12} f   \rvert_{m', p+\s_1, \alpha+\s_1}\,,\\
\mathbb{M}_{\Delta_{12} R_{\rho} } (p, \tb) \le_{p, \rho, m, m'} 
&\lvert \Delta_{12} a   \rvert_{m+1, p+\s_1, \s_1}
\lvert b \rvert_{m', p+\s_1, \s_1}\\
&+\lvert a \rvert_{m, p+\s_1, \s_1}
\lvert \Delta_{12} b   \rvert_{m'+1, p+\s_1, \s_1}\,,
\end{aligned}
\end{equation}
for all $0\le \tb \le \rho-3$ and where $p$ is the constant given in Definition \ref{ellerho}.
%$s_0\le p\le \mathcal{S}$.
In particular the symbol $c(\vphi,x,\x)$ admits the expansion
 \begin{equation}\label{castello1}
 c=a +  \{ f, a\} +  \frac12 \{ f, \{ f, a \} \} + r\,,
 \end{equation}
 where  $r\in S^{q}$ where $q=m'+\max\{m-3, 2 m - 4,3 m - 3\}$.
 $(ii)$ Assume that $a\in S^{m'}_{k'}$, $f\in S_{k}^{m}$
 and that \eqref{IpotesiPiccolezzaIdeltaDP} holds. Then, for $\e>0$ small enough,
 the result of item $(i)$ holds with $c\in S^{m'}_{\max\{k,k'\}}$ 
 and $R_{\rho}\in \mathfrak{L}_{\rho,p}^{\max\{k,k'\}}$.
\end{lemma}

\begin{proof}
$(i)$
To study the conjugated 
 operator $M:= \Phi^{1}(\vphi)\opw(a(\vphi,x,\x))(\Phi^{1}(\vphi))^{-1}$
we apply the usual 
Lie expansion up to order $L\geq 1$, i.e. 
\begin{equation}\label{escobar21}
\begin{aligned}
M&=\opw(a(\vphi,x,\x))
+\sum_{q=1}^{L}\frac{1}{q!}{\rm ad}_{\ii\opw(f(\vphi,x,  \x)) }^{q}[\opw(a(\vphi,x,\x))]
 \\&+  \frac{1}{L!}
  \int_{0}^{1}  (1- \theta)^{L} {\Phi}^{\theta} 
  {\rm ad}_{\ii\opw(f(\vphi,x,  \x)) }^{L+1}[\opw(a(\vphi,x,  \x)) ] 
  ({ \Phi}^{\theta})^{-1} 
d \theta\,,
\end{aligned}
\end{equation}
where, given two linear operators $M$ and  $B$, we  defined
\begin{equation*}%\label{adjlinear}
{\rm ad}_{M}[B]:=[M,B]\,, \qquad 
{\rm ad}^{q}_{M}[B]:=[M,{\rm ad}^{q-1}_{M}[B]]\,,\quad q\geq1\,.
\end{equation*}
By applying Lemma \ref{James2}
 we get (recall \eqref{moyal})
\[
{\rm ad}_{\opw(\ii f)}[\opw(a)]  =\big[\opw(\ii f), \opw(a)\big] 
 = \opw\big( \{  f, a \} + r_1 \big) \, , \quad  r_1\in S^{m+m' -3} \, , 
\]
up to a smoothing operator in $\gotL_{\rho,p}$.
Moreover (choosing $\rho$ possibly larger)
\[
{\rm ad}^{2}_{\opw(\ii f)}[\opw(a)] =\opw( \{ f, \{  f, a \} \} + r_2) \, , \quad  
r_2 \in S^{2m+m' -4} \, , 
\]
up to a smoothing operator in $\gotL_{\rho,p}$.
By  induction, for $ k \geq 3 $ we have  
\[
{\rm ad}^{k}_{\opw(\ii f)}[\opw(a)]=\opw( b_k),\quad 
b_k\in
S^{k(m-1)+m'}  \, , 
\]
up to a smoothing operator in $\gotL_{\rho,p}$.
We choose $L$  in such a way that  
$  (L+1)(1-m) - m' \geq \rho$ and $ L + 1 \geq 3 $,  
so that the operator 
$ \opw( b_{L+1}) $ belongs to $\gotL_{\rho,p}$.
The integral Taylor remainder in \eqref{escobar21} 
belongs to $\gotL_{\rho,p}$ as well by Lemma B.2 in \cite{FGP}.
The estimates \eqref{crawford100}-\eqref{jamal2100}
follows by using \eqref{crawford}-\eqref{jamal2}.
$(ii)$ The second item follows by reason exactly as in item $(i)$ but using the 
composition Lemma \ref{Jameshomo}
instead of Lemma \ref{James}, \ref{James2}.
\end{proof}

\begin{lemma}{\bf (Conjugation of $\omega\cdot\pa_{\vphi} $).} \label{flusso12basso2}
Let $\Phi^{\theta}(\vphi)$ be the flow of \eqref{linprob}  
with symbol $ f(\vphi, x, \xi ) \in S^{m}$ 
with $ m \leq 1/ 2  $
of the form  \eqref{sim2} or \eqref{sim3}. 
Assume also that 
$f=f(\omega, \mathfrak{I}(\omega))$ 
depends 
  $\omega\in \calO\Subset 
\mathbb{R}^{\nu}$ and on $\mathfrak{I}$  
in a Lipschitz way. 
$(i)$ There exist $\su:=\su(m,\rho)$
and $\delta:=\delta(m,\rho)$ such that, if
\eqref{ipoipo} holds then
\begin{equation}\label{coniugato1002}
\Phi^{1}(\vphi)\circ\omega\cdot\pa_{\vphi}\circ\Phi^{-1}(\vphi)
 = \omega\cdot\pa_{\vphi}-
\ii  \opw(c)  + R_{\rho}
\end{equation}
where $ c \in S^{m}$  and $ R_{\rho} \in \gotL_{\rho,p}$.
In particular one has
\begin{equation}\label{crawford100tempo}
\begin{aligned}
\lvert c \rvert^{\g, \calO}_{m, s, \alpha}&\le_{s, \rho, \alpha, m} 
\lvert f \rvert^{\g, \calO}_{m, s+\s_1, \alpha+\s_1}
+\lvert f \rvert^{\g, \calO}_{m, s_0+\s_1, \alpha+\s_1}
\lvert f \rvert^{\g, \calO}_{m, s+\s_1, \alpha+\s_1}\,,\\
\mathbb{M}^{\g}_{R_\rho}(s, \tb)
&\le_{s, \rho,  m }  
\lvert f \rvert^{\g, \calO}_{m, s+\s_1, \s_1}
+\lvert f \rvert^{\g, \calO}_{m, s_0, \s_1}
\lvert f \rvert^{\g, \calO}_{m, s+\s_1, \s_1}\,,
\end{aligned}
\end{equation}
for all $0\le \tb\le \rho-2$ and $s_0\le s\le \mathcal{S}$.
Moreover one has
\begin{align}
\lvert \Delta_{12} c   \rvert_{m, p, \alpha} &\le_{p, \alpha, \rho, m} 
\lvert \Delta_{12}f \rvert_{m, p+\s_1, \alpha+\s_1}
+\lvert f \rvert_{m, p+\s_1, \alpha+\s_1}
\lvert \Delta_{12} f   \rvert_{m, p+\s_1, \alpha+\s_1}\,,\label{crawford2100tempo}\\
\mathbb{M}_{\Delta_{12} R_{\rho} } (p, \tb) &\le_{p, \rho, m, m'} 
\lvert \Delta_{12} f   \rvert_{m+1, p+\s_1, \s_1}
\lvert f \rvert_{m, p+\s_1, \s_1}\,,\label{jamal2100tempo}
\end{align}
for all $0\le \tb \le \rho-3$ and where $p$ is the constant given in Definition \ref{ellerho}.
In particular the symbol $c(\vphi,x,\x)$ admits the expansion
 \begin{equation}\label{castello1tempo}
 c=\omega\cdot\pa_{\vphi}f 
+ \frac{1}{2} \{  f,  \omega\cdot\pa_{\vphi} f \}+r\,,
 \end{equation}
 where  $r\in S^{2m-3}$ satisfies bounds as \eqref{crawford100tempo}
 and \eqref{crawford2100tempo}.
  $(ii)$ Assume that  $f\in S_{k}^{m}$
 and that \eqref{IpotesiPiccolezzaIdeltaDP} holds. Then, for $\e>0$ small enough,
 the result of item $(i)$ holds with $c\in S^{m}_{k}$ 
 and $R_{\rho}\in \mathfrak{L}_{\rho,p}^{k}$.
\end{lemma}

\begin{proof}
By using  the Lie expansion \eqref{escobar21} we have 
$$
\begin{aligned}
\Phi^{1}\circ\omega\cdot\pa_{\vphi}\circ\Phi^{-1} & =\omega\cdot{\pa_{\vphi}}
- \opw(\ii \omega\cdot{\pa_{\vphi}} f )
-\sum_{k=2}^{L}\frac{1}{k!}{\rm ad}^{k-1}_{\opw(\ii f)}
[\opw(\ii\omega\cdot\pa_{\vphi} f )] \\
&+\frac{1}{L!}\int_{0}^{1} (1-\theta)^{L} \Phi^{\theta}
\Big({\rm ad}^{L}_{\opw(\ii f)}[\opw(\ii\omega\cdot\pa_{\vphi}  f )] \Big)(\Phi^{\theta})^{-1} d\theta
\end{aligned}
$$
and the lemma   follows (both items $(i)$ and $(ii)$) using  
Lemma \ref{James2}, noting that 
$ f \#^{W}_\rho \omega\cdot\pa_{\vphi}  - \omega\cdot\pa_{\vphi}   \#_\rho f 
= \frac{1}{ \ii}\{ f, f_t \}  $ 
plus a symbol of order $ 2 m - 3 $
and reasoning as in the proof of Lemma \ref{flusso12basso}.
\end{proof}

\begin{lemma}\label{preparailsugo}
Fix $\rho\geq3$, consider a compact subset $ \calO \Subset\R^\nu$ 
and let $R\in \gotL_{\rho, p}( \calO )$
(see Def. \ref{ellerho}).
Consider a function $\beta$ such that
$\beta:=\beta(\omega, i(\omega))\in H^s(\T^{\nu+1})$  
for some $s\geq s_0$, assume that it is  Lipschitz in 
$\omega\in  \calO $ and $\mathfrak{I}$. 
Let $\mathcal{A}^{\tau}$ be the operator 
defined in \eqref{ignobelSymp}.\\
$(i)$ There exists $\mu=\mu(\mathtt{\rho})\gg1$, $\s=\s(\rho)$ 
and $\delta>0$ small such that
if $\| \beta\|^{\gamma,  \calO }_{s_0+\mu}\leq \delta$ 
and $\| \Delta_{12}\beta\|^{\gamma,  \calO }_{p+\s}\leq 1$, 
then the operator 
$M^{\tau}:=\mathcal{A}^{\tau}R(\mathcal{A}^{\tau})^{-1}$
belongs to the class $\gotL_{\rho,p}$.
In particular one has, for $s_0\le s\le \mathcal{S}$,
\begin{equation}\label{casalotti}
\begin{aligned}
\mathbb{M}^{\gamma}_{M^{\tau}}(s,\mathtt{b}) 
&\leq \mathbb{M}^{\gamma}_{R}(s,\mathtt{b})+
\|\beta\|_{s+\mu}^{\gamma, \calO }
\mathbb{M}^{\gamma}_{R}(s_0,\mathtt{b})\,, 
\qquad \mathtt{b}\leq \rho-2\,,\\
\mathbb{M}_{\Delta_{12} M^{\tau}  }(p,\mathtt{b}) &\leq  
\mathbb{M}_{\Delta_{12} R^{\tau}  }(p,\mathtt{b})
+\lVert \Delta_{12} \beta   \rVert_{p+\mu}
\mathbb{M}_{R^{\tau}}^{\gamma}(p,\mathtt{b})\,, 
\qquad \mathtt{b}\leq \rho-3\,.
\end{aligned}
\end{equation}
$(ii)$ Assume that $\beta\in S_{k}^{0}$, $R\in \mathfrak{L}_{\rho,p}^{k'}$ and 
that \eqref{IpotesiPiccolezzaIdeltaDP} holds. 
Then, for $\e>0$ small enough and $\rho$ large enough, the result of item $(i)$
follows with $M^{\tau}\in \mathfrak{L}_{\rho',p}^{\max\{k,k'\}}$ with $\rho'= \rho-28+4 \max\{ k, k'\}$.
\end{lemma}

\begin{proof}
$(i)$ It follows following word by word the proof of Lemma B.10 
in \cite{FGP} and using Lemma \ref{CoroDPdiffeo}.
$(ii)$ We Taylor expand the flow $\mathcal{A}^{\tau}$ by using \eqref{diffeotot}.
Notice that
\[
\pa_{\tau\tau}\mathcal{A}^{\tau}=\opw(b(\tau,\vphi,x)\x)^{2}\mathcal{A}^{\tau}
+\opw(\pa_{\tau}b(\tau,\vphi,x)\x)\mathcal{A}^{\tau}\,.
\]
Recalling $\beta\in S_{k}^{0}$, the \eqref{diffeotot} and using Lemma \ref{Jameshomo}
we deduce
\[
\pa_{\tau\tau}\mathcal{A}^{\tau}=\opw(B^{(2)}(\tau,\vphi,x,\x))\mathcal{A}^{\tau}+R^{(2)}\circ\mathcal{A}^{\tau}
\]
where $B^{(2)}\in S_{k}^{2}$ and ${R}^{(2)}\in \mathfrak{L}_{\rho,p}^{k}$.
Using \eqref{diffeotot} one can prove inductively that
\begin{equation*}
\pa_{\tau}^{q}\mathcal{A}^{\tau}=\opw(B^{(q)}(\tau,\vphi,x,\x))\mathcal{A}^{\tau}+R^{(q)}\circ\mathcal{A}^{\tau}
\end{equation*}
where $B^{(q)}\in S_{k}^{2}$ and ${R}^{(q)}\in \mathfrak{L}_{\rho,p}^{k}$, for 
$2\leq q\leq15-2k$.
In conclusion we get that
% and the expansion given of $\beta$. Up to terms in $\mathfrak{L}_{\rho, p}$ 
 \begin{equation}\label{pirati2}
\mathcal{A}^{\tau}=\sum_{i=1}^{14-2k} 
\varepsilon^i \opw(\mathtt{A}_i(\tau; \varphi, x, \xi))+
\opw(\widetilde{\mathtt{A}}(\tau; \varphi, x, \xi))+
\sum_{i=1}^{14-2k}\e^{i}\mathcal{R}_{i}+\widetilde{\mathcal{R}}
+\int_0^{\tau} \frac{(\tau-\sigma)^{14-2k}}{(14-2k)!} 
\partial^{15-2k}_{\tau} \mathcal{A}^{\sigma}\,d\sigma
\end{equation}
where $\mathtt{A}_i\in S^i$, $\mathcal{R}_i\in \mathfrak{L}_{\rho,p}$ have the form respectively 
\eqref{simboOMO2DEF}, \eqref{simboOMO2DEF2}, 
$\widetilde{\mathtt{A}}\in S^{15-2k}$, $\widetilde{R}\in \mathfrak{L}_{\rho,p}$
satisfy estimates \eqref{AVBDEF}, \eqref{AVBDEF2}.
Moreover, by assumption, we have $R=\sum_{q=1}^{14-2k'}\e^{q}R_{q}+\widetilde{R}$
with $R_i$ as in \eqref{simboOMO2DEF2} and $\widetilde{R}$ satisfying \eqref{AVBDEF2}.
Reasoning as in Lemma B.10 
in \cite{FGP} one can check that the operators
\[
\opw(\mathtt{A}_i(\tau; \varphi, x, \xi))R_{q}\,,\quad \mathcal{R}_{i}R_{q}
\]
are $(i+q)$-homogeneous remainder as in \eqref{simboOMO2DEF2}
belonging to $\mathfrak{L}_{\rho',p}$. Similarly one has that
\[
\Big(\opw(\mathtt{A}_i(\tau; \varphi, x, \xi))+\opw(\widetilde{\mathtt{A}}(\tau; \varphi, x, \xi))
+\mathcal{R}_{i}+\widetilde{\mathcal{R}}\Big)\widetilde{R}\,,
\qquad \big(\opw(\widetilde{\mathtt{A}}(\tau; \varphi, x, \xi))+\widetilde{\mathcal{R}}\big)
R_{q}
%\widetilde{\mathcal{R}}_{q}\widetilde{R}\,,
%\qquad (\partial^{15-2k}_{\tau} \mathcal{A}^{\sigma})\big(R_q+\widetilde{R}\big)
\]
belong to $\mathfrak{L}_{\rho',p}$ and satisfy \eqref{AVBDEF2} with 
$k\rightsquigarrow \max\{k,k'\}$.
Reasoning in the same way one gets that 
$(\partial^{15-2k}_{\tau} \mathcal{A}^{\sigma})R$
is in $\mathfrak{L}^{\max\{k,k'\}}_{\rho',p}$.
The inverse flow $(\mathcal{A}^{\tau})^{-1}$ admits an expansion similar to \eqref{pirati2}.
Then the Lemma follows.
%
% and one can prove iteratively that
%\[
%\partial^{15-2k}_{\tau} \mathcal{A}^{\sigma}=(\op(a(\varphi, x , \xi))+R_{\rho})\circ \mathcal{A}^{\tau}
%\]
%with $R_{\rho}\in\mathfrak{L}_{\rho, p}$.
%% It follows by reasoning as done for the remainder $R$ in \eqref{pirati}
%%in Theorem \ref{EgorovQuantitativo}.
\end{proof}

\begin{lemma}\label{CoroDPdiffeo200}
Fix $\rho\geq3$, consider a compact subset $ \calO \subset\R^\nu$ 
and let $R\in \gotL_{\rho, p}( \calO )$.
$(i)$ There exist $\mu=\mu(\rho)$, $\s=\s(\rho)$ and $\delta>0$ such that, if 
$| f|^{\gamma, \calO}_{m,s_0+\s,\alpha}\leq \delta$ and
$| \Delta_{12}f|^{\gamma,  \calO }_{p+\s}\leq 1$, 
with $f$
in \eqref{sim2} or \eqref{sim3}, 
then the operator 
$M^{\tau}:=\Phi^{\tau}R(\Phi^{\tau})^{-1}$, where
$\Phi^{\tau}$ is the flow of \eqref{linprob},
belongs to the class $\gotL_{\rho,p}$.
In particular, for $s_0\le s\le \mathcal{S}$,
the bounds \eqref{casalotti} hold
with 
$\| \beta\|^{\gamma,  \calO }_{s+\mu}$
and $\| \Delta_{12}\beta\|^{\gamma,  \calO }_{p+\s}$ replaced by
$| f|^{\gamma, \calO}_{m,s+\s,\alpha}$ and 
$|\Delta f|^{\gamma, \calO}_{m,p+\s,\alpha}$.
$(ii)$ Assume that $f\in S_{k}^{m}$, $R\in \mathfrak{L}_{\rho,p}^{k'}$ and 
that \eqref{IpotesiPiccolezzaIdeltaDP} holds. 
Then, for $\e>0$ small enough, the result of item $(i)$
follows with $M^{\tau}\in \mathfrak{L}_{\rho,p}^{\max\{k,k'\}}$.
\end{lemma}

\begin{proof}
By Lemma \ref{CoroDPdiffeo2} we have that the flow $\Phi^{\tau}$
in \eqref{linprob} satisfies the same bounds given in Lemma \ref{CoroDPdiffeo}.
Then one can reason as in Lemma \ref{preparailsugo}.
%Hence the result follows by using the same arguments
%in the proof of Lemma B.10 in \cite{FGP}.
\end{proof}

%%%%%%%%%%%%%%%%%%%%%%%%%%%%%%%%%%%%%%%%%%%%%%%%%%

\section{Technical Lemmata} 

\subsection{Dirichlet-Neumann operator}\label{propDirNeu}
Here we collect some results about the Dirichlet-Neumann operator.
For more details we refer, for instance, to 
%Without trying to be exhaustive we
%refer to 
%\cite{ADel1,ADel2,Lan2} 
\cite{ADel1}, \cite{Lan2}
and reference therein.
We also remark that, in our context, we shall also consider 
$C^{\infty}$ profile $\eta(x)$. This is due to the fact that
at each step of the Nash-Moser we  perform a 
$C^{\infty}$-regularization.

Let $\eta\in C^{\infty}$. It is known that 
the Dirichlet-Neumann operator is 
(in the infinite depth case)
a \emph{pseudo differential} operator of the form
\begin{equation}\label{Dir-Neum}
G(\eta):=|D|+R_{G}(\eta)\,,
\end{equation}
where
$G(0)=|D|$ and the remainder $R_{G}(\eta)\in OPS^{-\infty}$ 
(see Def. \ref{pseudoR}).
The key result of this section is the following.
% Proposition, which is essentially
%Proposition 2.28 in \cite{BM1}.

\begin{proposition}{\bf (Dirchlet-Neumann).}\label{DNDNDN}
Assume that $\eta(\omega,\vphi,x)$, for $\omega\in\mathcal{O}\Subset\mathbb{R}^{\nu}$, 
is $C^{\infty}$ in the variables $\vphi,x$
and Lipschitz in $\omega$.
Then there exist constants $\s=\s(s_0)$, $\delta=\delta(s_0)$, $s_0\geq (d+1/2)$,
such that, if
\begin{equation*}%\label{DN1}
\|\eta\|_{s_0+\s}^{\gamma,\calO}\leq \delta\,,
\end{equation*}
then the Dirichlet-Neumann operator in \eqref{eq:112a}
has the form \eqref{Dir-Neum}
where $R_{G}(\eta)$ is a pseudo differential operator
satisfying the following.
For any $m,s,\alpha\in \mathbb{N}$ there exists $\s_{1}=\s_1(m,s,\alpha)$
 such that
 \begin{equation}\label{DN2}
 |R_{G}(\eta)|^{\gamma,\mathcal{O}}_{-m,s,\alpha}
 \lesssim_{m, s, \alpha} \|\eta\|^{\gamma,\calO}_{s+\s_1}\,.
 \end{equation}
 Let $p\geq s_0+\s$, consider $\eta_1,\eta_2\in C^{\infty}$
 and set $\Delta_{12}R_{G}:=R_{G}(\eta_1)-R_{G}(\eta_2)$.
 There exists $\delta(p)>0$ 
 such that, if 
 \[
 \|\eta_1\|_{p}+\|\eta_2\|_{p}\leq \delta_1(p)\,,
 \] 
 then, for any $m,\alpha\in \mathbb{N}$,
 \begin{equation*}%\label{DN3}
 |\Delta_{12} R_{G}|_{-m,p,\alpha}\lesssim_{p} \|\eta_1-\eta_2\|_{p}\,.
 \end{equation*}
\end{proposition}

\begin{proof}
One can deduce the result by following almost word by word
the proof
Proposition 2.37 in \cite{BM1}.
\end{proof}
We also have the following.
\begin{proposition}{\bf (Tame estimates on the Dirichlet-Neumann).}
\label{DNDNDN2}
There exist $\s=\s(s_0)>0$, $\delta=\delta(s_0)>0$ such that,
if $\|\eta\|^{\gamma,\calO}_{s_0+\s}\leq \delta$, then,
for all $s\geq s_0$
\begin{align*}
\|(G(\eta)-|D|)\psi\|^{\gamma,\calO}_{s}&\leq_s
\|\eta\|^{\gamma,\calO}_{s+\s}\|\psi\|^{\gamma,\calO}_{s_0}+
\|\eta\|^{\gamma,\calO}_{s_0+\s}\|\psi\|^{\gamma,\calO}_{s}\,,
%\label{DNDN1}
\\
\|(G'(\eta)[\hat{\eta}]\psi)\psi\|^{\gamma,\calO}_{s}&\leq_s
\|\psi\|^{\gamma,\calO}_{s+2}\|\hat{\eta}\|^{\gamma,\calO}_{s_0+1}
+
\|\psi\|^{\gamma,\calO}_{s_0+2}\|\hat{\eta}\|^{\gamma,\calO}_{s+1}
+
\|\eta\|^{\gamma,\calO}_{s+\s}
\|\psi\|^{\gamma,\calO}_{s_0+2}\|\hat{\eta}\|^{\gamma,\calO}_{s_0+1}\,,
%\label{DNDN2}
\\
\|(G''(\eta)[\hat{\eta},\hat{\eta}]\psi)\psi\|^{\gamma,\calO}_{s}&\leq_s
\|\psi\|^{\gamma,\calO}_{s+3}(\|\hat{\eta}\|^{\gamma,\calO}_{s_0+2})^{2}
+
\|\psi\|^{\gamma,\calO}_{s_0+3}\|\hat{\eta}\|^{\gamma,\calO}_{s+2}
\|\hat{\eta}\|^{\gamma,\calO}_{s_0+2}
%\label{DNDN3}
\\
&\qquad\qquad +
\|\eta\|^{\gamma,\calO}_{s+\s}
\|\psi\|^{\gamma,\calO}_{s_0+3}(\|\hat{\eta}\|^{\gamma,\calO}_{s_0+2})^{2}\,.
%\nonumber
\end{align*}
\end{proposition}

\begin{proof}
It follows by Lemma 2.32 in \cite{BM1}.
\end{proof}

\noindent
{\bf Algebraic properties.}
One has that, setting
$f^\vee (x) := f (-x)$, 
 the Dirichlet-Neumann operator satisfies
\[
G( \eta^\vee ) [ \psi^\vee ] (x) = G( \eta ) [ \psi ] ( - x) \, . 
\]
This implies that
\begin{equation*}%\label{muscadet}
B( \eta^\vee , \psi^\vee )(x) = B( \eta , \psi ) ( - x) \,,\qquad
V( \eta^\vee , \psi^\vee )(x) = -V( \eta , \psi ) ( - x) \,.
\end{equation*}
Morever, if $(\eta,\psi)$ satisfy
\begin{equation*}%\label{rev1}
\eta(-t,-x)=\eta(t,x)\,,\;\qquad \psi(-t,-x)=-\psi(t,x)\,,
\end{equation*}
% \eqref{rev1} (i.e. are reversible) 
 then
\begin{equation*}%\label{muscadet2}
B( \eta, \psi )(-t,-x) = -B( \eta , \psi ) ( t,x) \,,\qquad
V( \eta, \psi )(-t,-x) = V( \eta , \psi ) ( t,x) \,.
\end{equation*}
The following Lemma is fundamental for our scopes.

\begin{lemma}\label{MomentumandDN}
Let $\eta,\psi\in S_{\mathtt{v}}$ be $C^{\infty}$ functions in the variables $(\vphi,x)\in \mathbb{T}^{\nu+1}$. 
Then the following holds.
$(i)$ $G(\eta)\psi\in S_{\mathtt{v}}$. $(ii)$ $V(\eta,\psi), B(\eta,\psi)\in S_{\mathtt{v}}$.
$(iii)$ One has
\begin{equation}\label{DmomentoDN}
\Big(-(\mathtt{v}\cdot\pa_{\vphi}G)(\eta)+G(\eta)\circ\pa_{x}-\pa_{x}\circ G(\eta)\Big)h=0\,,
\end{equation}
for any $h\in H^{s}(\mathbb{T}^{\nu+1};\mathbb{R})$.
\end{lemma}
\begin{proof}
Since $\eta,\psi\in S_{\mathtt{v}}$, Lemma \ref{funzioniSVVV}
implies that
$\mathtt{v}\cdot\pa_{\vphi}\psi+\pa_{x}\psi
=\mathtt{v}\cdot\pa_{\vphi}\eta+\pa_{x}\eta=0$.
Then, using the ``shape derivative'' formula in \eqref{shapeDer}
we deduce
\begin{equation}\label{pennaNera}
\begin{aligned}
\big(\mathtt{v}\cdot\pa_{\vphi}+\pa_{x}\big)G(\eta)\psi
&=G(\eta)\big[
\mathtt{v}\cdot\pa_{\vphi}\psi+\pa_{x}\psi
\big]+
G'(\eta)\big[\mathtt{v}\cdot\pa_{\vphi}\eta+\pa_{x}\eta\big]\psi=0\,.
\end{aligned}
\end{equation}
Then Lemma \ref{funzioniSVVV} implies item $(i)$.
Item $(ii)$ follows by item $(i)$ and formul\ae\,  \eqref{def:V} and \eqref{form-of-B}.
Reasoning as in \eqref{pennaNera} we have
\[
(\mathtt{v}\cdot\pa_{\vphi}G)(\eta) h= G'(\eta)[\mathtt{v}\cdot\pa_{\vphi}\eta]h\,,
\qquad 
\pa_{x}\circ G(\eta) h=G(\eta)h+G'(\eta)[\eta_{x}]h\,.
\]
Then \eqref{DmomentoDN} follows using that $\eta\in S_{\mathtt{v}}$.
\end{proof}

\vspace{0.5em}
\noindent
{\bf Homogeneity expansion of the function $V$.}
We now compute an expansion in degree of homogeneity of
$V$ by
%In this subsection we shall compute explicitly 
%the coefficients
%the first terms $\mathtt{V}_1$, $\mathtt{V}_2$
%of such expansion. 
%In order to do this, we shall 
following the strategy used in section 3.2 in \cite{BFP}.
Recalling the good unknown map $\mathcal{G}$ in \eqref{flussoG}
ad the map $\Lambda$ in \eqref{CVWW} we define
\begin{equation}\label{levariabili}
%\omega:=\psi-B\eta\,,\quad 
u:=\frac{1}{\sqrt{2}}|D|^{-\frac{1}{4}}\eta+\frac{\ii}{\sqrt{2}}|D|^{\frac{1}{4}}(\psi-B\eta)\,.
\end{equation}
%\comment{bisogna trovare nome diverso per omega}
Notice that the variable $u$ is just the first component of $\Lambda\mathcal{G}^{-1}\vect{\eta}{\psi}$.
We have the following Lemma.
\begin{lemma}{\bf (Expansion of $ V $).}\label{siespandeV}
The function $ V $ defined in \eqref{def:V} admits the expansion 
\be\label{expV}
V = (\psi-B\eta)_x +(|D| (\psi-B\eta)_{x}) \eta  + V_{\geq 3}
\ee
where  $V_{\geq 3}\sim O(u^{3}) $ is a cubic function in $ (\eta, \psi)$.
In the complex variable $(u, \bar{u})$ defined in  \eqref{levariabili}, $V$ can be expanded as in \eqref{expTaylorFunz1000}, more precisely  we have 
\begin{align}
\widetilde{\mathtt{V}}_1 & = \frac{1}{\ii \sqrt{2}} \pa_x |D|^{- \frac14} (u - \bar u )  \label{coeV1},\\
 \widetilde{\mathtt{V}}_2 & = \frac{1}{ 2\ii }  \Big(  |D|^{\frac34}\pa_{x} (u- \bar u) \Big)  \big( |D|^{\frac14} (u + \bar u) \big)
 \, .  \label{coeV2}
\end{align}
\end{lemma}
\begin{proof}
By \eqref{form-of-B} and using the expansion \eqref{DNexp3}, we deduce  $  B = |D| \psi  $ up to 
a quadratic function in $u$.
 As a consequence, 
by \eqref{def:V} and \eqref{levariabili}, we have
$$
V = 	\psi_x - B \eta_x %= ( \omega +  B \eta )_x - B \eta_x  = 
=(\psi-B\eta)_x + \pa_x \big( (|D| \psi) \eta \big) - (|D| \psi) \eta_x  
$$
up to terms $O_2(\eta, \psi)$.
Since $B\eta$ is 
%$ \psi = \omega $ plus a term 
$O_2(\eta, \psi)$  we get \eqref{expV}. 
The \eqref{coeV1}, \eqref{coeV2} follow by passing to the complex variable $u$ in 
\eqref{levariabili}.
\end{proof}

\begin{lemma}{\bf (Coefficients of $ \widetilde{\mathtt{V}}_1 $ and $  \widetilde{\mathtt{V}}_2 $).}\label{lem:V1} 
The coefficients of the functions $  \widetilde{\mathtt{V}}_1 $ and $  \widetilde{\mathtt{V}}_2 $ 
in \eqref{coeV1}-\eqref{coeV2} are, 
for all $ n \in \Z \setminus \{0\} $  
\begin{equation}\label{coV1} 
( \widetilde{\mathtt{V}}_1)^{+}_n = ( \widetilde{\mathtt{V}}_1)^{-}_n = \frac{1}{\sqrt{2}} n |n|^{-1/4} \, , 
\qquad 
(\widetilde{\mathtt{V}}_2)^{+-}_{n, n}=(\widetilde{\mathtt{V}}_2)^{-+}_{n, n}  
= \frac{1}{2}n | n | \, .
%\quad ( \widetilde{\mathtt{V}}_2)^{+-}_{n,-n}=0 \, . 
\end{equation}
\end{lemma}

\begin{proof}
By \eqref{coeV1} and expanding in Fourier series the function $u$ as 
\begin{equation*}%\label{complex-uU}
u(x) = \sum_{n \in \Z\setminus\{0\} } u_{n}\frac{e^{\ii n x }}{\sqrt{2\pi}} \, , \qquad 
u_{n} := \frac{1}{\sqrt{2\pi}} \int_\T u(x) e^{-\ii n x } \, dx \, ,
\end{equation*}
 we have
\begin{align*}
 \widetilde{\mathtt{V}}_1   
& = \frac{1}{\ii \sqrt{2}} \pa_x | D |^{-1/4} (u - \bar u )
= \frac{1}{\sqrt{2\pi}}\frac{1}{ \sqrt{2}}   \sum_{n \neq 0} n | n |^{-1/4} u_n e^{\ii n x} 
+  n | n |^{-1/4} \ov{u_n} e^{- \ii n x}  \,,
\end{align*}
which implies the expressions  for $ (\widetilde{\mathtt{V}}_1)^{\pm}_n $  in  \eqref{coV1}. 
By \eqref{coeV2} 
%and \eqref{complex-uU} 
we have
\[
\begin{aligned}
 \widetilde{\mathtt{V}}_2=\frac{1}{4i\pi}
\big(
\sum_{j\in \mathbb{Z}\setminus\{0\}}|j|^{\frac{3}{4}}(\ii j)u_{j}e^{\ii jx}
&-\sum_{j\in \mathbb{Z}\setminus\{0\}}|j|^{\frac{3}{4}}(-\ii j)\ov{u_{j}}e^{-\ii jx}
\big)\times \\
&\times\big(
\sum_{j\in \mathbb{Z}\setminus\{0\}}|j|^{\frac{1}{4}}u_{j}e^{\ii jx}
+\sum_{j\in \mathbb{Z}\setminus\{0\}}|j|^{\frac{1}{4}}\ov{u_{j}}e^{-\ii jx}
\big)\,.
\end{aligned}
\]
By collecting form the formula above the terms $u_{j_1}\ov{u_{j_2}}$
and $\ov{u_{j_1}}{u_{j_2}}$
we obtain
\[
\begin{aligned}
&\frac{1}{4\pi}\sum_{j_1,j_2\in \mathbb{Z}\setminus\{0\}}u_{j_1}\ov{u_{j_2}}
( \widetilde{\mathtt{V}}_2)^{+-}_{j_1,j_2}
e^{\ii (j_1-j_2)x}\,,
\qquad 
\frac{1}{4\pi}\sum_{j_1,j_2\in \mathbb{Z}\setminus\{0\}}\ov{u_{j_1}}{u_{j_2}}
( \widetilde{\mathtt{V}}_2)^{-+}_{j_1,j_2}
e^{-\ii (j_1-j_2)x}\,,
\\&
( \widetilde{\mathtt{V}}_2)^{+-}_{j_1,j_2}:=\frac{1}{2}|j_1|^{\frac{3}{4}}|j_2|^{\frac{1}{4}}j_1
\qquad
( \widetilde{\mathtt{V}}_2)^{-+}_{j_1,j_2}:=
\frac{1}{2}
|j_1|^{\frac{3}{4}}|j_2|^{\frac{1}{4}}j_1
\,.
\end{aligned}
\]
Taking $j_1=j_2=n$ or $j_1=n=-j_2$ we obtain the \eqref{coV1}
for the coefficients $( \widetilde{\mathtt{V}}_2)^{+-}_{n,n}$,
 $( \widetilde{\mathtt{V}}_2)^{-+}_{n,n}$ .
%$( \widetilde{\mathtt{V}}_2)^{+-}_{n,-n}$ .
\end{proof}

\subsection{Technical Lemmata of Section \ref{sec:linBNF}}\label{sec:proofs}

Here we collects the proofs of some technical Lemmata appearing in the paper. 

\begin{proof}[{\bf Proof of Lemma \ref{lowerboundDel1}}]\label{10.3proofLem}
We use the following notation: given a set of integer indexes $j_1, \dots, j_n$
we write 
\[
\max({p}):= p{\rm -largest} \; {\rm value \; among} \; |j_1|, \dots, |j_n|\,.
\]

\emph{Proof of \eqref{lowerboundDel1BIS}}.
Let us call $j_*$ the only component of $\ell$ such that $\ell_{j_*}=1$. 
We first consider the case $\sigma\,=-\sigma'$. We have that
\[
\lvert \delta^{(1)}_{\sigma, \sigma', j, k}(\ell) \rvert=\lvert \sqrt{\lvert j_* \rvert}+\sigma (\sqrt{\lvert j \rvert}+\sqrt{\lvert k\rvert} )\rvert.
\]
If $\sigma=1$ then the lower bound is given by $3$. If $\sigma=-1$ then by the momentum conservation we have that $j+k=j_*$, and so $\lvert j \rvert+\lvert k \rvert-\lvert j_*\rvert\geq 0$. If $j_*\neq\max(1)$ then
\[
\lvert \sqrt{\lvert j \rvert}+\sqrt{\lvert k\rvert}-\sqrt{\lvert j_*\rvert} \rvert\geq \sqrt{\lvert j \rvert}\geq 1.
\]
If $j_*=\max(1)$, suppose without loss of generality that $j=\max(2)$, then by Remark \ref{max12}
\begin{align*}
&\lvert \sqrt{\lvert j \rvert}+\sqrt{\lvert k\rvert}-\sqrt{\lvert j_*\rvert} \rvert
=\frac{(\sqrt{\lvert j \rvert}+\sqrt{\lvert k \rvert})^2
-\lvert j_*\rvert}{\sqrt{\lvert j \rvert}+\sqrt{\lvert k\rvert}
+\sqrt{\lvert j_*\rvert}}\geq \frac{2 \sqrt{\lvert j \rvert\lvert k \rvert}}{\sqrt{\lvert j \rvert}
+\sqrt{\lvert k\rvert}+\sqrt{\lvert j_*\rvert}}
%\\&
=\left(\frac{\sqrt{\lvert j \rvert}}{\sqrt{\lvert j_* \rvert}}\right)\frac{2 \sqrt{\lvert k \rvert}}{(1+\frac{\sqrt{\lvert k\rvert}+\sqrt{\lvert j\rvert}}{\sqrt{\lvert j_* \rvert}})}\geq \frac{2}{9}.
\end{align*}
Now we consider the case $\sigma=\sigma'$. The only non-trivial case is when $j=\max(1)$ and 
\[
\delta^{(1)}_{\sigma, \sigma', j, k}(\ell)=\sqrt{\lvert j_*\rvert}+\sqrt{\lvert k \rvert}-\sqrt{\lvert j \rvert}=\sqrt{\lvert j_*\rvert}-\frac{\lvert j \rvert-\lvert k \rvert}{\sqrt{\lvert j \rvert}+\sqrt{\lvert k \rvert}}.
\]
By the momentum conservation (see \eqref{deltaUNO2})  
we have that 
$\lvert j \rvert-\lvert k \rvert\le \lvert j-k\rvert=\lvert j_*\rvert$. 
If $2\sqrt{\lvert j_*\rvert}\le \sqrt{\lvert j \rvert}+\sqrt{\lvert k \rvert}$, 
then $\lvert \delta^{(1)}_{\sigma, \sigma', j, k}(\ell) \rvert\geq 1/2$. 
If $\sqrt{\lvert j \rvert}+\sqrt{\lvert k \rvert}<2\sqrt{\lvert j_*\rvert} $ 
then
\[
\lvert \delta^{(1)}_{\sigma, \sigma', j, k}(\ell) \rvert
>\inf_{j, k} \frac{1}{\sqrt{\lvert j \rvert}+\sqrt{\lvert k \rvert}}=:C_1(S)
\]
indeed, by using that 
$\lvert j_*\rvert\geq \lvert \lvert j \rvert-\lvert k \rvert \rvert$ 
and $\lvert j \rvert\geq \lvert j_*\rvert$
\[
(\sqrt{\lvert j \rvert}+\sqrt{\lvert k \rvert}) \lvert \delta^{(1)}_{\sigma, \sigma', j, k}(\ell) \rvert\geq \sqrt{\lvert j \,j_*\rvert}-\lvert j_*\rvert+\sqrt{\lvert j_* k \rvert}\geq \sqrt{\lvert j_* k \rvert}\geq 1.
\]
This implies the \eqref{lowerboundDel1BIS}.

\smallskip
\noindent
\emph{Proof of \eqref{lowerboundDel2BIS}}.
By \eqref{deltaUNO}, \eqref{deltaUNO2} we have to bound from below
the function
\begin{equation*}%\label{dono}
\sqrt{|j_1|}\pm\sqrt{|j_2|}\pm\sqrt{|j_3|}\pm \sqrt{|j_4|}\quad {\rm when}
\quad j_1\pm j_2\pm j_3\pm j_4=0\,,
\end{equation*}
for all possible combinations of signs. 
\noindent
\emph{Case} $(+,+,+)$. This case is trivial since
\begin{equation}\label{dono2}
|\sqrt{|j_1|}+\sqrt{|j_2|}+\sqrt{|j_3|}+ \sqrt{|j_4|}|\geq \max(3)^{1/2}\,.
\end{equation}
We have to consider the case when two indexes are in $S$ and the other two are in $S^{c}$.
Assume, w.l.o.g that $\max(3)=|j_1|$. Then if $j_1\in S$ the \eqref{dono2}
implies the \eqref{lowerboundDel2BIS}. If $j_1\in S^{c}$ then there exists $i=2,3,4$
such that $j_i\in S$ and $|j_i|>|j_1|$. Then again we obtain the \eqref{lowerboundDel2BIS}.

\noindent
\emph{Case} $(+,+,-)$. We have
\[
|\sqrt{|j_1|}+\sqrt{|j_2|}+\sqrt{|j_3|}- \sqrt{|j_4|}|=
\frac{||j_1|+|j_2|+|j_3|-|j_4|+
2\sqrt{|j_1j_2|}+2\sqrt{|j_2j_3|}+2\sqrt{|j_1j_3|}|}{
|\sqrt{|j_1|}+\sqrt{|j_2|}+\sqrt{|j_3|}+ \sqrt{|j_4|}|}\,.
\]
The momentum condition implies $|j_1|+|j_2|+|j_3|-|j_4|\geq0$. 
Therefore the quantity above is bounded
from below by ${\rm max}(3)^{1/2}$. We get the \eqref{lowerboundDel2BIS} as in the previous case.

\noindent
\emph{Case} $(-,+,-)$. This is the most difficult case.
Consider the function
\[
\psi(j_1,j_2,j_3,j_4)=\sqrt{|j_1|}-\sqrt{|j_2|}+\sqrt{|j_3|}- \sqrt{|j_4|}\,,
\]
on the hyperplane $j_1-j_2+j_3-j_4=0$. We claim that
\begin{equation}\label{dono4}
\begin{aligned}
{\rm either} \;\;\;\psi=0\qquad &{\rm or }\quad j_i=-j_k\quad {\rm for \; some } \;\;\; i,k\in \{1,2,3,4\}\,,\\
&{\rm or}\;\;\;  |\psi|\geq \frac{c}{({\rm max}(3))^{9}}\,,
\end{aligned}
\end{equation}
for some absolute constant $0<c\ll 1$.
The \eqref{dono4} implies the \eqref{lowerboundDel2BIS}. Indeed 
if the lower bound holds the one concludes as in the previous cases.
If, for instance $j_1=-j_2$ then, since $\psi\neq0$,
\begin{equation}\label{dono5}
|\psi|=\frac{||j_3|-|j_4||}{\sqrt{|j_3|}+\sqrt{|j_4|}}\geq 
\frac{1}{\max\{|j_3|,|j_4|\}^{1/2}}\,.
\end{equation}
By the choice of $S$ in \eqref{TangentialSitesDP}
it is not possible that both $j_1,j_2$ are in $S$. Hence we can assume
$j_1,j_3\in S$ and $j_2,j_4\in S^{c}$. 
By the momentum and $j_1=-j_2$ we deduce
$|j_4|\leq C(S)$. Then  \eqref{dono5} implies the \eqref{lowerboundDel2BIS}.

\smallskip
\noindent
\emph{Proof of the Claim \ref{dono4}}.
Assume that $\max(1)\geq 100(\max(3))^{3}$ and
without loss of generality we
 assume $\max(1)=|j_1|$, $\max(2)=|j_2|$. We write
 \begin{equation}\label{dono6}
 \psi=\frac{|j_3|-|j_4|}{\sqrt{|j_3|}+\sqrt{|j_4|}}+
 \frac{|j_1|-|j_2|}{\sqrt{|j_1|}+\sqrt{|j_2|}}\,.
 \end{equation}
 If $j_1=j_2$ then $\psi=0$. If $j_1=-j_2$ then  
 or $\psi=0$ or $|\psi|\geq 1/(\max(3))^{1/2}$ which implies 
 the \eqref{lowerboundDel2BIS}. If $j_1\neq j_2$ we have, by \eqref{dono6},
 \[
|\psi|\gtrsim
\frac{||j_3|-|j_4||}{\sqrt{|j_3|}+\sqrt{|j_4|}}-\frac{|j_1-j_2|}{\sqrt{|j_1|}+\sqrt{|j_2|}}\gtrsim
\frac{1}{(\max(3))^{1/2}}-\frac{2\max(3)}{\max(1)^{1/2}}\gtrsim 
\frac{1}{2(\max(3))^{1/2}}\,,
 \]
since  $\max(1)\geq 100(\max(3))^{3}$.
 It remains to study the case $\max(1)< 100(\max(3))^{3}$. 
 In this case to get the \eqref{lowerboundDel2BIS} it is sufficient to prove
 \begin{equation}\label{dono7}
 |\psi|\geq \frac{1}{10(\max(1))^{N_0}}\,,
 \end{equation}
 for some $N_0>0$. The \eqref{dono7} follows by Proposition
 $6.3$ in \cite{BFP}. %This conclude the proof.
 
 \smallskip
\noindent
\emph{Proof of \eqref{smallDivLinearBNF2}}.
We prove the result for $\s=\s'=+$ which is the most difficult case.
By a generic choice of the set $S$ one can assume that
$\bar{\omega}\cdot \ell\neq0$ for $0<\lvert \ell\rvert\le 6$. Therefore one has that there is a constant $K(S)$
such that
\[
|\bar{\omega}\cdot \ell|\geq 2 K(S)\,,
\]
since $|\ell|\leq 6$. Moreover
by the constraint in \eqref{deltaUNO2} we deduce that $||j|-|k||\leq \tilde{C}(S)$, for some
$\tilde{C}(S)$ depending on the tangential sites in $S$.
Then, by  taking $|j|,|k|\geq C(S):=\tilde{C}^{2}(S)/K^{2}(S)$ and $\e$ small enough 
we get
(see \eqref{deltaUNO})
\[
|\delta^{(p)}_{\s,\s',j,k}(\ell)|\geq |\bar{\omega}\cdot\ell|-\frac{||j|-|k||}{\sqrt{|j|}+\sqrt{|k|}}
%-\e^{2}\hat{C}(S)
\geq 2K(S)-K(S)
\]
which implies the \eqref{smallDivLinearBNF2}.
\end{proof}

\begin{proof}[{\bf Proof of Lemma \ref{almostC12}}]\label{10.6prooflemma}
If the operator $\mathbf{B}$ is bounded, then $\lvert (\mathbf{B})_{\sigma, j}^{\sigma', k}(\ell) \rvert\le c$ for some constant $c>0$. Since $\mathbf{A}$ is almost-diagonal and it preserves the momentum, there exist two constants $C_1, C_2$ depending on $S$ such that the following holds: for all $j, k$ such that $ (\mathbf{A})_{\sigma, j}^{\sigma', k}(\ell) \neq 0$ we have $C_1 \,\langle k \rangle \le \langle j \rangle\le C_2 \langle k \rangle$.
 Suppose that $\lvert j \rvert\geq \lvert k \rvert$, then
\[
c\geq \lvert (\mathbf{B})_{\sigma, j}^{\sigma', k}(\ell) \rvert=\lvert (\mathbf{A})_{\sigma, j}^{\sigma', k}(\ell) \rvert \langle j \rangle^{m}\, \langle k \rangle^{m}\geq C_2^{-1} \lvert (\mathbf{A})_{\sigma, j}^{\sigma', k}(\ell) \rvert {\langle j, k\rangle^{2m}}.
\]
This implies the thesis.
Assume now that \eqref{stimaCoeffalmost} holds.
Then, since ${\bf A}$ is almost-diagonal it is easy to note that
$\lvert (\mathbf{B})_{\sigma, j}^{\sigma', k}(\ell) \rvert\le\tilde{C}(S)|j-k|^{-1}|\ell|^{-\nu-1}$. Therefore,
for any $v\in H_{S^{\perp}}^{s}$, we have (using Cauchy-Schwarz)
\[
\|\underline{{\bf B}_{\s}^{\s'}v}\|_{{s}}^{2}
\stackrel{\eqref{space}}{\lesssim}
\sum_{k,p}|v_{k}(p)|^{2}\langle k,p\rangle^{2s}
\sum_{j,\ell} \langle j-k,\ell-p\rangle^{2s}
|{\bf B}_{\s,j}^{\s',k}(\ell-p)|^{2}\lesssim_{s,S}\|v\|_{s}\,,
\]
which implies the thesis.
\end{proof}

\begin{proof}[{\bf Proof of Lemma \ref{LemmaAggancio}}]\label{10.7prooflemma}
Recall that
$\mathtt{q}\in S^{-1/2}_{4}$ and $\mathcal{Q}\in \mathfrak{L}_{\rho,p}^{4}\otimes
\mathcal{M}_{2}(\mathbb{C})$.
 Moreover
(see  \eqref{espandoenonmifermo}, \eqref{espandoenonmifermo2}) 
we have
\[
\mathtt{q}=\sum_{i=1}^{6}\e^i \mathtt{q}_i+\mathtt{q}_{\geq7}\,,
\qquad
\mathcal{Q}=\sum_{i=1}^{6}\e^i \mathcal{Q}_i+\mathcal{Q}_{\geq7}\,,
\qquad \mathtt{Q}_i:=\opw\sm{\ii \mathtt{q}_i
(\vphi,x,\x)}{0}{0}{-\ii \ov{\mathtt{q}_i(\vphi,x,\x)}}
+\mathcal{Q}_i\,,
\]
where $\mathtt{q}_i$, $\mathcal{Q}_i$, $i=1,\ldots,6$, 
are respectively $i$-homogeneous symbols
and operators as in \eqref{simboOMO2DEF}, \eqref{simboOMO2DEF2}.
%By Lemma $C.7$ in \cite{FGP1} and b
By Lemma $A.4$ in \cite{FGP}, where it is proved the immersion of smoothing pseudo differential tame operators in the class of modulo tame operators,
we have that 
$\mathtt{Q}_{\geq7}:=\opw(\mathtt{q}_{\geq7})+\mathcal{Q}_{\geq 7}$
is Lip-$(-1/2)$-modulo tame with
\begin{equation}\label{scossa3}
\begin{aligned}
&
{\mathfrak M}_{\mathtt{Q}_{\geq7}}^{{\sharp, \g}} (-1/2,s,\mathtt{b_0}) \lesssim_{s}
|\mathtt{q}_{\geq7}|_{-\frac{1}{2},s+\mu,\alpha}^{\gamma,\Omega_{\infty}^{2\gamma}}+
\mathbb{M}^{\gamma}_{\mathcal{Q}_{\geq7}}(s, \mathtt{b})\,,\\
&{\mathfrak M}_{\Delta_{12}\mathtt{Q}_{\geq7}}^{\sharp}(-1/2,p,\mathtt{b}_0)\lesssim_{p}
|\Delta_{12}\mathtt{q}_{\geq7}|_{-\frac{1}{2},p+\mu,\alpha}+
\mathbb{M}_{\Delta_{12}\mathcal{Q}_{\geq7}}(s, \mathtt{b})
\|i_{1}-i_{2}\|_{p+\mu}\,,
%\label{scossa4}
\end{aligned}
\end{equation}
for some $\mu=\mu(\nu)>0$ and $\mathtt{b}=s_0+\mathtt{b}_0$. 
%defined in Def. \ref{Cuno}.
By estimates
\eqref{AVBDEF}, %\eqref{AVB2DEF}, 
\eqref{AVBDEF2}
%, \eqref{AVB2DEF2}
with $k=4$
we have that the \eqref{scossa3} %, \eqref{scossa4}
imply the \eqref{scossa1}.
%, \eqref{scossa2}.

By Remark \ref{homogalmostdiag} the operators $\mathtt{Q}_i$, $i=1, \dots, 6$ are 
almost-diagonal
and $x$-translation invariant.
Moreover we have that
$
\langle D\rangle^{1/4} (\mathtt{Q}_{i})_{\s}^{\s}\langle D\rangle^{1/4}$\,,
is a bounded operator
since $(\mathtt{Q}_{i})_{\s}^{\s}=\opw(\ii \mathtt{q}_i)+\mathcal{Q}_i$
is a pseudo differential operator of order $-1/2$ plus a smoothing operator.
The operator
$
\langle D\rangle^{\rho/2} (\mathtt{Q}_{i})_{-\s}^{\s}\langle D\rangle^{\rho/2}$\,,
is bounded since $\mathtt{Q}_i\in\mathfrak{L}_{\rho, p}$.
Then \eqref{formaResto2}
follows by Lemma \ref{almostC12}.
\end{proof}

\begin{proof}[{\bf Proof of Lemma \ref{lemmaeqomologica}}]\label{10.10prooflemma}
Note that the operator ${\bf A}$ given in \eqref{omoeqCCC} is well-defined
 by \eqref{smallDivLinearBNF2} in Lemma \ref{lowerboundDel1}.
Since ${\bf B}$ is $x$-translation invariant, almost-diagonal, using Lemma \ref{lem:momentocomp} and formul\ae\, \eqref{omoeqCCC}, \eqref{deltaUNO2},  
it is easy to check
that also ${\bf A}$ is $x$-translation invariant, almost-diagonal.
Moreover we have that ${\bf B}$ is Hamiltonian,
i.e. ${\bf B}=\ii E \widetilde{{\bf B}}$ with 
$\widetilde{{\bf B}} $ self-adjoint (it satisfies \eqref{self-adjT}).
By Lemma \ref{coeffFouHamilto} we have
\[
\widetilde{\bf B}_{\s,j}^{\s,k}(\ell)=\ov{\widetilde{\bf B}_{\s,k}^{\s,j}(-\ell)}\,, \qquad 
\widetilde{\bf B}_{\s,-j}^{-\s,-k}(\ell)={\widetilde{\bf B}_{\s,k}^{-\s,j}(\ell)}\,.
\]
Using the explicit expression of $\delta^{(p)}_{\s,\s',j,k}(\ell)$ in \eqref{deltaUNO}
one can check that ${\bf A}$ is Hamiltonian.
Passing to the coefficients of the matrices we have that 
the left hand side of 
equation
 \eqref{omoeqCCC2} is given by 
 \begin{equation*}%\label{omoeqCCC2Fou}
-\mathrm{i}\, \delta^{(p)}_{\s,\s',j,k}(\ell){\bf A}_{\s,j}^{\s',k}(\ell)+{\bf B}_{\s,j}^{\s',k}(\ell)\,.
 \end{equation*}
 Then by definition \eqref{omoeqCCC} the equation \eqref{omoeqCCC2} is verified.
 The  \eqref{omoeqCCC4} follows by 
 Lemma \ref{lowerboundDel1},
% \ref{smallDivLinearBNF}, 
 indeed the condition $\max\{|j|,|k|\}\geq \mathtt{C}$ implies that $\lvert j \rvert, \lvert k \rvert\geq C(S)$ by the definition of $\mathtt{C}$ and the momentum relation \eqref{deltaUNO2}.
 %Lemmata \ref{almostC12}, \ref{almostC12bis}.
\end{proof}

\begin{proof}[{\bf Proof of Lemma \ref{compoC12}}]\label{10.11prooflemma}
We observe that
\[
\big(\langle \pa_{\vphi}\rangle^{\mathtt{b}_0}\langle D\rangle^{1/2}
{\bf A}\langle D\rangle^{1/2}\big)_{\s,j}^{\s',k}(\ell)=
\langle \ell\rangle^{\mathtt{b}_0}\langle j\rangle^{\frac{1}{2}}
\langle k\rangle^{\frac{1}{2}}{\bf A}_{\s,j}^{\s',k}(\ell)\,.
\]
Then it is easy to deduce the \eqref{UPC1} by 
using that $|\ell|\leq i$
and $\mathtt{v}\cdot\ell=\s j-\s'k$.
Item $(ii)$ follows by Lemma $A.5$ in \cite{FGP} where estimates for composition of modulo-tame operators are provided.
Regarding item $(iii)$ we have, by the bounds in item $(i), (ii)$,
\begin{align*}
\mathfrak{M}^{ \sharp, \gamma}_{\bf F}(s, \mathtt{b}_0) &\lesssim_s  \mathfrak{M}_{{\bf A}}^{\sharp,\gamma}(s,\mathtt{b_0}) \,\mathfrak{M}_{{\bf B}}^{\sharp,\gamma}(s_0,\mathtt{b_0})
\varepsilon^n (\mathfrak{M}_{{\bf A}}^{\sharp,\gamma}(s_0,\mathtt{b_0}))^{n-1} \sum_{k-n\geq 0} \frac{\varepsilon^{k-n}}{k!}(C({\bf A}))^{k-n}\,\\
&
+\mathfrak{M}_{{\bf B}}^{\sharp,\gamma}(s,\mathtt{b_0})\varepsilon^n (\mathfrak{M}_{{\bf A}}^{\sharp,\gamma}(s_0,\mathtt{b_0}))^{n} \sum_{k- n\geq 0} \frac{\varepsilon^{k-n}}{k!} C({\bf A})^{k-n}\,.
\end{align*}
Hence for $\varepsilon$ small enough we get the bound \eqref{stimaserielie}. To prove \eqref{stimaserieliedelta12} is sufficient to reason as above and using the estimates of item $(i)$ and $(ii)$.
\end{proof}

\begin{proof}[{\bf Proof of Lemma \ref{mercato}}]\label{10.12prooflemma}
$(i)$
We have that
\begin{equation}\label{UPC4}
{\bf C}_{\s,j}^{\s',k}(\ell)=\sum_{\ell_1,k_1, \s''}
{\bf A}_{\s,j}^{\s '',k_1}(\ell_1){\bf B}_{\s'',k_1}^{\s',k}(\ell-\ell_1)
-{\bf B}_{\s, j}^{\s '', k_1}(\ell_1){\bf A}_{\s'',k_1}^{\s',k}(\ell-\ell_1)
\end{equation}
and we observe that when $\s=\s'$ then the matrices 
involved are $1/2$-smoothing, if $\s\neq \s'$ then 
there is always an out of diagonal component of 
$\mathbf{A}$ or $\mathbf{B}$ that is $\rho$-smoothing.
The sum in \eqref{UPC4} is restricted to indexes (recall ${\bf A}$, ${\bf B}$
are almost-diagonal and $x$-translation invariant)
satisfying
\begin{equation}\label{UPC3}
|j-k|=|\mathtt{v}\cdot\ell|\,,\quad |j-k_1|=|\mathtt{v}\cdot\ell_{1}|\,,
\quad |k_1-k|=|\mathtt{v}\cdot(\ell-\ell_1)|\,.
\end{equation}
First of all we have $|j|\leq |k|(1+|\mathtt{v}\cdot\ell|)\leq |k|(1+C(S)(i_1+i_2))$\,.
Therefore
\[
\frac{|k|}{(1+C(S)(i_1+i_2))}\leq |j|\leq |k|(1+C(S)(i_1+i_2))\,, \quad 
\Rightarrow
\quad
\langle j\rangle^{1/2}\langle k\rangle^{1/2}\geq \frac{\langle j,k\rangle}{(1+C(S)(i_1+i_2))}\,.
\]
Notice also that, fixed $j$ and $k$,
the cardinality of indexes $\ell_1,k_1$ satisfying the constraints \eqref{UPC3}
is bounded by 
\[
K(i_1)^{\nu-1}C(S)(i_1+i_2)\,,
\]
with $K>0$ some absolute constant. Then,
using the bounds \eqref{formaResto2AA}
on ${\bf A}, {\bf B}$ and the \eqref{UPC4}
we deduce \eqref{UPC2} with
\[
C({\bf C}) \lesssim C({\bf A})C({\bf B}) (i_1)^{\nu-1}C(S)(i_1+i_2)(1+C(S)(i_1+i_2))\,.
\]

(ii) By an explicit computation on can check
\[
{\bf M}_{+}^{+}:= [{\bf A}_{+}^{+},\mathcal{D}_{+}^{+}]\,, \quad
{\bf M}_{+}^{-}:= {\bf A}_{+}^{-}\mathcal{D}_{-}^{-}-\mathcal{D}_{+}^{+}{\bf A}_{+}^{-}\,.
\]
The bound on the coefficients of ${\bf M}_{+}^{+}$ 
follows reasoning as in item $(i)$.
Using \eqref{simboM}, \eqref{espandoSimboM}
we have that
the coefficients  of ${\bf M}_{+}^{-}$ are
\[
({\bf M})_{+,j}^{-,k}(\ell)=\ii \sum_{k=1}^{3}
\e^{2k}{m}^{(2k)}_{1/2}(\sqrt{|k|}+\sqrt{|k|}){\bf A}_{+,j}^{-,k}(\ell)\,.
\]
Then the second in \eqref{formaResto2AAbis}
follows by the estimates \eqref{formaResto2AA}
on the coefficients of ${\bf A}_{+}^{-}$.
Item $(iii)$ follows by  item $(ii)$ using the estimates in \eqref{merstima}, 
%\eqref{mer12stima}, \eqref{mer12jstima}. 
and by item $(i)$ of Lemma \ref{compoC12}.
\end{proof}

\end{document}